\documentclass[11pt,a4paper,reqno]{amsart}
\usepackage[applemac]{inputenc}
\usepackage[T1]{fontenc}
\usepackage{amsmath}
\usepackage{amsthm}
\usepackage{amsfonts}
\usepackage{amssymb}
\usepackage{graphicx}
\usepackage{amsbsy}
\usepackage{mathrsfs}
\usepackage{bbm}
\newtheorem{theorem}{Theorem}[section]
\newtheorem{lemma}[theorem]{Lemma}

\newtheorem{proposition}[theorem]{Proposition}
\newtheorem{corollary}[theorem]{Corollary}
\newtheorem{definition}[theorem]{Definition}
\theoremstyle{remark}
\newtheorem{remark}[theorem]{\it \bf{Remark}\/}

\numberwithin{equation}{section}
\catcode`@=11
\def\section{\@startsection{section}{1}%
  \z@{1.5\linespacing\@plus\linespacing}{.5\linespacing}%
  {\normalfont\bfseries\large\centering}}
\catcode`@=12
\newcommand{\be}{\begin{equation}}
\newcommand{\ee}{\end{equation}}
\newcommand{\bea}{\begin{eqnarray}}
\newcommand{\eea}{\end{eqnarray}}
\newcommand{\bee}{\begin{eqnarray*}}
\newcommand{\eee}{\end{eqnarray*}}

\def\pa{\partial}

\def\RR{\mathbb{R}}

\def\fref#1{{\rm (\ref{#1})}}

\catcode`@=11
\def\supess{\mathop{\operator@font Sup\,ess}}
\catcode`@=12
\def\un{{\mathbbmss{1}}}

\def\RR{\mathbb{R}}

\def\e{\varepsilon}

\def\bar#1{{\overline #1}}
\def\fref#1{{\rm (\ref{#1})}}

\def\R2+{\RR ^2_+}

\def\pa{\partial}

\def\lim{\mathop{\rm lim}}

\def\l{\lambda}

\def\log{{\rm log}}

\def\pa{\partial}

\def\pa{\partial}

\def\matchal{\mathcal}

\def\bos{\boldsymbol}

\title[]{Type II blow up manifolds for the energy supercritical semilinear wave equation}
\author[C.Collot]{Charles Collot}
\address{Laboratoire J.A. Dieudonn\'e, Universit\'e de Nice-Sophia Antipolis, France}
\email{ccollot@unice.fr}


\begin{document}
\maketitle

\begin{abstract}
We consider the semilinear focusing wave equation 
$$\pa_{tt}u-\Delta u-u|u|^{p-1}=0$$
in large dimensions $d\geq 11$ and in the radial case. For a range of supercritical nonlinearities $p>p(d)>1+\frac{4}{d-2}$, for each integer large enough $\ell>\alpha (d,p)>2$, we construct a Lipschitz manifold of codimension $\ell -1$ of solutions blowing up in finite time $T$ by concentrating the soliton (stationnary state) profile:
$$u(t,r)\sim \frac{1}{\l(t)^{\frac 2{p-1}}}Q\left(\frac{r}{\l(t)}\right)$$
at the quantized  blow up rate : 
$$\l(t)\sim c_u(T-t)^{\frac \ell\alpha}.$$
The solutions can be chosen $C^{\infty}$ and compactly supported. In that case the blow up is of type II i.e all norms below scaling remain bounded 
$$\limsup_{t\uparrow T}\|\nabla^su(t),\nabla^{s-1}\partial_tu(t)\|_{L^2}<+\infty\ \ \mbox{for}\ \ 1\leq s<s_c=\frac{d}{2}-\frac{2}{p-1}.$$
Our analysis adapts the robust energy method developed for the study of energy critical bubbles \cite{MRRod1}, \cite{RaphRod}, \cite{RSc1,RSc2}, the study of this issue for the supercritical semilinear heat equation \cite{Velas}, \cite{MaM2}, \cite{Mizo1,Mizo2} and the analogous result for the energy supercritical schr\"odinger equation \cite{MRRod2}. 
\end{abstract}


\section{Introduction}



\subsection{The semilinear wave equation}


We study in this paper the focusing semilinear wave equation in the radial case:
\be
\label{nlw}
(NLW)\ \ \left\{\begin{array}{ll}\pa_{tt}u-\Delta u-|u|^{p-1}u=0, \\u_{|t=0}=u_0, \ \ \pa_tu|_{|t=0}=u_1,\end{array}\right.\ \ (t,x)\in \Bbb R^+\times \Bbb R^d, \ \ u(t,x)\in \Bbb R.
\ee
If $u(t,x)$ is a solution then $u_\l(\lambda t,x)=\l^{\frac2{p-1}}u(\l t,\l x)$ for $\l>0$ is also a solution. This scaling symmetry is an isometry of the critical homogeneous Sobolev space 
$$\|u_\l(\lambda t,\cdot),\partial_t (u_\l(\lambda t,\cdot))\|_{\dot{H}^{s_c}\times\dot{H}^{s_c-1}}=\|u(\l t,\cdot),(\partial_t u)(\lambda t,\cdot)\|_{\dot{H}^{s_c}\times \dot{H}^{s_c-1}}$$
for $s_c=\frac d 2-\frac {2}{p-1}.$ Here we consider energy supercritical nonlinearities: 
$$p> 2^*-1=1+\frac{4}{d-2}, \ \ d\ge 3\ \ (\mbox{ie}\ \ s_c> 1).$$
Under these conditions (NLW) is locally well posed in $H^{s_c}\times H^{s_c-1}$ (see \cite{LS}, \cite{Tao} and references therein). If the nonlinearity is analytic $p=2q+1, \ \ q\in \Bbb N^*,$ the flow propagates Sobolev regularity. If a solution has a finite maximal time of existence $T<+\infty$ (we then say that it blows up) all supercritical norms must explode: 
$$ \lim_{t\uparrow T}\|u(t),u_t(t)\|_{\dot{H}^s\times \dot{H}^{s-1}}=+\infty \ \ \mbox{for}\ \ s>s_c.$$
This paper is part of the study of the blow up phenomenon. We investigate here a special blow up scenario: the concentration at a point.


\subsection{Blow up for (NLW)}


The question of singularity formation for (NLW) has attracted a considerable attention since the pioneering works by John \cite{john}. From finite speed of propagation, using the constant in space solution
$$
u(t,x)=\frac{C}{(T-t)^{\frac{2}{p-1}}}
$$
one can construct solutions blowing up like the ODE $u_{tt}=u^p$. They are called type I blow up and correspond to a {\it complete} blow up 
$$\lim_{t\to T}\| u(t),\partial_tu(t)\|_{\dot{H}^1\times L^2}\to +\infty.$$
In the subconformal case ($s_c\leq \frac{1}{2}$) the recent works by Merle and Zaag \cite{MZ1,MZ5,MZ2} give in particular a complete description of the local singularity, being always a type I blow up bubble, and we refer to this monumental series of works for complete references on the history of the problem. Recently also, general upper bounds on the blow up rates have been obtained \cite{HZ}, \cite{KV} in the superconformal, energy subcritical case ($\frac{1}{2}<s_c<1$).\\

The situation is much more poorly understood in the energy critical and super critical regime $s_c\geq 1$. In this case, a new stationary solution arises: the soliton profile $Q$ which is the unique up to scaling radially symmetric solution to $$\Delta Q+Q^p=0, \ \ d\geq 3, \ \ s_c\geq 1.$$ Other blow ups than Type I appear, because this profile is at the heart of a new concentration phenomenon. The first construction of such blow up solution in the energy critical setting goes back to Krieger, Schlag and Tataru \cite{KSTwave} (the result being improved in \cite{KS}) in dimension 3 where blow up bubbles of the type 
\be
\label{cbejbevbive}
u(t,r)\sim \frac{1}{\l(t)^{\frac12}}Q\left(\frac{r}{\l(t)}\right),  \ \ \l(t)\sim (T-t)^{\nu}
\ee
 for all $\nu>1$ are constructed. This result is a by product of the approach developed for the 2-dimensional energy critical wave map problem in the seminal work \cite{KST}. A different approach is implemented in \cite{HR} in the continuation of the energy method developed by Merle, Rapha\"el and Rodnianski for the study of the energy critical wave map problem \cite{RaphRod} and the energy critical Schr\"odinger map problem \cite{MRRod1}. In particular Hillairet and Rapha\"el obtained in the energy critical case in dimension 4 blow up bubbles of the form
\be
\label{cbejibeiiev}
u(t,r)\sim \frac{1}{\l(t)}Q\left(\frac{r}{\l(t)}\right),  \ \ \l(t)\sim (T-t)e^{-\sqrt{|\log (T-t)|}}.
\ee
 An essential difference between these two constructions is the rigidity in the law \fref{cbejibeiiev} with respect to the continuum of blow up speeds \fref{cbejbevbive} which reflects the fact that all solutions corresponding to \fref{cbejibeiiev} are arbitrarily smooth, while the continuum \fref{cbejbevbive} generically corresponds to the propagation of a singularity on the light cone. In the related work \cite{RSc2} for the energy critical heat flow, the existence of a countable family of blow up rates for $\matchal C^{\infty}$ data is showed, and this result could be propagated to the energy critical wave equation as well.
\\

In the energy super critical setting $s_c>1$, much less is known. Recently, Duyckaerts, Kenig and Merle showed in \cite{DKM} the explosion of the critical norm at blow up time in dimension $3$ (see also \cite{KM} and \cite{Dod} in dimension $5$). Type I or ODE blow up solutions of course still exist, and their stability is addressed by Donninger and Sch\"orkhuber in \cite{Donn}. The existence of large global solutions is proved by Krieger and Schlag in \cite{KS2}. In comparison, the nonlinear heat equation is better understood. Another type of blow up solution was predicted in large dimensions $d\ge 11$ and large nonlinearities $p\geq p(d)$ in the pioneering work by Herrero and Velasquez \cite{Velas}. These so called type II blow up bubbles are rigorously constructed in \cite{Mizo1, Mizo2} using the breakthrough approach developed by Matano and Merle \cite{MaM1,MaM2}. The collection of these works yields a complete classification of the type II blow up scenario for the radially symmetric energy supercritical heat equation. The main restriction of these techniques however is the systematic use of the maximum principle which cannot be extended to the dispersive setting. 
\\

\noindent In the breakthrough work \cite{MRRod2}, the authors fully revisit the construction of type II blow up bubbles and show how the energy critical approach developed in \cite{RaphRod}, \cite{MRRod1}, \cite{RSc1, RSc2} can be extended to the energy supercritical setting to construct type II blow up solutions for the Schr\"odinger equation. Our main claim in this paper is that this analysis can be propagated to the wave equation to construct the first family of type II blow up bubbles in the energy supercritical setting.


\subsection{Statement of the result}


Let us introduce some numbers attached to the super critical numerology. Let $d\geq 11$ and let the Joseph-Lundgren exponent be 
\be
\label{exponentpjl}
p_{JL}=1+\frac{4}{d-4-2\sqrt{d-1}}.
\ee
Then for $p>p_{JL}$, the soliton profile admits an asymptotic expansion \be
  \label{expansionq}
  Q(r)=\frac{c_{\infty}}{r^{\frac{2}{p-1}}}+\frac{a_1}{r^{\gamma}}+o\left(\frac1{r^{\gamma}}\right), \ \ a_1\neq 0,
  \ee 
  with 
 \be
 \label{defgamma}
 c_\infty=\left[\frac{2}{p-1}\left(d-2-\frac 2{p-1}\right)\right]^{\frac 1{p-1}}, \ \ \gamma=\frac12(d-2-\sqrt{{\triangle}})>0
 \ee
 and where $${\triangle}=(d-2)^2-4pc_{\infty}^{p-1}\ \ (\triangle>0 \  \mbox{for}\ \ p>p_{JL}).$$ These numbers describing the asymptotic behavior of the soliton are essential in the analysis of type II blow up bubbles and we claim:

\begin{theorem}[Type II blow up for the energy super critical wave equation]
\label{thmmain}
Let $d\geq 11$, $p_{JL}$ be given by \fref{exponentpjl} and a nonlinearity 
\be
\label{plarge}  p=2q+1,\ \ q\in \Bbb N^*,\ \ p> p_{JL}.
\ee 
Let $\gamma$ be given by \fref{defgamma} and define:
\be \label{intro:eq:def alpha}
\alpha=\gamma-\frac{2}{p-1}.
\ee
Assume moreover:
\be
\label{conditiongamma}
\left(\frac d2-\gamma\right)\notin \Bbb N.
\ee
Pick an integer 
\be
\label{codnionrionfeell}
 \ell\in \Bbb N \ \ \mbox{with}\ \ \ell>\alpha,
 \ee
 and a large enough regularity exponent 
$$s_+\in \Bbb N, \ \ s_+\geq s(\ell) \ (s(\ell) \to +\infty\ \ \mbox{as}\ \ \ell\to +\infty).$$ Then there exists a radially symmetric initial data $(u_0, u_1)\in H^{s_+}\times H^{s_+-1}(\Bbb R^d)$ such that the corresponding solution to \fref{nlw} blows up in finite time $0<T<+\infty$ by concentrating the soliton profile:
\be
\label{concnenergy}
u(t,r)=\frac{1}{\l(t)^{\frac 2{p-1}}}(Q+\e)\left(\frac{r}{\l(t)}\right)
\ee
with:\\
\noindent{\em (i) Blow up speed}: 
\be
\label{Pexciitedlaw}
\l(t)=c(u_0)(1+o_{t\uparrow T}(1))(T-t)^{\frac{\ell}{\alpha}}, \ \ c(u_0)>0;
\ee
\noindent{\em (iii) Asymptotic stability above scaling in renormalized variables}: 
\be 
\label{intro:eq:convergence surcritique}
\lim_{t\uparrow T}\|\e(t,\cdot),\lambda(\partial_t u)_{\lambda}(t,\cdot)\|_{\dot{H}^s\times \dot{H}^{s-1}}=0  \ \ \mbox{for all}\ \ s_c<s\leq s_+;
\ee
\noindent{\em (iv) Boundedness below scaling}:
\be
\label{intro:eq:bornitude sous critique}
\limsup_{t\uparrow T}\|u(t),\partial_t u(t)\|_{\dot{H}^s\times \dot{H}^{s-1}}<+\infty  \ \ \mbox{for all}\ \ 1\leq s<s_c;
\ee
\noindent{\em (v) Behavior of the critical norms}:
\be
\label{intro:eq:comportement norme critique}
\|u(t)\|_{\dot{H}^{s_c}}=\left[c(d,p)\sqrt{\ell}+o_{t\uparrow T}(1)\right]\sqrt{|\log(T-t)|},
\ee
\be
\label{intro:eq:comportement norme critique 2}
\limsup_{t\uparrow T} \|\partial_t u(t)\|_{\dot{H}^{s_c-1}}<+\infty.
\ee
\end{theorem}

The proof of Theorem \ref{thmmain} relies on an explicit construction of blow up solutions. It allows us to find a whole set of initial data leading to such a blow up, and to investigate its topological properties:

\begin{theorem} \label{thmmain2}
We keep the notations and assumptions of Theorem \ref{thmmain}. Let a slightly supercritical regularity exponent $\sigma=\sigma(\ell)$ satisfying:
$$
0<\sigma-s_c\ll 1 \ \ \text{small enough}.
$$
There exists a locally Lipschitz manifold of codimension $\ell-1$ in the Banach space $\dot{H}^{\sigma}\cap \dot{H}^{s_+}\times \dot{H}^{\sigma-1}\cap \dot{H}^{s_+-1}$ of initial data leading to the blow up scenario described by Theorem \ref{thmmain}. We point out that as $\alpha>2$, the codimension satisfies $\ell-1>2$.
\end{theorem}

\noindent {\it Comments on Theorem \ref{thmmain} and Theorem \ref{thmmain2} }\\

\noindent{\it 1. On the assumptions on the nonlinearity $p$}. The assumption \fref{conditiongamma} is a technical one that avoids the presence of logarithmic losses in some weighted Hardy inequalities that we use; a similar assumption can be found in \cite{MRRod2}. These logarithms appear in some analysis tools, but not in the construction of the approximate blow up solution as in the critical settings \cite{HR}, \cite{RaphRod}, \cite{MRRod1}. This is why we believe that the assumption \fref{conditiongamma} could be removed. The assumption $p=2q+1$ makes the nonlinearity analytic and hence $ C^{\infty}$ regularity is propagated by the flow. For a nonlinearity with limited regularity, given a large integer $\ell$, a blow up solution satisfying \fref{Pexciitedlaw} can be constructed for $p\geq p(\ell)$ large enough using the same methodology.\\

\noindent{\it 2. The manifold construction}. To prove Theorem \ref{thmmain} we employ an approximate blow up profile having $\ell -1$ directions of unstability. We use Brouwer's fixed point theorem to obtain the existence of an initial datum for which these unstable modes stay under control, following \cite{MRRod2}, \cite{RSc2}. Here we further show the Lipschitz regularity of the set of initial data we consider. All other solutions starting close to that manifold undergo finite time instabilities and leave a neighborhood of that manifold. Nonetheless we cannot say anything once they have left. Manifold construction is an important step toward the understanding of the dynamics near the ground state (see \cite{KNS} for the energy critical wave equation) and the control of some parameters in blow up dynamics can be subtle \cite{KS3}.\\

\noindent{\it 3. On quantization of blow up rates}. The quantization of blow up rates \fref{Pexciitedlaw} exists in the case of the heat equation where it is sharp (a classification theorem is in \cite{Mizo2}), and for the harmonic heat flow \cite{RSc2}. It is a consequence of the regularity and decay associated to our initial data which in particular can be chosen in $\mathcal C^{\infty}_{c}(\Bbb R^d)$.\\

The strength and robustness of our approach is first that it relies on the derivation of the universal system of ODE's driving the evolution of the approximate blow up profile avoiding any sort of matching procedure, second that the control of the error term is performed using energy estimates only and not spectral estimates. For both these reasons, we expect that our analysis can be propagated to the non radial problem as well, this will be addressed in a forthcoming work.\\

\noindent{\bf Acknowledgment}. The author is supported by the ERC advanced grant BLOW-DISOL. This paper is part of the author PhD, and I would like to thank my advisor P. Rapha\"el for his guidance and advice during the preparation of this work.\\
\\


\noindent{\bf Notations}: Here are the main notations and relations used all along the paper.\\
\underline{Super critical numerology:} Given $d\geq 11$, $p> p_{JL}$ (defined in \fref{exponentpjl}), we let $\alpha$ and $\alpha_2$ be the roots of the polynomial $X^2-(d-2-\frac{4}{p-1})X+2(d-2-\frac{2}{p-1})$ satisfying $\alpha<\alpha_2$. One can check that the condition $p>p_{JL}$ ensures the reality of $\alpha$ and $\alpha_2$, and that they are not equal (see Lemma \ref{annexe:lem:valeur dalpha}). This definition is coherent with the formula \fref{intro:eq:def alpha}. We recall the following relation:
$$
\alpha = \gamma -\frac{2}{p-1}>2,
$$
where $\gamma$ was defined in \fref{defgamma}. We define\footnote{where we recall the definition of the entire part $E[x]\leq x < E[x]+1, \ E(x)\in \mathbb{Z}$.}:

\begin{equation} \label{intro:eq:def k0}
\left\{ \begin{array}{l l} k_0:=E[\frac{d}{2}-\gamma]>1, \\
\delta_0:=\frac{d}{2}-\gamma-k_0, \ 0< \delta_0<1.
\end{array} \right.
\end{equation}

\noindent because we are assuming $\left( \frac{d}{2}-\gamma \right) \notin \mathbb{N}$, so that

\begin{equation}
d=2\gamma+2k_0+2\delta_0.
\end{equation}
We let
\begin{equation}\label{eq:def:gain1}
g:= \text{min}(\alpha,\alpha_2-\alpha_1)-\epsilon>0
\end{equation}
and
\begin{equation}\label{eq:def:gain2}
g':=\text{min}(g,2,1+\delta_0-\epsilon)>0
\end{equation}
be the two real numbers that will quantify some gain in the asymptotics of our objects later on. $\epsilon$ stands for a very small constant $0<\epsilon\ll 1$ that can be chosen independently of the sequel. The presence of $-\epsilon$ and $1+\delta_0$ is just a way to simplify the writing of results later on.\\
\noindent \underline{Notations for the analysis:} For the sake of simplicity, we will use the following equivalent formulation for the focusing nonlinear wave equation (NLW):

\begin{equation}\label{eq:NLW}
(NLW) \ \left\{ \begin{array}{l l} \partial_{t} \bos{u} =\bos{F}(\bos{u}), \\  \bos{u}_{|t=0}=\bos{u}_0 \end{array}  \right. \ \ \ (t,x)\in\mathbb{R}^+ \times \mathbb{R}^d, \ \bos{u}(t,x): \mathbb{R}^d\rightarrow \mathbb{R}\times\mathbb{R}.
\end{equation}

\noindent We will consider radial solutions: $\bos{u}(x)=\bos{u}(r)$ where $r=|x|$. We refer to the coordinates of a function $\bos{u}$ as $u^{(1)}$ and $u^{(2)}$:
\be \label{intro:eq:def vecteur}
\bos{u}=\begin{pmatrix}
u^{(1)} \\
u^{(2)}
\end{pmatrix} .
\ee
We let the expression $\bos{F}$ be:

\begin{equation}
\bos{F}(\bos{u}):=\begin{pmatrix}
u^{(2)} \\ \Delta u^{(1)}+f(u^{(1)})
\end{pmatrix}, \ f(t):=|t|^{p-1}t .
\end{equation}

\noindent The bold notations will always refer to vectors. We make an abuse of notation (regarding \fref{intro:eq:def vecteur}) by still denoting the stationnary state introduced earlier by $\bos{Q}$:
$$
\bos{Q}:=\begin{pmatrix}
Q \\ 0
\end{pmatrix} .
$$
Given a large integer $L\gg 1$, we define the Sobolev exponent:

\begin{equation} \label{intro:eq:def sL}
s_L:=k_0+1+L .
\end{equation}

\noindent We will use the standard scalar product on $L^2(\mathbb{R}^d)$ and $L^2(\mathbb{R}^d)\times L^2(\mathbb{R}^d)$:
$$
\langle u, v \rangle := \int_{\mathbb{R}^d} uv \ \ \text{and}\ \ \ \langle \bos{u}, \bos{v} \rangle := \int_{\mathbb{R}^d} u^{(1)}v^{(1)}+\int_{\mathbb{R}^d} u^{(2)}v^{(2)} .
$$
Let $0<\lambda$, we denote the renormalized function by:
\begin{equation}\label{intro:eq: def fonction renormalisee}
\bos{u}_{\lambda}(x):=\begin{pmatrix}
\lambda^{\frac{2}{p-1}}u^{(1)}(\lambda y) \\
\lambda^{\frac{2}{p-1}+1}u^{(2)}(\lambda y)
\end{pmatrix} .
\end{equation}
The rescaled coordinates are then:
\be
\bos{u}_{\lambda}:=\begin{pmatrix}
u^{(1)}_{\lambda } \\
u^{(2)}_{\lambda}
\end{pmatrix} .
\ee
We let the generator of the scaling be:
$$
\bos{\Lambda} \bos{u} := \begin{pmatrix}
\Lambda^{(1)}u^{(1)} \\
\Lambda^{(2)}u^{(2)}
\end{pmatrix}
:= \begin{pmatrix}
\left( \frac{2}{p-1} +y.\nabla\right)u^{(1)} \\
\left( \frac{2}{p-1}+1 +y.\nabla\right)u^{(2)}
\end{pmatrix} .
$$
We introduce the renormalized space variable:
$$
y:=\frac{r}{\lambda} .
$$
Given $b_1>0$, we define:

\begin{equation}\label{eq:def:B1etB0}
B_0:=\frac{1}{b_1}, \ B_1:=B_0^{1+\eta}
\end{equation}

\noindent where $\eta$ is a small number $0<\eta\ll 1$ which will be choosen later. We denote by
$$
\begin{array}{r c l}
\mathcal{B}^n(R)&:=&\{x=(x_1,...,x_n)\in \mathbb{R}^n, \ \sum_{i=1}^d x_i^2\leq R^2 \}, \\
\mathcal{S}^n(R)&:=&\{x=(x_1,...,x_n)\in \mathbb{R}^n, \ \sum_{i=1}^d x_i^2=R^2 \} ,\\
\mathcal{C}^n(r,R)&:=&\{x=(x_1,...,x_n)\in \mathbb{R}^n, \ r^2\leq \sum_{i=1}^d x_i^2\leq R^2 \},
\end{array}
$$
the standard closed ball, sphere and ring of the standard euclidian n-dimension real space. For $u\in \mathbb{R}^n$ we denote the standard euclidian norm by:
$$
|u|:=\left(\sum_{i=1}^n u_i^2\right)^{\frac{1}{2}}.
$$ 
We introduce a generic radial, $C^{\infty}$ cut-off function:
\be \label{intro:eq:definition chi}
\chi \equiv 1 \ on \ \mathcal{B}^d(1), \ \chi\equiv 0 \ on \ \mathbb{R}^d\backslash \mathcal{B}^d(2).
\ee
And we adjust the zone of the cut by denoting, for $B>0$:
\be \label{intro:eq:definition scaling chi}
\chi_B:y\mapsto \chi\left(\frac{y}{B} \right).
\ee 

\noindent We use the Kronecker delta notation:
\be
\delta_{i,j}:= \left\{ \begin{array}{l l} 1 \ \text{if} \ i=j ,\\ 0 \ \text{otherwise}. \end{array} \right.
\ee

\noindent \underline{Analysis near the ground state:} The linearized operator near $\bos{Q}$ of equation \fref{eq:NLW} is given by:

\begin{equation}\label{eq:def:H}
\bos{H}\bos{\varepsilon}:=\begin{pmatrix}
-\varepsilon^{(2)} \\ -\Delta \varepsilon^{(1)} - pQ^{p-1}\varepsilon^{(1)}\end{pmatrix}=
\begin{pmatrix}
 0 & -1 \\ -\Delta -pQ^{p-1} & 0
\end{pmatrix}
\bos{\varepsilon} ,
\end{equation}
so that:
\begin{equation}
\bos{F}(\bos{Q}+\bos{\varepsilon})=-\bos{H}\bos{\varepsilon}+\bos{NL}.
\end{equation}
Here $\bos{NL}$ stands for the purely nonlinear term:

\begin{equation}
\bos{NL}:=\begin{pmatrix}
0 \\ f(Q+\varepsilon^{(1)})-f(Q)-pQ^{p-1}\varepsilon^{(1)}
\end{pmatrix}.
\end{equation}
We define:
\begin{equation}
\mathcal{L}:=-\Delta -pQ^{p-1},
\end{equation}
so that:
\begin{equation}
\bos{H}=\begin{pmatrix}
0 & -1 \\
\mathcal{L} & 0
\end{pmatrix}.
\end{equation}
Eventually, we note the potential:
\begin{equation}
V:=pQ^{p-1} .
\end{equation}


\subsection{Strategy of the proof}


We start by a summary of the main ideas involved in the proof of Theorem \ref{thmmain} and Theorem \ref{thmmain2}. We employ the same notations as in the critical settings, and use the formulation of (NLW) via \fref{eq:NLW}.\\

\emph{(i) Constructing of an approximate blow-up profile:} We study the dynamics close to the family of solitons $\left(\bos{Q}_{\frac{1}{\lambda}}\right)_{\lambda>0}$. We start at the scale $\lambda (0)=1$ and look for a perturbation $\bos{T}_1$ such that at first order the dynamics moves along the branch:
\be \label{intro:eq:matching premier ordre}
-\lambda_t\bos{\Lambda} \bos{Q} =\lambda_t \frac{\partial}{\partial \lambda}\left(\bos{Q}_{\frac{1}{\lambda}}\right)_{|\lambda=1} \sim \partial_t\left(\bos{Q}_{\frac{1}{\lambda}}+b_1\bos{T}_{1,\frac{1}{\lambda}}\right)_{|t=0}= \bos{F}(\bos{Q}+b_1\bos{T}_1)\sim -b_1\bos{H}(\bos{T}_1).
\ee
So $\bos{T}_1$ is given by: $\bos{T}_1=-\bos{H}^{-1} \bos{\Lambda} \bos{Q}$. When applying a scale change we get: 
$$
\bos{F}\left(\bos{Q}_{\frac{1}{\lambda}}+b_1\bos{T}_{1,\frac{1}{\lambda}}\right)\sim -b_1 \frac{\partial}{\partial \lambda'}\left(\bos{Q}_{\frac{1}{\lambda'}}\right)_{|\lambda'=\lambda}.
$$
Consequently, for the approximate solution $\bos{Q}_{\frac{1}{\lambda(t)}}+b_1(t)\bos{T}_{1,\frac{1}{\lambda(t)}} $ the evolution of the scaling is given by $\lambda_t=-b_1$. $b_1$ is supposed to be a small parameter. In the previous equations, we omitted the time evolution of $b_1\bos{T}_{1,\frac{1}{\lambda}}$, and the non linear terms $\bos{NL}$ because we expect them to be of higher order. We now include them in \fref{intro:eq:matching premier ordre} to look for a time evolution of $b_1$ given by higher order terms:
$$
b_{1,t}\bos{T}_{1,\frac{1}{\lambda}}+\frac{b_1^2}{\lambda} \bos{\Lambda} \bos{T}_{1,\frac{1}{\lambda}} \sim b_{1,t}\bos{T}_{1,\frac{1}{\lambda}}-\lambda_t \frac{b_1}{\lambda} \bos{\Lambda} \bos{T}_{1,\frac{1}{\lambda}}  = \partial_t\left(b_1\bos{T}_{1,\frac{1}{\lambda}}\right) \sim \bos{NL}.
$$
Surprisingly, as will be explained just afterwards, one has that $\bos{\Lambda} \bos{T}_{1,\frac{1}{\lambda}}\sim (1-\alpha)\bos{T}_{1,\frac{1}{\lambda}}$, and that $\bos{NL}$ is negligeable compared to $\frac{b_1^2}{\lambda}\bos{T}_{1,\frac{1}{\lambda}}$. So we end up with: $b_{1,t}=-\frac{1-\alpha}{\lambda}b_1^2$. In short: we have a perturbation that at first order makes the solution move along the branch, and at second order influences its own time evolution, the error in this approximation being of third order.\\
\\
In the same spirit, to allow additional movement along $\bos{T}_1$ we let $\bos{T}_2=-\bos{H}^{-1}(\bos{T}_1)$, and do the same matching technique for the profile $\bos{Q}_{\frac{1}{\lambda}}+b_1\bos{T}_{1,\frac{1}{\lambda}}+b_2\bos{T}_{2,\frac{1}{\lambda}}$ with $b_2$ of order $ b_1^2$ since we already know that $\lambda b_{1,t}$ should be of this order. This gives: $\lambda_t=-b_1$, $b_{1,t}=\frac{1}{\lambda}(-(1-\alpha)b_1^2+b_2) $, and $b_{2,t}=-\frac{2-\alpha}{\lambda}b_1b_2$.\\
\\
Letting $\bos{T}_{i}=(-1)^i\bos{H}^{-i}\bos{\Lambda} \bos{Q}$ and considering a general approximate profile of the form $\bos{Q}_{b,\frac{1}{\lambda}}:=\bos{Q}_{\frac{1}{\lambda}}+\sum_{i=1}^L b_i \bos{T}_{i,\frac{1}{\lambda}}$ gives in turn at first orders\footnote{with the convention $b_{L+1}\equiv 0$.}:
\be\label{intro:sd des parametres}
\left\{
\begin{array}{r c l}
\lambda_t&=&-b_1, \\
b_{i,t}&=&\frac{1}{\lambda}(-(i-\alpha)b_1b_i+b_{i+1}).
\end{array}
\right.
\ee
We point out at this stage that what we are doing is to build an approximate center manifold $\mathcal{M}_{ap}=\{ (\bos{Q}_{b,\frac{1}{\lambda}})_{b,\lambda}\}$ close to $\left(\bos{Q}_{\frac{1}{\lambda}} \right)_{\lambda>0}$, tangent to the vector space $Span(\bos{T}_i)$ being the generalized kernel of the operator $\bos{H}$. This manifold is determined by $L+1$ parameters. Thanks to a matching technique we have an insight for the parameters behavior under the dynamics of (NLW): their time evolution should be given by \fref{intro:sd des parametres}. We now explain what is the matching technique.\\
\\

\emph{(ii) Tail dynamics:} When constructing the profiles $\bos{T}_i$ one has\footnote{where mod stands for the Euclidean division $a=a \ mod2+b$, $0\leq b \leq 1$.}:
\be \label{intro:eq:comportement asymptotique Ti}
\bos{T}_i(r)\sim r^{-\gamma+i-(i \ mod 2)}, \ as \ r\rightarrow +\infty.
\ee
Hence for $i$ big enough it has an irrelevant growth at infinity. For this reason, to obtain a reasonable approximate profile we cut the $\bos{T}_{i,\frac{1}{\lambda}}$'s in the zone $y\sim B_1$ because it is the zone where $b_i\bos{T}_i$ has the same size as $\bos{\Lambda} \bos{Q}$. The true approximate profile is in fact of the form $(\bos{Q}+\chi_{B_1}\sum_{i=1}^L \bos{T}_i)_{\frac{1}{\lambda}}$.\\
\\
The important computations of the analysis are then done in the zone $r\sim \lambda B_1$ (in renormalized variable $y\sim B_1$), where $b_i\bos{T}_{i,\frac{1}{\lambda}}$ behaves like \fref{intro:eq:comportement asymptotique Ti} because $B_1\gg 1$. As we will compute later in the analysis:
$$
\Lambda \bos{T}_i = (i-\alpha)\bos{T}_i+O(r^{-\gamma+i-(i \ \text{mod}2)-g'}), \ as \ r\rightarrow +\infty.
$$
This explains why we say that $\Lambda \bos{T}_{i,\frac{1}{\lambda}} \sim (i-\alpha)\bos{T}_{i,\frac{1}{\lambda}}$: their difference is of lower order in the relevant zone $y\sim B_1$. To truly understand that point, one has to read the analysis to see how the size of a profile in the zone $y\sim  B_1$ is directly related to a polynomial size in terms of the main parameter $b_1$ for some importants norms of this profile.\\
\\
This way, the system of ODE's \fref{intro:sd des parametres} is just computed on the asymptotics of the profiles. This heuristic has been extensively used in blow-up problems.\\

\emph{(iii) Approximate blow-up profiles:}
The natural question is: what type of special solutions does the approximate dynamics possess? For $\ell>\alpha$, there exists\footnote{see Lemma \ref{lem:soldusystemd}. }  a solution $(\lambda^e(t),b^e(t))$ of \fref{intro:sd des parametres} such that $\lambda^e(t)$ goes to $0$ in finite time $T$ with asymptotics $\lambda^e\sim (T-t)^{\frac{\ell}{\alpha}}$. This means that the approximate dynamics makes $\bos{Q}_{b^e(t),\frac{1}{\lambda^e(t)}}$ blow up in finite time. It is the approximate blow-up profile we are going to work with. We note that for this special solution, the parameters have the following size:
\be \label{intro:estimation parametres}
b_i^e\lesssim (b_1^e)^i, \ b_{i,s}^e\lesssim (b_1^e)^{i+1}
\ee
We write the approximate dynamics under the form:
$$
\begin{array}{r c l}
\bos{F}(\bos{Q}_{b,\frac{1}{\lambda}})&=&-b_1\frac{\partial }{\partial \lambda'}\left(\bos{Q}_{b,\frac{1}{\lambda'}} \right)_{|\lambda'=\lambda} + \frac{1}{\lambda}\sum_{i=1}^L (-(i-\alpha)b_1b_i+b_{i+1})\frac{\partial}{\partial b_i} \left(\bos{Q}_{b',\frac{1}{\lambda}} \right)_{b'=b} \\
&& +\bos{\psi}
\end{array}
$$
where $\bos{\psi}$ denote the remainder which is of higher order.\\

\emph{(iv) Obtaining a blow-up solution for the full dynamics:} We now want to prove that this special solution persists in the full (NLW) dynamics. We look for a true solution under the form $\bos{u}(t)=\bos{Q}_{b(t),\frac{1}{\lambda(t)}}+\bos{\varepsilon}_{\frac{1}{\lambda}}(t)$. $\bos{\varepsilon}$ is the error term "orthogonal" to the manifold $\mathcal{M}_{ap}$. $b(t)=b^e(t)+b'(t)$ and $\lambda(t)=\lambda^e(t)+\lambda'(t)$ are perturbations of the special trajectory $(b^e(t),\lambda^e(t))$, they represent the projection of $\bos{u}$ on the manifold $\mathcal{M}_{ap} $. We hope to find a solution for which $\bos{\varepsilon}$, $b'$ and $\lambda'$ stay small, so that the blow-up still happens.\\
\\
To do that we use a bootstrap technique. We look at all the solutions starting in a neighborhood $\mathcal{O}$ of the curve $\left(\bos{Q}_{b^e(t),\frac{1}{\lambda^e(t)}}\right)_{0\leq t\leq T}\subset \mathcal{M}_{ap}$, and we prove that at least one has to stay in this neighborhood, leading to a blow up. We write:
$$
\mathcal{O}=\left(\bos{Q}_{b^e(t),\frac{1}{\lambda^e(t)}}\right)_{0\leq t \leq T}+\mathcal{O}_1\times \mathcal{O}_2,
$$
meaning that $U\in \mathcal{O}$ if and only if $\bos{\varepsilon} \in \mathcal{O}_1$ and $(\lambda',b')\in \mathcal{O}_2$. To measure the size of the objects, as \fref{intro:estimation parametres} holds, $b_1$ will be the quantity of reference. Our analysis has three main steps.\\
\\
\underline{Modulation:} We compute the time evolution of the parameters $\lambda$ and $b$. We show an inequality of the type:
$$
\left|(b_i^e+b'_i)_t+\frac{1}{\lambda}((i-\alpha)(b^e_1+b'_1)(b^e_i+b'_i)-(b^e_{i+1}+b'_{i+1})) \right|\leq  \frac{1}{\lambda} (\parallel \bos{\varepsilon} \parallel_{\text{loc}}+b_1^{L+3}).
$$
$\parallel \bos{\varepsilon} \parallel_{\text{loc}}$ comes from a local interaction term. It means that as long as $\bos{\varepsilon}$ stays in $\mathcal{O}_1$, it does not influence too much the evolution of the parameters. That is to say, as long as $\bos{u}(t)\in \mathcal{O}$, the dynamics of $\lambda^e+\lambda'$ and $b_e+b'$ are given at first orders by \fref{intro:sd des parametres}.\\
\\
\underline{Energy method:} We want to estimate the size of the error term $\bos{\varepsilon}$. Its time evolution is given by:
$$
\lambda \partial_t \bos{\varepsilon} = -\bos{H}_{\frac{1}{\lambda}}\bos{\varepsilon} +\bos{NL} +\bos{\psi}+\tilde{\bos{\psi}} ,
$$
where $\tilde{\bos{\psi}}$ is a corrective term as $\bos{\varepsilon}$ is orthogonal to $\mathcal{M}_{ap}$. Under the smallness assumption $(b',\lambda')\in \mathcal{O}_2$, $\bos{\psi}$ can be estimated, and under the smallness assumption $\bos{\varepsilon} \in \mathcal{O}_1$ so can be $\tilde{\bos{\psi}}$. To measure the size of $\bos{\varepsilon}$ we introduce two norms. The first one at high regularity:
$$
\mathcal{E}_{s_L}=\int \varepsilon^{(1)} \mathcal{L}^{s_L}\varepsilon^{(1)}+\int \varepsilon^{(2)} \mathcal{L}^{s_L-1}\varepsilon^{(2)}.
$$
This quantity is coercive, and in particular it controls the usual Sobolev norm (see Corollary \ref{annexe:cor:coercivite mathcalEsL}):
$$
\mathcal{E}_{s_L}\gtrsim \parallel \bos{\varepsilon} \parallel_{\dot{H}^{s_L}\times \dot{H}^{s_L-1}}^2.
$$
The second norm we use is at a low regularity level:
$$
\mathcal{E}_{\sigma}= \int |\nabla^{\sigma}\varepsilon^{(1)}|^2 + \int |\nabla^{\sigma-1}\varepsilon^{(2)}|^2
$$
for $\sigma>s_c$ slightly supercritical. The first one is the most essential for the analysis, because it is with this adapted norm that one can see that the error stays smaller than the perturbation involved in the approximate profile. We exhibit a Lyapunov type monotonicity formula for this term:
$$
\frac{d}{dt}\left\{\frac{\mathcal{E}_{s_L}}{\lambda^{2(s_L-s_c)}} \right\}\lesssim \frac{b_1^{2L+1+\delta}}{\lambda^{2(s_L-s_c)+1}}
$$
for $\delta=\delta(d,p,L)>0$. This can be integrated to obtain:
$$
\mathcal{E}_{s_L}\lesssim b_1^{2L+\delta}.
$$
When deriving this estimate, we need to control derivatives at a lower level to deal with the non-linear term. This is why we also aim at controling $\mathcal{E}_{\sigma}$. For this norm we exhibit a similar estimate:
$$
\frac{d}{dt}\left\{\frac{\mathcal{E}_{\sigma}}{\lambda^{2(\sigma-s_c)}} \right\}\lesssim \frac{b_1^{1+\delta'}}{\lambda^{2(\sigma-s_c)+1}}.
$$
When integrated in time it gives:
$$
\mathcal{E}_{\sigma} \lesssim b_1 ^{2(\sigma-s_c)(1+\delta'')} .
$$
When establishing the monotonicity formula for $\mathcal{E}_{s_L}$, we also need to control a local term that cannot be estimated directly with $\mathcal{E}_{s_L}$ and $\mathcal{E}_{\sigma}$. This is done through the use of a third tool: a Morawetz type quantity whose time evolution controls this local term.\\
\\
All these estimates show the following fact: if $\bos{u}(t')\in \mathcal{O}$ for $0\leq t'\leq t$, $\bos{\varepsilon}$ enjoys in fact better estimates giving in particular $\bos{\varepsilon}(t) \in \mathring{\mathcal{O}_1}$.\\

\noindent \underline{Conclusion a la Brouwer:} We recapitulate what we have shown so far in the analysis: as long as $\bos{u}(t)\in \mathcal{O}$, the parameters evolve according to \fref{intro:sd des parametres} plus a small perturbation, and the error enjoys a better estimate $\bos{\varepsilon} \in\mathring{\mathcal{O}_1}$. So a solution escapes from $\mathcal{O}$ if and only if $(b',\lambda')$ escape from $(\lambda^e,b^e)+\mathcal{O}_2$. We look at the dynamics given by \fref{intro:sd des parametres} in the set $(\lambda^e,b^e)+\mathcal{O}_2$. It admits $(\lambda^e,b^e) $ as an hyperbolic equilibrium. From standard argument a la Brouwer, even perturbed this equilibrium should persist in some sense: there must exist at least one orbit staying forever in $(\lambda^e,b^e)+\mathcal{O}_2$. This ends the proof of the existence of a true blow-up profile.\\

\emph{(v) The manifold construction:} Once we have the existence of our special blow up solutions, we investigate the topological properties of the set of their initial data. If we assume that two solutions $\bos{Q}_{b,\frac{1}{\lambda}}+\bos{\varepsilon}$ and $\bos{Q}_{b',\frac{1}{\lambda'}}+\bos{\varepsilon}'$ blow up according to the scenario we previously set up, we have enough informations and estimates to study their difference. We analyse the evolution equations for the differences of parameters $b_i-b'_i$ and errors $\bos{\varepsilon}-\bos{\varepsilon}'$. We find that $\ell-1$ differences of parameters evolve according to an unstable linear dynamics, and that the dynamics of the $L-\ell+1$ others and the difference of errors is stable. The differences of the stable parameters and errors only have a small feedback on the time evolution of the unstable parameters. Thus, if the initial difference of the unstable parameters is too big compared to the initial differences of the stable parameters and errors, the unstable linear dynamics wins and expells the differences of unstable parameters away from $0$. Hence one of the two solutions cannot blow up according to our scenario, yielding a contradiction. This gives that the unstable parameters have Lipschitz dependance on the stable parameters and on the error and proves Theorem \ref{thmmain2}.\\

The paper is organized as follows. In section 2 we present the main tools to understand the linear operator $\bos{H}$. After that we are able to construct or primary approximate profile in Proposition \ref{pr:constructionprofilnoncoupe}. We then localize this profile in the zone $y\leq B_1$ and estimate the remainder of the approximate dynamics in Proposition \ref{pr:profilapprochecoupe}. We end this section by studying the special solutions of the approximate dynamics: the existence of special solutions for \fref{intro:sd des parametres} is done in Lemma \ref{lem:soldusystemd}, their linear stability is studied in Lemma \ref{lem:linearisationsystemdyn}. In section 3 we implement our bootstrap method and state our main result of existence in Proposition \ref{prop:bootstrap}. First we explain how to "project" the full (NLW) on the manifold of approximate solutions in Lemma \ref{lem:conditiondortho}. Then we estimate the impact of $\bos{\varepsilon}$ on the dynamics of the parameters $b$ and $\lambda$ by computing the modulations equations in Lemmas \ref{lem:modulation} and \ref{trappedregime:improvedmodulation:lem:improvedmodulation}. In the second part we estimate the error term $\bos{\varepsilon}$. We start by deriving the monotonicity formula for the low Sobolev norm in Proposition \ref{trappedregime:pr:low sobo}, then we do it for the high regularity norm in Proposition \ref{trappedregime:pr:high sobo}, which is the main result of the section. We end the section with deriving a Morawetz identity to control a local term that appeared earlier in the computations in Proposition \ref{thetrapped:pr:morawetz}. In section 4 we end the proof of Proposition \ref{prop:bootstrap}. We show that in fact better bounds hold for the error term $\bos{\varepsilon}$ in Lemma \ref{endoftheproof:lem:integration des equations pour les normes}. We then examine the dynamics for the parameters in Lemmas \ref{endoftheproof:lem: controle des modes stables} and \ref{endoftheproof:lem:proprietes de f}, we show the existence of a true blow-up solution by topological arguments. For the completeness of the result we study the behavior of Sobolev norms in subsection \ref{end:subsection:normes}. In Section 5 we investigate the topological properties of the set of initial data leading to such a blow up scenario. In Proposition \ref{variete:prop:parametres lipschitz} we show that the for such solutions starting at the same scale with some additional regularity, we have Lipschitz dependence in adapted variables. We remove the extra assumptions in Proposition \ref{variete:prop:parametres lipschitz 2}, which allows us to prove that the set of initial data staying in our blow up scenario is a Lipschitz manifold whose codimension is explicit.


\section{The linearized dynamics and the construction of the approximate blow-up profile}

To understand the dynamics close to the 1-parameter family of ground states $\left(\bos{Q}_{\frac{1}{\lambda}}\right)_{\lambda>0}$ we study first its linearization. In this section we start by the presentation of appropriate notions, and technical lemmas about the linearized operator $\bos{H}$. Once we have these tools, we are able to create an approximate blow up profile in the second part of this section.

\subsection{The stationnary state and its numerology}

From standard argument, all smooth radially symmetric solutions to:
$$
-\Delta \phi - \phi^p=0,
$$
are dilates of a given normalized ground state profile:
$$
\phi=Q_{\lambda}, \ \lambda>0, \ \left\{ \begin{array}{l l}
-\Delta Q-Q^p=0 \\
Q(0)=1
\end{array}
\right. .
$$
We will now recall the asymptotic behavior of $Q$. Most of them are known properties, see \cite{YiLi}, \cite{KaS}.
 
\begin{lemma}[Asymptotic expansion of the ground state] Let $p>p_{JL}$ (defined in \fref{exponentpjl}). We recall that $g>0$, $c_{\infty}$ and $\gamma $ are defined in \fref{defgamma} and \fref{eq:def:gain1}. One has:
\label{lem:Q}
\begin{itemize}
\item[(i)]\emph{Asymptotics at infinity:}
\begin{equation} \label{linearized:eq:asymptotique soliton}
\forall k\geq 0, \ \partial^k_y Q=\partial^k_y \left[ \frac{c_{\infty}}{y^{\frac{2}{p-1}}}+\frac{a_1}{y^{\gamma}}\right]+O\left( \frac{1}{y^{\gamma+g+k}}\right), \ \text{as} \ y \rightarrow +\infty,
\end{equation}
for a non null constant $a_1\neq 0$.
\item[(ii)]\emph{Degeneracy:}
\begin{equation}\label{eq:degenerescencesoliton}
\forall k\geq 0, \ \partial^k_y \Lambda^{(1)} Q =\partial^k_y \left[ \frac{c}{y^{\gamma}} \right]+O\left( \frac{1}{y^{\gamma+g+k}}\right), \ \ \text{as} \ y \rightarrow +\infty,
\end{equation}
for a non null constant $c\neq 0$.
\item[(iii)]\emph{Positivity of $\mathcal{L}$:}
\begin{equation}\label{positivitelinearise}
\mathcal{L}>\frac{\delta(p)}{y^2}>0 \ on \ H^1(\mathbb{R}^d),
\end{equation} 
\item[(iv)]\emph{Positivity of $\Lambda^{(1)}Q$:}
\begin{equation}\label{cons:eq:positivite lambdaQ}
\Lambda^{(1)}Q>0.
\end{equation} 
\end{itemize}
\end{lemma}

\begin{proof}[Proof of lemma \ref{lem:Q}]
Only the fact that $a_1\neq 0$ is not proven in the references we quoted. To prove it, we have to enter in details in their proof of the asymptotic expansion. This is done in Lemma \ref{lem:expansion_soliton} of Appendix A.
\end{proof}

We now state important properties of the numbers attached to the asymptotic expansion of the ground state. A proof can be found in \cite{MRRod2}, Lemma A1.

\begin{lemma}[supercritical numerology]\label{annexe:lem:valeur dalpha}
Let $d\geq 11$, $p_{JL}$ and, $\alpha$ be given by \fref{exponentpjl} and \fref{intro:eq:def alpha}. Then:
\begin{itemize}
\item[(i)]the condition $p>p_{JL}$ is equivalent to:
$$
2+\sqrt{d-1}<s_c<\frac{d}{2} .
$$
\item[(ii)]$\alpha$ is real if and only if $p>p_{JL}$. In that case there holds the bounds:
$$
2<\alpha<\frac{d}{2}-1 .
$$
\end{itemize}
\end{lemma}


\subsection{factorization of $\mathcal{L}$}

The positivity of $\Lambda^{(1)}Q$ \fref{cons:eq:positivite lambdaQ} implies from a direct calculation the factorization of this operator.
\begin{lemma}[Factorization of $\mathcal{L}$]\label{lem:factorisation} Let:
\begin{equation}
W:=\partial_y(\text{log}(\Lambda^{(1)}Q)),
\end{equation}
and define the first order operators on radial functions:
\begin{equation}
A:u\mapsto -\partial_y u+Wu, \ A^*:u\mapsto \frac{1}{y^{d-1}}\partial_y(y^{d-1}u)+Wu.
\end{equation}
Then we have:
\begin{equation}
\mathcal{L}=A^*A.
\end{equation}
\end{lemma}

\begin{remark}
The adjunction is taken with respect to the radially symmetric Lebesgue measure:
$$
\int_{y>0} (Au)vy^{d-1}dy=\int_{y>0} u(A^*v)y^{d-1}dy.
$$
\end{remark}

\begin{proof}[Proof of Lemma \ref{lem:factorisation}] 
This factorization relies on the fact that $\Lambda^{(1)}Q>0$, and then it is a standard property of Schr\"odinger operators with a non-vanishing zero. One can compute:
$$
A^*Au=-\Delta u+(\frac{d-1}{y}W+\partial_y W+W^2)u.
$$
Then the result follows from:
$$
\frac{d-1}{y}W+\partial_y W+W^2=\frac{\Delta \Lambda^{(1)}Q}{\Lambda^{(1)}Q}=\frac{-\mathcal{L}\Lambda^{(1)}Q-V\Lambda^{(1)}Q}{\Lambda^{(1)}Q}=-V,
$$
where we used the fact that $\mathcal{L}\Lambda^{(1)}Q=0$.
\end{proof}

We collect here the informations about the asymptotic behavior of the potentials $V$ and $W$ which will be used many times in the sequel. These results are a direct implication of the previous Lemma \ref{lem:Q}.

\begin{lemma}(Asymptotic behavior of the potentials:)\label{lem:asymptoticdupotentiel}
There holds:
\begin{itemize}
\item[(i)]\emph{Asymptotics:}
\begin{equation}
\partial_y^k V= \left\{ \begin{array}{l l} O(1) \ as \ y\rightarrow 0 \\ \frac{c_k}{y^{2+k}}+O\left( \frac{1}{y^{2+\alpha+k}}\right) \ \text{as} \ y\rightarrow +\infty  \end{array}  \right. ,
\end{equation}
\begin{equation} \label{linearized:eq:asymptotique W}
\partial_y^k W= \left\{ \begin{array}{l l} O(1) \ as \ y\rightarrow 0 \\ \frac{c'_k}{y^{1+k}}+O\left( \frac{1}{y^{1+g+k}}\right) \ \text{as} \ y\rightarrow +\infty  \end{array}  \right. ,
\end{equation}
with $c_k\neq 0$, $c'_k\neq 0$ and $c'_1=-\gamma$.
\item[(ii)]\emph{Degeneracy:}
\begin{equation} \label{linearized:eq:degenerescence scaling}
\partial_y^k \left( \frac{d}{d\lambda} [(Q_{\lambda})^{p-1}]_{|\lambda=1} \right)= O\left(\frac{1}{y^{2+\alpha+k}}\right) \ \text{as} \ y \rightarrow +\infty .
\end{equation}
\end{itemize}

\end{lemma}


\subsection{Inverting $\bos{H}$ on radially symmetric functions}

We first start by inverting $\mathcal{L}$. We are only considering radially symmetric functions, so $\Delta=\partial_{yy}+(d-1)\frac{\partial_y}{y}$, and we can apply basic results from ODE theory. We will do this thanks to the explicit knowledge of the kernel of $\mathcal{L}$. Indeed from the rewriting:
\begin{equation}\label{eq:reecrituredeA}
A:u\mapsto -\Lambda^{(1)}Q\partial_y\left( \frac{u}{\Lambda^{(1)} Q}\right), \ A^*:u\mapsto \frac{1}{y^{d-1}\Lambda^{(1)}Q}\partial_y (y^{d-1}\Lambda^{(1)}Q u),
\end{equation}
we note that:
\begin{equation}
Au=0 \ \text{iff} \ u\in \text{Span}(\Lambda^{(1)}Q), \ A^*u=0 \ \text{iff} \ u\in \text{Span}\left( \frac{1}{y^{d-1}\Lambda Q}\right).
\end{equation}
It implies that for radially symetric functions:
\begin{equation}
\mathcal{L}u=0 \ \text{iff} \ u\in \text{Span}(\Lambda^{(1)}Q,\Gamma),
\end{equation}
with:
\begin{equation}\label{eq:definitionGamma}
\Gamma (y):=\Lambda^{(1)}Q(y) \int_1^y \frac{dx}{x^{d-1}(\Lambda^{(1)} Q(x))^2}.
\end{equation}
We already knew $\Lambda^{(1)}Q$ was in the kernel of $\mathcal{L}$ since it is the tangent vector to the branch of stationnary solutions $(Q_{\lambda})_{\lambda>0}$. We just found the second vector in the kernel: $\Gamma$. From the asymptotic behavior \fref{eq:degenerescencesoliton} of $\Lambda^{(1)}Q$, we deduce the following asymptotic for $\Gamma$:
\begin{equation}
\Gamma \underset{y\rightarrow 0}{\sim} \frac{-c}{y^{d-2}} \ \ \text{and} \ \ \Gamma \underset{y\rightarrow +\infty}{\sim} \frac{c'}{y^{\gamma}},
\end{equation}
$c$ and $c'$ being two positive constants. Both results are obtained from \fref{eq:definitionGamma}, with the fact that $\Lambda^{(1)}Q>0$ and the asymptotic \fref{eq:degenerescencesoliton} that implies:
$$
0<\int_1^{+\infty} \frac{dx}{x^{d-1}(\Lambda^{(1)}Q)^2} \leq C \int_1^{+\infty} \frac{dx}{x^{d-1-2\gamma}}<+\infty ,
$$
where we used the relation from \fref{defgamma}: $d-1-2\gamma$>1.\\
\\
Now that we know the Green's functions of $\mathcal{L}$ we can introduce the formal inverse:
\begin{equation}\label{eq:definitioninverseformel}
\mathcal{L}^{-1}f:=-\Gamma(y)\int_0^y f\Lambda^{(1)} Q x^{d-1} dx +\Lambda^{(1)}Q(y)\int_0^y f\Gamma x^{d-1} dx .
\end{equation}
One can check that for $f$ smooth and radial we have indeed $\mathcal{L}(\mathcal{L}^{-1}f)=f$. As we do not have uniqueness for the equation $\mathcal{L}u=f$, one may wonder if this definition is the "right" one. The answer is yes because this inverse has the good asymptotic behavior at the origin and $+\infty$, see Lemma \ref{lem:actiondeHetH-1suradmissile}. To compute easily the asymptotic, we will use the following computational lemma.

\begin{lemma}\label{lem:calculdelinversedeL}(Inversion of $\mathcal{L}$:) Let $f$ be a $C^{\infty}$ radially symmetric function, and denote by $u$ its inverse by $\mathcal{L}$: $u=\mathcal{L}^{-1}f$ given by \fref{eq:definitioninverseformel}, then:
\begin{equation}
Au=\frac{1}{y^{d-1}\Lambda^{(1)}Q}\int_0^y f\Lambda^{(1)} Q x^{d-1} dx, \ u =-\Lambda^{(1)}Q \int_0^y \frac{Au}{\Lambda^{(1)}Q}dx .
\end{equation} 
\end{lemma}
This lemma says that to compute $u=\mathcal{L}^{-1}f $, we can do it in a rather easy way in two times: first we compute $Au$, then we compute $u$ knowing $Au$.

\begin{proof}[Proof of Lemma \ref{lem:calculdelinversedeL}]
We compute from the definition of $\Gamma$  \fref{eq:definitionGamma}:
$$
A\Gamma=-\partial_y \Gamma +\frac{\partial_y(\Lambda^{(1)}Q)}{\Lambda Q}\Gamma=-\frac{1}{y^{d-1}\Lambda^{(1)}Q} .
$$
We therefore apply $A$ to the definition of $u$ given by \fref{eq:definitioninverseformel}, and using the cancellation $A(\Lambda Q)=0$, we find:
$$
Au=\frac{1}{y^{d-1}\Lambda^{(1)}Q}\int_0^y f\Lambda^{(1)}Q x^{d-1}dx .
$$
which, together with the definition of $A$ \fref{eq:reecrituredeA} gives:
$$
u= -\Lambda^{(1)}Q\int_0^y \frac{Au}{\Lambda^{(1)}Q}dx+c_u\Lambda Q,
$$
$c_u$ being an integration constant. But from \fref{eq:definitioninverseformel} we see that: $u=O(y^2)$ and $Au = O(y)$ as $y\rightarrow 0$. From that we deduce the nullity of the constant: $c_u=0$, which establishes the formula.
\end{proof}

Knowing how to invert $\mathcal{L}$, we define the inverse of $\bos{H}$ by the following formula:
\begin{equation}
\bos{H}^{-1}:=\begin{pmatrix}
0& \mathcal{L}^{-1} \\
-1 & 0
\end{pmatrix} .
\end{equation}


\subsection{Adapted derivatives, admissible and homogeneous functions}

The usual derivatives, that is to say the $\nabla^k$ ones, are not fit for the study of (NLW) close to the family of ground states $(\bos{Q}_{\lambda})_{\lambda>0}$, because they do not commute with the linearized operator $\mathcal{L}$. In this subsection we describe the adapted derivatives we will use. The asymptotic behavior of the adapted derivatives of the profiles, at the origin and at infinity, is going to play an important role. The second significant property is the vectorial position (when a function $\bos{f}$ has only one of its coordinate being non null). For the profiles we will use later, these informations are contained in the notion of admissible function. Given a radial function $f(x)=f(|x|)$, we define the sequence:
$$
f_k=\mathcal{A}^kf 
$$
of adapted derivatives of $f$ by induction:
\begin{equation} \label{linearized:eq:def derivees adaptees}
f_0:=f \ \text{and} \ f_{k+1}:=\left\{ \begin{array}{l l}
A f_k \ \text{for} \ k \ \text{even} ,\\
A^* f_k \ \text{for} \ k \ \text{odd} .
\end{array}   \right.
\end{equation}

\begin{definition}(Admissible functions:) \label{linearized:def:fonctions admissibles}
Let $p_1$ be a positive integer, $p_2$ be a real number, and $\iota$ an indice $\iota \in \{0;1\}$.\\
We say that a vector of functions $\bos{f}=\begin{pmatrix}f^{(1)} \\ f^{(2)} \end{pmatrix}$ of two $C^{\infty}$ radially symmetric functions is admissible of degree $(p_1,p_2,\iota)$ if:
\begin{itemize}
\item[(i)]\emph{$\iota$ is the position:}
\begin{equation}
\bos{f}=\begin{pmatrix} f^{(1)} \\ 0 \end{pmatrix} \ (\text{ie} \ f^{(2)}=0) \ \text{if} \ \iota=0, \ \ \text{and} \ \bos{f}=\begin{pmatrix} 0 \\ f^{(2)} \end{pmatrix} \ (\text{ie} \ f^{(1)}=0) \ \text{if} \ \iota=1 .
\end{equation}
We will then write indifferently $f$ to denote $f^{(1)}$ or $f^{(2)}$ in the two cases.
\item[(ii)]\emph{$p_1$ describes the behavior near $0$:}
\begin{equation}
\forall 2p\geq p_1, \ f(y)=\sum_{k=p_1-\iota, \ k \ \text{even}}^{2p} c_ky^k+ O(y^{2p+2}), \ \text{as} \ y\rightarrow 0 .
\end{equation}
\item[(iii)]\emph{$p_2$ describes the behavior at infinity:}
\begin{equation} \label{linearized:eq:asymptotique infini fonction admissible}
\forall k\in \mathbb{N}, \ \ |f_k(y)|=O (y^{p_2-\gamma-\iota-k}) \ \text{a}s \ y\rightarrow +\infty .
\end{equation}
\end{itemize}
\end{definition}

The actions of $\bos{H}$ and $\bos{H}^{-1}$ on admissible functions enjoy the following properties:

\begin{lemma}\label{lem:actiondeHetH-1suradmissile}(Action of $\bos{H}$ and $\bos{H}^{-1}$ on admissible functions:) Let $\bos{f}$ be an admissible function of degree $(p_1,p_2,\iota)$, with $p_2\geq -1$ then:
\begin{itemize}
\item[(i)] $\forall i\geq 0$, $\bos{H}^i \bos{f}$ is admissible of degree $(\text{max}(p_1-i,\iota),p_2-i,\iota+i\text{mod}2)$.
\item[(ii)] $\forall i\geq 0$, $\bos{H}^{-i}\bos{f}$ is admissible of degree $(p_1+i,p_2+i,\iota+i\text{mod}2)$.
\end{itemize}
\end{lemma}

\begin{proof}[Proof of Lemma \ref{lem:actiondeHetH-1suradmissile}]
\underline{Action of $\bos{H}$:}
We compute:
\be \label{thetrapped:eq:puissances de H}
\bos{H}^{2k}=(-1)^k\begin{pmatrix}
\mathcal{L}^k & 0 \\
0 & \mathcal{L}^k
\end{pmatrix}
, \ and \ \bos{H}^{2k+1}=(-1)^k\begin{pmatrix}
0 & -\mathcal{L}^k \\
\mathcal{L}^{k+1} & 0
\end{pmatrix} .
\ee
So that the property we claim holds by a direct check at the definitions of adapted derivatives and admissible functions.\\

\underline{Action of $\bos{H}^{-1}$:} We are going to prove the property by induction on $i$. We will prove it for $\iota=0$, the proof being the same for $\iota=1$. We can suppose without loss of generality that $p_1$ is even. The property is true, of course, for $i=0$. Suppose now it is true for $i$. If $i$ is even, then:
$$
\bos{H}^{-(i+1)}\bos{f}=\bos{H}^{-1}\bos{H}^{-i}\bos{f}=\begin{pmatrix}
0 & \mathcal{L}^{-1} \\
-1 & 0
\end{pmatrix}
\begin{pmatrix}
(H^{-i}f)^{(1)} \\
0
\end{pmatrix}
=\begin{pmatrix}
0 \\
-(H^{-i}f)^{(1)}
\end{pmatrix} .
$$
The induction hypothesis for $\bos{H}^{-i}\bos{f}$ implies that the function $\bos{H}^{-(i+1)}\bos{f}$ is of degree $(p_1+i+1,p_2+i+1,1)$. Suppose now $i$ is odd. Then we have:
$$
\bos{H}^{-(i+1)}\bos{f}=\begin{pmatrix} 0 & \mathcal{L}^{-1} \\ -1 & 0 \end{pmatrix} \begin{pmatrix} 0 \\ (H^{-i}f)^{(2)} \end{pmatrix}=\begin{pmatrix} \mathcal{L}^{-1}(H^{-i}f)^{(2)} \\ 0 \end{pmatrix} .
$$
We write $u=\mathcal{L}^{-1}(H^{-i}f)^{(2)})$. We have from the induction hypothesis:
$$
(H^{-i}f)^{(2)}= \sum_{k=p_1+i-1, \ k \ \text{even}}^{2p}c_k y^{k} +O(y^{2p+2}), \ as \ y\rightarrow 0 .
$$
From \fref{eq:definitioninverseformel} one can see the gain:
$$
u= \sum_{k=p_1+i+1, \ k \ \text{even}}^{2p} c'_k y^{k} +O(y^{2p+2}), \ as \ y \rightarrow 0 ,
$$
and since $\iota(\bos{H}^{-(i+1)}\bos{f})=0$, we get $p_1(\bos{H}^{-(i+1)}\bos{f})=p_1+1$.\\
\\
From the induction hypothesis for $\bos{H}^{-i}\bos{f}$, and the relation $u_k=(H^{-i}f)^{(2)}_{k-2}$ for $k\geq 2$, the asymptotic \fref{linearized:eq:asymptotique infini fonction admissible} at $+\infty$ for $u$ is true for $k\geq 2$. One only needs to check the asymptotic at $+\infty$ for $k=0$ and $k=1$. We use the computational Lemma \ref{lem:calculdelinversedeL}:
$$
\begin{array}{r c l}
Au=\frac{1}{y^{d-1}\Lambda^{(1)} Q}\int_0^y (H^{-i}f)^{(2)}\Lambda^{(1)}Q x^{d-1}dx &=& O\left( \frac{1}{y^{d-1-\gamma}}\int_0^y x^{p_2+i-1-2\gamma+d-1}dx \right) \\
&=& O(y^{p_2+i-\gamma}) ,
\end{array}
$$
where we used the asymptotic \fref{eq:degenerescencesoliton} of $\Lambda^{(1)}Q$. Indeed the integral in the right hand side is divergent from:
$$
p_2+i-1-2\gamma+d=p_2+i+\sqrt{\triangle}+1>0 .
$$
We then do the same for $u$:
$$
u=-\Lambda^{(1)}Q\int_0^y \frac{Au}{\Lambda^{(1)} Q}dx= O\left( y^{-\gamma}\int_0^y x^{p_2+i-\gamma+\gamma} \right)=O(y^{p_2+i+1-\gamma}) ,
$$
and from $\iota(\bos{H}^{-1}f)=0$ we deduce $p_2(\bos{H}^{-1}f)=p_2+i+1$.
\end{proof}

This notion of admissible function will be helpful to construct the approximate blow-up profile. The building blocks of this profile are the generators of the kernel of the iterates of $\bos{H}$.

\begin{lemma}\label{lem:profilsTi}(\emph{Generators of the kernel of $\bos{H}^i$}:) We recall that the numbers $\alpha$ and $g'$ are defined in \fref{intro:eq:def alpha}, \fref{eq:def:gain2}. Let $(\bos{T}_i)_{i\in \mathbb{N}}$ denote the sequence of profiles given by:
\begin{equation}
\bos{T}_0:=\bos{\Lambda} \bos{Q}, \ \ \bos{T}_{i+1}:=-\bos{H}^{-1}\bos{T}_{i}, \ i \in \mathbb{N} .
\end{equation}
Let $(\bos{\Theta}_i)_{i\in \mathbb{N}}$ be the associated sequence defined by:
\begin{equation}
\bos{\Theta}_i:=\bos{\Lambda} \bos{T}_i-(i-\alpha)\bos{T}_i, \ i\in \mathbb{N} .
\end{equation}
Then:
\begin{itemize}
\item[(i)] $\bos{T}_i$ is admissible of degree $(i,i,i\text{mod} \ 2)$.
\item[(ii)] $\bos{\Theta}_i$ is admissible of degree $(i,i-g',i,i \text{mod} \ 2)$.
\end{itemize}
\end{lemma}

This lemma states that the $T_i$'s and $\Theta_i$'s have only one coordinate being non null, depending on the parity of $i$. We will then make the following abuse of notation (with respect to \fref{intro:eq:def vecteur}):
\be \label{linearized:eq:localisation Ti}
\bos{T}_{2i}= \begin{pmatrix} T_{2i} \\ 0  \end{pmatrix} , \ \bos{T}_{2i+1}= \begin{pmatrix} 0 \\  T_{2i+1}  \end{pmatrix}, \ \bos{\Theta}_{2i}= \begin{pmatrix} \Theta_{2i} \\ 0  \end{pmatrix} \ \text{and} \ \bos{\Theta}_{2i+1}= \begin{pmatrix} 0 \\  \Theta_{2i+1}  \end{pmatrix}
\ee

\begin{proof}[Proof of Lemma \ref{lem:profilsTi}] From the degenerescence \fref{eq:degenerescencesoliton} and the fact that $A\Lambda^{(1)} Q=0$, $\bos{\Lambda} \bos{Q}$ is admissible of degree $(0,0,0)$. Hence due to the properties of the action of $\bos{H}^{-1}$ on admissible functions, the previous Lemma \ref{lem:actiondeHetH-1suradmissile}, we get that $\bos{T}_i$ is admissible of degree $(i,i,i\text{mod}2)$.\\
\\
To prove the second part about the $\bos{\Theta}_i$'s we will procede by induction. The asymptotic behavior of the solitary wave \fref{eq:degenerescencesoliton} ensures that the property is true for $\bos{\Theta}_0=\bos{\Lambda}(\bos{\Lambda} \bos{Q})+\alpha\bos{\Lambda}\bos{Q}$. For $i$ odd we have:
$$
\begin{array}{r c l}
\bos{\Theta}_i&=&\begin{pmatrix}
0 \\ \Lambda^{(2)}T_i^{(2)}-(i-\alpha)T_i^{(2)}
\end{pmatrix} =\begin{pmatrix}
0 \\ -\left((\Lambda^{(1)}+1)T_{i-1}^{(1)}-(i-1+1-\alpha)T_{i-1}^{(1)}\right)
\end{pmatrix} \\
&=& \begin{pmatrix}
0 \\ -\Theta_{i-1}^{(1)}
\end{pmatrix} .
\end{array} 
$$
So if the property is true for $i$ even, it is true for $i+1$ from a direct check at the definition of the degree. Let us now assume that $i$ is even, $i\geq 2$. We compute the following relation:
\begin{equation}
\mathcal{L}(\Lambda^{(1)}u)=2\mathcal{L}u+\Lambda^{(1)}\mathcal{L} u+(2V+y.\nabla V)u .
\end{equation}
The asymptotic behavior of the potential (Lemma \ref{lem:asymptoticdupotentiel}), implies the improved decay:
\begin{equation}
2V+y.\nabla V=O\left( \frac{1}{y^{2+\alpha}} \right) .
\end{equation}
We then compute:
\begin{equation}
\mathcal{L}(\Theta_{i}^{(1)})=-\Theta_{i-2}^{(1)}+(2V+y.\nabla V)T_{i}^{(1)} .
\end{equation}
The induction hypothesis, together with the decay property of the potential and the degree of $\bos{T}_i$ give that $\bos{H}\bos{\Theta}_i$ is of degree $(i-1,i-1-g',1)$. As $0<g'\leq 2$ we have that $p_2(\bos{H}\bos{\Theta}_i)=i-1-g'\geq -1$ and we can apply the inversion Lemma \ref{lem:actiondeHetH-1suradmissile} about admissible functions: $\bos{H}^{-1}(\bos{H}\bos{\Theta}_i)$ is of degree $(i,i-g',0)$. One has $\mathcal{L}^{-1}\mathcal{L}(\Theta_i)=\Theta_i+a\Lambda^{(1)}Q+b\Gamma $, with $a$ and $b$ two integration constants. From the asymptotics $\Theta_i(y)\underset{y\rightarrow 0}{\rightarrow} 0$, $\mathcal{L}^{-1}\mathcal{L}(\Theta_i) \underset{y\rightarrow 0}{\rightarrow} 0$, $\Lambda^{(1)}Q(y)\underset{y\rightarrow 0}{\rightarrow} c>0$ and $\Gamma (y)\underset{y\rightarrow 0}{\rightarrow} +\infty$ one deduces $a=b=0$. This means that $\Theta_i=\mathcal{L}^{-1}\mathcal{L}(\Theta_i)$ is of degree $(i,i-g',0)$.
\end{proof}

In the following, we will have to deal with polynomial functions of the coefficients $b_i$. Knowing in advance that $b_i\approx b_1^i$ for the approximate blow-up profile\footnote{see Lemma \ref{lem:soldusystemd}.}, we have that $\prod b_i^{J_i}\approx b_1^{\sum iJ_i}$. Given a $L$-tuple $J$ of integers, we define:
\begin{equation}
|J|_1=\sum_1^L J_i, \  \text{and} \ |J|_2=\sum_1^L iJ_i .
\end{equation}

\begin{definition}[Homogeneous functions]\label{def:homogeneousfunctions}
$b$ denotes a $L$-tuple $(b_i)_{1\leq i \leq L}$. $p_1$ is an integer, $p_2$ is a real number, $\iota$ is an indice $\iota \in \{0;1\}$ and $p_3$ is an integer. We say that a function $\bos{S}(b,y)$ is homogeneous of degree $(p_1,p_2,\iota,p_3)$ if it can be written as a finite sum:
$$
\bos{S}=\sum_{J\in \mathcal{J}, \ |J|_2=p_3} \left( \prod_{i=1}^L b_i^{J_i} \bos{S}_J(y) \right) , 
$$
$\# \mathcal{J}<+\infty$, where for each $J$, $\bos{S}_J$ is an admissible function of degree $(p_1,p_2,\iota)$.
\end{definition}

Because of the asymptotics of the potential $W$, see \fref{lem:asymptoticdupotentiel}, asking that $\mathcal{A}^k f$ behave like $y^{-\gamma+k+p_2}$ at infinity is equivalent to say that $\partial_y^k f$ behaves the same way. As a consequence, the asymptotics can be multiplied, derived etc... which is the object of the following computational lemma. It is a straightforward application of Lemma \ref{annexe:lem:equivalence degre} from the Appendix.

\begin{lemma}[Calculus on homogeneous functions:]\label{lem:calculsurfonctionsadmissibles}
Let $\bos{f}=\begin{pmatrix}f\\ 0 \end{pmatrix}$, $\bos{g}=\begin{pmatrix}g\\ 0 \end{pmatrix}$ be homogeneous of degree\footnote{we just state the result for $\iota=0$ as in (NLW) the nonlinearity only acts on the first coordinate.} $(p_1,p_2,0,p_3)$ and $(p'_1,p'_2,0,p_3')$ ($p_1$ and $p_1'$ even). Then:
\begin{itemize}
\item[(i)]\emph{Multiplication:} the product $\bos{fg}:=\begin{pmatrix} fg \\ 0 \end{pmatrix}$ is an homogeneous profile of degree $(p_1+p_1',p_2+p_2'-\gamma,0,p_3+p_3')$.
\item[(ii)]\emph{Multiplication by the potentials involved in the analysis:} $\bos{fQ}^k:=\begin{pmatrix} fQ^k \\ 0 \end{pmatrix}$ is an homogeneous profile of degree $(p_1,p_2-k\frac{2}{p-1},0,p_3)$
\end{itemize}
\end{lemma}


\subsection{Slowly modulated blow profiles and growing tails}

We now construct an approximate blow up profile using the tools we previously displayed. First, we construct an approximate blow-up profile generating a blow up locally around the origin, but far away nonetheless it is irrelevant because it has polynomial growth (Proposition \ref{pr:constructionprofilnoncoupe}). Secondarily we cut this profile in a relevant zone to avoid this problem (Proposition \ref{pr:profilapprochecoupe}). This cutting procedure creates additional error terms which will be estimated.\\
\\ 
To manipulate the topological properties of the dynamics we will make use of the following adapted norms for $k\in \mathbb{N}$:
\begin{equation} \label{linearized:eq:def normes adaptees}
\begin{array}{r c l}
\parallel \bos{u}\parallel_k^2 &=& \parallel u^{(1)}_{k_0+1+k}\parallel_{L^2}^2 +\parallel u^{(2)}_{k_0+k}\parallel_{L^2}^2 \\
&=& \int u^{(1)}\mathcal{L}^{k_0+1+k}u^{(1)} + \int u^{(2)}\mathcal{L}^{k_0+k}u^{(2)} ,
\end{array} 
\end{equation}
involving the $k$-th adapted derivative of $u$ defined in \fref{linearized:eq:def derivees adaptees}. We will also the local version of these norms:
\begin{equation} \label{linearized:eq:def normes adaptees locales}
\parallel \bos{u} \parallel_{k,(y\leq M)}^2=\parallel u^{(1)}_{k_0+1+k}\parallel_{L^2(|y|\leq M)}^2 +\parallel u^{(2)}_{k_0+k}\parallel_{L^2(|y|\leq M)}^2 .
\end{equation}
As the scale $\lambda$ of our solution is changing with time, we want to work with the appropriate space variable $y=\frac{r}{\lambda}$. The appropriate renormalized time is:
\be \label{linearized:eq:def s}
s(t)=s_0+\int_{t_0}^t \frac{1}{\lambda(\tau)}d\tau .
\ee
Let $\bos{u}$ be a solution of (NLW) on the time interval $[0,T[$, and $\lambda:[0,T[\rightarrow \mathbb{R}^*_+$ be a $C^1$ function. We define the associated renormalized solution by: 
$$
\bos{v}(y,s)=\bos{u}_{\lambda(t)}(y,t) .
$$
The time evolution of $\bos{v}$ is then given by:
\be \label{linearized:eq:def NLW renormalisee}
\partial_s \bos{v} =\bos{F}(\bos{v})+\frac{\lambda_s}{\lambda} \bos{\Lambda} \bos{v} .
\ee
It is often easier to work with this renormalized flow.\\
\\

In the next proposition we state the existence of a primary blow up profile. This construction is related to the so-called center manifolds. The idea is to construct a manifold, tangent to the vector space of the generalized kernel of the linearized operator at the point $\bos{Q}_{\lambda}$, displaying a special dynamics. At the linear level, this dynamics is driven by the linearized operator. At the quadratic level it is driven by the scaling. The non linear terms only affect the dynamics at higher order, thus being invisible as we work in a perturbative setting\footnote{this point will be made clearer when studying the full non-linear dynamics.}. The dynamics on this manifold is then easy to write down.

\begin{proposition}\label{pr:constructionprofilnoncoupe}\emph{(Construction of the approximate profile)} Let a very large odd integer \footnote{we take $L$ to be odd just to know the coordinates of the objects we are manipulating, but it is not important.}:
\begin{equation} \label{linearized:eq:def L}
L\gg 1
\end{equation}
and let $b=(b_1,...,b_L)$ denote a $L$-tuple of real numbers, with $b_1>0$. There exists a $L$-dimensional manifold of $C^{\infty}$ radially symmetric functions $(Q_{b})_{b\in \mathbb{R}_+^* \times \mathbb{R}^{L-1}}$ satisfying the following identity:
\begin{equation} \label{linearized:eq:dynamique Qb}
\bos{F}(\bos{Q}_b)= b_1\bos{\Lambda} \bos{Q}_b+\sum_{i=1}^{L} (-(i-\alpha)b_1b_i+b_{i+1}) \frac{\partial \bos{Q}_b}{b_i} -\bos{\psi}_b,
\end{equation}
where we used the convention $b_{L+1}=0$. $\bos{\psi}_b$ stands for a higher order remainder term situated on the second coordinate:
\be \label{linearized:eq:position psib}
\bos{\psi}_b= \begin{pmatrix} 0 \\ \psi_b \end{pmatrix} .
\ee
Let $B_1$ be defined by \fref{eq:def:B1etB0}. In the regime in which $|b_i|\lesssim|b_1|^i$, $0<b_1\ll 1$, it enjoys the following estimates (the adapted norm is defined by  \fref{linearized:eq:def normes adaptees locales}):

\begin{itemize}

\item[(i)]\emph{Global\footnote{here the zone $y\leq B_1$ is called global because we will cut the profile $Q_b$ in the next section at this precise location.} bounds:} For $0\leq j\leq L$:
\begin{equation}
\parallel \bos{\psi}_b \parallel_{j,(y\leq 2B_1)}^2 \leq C(L) b_1^{2j+2+2(1-\delta_0)+2g'-C\eta} 
\end{equation}

\item[(ii)]\emph{Local improved bounds:}
\begin{equation} \label{linearized:eq:estimations locales non coupe}
\forall j\geq 0, \ \forall B >1, \ \int_{y\leq B} |\nabla^{j}\psi_b^{(1)}|^2+|\nabla^{j}\psi_b^{(2)}|^2 \leq C(j,L)B^{C(j,L)} b_1^{2L+6} .
\end{equation}
\end{itemize}

\noindent The profile $\bos{Q}_b$ is of the form:
\begin{equation}\label{eq:def:Qb}
\bos{Q}_{b}:=\bos{Q}+\bos{\alpha}_{b}, \ \bos{\alpha}_{b}:=\sum_{i=1}^L b_i\bos{T}_i+\sum_{i=2}^{L+2} \bos{S}_i ,
\end{equation}
where $\bos{T}_i$ is given by Lemma \ref{lem:profilsTi}, and the $\bos{S}_i$'s is are homogeneous functions in the sense of definition \ref{def:homogeneousfunctions}:
$$
\left\{ \begin{array}{l l}
\bos{S}_i:=\bos{S}_i(b,y), \ 1\leq i \leq L+2 \\
\bos{S}_1=0
\end{array}
\right. ,
$$
with:
\begin{equation} \label{linearized:eq:degre Si}
\left\{ \begin{array}{l l}
deg(\bos{S}_i)=(i,i-g',i \ \text{mod}2,i) \\
\frac{\partial \bos{S}_i}{\partial b_j}=0 \ \text{for} \ 2\leq i\leq j \leq L
\end{array}
\right. .
\end{equation}
\end{proposition}

\begin{remark}
Because of the form \fref{eq:def:Qb} of the profile $\bos{Q}_b$, including its time evolution in \fref{linearized:eq:dynamique Qb} yields:
\begin{equation}
\partial_s \bos{Q}_b-\bos{F}(\bos{Q}_b)+b_1\bos{\Lambda} \bos{Q}_b= \bos{Mod}(t) +\bos{\psi}_b ,
\end{equation}
where:
\begin{equation} \label{linearized:eq:def Mod}
\bos{Mod}(t)=\sum_{i=1}^L [b_{i,s}+(i-\alpha)b_1b_i-b_{i+1}]\left[ \bos{T}_i+\sum_{j=i+1}^{L+2}\frac{\partial \bos{S}_j}{\partial b_i} \right] .
\end{equation}
From the homogeneity property of the $\bos{S}_i$'s \fref{linearized:eq:degre Si}, we have the following position depending on the parity of $i$, and make the abuse of notation (regarding \fref{intro:eq:def vecteur}):
\be \label{linearized:eq:position Si}
\bos{S}_{2i}=\begin{pmatrix} S_{2i} \\ 0 \end{pmatrix}, \ \bos{S}_{2i+1}=\begin{pmatrix} 0 \\ S_{2i+1} \end{pmatrix}.
\ee
\end{remark}

\begin{proof}[Proof of Proposition \ref{pr:constructionprofilnoncoupe}]
\underline{Step 1:} Computation of the error. We take a profile having the form \fref{eq:def:Qb} and compute the following identity:
$$
-\bos{F}(Q_b)+b_1\bos{\Lambda} \bos{Q}_b=\bos{A}_1-\bos{A}_2 ,
$$
with:
$$
\begin{array}{l l}
\bos{A}_1:=b_1\bos{\Lambda} \bos{Q} + \sum_{i=1}^L[\bos{T}_i+b_i\bos{H}\bos{T}_i+b_1b_i\bos{\Lambda} \bos{T}_i]+\sum_{i=2}^{L+2}[\bos{H} \bos{S}_i+b_1 \bos{\Lambda} \bos{S}_i], \\
\bos{A}_2:=\begin{pmatrix}
0 \\ f(Q+\alpha_b^{(1)})-f(Q)+f'(Q)\alpha_b^{(1)} 
\end{pmatrix}.
\end{array} 
$$
Knowing in advance the fact that $S_i\sim b_1^i$ and $b_i\sim b_1^i$ we rearange all the term according to the power of $b_1$:
$$
\begin{array}{r c l}
\bos{A}_1&=&  b_1(\bos{\Lambda} \bos{Q} +\bos{H}\bos{T}_1)  +\sum_1^{L-1}[b_1b_i\bos{\Lambda} \bos{T}_i+b_{i+1}\bos{H}\bos{T}_{i+1}+\bos{H}\bos{S}_{i+1}+b_1\bos{\Lambda} \bos{S}_{i}] \\
&& + b_1b_L \bos{\Lambda} \bos{T}_L +\bos{H}\bos{S}_{L+1}+b_1\bos{\Lambda} \bos{S}_{L}  + b_1 \bos{\Lambda} \bos{S}_{L+1} +\bos{H}\bos{S}_{L+2}  +b_1\bos{\Lambda} \bos{S}_{L+2} \\
&=& \sum_1^{L-1}[b_1b_i\bos{\Lambda} \bos{T}_i-b_{i+1}\bos{T}_i+\bos{H}\bos{S}_{i+1}+b_1\bos{\Lambda} \bos{S}_{i}] \\
&&+ b_1b_L \bos{\Lambda} \bos{T}_L +\bos{H}\bos{S}_{L+1}+b_1\bos{\Lambda} \bos{S}_{L}  + b_1 \bos{\Lambda} \bos{S}_{L+1} +\bos{H}\bos{S}_{L+2}  +b_1\bos{\Lambda} \bos{S}_{L+2}. \\
\end{array} 
$$
Because we have assumed $p$ to be an integer, and from the localization of the $\bos{T}_i$'s \fref{linearized:eq:localisation Ti}, we can expand\footnote{For the moment we include all the $S_i^{(1)}$ because we still have not proved their localization.} $\bos{A}_2$ as a sum of polynomials of order higher or equal to 2:
$$
A_2^{(2)}=\sum_{j=2}^{p} C_j Q^{p-j}(\alpha_b^{(1)})^j=\sum_{j=2}^{p} C_j Q^{p-j}\left( \sum_{i=2, \ i \ \text{even}}^{L-1}b_iT_i+\sum_{i=2}^{L+2} S_i^{(1)}   \right)^j  .
$$
Again, we reorder these polynomials according to:
$$
A_2^{(2)}= \sum_{i=2}^{L+2}P_i+R .
$$
where:
$$
P_i=\sum_{j=2}^{p}C_jQ^{p-j} \left( \sum_{J, |J|_1=j, |J|_2=i}  \prod_{k=2, \ k \ \text{even}}^{L-1} b_i^{J_k}T_k^{J_k} \prod_{k=2}^{L+2} (S_k^{(1)})^{\tilde{J}_k} \right) ,
$$
where here $J=(J_2,...,J_{L-1},\tilde{J}_2,...,\tilde{J}_{L+2})$ and the way to count the powers of $b_1$ is: $|J|_2=\sum_{k=1}^{\frac{L-1}{2}} 2kJ_{2k}+\sum_{k=1}^{L+2} k\tilde{J}_k$. The remainder is:
$$
R=\sum_{j=2}^{p}C_jQ^{p-j} \sum_{J, |J|_1=j, |J|_2\geq L+3} \left(\prod_{k=2, \ k \ \text{even}}^{L-1} b_k^{J_k}T_k^{J_k} \prod_{k=1}^{L+2} (S_k^{(1)})^{\tilde{J}_k} \right) .
$$
We make an abuse of notation by denoting $\bos{P}_i:=\begin{pmatrix} 0 \\ P_i\end{pmatrix}$ and $\bos{R}:=\begin{pmatrix} 0 \\ R \end{pmatrix}$. The error term $\bos{\psi}_b$ has then the following expression (anticipating that $\frac{\partial \bos{S}_j}{\partial b_i}=0 \ \text{for} \ j\leq i$):
\be \label{linearized:eq: expression psib}
\begin{array}{r c l}
\bos{\psi}_b&=& \sum_{i=1}^{L} (-(i-\alpha)b_1b_i+b_{i+1})\frac{\partial \bos{Q}_b}{\partial b_i}+\bos{A}_1-\bos{A}_2 \\
&=& \sum_1^L (-(i-\alpha)b_1b_i+b_{i+1})\left[\bos{T}_i+\sum_{j=i+1}^{L+2} \frac{\partial \bos{S}_j}{\partial b_i}\right]+\bos{A}_1-\bos{A}_2 \\
&=& \sum_1^L [\bos{H}(\bos{S}_{i+1})+b_1b_i\bos{\Theta}_i+b_1\bos{\Lambda} \bos{S}_i + \bos{P}_{i+1}+\sum_{j=2}^{i-1}((j-\alpha)b_1b_j-b_{j+1})\frac{\partial \bos{S}_i}{\partial b_j}] \\
&& + \bos{H}(\bos{S}_{L+2})+b_1 \bos{\Lambda} \bos{S}_{L+1}+\bos{P}_{L+2}+\sum_{j=2}^L (-(j-\alpha)b_1b_j+b_{j+1})\frac{\partial \bos{S}_{L+1}}{\partial b_i} \\
&& + b_1 \bos{\Lambda} \bos{S}_{L+2}+\sum_{j=2}^L (-(j-\alpha)b_1b_j+b_{j+1})\frac{\partial \bos{S}_{L+2}}{\partial b_i} +\bos{R}_1 .
\end{array} 
\ee

\underline{Step 2:} Expression of the $\bos{S}_i$'s, simplification of $\bos{\psi}_b$. We define the $\bos{S}_i$'s by induction, in order to cancel the terms with a power of $b_1$ less than $L+2$ in \fref{linearized:eq: expression psib}:
\begin{equation} \label{linearized:eq:expression Si}
\left\{ \begin{array}{l l}
\bos{S}_1=0 ,\\
\bos{S}_{i}=-\bos{H}^{-1}(\bos{\Phi} _i) \ \text{for} \ 2\leq i \leq L+2,
\end{array}
\right. 
\end{equation}
with the following expression for the profiles $\bos{\Phi}_i$:
\begin{equation} \label{linearized:eq:def Phii}
\left\{ \begin{array}{l l}
\bos{\Phi}_{i+1}=b_1b_i\bos{\Theta}_i+b_1\bos{\Lambda} \bos{S}_i + \bos{P}_{i+1} +\sum_{j=1}^{i-1}(-(j-\alpha)b_1b_j-b_{j+1})\frac{\partial \bos{S}_i}{\partial b_j} \ \text{for} \ 1\leq i \leq L, \\
\bos{\Phi}_{L+2}=b_1\bos{\Lambda} \bos{S}_{L+1} +\bos{P}_{L+2}+\sum_{j=1}^{L-1}(-(j-\alpha)b_1b_j-b_{j+1})\frac{\partial \bos{S}_{L+1}}{\partial b_j} .
\end{array} \right.
\end{equation}
The $\bos{S}_i$'s being defined by \fref{linearized:eq:expression Si}, $\bos{\psi}_b$ has now the following expression:
\be \label{linearized:eq:expression psib 2}
\bos{\psi}_b= b_1 \bos{\Lambda} \bos{S}_{L+2}+\sum_{j=1}^L (-(j-\alpha)b_1b_j+b_{j+1})\frac{\partial \bos{S}_{L+2}}{\partial b_i} +\bos{R} .
\ee

\underline{Step 3:} Properties of the $\bos{S}_i$'s. We claim the following facts (we recall that the homogeneity is defined in Definition \ref{def:homogeneousfunctions}):
\begin{itemize}
\item[(i)] $\bos{S}_i$ is homogeneous of degree $(i,i-g',i \ \text{mod}2,i)$
\item[(ii)] $\bos{P}_i=0$ for $i$ odd,
\item[(iii)] the condition $\frac{\partial \bos{S}_j}{\partial b_i}=0$ for $j\leq i$ is fullfiled.
\end{itemize}
The proof of the fact that $\bos{P}_i=0$ for $i$ odd is an easy induction left to the reader. We will also prove the two other facts by induction. For $i=2$ we have:
$$
\bos{S}_2=\bos{H}^{(-1)}(b_1^2\bos{\Theta}_1+\bos{P}_2) ,
$$
and it is straightforward to check that $\bos{P}_2=0$. Hence from the result about the $\bos{\Theta}_i$'s given by Lemma \ref{lem:profilsTi}, we have that $\bos{S}_2$ is of degree $(2,2-g',0,2)$. It is also clear from the previous identity that $\frac{\partial S_2}{\partial b_i}=0$ for $2\leq i \leq L$.\\
\\
We now suppose $i\geq 3$, and that the properties (i) and (iii) are true for all $2\leq j<i$, which is our induction hypothesis. We look at all the terms in the right hand side of \fref{linearized:eq:def Phii}. $b_1b_{i-1} \bos{\Theta}_{i-1}$ is of degree $(i-1,i-1-g',i-1\text{mod}2,i)$. By the induction hypothesis, $b_1\bos{\Lambda} \bos{S}_{i-1}$ is of degree $(i-1,i-1-g',i-1 \ \text{mod}2,i)$, and so is the profile $(-(j-\alpha)b_1b_j-b_{j+1})\frac{\partial \bos{S}_{i-1}}{\partial b_j}$. If $i$ is odd, $\bos{P}_i=0$ and there is nothing to prove. If $i$ is even, from the position of the $\bos{T}_i$'s \fref{linearized:eq:localisation Ti}, and the position \fref{linearized:eq:position Si} of the $\bos{S}_j$'s for $j<i$ given by the induction hypothesis (i), $\bos{P}_{i}$ is a linear combination of terms of the form:
$$
Q^{p-j}\prod_{k<i, \ k \ \text{even}} b_k^{J_k}T_k^{J_k} \prod_{k<i, \ k \ \text{even}} S_k^{\tilde{J}_k} ,
$$
for $2 \leq j\leq p$, $|J|_1=j$ and $|J|_2=i$. From the induction hypothesis and the Calculus Lemma for admissible functions \ref{lem:calculsurfonctionsadmissibles}, we deduce the asymptotics:
$$
\begin{array}{r c l}
Q^{p-j}\underset{k<i, \ k \ \text{even}}{\prod} b_k^{J_k}T_k^{J_k} \underset{k<i, \ k \ \text{even}}{\prod} S_k^{\tilde{J}_k} &=& O\left(b_1^{i}\frac{1}{1+y^{(p-j)\frac{2}{p-1}+\sum J_k(\gamma-k) +\sum \tilde{J}_k(\gamma-k+g')}}\right) \\
&=& O\left(b_1^{i}\frac{1}{1+y^{2+\frac{2}{p-1}+j\alpha+\sum \tilde{J}_kg'-i}} \right)\\
&=& O\left( b_1^i \frac{1}{1+y^{2+\gamma +(j-1)\alpha+\sum \tilde{J}_k g'-i}} \right),
\end{array}
$$
which adapts for higher derivatives (ie deriving $k$ times the left hand side amounts to divide the right hand side by $y^k$). As $j\geq 2$ and $\alpha\geq 2 \geq g'$ we conclude that $\bos{P}_i$ is of degree $(i-1,i-1-g',1,i)$ (the expansion at the origin can be checked the same way). In this step, so far, we have proven that $\bos{\Phi}_i$ is of degree $(i-1,i-1-g',i-1 \ \text{mod}2,i)$, hence from the inversion Lemma \ref{lem:actiondeHetH-1suradmissile} $\bos{S}_i$ is of degree $(i,i-g',i\text{mod}2,i)$.\\

\underline{Step 4:} Bounds for the error term. We now turn to the expression of the error $\bos{\psi}_b$ given by $\fref{linearized:eq:expression psib 2}$, and estimate all terms in the right hand side. We showed in step 3 that $\bos{S}_{L+2}$ is of degree $(L+2,L+2-g',L+2 \ \text{mod}2,L+2)$. As $L$ is odd, and as $\bos{R}$ is situated on the second coordinate we obtain the localization of $\bos{\psi}_b$:
$$
\bos{\psi}_b=\begin{pmatrix}
0 \\
b_1 \Lambda^{(2)}S_{L+2}^{(2)}+\sum_{j=1}^L(-(j-\alpha)b_1b_j+b_{j+1})\frac{\partial S_{L+2}^{(2)}}{\partial b_i}+R
\end{pmatrix} .
$$
We start by estimating the first two terms. We already know that $b_1\bos{\Lambda} \bos{S}_{L+2}$ and $\sum_{j=1}^L(-(j-\alpha)b_1b_j+b_{j+1})\frac{\partial \bos{S}_{L+2}^{(2)}}{\partial b_i}$ are of degree $(L+2,L+2-g',1,L+3)$. This leads to the following estimates (the local adapted norm was defined in \fref{linearized:eq:def normes adaptees locales}):
$$
\begin{array}{r c l}
& \parallel b_1\bos{\Lambda} \bos{S}_{L+2}+\sum_{j=1}^L(-(j-\alpha)b_1b_j+b_{j+1})\frac{\partial \bos{S}_{L+2}^{(2)}}{\partial b_i} \parallel_{j,(\leq B_1)}^2 \\
\leq& C(L) \int_0^{B_1} \left|\frac{|b_1|^{L+3}}{y^{\gamma-(L+2-g')+1+k_0+j}}\right|^2 y^{d-1}dy \\
=& C(L) b_1^{2L+6}\int_0^{B_1} y^{2\delta_0-2g'+2L+2-2j-1}dy = C(L)b_1^{2j+2+2(1-\delta_0)+2g'}.\\
\end{array} 
$$
The integral in the right hand side is always divergent as $j \leq L$, and as $1+\delta_0-g'\geq 0$ (see the definition of $g'$ \fref{eq:def:gain2}, the presence of $1+\delta_0$ was made to produce this result). We now prove the local estimates. We recall that we proved in step 3 that $b_1\bos{\Lambda} \bos{S}_{L+2}+\sum_{j=1}^L(-(j-\alpha)b_1b_j+b_{j+1})\frac{\partial \bos{S}_{L+2}}{\partial b_i}$ is homogeneous of degree $p_3=L+3$. This means that:
$$
b_1\bos{\Lambda} S_{L+2}+\sum_{j=1}^L(-(j-\alpha)b_1b_j+b_{j+1})\frac{\partial S_{L+2}}{\partial b_i} = \sum_{|J|_2= L+3}b^J f_J ,
$$
for a finite number of functions $f_J$ such that $|\partial_y^kf_J|\lesssim y^{-\gamma+L+2-1-g'-k}$ at infinity, and with $b^J=\prod b_i^{J_i}$. Hence the brute force upper bound:
$$
\left| \partial_y^k \left(b_1\Lambda S_{L+2}+\sum_{j=1}^L(-(j-\alpha)b_1b_j+b_{j+1})\frac{\partial S_{L+2}^{(2)}}{\partial b_i}\right)\right| \lesssim b_1^{L+3} (1+y)^{-\gamma+L+2-1-g'-k}
$$
which implies the local bound \fref{linearized:eq:estimations locales non coupe} for this term. We now turn to the bounds for the $\bos{R}$ term. Thanks to the homogeneity property of the $\bos{S}_i$'s, $R$ is of the form:
$$
R=\sum_{|J|_2\geq L+3} \prod_{i=1}^L b_i^{J_i} g_J ,
$$
for a finite number of functions $g_J$ whose derivatives have polynomial growth at infinity. This directly implies the local bounds \fref{linearized:eq:estimations locales non coupe} for this term. For the global bounds, we rewrite $R$ as a linear sum of terms of the form:
$$
Q^{p-j} \left(\prod_{i=2, \ i \ \text{even}}^{L} b_i^{J_i}T_i^{J_i} \prod_{i=2, \ i \ \text{even}}^L S_i^{\tilde{J}_i} \right) ,
$$
for $|J|_2\geq L+3$ and $2\leq j \leq p$. Using again the Calculus Lemma for admissible functions \ref{lem:calculsurfonctionsadmissibles}, each term has the asymptotic behavior:
$$
\begin{array}{r c l}
Q^{p-j} \left(\prod b_i^{J_i}T_i^{J_i} \prod S_i^{\tilde{J}_i} \right) &=& O\left( \frac{b_1^{|J|_2}}{1+y^{\frac{2}{p-1}(p-j)+\sum (\gamma-J_i)+\sum (\gamma-\tilde{J}_i+g')}}\right) \\
&=& O\left( \frac{b_1^{|J_2|}}{1+y^{2+\gamma+(j-1)\alpha +(\sum\tilde{J}_i)g'-|J_2|}} \right).
\end{array} 
$$
For all $k\in \mathbb{N}$:
$$
\partial_y^k \left(Q^{p-j} \left(\prod b_i^{J_i}T_i^{J_i} \prod S_i^{\tilde{J}_i} \right)\right)=O\left(  \frac{b_1^{|J_2|}}{1+y^{2+\gamma+(j-1)\alpha +(\sum\tilde{J}_i)g'-|J_2|+k}} \right).
$$
From the fact that $(j-1)\alpha>2\geq g'$ we conclude that the global estimates of the term $\bos{R}$ are in all cases better (ie with a higher power of $b_1$, $b_1$ being small $0< b_1\ll1$) than the ones for $b_1\bos{\Lambda} \bos{S}_{L+2}+\sum_{j=1}^L(-(j-\alpha)b_1b_j+b_{j+1})\frac{\partial \bos{S}_{L+2}^{(2)}}{\partial b_i}$, which concludes the proof.
\end{proof}

As we have seen with the previous estimates of the error term $\bos{\psi}_b$, we have a good approximate dynamics for $y\leq B_1$. However, as
$$
T_i\sim y^{-\gamma+i-\delta_{i \ \text{odd}}}\rightarrow +\infty \ \text{as} \ y \rightarrow +\infty \ (\text{as soon as} \ i> \gamma+1),
$$
the approximate dynamic is irrelevant far away of the origin. Consequently, we will now localise the profiles of Proposition \ref{pr:constructionprofilnoncoupe} in the zone $y\leq B_1$, where $\frac{b_{2i}T_{2i}}{\Lambda^{(1)}Q}$ is nearly of order 1. To do this, we will simply multiply by a cut-off function. This cut will create additional error terms that we will estimate in the next proposition. We recall that our cut-off function $\chi$ is defined by \fref{intro:eq:definition chi}. We denote by $\chi_{B_1}\bos{\alpha}_b$:
\begin{equation}
\chi_{B_1}\bos{\alpha}_b:=\begin{pmatrix}
\chi_{B_1} \alpha_b^{(1)} \\ \chi_{B_1}\alpha_b^{(2)}
\end{pmatrix} .
\end{equation}

\begin{proposition}[Localization of the approximate profile]\label{pr:profilapprochecoupe}
We use the assumptions and notations of Proposition \ref{pr:constructionprofilnoncoupe}. Let $I=]s_0,s_1[$ denote a renormalized time interval, and 
$$
\begin{array}{r c l}b:I &\rightarrow& \mathbb{R}^L \\ s &\mapsto& (b_i(s))_{1\leq i\leq L} \end{array} 
$$
be a $C^1$ map such that: $|b_i|\lesssim b_1^i$ with $0<b_1\ll 1$. Assume the a priori bound:
\begin{equation}
|b_{1,s}|\lesssim b_1^2 .
\end{equation}
Let $\tilde{\bos{Q}}_b$ denote the localized profile, given by:
\begin{equation}\label{eq:def:Qbtilde}
\tilde{\bos{Q}}_b=\bos{Q}+\chi_{B_1}\bos{\alpha}_b .
\end{equation}
Then for $0<\eta\ll 1$ small enough one has the following identity ($\bos{Mod}(t)$ being defined by \fref{linearized:eq:def Mod}):
\begin{equation}
\partial_s \tilde{\bos{Q}}_b-\bos{F}(\tilde{\bos{Q}}_b)+b_1 \bos{\Lambda} \tilde{\bos{Q}}_b = \tilde{\bos{\psi}}_b +\chi_{B_1}\bos{Mod}(t).
\end{equation}
$\tilde{\bos{\psi}}_b$, the new error term, satisfies (the adapted norm being defined in \fref{linearized:eq:def normes adaptees locales}):
\begin{itemize}
\item[(i)]\emph{Global weighted bounds:}
\begin{equation} \label{linearized:eq:global bound jleqL}
\forall 0\leq j \leq L-1, \ \parallel \tilde{\bos{\psi}}_b \parallel_j^2\leq C(L) b_1^{2j+2+2(1-\delta_0)-C_j\eta} ,
\end{equation}
\begin{equation} \label{linearized:eq:global bound L}
\text{for} \ j=L, \ \parallel \tilde{\bos{\psi}}_b \parallel_L^2\leq C(L) b_1^{2L+2+2(1-\delta_0)(1+\eta)}.
\end{equation}
\item[(ii)]\emph{Local improved bounds:} For $x\leq \frac{B_1}{2}$, $\tilde{\bos{\psi}}_b(x)=\bos{\psi}_b(x)$, where $\bos{\psi}_B$ is the former error term of Proposition \ref{pr:constructionprofilnoncoupe}. Hence $\forall j\geq 0, \ \forall 1\leq B\leq \frac{B_1}{2}:$
\begin{equation}
\int_{|y|\leq  B} |\nabla^j \tilde{\psi}_b^{(1)}|^2 +|\nabla^j \tilde{\psi}_b^{(2)}|^2 = \int_{|y|\leq B}|\nabla^j \psi_b^{(2)}|^2 \lesssim C(L,j)B^{C(L,j)}b_1^{2L+6} .
\end{equation}
\end{itemize}
\end{proposition}

\begin{remark}
When comparing the estimates given by this proposition, and the ones given in the proposition \ref{pr:constructionprofilnoncoupe}, we note a loss. Indeed the first non cut profile creates an error seen on the corrective terms $\bos{S}_{L+2}$ and $\bos{R}$ which enjoy additional gains $y^{-g'}$ or $y^{-\alpha}$ away from the origin compared to the $\bos{T}_i$'s. When cutting, we see in the additional error term the profiles $\bos{T}_i$'s, giving a worst estimate as they do not have this additional gain.\\
\\
However, the error created in the zone $\leq B_1$ is left unperturbed by the cut. The fact that the error enjoys two different estimates: a good one in the zone $y\leq B_1$ and a bad one in the zone $B_1\leq y \leq 2B_1$ will be helpful in the analysis later.

\end{remark}

\begin{proof}[Proof of Proposition \ref{pr:profilapprochecoupe}]
We compute the error in localizing:
$$
\begin{array}{r c l}
\partial_s \tilde{\bos{Q}}_b -\bos{F}(\tilde{\bos{Q}}_b)+b_1 \bos{\Lambda} \tilde{\bos{Q}}_b &=& \chi_{B_1}\bos{\psi}_b +\chi_{B_1}\bos{Mod}(t)\\
&& + \partial_s(\chi_{B_1})\tilde{\bos{\alpha}_b} + b_1(\bos{\Lambda} \tilde{\bos{Q}_b}-\chi_{B_1}\bos{\Lambda} \bos{Q}_b) \\
&& -(\bos{F}(\tilde{\bos{Q}}_b)-\bos{F}(\bos{Q})-\chi_{B_1}(\bos{F}(\bos{Q}_b)-\bos{F}(\bos{Q})) .
\end{array} 
$$
So we have the following expression for the new error term:
\begin{equation} \label{eq:def:psibtilde}
\begin{array}{ r c l}
\tilde{\bos{\psi}}_b&=& \chi_{B_1}\bos{\psi}_b + \partial_s(\chi_{B_1})\bos{\alpha}_b+b_1(\bos{\Lambda} \tilde{\bos{Q}_b}-\chi_{B_1}\bos{\Lambda} \bos{Q}_b) \\
&& -(\bos{F}(\tilde{\bos{Q}}_b)-\bos{F}(\bos{Q})-\chi_{B_1}(\bos{F}(\bos{Q}_b)-\bos{F}(\bos{Q})) ,
\end{array} 
\end{equation}
and we aim at estimating all these terms in global and local norms.\\

\underline{Local bounds:} From \fref{eq:def:psibtilde} we clearly see that $\tilde{\bos{\psi}}_b\equiv \bos{\psi}_b$ for $|y|\leq \frac{B_1}{2}$, because the new error terms appearing when cutting are created in the zone $B_1\leq |y| \leq 2B_1$. Therefore the local bounds are a direct consequence of the local ones established in \fref{linearized:eq:estimations locales non coupe}. \\

\underline{Global bounds:} We recall that $\parallel \bos{f} \parallel_j^2=\parallel f^{(1)}_j \parallel_{L^2}^2+\parallel f^{(2)}_{j-1} \parallel_{L^2}^2$ where the $j$-th adapted derivative of a function is defined by \fref{linearized:eq:def derivees adaptees}. We will now compute this norm for all the terms in the right hand side of \fref{eq:def:psibtilde}.\\
$\bullet$ \emph{$\chi_{B_1}\bos{\psi}_b$ term:} When applying the differential operators $A$ or $A^*$ to any product $\chi_{B_1}f$, we have:
\begin{equation}\label{eq:actiondeAA*surchif}
\begin{array}{r c l}
A(\chi_{B_1}f)&=&\chi_{B_1}f_1-b^{1+\eta}\partial_y \chi \left(\frac{y}{B_1}\right)f, \\
A^*A(\chi_{B_1}f)&=&\chi_{B_1}f_2+b^{1+\eta}\partial_y \chi\left(\frac{y}{B_1} \right)f_1 \\ 
&&-\left[b^{2+2\eta}\partial_y^2 \chi (\frac{y}{B_1})+b^{1+\eta}\partial_y \chi \left(\frac{y}{B_1} \right)\left(2W+\frac{d-1}{y}\right)\right]f.
\end{array}
\end{equation}
And so on for higher powers of $A$ and $A^*$. Because of the asymptotic of $W$, see Lemma \ref{lem:asymptoticdupotentiel}, the general expression is of the form:
$$
(\chi_{B_1}f)_{i}=\chi_{B_1}f_i+\un _{B_1\leq y \leq 2B_1} \sum_{j=1}^i a_jf_j ,
$$
where $a_i(y)=O( y^{-(i-j)})$. It means that deriving $\chi_{B_1}$ amounts to dividing by $B_1$ and localizing in the zone $B_1\leq y \leq 2B_1$. Hence for $0\leq j\leq L$:
\be \label{linearized:eq:estimation finale psib}
\begin{array}{r c l}
\parallel \chi_{B_1}\bos{\psi}_b \parallel_j^2&=&\int \left|(\chi_{B_1}\psi_b^{(2)})_{k_0+j}\right|^2 \\
&\leq&C(L) \sum_{i=1}^{k_0+j} \int_{B_1\leq |y|\leq 2B_1} b^{2(1+\eta)i}|\psi_{b,k_0+j-i}^{(2)}|^2 + \int_{|y|\leq 2B_1} |\psi_{b,k_0+j}^{(2)}|^2 \\
&\leq & C(L) \parallel \psi_b \parallel_{j,\leq2B_1}^2 +C(L)\sum_{i=1}^{k_0+j} \int b^{2(1+\eta)i}\left|  \frac{b^{L+3}}{y^{\gamma-(L+2)+1+g'+k_0+j-i}}\right|^2 \\
&\leq & C(L)b^{2j+2+2(1-\delta_0)(1+\eta)},
\end{array}
\ee
thanks to the Proposition \ref{pr:constructionprofilnoncoupe}.
\\
$\bullet$ \emph{$\partial_s(\chi_{B_1})\bos{\alpha}_b$ term:} We have from the assumption $|b_{1,s}|\lesssim b_1^2$:
$$
\partial_s(\chi_{B_1})=(1+\eta)b^{\eta}b_sy\partial_y \chi (\frac{y}{B_1})\lesssim b_1 b_1^{1+\eta}y \partial_y \chi (\frac{y}{B_1}).
$$
Again, deriving $y\partial_y \chi {\frac{y}{B_1}}$ amounts to dividing by $B_1$,  we get:
\be \label{linearized:eq:estimation intermediaire alpha}
\begin{array}{r c l}
\parallel \partial_s(\chi_{B_1})\bos{\alpha}_b\parallel_j^2 &=& \int |(\partial_s(\chi_{B_1})\alpha_b^{(1)})_{k_0+j+1}|^2 + |(\partial_s(\chi_{B_1})\alpha_b^{(2)})_{k_0+j}|^2 \\
&\leq& C(L)b_1^2\int_{B_1\leq y\leq 2B_1} |\alpha_{b,k_0+j+1}^{(1)}|^2+ |\alpha_{b,k_0+j}^{(2)}|^2.
\end{array}
\ee
We estimate the two terms using the asymptotic of the $T_i$'s from Lemma \ref{lem:profilsTi} and \fref{linearized:eq:degre Si} for the $S_i$'s:
\be \label{linearized:eq:estimation intermediaire alpha1}
\begin{array}{r c l}
\int_{B_1\leq |y|\leq 2B_1} |\alpha_{b,k_0+j+1}^{(1)}|^2&\leq &\int_{B_1\leq |y| \leq 2B_1} \sum_{i=2, \ i \ \text{even}}^{L-1} |b_iT_{i,k_0+j+1}|^2 \\
&&+\int_{B_1\leq |y| \leq 2B_1}  \sum_{i=2, \ i \ even}^{L+1} |S_{i,k_0+j+2}|^2 \\
&\leq& C(L)\sum_{i=2, \ i \ \text{even}}^{L-1} b_1^{2i} \int_{B_1}^{2B_1} \frac{1}{y^{2\gamma-2i+2k_0+2j+2}}y^{d-1}dy \\
&& +C(L)\sum_{i=2, \ i \ \text{even}}^{L+1} b_1^{2i} \int_{B_1}^{2B_1} \frac{1}{y^{2\gamma-2i+2k_0+2j+2+2g'}}y^{d-1}dy \\
&=&C(L)\sum_{i=2, \ i \ \text{even}}^{L-1} b_1^{2i} \int_{B_1}^{2B_1} y^{2\delta_0-2+2i-2j-1}dy \\
&& +C(L)\sum_{i=2, \ i \ \text{even}}^{L+1} b_1^{2i} \int_{B_1}^{2B_1} y^{2\delta_0-2+2i-2j-2g'-1}dy .
\end{array}
\ee
Similarly:
\be
\begin{array}{r c l} \label{linearized:eq:estimation intermediaire alpha2}
\int_{B_1}^{2B_1} |\alpha_{b,k_0+j}^{(2)}|^2 &\leq& C(L) \sum_{i=1, \ i \ \text{odd}}^{L} b_1^{2i} \int_{B_1}^{2B_1} y^{2\delta_0-2+2i-2j-1}dy \\
&& +C(L)\sum_{i=3, \ i \ \text{odd}}^{L+2} b_1^{2i} \int_{B_1}^{2B_1} y^{2\delta_0-2+2i-2j-2g'-1}dy .
\end{array}
\ee
The first upper bound \fref{linearized:eq:estimation intermediaire alpha}, combined with the two we just proved,  \fref{linearized:eq:estimation intermediaire alpha1} and \fref{linearized:eq:estimation intermediaire alpha2}, lead to (because $0<\delta_0<1$ avoids a possible log-term in the first sum):
\be \label{linearized:eq:estimation finale alpha}
\begin{array}{r c l}
\parallel \partial_s(\chi_{B_1})\bos{\alpha}_b\parallel_j^2 &\leq& C(L)\sum_{i=1}^L b_1^{2i} B_1^{2\delta_0-2+2i-2j}\\
&& +C(L) \sum_{i=2}^{L+2} b_1^{2i} B_1^{2\delta_0-2+2i-2j-2g'}\text{log}(B_1) \\
&\leq& C(L)\sum_1^L b^{2j+2(1+\eta)(1-\delta_0)-\eta(2i-2j)} \\
&& +C(L) \sum_2^{L+2} b^{2j+2(1+\eta)(1-\delta_0)-\eta(2i-2j-2g')+2g'} \text{log}(B_1) \\
&\leq& \left\{ \begin{array}{l l}
C(L)b^{2j+2(1-\delta_0)-C_j\eta} \ \text{for} \ j\leq L-1, \\
C(L)b^{2L+2(1+\eta)(1-\delta_0)} \ \text{for} \ j=L,
\end{array} .
\right.
\end{array}
\ee
for $\eta$ small enough.\\
$\bullet$ \emph{$\bos{F}(\tilde{\bos{Q}_b})-\bos{F}(\bos{Q})-\chi_{B_1}(\bos{F}(\bos{Q}_b)-\bos{F}(\bos{Q}))$ term:} We compute:
\be \label{linearized:eq:expression erreur terme F}
\begin{array}{r c l}
&\bos{F}(\tilde{\bos{Q}_b})-\bos{F}(\bos{Q})-\chi_{B_1}(\bos{F}(\bos{Q}_b)-\bos{F}(\bos{Q}))\\
=&\begin{pmatrix}
0 \\ \Delta (\chi_{B_1}\alpha_b^{(1)})-\chi_{B_1}\Delta(\alpha_b^{(1)})+f(\tilde{Q}_b)-f(Q)-\chi_{B_1}(f(Q_b)-f(Q))
\end{pmatrix} .
\end{array}
\ee
We estimate the two terms in the right hand side of \fref{linearized:eq:expression erreur terme F}:
$$
\begin{array}{r c l}
\Delta (\chi_{B_1}\alpha_b^{(1)})-\chi_{B_1}\Delta(\alpha_b^{(1)})&=& \partial_y(\chi_{B_1})\partial_y(\alpha_b^{(1)}) +\Delta(\chi_{B_1})\alpha_b^{(1)} \\
&=& b^{1+\eta}\partial_y \chi (\frac{y}{B_1})\partial_y(\alpha_b^{(1)})+ b^{2(1+\eta)}\Delta \chi (\frac{y}{B_1}) \alpha_b^{(1)} .
\end{array}
$$
Considering the asymptotics of $\alpha_b^{(1)}$ we have:
\be \label{linearized:eq:estimation intermediaire erreur terme F}
\begin{array}{r c l}
&&\int |(\Delta (\chi_{B_1}\alpha_b^{(1)})-\chi_{B_1}\Delta(\alpha_b^{(1)}))_{k_0+j}|^2 \\
&\leq& C(L)b^{2(1+\eta)}\int_{B_1}^{2B_1} \left( \sum_{i=2}^{L-1} b^{2i}y^{2\delta_0-2j+2i-2}+ \sum_{i=2,L+1} b^{2i}y^{2\delta_0-2j+2i-2-2g'}\right)dy \\
&+& C(L)b^{4(1+\eta)}\int_{B_1}^{2B_1} \left( \sum_{i=2}^{L-1} b^{2i}y^{2\delta_0-2j+2i}+ \sum_{i=2,L+1} b^{2i}y^{2\delta_0-2j+2i-2g'} \right)dy \\
&\leq & \left\{ \begin{array}{l l} 
C(L)b^{2+2j+2(1-\delta_0)-C_j\eta} \ for \ 0\leq j\leq L, \\
C(L)b^{2+2L+2(1-\delta_0)(1+\eta)} \ for \ j= L.
\end{array} \right. 
\end{array}
\ee
because $i<L-1$ in the sum concerning the $\bos{T}_i$'s and because of the gain $g'>0$ in the one of the $\bos{S}_i$'s. The second term is:
$$
f(\tilde{Q}_b)-f(Q)-\chi_{B_1}(f(Q_b)-f(Q))= \chi_{B_1}\sum_{k=2}^{p}C_kQ^{p-k}(\chi_{B_1}^{k-1}-1)\alpha_b^{(1)k} .
$$
For each $2\leq k\leq p$, we can expand the polynomial and we have a linear sum of terms of the form:
$$
\chi_{B_1}Q^{p-k}(\chi_{B_1}^{k-1}-1) \prod_{i=2, \ i \ \text{even}}^{L-1} (b_iT_i)^{J_i}\prod_{i=2, \ i \ \text{even}}^{L+1} (S_i)^{\tilde{J}_i} ,
$$
for $|J|_1=k$. According to the calculus Lemma \ref{lem:calculsurfonctionsadmissibles} for homogeneous functions:
$$
\begin{array}{r c l}
\partial_y^l \left(Q^{p-k}\underset{i=2, \ i \ \text{even}}{\overset{L-1}{\prod}}(b_iT_i)^{J_i}\underset{i=2, \ i \ \text{even}}{\overset{L+1}{\prod}} (S_i)^{\tilde{J}_i}\right) &\underset{y\rightarrow +\infty}{=}& O\left( \frac{b_1^{|J|_2}}{y^{(p-k)\frac{2}{p-1}+k\gamma +\sum \tilde{J}_i g'-|J|_2+l}}\right) \\
&\underset{y\rightarrow +\infty}{=}& O\left(\frac{b_1^{|J|_2}}{y^{2+\gamma+(k-1)\alpha-|J|_2+\sum \tilde{J}_i g'+l}}\right) .
\end{array}
$$
As we have seen before, the presence of the term $\chi_{B_1}$ does not affect the computation (deriving $\chi_{B_1}$ amounts to divide by $y$):
\be \label{linearized:eq:estimation intermediaire f}
\begin{array}{r c l}
&&\int_{B_1}^{2B_1} |(Q^{p-k}\prod_{i=2, \ i  \ \text{even}}^{L-1} (b_iT_i)^{J_i}\prod_{i=2, \ i \ \text{even}}^{L+1} (S_i)^{\tilde{J}_i})_{k_0+j}|^2 \\
&\leq& C(L)\int_{B_1}^{2B_1} \frac{b_1^{2|J|_2}}{y^{4+2\gamma+2(k-1)\alpha-2|J|_2+2\sum \tilde{J}_i g'+2k_0+2j}}y^{d-1}dy \\
&\leq& C(L)\int_{B_1}^{2B_1} b_1^{2|J|_2}y^{-4+2\delta_0-2(k-1)\alpha+2|J|_2-2\sum \tilde{J}_i g'-2j-1}dy \\
&\leq & C(L)b^{2+2j+2(1-\delta_0)(1+\eta)} \ \text{for} \ 0\leq j\leq L .
\end{array}
\ee
because of the gain $(k-1)\alpha>\alpha>2$. The bound \fref{linearized:eq:estimation intermediaire f} then implies the bound for $1\leq j \leq L$:
\be \label{linearized:eq:estimation intermediaire erreur term F 2}
\int |(f(\tilde{Q}_b)-f(Q)-\chi_{B_1}(f(Q_b)-f(Q)))_{j+k_0}|^2\leq C(L)b^{2+2j+2(1-\delta_0)(1+\eta)} \ \text{for} \ 0\leq j\leq L.
\ee
The primary decomposition \fref{linearized:eq:expression erreur terme F}, with the bounds \fref{linearized:eq:estimation intermediaire erreur terme F} and \fref{linearized:eq:estimation intermediaire erreur term F 2} implies the bound we were looking for:
\be \label{linearized:eq:estimation finale F}
\parallel \bos{F}(\tilde{\bos{Q}_b})-\bos{F}(\bos{Q})-\chi_{B_1}(\bos{F}(\bos{Q}_b)-\bos{F}(\bos{Q}))\parallel_{j}^2 \leq \left\{ \begin{array}{l l} 
C(L)b^{2+2j+2(1-\delta_0)-C_j\eta} \ \text{for} \ 0\leq j\leq L, \\
C(L)b^{2+2L+2(1-\delta_0)(1+\eta)} \ \text{for} \ j= L. \end{array}\right.
\ee
$\bullet$ \emph{$b_1(\bos{\Lambda} \tilde{\bos{Q}}_b-\chi_{B_1}\bos{\Lambda} \bos{Q}_b)$ term:} We compute:
$$
\bos{\Lambda} \tilde{\bos{Q}}_b-\chi_{B_1}\bos{\Lambda} \bos{Q}_b=(1-\chi)\bos{\Lambda} \bos{Q} + y\partial_y(\chi_{B_1})\bos{\alpha}_b .
$$
We have that:
$$
y\partial_y(\chi_{B_1})=b_1^{1+\eta}y\partial_y \chi (\frac{y}{B_1}) .
$$
So the term $y\partial_y(\chi_{B_1})\bos{\alpha}_b$ behaves the same way as the term $\partial_s(\chi_{B_1})\bos{\alpha}_b$ previously treated and enjoys the same estimations. Finally we estimate the soliton contribution, because of which we had to derive $k_0$ times at least in order to have integrability. We again use the fact that deriving k times $\chi_{B_1}$ amounts to divide by $y^k$ and to localize in the zone $B_1\leq y \leq 2B_1$.
$$
\begin{array}{r c l}
\int |b_1 (1-\chi_{B_1})\Lambda^{(1)}Q_{k_0+j+1}|^2 & \leq & C(L)b_1^2 \int_{B_1}^{\infty} y^{-2\gamma-2k_0-2-2j+d-1} dy\\
& \leq & C(L)b_1^{2+2j+2(1-\delta_0)+(2j+2(1-\delta_0))\eta} .
\end{array}
$$
So that finally:
\be \label{linearized:eq:estimation finale lambdaQb}
\parallel b_1(\bos{\Lambda} \tilde{\bos{Q}}_b-\chi_{B_1}\bos{\Lambda} \bos{Q}_b) \parallel_{j}^2 \lesssim \left\{ \begin{array}{l l}
C(L)b^{2j+2(1+\eta)(1-\delta_0)-C_j\eta} \ \text{for} \ j\leq L-1, \\
C(L)b^{2L+2(1+\eta)(1-\delta_0)} \ \text{for} \ j=L,
\end{array} .
\right.
\ee

The decomposition \fref{eq:def:psibtilde}, with the bounds for each term \fref{linearized:eq:estimation finale psib}, \fref{linearized:eq:estimation finale alpha}, \fref{linearized:eq:estimation finale F} and \fref{linearized:eq:estimation finale lambdaQb} give the global bounds \fref{linearized:eq:global bound jleqL} and \fref{linearized:eq:global bound L} we had to prove.
\end{proof}


\subsection{Study of the dynamical system driving the evolution of the parameters $(b_i)_{1\leq i \leq L}$}

We have constructed in the preceding propositions \ref{pr:constructionprofilnoncoupe} and \ref{pr:profilapprochecoupe} a manifold of functions near the solitary wave such that:
$$
\bos{F}(\tilde{\bos{Q}}_b)\sim b_1\bos{\Lambda} \tilde{\bos{Q}}_b+\sum_{i=1}^L (-(i-\alpha)b_1b_i+b_{i+1})\frac{\partial \tilde{\bos{Q}}_b}{\partial b_i} .
$$
By applying scaling, and the identity $\frac{\partial (\bos{f}_{\lambda})}{\partial \lambda}=\frac{1}{\lambda}\bos{\Lambda}\bos{f}_{\lambda}$ we have that:
$$
\bos{F}(\tilde{\bos{Q}}_{b,\frac{1}{\lambda}})\sim \frac{b_1}{\lambda} (\bos{\Lambda} \tilde{\bos{Q}}_b)_{\lambda}+\sum_{i=1}^L \frac{1}{\lambda}(-(i-\alpha)b_1b_i+b_{i+1})\frac{\partial \tilde{\bos{Q}}_b}{\partial b_i} .
$$
Hence approximately a solution of (NLW) on this manifold gives:
$$
\begin{array}{r c l}
-\frac{\lambda_t}{\lambda}\bos{\Lambda}(\tilde{\bos{Q}}_b)_{\frac{1}{\lambda}}+\sum b_{i,t} \left(\frac{\partial \tilde{\bos{Q}}_{b}}{\partial b_i}\right)_{\frac{1}{\lambda}}&=&\partial_t(\tilde{\bos{Q}}_{b,\frac{1}{\lambda}}) \\
&=& \bos{F}(\tilde{\bos{Q}}_{b,\frac{1}{\lambda}}) \\
&\sim& \frac{b_1}{\lambda} (\bos{\Lambda} \tilde{\bos{Q}}_b)_{\frac{1}{\lambda}}+\sum(-(i-\alpha)b_1b_i+b_{i+1})\left( \frac{\partial \tilde{\bos{Q}}_b}{\partial {b_i}} \right)_{\frac{1}{\lambda}} .
\end{array}
$$
By identifying the termswe obtain:
\begin{equation}
\left\{
\begin{array}{l l}
\lambda_t=-b_1, \\
b_{i,t}=\frac{1}{\lambda}(-(i-\alpha)b_1b_i+b_{i+1}) \ \text{for} \ 1\leq i\leq L+1, \\
b_{L,t}=-\frac{1}{\lambda}(L-\alpha)b_1b_L.
\end{array}
\right.
\end{equation}
We thus want to study the behavior of the solutions of this dynamical system in order to understand the behavior of a real solution close to the manifold of approximate solutions. Writing it in renormalized variables (the renormalized time being defined by \fref{linearized:eq:def s}), the evolution of the $b_i$'s is given by:
\begin{equation}\label{eq:systemdyndesb}
\left\{
\begin{array}{l l}
b_{i,s}=-(i-\alpha)b_1b_i+b_{i+1} \ \ \text{for} \ \ 1\leq i\leq L-1, \\
b_{L,s}=-(L-\alpha)b_1b_L .
\end{array}
\right.
\end{equation}
We show in this section that this dynamical system admits exceptional solutions leading to an explosive scenario, and that the stability of such solutions can be explicitly computed.

\begin{lemma}(Special solutions for the dynamical system:)\label{lem:soldusystemd} Let $\ell$ be an integer such that $\alpha<\ell$. Then\footnote{We forget the dependence with $\ell$ and write $b^e$ to avoid additional notations, as $\ell$ will be fixed throughout the paper} $b^e:]0,+\infty[\rightarrow \mathbb{R}^L$ given by:
\begin{equation} \label{linearized:eq:def ci}
\left\{
\begin{array}{l l}
b^e_i(s)=\frac{c_i}{s^i} \ \ \text{for} \ \ 1\leq i \leq \ell ,\\
b^e_i \equiv 0 \ \ \text{for} \ \ \ell<i ,
\end{array}
\right.
\end{equation}
with the constant $c_i$ given by:
\begin{equation}
c_1=\frac{\ell}{\ell-\alpha} \ \ \text{and} \ \ c_{i+1}=-\frac{\alpha(\ell-i)}{\ell-\alpha}c_i \ \ \text{for} \ \ 1\leq i \leq \ell-1 ,
\end{equation}
is a solution of \fref{eq:systemdyndesb}. Moreover, if the renormalized time $s$ and the scaling satisfy:
$$
\frac{ds}{dt}=\frac{1}{\lambda}, \ s(0)=s_0>0, \ \frac{d}{dt}\lambda =-b_1, \ \lambda(0)=1 ,
$$
then there exists $T>0$ with $s(t)\rightarrow +\infty$ as $t\rightarrow T$, and there holds:
$$
\lambda(t)\underset{t\rightarrow T}{\sim} (T-t)^{\frac{\ell}{\alpha}}
$$
\end{lemma}

We do not write here the proof as it is a direct computation. When dealing with the real equation (NLW), we want these special solutions to persist. A real solution will imply a corrective term "orthogonal" to the manifold $\left(\tilde{Q}_{b,\lambda}\right)_{b,\lambda}$ and a corrective term for the parameters. Therefore, to understand the time evolution of the part of the error on the manifold $\left(\tilde{Q}_{b,\lambda} \right)_{b,\lambda}$, we have to understand the dynamics of \fref{eq:systemdyndesb} close to the special solution $(b^e(s))_{s>0}$.

\begin{lemma}(Linearization around the special trajectories)\label{lem:linearisationsystemdyn}
Let us denote a perturbated solution around $b^e$ by:
\begin{equation}
b_k(s)=b^e_k(s)+\frac{U_k(s)}{s^k}, \ \text{for} \ 1\leq k\leq L ,
\end{equation}
and note $U=(U_1,...,U_L)$ the perturbation. Suppose $b$ is a solution of \fref{eq:systemdyndesb}, then the evolution of $U$ is given by:
\begin{equation}
\partial_t U=\frac{1}{s}A_{\ell}U +O\left( \frac{|U|^2}{s}\right) ,
\end{equation}
with:
\begin{equation}
A_{\ell}=\begin{pmatrix} -(1-\alpha)c_1 +\alpha \frac{\ell-1}{\ell-\alpha} & 1 &  & & & & & & &  \\ . & . & . & & & & & & &  \\ -(i-\alpha)c_i & & \alpha\frac{\ell-i}{\ell-\alpha} & 1 & & & & & &  \\ . &  & & . &. & & & (0)& &  \\ -(\ell-\alpha)c_{\ell} & & & & 0 & 1 & & & &  \\ 0 & & & & & \alpha\frac{-1}{\ell-\alpha} & . & & &  \\  . & & & & & & . & 1 & &  \\ 0 & & (0) & & & & & \alpha\frac{\ell-i}{\ell-\alpha} &. &  \\ . & & & & & & & &  . & 1  \\ 0 & & & & & & & & & \alpha \frac{(\ell-i)}{\ell-\alpha}   \end{pmatrix}
\end{equation}
$A_{\ell}$ is diagonalizable into the matrix $\text{diag}(-1, \frac{2\alpha}{\ell-\alpha},.,\frac{i\alpha}{\ell-\alpha},., \frac{\ell\alpha}{\ell-\alpha},\frac{-1}{\ell-\alpha},.,\frac{\ell-L}{\ell-\alpha})$. We denote the eigenvector associated to the eigenvalue $-1$ by $v_1$ and the eigenvectors associated to the unstable modes $\frac{2\alpha}{\ell-\alpha},...,\frac{\ell \alpha}{\ell-\alpha}$ by $v_2,...,v_{\ell}$. They are a linear combination of the $\ell$ first components only. That is to say there exists a $L\times L$ matrix coding a change of variables:
\be \label{linearized:eq:def P}
P_{\ell}:=\begin{pmatrix} P_{\ell}' & 0 \\ 0 & \text{Id}_{L-\ell} \end{pmatrix},
\ee
with $P_{\ell}'$ an invertible $\ell \times \ell$ matrix and $\text{Id}_{L-\ell}$ the $L-\ell\times L-\ell$ identity such that:
\be \label{linearized:eq:diagonalisation}
P_{\ell}A_{\ell}P_{\ell}^{-1}=\begin{pmatrix} -1 &  & (0)  & & q_1 & & &  \\  & \frac{2\alpha}{\ell-\alpha} &  &  & q_2 & & &  \\  & & . &  & & & &  \\  &  & & \frac{\ell \alpha}{\ell-\alpha} &  q_{\ell} & &  (0) &  \\  & & & & \frac{-\alpha}{\ell-\alpha} & 1 & &  \\   & & & & & . & . &  \\  & & (0) & & & & .& 1   \\  & &  & & & &  & \alpha\frac{\ell-L}{\ell-\alpha}  \end{pmatrix}.
\ee
with $q_i$ being some coefficient $q_i\in \mathbb{R}$ for $1\leq i \leq \ell$.
\end{lemma}

\begin{proof}[Proof of Lemma \ref{lem:linearisationsystemdyn}]
 \underline{Step 1:} Linearization. We compute:
$$
\begin{array}{ r c l}
0&=& b_{k,s}+(k-\alpha)b_1b_k-b_{k+1} \\
&=&\frac{1}{s^{k+1}}[s(U_{k,s}-kU_k+(k-\alpha)c_1 U_k (k-\alpha)c_k U_1-U_{k+1}+O(U_1 U_k)] \\
&=&\frac{1}{s^{k+1}}[s(U_{k,s}+\alpha\frac{k-\ell}{\ell-\alpha}U_k+(k-\alpha)c_k U_1-U_{k+1}+O(U_1 U_k)] .
\end{array}
$$
which gives the expression of $A_{\ell}$.\\

\underline{Step 2:} Diagonalization. We will compute by induction the characteristic polynomial. The case $\ell=3$ can be done by hand. We now assume $\ell\geq 4$ and let:
$$
\mathcal{P}_{\ell}(X)=\text{det}(A_{\ell}-XId) .
$$
We first notice that: $\mathcal{P}_{\ell}(X)=\text{det}(A_{\ell}'-XId)\text{det}(A_{\ell}''-XId)$ where $A_{\ell}'$ stands for the $\ell\times \ell$ matrix on the top left corner, and $A_{\ell}''$ for the $(L- \ell)\times (L-\ell)$ matrix on the bottom right corner:
\begin{equation}
A_{\ell}'=\begin{pmatrix} -(1-\alpha)c_1 +\alpha \frac{\ell-1}{\ell-\alpha} & 1 &  & & (0)  \\ . & . & . & &   \\ -(i-\alpha)c_i & & \alpha\frac{\ell-i}{\ell-\alpha} & 1 &   \\ . & (0) & & . &.   \\ -(\ell-\alpha)c_l & & & & 0  \end{pmatrix} ,
\end{equation}
\begin{equation}
A_{\ell}''=\begin{pmatrix}
-\frac{\alpha}{\ell-\alpha}& 1 & & &(0) \\
 &\cdot & \cdot & & \\
& & -\alpha\frac{i-\ell}{\ell-\alpha}& 1& \\
& & & \cdot& \cdot \\
(0) & & & & -\alpha \frac{L-\alpha}{\ell-\alpha}\\
\end{pmatrix}.
\end{equation}
We have:
\begin{equation}
\text{det}(A_{\ell}''-XId)=\prod_{i=\ell+1}^L (-1)\left(X+\frac{(i-\ell)\alpha}{\ell-\alpha}\right) .
\end{equation}
We write $\mathcal{P}_{\ell}'=\text{det}(A_{\ell}'-XId)$. We develop this determinant with respect to the last row and iterate this process. It gives for $\mathcal{P}_{\ell}'$ an expression of the form:
$$
\begin{array}{r c l}
\mathcal{P}_{\ell}'=(-1)^{\ell+1}(-1)(\ell-\alpha)c_{\ell} &+&(-X)\Bigl{[}(-1)^{\ell}(-1)(\ell-1-\alpha)+(\frac{\alpha}{\ell-\alpha}-X)\\
&\times&\left[(-1)^{\ell-1}(-1)(\ell-2-\alpha)c_{\ell-2}+(\frac{2\alpha}{\ell-\alpha}-X)[...]\right]\Bigr{]} .
\end{array}
$$
We let for $1 \leq i \leq \ell$:
\begin{equation}
A_i:=(-1)^{\ell+2-i}(-1)(\ell+1-i-\alpha)c_{\ell+1-i} ,
\end{equation}
and
\begin{equation}
B_i:=(i-1)\frac{\alpha}{\ell-\alpha}-X .
\end{equation}
We then rewrite:
$$
\mathcal{P}_{\ell}'=A_1+B_1\left(A_2+B_2\left[A_3+B_3\left[...]\right]\right]\right) .
$$
We now let for $1\leq i \leq \ell-1$:
\begin{equation}
C_i:=(-1)^{\ell+1-i}(X(\ell-i-\alpha)c_{\ell-i}+\frac{\ell-\alpha}{i}c_{\ell-i+1}) .
\end{equation}
We have the following relation for $1\leq i \leq \ell-2$:
\begin{equation}
C_i+B_1B_2A_{i+2}=B_{i+2}C_{i+1} .
\end{equation}
Indeed we compute:
$$
\begin{array}{r c l}
C_i+B_1B_2A_{i+2}&=&(-1)^{\ell+1-i}(X(\ell-i-\alpha)c_{\ell-i}+\frac{\ell-\alpha}{i}c_{\ell-i+1}) \\
&&+(-X)(\frac{\alpha}{\ell-\alpha}-X)(-1)^{\ell-i}(-1)(\ell-i-1-\alpha)c_{\ell-i-1} \\
&=& (-1)^{\ell-i}\big{(}-X(\ell-i-\alpha)c_{\ell-i}-\frac{\ell-\alpha}{i}c_{\ell-i+1} \\
&&+X(\frac{(i+1)\alpha}{\ell-\alpha}-i\frac{\alpha}{\ell-\alpha}-X)(\ell-i-1-\alpha)c_{\ell-i-1}\big{)} \\
&=& B_{i+2}(-1)^{\ell-i}(\ell-i-1-\alpha)c_{\ell-i-1} \\
&& + (-1)^{\ell-i}\big{[} -X(\ell-i-\alpha)c_{\ell-i}-\alpha c_{\ell-i}\\
&&-i\frac{\alpha}{\ell-\alpha}X(\ell-i-1-\alpha)(-\frac{\ell-\alpha}{\alpha(i+1)}c_{\ell-i})\big{]} \\
&=& B_{i+2}(-1)^{\ell-i}(\ell-i-1-\alpha)c_{\ell-i-1} \\
&& + (-1)^{\ell-i}c_{\ell-i}(-X(\ell-i-\alpha)+\alpha+\frac{i}{i+1}(\ell-i-1-\alpha)X) \\
&=& B_{i+2}(-1)^{\ell-i}(\ell-i-1-\alpha)c_{\ell-i-1}  + (-1)^{\ell-i}c_{\ell-i}(-\frac{\ell-\alpha}{i+1}X+\alpha)\\
&=& B_{i+2}(-1)^{\ell-i}(\ell-i-1-\alpha)c_{\ell-i-1}  + (-1)^{\ell-i}\frac{\ell-\alpha}{i+1}c_{\ell-i}B_{i+2} \\
&=& B_{i+2}C_{i+1} .
\end{array}
$$
We also have:
$$
A_1+B_1A_2=C_1 .
$$
By iterations we get:
$$
\begin{array}{r c l}
\mathcal{P}_{\ell}'&=&A_1+B_1A_2+B_1B_2A_3+B_1B_2B_3(A_4+B_4(...))\\
&=& C_1+B_1B_2A_3+B_1B_2B_3(A_4+B_4(...)) \\
&=& C_2B_3+B_1B_2B_3(A_4+B_4(...)) = B_3(C_2+B_1B_2(A_4+B_4(...)) \\
&=& B_3(B_4C_3+B_1B_2B_4(A_5+B_5(...))= B_3B_4(C_3+B_1B_2(A_5+B_5(...)) \\
&...& \\
&=& B_3...B_{\ell}(C_{\ell-1}+B_1B_2) .\\
\end{array}
$$
We compute the last polynomial:
$$
C_{\ell-1}+B_1B_2=X(1-\alpha)c_1+\frac{\ell-\alpha}{\ell-1}c_2+(-X)\left(\frac{\alpha}{\ell-\alpha}-X\right)=(X+1)\left(X-\frac{\alpha \ell}{\ell-\alpha}\right) .
$$
So:
$$
\mathcal{P}_{\ell}'=(X+1)\prod_{i=2}^{\ell}\left(\frac{i\alpha}{\ell-\alpha}-X\right) .
$$
This result, together with the result concerning $\mathcal{P}_{\ell}''$ shows that $A_{\ell}$ is diagonalizable and that its eigenvalues are: $(-1,\frac{2\alpha}{\ell-\alpha},...,\frac{\ell\alpha}{\ell-\alpha},\frac{-\alpha}{\ell-\alpha},...,\frac{(L-\ell)\alpha}{\ell-\alpha})$.\\
\\
In addition, from the form of $A_{\ell}$, one sees that the $\ell$ first components do not affect the $L-\ell$ last ones: $\mathbb{P}_{(\ell+1,L)}A\mathbb{P}_{(1,\ell)}=0$ where $\mathbb{P}_{(\ell+1,L)}$ and $\mathbb{P}_{(1,\ell)}$ are the projectors:
$$
\mathbb{P}_{(\ell+1,L)}(U_1,...,U_L)=(0,...,0,U_{\ell+1},...,U_L), \ \mathbb{P}_{(1,\ell)}(U_1,...,U_L)=(U_1,...,U_{\ell},0,...,0).
$$ 
This gives the last result stated in the lemma. The $v_i$'s are a linear combination of the $\ell$ first components only.
\end{proof}


\section{The trapped regime}

In this section we are considering a real solution of (NLW). We fix $1\ll L$ odd and $\alpha<\ell$. Our aim is to show that the approximate solution $(\tilde{\bos{Q}}_{b^e})_{\frac{1}{\lambda^e}}$ constructed in the last section does persist. That is to say that there exists an orbit of the (NLW) equation that stays asymptotically (with respect to renormalized time $s$) close to the family of special approximate solutions $(\tilde{\bos{Q}}_{b^e})_{\frac{1}{\lambda}}$. Note that we do not prescribe in advance the behavior of the scaling $\lambda$, but it will be shown to have the same asymptotical behavior as $\lambda^e $.\\
\\
In order to do that, we need to understand how the full dynamics affects the approximate one we exhibited in the last section. We decompose a true solution under the form $\bos{u}(t)=(\tilde{\bos{Q}}_b+\bos{\varepsilon})_{\frac{1}{\lambda}}$. We aim at estimating the contribution of the error $\bos{\varepsilon}$ on the parameters dynamics, and at estimating the size of $\bos{\varepsilon}$ in adapted norms.\\
\\
The special approximate solutions $(\tilde{\bos{Q}}_{b^e})_{\frac{1}{\lambda}}$ for $\lambda\sim \lambda^e$,  generate a reasonable error term, because as $|b_i^e||\lesssim s^{-i}\approx (b_1^e)^i$ the estimates on the error term $\bos{\psi}_b$ in Proposition \ref{pr:profilapprochecoupe} apply. But they are not stable along the unstable directions $(v_2,...,v_{\ell})$ (defined in Lemma \ref{lem:linearisationsystemdyn}), and if the parameters $b_i$'s move too much, the error term in the approximate dynamics grows too big, consequently making a control over $\bos{\varepsilon}$ impossible. Therefore we cannot work close to the full approximate manifold $\left(\tilde{\bos{Q}}_{b,\lambda} \right)_{b,\lambda}$: we are restricted to work close to the subset of these approximate trajectories $\left( \tilde{\bos{Q}}_{b^e(s),\lambda} \right)_{s>0,\lambda>0}$. We work in a neighborhood of these approximate trajectories, study all the real trajectories starting from that neighborhood, and show that at least one must stay in that neighborhood for all time. We make a proof based on a bootstrap technique. We in particular argue "forward" in time what allows us to measure precisely the stabilities and instabilities.\\
\\
The fact that staying in an appropriate neighborhood of a special approximate solution leads to a blow-up, whose blow-up rate and asymptotic behavior can be computed, will be shown in the next section.


\subsection{Setting up the bootstrap}

We are now going to define in which neighborhood of the family of approximate solutions $\left(\tilde{\bos{Q}}_{b^e(s),\frac{1}{\lambda}}\right)_{s,\lambda}$ we want to work. We start by defining how we decompose our solution into the sum $\bos{u}=(\tilde{\bos{Q}}_{b}+\bos{\varepsilon})_{\frac{1}{\lambda}} $. After that we describe the neighborhood and state the main Proposition of the paper claiming the existence of an orbit staying inside.


\subsubsection{Projection onto the approximate solutions manifold}

Close to $\bos{Q}$, the manifold $\left(\bos{Q}_{b,\lambda} \right)_{b,\lambda}$ is tangent to the vector space $Span(\bos{T}_i)$. It is consequently appealing to ask $\langle \bos{T}_i,\bos{\varepsilon}\rangle=0$ for all $i$. However, the $\bos{T}_i$'s are not in appropriate functional spaces, and in particular cannot be used to generate orthogonality conditions. Instead, we will create a sequence of profiles with compact support that approximate such orthogonality conditions. We let the adjoint of $\bos{H}$ be the operator:
\be
\bos{H}^*= \begin{pmatrix} 0 & \mathcal{L} \\
-1 & 0
\end{pmatrix}.
\ee
We have the following relations: $\langle \bos{H} \bos{u},\bos{v}\rangle=\langle \bos{u},\bos{H}^*\bos{v} \rangle $, and
\be
\bos{H}^{*2i}=  \begin{pmatrix} (-1)^i\mathcal{L}^i & 0 \\
0 & (-1)^i\mathcal{L}^i
\end{pmatrix}, \ \bos{H}^{*(2i+1)}=  \begin{pmatrix} 0 & (-1)^{i}\mathcal{L}^{i+1}  \\
 (-1)^{i+1}\mathcal{L}^{i} &0
\end{pmatrix}.
\ee
We recall that $L$ is an odd, large integer. We let $M$ be a large constant, and define:
\begin{equation}\label{thetrapped:eq:def:phi M}
\bos{\Phi}_M=\sum_{p=0}^Lc_{p,M} \bos{H}^{*p}(\chi_M \bos{\Lambda} \bos{Q}) ,
\end{equation}
with the constants $c_{p,M}$ defined by:
\begin{equation}
c_{0,M}=1 \ \ \text{and} \ \ c_{k,M}=(-1)^{k+1} \frac{\sum_{p=0}^{k-1} c_{p,M} \langle \bos{H}^{*p}(\chi_M \bos{\Lambda} \bos{Q}),\bos{T}_k \rangle}{\langle \chi_M \bos{\Lambda} \bos{Q}, \bos{\Lambda} \bos{Q} \rangle} .
\end{equation}

\begin{lemma}(Generation of orthogonality conditions:)\label{lem:conditiondortho}
The profile $\bos{\Phi}_M$ is located on the first coordinate:
\begin{equation} \label{thetrapped:eq:localisation PhiM}
\bos{\Phi}_M=\begin{pmatrix}
\Phi_M \\ 0
\end{pmatrix} ,
\end{equation}
because for $1\leq k=2i+1\leq L$ an odd integer one has $c_{k,M}=0$. Moreover the following bounds hold:
\begin{equation} \label{thetrapped:eq:proprietes PhiM}
\left\{ \begin{array}{l l}
 | \langle \Phi_M , \Lambda Q \rangle |\sim cM^{2k_0+2\delta_0} ,\\
|c_{p,M}| \leq CM^p ,\\
\int \Phi_M^2 \leq CM^{2k_0+2\delta_0}.
\end{array} \right.
\end{equation}
for two positive constants $c,C>0$. In addition, the following orthogonality conditions are met for $1\leq j \leq L$ and $i\in \mathbb{N}$:
\begin{equation} \label{thetrapped:eq:orthogonalite PhiM}
\langle \bos{\Phi}_M, \bos{H}^i \bos{T}_j \rangle = \langle \chi_M \bos{\Lambda} \bos{Q}, \bos{\Lambda} \bos{Q} \rangle \delta_{i,j} .
\end{equation}
\end{lemma}

\begin{proof}[Proof of Lemma \ref{lem:conditiondortho}]
\underline{Proof of the orthogonality conditions:}
$$
\begin{array}{r c l}
\langle \bos{\Phi}_M , \bos{\Lambda} \bos{Q}\rangle &=& c_{0,M} \langle \chi_M \bos{\Lambda} \bos{Q}, \bos{\Lambda} \bos{Q} \rangle + \sum_{p=1}^L c_{p,M} \langle \chi_M \bos{\Lambda} \bos{Q} , \bos{H}^{p}(\bos{\Lambda} \bos{Q})\rangle \\
&=& \langle \chi_M \bos{\Lambda} \bos{Q}, \bos{\Lambda} \bos{Q} \rangle \\
&\sim & cM^{d-2\gamma} ,
\end{array}
$$
$c>0$, from the asymptotic $\Lambda^{(1)} Q \sim \frac{c'}{y^{\gamma}}$, $c'\neq 0$. This proves the first property of \fref{thetrapped:eq:proprietes PhiM}. The orthogonality with respect to the $\bos{T}_i$'s is created on purpose by the definition of the constants $c_{p,M}$:
$$
\langle \bos{\Phi}_M, \bos{T}_k \rangle = \sum_{p=0}^{k-1} c_{p,M}\langle \bos{H}^{*p}(\chi_M \bos{\Lambda} \bos{Q}),\bos{T}_k \rangle + c_{k,M} \langle \chi_M \bos{\Lambda} \bos{Q}, \bos{H}^{k}\bos{T}_k \rangle =0 .
$$
Hence by duality:
$$
\langle \bos{\Phi}_M, \bos{H}^{i}\bos{T}_j \rangle = \langle \chi_M \bos{\Lambda} \bos{Q}, \bos{\Lambda} \bos{Q} \rangle \delta_{i=j} .
$$
This proves \fref{thetrapped:eq:orthogonalite PhiM}. \\

\underline{Bounds on the constants:}
We notice by induction that $c_{p,M}=0$ for $p$ odd. This implies that $\Phi_M^{(2)}=0$. We prove the estimate on the constants $c_{p,M}$ by induction. Since $c_0=1$, the estimation is true for $k=0$. We assume now $k$ to be even. By definition we have:
$$
\begin{array}{ r c l}
|c_{k,M}|&=&\frac{|\sum_{p=0}^{k-1}\langle \bos{H}^{*p}(\chi_M \bos{\Lambda} \bos{Q}),\bos{T}_k \rangle|}{|\langle \chi_M \bos{\Lambda} \bos{Q},\bos{\Lambda} \bos{Q} \rangle|} \\
&\leq & CM^{-d+2\gamma} \sum_{p=0}^{k-1} |c_{p,M}||\langle \bos{H}^{*p}(\chi_M \bos{\Lambda} \bos{Q}),\bos{T}_k \rangle | \\
&=&CM^{-d+2\gamma} \sum_{p=0}^{k-1} |c_{p,M}||\langle \chi_M \bos{\Lambda} \bos{Q},\bos{T}_{k-p} \rangle | .
\end{array}
$$
In the sum, for $k-p$ odd this term equals $0$. So we have $k-p\geq 2$. Using the asymptotics $\Lambda^{(1)} Q \sim cy^{-\gamma}$ and $T_{k-p} \sim cy^{-\gamma+k-p}$ the integral in the scalar product is divergent and we estimate:
$$
|\langle \chi_M \bos{\Lambda} \bos{Q},\bos{T}_{k-p} \rangle| \sim cM^{d-2\gamma+k-p} .
$$
Using the induction hypothesis we get:
$$
M^{-d+2\gamma}|c_{p,M}||\langle \bos{H}^{*p} (\chi_M \bos{\Lambda} \bos{Q} , \bos{T}_k \rangle \leq CM^k ,
$$
and so the estimate is true for $c_{k,M}$. We have proven the second assertion of \fref{thetrapped:eq:proprietes PhiM}.\\

\underline{$L^2 estimate$:} $\int |\Phi_M|^2$ is a finite sum of terms of the following form enjoying the bound (from the asymptotic \fref{eq:degenerescencesoliton}):
$$
\begin{array}{r c l}
&|\langle c_{p_1,M}\bos{H}^{*p_1}(\chi_M \bos{\Lambda} \bos{Q}),c_{p_2,M}\bos{H}^{*p_2}(\chi_M \bos{\Lambda} \bos{Q}) \rangle| \\
\leq& CM^{p_1+p_2}|\langle \mathcal{L}^{\frac{p_1+p_2}{2}}(\chi_M \Lambda Q),\chi_M \Lambda Q \rangle| \leq&  CM^{-2\gamma+d} ,
\end{array}
$$
because we assumed $\frac{d}{2}-\gamma$ not to be an integer. It implies the last bound in \fref{thetrapped:eq:proprietes PhiM}
\end{proof}


\subsubsection{Modulation:}
\label{thetrapped:subsubsection:modulation}
We want to decompose a function $\bos{u}$ close to $\bos{Q}_{\lambda}$ as a unique sum $\bos{u}=(\bos{Q}_{b}+\bos{\varepsilon})$, with $\bos{\varepsilon}$ "orthogonal" to the manifold $(\bos{Q}_{b,\lambda})_{b,\lambda}$. We make the following change of variable for the parameter $b$: $\tilde{b}_1:=(b_1,0,...,0)$ and $\tilde{b}_i=(b_1,0...,0,b_i,0,...,0)$ and introduce the application $\phi:(\lambda,b)\mapsto (\langle \tilde{\bos{Q}}_b,\bos{H}^{*i}\bos{\Phi}_M\rangle)_{0\leq i \leq L}$. We denote by $D\phi$ the jacobian matrix of $\phi$ at the point $(1,(0,...,0))$ in the $(\lambda,\tilde{b})$ basis. From the properties \fref{thetrapped:eq:proprietes PhiM} and \fref{thetrapped:eq:orthogonalite PhiM} of the profile $\bos{\Phi}_M$ that we previously established, one has:
$$
D\phi= \langle \bos{\Lambda Q},\chi_M \bos{\Lambda Q}\rangle
\begin{pmatrix} 1 & 0 & & & (0) \\
&1&1&& \\
&& 1 &.& \\
&&&.&1 \\
(0)&&&&1 
\end{pmatrix}.
$$
This proves that $\phi$ is a local diffeomorphism around $(1,(0,...,0))$. The implicit function theorem gives for $\bos{u}$ close enough\footnote{the closeness assumption is described in the next subsection and is compatible with what we are saying here.} to $\bos{Q}$ the existence of a unique decomposition:
\begin{equation}
\bos{u}=(\tilde{\bos{Q}}_b)_{\frac{1}{\lambda}}+\bos{w}=(\tilde{\bos{Q}}_b)+\bos{\varepsilon})_{\frac{1}{\lambda}} ,
\end{equation}
with $\bos{\varepsilon}$ verifying the $L+1$ orthogonality conditions:
\begin{equation}\label{eq:condition d'ortho pour epsilon}
\langle \bos{\varepsilon},\bos{H}^{*i}\bos{\Phi}_M\rangle=0, \ for \ 0\leq i \leq M .
\end{equation}
Hence for a real solution to (NLW) starting close enough to $\bos{Q}$, and by scaling argument, we have as long as $\bos{u}$ is close enough to $\bos{Q}_{\lambda}$ a decomposition:
\begin{equation}
\bos{u}=(\tilde{\bos{Q}}_{b(t)}+\bos{\varepsilon})_{\frac{1}{\lambda(t)}} ,
\end{equation}
with $b$ and $\lambda$ being $C^1$ in time\footnote{As the dynamic will be smooth enough.}, and $\bos{\varepsilon}$ satisfying \fref{eq:condition d'ortho pour epsilon}.


\subsubsection{Adapted norms:}

We quantify the smallness of $\bos{\varepsilon}$ through the following norms:
\begin{itemize}
\item[(i)] \emph{High order Sobolev norm adapted to the linearized operator:} Remember that $s_L=L+k_0+1$ and that the $k$-th adapted derivative of a function $f$, $f_k$, is defined in \fref{linearized:eq:def derivees adaptees}. We define:
\begin{equation} \label{thetrapped:eq:def mathcalEsL}
\begin{array}{r c l}
\mathcal{E}_{s_L}&:=&\int |\varepsilon_{k_0+L+1}^{(1)}|^2+\int |\varepsilon_{k_0+L}^{(2)}|^2 \\
&=& \int \varepsilon^{(1)}\mathcal{L}^{k_0+L+1}\varepsilon^{(1)} + \int \varepsilon^{(2)}\mathcal{L}^{k_0+L}\varepsilon^{(2)} ,
\end{array}
\end{equation}which is coercive thanks to the result of Lemma \ref{annexe:lem:coercivite des normes adaptees}. In particular:
$$
\mathcal{E}_{s_L}\gtrsim \parallel \bos{\varepsilon} \parallel_{\dot{H}^{s_L}\times \dot{H}^{s_L-1}}^2 .
$$
As we will see later on in this paper, a local part of this norm will have to be treated separately. Let $N>0$, we define:
\begin{equation} \label{thetrapped:eq:def mathcalEsLloc}
\mathcal{E}_{s_L,\text{loc}}:=\int_{y\leq N} |\varepsilon_{k_0+L+1}^{(1)}|^2+\int_{y\leq N} |\varepsilon_{k_0+L}^{(2)}|^2 .
\end{equation}

\item[(ii)] \emph{Low order slightly supercritical Sobolev norm:} We choose a real number $\sigma$ such that:
\begin{equation} \label{thetrapped:eq:def sigma}
0< \sigma -s_c\ll 1 ,
\end{equation}
and we define:
\begin{equation} \label{thetrapped:eq:def mathcalEsigma}
\mathcal{E}_{\sigma}:=\int |\nabla^{\sigma} \varepsilon^{(1)}|^2+\int |\nabla^{\sigma-1} \varepsilon^{(2)}|^2 .
\end{equation}
\end{itemize}


\subsubsection*{Estimates we want to bootstrap and main Proposition:}

Let $s_0$ denote a large enough real number $s_0\gg 1$. We recall the definition of the renormalized variables:
\begin{equation} \label{thetrapped:eq:def s}
y=\frac{r}{\lambda(t)}, \ s(t)=s_0+\int_0^t \frac{d \tau}{\lambda(\tau)} .
\end{equation}
We introduce notations for the decomposition of the solution in both real and renormalized time:
\begin{equation}
\bos{u}=\tilde{\bos{Q}}_{b(t),\frac{1}{\lambda(t)}}+\bos{w}=(\tilde{\bos{Q}}_{b(s)}+\bos{\varepsilon}(s))_{\frac{1}{\lambda(s)}} .
\end{equation}
The parameters $b_i$ are chosen as a perturbation of the solution $b^e$:
\begin{equation}
b_i(s)=b^e_i(s)+\frac{U_i(s)}{s^i} .
\end{equation}
To treat the stable and unstable modes separately, we employ the change of variables coded by the matrix $P_{\ell}$ defined by \fref{linearized:eq:def P}. Instead of $U_1,...,U_{\ell}$ we consider:
\begin{equation} \label{thetrapped:eq:def Vi}
V_i:=(P_{\ell} U)_i \ \ \text{for} \ 1\leq i \leq \ell .
\end{equation}

We assume initially\footnote{the choice of the constants is done in the next proposition.}:
\begin{itemize}
\item[(i)] Smallness of the unstable modes: Let $0<\tilde{\eta}$ be a constant to be defined later.
\begin{equation}\label{eq:estimation modes instables}
(V_2(s_0),...,V_{\ell}(s_0)) \in \mathcal{B}^{l-1}\left(\frac{1}{s_0^{\tilde{\eta}}}\right) .
\end{equation}
\item[(ii)] Smallness of the stable modes\footnote{the $\frac{1}{10}$ is arbitrary: we just want the initial condition to be smaller than the information we want to bootstrap, see next proposition.}:
\begin{equation}\label{eq:estimation initiale des modes stables}
V_1(s_0)\leq\frac{1}{10s_0} , \ \text{and} \ |b_i(s_0)|\leq \frac{\epsilon_i}{10s_0^{(i-\alpha)c_1}} \ \text{for} \ \ell+1 \leq i \leq L .
\end{equation}
\item[(iii)] Smallness of the initial perturbation in high and low Sobolev norms:
\begin{equation}\label{eq:estimation initiale de epsilon0}
\mathcal{E}_{s_L}(s_0)+\mathcal{E}_{\sigma}(s_0)< \frac{1}{s_0^{2L+2+2(1-\delta_0)(1+\eta)}} .
\end{equation}
\item[(iv)] Normalization: up to a fix rescaling, we may always assume:
\begin{equation}
\lambda(s_0)=1 .
\end{equation}
\end{itemize}

\begin{proposition}\emph{(Existence of an initial datum for which the solution stays in yhe trapped regime:)}\label{prop:bootstrap}
There exists universal constants for the analysis:
\begin{equation}
\begin{array}{l l}
0<\eta=\eta(d,p,L)\ll 1,
M=M(d,p,L)\gg 1, \ N=N(d,p,L,M)\gg 1 ,\\
K_i=K_i(d,p,L,M)\gg 1, \ \text{for} \ i=1,2, \ s_0=s_0(l,d,p,L,M,K)\gg 1,
\end{array}
\end{equation}
and constants for smallness:
\begin{equation}
0<\epsilon_i \ for \ \ell+1\leq i \leq L, \ 0<\epsilon_1, \ \text{and} \ 0<\tilde{\eta} \ \ \ (all \ \ll 1),
\end{equation}
such that the following fact holds. Given $\bos{\varepsilon}(s_0)$ satisfying \fref{eq:condition d'ortho pour epsilon}, \fref{eq:estimation initiale de epsilon0}, and stable parameters $V_1(s_0)$, $(b_{\ell+1}(s_0),...,b_L(s_0))$ satisfying \fref{eq:estimation initiale des modes stables}, there exists initial conditions for the unstable parameters $(V_2(s_0),...V_{\ell}(s_0))$ satisfying \fref{eq:estimation modes instables} for which the solution to (NLW) with initial data $\tilde{\bos{Q}}_{b(s_0)}+\bos{\varepsilon}(s_0)$ with:
$$
\begin{array}{r c l}
b(s_0)&=&b^e(s_0)+(0,...,0,b_{\ell+1}(s_0),...,b_L(s_0))\\
&&+\left(\frac{(P_{\ell}^{-1}(V_1(s_0),...,V_{\ell}(s_0),0,...,0))_1}{s_0},...,\frac{(P_{\ell}^{-1}(V_1(s_0),...,V_{\ell}(s_0),0,...,0))_{\ell}}{s_0^{\ell}},0,...,0 \right) ,
\end{array}
$$ admits the following bounds for all $s\geq s_0$:
\begin{itemize}
\item[-]\emph{control of the part on the approximate profiles manifold:} for the unstable modes:
\begin{equation}\label{eq:bootstrap modes instables}
(V_2(s),...V_{\ell}(s))\in \mathcal{B}^{\ell-1}\left(\frac{1}{s^{\tilde{\eta}}}\right) .
\end{equation}
for the stable modes:
\begin{equation}\label{eq:bootstrap modes stables}
|V_1(s)|\leq \frac{1}{s^{\tilde{\eta}}}, \ |b_k(s)|\leq \frac{\epsilon_k}{s^{k+\tilde{\eta}}}, \ \text{for} \ \ell+1 \leq k \leq L.
\end{equation}
\item[-]\emph{control of the error term:}
\begin{equation}\label{eq:bootstrap estimations sur epsilon}
\begin{array}{l l}
\mathcal{E}_{s_L}(s) \leq K_1 b_1^{2L+2(1-\delta_0)(1+\eta)} ,\\
\mathcal{E}_{\sigma}(s) \leq K_2 b_1^{2(\sigma-s_c)\frac{\ell}{\ell-\alpha}} .
\end{array}
\end{equation}
\end{itemize}
\end{proposition}

To prove Proposition \ref{prop:bootstrap} we argue by contradiction and suppose that for all initial data of the unstable modes $(V_2,...,V_{\ell})\in \mathcal{B}^{\ell-1}(s_0^{-\tilde{\eta}})$, the conditions are not met for all time:
\begin{equation}
\begin{array}{r c l}
s^*&=& s^*(\varepsilon(s_0),s_0,V_1(s_0),...,V_{\ell}(s_0),b_{\ell+1}(s_0),...,b_L(s_0)) \\
&=& sup\{ s\geq s_0 \ \text{such} \ \text{that} \ \fref{eq:bootstrap estimations sur epsilon}, \  \fref{eq:bootstrap modes instables} \ \text{and} \ \fref{eq:bootstrap modes stables} \ \text{hold} \ \text{on} [s_0,s] \} \\
&<& +\infty .
\end{array}
\end{equation}

By continuity of the flow and the smallness of the initial perturbation, we know that $s^*>0$. We perform a three steps reasoning to prove the contradiction:

\begin{itemize}
\item[(i)] First we show that as long as $\bos{\varepsilon}$ is controlled by the estimates \fref{eq:bootstrap estimations sur epsilon}, it does not perturb too much the dynamical system \fref{eq:systemdyndesb}. That is to say we have a sufficient control over the evolution of the $b_i$'s to show that the perturbation $U$ of the trajectory $b^e$ evolves according to the linearisation at the leading order.

\item[(ii)] (i) has given us control over the part of the solution on the approximate manifold, this allows us to compute the evolution of the scale $\lambda$. Under the bootstrap conditions we know the size of the error term $\tilde{\bos{\psi}}_b$ generated by the approximate dynamics. Once we know the behavior of $\tilde{\bos{\psi}}_b$ and $\lambda$, we can look for better informations about $\bos{\varepsilon}$. Indeed we apply an energy method and find out that we control the time evolution of $\mathcal{E}_{s_L}$ and $\mathcal{E}_{\sigma}$. As $\bos{\varepsilon}$ is a stable perturbation, we find that we have in fact a better estimate for this term: $\bos{\varepsilon}$ is smaller than the estimate given by \fref{eq:bootstrap estimations sur epsilon}. Hence at time $s^*$ we have:
\begin{equation}
\begin{array}{l l}
\mathcal{E}_{s_L}(s^*) < K_1 b_1^{2L+2(1-\delta_0)(1+\eta)} ,\\
\mathcal{E}_{\sigma}(s^*) < K_2 b_1^{2(\sigma-s_c)\frac{\ell}{\ell-\alpha}}.
\end{array}
\end{equation}
This implies that the exit of the trapped regime is only when the parameters do not satisfy the estimates \fref{eq:bootstrap modes stables} and \fref{eq:bootstrap modes stables} anymore.

\item[(iii)] With the estimates we have found regarding the parameters dynamics in (i) we are able to say that this is impossible. Indeed, the stable parameters cannot go away because their dynamics is stable. It is possible for some unstable parameters to go away, but they cannot all leave the ball $\mathcal{B}^{\ell-1}\left( \frac{1}{(s^*)^{\tilde{\eta}}} \right) $ in finite time. We have seen in Lemma \ref{lem:linearisationsystemdyn} that the $V_i$'s for $2\leq i\leq \ell $ evolve as a linearized system around a repulsive equilibrium. The true dynamics, adding a small error term to their time evolution, preserves this structure. The dynamics in our case cannot expulse all the orbits away from the equilibrium point: we will show how in that case it would be a contradiction to Brouwer's fixed point theorem.
\end{itemize}


\subsection{Evolution equations for $\bos{\varepsilon}$ and $\bos{w}$:}

We recall that we are studying a solution under the form: 
$$
\bos{u}=\tilde{\bos{Q}}_{b(t),\frac{1}{\lambda (t)}}+\bos{w}=(\tilde{\bos{Q}}_{b(s)}+\bos{\varepsilon} (s))_{\frac{1}{\lambda(s)}} ,
$$
where $\tilde{\bos{Q}}_b$ is defined by \fref{eq:def:Qbtilde} and $\bos{\varepsilon}$ satisfies the orthogonality conditions \fref{eq:condition d'ortho pour epsilon}, this decomposition being explained in Subsubsection \ref{thetrapped:subsubsection:modulation}.
The evolution of $\bos{\varepsilon}$ and $\bos{w}$ is given by:
\begin{equation}\label{eq:evolution epsilon}
\begin{array}{l l l l}
\partial_s \bos{\varepsilon} - \frac{\lambda_s}{\lambda} \bos{\Lambda} \bos{\varepsilon} +\bos{H}(\bos{\varepsilon})&=& -\bos{Mod}(t)+(\frac{\lambda_s}{\lambda}+b_1)\bos{\Lambda} \tilde{\bos{Q}}_b  - \tilde{\bos{\psi}}_b & \\
&& + \bos{F}(\tilde{\bos{Q}}_b+\bos{\varepsilon})-\bos{F}(\tilde{\bos{Q}}_b)+\bos{H}_b(\bos{\varepsilon}) &\left\} \mathbf{:=  NL(\varepsilon)}\right. \\
&& +\bos{H}(\bos{\varepsilon})-\bos{H}_b(\bos{\varepsilon})& \left\} \mathbf{:= L(\varepsilon)}\right. ,
\end{array}
\end{equation}
where $\bos{H}_b$ denotes the linearization close to $\tilde{\bos{Q}}_b$:
\begin{equation}
\bos{H}_b:=\begin{pmatrix}
0 & -1 \\
-\Delta -p\tilde{Q}_b^{p-1} & 0
\end{pmatrix} ,
\end{equation}
and:
\begin{equation}\label{eq:evolution W}
\begin{array}{l l l l}
\partial_t \bos{w}+\bos{H}_{\frac{1}{\lambda}} \bos{w}&=&\frac{1}{\lambda}(-\bos{Mod}(t)+(\frac{\lambda_s}{\lambda}+b_1)\bos{\Lambda} \tilde{\bos{Q}}_b)_{\frac{1}{\lambda}} -\frac{1}{\lambda}\tilde{\bos{\psi}}_{b,\frac{1}{\lambda}} &\\
&& + \bos{F}(\tilde{\bos{Q}}_{b,\frac{1}{\lambda}}+\bos{w})-\bos{F}(\tilde{\bos{Q}}_{b,\frac{1}{\lambda}})+\bos{H}_{b,\frac{1}{\lambda}} \bos{w} & \left\} \mathbf{:= NL(w)} \right. \\
&& + \bos{H}_{\frac{1}{\lambda}}\bos{w}-\bos{H}_{b,\frac{1}{\lambda}}\bos{w} & \left\} \mathbf{:= L(w)} \right. ,
\end{array}
\end{equation}
where:
\begin{equation}
\bos{H}_{\frac{1}{\lambda}}:=\begin{pmatrix}
0 & -1 \\
-\Delta -p (Q_{\frac{1}{\lambda}})^{p-1} & 0
\end{pmatrix}, \ \text{and} \ \bos{H}_{b,\frac{1}{\lambda}}:= \begin{pmatrix}
0 & -1 \\
-\Delta -p(\tilde{Q}_{b,\frac{1}{\lambda}})^{p-1} & 0
\end{pmatrix} .
\end{equation}
We notice that the $\bos{NL}$ and $\bos{L}$ terms are situated on the second coordinate:
\be \label{thetrapped:eq:localisation NL et L}
\bos{NL}(\bos{\varepsilon})=\begin{pmatrix}0 \\ NL(\varepsilon) \end{pmatrix}, \ \bos{NL}(\bos{w})=\begin{pmatrix}0 \\ NL(w) \end{pmatrix}, \ \bos{L}(\bos{\varepsilon})=\begin{pmatrix}0 \\ L(\varepsilon) \end{pmatrix}, \ \bos{L}(\bos{w})=\begin{pmatrix}0 \\ L(w) \end{pmatrix}.
\ee
We let the new modulation term that now includes the scale change be:
\begin{equation}
\tilde{\bos{Mod}(t)}:= \bos{Mod}(t)-\left(\frac{\lambda_s}{\lambda}+b_1\right)\bos{\Lambda} \tilde{\bos{Q}}_b .
\end{equation}


\subsection{Modulation equations} In this section we compute the influence of $\bos{\varepsilon}$ on the equations governing the evolution of the parameters $\lambda$ and $b$.

\begin{lemma}[Modulation estimates] \label{lem:modulation}
Assume that all the constants involved in Proposition \ref{prop:bootstrap} are fixed in their range\footnote{It means that, for example, if we wrote $0<C\ll 1$ that $C$ is fixed very small}, except $s_0$. Then for $s_0$ large enough there holds the bounds for $s_0\leq s<s^*$:
\begin{equation} \label{thetrapped:eq:modulation leq L-1}
\begin{array}{r c l}
&&\left|\frac{\lambda_s}{\lambda}+b_1\right| +\sum_{i=1}^{L-1}|b_{i,s}+(i-\alpha)b_1b_i+b_{i+1} |  \\ 
&\leq& C(M)b_1^{L+3}+C(L,M)b_1\sqrt{\mathcal{E}_{s_L}} ,
\end{array}
\end{equation}
\begin{equation} \label{thetrapped:eq:premiere modulation L}
|b_{L,s}+(L-\alpha)b_1 b_L|\leq C(M)\sqrt{\mathcal{E}_{s_L}}+C(M)b_1^{L+3} .
\end{equation}
\end{lemma}

\begin{remark}\label{thetrappedregime:modulationequation:re:premiere estimation sur bl insuffisante}
Under the assumption on the smallness of $\bos{\varepsilon}$ \fref{eq:bootstrap estimations sur epsilon} This implies in particular that:
$$
\frac{\lambda_s}{\lambda}=-b_1+O(b_1^2) 
$$
and
$$
b_{i,s}=-(i-\alpha)b_1b_i+b_{i+1}+O(b_1^{i+2}) 
$$
for $1\leq i \leq L-1$. If we had also $b_{L,s}=-(L-\alpha)b_1b_L+O(b_1^{L+1+c})$ for a small constant $c>0$, this would be enough to conclude that the dynamics of the parameters is given at the first order by \fref{eq:systemdyndesb}. Unfortunately this last condition is not met. We will see how to skirt this problem in the next Lemma \ref{trappedregime:improvedmodulation:lem:improvedmodulation}.
\end{remark}

\begin{proof}[Proof of Lemma \ref{lem:modulation}] We let:
\begin{equation}
D(t)=\left| \frac{\lambda_s}{\lambda}+b_1 \right|+\sum_{i=1}^L |b_{i,s}+(i-\alpha)b_1b_i-b_{i+1}| .
\end{equation}

For $0\leq i \leq L$ we take the sclar product of \fref{eq:evolution epsilon} with $\bos{H}^{*i}\bos{\Phi}_M$:
\be \label{thetrapped:eq:modulation}
\begin{array}{r c l}
\langle \tilde{\bos{Mod}}(t), \bos{H}^{*i}\bos{\Phi}_M \rangle&=& \langle -\bos{H}(\bos{\varepsilon}), \bos{H}^{*i}\bos{\Phi}_M \rangle  + \langle \frac{\lambda_s}{\lambda}\bos{\Lambda} \bos{\varepsilon},\bos{H}^{*i}\bos{\Phi}_M\rangle  - \langle \tilde{\bos{\psi}}_b,\bos{H}^{*i}\bos{\Phi}_M\rangle \\
&& + \langle \bos{NL}(\bos{\varepsilon}),\bos{H}^{*i}\bos{\Phi}_M\rangle  + \langle \bos{L}( \bos{\varepsilon}),\bos{H}^{*i}\bos{\Phi}_M\rangle .
\end{array}
\ee

\underline{Step 1:} \emph{law for $\lambda$.} We take $i=0$ in the preceding equation \fref{thetrapped:eq:modulation} and compute all the terms. As $\bos{\Phi}_M$ is located on the first coordinate, see \fref{thetrapped:eq:proprietes PhiM}, it gives:
\be \label{thetrapped:modulation:lambda NL L}
\langle \bos{NL}(\bos{\varepsilon}),\bos{\Phi}_M\rangle  = \langle \bos{L}( \bos{\varepsilon}),\bos{\Phi}_M\rangle =0.
\ee
$\bos{\Phi}_M$ is of compact support in $|y|\leq 2M$ and situated on the first coordinate. For $b_1$ small enough one has $\tilde{\bos{\psi}}_b(y)=\bos{\psi}_b(y)$, and $\bos{\psi}_b $ is situated on the second coordinate from \fref{linearized:eq:position psib}. Hence:
\be \label{thetrapped:modulation:lambda psib}
\langle \tilde{\bos{\psi}}_b, \bos{\Phi}_M \rangle =0.
\ee
The linear term is equal to $0$ because of the orthogonality conditions \fref{eq:condition d'ortho pour epsilon}:
\be \label{thetrapped:modulation:lambda lineaire}
\langle -\bos{H}(\bos{\varepsilon}),\bos{\Phi}_M \rangle=0 .
\ee
The left hand side, the modulation term, is the one catching the evolution of $\lambda_s$:
\be \label{thetrapped:modulation:lambda mod}
\begin{array}{ r c l}
\langle \tilde{\bos{Mod}}(t), \bos{\Phi}_M \rangle&=& (\frac{\lambda_s}{\lambda}+b)\langle \bos{\Lambda} \tilde{\bos{Q}_b}, \bos{\Phi}_M \rangle \\
&& + \sum_{i=1}^{L} (b_{i,s}+(i-\alpha)b_1b_i-b_{i+1})\langle \bos{T}_i+\sum_{j=i+1}^L \frac{\partial \bos{S}_j}{\partial b_i},\bos{\Phi}_M \rangle \\
&=& (\frac{\lambda_s}{\lambda}+b_1)\langle \bos{\Lambda} \bos{Q}, \bos{\Lambda} \bos{Q} \rangle +O(b_1D(t)).
\end{array}
\ee
We now estimate the scaling term:
$$
\begin{array}{r c l}
|\langle \frac{\lambda_s}{\lambda}\bos{\Lambda} \bos{\varepsilon}, \bos{\Phi}_M\rangle|&\leq& |\frac{\lambda_s}{\lambda}+b_1||\langle \Lambda^{(1)} \varepsilon^{(1)} , \Phi_M\rangle| + b_1|\langle \Lambda^{(1)} \varepsilon^{(1)} , \Phi_M\rangle| \\
&\leq& (b_1+D(t))|\parallel \Lambda \varepsilon^{(1)} \parallel_{L^2(\leq M)}\parallel \Phi_M\parallel_{L^2} .
\end{array}
$$
We use the coercivity estimate from Corollary \ref{annexe:cor:coercivite mathcalEsL} to relate the $L^2$ norm on the compact set $y\leq M$ to $\mathcal{E}_{s_L}$:
$$
\int_{y\leq M} |\varepsilon^{(1)}|^2 = \int_{y\leq M} (1+y)^{2k_0+2L+2}\frac{|\varepsilon^{(1)}|^2}{1+y^{2k_0+2L+2}} \leq  C(M)\mathcal{E}_{s_L} ,
$$
$$
\int_{y\leq M} |y\partial_y\varepsilon^{(1)}|^2 \leq \int_{y\leq M} (1+y)^{2k_0+2L+2}\frac{|\partial_y \varepsilon^{(1)}|^2}{1+y^{2k_0+2L}} \leq C(M)^{2(k_0+L+1)}\mathcal{E}_{s_L} .
$$
This gives:
\be \label{thetrapped:modulation:lambda lambdavarepsilon}
|\langle \frac{\lambda_s}{\lambda}\Lambda^{(1)} \varepsilon^{(1)}, \Phi_M\rangle|\leq C(M)(b_1+D(t))\sqrt{\mathcal{E}_{s_L}}.
\ee
Now that we have computed all the terms in \fref{thetrapped:eq:modulation} for $i=0$, in \fref{thetrapped:modulation:lambda NL L}, \fref{thetrapped:modulation:lambda psib}, \fref{thetrapped:modulation:lambda lineaire}, \fref{thetrapped:modulation:lambda mod} and \fref{thetrapped:modulation:lambda lambdavarepsilon}, we end up with:
\begin{equation}\label{eq:estimation modulation lambda}
\left| \frac{\lambda_s}{\lambda}+b_1 \right|=O(b_1 D(t))+O((b_1+D(t))C(M)\sqrt{\mathcal{E}_{s_L}}) .
\end{equation}

\underline{Step 2:} \emph{law of $b_i$ for $1\leq i \leq L-1$}.
We take again equation \fref{thetrapped:eq:modulation} and do the same computations. The $\tilde{\bos{Mod}}$ term represents the approximate dynamics:
\be \label{thetrapped:modulation:i mod}
\langle \tilde{\bos{Mod}}(t), \bos{H}^{*i}\bos{\Phi}_M \rangle=\langle \bos{\Lambda} \bos{Q},\bos{\Phi}_M\rangle (b_{i,s}+(i-\alpha)b_1b_i-b_{i+1})+O(b_1 D(t)) .
\ee
The linear term still disappears because of the orthogonality conditions:
\be \label{thetrapped:modulation:i lineaire}
\langle -\bos{H}(\bos{\varepsilon}), \bos{H}^{*i}\bos{\Phi}_M \rangle=0 .
\ee
For the scale changing term, as before, thanks to the coercivity of $\mathcal{E}_{s_L}$ and to \fref{eq:estimation modulation lambda}:
\be \label{thetrapped:modulation:i lambdavarepsilon}
|\langle \frac{\lambda_s}{\lambda}\bos{\Lambda} \bos{\varepsilon},\bos{H}^{*i}\bos{\Phi}_M\rangle| \leq (b_1+D(t)) C(M) \sqrt{\mathcal{E}_{s_L}} .
\ee
The error contribution, as $\tilde{\bos{\psi}}_b=\bos{\psi}_b$ for $y\leq 2M$ (for $s_0$ small enough) is estimated thanks to Proposition \ref{pr:constructionprofilnoncoupe}:
\be \label{thetrapped:modulation:i psib}
|\langle \tilde{\bos{\psi}}_b,\bos{H}^{*i}\bos{\Phi}_M \rangle| \leq C(M)b_1^{L+3} .
\ee
We now want to estimate the nonlinear contribution. Since $\bos{NL}$ is a linear sum of terms of the form $\tilde{Q}_b^{p-k}\varepsilon^{(1)k}$ for $k\geq 2$ we estimate using Cauchy-Schwarz, the $L^{\infty}$ estimate given in Lemma \ref{annexe:lem:interpolation varepsilon}, and again the coercivity estimate:
\be \label{thetrapped:modulation:i NL}
\begin{array}{r c l}
\langle \tilde{\bos{Q}}_b^{p-k}\bos{\varepsilon}^{(1)k},\bos{H}^{*i}\bos{\Phi}_M\rangle &\leq & C(M)\parallel \epsilon^{(1)} \parallel_{L^{\infty}}^{k-2}\mathcal{E}_{s_L} \\
&=& o(b_1 \sqrt{\mathcal{E}_{s_L}}) ,
\end{array}
\ee
in the regime \fref{eq:bootstrap estimations sur epsilon}. Because $(\tilde{Q}_b^{(1)})^{p-1}-Q^{p-1}=O(b_1)$ there holds for the small linear term:
\be \label{thetrapped:modulation:i L}
|\langle L(\varepsilon),H^{*i}\Phi_M \rangle|\leq b_1C(M) \sqrt{\mathcal{E}_{s_L}} .
\ee
We have estimated all the terms in \fref{thetrapped:eq:modulation} for $1\leq i\leq L-1$, in \fref{thetrapped:modulation:i mod}, \fref{thetrapped:modulation:i lineaire}, \fref{thetrapped:modulation:i lambdavarepsilon}, \fref{thetrapped:modulation:i psib}, \fref{thetrapped:modulation:i NL} and \fref{thetrapped:modulation:i L}, it yields:
\begin{equation}\label{eq:modulation 1<i<L-1}
|b_{i,s}-(i-\alpha)b_1b_i| \leq O(b_1D(t))+C(M)b_1^{L+3}+C(M)b_1\sqrt{\mathcal{E}_{s_L}} .
\end{equation}

\underline{Step 3:} \emph{the law of $b_L$}. We compute:
$$
\langle \tilde{\bos{Mod}(t)},\bos{H}^{*L}\bos{\Phi}_M\rangle = O(b_1D(t))+(b_{L,s}+(L-\alpha)b_1b_L)\langle \bos{\Lambda} \bos{Q}, \bos{\Phi}_M \rangle .
$$
The terms that we previously estimated still admits the same bounds. But the linear term does not disappear in this case. We recall that we have chosen $L$ odd. From the identity \fref{thetrapped:eq:puissances de H} relating $\bos{H}^k$ to $\mathcal{L}$:
$$
|\langle \bos{H}(\bos{\varepsilon}),\bos{H}^{*L}\bos{\Phi}_M\rangle|=|\langle \bos{H}^{L+1}\bos{\varepsilon},\bos{\Phi}_M \rangle| = |\int \mathcal{L}^{\frac{L+1}{2}}\varepsilon^{(1)}\Phi_M | \leq C(M)\sqrt{\mathcal{E}_{s_L}} .
$$
This gives:
\begin{equation}
\left| \frac{\langle \bos{H}(\bos{\varepsilon}),\bos{H}^{*L}\bos{\Phi}_M \rangle}{\langle \bos{\Phi}_M, \bos{\Lambda} \bos{Q} \rangle}\right|\lesssim M^{-\delta_0}\sqrt{\mathcal{E}_{s_L}} .
\end{equation}
We then conclude that:
\begin{equation}\label{eq:estimation modulation bL}
|b_{L,s}-(L-\alpha)b_1b_L|\leq C(M)(b_1D(t)+b_1^{L+3})+C(M)\sqrt{\mathcal{E}_{s_L}}.
\end{equation}

\underline{Step 4:} \emph{reinjection of the bounds}. Summing \fref{eq:estimation modulation bL}, \fref{eq:modulation 1<i<L-1} and \fref{eq:estimation modulation lambda} we find that:
\begin{equation}
D(t)=O(\sqrt{\mathcal{E}_{s_L}}+b_1^{L+3}).
\end{equation}
This allows us to go back to the previous estimate of the law of $\lambda $ \fref{eq:estimation modulation lambda}, of the $b_i$'s \fref{eq:modulation 1<i<L-1}, and of $b_L$ \fref{eq:estimation modulation bL} to obtain the desired estimates \fref{thetrapped:eq:modulation leq L-1} and \fref{thetrapped:eq:premiere modulation L}
\end{proof}


\subsection{Improved modulation equation for $b_L$}
We have seen in remark \ref{thetrappedregime:modulationequation:re:premiere estimation sur bl insuffisante} that the control over the evolution of $b_L$ we found in the last Lemma \ref{lem:modulation} is not sufficient. In fact, this is because our orthogonality conditions approximate a true orthogonal decomposition (which would have been to ask $\langle \bos{\varepsilon}, \bos{T}_i \rangle=0$ and would have implied the vanishing of the bad term $\langle \bos{H} \bos{\varepsilon}, \bos{T}_L\rangle=\langle  \bos{\varepsilon}, -\bos{T}_{L-1}\rangle=0$). Nevertheless, we are able to determine which part of $\bos{\varepsilon}$ contributes in the worst way to the evolution of $b_L$ and to control it. This is the subject of the following lemma:

\begin{lemma}[Improved modulation equation for $b_L$:]
\label{trappedregime:improvedmodulation:lem:improvedmodulation}
We recall that $B_0$ is given by \fref{eq:def:B1etB0}. Assume all the constants involved in Proposition \ref{prop:bootstrap} are fixed in their range except $s_0$. Then for $s_0$ large enough there holds\footnote{The denominator being non null from \fref{thetrapped:eq:improved modulation taille denominateur}.} for $s_0\leq s<s^*$:
\begin{equation} \label{thetrapped:eq:improved modulation}
\begin{array}{r c l}
&& \left| b_{L,s}+(L-\alpha)b_1 b_L -\frac{d}{ds}\left[ \frac{\langle \bos{H}^{L}\bos{\varepsilon} , \chi_{B_0} \bos{\Lambda} \bos{Q}\rangle}{\Bigl\langle  \chi_{B_0} \Lambda^{(1)} Q, \Lambda^{(1)} Q +(-1)^{\frac{L-1}{2}}\left(\frac{\partial S_{L+2}}{\partial b_L}\right)_{L-1}\Bigr\rangle}\right]\right| \\
&\leq& \frac{1}{B_0^{\delta_0}}C(L,M)\left[  \sqrt{\mathcal{E}_{s_L}} +b_1^{L+1-\delta_0+g'}\right] ,
\end{array}
\end{equation}
where $g'$ is the gain in the asymptotic of the profiles $S_i$ defined by \fref{eq:def:gain2}.
\end{lemma}

\begin{proof}[Proof of Lemma \ref{trappedregime:improvedmodulation:lem:improvedmodulation}] \underline{Step 1:} Expression of the time derivative of the numerator. We first compute the time evolution of the numerator of the new term we introduced in \fref{trappedregime:improvedmodulation:lem:improvedmodulation}: $\langle \bos{H}^{L}\bos{\epsilon} , \chi_{B_0} \bos{\Lambda} \bos{Q}\rangle$. From the evolution equation for $\bos{\varepsilon}$ given by \fref{eq:evolution epsilon}:
\be \label{thetrapped:eq:expression numerateur}
\frac{d}{ds} \left( \langle \bos{H}^L \bos{\varepsilon}, \chi_{B_0}\bos{\Lambda} \bos{Q} \rangle \right)= \langle \bos{H}^L \bos{\varepsilon}_s, \chi_{B_0}\bos{\Lambda} \bos{Q} \rangle + \langle \bos{H}^L \bos{\varepsilon}, b_{1,s} y\partial_y \chi (\frac{y}{B_0})\bos{\Lambda} \bos{Q} \rangle .
\ee
We will now compute each term in the right hand side. We first estimate the second term. From the modulation equation \fref{thetrapped:eq:modulation leq L-1}, and under the bootstrap assumptions \fref{eq:bootstrap estimations sur epsilon} one has $|b_{1,s}|\leq Cb_1^2$. We use the expression of $\bos{H}^L$ given by \fref{thetrapped:eq:puissances de H}, $L$ being odd, and the coercivity of $\mathcal{E}_{s_L}$, see Corollary \ref{annexe:cor:coercivite mathcalEsL}:
\be \label{thetrapped:eq:expression deuxieme terme improved}
\begin{array}{r c l}
\left| \langle \bos{H}^L \bos{\varepsilon}, b_{1,s} \partial_y \chi (\frac{y}{B_0})\bos{\Lambda} \bos{Q} \rangle \right|&=& \left| \int (-1)^{\frac{L+1}{2}}\mathcal{L}^{\frac{L-1}{2}}\varepsilon^{(2)}b_{1,s} y \partial_y \chi (\frac{y}{B_0})\Lambda^{(1)}Q  \right| \\
&\leq& Cb_1^2 \int_{B_0}^{2B_0} |\varepsilon^{(2)}_{L-1}| \frac{y}{y^{\gamma}} = Cb_1^2 \int_{B_0}^{2B_0} \frac{|\varepsilon^{(2)}_{L-1}|}{y^{k_0+1}} y^{k_0-\gamma+2} \\
&\leq & C(M) b_1^2 \sqrt{\mathcal{E}_{S_L}} \left( \int_{B_0}^{2B_0} y^{2k_0-2\gamma+4} \right)^{\frac{1}{2}} \\
&\leq&  C(M)b_1^2 \sqrt{\mathcal{E}_{S_L}} b_1^{-(2k_0+\delta_0+2)} \\
&\leq& C(M)\sqrt{\mathcal{E}_{S_L}} b_1^{-(2k_0+\delta_0)} 
\end{array}
\ee
where we used the asymptotic \fref{eq:degenerescencesoliton} of $\Lambda^{(1)}Q$ (and we recall that $f_k$ stands for the $k$-th adapted derivative of $f$ given by \fref{linearized:eq:def derivees adaptees}). We now aim at estimating the other term in the right hand side of \fref{thetrapped:eq:expression numerateur}. We compute using again the expression of $\bos{H}^L$ given by \fref{thetrapped:eq:puissances de H} and the fact that $L$ is odd:
\be \label{thetrapped:eq:expression premier terme improved}
\begin{array}{r c l}
&(-1)^{\frac{L+1}{2}} \langle \bos{H}^{L} \bos{\varepsilon}_s, \chi_{B_0}\bos{\Lambda} \bos{Q}\rangle= \int \mathcal{L}^{\frac{L-1}{2}} \varepsilon^{(2)}_s \Lambda^{(1)}Q \\
=& \int \chi_{B_0} \Lambda^{(1)}Q \left(-\mathcal{L}\varepsilon^{(1)}+\frac{\lambda_s}{\lambda}\Lambda^{(2)}\varepsilon^{(2)}-\tilde{Mod}(t)^{(2)}-\tilde{\psi}_b^{(2)}+NL(\varepsilon)+L(\varepsilon)\right)_{L-1} ,
\end{array}
\ee
and we now estimate all the terms in the right hand side. \\
$\bullet$ \emph{$\mathcal{L}\varepsilon^{(1)}$ term:} There holds using coercivity and the fact that $A(\Lambda^{(1)}Q)=0$:
\be \label{thetrapped:eq:estimation premier terme improved intermediaire lineaire}
\begin{array}{r c l}
\left| \int \chi_{B_0}\Lambda^{(1)}Q (\mathcal{L}\varepsilon^{(1)})_{L-1}  \right| &\leq& C \int_{B_0}^{2B_0} \frac{1}{y^{\gamma+1}} |\varepsilon^{(1)}_{L}|  \leq C \int_{B_0}^{2B_0} \frac{\varepsilon^{(1)}_{L}}{y^{k_0+1}}y^{k_0-\gamma} \\
&\leq& C(M) \sqrt{\mathcal{E}_{s_L}} b_1^{-(2k_0+\delta_0)} .
\end{array}
\ee
$\bullet$\emph{$\Lambda^{(2)}\varepsilon^{(2)}$ term:} Again, using the same arguments, as $\frac{|\lambda_s|}{\lambda}\leq Cb_1$ from \fref{thetrapped:eq:modulation leq L-1}:
\be  \label{thetrapped:eq:estimation premier terme improved intermediaire lambdavarepsilon}
\begin{array}{r c l}
\left| \int\chi_{B_0}\Lambda^{(1)}Q \frac{\lambda_s}{\lambda} (\Lambda^{(2)}\varepsilon^{(2)})_{L-1} \right| &\leq& Cb_1 \int_{B_0}^{2B_0} \frac{1}{y^{\gamma+1}} |\epsilon^{(2)}_{L-1}| \leq C(M)b_1 \sqrt{\mathcal{E}_{s_L}} b_1^{-(2k_0+1+\delta_0)} \\
&\leq& C(M) \sqrt{\mathcal{E}_{s_L}} b_1^{-(2k_0+\delta_0)} .
\end{array}
\ee
$\bullet$ \emph{$\tilde{\psi}_b$ term:} Because we are in the zone $\sim B_0$ we do not see the bad tail. We can then use the improved bound of Proposition \ref{pr:constructionprofilnoncoupe}:
\be \label{thetrapped:eq:estimation premier terme improved intermediaire psib}
\begin{array}{r c l}
\left| \int \chi_{B_0}\Lambda^{(1)}Q (\tilde{\psi_b}^{(2)})_{L-1} \right| &=& \left| \int \chi_{B_0}\Lambda^{(1)}Q (\psi_b^{(2)})_{L-1} \right|\\
& \leq& \parallel \Lambda^{(1)} Q\parallel_{L^2(\leq 2B_0)} \parallel \psi_{b,L-1}^{(2)}\parallel_{L^2(\leq 2B_0)}\\
&\leq & C b_1^{L+1-2k_0-2\delta_0+g'} .
\end{array}
\ee
$\bullet$ \emph{$NL(\varepsilon)$ term:} By duality we put all the derivatives on $\Lambda^{(1)}Q$:
$$
\begin{array}{r c l}
\left| \int \chi_{B_0} \Lambda^{(1)}Q (NL(\varepsilon))_{L-1}  \right|&=&\left| \int (\chi_{B_0} \Lambda^{(1)}Q)_{L-1} NL(\varepsilon)  \right| \leq C  \int_{B_0}^{2B_0} \frac{1}{y^{\gamma+L-1}} |NL(\varepsilon)| .
\end{array}
$$
We know that $NL(\varepsilon)$ is a sum of terms of the form: $Q^{p-k}\varepsilon^{(1)k}$ for $k>2$. So from the asymptotic \fref{linearized:eq:asymptotique soliton} of $Q$ and using coercivity:
$$
\begin{array}{r c l}
\left|\int_{B_0}^{2B_0} \frac{1}{y^{\gamma+L-1}}Q^{p-k}\varepsilon^{(1)k}\right| &\leq& C\parallel \varepsilon^{(1)} \parallel_{L^{\infty}}^{k-1} \int_{B_0}^{2B_0} \frac{|\varepsilon^{(1)}|}{y^{\gamma+L-1+\frac{2}{p-1}(p-k)}} \\
&\leq & C(M)\parallel \varepsilon^{(1)} \parallel_{L^{\infty}}^{k-1} \sqrt{\mathcal{E}_{s_L}} b_1^{-(2k_0+\delta_0)}b_1^{-2+\frac{2}{p-1}(p-k)} .
\end{array}
$$
We now use the estimate provided by Lemma \ref{annexe:lem:interpolation varepsilon}:
$$
\begin{array}{r c l}
\parallel \varepsilon^{(1)} \parallel_{L^{\infty}}& \leq &C(M,K_1,K_2) \sqrt{\mathcal{E}_{\sigma}}b_1^{\frac{d}{2}-\sigma+\frac{\frac{2}{p-1}\alpha}{L}+O\left(\frac{\sigma-s_c}{L} \right)} \\
&\leq& C(M,K_1,K_2)  \left(\frac{\mathcal{E}_{\sigma}}{b_1^{\sigma-s_c}} \right)b_1^{\frac{2}{p-1}+\frac{\frac{2}{p-1}\alpha}{L}+O\left(\frac{\sigma-s_c}{L} \right)} .
\end{array}
$$
Therefore:
$$
\left| \int_{B_0}^{2B_0} \frac{Q^{p-k}\varepsilon^{(1)k}}{y^{\gamma+L-1}}\right|  \leq C(M,K_1,K_2) \left(\frac{\mathcal{E}_{\sigma}}{b_1^{\sigma-s_c}} \right)^{k-1} \sqrt{\mathcal{E}_{s_L}} b_1^{-(2k_0+\delta_0)+\frac{2(k-1)\alpha}{(p-1)L} +O(\frac{\sigma-s_c}{L})} .
$$
Under the bootstrap estimate, for $s_0$ small enough this gives:
\be \label{thetrapped:eq:estimation premier terme improved intermediaire NL}
\left| \int \chi_{B_0} \Lambda^{(1)} Q NL(\varepsilon^{(1)})_{L-1} \right|\leq\sqrt{\mathcal{E}_{s_L}}b_1^{-(2k_0+\delta_0)} .
\ee
Indeed, the constant $s_0$ being chosen after all the other constants, we can increase $s_0$ to erase the dependence on the other constant in the preceding equation.\\
$\bullet$ \emph{$L(\varepsilon)$ term:}
$$
\begin{array}{r c l}
\left| \int \chi_{B_0} \Lambda^{(1)}Q (L(\varepsilon))_{L-1} \right| \leq C \int_{B_0}^{2B_0} \frac{1}{y^{\gamma+L-1}}|Q_b^{p-1}-Q^{p-1}||\varepsilon^{(1)}| .
\end{array}
$$
We use the degeneracy of the potential: $Q^{p-1}_b-Q^{p-1}\leq \frac{C}{1+y^{2+\alpha}}$ to estimate:
\be \label{thetrapped:eq:estimation premier terme improved intermediaire L}
\begin{array}{r c l}
\left| \int \chi_{B_0} \Lambda^{(1)}Q (L(\varepsilon))_{L-1} \right| &\leq & C \int_{B_0}^{2B_0} \frac{|\varepsilon^{(1)}|}{y^{\gamma+L-1+2+\alpha}} \\
&\leq & C(M) \sqrt{\mathcal{E}_{s_L}} b_1^{-(2k_0+\delta_0)}b_1^{\alpha} .
\end{array}
\ee
$\bullet$ \emph{$\tilde{Mod(t)}^{(2)}$ term:} From the localization of the $T_i$' and $S_i$'s (\fref{linearized:eq:localisation Ti} and \fref{linearized:eq:degre Si}):
$$
\begin{array}{r c l}
&\int \tilde{Mod(t)}^{(2)}_{L-1}\chi_{B_0}\Lambda^{(1)} Q \\
=& \int (\sum_{i=1}^L(b_{i,s}+(i-\alpha)b_1b_i-b_{i+1})(T_i\delta_{i \text{mod}2,1}+\underset{j= i+1 \ \text{odd}}{\overset{L+2}{\sum}}\frac{\partial S_j}{\partial b_i}) )_{L-1}\chi_{B_0}\Lambda^{(1)} Q \\
&-\int (\frac{\lambda_s}{\lambda}+b_1)\Lambda^{(2)}\tilde{\alpha}_b^{(2)})_{L-1}\chi_{B_0}\Lambda^{(1)}Q\\
=& (b_{L,s}+(L-\alpha)b_1b_L)\int (T_L+\frac{\partial S_{L+2}}{\partial b_L})_{L-1}\chi_{B_0}\Lambda^{(1)}Q \\
& + \int (\sum_{i=1}^{L-1}(b_{i,s}+(i-\alpha)b_1b_i-b_{i+1})(T_i\delta_{i \text{mod}2,1}+\underset{j= i+1 \ \text{odd}}{\overset{L+2}{\sum}}\frac{\partial S_j}{\partial b_i}) )_{L-1}\chi_{B_0}\Lambda^{(1)} Q\\
&-\int (\frac{\lambda_s}{\lambda}+b_1)(\Lambda^{(2)}\tilde{\alpha}_b^{(2)})_{L-1}\chi_{B_0}\Lambda^{(1)}Q .
\end{array}
$$
We compute from the fact that $\bos{H}(\bos{T_L})=(-1)^{L}\bos{\Lambda Q}$:
$$
\int (T_L)_{L-1}\chi_{B_0}\Lambda^{(1)} Q=(-1)^{\frac{L-1}{2}}\int |\Lambda^{(1)}Q|^2\chi_{B_0} .
$$
For $i<L$, as $(T_i)_{L-1}=0$ we have:
$$
\begin{array}{r c l}
&\left| \int (T_i\delta_{i\text{mod}2,1}+\sum_{j\geq i+1 \  \text{odd}}^{L+2}\frac{\partial S_j^{(2)}}{\partial b_i})_{L-1}\chi_{B_0}\Lambda^{(1)}Q\right| \\
=&\left| \int \sum_{j\geq i+1}^{L+2}\left(\frac{\partial S_j^{(2)}}{\partial b_i}\right)_{L-1}\chi_{B_0}\Lambda^{(1)}Q\right| \leq  Cb_1^{L-i+g'}b_1^{-(2k_0+2\delta_0)} .
\end{array}
$$
And for the last term there holds the bound:
$$
\left| \int (\Lambda^{(2)}\tilde{\alpha}_{b}^{(2)})_{L-1}\chi_{B_0}\Lambda^{(1)}Q \right|\leq C b_1^{L-(2k_0+2\delta_0)}
$$
We then conclude, using the majoration obtained in the previous Lemma \ref{lem:modulation} for the evolution of the $b_i$'s and $\lambda$, that for the $\tilde{Mod(t)}$ term:
\be \label{thetrapped:eq:estimation premier terme improved intermediaire mod}
\begin{array}{r c l}
&\int \tilde{Mod(t)}^{(2)}_{L-1}\chi_{B_0}\Lambda^{(1)} Q \\
=& (b_{L,s}+(L-\alpha)b_1b_L)\left( (-1)^{\frac{L-1}{2}} \int (\Lambda^{(1)}Q)^2\chi_{B_0}+\left(\frac{\partial S_{L+2}^{(2)}}{\partial b_L}\right)_{L-1}\chi_{B_0} \Lambda^{(1)}Q \right) \\
&+O(\sqrt{\mathcal{E}_{s_L}}b_1^{-(2k_0+\delta_0)}+b_1^{L+3-(2k_0+\delta_0)})
\end{array}
\ee
(From now on we use the $O()$ notation, the constants hidden depending only on M). We now collect all the estimates \fref{thetrapped:eq:estimation premier terme improved intermediaire lineaire}, \fref{thetrapped:eq:estimation premier terme improved intermediaire lambdavarepsilon}, \fref{thetrapped:eq:estimation premier terme improved intermediaire psib}, \fref{thetrapped:eq:estimation premier terme improved intermediaire NL}, \fref{thetrapped:eq:estimation premier terme improved intermediaire L} and \fref{thetrapped:eq:estimation premier terme improved intermediaire mod}, inject them in \fref{thetrapped:eq:expression premier terme improved} to find that the first term in the right hand side of \fref{thetrapped:eq:expression numerateur} is:
\be \label{thetrapped:eq:estimation premier term improved modulation}
\begin{array}{r c l}
&\langle \bos{H}^L \bos{\varepsilon}_s, \chi_{B_0}\bos{\Lambda} \bos{Q} \rangle +O(\sqrt{\mathcal{E}_{s_L}}b_1^{-(2k_0+\delta_0)})+O(b_1^{L+1-2k_0-2\delta_0+g'}).\\
=& (b_{L,s}+(L-\alpha)b_1b_L)\Bigl\langle  \chi_{B_0}\Lambda^{(1)}Q,\Lambda^{(1)}Q+(-1)^{\frac{L-1}{2}}\left(\frac{\partial S_{L+2}^{(2)}}{\partial b_L}\right)_{L-1} \Bigr\rangle\\
\end{array}
\ee
With the two computations \fref{thetrapped:eq:estimation premier term improved modulation} and \fref{thetrapped:eq:expression deuxieme terme improved}, the time evolution of the numerator given by \fref{thetrapped:eq:expression numerateur} is now:
\be\label{thetrapped:eq:evolution temporelle numerateur}
\begin{array}{r c l}
&\frac{d}{ds}\langle \bos{H}^L \bos{\varepsilon}, \chi_{B_0}\bos{\Lambda} \bos{Q} \rangle +O(\sqrt{\mathcal{E}_{s_L}}b_1^{-(2k_0+\delta_0)})+O(b_1^{L+1-2k_0-2\delta_0+g'}).\\
=& (b_{L,s}+(L-\alpha)b_1b_L)\Bigl\langle  \chi_{B_0}\Lambda^{(1)}Q,\Lambda^{(1)}Q+(-1)^{\frac{L-1}{2}}\left(\frac{\partial S_{L+2}^{(2)}}{\partial b_L}\right)_{L-1} \Bigr\rangle\\
\end{array}
\ee

\underline{Step 2:} end of the computation. We have thanks to our previous estimate \fref{thetrapped:eq:evolution temporelle numerateur}:
\be \label{thetrapped:eq:improved modulation expression intermediare}
\begin{array}{r c l}
&&\frac{d}{ds}\left[ \frac{\langle \bos{H}^{L}\bos{\varepsilon} , \chi_{B_0} \bos{\Lambda} \bos{Q}\rangle}{\Bigl\langle  \chi_{B_0}\Lambda^{(1)}Q,\Lambda^{(1)}Q+(-1)^{\frac{L-1}{2}}\left(\frac{\partial S_{L+2}^{(2)}}{\partial b_L}\right)_{L-1} \Bigr\rangle}\right] \\
&=& (b_{L,s}+(L-\alpha)b_1b_L)  + \frac{O(\sqrt{\mathcal{E}_{s_L}}b_1^{-(2k_0+\delta_0)}+b_1^{L+1-(2k_0+2\delta_0)+g'})}{\Bigl\langle  \chi_{B_0}\Lambda^{(1)}Q,\Lambda^{(1)}Q+(-1)^{\frac{L-1}{2}}\left(\frac{\partial S_{L+2}^{(2)}}{\partial b_L}\right)_{L-1} \Bigr\rangle} \\
&&-  \frac{\langle \bos{H}^{L}\bos{\varepsilon} , \chi_{B_0} \bos{\Lambda} \bos{Q}\rangle \times \frac{d}{ds}\left[ \Bigl\langle  \chi_{B_0}\Lambda^{(1)}Q,\Lambda^{(1)}Q+(-1)^{\frac{L-1}{2}}\left(\frac{\partial S_{L+2}^{(2)}}{\partial b_L}\right)_{L-1} \Bigr\rangle \right]}{\Bigl\langle  \chi_{B_0}\Lambda^{(1)}Q,\Lambda^{(1)}Q+(-1)^{\frac{L-1}{2}}\left(\frac{\partial S_{L+2}^{(2)}}{\partial b_L}\right)_{L-1} \Bigr\rangle^2} .
\end{array}
\ee
From the asymptotic of $\Lambda^{(1)}Q$ and $S_{L+2}$, the denominator has the following size:
\begin{equation} \label{thetrapped:eq:improved modulation taille denominateur}
\Bigl\langle  \chi_{B_0}\Lambda^{(1)}Q,\Lambda^{(1)}Q+(-1)^{\frac{L-1}{2}}\left(\frac{\partial S_{L+2}^{(2)}}{\partial b_L}\right)_{L-1} \Bigr\rangle\sim C b_1^{-2k_0-2\delta_0},
\end{equation}
for some constant $C>0$. So the second term in the right hand side of \fref{thetrapped:eq:improved modulation expression intermediare} is:
\be \label{thetrapped:eq:improved modulation second terme expression intermediaire}
\left| \frac{O(\sqrt{\mathcal{E}_{s_L}}b_1^{-(2k_0+\delta_0)}+b_1^{L+1-(2k_0+2\delta_0)+g'})}{\Bigl\langle  \chi_{B_0}\Lambda^{(1)}Q,\Lambda^{(1)}Q+(-1)^{\frac{L-1}{2}}\left(\frac{\partial S_{L+2}^{(2)}}{\partial b_L}\right)_{L-1} \Bigr\rangle}\right| \leq C(M)\left( b_1^{-\delta_0}\sqrt{\mathcal{E}_{s_L}} +b_1^{L+1+g'}\right)
\ee
We now estimate the third term in the right hand side of \fref{thetrapped:eq:improved modulation expression intermediare}. We have by coercivity of the adapted norm:
\begin{equation} \label{thetrapped:eq:improved modulation estimation numerateur}
|\langle \bos{H}^L \bos{\varepsilon} ,\chi_{B_0}\bos{\Lambda} \bos{Q}\rangle |\leq C  \int_{B_0}^{2B_0}\frac{\varepsilon{(2)}}{y^{\gamma+L-1}} \leq C(M)  \sqrt{\mathcal{E}_{s_L}} b_1^{-(2k_0+\delta_0)-1} .
\end{equation}
As $\frac{\partial S_{L+2}^{(2)}}{\partial b_L}$ does not depend on $b_L$, we obtain using the modulation bound \fref{thetrapped:eq:modulation leq L-1} for $b_1$,...,$b_{L-1}$:
$$
\left| \frac{\frac{d}{ds}\left[ \Bigl\langle  \chi_{B_0}\Lambda^{(1)}Q,\Lambda^{(1)}Q+(-1)^{\frac{L-1}{2}}\left(\frac{\partial S_{L+2}^{(2)}}{\partial b_L}\right)_{L-1} \Bigr\rangle\right]}{\Bigl\langle  \chi_{B_0}\Lambda^{(1)}Q,\Lambda^{(1)}Q+(-1)^{\frac{L-1}{2}}\left(\frac{\partial S_{L+2}^{(2)}}{\partial b_L}\right)_{L-1} \Bigr\rangle}\right| \leq C(M) b_1 .
$$
The third term in the right hand side of \fref{thetrapped:eq:improved modulation expression intermediare} then admits the bound:
\be \label{thetrapped:eq:improved modulation troisieme terme expression intermediaire}
\begin{array}{r c l}
&\left|   \frac{\langle \bos{H}^{L}\bos{\varepsilon} , \chi_{B_0} \bos{\Lambda} \bos{Q}\rangle \times \frac{d}{ds}\left[ \Bigl\langle  \chi_{B_0}\Lambda^{(1)}Q,\Lambda^{(1)}Q+(-1)^{\frac{L-1}{2}}\left(\frac{\partial S_{L+2}^{(2)}}{\partial b_L}\right)_{L-1} \Bigr\rangle \right]}{ \Bigl\langle  \chi_{B_0}\Lambda^{(1)}Q,\Lambda^{(1)}Q+(-1)^{\frac{L-1}{2}}\left(\frac{\partial S_{L+2}^{(2)}}{\partial b_L}\right)_{L-1} \Bigr\rangle^2} \right| \\
\leq& C(M)b_1^{-\delta_0}\sqrt{\mathcal{E}_{s_L}}.
\end{array}
\ee
The identity \fref{thetrapped:eq:improved modulation expression intermediare}, with the bounds on the terms \fref{thetrapped:eq:improved modulation second terme expression intermediaire} and \fref{thetrapped:eq:improved modulation troisieme terme expression intermediaire}, gives:
$$
\begin{array}{r c l}
\frac{d}{ds}\left[ \frac{\langle \bos{H}^{L}\bos{\epsilon} , \chi_{B_0} \bos{\Lambda} \bos{Q}\rangle}{\Bigl\langle  \chi_{B_0}\Lambda^{(1)}Q,\Lambda^{(1)}Q+(-1)^{\frac{L-1}{2}}\left(\frac{\partial S_{L+2}^{(2)}}{\partial b_L}\right)_{L-1} \Bigr\rangle}\right]
&=& (b_{L,s}+(L-\alpha)b_1b_L) \\
&& + O(\sqrt{\mathcal{E}_{s_L}}b_1^{\delta_0}+b_1^{L+1+g'}) ,
\end{array}
$$
the constant hidden in the $O()$ depending on $M$ (and $L$ of course but we do not track the dependence on this constant anymore).
\end{proof}


\subsection{Lyapunov monotonicity for the low Sobolev norm:} As it appeared in the previous subsections, the key estimate in our analysis is the one concerning the high Sobolev norm. Nonetheless, to have an idea on how the lower derivatives behave, and to close an estimate for the nonlinear term in the next section, we start by computing an energy estimate on the low Sobolev norm. We define:
\be
\nu:=\frac{\alpha}{\ell-\alpha},
\ee
so that $1+\nu=\frac{\ell}{\ell-\alpha}$ and that the condition \fref{eq:bootstrap estimations sur epsilon} for $\mathcal{E}_{\sigma}$ can be rewriten as:
\be
\mathcal{E}_{\sigma}\leq K_2 b_1^{2(\sigma-s_c)(1+\nu)}
\ee

\begin{proposition}\label{trappedregime:pr:low sobo}
\emph{(Lyapunov monotonicity for the low Sobolev norm:)} Assume all the constants involved in Proposition \ref{prop:bootstrap} are fixed in their range, except $s_0$ and $\eta$. Then for $s_0$ large enough and $\eta$ small enough there holds for $s_0\leq s <s^*$:
\begin{equation}
\frac{d}{dt}\left\{ \frac{\mathcal{E}_{\sigma}}{\lambda^{2(\sigma-s_c)}} \right\} \leq \frac{b_1\sqrt{\mathcal{E}_{\sigma}}b_1^{(\sigma-s_c)(1+\nu)}}{\lambda^{2(\sigma-s_c)+1}}  \left[ b_1^{\frac{\alpha}{2L}+O\left(\frac{\sigma-s_c}{L}\right)}+b_1^{\frac{\alpha}{2L}+O\left(\frac{\sigma-s_c}{L}\right)}\sum_{k=2}^p\left( \frac{\sqrt{\mathcal{E}_{\sigma}}}{b_1^{\sigma-s_c}}\right)^{k-1}\right] 
\end{equation}
(the norm $\mathcal{E}_{\sigma}$ was defined in \fref{thetrapped:eq:def mathcalEsigma}).
\end{proposition}

\begin{proof}[Proof of Proposition \ref{trappedregime:pr:low sobo}]
To prove this proposition we will compute the derivative with respect to time of $\frac{\mathcal{E}_{\sigma}}{\lambda^{2(\sigma-s_c)}}$ and estimate it in the trapped regime using \fref{eq:bootstrap estimations sur epsilon} and the size of the error given by Proposition \ref{pr:profilapprochecoupe}. From the evolution of $\bos{w}$ given by \fref{eq:evolution W} we first compute the following identity:
\be \label{thetrapped:lowsobo:eq:expression derivee}
\begin{array}{r c l}
&\frac{d}{dt}\left\{ \frac{\mathcal{E}_{\sigma}}{\lambda^{2(\sigma-s_c)}} \right\} = \frac{d}{dt} \left\{ \int |\nabla^{\sigma}w^{(1)} |^2+\int |\nabla^{\sigma-1}w^{(2)} |^2    \right\} \\
=& \int \nabla^{\sigma}w^{(1)}.\nabla^{\sigma}(w^{(2)}+\frac{1}{\lambda}(-\tilde{Mod}(t)_{\frac{1}{\lambda}}^{(1)}-\tilde{\psi_b}_{\frac{1}{\lambda}}^{(1)})) \\
& + \int \nabla^{\sigma-1}w^{(2)}.\nabla^{\sigma-1}(-\mathcal{L}w^{(1)}+\frac{1}{\lambda}(-\tilde{Mod}(t)_{\frac{1}{\lambda}}^{(2)}-\tilde{\psi_b}_{\frac{1}{\lambda}}^{(2)})+NL(w)+L(w)) .
\end{array}
\ee

\underline{Step 1:} estimate on each term. We will now estimate everything in the right hand side of \fref{thetrapped:lowsobo:eq:expression derivee}.\\
$\bullet$ \emph{Linear terms:} Because the norm we are using is adapted to a wave equation we have:
$$
\begin{array}{r c l}
\int \nabla^{\sigma}w^{(1)}.\nabla^{\sigma}w^{(2)}- \nabla^{\sigma-1}w^{(2)}.\nabla^{\sigma-1}\mathcal{L} w^{(1)}&=&\int \nabla^{\sigma-1}w^{(2)}. \nabla^{\sigma-1}(pQ_{\frac{1}{\lambda}}^{p-1}w^{(1)}) \\
&\leq & \parallel \nabla^{\sigma}w^{(2)} \parallel_{L^2} \parallel \nabla^{\sigma-2}(Q_{\frac{1}{\lambda}}^{p-1}w^{(1)}) \parallel_{L^2} .
\end{array}
$$
We now use the asymptotic behavior $Q^{p-1}\sim\frac{c}{x^{2}}$ ($c>0$) and the weighted Hardy estimate from Lemma \ref{annexe:lem: hardy frac a poids}:
$$
\parallel \nabla^{\sigma-2}(Q_{\frac{1}{\lambda}}^{p-1}w^{(1)}) \parallel_{L^2} \leq C \parallel \nabla^{\sigma} w^{(1)}\parallel_{L^2} = C \frac{\sqrt{\mathcal{E}_{\sigma}}}{\lambda^{\sigma-s_c}} .
$$
The other term is estimated by interpolation. Indeed as $\parallel \nabla^{s_L-1}\varepsilon^{(2)} \parallel_{L^2}^2\leq c\mathcal{E}_{s_L}$ from Corollary \ref{annexe:cor:coercivite mathcalEsL}:
$$
\parallel \nabla^{\sigma} w^{(2)} \parallel_{L^2}\leq \frac{C(M)}{\lambda^{\sigma-s_c+1}} \sqrt{\mathcal{E}_{\sigma}}^{1-\frac{1}{s_L-\sigma}}\sqrt{\mathcal{E}_{s_L}}^{\frac{1}{s_L-\sigma}} 
$$
We have the following estimate under the bootstrap conditions \fref{eq:bootstrap estimations sur epsilon} :
$$
\sqrt{\mathcal{E}_{\sigma}}^{1-\frac{1}{s_L-\sigma}}\sqrt{\mathcal{E}_{s_L}}^{\frac{1}{s_L-\sigma}} \leq C(K_1,K_2,M) b_1^{(\sigma-s_c)(1+\nu)}  b_1^{(\frac{1}{s_L-\sigma})(L+(1-\delta_0)(1+\eta)-(\sigma-s_c)(1+\nu))}
$$
and from: $\frac{L+(1-\delta_0)(1+\eta)-(\sigma-s_c)(1+\nu)}{s_L-\sigma}=1+\frac{(1-\delta_0)\eta+\alpha}{L}+O(\frac{\sigma-s_c}{L}) $ we conclude that:
\be \label{thetrapped:lowsobo:eq:estimation lineaire}
\begin{array}{r c l}
&\left| \int \nabla^{\sigma}w^{(1)}.\nabla^{\sigma}w^{(2)}- \nabla^{\sigma-1}w^{(2)}.\nabla^{\sigma-1}\mathcal{L} w^{(1)}\right|  \\
 \leq& C(K_1,K_2,M) \frac{\sqrt{\mathcal{E}_{\sigma}}b_1^{(\sigma-s_c)(1+\nu)}}{\lambda^{2(\sigma-s_c)}}\frac{b_1}{\lambda}b_1^{\frac{\alpha}{L}+O\left(\frac{\sigma-s_c}{L} \right)} .
\end{array}
\ee
$\bullet$ \emph{$\tilde{\bos{Mod}}(t)$ terms:} We only treat the $\tilde{Mod(t)}^{(2)}$ terms, the computation being the same for the first coordinate.
$$
\left| \frac{1}{\lambda}\int \nabla^{\sigma-1}w^{(2)}.\nabla^{\sigma-1} \tilde{Mod(t)}^{(2)}_{\frac{1}{\lambda}}\right|\leq \frac{1}{\lambda^{2(\sigma-s_c)}}\frac{1}{\lambda}\sqrt{\mathcal{E}_{\sigma}} \parallel \nabla^{\sigma-1}\tilde{Mod(t)}\parallel_{L^2} .
$$
We compute thanks to the previous estimate on the modulation, see Lemma \ref{lem:modulation}:
$$
\begin{array}{r c l}
&\parallel \nabla^{\sigma-1}\tilde{Mod}^{(2)} \parallel_{L^2}\\
\lesssim& (\sqrt{\mathcal{E}_{s_L}}+b^{L+3})\left( \underset{i<j\leq L+2}{\sum}\parallel \nabla^{\sigma-1}\left(\chi_{B_1}\frac{\partial S_j^{(2)}}{\partial b_i}\right)\parallel_{L^2} +\sum_0^L \parallel \chi_{B_1}\nabla^{\sigma-1}T_i^{(2)} \parallel_{L^2} \right) \\
=& (\sqrt{\mathcal{E}_{s_L}}+b^{L+3})b_1^{(1+\eta)(-k_0-\delta_0-L+\sigma)} \\
\leq & C(M) b_1^{\alpha+(1-\delta_0)+(\sigma-s_c)+\eta(1-\delta_0+\alpha+(\sigma-s_c)-L)} .
\end{array}
$$
Hence, treating similarly the other coordinate:
\be \label{thetrapped:lowsobo:eq:estimation mod}
\left| \frac{1}{\lambda}\int \nabla^{\sigma-1}w^{(2)}.\nabla^{\sigma-1} \tilde{Mod(t)}^{(2)}_{\frac{1}{\lambda}}+\nabla^{\sigma}w^{(1)}.\nabla^{\sigma} \tilde{Mod(t)}^{(1)}_{\frac{1}{\lambda}} \right| \leq C(M) \frac{b_1\sqrt{\mathcal{E}_{\sigma}}b_1^{(\sigma-s_c)+\alpha}}{\lambda^{2(\sigma-s_c)+1}} .
\ee
$\bullet$ \emph{$\tilde{\bos{\psi}}_b$ term:} Again we just compute for the first coordinate $\tilde{\psi}_b^{(1)}$, because we can treat the second one exactly the same way.
$$
\left|\frac{1}{\lambda}\int \nabla^{\sigma}w^{(1)}.\nabla^{\sigma}\tilde{\psi}_{b,\frac{1}{\lambda}}^{(1)}\right| \leq \frac{1}{\lambda^{2(\sigma-s_c)}}\frac{1}{\lambda}\sqrt{\mathcal{E}_{\sigma}}\parallel \nabla^{\sigma} \tilde{\psi}_b^{(1)} \parallel_{L^2} .
$$
We can estimate using proposition \ref{pr:profilapprochecoupe}:
$$
\parallel \nabla^{\sigma} \tilde{\psi}_b^{(1)} \parallel_{L^2}\leq C b_1^{(1-\delta_0)+\sigma-k_0-C\eta}=Cb_1^{(\sigma-s_c)+\alpha-C\eta+1} .
$$
Hence for $\eta$ small enough:
$$
\left| \frac{1}{\lambda}\int \nabla^{\sigma}w^{(1)}.\nabla^{\sigma}\tilde{\psi}_{b,\frac{1}{\lambda}}^{(1)} \right| \leq \frac{C}{\lambda^{(\sigma-s_c)}}\frac{b_1}{\lambda}\sqrt{\mathcal{E}_{\sigma}}b_1^{\sigma-s_c}b_1^{\frac{3\alpha}{4}} .
$$
The same computation for the second coordinate gives the same result, hence the error's contribution is:
\be \label{thetrapped:lowsobo:eq:estimation psib}
\left| \frac{1}{\lambda}\int \nabla^{\sigma}w^{(1)}.\nabla^{\sigma}\tilde{\psi}_{b,\frac{1}{\lambda}}^{(1)} +\frac{1}{\lambda}\int \nabla^{\sigma-1}w^{(2)}.\nabla^{\sigma-1}\tilde{\psi}_{b,\frac{1}{\lambda}}^{(2)}\right| \leq \frac{C}{\lambda^{(\sigma-s_c)}}\frac{b_1}{\lambda}\sqrt{\mathcal{E}_{\sigma}}b_1^{\sigma-s_c}b_1^{\frac{3\alpha}{4}} .
\ee
$\bullet$ \emph{$L(w)$ term:} First using Cauchy-Schwarz:
$$
\left| \int \nabla^{\sigma-1}w^{(2)}.\nabla^{\sigma-1}(L(w))\right|\leq \frac{\sqrt{\mathcal{E}_{\sigma}}}{\lambda^{2(\sigma-s_c)+1}}\parallel \nabla^{\sigma-1}L(w) \parallel_{L^2}.
$$
Now we have that $L(w)=(pQ^{p-1}-p\tilde{Q}_b^{p-1})w^{(1)}$. From the asymptotics of the profiles $T_i$ and $S_i$, the potential here enjoys the following bounds:
\be \label{thetrapped=lowsobo:eq:asymptotique L}
\left| \partial_y^k (Q^{p-1}-\tilde{Q}_b^{p-1}) \right| \leq Cb_1 \frac{1}{y^{1+\alpha-C(L)\eta}}
\ee
It allows us to use the fractionnal Hardy estimate from Lemma \ref{annexe:lem: hardy frac a poids}:
$$
\parallel \nabla^{\sigma-1}L(w) \parallel_{L^2}\leq C b_1 \parallel \nabla^{\sigma +\frac{1}{p-1}} w^{(1)}\parallel_{L^2},
$$
because $\sigma +\frac{1}{p-1}<\frac{d}{2}$, and because for $\eta$ small enough one has: $\alpha- C(L)\eta\geq \frac{1}{p-1}$ (as $\alpha>2$). In the trapped regime one has by interpolation:
$$
\begin{array}{r c l}
\parallel \nabla^{\sigma +\frac{1}{p-1}}w^{(1)}\parallel_{L^2}&\leq& C(M)\sqrt{\mathcal{E}_{\sigma}}^{1-\frac{1}{(p-1)(s_L-\sigma)}}\sqrt{\mathcal{E}_{s_L}}^{\frac{1}{(p-1)(s_L-\sigma)}} \\
&\leq C(K_1,K_2,M) & b_1^{(\sigma-s_c)(1+\nu)}b_1^{\frac{1}{p-1}+O\left( \frac{1}{L}\right)}.
\end{array}
$$
Therefore we end up with the following bound on the small linear term:
\be \label{thetrapped:lowsobo:eq:estimation L}
\left| \int \nabla^{\sigma-1}w^{(2)}.\nabla^{\sigma-1}(L(w))\right|\leq C(K_1,K_2,M) \frac{b_1\sqrt{\mathcal{E}_{\sigma}}}{\lambda^{2(\sigma-s_c)+1}}b_1^{(\sigma-s_c)(1+\nu)+\frac{1}{p-1}+O\left(\frac{1}{L} \right)} 
\ee
$\bullet$ \emph{$NL$ term:} We start by integrating by parts and using Cauchy-Schwarz:
\be \label{thetrapped:lowsobo:eq:expression NL}
\begin{array}{r c l}
&\left| \int \nabla^{\sigma-1} w^{(2)}\nabla^{\sigma -1}NL(w)\right| \\
\leq&  \frac{1}{\lambda^{2(\sigma-s_c)+1}} \parallel \nabla^{\sigma-(k-1)(\sigma-s_c)}\varepsilon^{(2)} \parallel_{L^2} \parallel \nabla^{\sigma-2+(k-1)(\sigma-s_c)} NL(\varepsilon) \parallel_{L^2} .
\end{array}
\ee
The first term is estimated via interpolation, and gives under the bootstrap conditions:
\be \label{thetrapped:lowsobo:eq:estimation NL 1}
\begin{array} {r c l}
\parallel \nabla^{\sigma-(k-1)(\sigma-s_c)} \varepsilon^{(2)}\parallel_{L^2} &\leq& C(M)\sqrt{\mathcal{E}_{\sigma}}^{1-\frac{1-(k-1)(\sigma-s_c)}{s_L-\sigma}} \sqrt{\mathcal{E}_{s_L}}^{\frac{1-(k-1)(\sigma-s_c)}{s_L-\sigma}}\\
&\leq& C(M,K_1,K_2) b_1^{(\sigma-s_c)(1+\nu)+1-(k-1)(\sigma-s_c)+\frac{\alpha}{L}+O\left( \frac{\sigma-s_c}{L}\right)} .
\end{array}
\ee
We now estimate the second one. We know that $NL(\varepsilon)$ is a linear combination of terms of the form: $\tilde{Q}_b^{(1)(p-k)}\varepsilon^{(1)k}$ for $2\leq k\leq p$. We know also that here we have: $\partial_y^j\tilde{Q}_b^{(1)(p-k)}\leq\frac{c}{y^{\frac{2}{p-1}(p-k)+j}} $. So using the weighted and fractional hardy estimate of Lemma \ref{annexe:lem: hardy frac a poids}:
$$
\parallel \nabla^{\sigma-2+(k-1)(\sigma-s_c)}(Q^{p-k}\varepsilon^{(1)k)} \parallel \leq C \parallel \nabla^{\sigma-2+\frac{2}{p-1}(p-k)+(k-1)(\sigma-s_c)}(\varepsilon^{(1)k}) \parallel_{L^2} .
$$
We let $\tilde{\sigma}=E[\sigma-2+\frac{2}{p-1}(p-k)+(k-1)(\sigma-s_c)]$ so that:
$$
\sigma-2+\frac{2}{p-1}(p-k)+(k-1)(\sigma-s_c)=\tilde{\sigma}+\delta_{\sigma} ,
$$
with $0\leq \delta_{\sigma}<1$. Developing the entire part of the derivative yields:
$$
\parallel \nabla^{\sigma-2+\frac{2}{p-1}(p-k)+(k-1)(\sigma-s_c)}(\varepsilon^{(1)k}) \parallel_{L^2}=\parallel \nabla^{\delta_{\sigma}}(\nabla^{\tilde{\sigma}}(\varepsilon^{(1)k})) \parallel_{L^2} .
$$
We develop the $\nabla^{\tilde{\sigma}}(v^{(1)k})$ term: it is a linear combination of terms of the form:
$$
\prod_{j=1}^k \nabla^{l_j}\varepsilon^{(1)} ,
$$
for $\sum_{j=1}^k l_j=\tilde{\sigma}$. We recall the standard commutator estimate:
$$
\parallel \nabla^{\delta_{\sigma}}(uv)\parallel_{L^q}\leq C \parallel \nabla^{\delta_{\sigma}}u\parallel_{L^{p_1}}\parallel v\parallel_{L^{p_2}}+C\parallel \nabla^{\delta_{\sigma}}v\parallel_{L^{p_1'}}\parallel u\parallel_{L^{p_2'}} ,
$$
for $\frac{1}{p_1}+\frac{1}{p_2}=\frac{1}{p_1'}+\frac{1}{p_2'}=\frac{1}{q}$, provided $1<q,p_1,p_1'<+\infty$ and $1\leq p_2,p_2'\leq +\infty$. So by iteration we have that:
$$
\parallel \nabla^{\sigma-2}(Q^{p-k}\varepsilon^{(1)k} )\parallel_{L^2}\leq C \sum_{j=1}^k \prod_{i=1}^k \parallel \nabla^{l(j)_i}\varepsilon^{(1)} \parallel_{L^{p(j)_i}} ,
$$
with $l(j)_i=l_i+\delta_{\sigma}\delta_{i=j}$ and with $\sum \frac{1}{p(j)_i}=\frac{1}{2}$. We have for any $i$ and $j$: $l(j)_i<\sigma$. Hence we can use Sobolev injection to find:
$$
\nabla^{l(j)_i}\varepsilon^{(1)} \in L^{p(j)^*_i} ,
$$
for $p(j)^*_i=\frac{2d}{d-2\sigma+2l(j)_i}$. We compute (the strategy was designed to obtain this):
$$
\frac{1}{p(j)^*}:=\sum_{i=1}^k \frac{1}{p(j)^*_i}=\frac{1}{2} .
$$
So we take $p(j)_i=p(j)^*i$. We then have:
\be \label{thetrapped:lowsobo:eq:estimation NL 2}
\parallel \nabla^{\sigma-2+\frac{2}{p-1}(p-k)+(k-1)(\sigma-s_c)}(NL(\varepsilon)) \parallel_{L^2} \leq C \sqrt{\mathcal{E}_{\sigma}}^k .
\ee
The Cauchy-Schwarz inequality \fref{thetrapped:lowsobo:eq:expression NL}, with the estimates for the two terms \fref{thetrapped:lowsobo:eq:estimation NL 1} and \fref{thetrapped:lowsobo:eq:estimation NL 2} give eventually:
\be \label{thetrapped:lowsobo:eq:estimation NL}
\begin{array}{r c l}
&\left| \int \nabla^{\sigma-1}w^{(2)}.\nabla^{\sigma-1}(NL(w))\right| \\
\leq& \frac{C(K_1,K_2,M)b_1\sqrt{\mathcal{E}_{\sigma}}}{\lambda^{2(\sigma-s_c)+1}}\left(\frac{\sqrt{\mathcal{E}_{\sigma}}}{b_1^{\sigma-s_c}} \right)^{k-1}b_1^{(\sigma-s_c)(1+\nu)+\frac{\alpha}{L}+O\left( \frac{\sigma-s_c}{L}\right)} .
\end{array}
\ee

\underline{Step 2:} Gathering the bounds. We have made the decomposition \fref{thetrapped:lowsobo:eq:expression derivee} and have found an upper bound for all terms in the right hand side in \fref{thetrapped:lowsobo:eq:estimation lineaire}, \fref{thetrapped:lowsobo:eq:estimation mod}, \fref{thetrapped:lowsobo:eq:estimation psib}, \fref{thetrapped:lowsobo:eq:estimation L} and \fref{thetrapped:lowsobo:eq:estimation NL}. So we get:
\be \label{thetrapped:lowsobo:eq:final 1 }
\begin{array}{r c l}
\frac{d}{dt}\left\{ \frac{\mathcal{E}_{\sigma}}{\lambda^{2(\sigma-s_c)}} \right\} &\leq& \frac{C(K_1,K_2,M)}{\lambda^{2(\sigma-s_c)}}\frac{b_1}{\lambda}\sqrt{\mathcal{E}_{\sigma}} b_1^{(\sigma-s_c)(1+\nu)} \times \Bigl( b_1^{\frac{\alpha}{L}+O\left(\frac{\sigma-s_c}{L} \right)} +b_1^{\alpha-\nu(\sigma-s_c)}\\
&&+b_1^{\frac{3}{4}\alpha-\nu(\sigma-s_c)}+b_1^{\frac{1}{p-1}+O\left(\frac{1}{L} \right)}  +b_1^{\frac{\alpha}{L}+O\left(\frac{\sigma-s_c}{L} \right)} \sum_{k=2}^p \left( \frac{\sqrt{\mathcal{E}_{\sigma}}}{b_1^{\sigma-s_c}}\right)^{k-1} \Bigr),
\end{array}
\ee
We see that if one choose $\sigma-s_c$ small enough there holds:
\be \label{thetrapped:lowsobo:eq:final 2}
\frac{\alpha}{2L} < \text{min}\left(\frac{\alpha}{L}+O\left(\frac{\sigma-s_c}{L} \right),\alpha-\nu(\sigma-s_c),\frac{3}{4}\alpha-\nu(\sigma-s_c),\frac{1}{p-1}+O\left(\frac{1}{L} \right)\right).
\ee
In the trapped regime we recall that $b_1\sim \frac{c_1}{s}$ is small, so that $b_1^a\ll b_1^b$ if $b< a$. Consequently by taking $s_0$ big enough to "erase" the constants, \fref{thetrapped:lowsobo:eq:final 1 } combined with \fref{thetrapped:lowsobo:eq:final 2} give the result of the proposition.
\end{proof}


\subsection{Lyapunov monotonicity for the high Sobolev norm:}
We have seen that in order to control the evolution of the parameters, we need to control the high Sobolev norm $\mathcal{E}_{s_L}$. Indeed, the law of $b_L$ is computed when projecting the dynamics onto $\bos{H}^{*L}\bos{\Phi}_M$, which involves at least to control $L$ derivative. Why do we look at the $k_0+1+L$-th derivative? Because it is only when deriving at least $k_0+1$ more times that we gain something on the error term $\tilde{\bos{\psi}}_b$: the $\eta$ gain (see proposition \ref{pr:profilapprochecoupe})\footnote{this is the reason why we need or approximate profile to expand till the zone $y\sim B_1$.}. However, if we look at a higher order derivative ($>k_0+L+1$) we loose the control of the solution by lack of Hardy inequalities (Corollary \ref{annexe:cor:coercivite mathcalEsL} does not work at a higher level of regularity). For these reasons, the choice $L+k_0+1$ is sharp.\\
\\
We state here a control on the evolution of $\mathcal{E}_{s_L}$, and prove it. We will not be able to estimate it directly, a local part will require the study of a Morawetz type quantity. This is the subject of the following subsection.

\begin{proposition}\label{trappedregime:pr:high sobo}
\emph{(Lyapunov monotonicity for the high Sobolev norm:)} Suppose all the constants of Proposition \ref{prop:bootstrap} are fixed, except $s_0$ and $\eta$. Then for $s_0$ large enough and $\eta$ small enough there holds for $s_0\leq s<s^*$:
\begin{equation} \label{thetrapped:eq:high sobo}
\begin{array}{r c l}
\frac{d}{dt}\left\{ \frac{\mathcal{E}_{s_L}}{\lambda^{2(s_L-s_c)}}+O\left( \frac{\mathcal{E}_{s_L}b_1^{\eta(1-\delta_0)}}{\lambda^{2(s_L-s_c)}}\right) \right\}  &\leq& \frac{C(M)}{\lambda^{2(s_L-s_c)}}\frac{b_1}{\lambda}  \Bigl[ \mathcal{E}_{s_L}b_1^{\frac{\alpha}{2L}+O\left( \frac{\sigma-s_c}{L} \right)}\underset{k=2}{\overset{p}{\sum}} \left[\frac{\sqrt{\mathcal{E}_{\sigma}}}{b_1^{\sigma-s_c}} \right]^{k-1} \\
&&+C(N)\mathcal{E}_{s_L,\text{loc}}+\frac{\mathcal{E}_{s_L}}{N^{\frac{\delta_0}{2}}} + \sqrt{\mathcal{E}_{s_L}}b_1^{L+(1-\delta_0)(1+\eta)} \Bigr]
\end{array}
\end{equation}
the constant hidden in the $O()$ in the left hand side depending on $M$ (the norms $\mathcal{E}_{s_L}$ and $\mathcal{E}_{s_L,\text{loc}}$ are defined by \fref{thetrapped:eq:def mathcalEsL} and \fref{thetrapped:eq:def mathcalEsLloc}).
\end{proposition}

\begin{proof}[Proof of Proposition \ref{trappedregime:pr:high sobo}]: First we compute the time evolution of $\mathcal{E}_{s_L}$:
\be \label{thetrapped:highsobo:eq:expression derivee}
\begin{array}{r c l}
&\frac{d}{dt}\left( \frac{\mathcal{E}_{s_L}}{\lambda^{2(s_L-s_c)}} \right) =
\frac{d}{dt}\left( \int |w^{(1)}_{k_0+1+L}|^2+|w^{(2)}_{k_0+L}|^2 \right) \\
=& \frac{d}{dt}\left( \int w^{(1)}\mathcal{L}_{\frac{1}{\lambda}}^{k_0+L+1}w^{(1)}+w^{(2)}\mathcal{L}_{\frac{1}{\lambda}}^{k_0+L}w^{(2)} \right) \\
=& 2\int w^{(1)}\mathcal{L}_{\frac{1}{\lambda}}^{k_0+L+1}w_t^{(1)}+w^{(2)}\mathcal{L}_{\frac{1}{\lambda}}^{k_0+L}w_t^{(2)} \\
& +\sum_{i=1}^{k_0+L+1} \int w^{(1)} \mathcal{L}_{\frac{1}{\lambda}}^{i-1}\frac{d}{dt}\left( \mathcal{L}_{\frac{1}{\lambda}}\right)\mathcal{L}_{\frac{1}{\lambda}}^{k_0+L+1-i} w^{(1)} \\
&+\sum_{i=1}^{k_0+L} \int w^{(2)} \mathcal{L}_{\frac{1}{\lambda}}^{i-1}\frac{d}{dt}\left( \mathcal{L}_{\frac{1}{\lambda}}\right) \mathcal{L}_{\frac{1}{\lambda}}^{k_0+L-i} w^{(2)} \\
=& 2\int w^{(1)}\mathcal{L}_{\frac{1}{\lambda}}^{k_0+L+1} (w^{(2)}-\frac{1}{\lambda}\tilde{\psi}_{b,\frac{1}{\lambda}}^{(1)} -\frac{1}{\lambda} \tilde{Mod}_{\frac{1}{\lambda}}(t)^{(1)})\\
&+2\int w^{(2)}\mathcal{L}_{\frac{1}{\lambda}}^{k_0+L} (-\mathcal{L}_{\frac{1}{\lambda}}w^{(1)}-\frac{1}{\lambda}\tilde{\psi}_{b,\frac{1}{\lambda}}^{(2)} -\frac{1}{\lambda} \tilde{Mod}(t)_{\frac{1}{\lambda}}^{(2)}+L(w)+NL(w))\\
&+\sum_{i=1}^{k_0+L+1} \int w^{(1)} \mathcal{L}_{\frac{1}{\lambda}}^{i-1}\frac{d}{dt}\left( \mathcal{L}_{\frac{1}{\lambda}}\right)\mathcal{L}_{\frac{1}{\lambda}}^{k_0+L+1-i} w^{(1)} \\
&+\sum_{i=1}^{k_0+L} \int w^{(2)} \mathcal{L}_{\frac{1}{\lambda}}^{i-1}\frac{d}{dt}\left( \mathcal{L}_{\frac{1}{\lambda}}\right) \mathcal{L}_{\frac{1}{\lambda}}^{k_0+L-i} w^{(2)} .
\end{array}
\ee
We aim at computing the effect of everything in the right hand side.\\

\underline{Step 1:} Terms that can be estimated directly. We claim that the quadratic term, the error term and the non-linear term can be estimated directly, transforming \fref{thetrapped:highsobo:eq:expression derivee} into:
\be \label{thetrapped:highsobo:eq:expression derivee step1}
\begin{array}{r c l}
&\frac{d}{dt}\left( \frac{\mathcal{E}_{s_L}}{\lambda^{2(s_L-s_c)}} \right)\\
&= 2\int w^{(1)}\mathcal{L}_{\frac{1}{\lambda}}^{k_0+L+1} ( -\frac{1}{\lambda} \tilde{Mod}_{\frac{1}{\lambda}}(t)^{(1)})+w^{(2)}\mathcal{L}_{\frac{1}{\lambda}}^{k_0+L} ( -\frac{1}{\lambda} \tilde{Mod}(t)_{\frac{1}{\lambda}}^{(2)}+L(w))\\
&+\sum_{i=1}^{k_0+L+1} \int w^{(1)} \mathcal{L}_{\frac{1}{\lambda}}^{i-1}\frac{d}{dt}\left( \mathcal{L}_{\frac{1}{\lambda}}\right)\mathcal{L}_{\frac{1}{\lambda}}^{k_0+L+1-i} w^{(1)} \\
&+\sum_{i=1}^{k_0+L} \int w^{(2)} \mathcal{L}_{\frac{1}{\lambda}}^{i-1}\frac{d}{dt}\left( \mathcal{L}_{\frac{1}{\lambda}}\right) \mathcal{L}_{\frac{1}{\lambda}}^{k_0+L-i} w^{(2)} \\
&+\frac{b_1}{\lambda^{2(s_L-s_c)+1}} \left[O(\sqrt{\mathcal{E}_{s_L}}b_1^{L+(1-\delta_0)(1+\eta)})+O\left(\mathcal{E}_{s_L}b_1^{\frac{\alpha}{L}+O\left( \frac{\sigma-s_c}{L} \right)}\underset{k=2}{\overset{p}{\sum}} \left[\frac{\sqrt{\mathcal{E}_{\sigma}}}{b_1^{\sigma-s_c}} \right]^{k-1}  \right)\right] .
\end{array}
\ee
where the constant hidden in the first $O()$ does not depend on $K_1$ and $K_2$. We now prove this intermediate estimate.\\
$\bullet$ \emph{The linear term:} Because this norm is adapted to the flow of the wave equation we have the fundamental cancellation:
\be \label{thetrapped:highsobo:eq:estimation lineaire}
\int w^{(1)} \mathcal{L}_{\frac{1}{\lambda}}^{k_0+L+1}w^{(2)}+w^{(2)} \mathcal{L}_{\frac{1}{\lambda}}^{k_0+L}(-\mathcal{L}_{\frac{1}{\lambda}}w^{(1)})=0 .
\ee
$\bullet$ \emph{the $\tilde{\bos{\psi}}_b$ term:} It is this term that gives the eventual estimate for $\mathcal{E}_{s_L}$ we want to prove. We recall that $f_j$, the $j$-th adapted derivative of a function $f$, is defined in \fref{linearized:eq:def derivees adaptees}. We just use Cauchy-Schwarz and the estimate provided in Proposition \ref{pr:profilapprochecoupe}:
\be \label{thetrapped:highsobo:eq:estimation psib}
\begin{array}{r c l}
& \left|\frac{1}{\lambda}\int w^{(1)} \mathcal{L}_{\frac{1}{\lambda}}^{k_0+L+1}\tilde{\psi}_{b,\frac{1}{\lambda}}^{(1)}+w^{(2)} \mathcal{L}_{\frac{1}{\lambda}}^{k_0+L}\tilde{\psi}_{b,\frac{1}{\lambda}}^{(2)}\right| \\
=& \left|\frac{1}{\lambda}\int w^{(1)}_{k_0+1+L} \left(\tilde{\psi}_{b,\frac{1}{\lambda}}^{(1)}\right)_{k_0+1+L}+w^{(2)}_{k_0+L} \left( \tilde{\psi}_{b,\frac{1}{\lambda}}^{(2)}\right)_{k_0+L}\right| \\
\leq & C\frac{1}{\lambda} \frac{1}{\lambda^{2(s_L-s_c)}} \sqrt{\mathcal{E}_{s_L}} b_1^{1+L+(1-\delta_0)(1+\eta)} .
\end{array}
\ee
for a constant C depending on $L$ only.\\
$\bullet$ \emph{The non linear term:} We begin by Cauchy-Schwarz inequality and by doing a change on the scaling:
$$
\left|\int w^{(2)}\mathcal{L}_{\frac{1}{\lambda}}^{k_0+L}NL(w)\right| \leq \frac{1}{\lambda^{2(s_L-s_c)+1}}\sqrt{\mathcal{E}_{s_L}}\parallel NL(\varepsilon)_{k_0+L}\parallel_{L^2} .
$$
We aim at estimating the last term in the right hand side. We know that $NL(\varepsilon)$ is a sum of terms of the form $\tilde{Q}_b^{(1)(p-k)}\varepsilon^{(1)k}$ for $2\leq k \leq p$. So by now we have to study quantities of the form: $(\tilde{Q}_b^{(1)(p-k)}\varepsilon^{(1)k})_{k_0+L}$. For $l=(l_0,...,l_k)$ we recall the notation: $|l|_1=\sum_{i=0}^k l_i$. Close to the origin, we have from the equivalence between Sobolev norms and adapted norms (Lemma \ref{equivalence of norms}), and because $H^{s_L}(y\leq 1)$ is an algebra:
$$
\int_{y\leq 1} \left(NL(\varepsilon)_{k_0+L} \right)^2\leq C \sum_{k=2}^p\parallel \varepsilon^{(1)} \parallel^{2k}_{H^{s_L}(y\leq 1)} \leq C(M) \mathcal{E}_{s_L}\leq C(M) \sqrt{\mathcal{E}_{s_L}}b_1^2 .
$$
For $y\geq 1$ we notice that when applying $A$ and $A^*$:
$$
(\tilde{Q}_b^{(1)(p-k)}\varepsilon^{(1)k})_{k_0+L}=\sum_{|l|_1=k_0+L} f_{\tilde{l_0}}\partial_y^{l_0}(\tilde{Q}_b^{(1)(p-k)})\prod_{i=1}^k \partial_y^{l_i} \varepsilon^{(1)} .
$$
with $f_{\tilde{l_0}}\sim \frac{1}{1+y^{\tilde{l_0}}}$. We have the following asymptotic for the potential: 
$$
\partial_y^{l_0}(\tilde{Q}_b^{(1)(p-k)})\leq \frac{C}{1+y^{\frac{2}{p-1}(p-k)+l_0}} .
$$
So, putting together the decay given by $\partial_y^{l_0}\tilde{Q}_b^{(1)(p-k)}$ and $f_{\tilde{l_0}}$ and renaming $l_0:=l_0+\tilde{l_0}$ we need to study integrals of the following form:
\be \label{thetrapped:highsobo:eq:estimation inter NL}
\int_{y\geq 1} |NL(\varepsilon)_{s_L-1}|^2 \lesssim \sum_{k=2}^p \sum_{|l|_1=k_0+L} \int_{y\geq 1} \frac{\prod_{i=1}^k |\partial_y^{l_i} \varepsilon^{(1)}|^2}{1+y^{\frac{4}{p-1}(p-k)+2l_0}} ,
\ee
for $\sum_{i=0}^k l_i=s_l-1$. We order the coefficient $l_i$ for $1\leq i \leq k$ by increasing order: $l_1\leq l_2\leq...\leq l_k$. \\
$\circ$ \emph{Case 1:} we suppose that: $\frac{2}{p-1}(p-k)+l_0+l_k\leq s_L$. It implies the integrability $\frac{\partial_y^{l_k}\varepsilon^{(1)}}{1+y^{l_0+\frac{2}{p-1}(p-k)}}\in L^2(y\geq 1)$ by the improved Hardy inequality from Lemma \ref{annexe:lem:interpolation varepsilon}. There also holds in that case for all $1\leq i \leq k-1$ that $l_i< s_L-\frac{d}{2} $ which implies $\partial_y^{l_i}\varepsilon^{(1)}\in L^{\infty}(y\geq 1)$. We then estimate:
$$
\left\Vert \frac{\prod_{i=1}^k |\partial_y^{l_i} \varepsilon^{(1)}|^2}{1+y^{\frac{4}{p-1}(p-k)+2l_0}} \right\Vert_{L^2(y\geq 1)} \leq C \left\Vert \frac{\partial_y^{l_k}\varepsilon^{(1)}}{1+y^{\frac{2}{p-1}(p-k)+l_0}}\right\Vert_{L^2(\geq 1)} \prod_{i=1}^{k-1} \left\Vert \partial_y^{l_i}\varepsilon^{(1)}\right\Vert \parallel_{L^{\infty}(y\geq 1)} .
$$
For $1\leq i \leq k-1$, from the equivalence between Laplace and $\partial_y$ derivatives for $y\geq 1$:
$$
\partial_y^{l_i} \varepsilon^{(1)} = \sum_{j=0}^{l_i}f_jD^{j}\varepsilon^{(1)} ,
$$
with $\partial_y^n f_j=O\left( \frac{1}{1+x^{l_i-j+n}}\right)$ for $y \geq 1$, we deduce:
$$
\begin{array}{r c l}
\parallel \partial_y^{l_i}\varepsilon^{(1)}\parallel_{L^{\infty}(y\geq 1)} &\leq& C \sum_{j=0}^{l_i} \parallel \frac{D^{j}\varepsilon^{(1)}}{1+x^{l_i-j}} \parallel_{L^{\infty}}\\
&\leq & C \sqrt{\mathcal{E}_{\sigma}}^{\frac{s_L-l_i-\frac{d}{2}}{s_L-\sigma}} \sqrt{\mathcal{E}_{s_L}}^{\frac{l_i+\frac{d}{2}-\sigma}{s_L-\sigma}} .
\end{array}
$$
We used Sobolev injection, interpolation and coercivity. For $i=k$ from Lemma \ref{annexe:lem:interpolation varepsilon}:
$$
\left\Vert \frac{\partial_y^{l_k}\varepsilon^{(1)}}{1+y^{\frac{2}{p-1}(p-k)+2l_0}}\right\Vert_{L^2(y\geq 1)} \leq C(M) \sqrt{\mathcal{E}_{\sigma}}^{\frac{s_L-l_k-l_0-\frac{2}{p-1}(p-k)}{s_L-\sigma}} \sqrt{\mathcal{E}_{s_L}}^{\frac{l_k+l_0+\frac{2}{p-1}(p-k)-\sigma}{s_L-\sigma}} .
$$
So that when combining the last two estimates we find:
$$
\begin{array}{r c l}
&\left\Vert \frac{\prod_{i=1}^k |\partial_y^{l_i} \varepsilon^{(1)}|^2}{1+y^{\frac{4}{p-1}(p-k)+2l_0}} \right\Vert_{L^2(y\geq 1)} \\
\leq& C \sqrt{\mathcal{E}_{\sigma}}^{\sum_{i=1}^{k-1} \left(\frac{s_L-l_i-\frac{d}{2}}{s_L-\sigma}\right)+\frac{s_L-l_k-l_0-\frac{2}{p-1}(p-k)}{s_L-\sigma}} \sqrt{\mathcal{E}_{s_L}}^{\sum_{i=1}^{k-1}\left( \frac{l_i+\frac{d}{2}-\sigma}{s_L-\sigma}\right)+\frac{l_k+l_0+\frac{2}{p-1}(p-k)-\sigma}{s_L-\sigma}} 
\end{array}
$$
$C$ depending on $M$. We can calculate the coefficients:
$$
\sum_{i=2}^k \left(\frac{s_L-l_i-\frac{d}{2}}{s_L-\sigma}\right)+\frac{s_L-l_k-l_0-\frac{2}{p-1}(p-k)}{s_L-\sigma}=\frac{(k-1)(s_L-\frac{d}{2}+1-\frac{2}{p-1}(p-k)}{s_L-\sigma} ,
$$
$$
\sum_{i=1}^{k-1}\left( \frac{l_i+\frac{d}{2}-\sigma}{s_L-\sigma}\right)+\frac{l_k+l_0+\frac{2}{p-1}(p-k)-\sigma}{s_L-\sigma}=1+\frac{1-(k-1)(\sigma-s_c)}{s_L-\sigma} .
$$
Under the bootstrap assumptions \fref{eq:bootstrap estimations sur epsilon} it gives:
\be \label{thetrapped:highsobo:eq:NL cas 1}
\left\Vert \frac{\prod_{i=1}^k |\partial_y^{l_i} \varepsilon^{(1)}|^2}{1+y^{\frac{4}{p-1}(p-k)+2l_0}} \right\Vert_{L^2}\leq C(K_1,K_2,M) b_1\sqrt{\mathcal{E}_{s_L}}\left( \frac{\sqrt{\mathcal{E}_{\sigma}}}{b_1^{\sigma-s_c}}\right)^{k-1}b_1^{\frac{\alpha}{L}+O\left( \frac{\sigma-s_c}{L}\right)} .
\ee
$\circ$ \emph{Case 2:} if the last condition does not hold, it implies that $l_k+l_0=s_L-1$ with $\frac{2}{p-1}(p-k)-1>0$, and that consequently for $1\leq i \leq k-1$, $l_i=1$. It means that we have to estimate an integral of the following form:
$$
\int_{y\geq 1} |\varepsilon^{(1)}|^{2(k-1)}  \frac{|\partial_y^{l_k}\varepsilon^{(1)}|^2}{1+y^{\frac{4}{p-1}(p-k)+2l_0}} .
$$
We rewrite it as:
$$
\int_{y\geq 1} |\varepsilon^{(1)}|^{2(k-2)}\frac{|\varepsilon^{(1)}|^2}{1+y^{\frac{4}{p-1}(p-k)-2}}  \frac{|\partial_y^{l_k}\varepsilon^{(1)}|^2}{1+y^{2+2l_0}} .
$$
The $L^{\infty}$ norm of $\varepsilon^{(1)}$ is estimated in Lemma \ref{annexe:lem:interpolation varepsilon}:
$$
\parallel \varepsilon^{(1)} \parallel_{L^{\infty}} \leq C(M,K_1,K_2) \sqrt{\mathcal{E}_{\sigma}}b_1^{\left( \frac{d}{2}-\sigma\right)+\frac{2\alpha}{(p-1)L} +O\left( \frac{\sigma-s_c}{L} \right) } 
$$
We use the improved Hardy estimate from Lemma \ref{annexe:lem:interpolation varepsilon} to estimate:
$$
\left\Vert \frac{\partial_y^{l_k}\varepsilon^{(1)}}{1+y^{1+l_0}} \right\Vert_{L^2(y\geq 1)}\leq C(M) \sqrt{\mathcal{E}_{s_L}} .
$$
And finally we use the weighted $L^{\infty}$ estimate (still from Lemma \ref{annexe:lem:interpolation varepsilon}):
$$
\left\Vert \frac{|\varepsilon^{(1)}|}{1+y^{\frac{2}{p-1}(p-k)-1}} \right\Vert_{L^{\infty}}\leq C(M,K_1,K_2) \sqrt{\mathcal{E}_{\sigma}}b_1^{\frac{2(p-k)}{p-1}-1+(\frac{d}{2}-\sigma)+\frac{2\alpha}{(p-1)L}+ O(\frac{(\sigma-s_c)}{L})} .
$$
With these last three estimates we have:
\be \label{thetrapped:highsobo:eq:NL cas 2}
\begin{array}{r c l}
&\left\Vert|\varepsilon^{(1)}|^{(k-2)}\frac{|\varepsilon^{(1)}|}{1+y^{\frac{2}{p-1}(p-k)-1}}  \frac{|\partial_y^{l_k}\varepsilon^{(1)}|}{1+y^{1+l_0}} \right\Vert_{L^2(y\geq 1)} \\
\lesssim & \sqrt{\mathcal{E}_{s_L}} \sqrt{\mathcal{E}_{\sigma}}^{(k-2)+1}b_1^{(k-2)(\frac{d}{2}-\sigma)+\frac{2(k-2)\alpha}{(p-1)L}} b_1^{\frac{2(p-k)}{p-1}-1+(\frac{d}{2}-\sigma)+(\frac{2}{p-1}+\frac{2(p-k)}{p-1}-1)\frac{\alpha}{L} +O(\frac{(\sigma-s_c)}{L})} \\
\leq& C(M,K_1,K_2) \sqrt{\mathcal{E}_{s_L}}b_1 \left( \frac{\sqrt{\mathcal{E}}_{\sigma}}{b_1^{\sigma-s_c}}\right)^{k-1}b_1^{\frac{\alpha}{L}+O\left(\frac{\sigma-s_c}{L}\right)} .
\end{array}
\ee
We now come back to \fref{thetrapped:highsobo:eq:estimation inter NL} and inject the bounds we have found. Putting together the result obtained in case 1, \fref{thetrapped:highsobo:eq:NL cas 1} and the result obtained in the second case, \fref{thetrapped:highsobo:eq:NL cas 2}, gives for the non linear term: 
\be \label{thetrapped:highsobo:eq:estimation NL}
\left| \int w^{(2)}\mathcal{L}_{\frac{1}{\lambda}}^{k_0+L}(NL(w))\right|\leq \frac{C(K_1,K_2,M)}{\lambda^{2(s_L-s_c)}}\frac{b_1}{\lambda}\mathcal{E}_{s_L}\left[ \sum_{k=2}^{p} \left( \frac{\sqrt{\mathcal{E}_{\sigma}}}{b_1^{\sigma-s_c}}\right)^{k-1}\right] b_1^{\frac{\alpha}{L}+O(\frac{(\sigma-s_c)}{L})} .
\ee
We now recapitulate: we have found directs bounds for the quadratic term \fref{thetrapped:highsobo:eq:estimation lineaire}, for the error term \fref{thetrapped:highsobo:eq:estimation psib}, and for the non linear term \fref{thetrapped:highsobo:eq:estimation NL}. We inject them in \fref{thetrapped:highsobo:eq:expression derivee} to obtain the intermediate identity \fref{thetrapped:highsobo:eq:expression derivee step1}, wich we claimed in this step 1.\\

\underline{Step 2:} Terms for which only a local part is problematic. The small linear term and the scale changing term involve a potential that, in both cases, has a better decay than $\frac{1}{y^2}$ far away of the origin. So away from the origin we can control them directly. Unfortunately, close to the origin we cannot. This is why we will have to use an additional tool, the study of a Morawetz type quantity, which will be done in the next subsection. We claim that \fref{thetrapped:highsobo:eq:expression derivee step1} yields:
\be \label{thetrapped:highsobo:eq:expression derivee step2}
\begin{array}{r c l}
&\frac{d}{dt}\left( \frac{\mathcal{E}_{s_L}}{\lambda^{2(s_L-s_c)}} \right) = 2\int w^{(1)}\mathcal{L}_{\frac{1}{\lambda}}^{k_0+L+1} ( -\frac{1}{\lambda} \tilde{Mod}_{\frac{1}{\lambda}}(t)^{(1)})+w^{(2)}\mathcal{L}_{\frac{1}{\lambda}}^{k_0+L} ( -\frac{1}{\lambda} \tilde{Mod}(t)_{\frac{1}{\lambda}}^{(2)})\\
&+\frac{b_1}{\lambda^{2(s_L-s_c)+1}} \left[O(\sqrt{\mathcal{E}_{s_L}}b_1^{L+(1-\delta_0)(1+\eta)})+O\left(\mathcal{E}_{s_L}b_1^{\frac{\alpha}{L}+O\left( \frac{\sigma-s_c}{L} \right)}\underset{k=2}{\overset{p}{\sum}} \left[\frac{\sqrt{\mathcal{E}_{\sigma}}}{b_1^{\sigma-s_c}} \right]^{k-1}\right)\right]\\
&+\frac{b_1}{\lambda^{2(s_L-s_c)+1}}O\left( \frac{\mathcal{E}_{s_L}}{N^{\delta}}+C(N)\mathcal{E}_{s_L,\text{loc}} \right) .
\end{array}
\ee
We are now going to prove this identity \fref{thetrapped:highsobo:eq:expression derivee step2} by establishing bounds on the small linear term and the scale changing term in \fref{thetrapped:highsobo:eq:expression derivee step1}.\\
$\bullet$ \emph{The $L(w)$ term:} We start by rescaling and using Cauchy-Schwarz:
$$
\left| \int w^{(2)} \mathcal{L}_{\frac{1}{\lambda}}^{k_0+L}(L(w))\right| \leq \frac{1}{\lambda^{2(s_L-s_c)+1}} \sqrt{\mathcal{E}_{s_L}} \parallel (L(\varepsilon))_{k_0+L}\parallel_{L^2} .
$$
We have: $L(\varepsilon)=p(Q^{p-1}-\tilde{Q}_b^{(1)(p-1)})\varepsilon^{(1)}$. From the asymptotic of the the profiles $T_i$ and $S_i$ there holds the degeneracy:
$$
|\partial_y^j (Q^{p-1}-\tilde{Q}_b^{(1)(p-1)})| \leq C(L) \frac{b_1}{1+y^{1+\alpha+j-C(L)\eta}} ,
$$
Let\footnote{We cannot expect to gain the weight $y^{-\alpha}$ because if $\alpha$ is too big the weighted coercivity does not apply. The limiting case is $\delta_0$ hence our choice for $\delta$.} $\delta=\frac{\delta_0}{2}$. We first estimate the integral close to the origin. $H^{s_L-1}(y\leq 1)$ is an algebra, from the equivalence between Laplace based derivatives and adapted ones (see Lemma \ref{equivalence of norms}), and from the weighted coercivity (Lemma \ref{annexe:lem:coercivite des normes adaptees}):
$$
\int_{y\leq 1}(L(\varepsilon))_{s_L-1}^2 \leq Cb_1^2 \int_{y\leq 1} \sum_{i=0}^{s_L} |D^{i}\varepsilon^{(1)}|^2 \leq C(M) b_1^2 \int \frac{|\varepsilon^{(1)}_{s_L}|}{1+y^{2\delta}} .
$$
Away from the origin we estimate using the weighted coercivity and the equivalence between $\partial_y$ derivatives and adapted derivatives (Lemma \ref{annexe:lem:equivalence degre}).
$$
\begin{array}{r c l}
\parallel (L(\varepsilon^{(1)})_{k_0+1}\parallel_{L^2(y\geq 1)}^2 &\leq & C\sum_{i=0}^{s_L-1} \parallel  \frac{b_1|\varepsilon^{(1)}_{i}|}{1+y^{1+\alpha+s_L-1-i-C\eta}} \parallel_{L^2(y\geq 1)}^2 \\
&\leq& C(M)b_1^2 \parallel \frac{\varepsilon^{(1)}_{s_L}}{1+y^{\delta}}\parallel_{L^2}^2 .
\end{array}
$$
With the two estimates, close and away from the origin, we have shown:
\be \label{thetrapped:highsobo:eq:estimation intermediaire L}
\parallel (L(\varepsilon))_{s_L-1} \parallel_{L^2}^2 \lesssim b_1^2 \left\Vert \frac{\varepsilon^{(1)}_{s_L}}{1+y^{\delta}} \right\Vert_{L^2}^2 .
\ee
We now split the term of the right hand side in two parts, one before $N$ and the other after, where $N>0$ is the large constant used in the definition of the local adapted norm (see \fref{thetrapped:eq:def mathcalEsLloc}):
$$
\left\Vert b_1\frac{\varepsilon^{(1)}_{s_L}}{1+y^{\delta}}\right\Vert_{L^2}\leq b_1 \parallel \varepsilon^{(1)}_{s_L}\parallel_{L^2(\leq N)}+b_1\frac{1}{N^{\delta}} \parallel \varepsilon^{(1)}_{s_L}\parallel_{L^2(\geq N)} .
$$
Finally:
$$
\left| \int w^{(2)}\mathcal{L}_{\frac{1}{\lambda}}^{s_L-1}(L(w))\right| \lesssim \frac{C(M)}{\lambda^{2(s_L-s_c)}}\frac{b_1\sqrt{\mathcal{E}_{s_L}}}{\lambda} \left(\frac{\sqrt{\mathcal{E}_{s_L}}}{N^{\delta}} +C(N)\sqrt{\mathcal{E}_{s_L,\text{loc}}} \right) .
$$
We now use Youngs inequality to reformulate it as:
\be \label{thetrapped:highsobo:eq:estimation L}
\left| \int w^{(2)}\mathcal{L}_{\frac{1}{\lambda}}^{s_L-1}(L(w))\right| \leq \frac{C(M)}{\lambda^{2(s_L-s_c)}}\frac{b_1}{\lambda} \left(\frac{\mathcal{E}_{s_L}}{N^{\delta}} +C(N) \mathcal{E}_{s_L,\text{loc}} \right) .
\ee
$\bullet$ \emph{The scale changing term:} The same reasoning applies to the scale changing term. Indeed one has:
$$
\frac{d}{dt}(\mathcal{L}_{\frac{1}{\lambda}})=-\frac{\lambda_s}{\lambda^2} pQ^{p-2}_{\frac{1}{\lambda}}(\Lambda^{(1)}Q)_{\frac{1}{\lambda}}=-\frac{\lambda_s}{\lambda^{4}}\tilde{V}\left(\frac{y}{\lambda}  \right)
$$
where the potential $\tilde{V}$ satisfies an improved decay property:
$$
\left|\partial_y^jV \right| \leq \frac{C}{1+\lambda y^{2+\alpha+j}}.
$$
Consequently, as $-\frac{\lambda_s}{\lambda}\approx b_1$ from the modulation equations, we have the same gain of a weight $y^{-\alpha}$ we had for the small linear term. Using verbatim the same techniques one obtains:
\be \label{thetrapped:highsobo:eq:estimation lambdaL}
\begin{array}{r c l}
& \left| \int \sum_{i=1}^{s_L} w^{(1)} \mathcal{L}_{\frac{1}{\lambda}}^{i-1}\frac{d}{dt}(\mathcal{L}_{\frac{1}{\lambda}})\mathcal{L}_{\frac{1}{\lambda}}^{s_L-i}w^{(1)}+\int \sum_{i=1}^{s_L-1} w^{(2)} \mathcal{L}_{\frac{1}{\lambda}}^{i-1}\frac{d}{dt}(\mathcal{L}_{\frac{1}{\lambda}})\mathcal{L}_{\frac{1}{\lambda}}^{s_L-1-i}w^{(2)}\right| \\
\leq &\frac{C(M)}{\lambda^{2(s_l-s_c)}}\frac{b_1}{\lambda} \left(\frac{\mathcal{E}_{s_L}}{N^{\delta}} +C(N) \mathcal{E}_{s_L,\text{loc}} \right) ,
\end{array}
\ee
We now come back to the identity \fref{thetrapped:highsobo:eq:expression derivee step1} established in step 1, and inject the bounds on the small linear term \fref{thetrapped:highsobo:eq:estimation L} and on the scale changing term \fref{thetrapped:highsobo:eq:estimation lambdaL}. This gives the identity \fref{thetrapped:highsobo:eq:expression derivee step2}  we claimed in this step 2.\\

\underline{Step 3:} Managing the modulation term. Eventually, we have to estimate the influence of the modulation term on \fref{thetrapped:highsobo:eq:expression derivee step2}. We claim that:
\be \label{thetrapped:highsobo:eq:estimation mod}
\begin{array}{r c l}
&\int w^{(1)}\mathcal{L}_{\frac{1}{\lambda}}^{s_L} \frac{1}{\lambda}\tilde{Mod}(t)_{\frac{1}{\lambda}}^{(1)} + \int w^{(2)}\mathcal{L}_{\frac{1}{\lambda}}^{s_L-1} \frac{1}{\lambda}\tilde{Mod}(t)_{\frac{1}{\lambda}}^{(2)} \\
=& \frac{d}{dt} O\left[ \frac{\mathcal{E}_{s_L}}{\lambda^{2(s_L-s_c)}}b_1^{\eta(1-\delta_0)} \right] +O\left( \frac{b_1\mathcal{E}_{s_L}}{\lambda^{2(s_L-s_c)+1}}b_1^{\eta(1-\delta_0)} + \frac{b_1\sqrt{\mathcal{E}_{s_L}}}{\lambda^{2(s_L-s_c)+1}}b_1^{L+(1-\delta_0)(1+2\eta)} \right) .
\end{array}
\ee
Once this bound is proven, we can finish the proof of the proposition by injecting it in \fref{thetrapped:highsobo:eq:expression derivee step2}. So to finish to proof, we will now prove \fref{thetrapped:highsobo:eq:estimation mod}. For $1\leq i \leq L-1$, the bound \fref{thetrapped:eq:modulation leq L-1} we found for the modulation equations provides a sufficient estimate for the terms $(b_{i,s}+(i-\alpha)b_1b_i-b_{i+1})(\bos{T}_i+\sum \frac{\partial \bos{S}_j}{\partial b_i})$. Indeed, pick an indice $1\leq i \leq L-1$ and suppose it is even (the odd case being exactly the same). We calculate:
$$
\begin{array}{r c l}
&\left| \frac{1}{\lambda} \int w^{(1)} \mathcal{L}_{\frac{1}{\lambda}}^{s_L}((b_{i,s}+(i-\alpha)b_1b_i-b_{i+1})\chi_{B_1}(T_i+\underset{{j=i+1, \ j \ \text{even}}}{\overset{L+2}{\sum}} \frac{\partial S_j}{\partial b_i}))_{\frac{1}{\lambda}} \right| \\
&+\left| \frac{1}{\lambda}\int w^{(2)} \mathcal{L}_{\frac{1}{\lambda}}^{s_L-1}((b_{i,s}+(i-\alpha)b_1b_i-b_{i+1})\chi_{B_1}(\underset{{j=i+1, \ j \ \text{odd}}}{\overset{L+2}{\sum}} \frac{\partial S_j}{\partial b_i}))_{\frac{1}{\lambda}} \right| \\
\leq &\frac{C(M)\sqrt{\mathcal{E}_{s_L}}}{\lambda^{2(s_L-s_c)}}(b_1\sqrt{\mathcal{E}_{s_L}}+b_1^{L+3})\left\Vert \left( \chi_{B_1}\left( T_i+\underset{{j=i+1, \ j \ \text{even}}}{\overset{L+2}{\sum}} \frac{\partial S_j}{\partial b_i}\right)\right)_{s_L} \right\Vert_{L^2} \\
& + \frac{C(M)\sqrt{\mathcal{E}_{s_L}}}{\lambda^{2(s_L-s_c)}}(b_1\sqrt{\mathcal{E}_{s_L}}+b_1^{L+3})\left\Vert \left( \chi_{B_1}\left( \underset{{j=i+1, \ j \ \text{odd}}}{\overset{L+2}{\sum}} \frac{\partial S_j}{\partial b_i}\right)\right)_{s_L-1} \right\Vert_{L^2} .
\end{array}
$$
Since:
$$
\left\Vert \left( \chi_{B_1}\Bigl{(} T_i+\underset{{j=i+1, \ \text{even}}}{\overset{L+2}{\sum}} \frac{\partial S_j}{\partial b_i}\Bigr{)}\right)_{s_L} \right\Vert_{L^2}+\left\Vert \left( \chi_{B_1}\Bigl{(} \underset{{j=i+1, \ \text{odd}}}{\overset{L+2}{\sum}} \frac{\partial S_j}{\partial b_i}\Bigr{)}\right)_{s_L-1} \right\Vert_{L^2}  \leq Cb_1^{(L-i)}  
$$
and that we assumed $i<L$, this bound implies the following identity for the modulation term:
\be \label{thetrapped:highsobo:eq:estimation mod 1}
\begin{array}{r c l}
&\int w^{(1)}\mathcal{L}_{\frac{1}{\lambda}}^{s_L} \frac{1}{\lambda}\tilde{Mod}(t)_{\frac{1}{\lambda}}^{(1)} + \int w^{(2)}\mathcal{L}_{\frac{1}{\lambda}}^{s_L-1} \frac{1}{\lambda}\tilde{Mod}(t)_{\frac{1}{\lambda}}^{(2)} \\
=& \frac{1}{\lambda} \int w^{(1)}\mathcal{L}_{\frac{1}{\lambda}}^{s_L}((b_{L,s}+(L-\alpha)b_1b_L)\chi_{B_1}( \frac{\partial S_{L+1}}{\partial b_L}))_{\frac{1}{\lambda}} +\frac{b_1O\left(b_1 \mathcal{E}_{s_L}+\sqrt{\mathcal{E}_{s_L}}b_1^{L+3}  \right)}{\lambda^{2(s_L-s_c)+1}} \\
&+ \frac{1}{\lambda} \int w^{(2)}\mathcal{L}_{\frac{1}{\lambda}}^{s_L-1}((b_{L,s}+(L-\alpha)b_1b_L)\chi_{B_1}(T_L+\frac{\partial S_{L+2}}{\partial b_L}))_{\frac{1}{\lambda}}
\end{array}
\ee
The bad term is the last one for $i=L$. But we know by the improved bound for the evolution of $b_L$, see Lemma \ref{trappedregime:improvedmodulation:lem:improvedmodulation} that $b_{L,s}+(L-\alpha)b_1b_L$ is small enough up to the derivative in time of the projection of $\varepsilon$ onto $\bos{H}^{*L}\chi_{B_1}\bos{\Lambda Q}$. Let\footnote{$\bos{\xi}$ can be seen as the coordinate of $\bos{\varepsilon}$ along the vector $\chi_{B_0}T_L$.}:
\be \label{thetrapped:highsobo:eq:definition xi}
\begin{array}{r c l}
\bos{\xi} &:=&  \frac{\langle \bos{H}^{L}\bos{\varepsilon} , \chi_{B_0} \bos{\Lambda} \bos{Q}\rangle}{\Bigl\langle  \chi_{B_0} \Lambda^{(1)} Q, \Lambda^{(1)} Q +(-1)^{\frac{L-1}{2}}\left(\frac{\partial S_{L+2}}{\partial b_L}\right)_{L-1}\Bigr\rangle} \left[\chi_{B_1}\left(\bos{T}_L+\frac{\partial \bos{S}_{L+1}}{\partial b_L}+\frac{\partial \bos{S}_{L+2}}{\partial b_L} \right)\right]_{\frac{1}{\lambda}} \\
&:=& C(\xi) \left[\chi_{B_1}\left( \bos{T}_L+\frac{\partial \bos{S}_{L+1}}{\partial b_L}+\frac{\partial \bos{S}_{L+2}}{\partial b_L} \right) \right]_{\frac{1}{\lambda}}
\end{array}
\ee
We claim that the bad part of the $L$-th modulation term can be integrated in time the following way:
\be \label{thetrapped:highsobo:eq:integration mod}
\begin{array}{r c l}
&\frac{d}{dt} \left( \int w^{(1)}\mathcal{L}_{\frac{1}{\lambda}}^{s_L} \xi^{(1)} +\int w^{(2)}\mathcal{L}_{\frac{1}{\lambda}}^{s_L-1} \xi^{(2)} +\frac{1}{2}\int \xi^{(1)}\mathcal{L}_{\frac{1}{\lambda}}^{s_L} \xi^{(1)} +\frac{1}{2} \int \xi^{(2)} \mathcal{L}_{\frac{1}{\lambda}}^{s_L-1} \xi^{(2)}  \right) \\
=& \frac{1}{\lambda} \int w^{(1)}\mathcal{L}_{\frac{1}{\lambda}}^{s_L}((b_{L,s}+(L-\alpha)b_1b_L)\chi_{B_1}( \frac{\partial S_{L+1}}{\partial b_L}))_{\frac{1}{\lambda}}   \\
&+   \frac{1}{\lambda} \int w^{(2)}\mathcal{L}_{\frac{1}{\lambda}}^{s_L-1}((b_{L,s}+(L-\alpha)b_1b_L)\chi_{B_1}(T_L+\frac{\partial S_{L+2}}{\partial b_L}))_{\frac{1}{\lambda}}  \\
&+\frac{b_1}{\lambda^{2(s_L-s_c)+1}}O(\mathcal{E}_{s_L}b_1^{\eta(1-\delta_0)})+\frac{b_1}{\lambda^{2(s_L-s_c)+1}}O(\sqrt{\mathcal{E}_{s_L}}b_1^{L+(1+2\eta)(1-\delta_0)})
\end{array}
\ee   
We will prove this identity at the end of this step 3. Once it is established, it allows us to prove the identity \fref{thetrapped:highsobo:eq:estimation mod}. Indeed, \fref{thetrapped:highsobo:eq:estimation mod 1} can be rewritten as:
\be \label{thetrapped:highsobo:eq:estimation mod 2}
\begin{array}{r c l}
&\int w^{(1)}\mathcal{L}_{\frac{1}{\lambda}}^{s_L} \frac{1}{\lambda}\tilde{Mod}(t)_{\frac{1}{\lambda}}^{(1)} + \int w^{(2)}\mathcal{L}_{\frac{1}{\lambda}}^{s_L-1} \frac{1}{\lambda}\tilde{Mod}(t)_{\frac{1}{\lambda}}^{(2)} \\
=&  \frac{d}{dt} \left( \int w^{(1)}\mathcal{L}_{\frac{1}{\lambda}}^{s_L} \xi^{(1)} +\int w^{(2)}\mathcal{L}_{\frac{1}{\lambda}}^{s_L-1} \xi^{(2)}-\frac{1}{2}\int \xi^{(1)}\mathcal{L}_{\frac{1}{\lambda}}^{s_L} \xi^{(1)} -\frac{1}{2} \int \xi^{(2)} \mathcal{L}_{\frac{1}{\lambda}}^{s_L-1} \xi^{(2)}  \right)  \\
&+\frac{b_1\sqrt{\mathcal{E}_{s_L}}}{\lambda^{2(s_L-s_c)+1}}O\left(b_1^{\eta(1-\delta_0)} \sqrt{\mathcal{E}_{s_L}}+b_1^{L+(1-\delta_0)(1+2\eta)}  \right)
\end{array}
\ee
We just have to check the gain obtained by the time integration. From the two estimates \fref{thetrapped:eq:improved modulation taille denominateur} and \fref{thetrapped:eq:improved modulation estimation numerateur} we used in the proof of the improved modulation equation, one has the following size for the coefficient $C(\xi)$:
\be \label{thetrapped:highsobo:estimation C(xi)}
|C(\xi)|\lesssim \sqrt{\mathcal{E}_{s_L}}b_1^{\delta_0-1}.
\ee
From the construction of the profiles $S_i$ in Proposition \ref{pr:constructionprofilnoncoupe}, one has the following asymptotics:
\be \label{thetrapped:highsobo:asymptotic SL+1 SL+2}
\left| \partial_y^j \left( \frac{\partial S_{L+1}}{\partial b_L}\right) \right|\leq  \frac{C(L)b_1}{1+y^{\gamma-L-1+g'+j}}, \ \ \text{and} \ \ \left| \partial_y^j \left( \frac{\partial S_{L+2}}{\partial b_L}\right) \right|\leq \frac{C(L)b_1^2}{1+y^{\gamma-L-1+g'+j}}.
\ee
The cancellation $\mathcal{L}^{\frac{L+1}{2}}T_L=0$ implies that the support of $(\chi_{B_1}T_L)_{s_L-1}$ is in the zone $B_1\leq y \leq 2B_1$, hence $\parallel (\chi_{B_1}T_L)_{s_L-1} \parallel_{L^2}\lesssim b_1^{(1-\delta_0)(1+\eta)}$. The two last estimates then imply:
\be \label{thetrapped:eq:gain inte dernier term mod}
\begin{array}{r c l}
&\left|  \int w^{(1)}\mathcal{L}_{\frac{1}{\lambda}}^{s_L} \xi^{(1)}  +  w^{(2)}\mathcal{L}_{\frac{1}{\lambda}}^{s_L-1} \xi^{(2)} \right| \\
\leq& \frac{\sqrt{\mathcal{E}_{s_L}}|C(\xi)| (\parallel (\chi_{B_1}\frac{\partial S_{L+1}}{\partial b_L})_{s_L} \parallel_{L^2} +\parallel (\chi_{B_1}(T_L+\frac{\partial S_{L+2}}{\partial b_L}))_{s_L-1} \parallel_{L^2} )}{\lambda^{2(s_L-s_c)}}  \leq C(M)\frac{\mathcal{E}_{s_L}}{\lambda^{2(s_l-s_c)}} b_1^{\eta(1-\delta_0)} ,
\end{array}
\ee
For the same reasons:
\be \label{thetrapped:eq:gain inte dernier term mod 2}
\left| \frac{1}{2}\int \xi^{(1)}\mathcal{L}_{\frac{1}{\lambda}}^{s_L} \xi^{(1)} +\frac{1}{2} \int \xi^{(2)} \mathcal{L}_{\frac{1}{\lambda}}^{s_L-1} \xi^{(2)}   \right| \lesssim \frac{1}{\lambda^{2(s_l-s_c)}} \mathcal{E}_{s_L}b_1^{2\eta(1-\delta_0)} ,
\ee
The injection of these last bounds \fref{thetrapped:eq:gain inte dernier term mod} and \fref{thetrapped:eq:gain inte dernier term mod 2} in the previous identity \fref{thetrapped:highsobo:eq:estimation mod 2} yields the identity \fref{thetrapped:highsobo:eq:estimation mod} we claimed in this step 3. To end the proof of the proposition, it just remains to prove \fref{thetrapped:highsobo:eq:integration mod}, what we are now going to do. Using the improved modulation bound \fref{thetrapped:eq:improved modulation} for $b_{L,s}$ one calculates:
\be \label{thetrapped:highsobo:eq:expression radiation}
\begin{array}{r c l}
& \frac{d}{dt} \left( \int w^{(2)}\mathcal{L}_{\frac{1}{\lambda}}^{s_L-1} \xi^{(2)} + \int w^{(1)}\mathcal{L}_{\frac{1}{\lambda}}^{s_L} \xi^{(1)} \right) \\
=& \frac{1}{\lambda} \int w^{(1)}\mathcal{L}_{\frac{1}{\lambda}}^{s_L}((b_{L,s}+(L-\alpha)b_1b_L)\chi_{B_1} \frac{\partial S_{L+1}}{\partial b_L})_{\frac{1}{\lambda}} \\
&+ \frac{1}{\lambda} \int w^{(2)}\mathcal{L}_{\frac{1}{\lambda}}^{s_L-1}((b_{L,s}+(L-\alpha)b_1b_L)\chi_{B_1}(T_L+\ \frac{\partial S_{L+2}}{\partial b_L}))_{\frac{1}{\lambda}} \\
&+ \frac{O(b_1^{\delta_0}\sqrt{\mathcal{E}_{s_L}}+b_1^{L+1+g'})}{\lambda}  [ \int w^{(1)} \mathcal{L}_{\frac{1}{\lambda}}^{s_L}(\chi_{B_1}\frac{\partial S_{L+1}}{\partial b_L})_{\frac{1}{\lambda}} +w^{(2)}\mathcal{L}_{\frac{1}{\lambda}}^{s_L-1}(\chi_{B_1}(T_L+\frac{\partial S_{L+2}}{\partial b_L}))_{\frac{1}{\lambda}}  ]\\
&+  \int w^{(1)} \mathcal{L}_{\frac{1}{\lambda}}^{s_L} C(\xi) \partial_t \left( \mathcal{L}_{\frac{1}{\lambda}}^{s_L} \left[\chi_{B_1}\left(  \frac{\partial S_{L+1}}{\partial b_L}\right)\right]_{\frac{1}{\lambda}} \right)  \\
&+ \int w^{(2)} \mathcal{L}_{\frac{1}{\lambda}}^{s_L-1} C(\xi)  \partial_t \left( \mathcal{L}_{\frac{1}{\lambda}}^{s_L-1}  \left[\chi_{B_1}\left(T_L+ \frac{\partial S_{L+2}}{\partial b_L}\right)\right]_{\frac{1}{\lambda}} \right)\\
&+  \int w^{(2)}_t\mathcal{L}_{\frac{1}{\lambda}}^{s_L-1} \xi^{(2)} + \int w^{(1)}_t\mathcal{L}_{\frac{1}{\lambda}}^{s_L} \xi^{(1)}    .
\end{array}
\ee
We show that all the other terms are small enough. From the modulation equations \fref{thetrapped:eq:modulation leq L-1} for $b_i$ for $i<L$ one has: $|\lambda_s\lambda^{-1}|\lesssim b_1$, $|b_{i,s}|\lesssim b_1^{i+1}$. As $\bos{\xi}$ does not depend on $b_L$, this gives us the following bounds when the time derivative applies to $\bos{\xi}$ or $\bos{\mathcal{L}}$:
\be \label{thetrapped:highsobo:eq:radiation lambdaL}
\begin{array}{r c l}
&\Bigl| \frac{O(b_1^{\delta_0}\sqrt{\mathcal{E}_{s_L}}+b_1^{L+1+g'})}{\lambda}  [ \int w^{(1)} \mathcal{L}_{\frac{1}{\lambda}}^{s_L}(\chi_{B_1}\frac{\partial S_{L+1}}{\partial b_L})_{\frac{1}{\lambda}} + \int w^{(2)}\mathcal{L}_{\frac{1}{\lambda}}^{s_L-1}(\chi_{B_1}(T_L+\frac{\partial S_{L+2}}{\partial b_L}))_{\frac{1}{\lambda}}  ]\\
&+  \int w^{(1)} \mathcal{L}_{\frac{1}{\lambda}}^{s_L} C(\xi) \partial_t \left( \mathcal{L}_{\frac{1}{\lambda}}^{s_L} \left[\chi_{B_1}\left(  \frac{\partial S_{L+1}}{\partial b_L}\right)\right]_{\frac{1}{\lambda}} \right)  \\
&+ \int w^{(2)} \mathcal{L}_{\frac{1}{\lambda}}^{s_L-1} C(\xi)  \partial_t \left( \mathcal{L}_{\frac{1}{\lambda}}^{s_L-1}  \left[\chi_{B_1}\left(T_L+ \frac{\partial S_{L+2}}{\partial b_L}\right)\right]_{\frac{1}{\lambda}} \right) \Bigr| \\
\leq& C(L,M)\frac{b_1\sqrt{\mathcal{E}_{s_L}}}{\lambda^{2(s_L-s_c)+1}}(\sqrt{\mathcal{E}_{s_L}}b_1^{\eta(1-\delta_0)}+b_1^{L+(1-\delta_0)(1+\eta)+g'}),
\end{array}
\ee
where we used coercivity, \fref{thetrapped:highsobo:estimation C(xi)} and \fref{thetrapped:highsobo:asymptotic SL+1 SL+2} and the fact that $\partial_t (\mathcal{L}^{s_L-1}\chi_{B_1}T_L) $ has its support in $B_1\leq y\leq 2B_1$. We have now to estimate the terms involving $\bos{w}_t$ in \fref{thetrapped:highsobo:eq:expression radiation}. We do exactly the same things we did to the proof of Lemma \ref{trappedregime:improvedmodulation:lem:improvedmodulation}. For the sake of simplicity we will only do it for the second coordinate, the first one being the same. We first compute the expression:
\be \label{thetrapped:highsobo:eq:radiation wt intermediaire}
\begin{array}{r c l}
\int w^{(2)}_t \mathcal{L}_{\frac{1}{\lambda}}^{s_L-1} \xi^{(2)}&=& \int -\mathcal{L}_{\frac{1}{\lambda}}w^{(1)} \mathcal{L}_{\frac{1}{\lambda}}^{s_L-1} \xi^{(2)} + \int -\frac{1}{\lambda}(\tilde{\psi}_b^{(2)}+\tilde{Mod}(t)^{(2)})_{\frac{1}{\lambda}} \mathcal{L}_{\frac{1}{\lambda}}^{s_L-1} \xi^{(2)} \\
&& +\int (L(w)+NL(w)) \mathcal{L}_{\frac{1}{\lambda}}^{s_L-1} \xi^{(2)} .
\end{array}
\ee
We use the bootstrap assumptions to put an upper bound on everything except the $b_{L,s}$ term. For the linear term one has the bound:
\be \label{thetrapped:highsobo:eq:radiation lineaire}
\left|\int -\mathcal{L}_{\frac{1}{\lambda}}w^{(1)} \mathcal{L}_{\frac{1}{\lambda}}^{s_L-1} \xi^{(2)}\right| \leq \frac{\sqrt{\mathcal{E}_{s_L}}}{\lambda^{(s_l-s_c)}}\parallel (\xi^{(2)})_{s_L} \parallel_{L^2}\leq C(M) \frac{b_1}{\lambda^{2(s_l-s_c)+1}}\mathcal{E}_{s_L}b_1^{\eta(1-\delta_0)} .
\ee
Using the bounds on the error $\tilde{\bos{\psi}}_b$ from Proposition \ref{pr:profilapprochecoupe}:
\be \label{thetrapped:highsobo:eq:radiation psib mod}
\begin{array}{r c l}
\left| \int -\frac{1}{\lambda}(\tilde{\psi}_b^{(2)})_{\frac{1}{\lambda}} \mathcal{L}_{\frac{1}{\lambda}}^{s_L-1} \xi^{(2)} \right| &\leq & \frac{1}{\lambda^{s_L-s_c+1}}\parallel (\tilde{\psi}_b^{(2)})_{s_L-1} \parallel_{L^2} \parallel \xi^{(2)}_{s_L-1} \parallel_{L^2} \\
&\leq & \frac{C(M)b_1}{\lambda^{2(s_L-s_c)+1}}\sqrt{\mathcal{E}_{s_L}}b_1^{L+(1-\delta_0)(1+2\eta)}.
\end{array}
\ee
The small linear term gives the same estimate as the linear one:
\be \label{thetrapped:highsobo:eq:radiation L}
\left| \int L(w) \mathcal{L}_{\frac{1}{\lambda}}^{s_L-1} \xi^{(2)} \right| \leq \frac{C(M)b_1}{\lambda^{2(s_l-s_c)+1}}\mathcal{E}_{s_L}b_1^{\eta(1-\delta_0)} .
\ee
Finally, we start by decomposing the nonlinear term as a sum of term of the form: $\tilde{Q}_{b,\frac{1}{\lambda}}^{(1)(p-k)}w^{(1)k} $ for $2\leq k \leq p$. For each term we let all the derivatives on $\xi^{(2)}$:
$$
\begin{array}{r c l}
\left| \int NL(w) \mathcal{L}_{\frac{1}{\lambda}}^{s_L-1} \xi^{(2)} \right| &\lesssim& \frac{ \sqrt{\mathcal{E}_{s_L}}b_1^{\delta_0-1} }{\lambda^{2(s_L-s_c)+1}}\int \frac{|\varepsilon|^{(1)}|^k}{1+y^{\frac{2}{p-1}(p-k}} (\chi_{B_1}(T_L+\frac{\partial S_{L+2}}{\partial_{b_L}}))_{2s_L-2} .
\end{array}
$$
We know from their construction that $(T_L+\frac{\partial S_{L+2}}{\partial_{b_L}}))_{2s_L-2}=O\left( \frac{1}{1+y^{\gamma+L+1+2k_0}}\right)$, and by using the coercivity of the adapted norm and the $L^{\infty}$ estimate for $w^{(1)}$:
$$
\begin{array}{r c l}
\left| \int \frac{|\varepsilon|^{(1)}|^k(\chi_{B_1}(T_L+\frac{\partial S_{L+2}}{\partial_{b_L}}))_{2s_L-2}}{1+y^{\frac{2}{p-1}(p-k}} \right|&\leq& C\int \frac{|\varepsilon^{(1)}|^k}{1+y^{\frac{2}{p-1}(p-k)+\gamma+L+2k_0+1}} \\
&\leq& C(M)b_1^{-L+\gamma-1+\frac{2}{p-1}(p-k)}\mathcal{E}_{s_L}\parallel \varepsilon^{(1)} \parallel_{L^{\infty}}^{k-2} \\
&\leq & C(M,K_1,K_2)\mathcal{E}_{s_L}b_1^{-L+1+\alpha+O(\frac{1}{L})}\left(\frac{\sqrt{\mathcal{E}_{\sigma}}}{b_1^{\sigma-s_c}}\right)^{k-2}
\end{array}
$$
where the integral in $y$ we used with the Cauchy-Schwarz inequality was indeed divergent.
Under the bootstrap assumptions it leads to:
$$
\frac{\sqrt{\mathcal{E}_{s_L}}b_1^{1-\delta_0}}{\lambda^{2(s_L-s_c)+1}}\int \frac{|\varepsilon|^{(1)}|^k}{1+y^{\frac{2}{p-1}(p-k)}} (\chi_{B_1}(T_L+\frac{\partial S_{L+2}}{\partial_{b_L}}))_{2s_L-2}
\leq \frac{b_1\mathcal{E}_{s_L}}{\lambda^{2(s_L-s_c)+1}}b_1^{\eta(1-\delta_0)+\frac{\alpha}{2}} 
$$
(as $C(M,K_1,K_2)b_1^{\alpha}\leq b_1^{\frac{\alpha}{2}}$ for $s_0$ large). Therefore for the non linear term we have:
\be \label{thetrapped:highsobo:eq:radiation NL}
\left| \int NL(w) \mathcal{L}_{\frac{1}{\lambda}}^{s_L-1} \xi \right| \leq \frac{b_1\mathcal{E}_{s_L}}{\lambda^{2(s_L-s_c)+1}}b_1^{\eta(1-\delta_0)+\frac{\alpha}{2}}.
\ee
We now treat the modulation terms, preserving the $L$-th one. With the bound \fref{thetrapped:eq:modulation leq L-1} on the modulation for $1\leq i \leq L-1$, one has:
\be \label{thetrapped:highsobo:eq:radiation mod}
\begin{array}{r c l}
&\left| \int \frac{1}{\lambda} \tilde{Mod}^{(2)}_{\frac{1}{\lambda}}\mathcal{L}^{s_L-1} \xi ^{(2)} - \int \frac{1}{\lambda} (b_{L,s}+(L-\alpha)b_1b_L)\bigl{(}\chi_{B_1}(T_L+\frac{\partial S_{L+2}}{\partial b_L}) \bigr{)}_{\frac{1}{\lambda}}\mathcal{L}^{s_L-1} \xi ^{(2)} \right| \\
\leq & C(M)\frac{b_1\sqrt{\mathcal{E}_{s_L}}}{\lambda^{2(s_L-s_c)}}(\sqrt{\mathcal{E}_{s_L}}b_1^{\eta(1-\delta_0)}+b_1^{L+3}).
\end{array}
\ee
We come back to the expression \fref{thetrapped:highsobo:eq:radiation wt intermediaire} of the term involving $w_t^{(2)}$, inject the bounds we have found for each term \fref{thetrapped:highsobo:eq:radiation lineaire}, \fref{thetrapped:highsobo:eq:radiation psib mod}, \fref{thetrapped:highsobo:eq:radiation L} and \fref{thetrapped:highsobo:eq:radiation NL}, yielding:
\be \label{thetrapped:highsobo:eq:radiation wt bas}
\begin{array}{r c l}
\int w^{(2)}_t \mathcal{L}_{\frac{1}{\lambda}}^{s_L-1} \xi^{(2)} &=&  \int \frac{1}{\lambda} (b_{L,s}+(L-\alpha)b_1b_L)\left(\chi_{B_1}(T_L+\frac{\partial S_{L+2}}{\partial b_L}) \right)_{\frac{1}{\lambda}}\mathcal{L}^{s_L-1} \xi ^{(2)} \\
&&+\frac{b_1}{\lambda^{2(s_L-s_c)+1}} O\left(\mathcal{E}_{s_L}b_1^{\eta(1-\delta_0)}+\sqrt{\mathcal{E}_{s_L}}b_1^{L+(1-\delta_0)(1+2\eta)}\right) .
\end{array}
\ee
As we said, the same computation can be done using verbatim the same techniques for the first coordinate, yielding:
\be \label{thetrapped:highsobo:eq:radiation wt haut}
\begin{array}{r c l}
\int w^{(1)}_t \mathcal{L}_{\frac{1}{\lambda}}^{s_L} \xi^{(1)} &=&  \int \frac{1}{\lambda} (b_{L,s}+(L-\alpha)b_1b_L)\left(\chi_{B_1}\frac{\partial S_{L+1}}{\partial b_L} \right)_{\frac{1}{\lambda}}\mathcal{L}^{s_L} \xi ^{(1)} \\
&&+\frac{b_1}{\lambda^{2(s_L-s_c)+1}} O\left(\mathcal{E}_{s_L}b_1^{\eta(1-\delta_0)}+\sqrt{\mathcal{E}_{s_L}}b_1^{L+(1-\delta_0)(1+2\eta)}\right) .
\end{array}
\ee
Now we look back at the identity \fref{thetrapped:highsobo:eq:expression radiation}. We estimated all terms in the right hand side in \fref{thetrapped:highsobo:eq:radiation lambdaL}, \fref{thetrapped:highsobo:eq:radiation wt bas} and \fref{thetrapped:highsobo:eq:radiation wt haut}. Therefore it gives the intermediate identity:
\be \label{thetrapped:highsobo:radiation intermediaire}
\begin{array}{r c l}
& \frac{d}{dt} \left( \int w^{(2)}\mathcal{L}_{\frac{1}{\lambda}}^{s_L-1} \xi^{(2)} + \int w^{(1)}\mathcal{L}_{\frac{1}{\lambda}}^{s_L} \xi^{(1)} \right) \\
=& \frac{1}{\lambda} \int w^{(1)}\mathcal{L}_{\frac{1}{\lambda}}^{s_L}((b_{L,s}+(L-\alpha)b_1b_L)\chi_{B_1} \frac{\partial S_{L+1}}{\partial b_L})_{\frac{1}{\lambda}} \\
&+ \frac{1}{\lambda} \int w^{(2)}\mathcal{L}_{\frac{1}{\lambda}}^{s_L-1}((b_{L,s}+(L-\alpha)b_1b_L)\chi_{B_1}(T_L+\ \frac{\partial S_{L+2}}{\partial b_L}))_{\frac{1}{\lambda}} \\
&-(b_{L,s}+(L-\alpha)b_1b_L)\left[ \int \bigl{(}\chi_{B_1}\frac{\partial S_{L+1}}{\partial b_L}\bigr{)}_{\frac{1}{\lambda}} \mathcal{L}^{s_L} \xi^{(1)} + \bigl{(}\chi_{B_1}(T_L+\frac{\partial S_{L+2}}{\partial b_L})\bigr{)}_{\frac{1}{\lambda}} \mathcal{L}^{s_L-1} \xi^{(2)}\right]\\
& +O\left(\frac{b_1\sqrt{\mathcal{E}_{s_L}}}{\lambda^{2(s_L-s_c)+1}}\left(\sqrt{\mathcal{E}_{s_L}}b_1^{\eta(1-\delta_0)}+b_1^{L+(1-\delta_0)(1+2\eta)} \right) \right).
\end{array}
\ee
We will now integrate in time the remaining term involving $b_{L,s}+(L-\alpha)b_1b_L$. From the improved modulation equation \fref{thetrapped:eq:improved modulation} for $b_{L}$, one compute using \fref{thetrapped:highsobo:radiation intermediaire}:
$$
\begin{array}{l l l l}
&\frac{d}{dt} \left( \frac{1}{2}\int \xi^{(1)}\mathcal{L}_{\frac{1}{\lambda}}^{s_L} \xi^{(1)} +\frac{1}{2} \int \xi^{(2)} \mathcal{L}_{\frac{1}{\lambda}}^{s_L-1} \xi^{(2)}  \right) = \int \xi^{(1)}_{s_L}\partial_t (\xi^{(1)}_{s_L}) +\int \xi^{(2)}_{s_L-1}\partial_t (\xi^{(2)}_{s_L-1})\\
=&O(b_1^{\delta_0}\sqrt{\mathcal{E}_{s_L}}+b_1^{L+1+g'})[ \int \xi^{(1)} \mathcal{L}^{s_L}_{\frac{1}{\lambda}}(\chi_{B_1}\frac{\partial S_{L+1}}{\partial b_L})_{\frac{1}{\lambda}} +  \xi^{(2)} \mathcal{L}^{s_L-1}_{\frac{1}{\lambda}}(\chi_{B_1}(T_L+\frac{\partial S_{L+2}}{\partial b_L})_{\frac{1}{\lambda}} ]\\
&+ (b_{L,s}+(L-\alpha)b_1b_L)\bigl{(} \int \xi^{(1)} \mathcal{L}^{s_L}_{\frac{1}{\lambda}}(\chi_{B_1}\frac{\partial S_{L+1}}{\partial b_L})_{\frac{1}{\lambda}} + \xi^{(2)} \mathcal{L}^{s_L-1}_{\frac{1}{\lambda}}(\chi_{B_1}(T_L+\frac{\partial S_{L+2}}{\partial b_L})_{\frac{1}{\lambda}} \bigr{)} \\
&+ \frac{C(\xi)}{2}\int \xi^{(1)} \partial_t \Bigl{(} \mathcal{L}^{s_L}_{\frac{1}{\lambda}}(\chi_{B_1}\frac{\partial S_{L+1}}{\partial b_L})_{\frac{1}{\lambda}}  \Bigr{)}+\frac{C(\xi)}{2}\int \xi^{(2)} \partial_t \Bigl{(} \mathcal{L}^{s_L-1}_{\frac{1}{\lambda}}(\chi_{B_1}(T_L+\frac{\partial S_{L+2}}{\partial b_L}))_{\frac{1}{\lambda}}  \Bigr{)}
\end{array}
$$
Using verbatim the same techniques employed throughout this step 3 we estimate the remaining terms in this identity and end up with:
\be \label{thetrapped:highsobo:radiation intermediaire 2}
\begin{array}{l l l l}
&\frac{d}{dt} \left( \frac{1}{2}\int \xi^{(1)}\mathcal{L}_{\frac{1}{\lambda}}^{s_L} \xi^{(1)} +\frac{1}{2} \int \xi^{(2)} \mathcal{L}_{\frac{1}{\lambda}}^{s_L-1} \xi^{(2)}  \right) \\
=& (b_{L,s}+(L-\alpha)b_1b_L)\Bigl{(} \int \xi^{(1)} \mathcal{L}^{s_L}_{\frac{1}{\lambda}}(\chi_{B_1}\frac{\partial S_{L+1}}{\partial b_L})_{\frac{1}{\lambda}} + \xi^{(2)} \mathcal{L}^{s_L-1}_{\frac{1}{\lambda}}(\chi_{B_1}(T_L+\frac{\partial S_{L+2}}{\partial b_L})_{\frac{1}{\lambda}} \Bigr{)} \\
&+\frac{b_1\sqrt{\mathcal{E}_{s_L}}}{\lambda^{2(s_L-s_c)+1}}O(\sqrt{\mathcal{E}_{s_L}}b_1^{2\eta(1-\delta_0)}+b_1^{L+(1-\delta_0)(1+2\eta)+g'}).
\end{array}
\ee 
We can now end the proof: combing the intermediate estimates \fref{thetrapped:highsobo:radiation intermediaire 2} and \fref{thetrapped:highsobo:radiation intermediaire} yields the identity \fref{thetrapped:highsobo:eq:integration mod}
\end{proof}


\subsection{Control from a Morawetz type quantity:}

As will be clear when we reintegrate the bootstrap equation in the next section, the term we still do no control in the monotonicity formula for the high regularity norm is the local one. We control it here via the study of a Morawetz type quantity. This term contributes to the time evolution of a bounded quantity (compared with $\mathcal{E}_{s_L}$), so when we integrate it with respect to time it should remain small. For $A>0$ and $\delta>0$ let:
\be \label{thetrapped:eq:def phiA}
\phi_A(x):=\int_0^x \chi_A (x')x^{'(1-\delta)}dx' 
\ee
be the primitive of the function $\chi_A (x)x^{1-\delta}$ and we still denote by $\phi_A$ its radial extension. The quantity we will now study is (we recall that the adapted derivative $f_k$ of a function is defined in \fref{linearized:eq:def derivees adaptees}):
\begin{equation}\label{thetrapped:eq:definition M}
\mathcal{M}= -\int [\nabla \phi_A .\nabla \varepsilon^{(1)}_{s_L-1}+(1-\delta) \frac{\Delta \phi_A}{2}\varepsilon^{(1)}_{s_L-1}] \varepsilon^{(2)}_{s_L-1} .
\end{equation}
From coercivity (Corollary \ref{annexe:cor:coercivite mathcalEsL}), it is controlled by the high Sobolev norm:
\begin{equation}\label{thetrapped:eq:morawetz controle par high sobo}
|\mathcal{M} |\leq  C(A,M) \mathcal{E}_{s_L}
\end{equation}
We start by a lemma describing how this quantity controls the local norm $\mathcal{E}_{s_L,\text{loc}}$ thanks to the fact that $\mathcal{L}>0$ on $\dot{H}^1$.

\begin{lemma}\emph{(control from the Morawetz identity at the linear level)} \label{thetrapped:lem:control du morawetz}
For $A$ big enough, $\delta$ small enough, there holds the following control:
\begin{equation} \label{thetrapped:morawetz:eq:controle lineaire}
\begin{array}{r c l}
&  \int [\nabla \phi_A .\nabla \varepsilon^{(1)}_{s_L-1}+\frac{(1-\delta)\Delta \phi_A}{2}\varepsilon^{(1)}_{s_L-1}] \mathcal{L}\varepsilon^{(1)}_{s_L-1} \\
&-\int [\nabla \phi_A .\nabla \varepsilon^{(2)}_{s_L-1}+\frac{(1-\delta)\Delta \phi_A}{2}\varepsilon^{(2)}_{s_L-1}] \varepsilon^{(2)}_{s_L-1}\\
\geq & C \frac{\delta}{N^{\delta}} \mathcal{E}_{s_L,loc} -\frac{C(M)}{A^{\delta}}\mathcal{E}_{s_L} ,
\end{array}
\end{equation}
for some constant $C>0$ that does not depend on the other constants.

\end{lemma}

\begin{proof}[Proof of Lemma \ref{thetrapped:lem:control du morawetz}]
We calculate each term in the left hand side of \fref{thetrapped:morawetz:eq:controle lineaire}. For the second one we have:
$$
-\int [\nabla \phi_A .\nabla \varepsilon^{(2)}_{s_L-1}+\frac{(1-\delta)\Delta \phi_A}{2}\varepsilon^{(2)}_{s_L-1}] \varepsilon^{(2)}_{s_L-1}=\delta \int \frac{\Delta \phi_A}{2}|\varepsilon^{(2)}_{s_L-1}|^2.
$$
As $\Delta \phi_A= \frac{(d-\delta)\chi_A}{y^{\delta}}+\frac{\partial_y \chi \left( \frac{y}{A}\right)}{Ay^{\delta-1}}$ we get a control over the second coordinate:
\be \label{thetrapped:morawetz:eq:controle lineaire 2}
-\int [\nabla \phi_A .\nabla \varepsilon^{(2)}_{s_L-1}+\frac{(1-\delta)\Delta \phi_A}{2}\varepsilon^{(2)}_{s_L-1}] \varepsilon^{(2)}_{s_L-1}\geq \delta \int \frac{\chi_A |\varepsilon^{(2)}_{s_L-1}|^2}{y^{\delta}}+O\left(\frac{\mathcal{E}_{s_L}}{A^{\delta}} \right)
\ee
We now turn to the first term in \fref{thetrapped:morawetz:eq:controle lineaire}. We start by calculating:
\be \label{thetrapped:morawetz:eq:terme 1 intermediaire}
\begin{array}{r c l}
&-\int [\nabla \phi_A .\nabla \varepsilon^{(1)}_{s_L-1}+\frac{(1-\delta)\Delta \phi_A}{2}\varepsilon^{(1)}_{s_L-1}] (-\mathcal{L}\varepsilon^{(1)}_{s_L-1})\\
=&\int (\partial_y^2 \phi_A-\frac{\delta \Delta \phi_A}{2} |\nabla \varepsilon^{(1)}_{s_L-1}|^2-\frac{1-\delta}{4}\int \Delta^2 \phi_A |\varepsilon^{(1)}_{s_L-1}|^2+\int \frac{\nabla V.\nabla \phi_A+\delta \Delta \phi_A V}{2} |\varepsilon^{(1)}_{s_L-1}|^2 .
\end{array}
\ee
We are now going to show that locally, the first term of the right hand side is bigger than the two others and control the first coordinate. We have $\partial_y^2 (\psi_A)=\frac{(1-\delta)\chi_A}{y^{-\delta}}+\frac{y^{1-\delta}}{A}\partial_y \chi (\frac{y}{A})$ which leads to:
\be \label{thetrapped:morawetz:terme gradient inte}
\int (\partial_y^2 \phi_A-\frac{\delta \Delta \phi_A}{2} |\nabla \varepsilon^{(1)}_{s_L-1}|^2=(1-O(\delta))\int \chi_A \frac{|\nabla \varepsilon^{(1)}_{s_L-1}|^2}{y^{\delta}} + O\left( \frac{1}{A^{\delta}}\mathcal{E}_{s_L} \right).
\ee
We claim the following weighted Hardy inequality for radial functions:
\be \label{thetrapped:morawetz:eq:hardy}
\int \frac{\chi_A}{y^{-\delta}}|\nabla u|^2 \geq \frac{(d-2-\delta)^2}{4}\int \chi_A\frac{u^2}{y^{2+\delta}}-C\int \frac{|y\partial \chi (\frac{y}{A})|}{y^{2+\delta}}u^2 .
\ee
We prove this general inequality now. For smooth radial functions we compute, performing integration by parts:
\be \label{thetrapped:morawetz:eq:preuve hardy 1}
\int \frac{\chi_A}{y^{1+\delta}}u\partial_y u = -\frac{d-2-\delta}{2} \int \frac{u^2}{y^{2+\delta}}\chi_A -\frac{1}{2}\int \frac{u^2}{y^{2+\delta}}\frac{y\partial_y\chi (\frac{y}{A})}{A} .
\ee
We can control the left hand side by using Cauchy-Schwarz and Young's inequality:
\be \label{thetrapped:morawetz:eq:preuve hardy 2}
\left| \int \frac{\chi_A}{y^{1+\delta}}u\partial_y u \right|\leq \frac{\epsilon}{2} \int \frac{\chi_A}{y^{2+\delta}}u^2 +\frac{1}{2\epsilon}\int \frac{\chi_A}{y^{\delta}}|\nabla u|^2 .
\ee
Combining the two equations \fref{thetrapped:morawetz:eq:preuve hardy 1} and \fref{thetrapped:morawetz:eq:preuve hardy 2} with the choice $\epsilon=\frac{d-2-\delta}{2}$ gives the analysis bound \fref{thetrapped:morawetz:eq:hardy} we claimed. We now come back to the identity \fref{thetrapped:morawetz:terme gradient inte}, which gives the following control thanks to the Hardy inequality \fref{thetrapped:morawetz:eq:hardy} we just proved:
\be \label{thetrapped:morawetz:terme gradient}
\begin{array}{r c l}
\int (\partial_y^2 \phi_A -\frac{\delta \Delta \phi_A}{2} )|\nabla \varepsilon^{(1)}_{s_L-1}|^2&\geq& \delta \int \chi_A \frac{|\nabla \varepsilon^{(1)}_{s_L-1}|^2}{y^{\delta}}+(1-O(\delta))^2\frac{(d-2-\delta)^2}{4} \int \chi_A \frac{|\varepsilon^{(1)}_{s_L-1}|^2}{y^{2+\delta}} \\
&&+O\left( \frac{1}{A^{\delta}}\mathcal{E}_{s_L} \right).
\end{array}
\ee
With this control coming from the "gradient" part, the equation \fref{thetrapped:morawetz:eq:terme 1 intermediaire} can be rewriten as:
\be \label{thetrapped:morawetz:eq:terme 1 intermediaire 2}
\begin{array}{r c l}
&-\int [\nabla \phi_A .\nabla \varepsilon^{(1)}_{s_L-1}+\frac{\Delta \phi_A}{2}\varepsilon^{(1)}_{s_L-1}] (-\mathcal{L}\varepsilon^{(1)}_{s_L-1})\\
\geq&   \delta \int \chi_A \frac{|\nabla \varepsilon^{(1)}_{s_L-1}|^2}{y^{\delta}}+(1-O(\delta))\frac{(d-2-\delta)^2}{4} \int \chi_A \frac{|\varepsilon^{(1)}_{s_L-1}|^2}{y^{2+\delta}} +O\left( \frac{1}{A^{\delta}}\mathcal{E}_{s_L} \right)     \\
&-\frac{1-\delta}{4}\int \Delta^2 \phi_A |\varepsilon^{(1)}_{s_L-1}|^2+\frac{1}{2} \int (\nabla V.\nabla \phi_A+\delta\Delta \phi_A V) |\varepsilon^{(1)}_{s_L-1}|^2.
\end{array}
\ee
We now prove that the last two terms are controled by the two first ones. We calculate:
\be  \label{thetrapped:morawetz:eq:expression Delta phi}
-\Delta^2 (\phi_A)= \frac{\delta(d-\delta)(d-2-\delta)}{2}\frac{\chi_A}{y^{2+\delta}} +O\left( \frac{1}{A^{\delta}}1_{A\leq y\leq 2A} \right) .
\ee
For the term involving the potential we have that because $\Lambda^{(1)}Q,Q>0$:
\begin{equation} \label{thetrapped:morawetz:eq:expression potentiel}
\begin{array}{r c l}
\frac{1}{2}y\partial_y V&=&\frac{y}{2}p(p-1)Q^{p-2}\partial_y Q=\frac{p}{2}(p-1)Q^{p-2}\Lambda^{(1)}Q -pQ^{p-1} \\
&\geq & -pQ^{p-1} \\
&\geq & -\frac{\frac{(d-2)^2}{4}-\delta_p}{y^2} ,
\end{array}
\end{equation}
for some $\delta_p>0$, because the potential is strictly smaller than the Hardy potential from Lemma \ref{lem:Q}. The expressions \fref{thetrapped:morawetz:eq:expression Delta phi} and \fref{thetrapped:morawetz:eq:expression potentiel} imply that \fref{thetrapped:morawetz:eq:terme 1 intermediaire 2} can be rewriten as:
$$
\begin{array}{r c l}
&\int \partial_y^2 \phi_A |\nabla \varepsilon^{(1)}_{s_L-1}|^2-\frac{1-\delta}{4}\int \Delta^2 \phi_A \varepsilon^{(1)2}_{s_L-1}+\frac{1}{2} \int (\nabla V.\nabla \phi+\delta \Delta \phi_A V) \varepsilon^{(1)2}_{s_L-1} \\
\geq& \delta \int \frac{\chi_A|\nabla \varepsilon^{(1)}_{s_L-1}|}{y^{\delta}}+ (\delta_p - O(\delta))\int \frac{|\varepsilon^{(1)}_{s_L-1}|^2}{y^{2+\delta}}+O\left( \frac{\mathcal{E}_{s_L}}{A^{\delta}}\right) .
\end{array}
$$
Hence the identity \fref{thetrapped:morawetz:eq:terme 1 intermediaire} becomes:
\be \label{thetrapped:morawetz:eq:controle lineaire 1}
\begin{array}{r c l}
&-\int [\nabla \phi_A .\nabla \varepsilon^{(1)}_{s_L-1}+\frac{\Delta \phi_A}{2}\varepsilon^{(1)}_{s_L-1}] (-\mathcal{L}\varepsilon^{(1)}_{s_L-1})\\
\geq &\delta \int \frac{\chi_A|\nabla \varepsilon^{(1)}_{s_L-1}|}{y^{\delta}}+ (\delta_p - O(\delta))\int \frac{\chi_A |\varepsilon^{(1)}_{s_L-1}|^2}{y^{2+\delta}}+O\left( \frac{\mathcal{E}_{s_L}}{A^{\delta}}\right)  .
\end{array}
\ee
We now come back to the left hand side of \fref{thetrapped:morawetz:eq:controle lineaire}. We have estimated the two terms in \fref{thetrapped:morawetz:eq:controle lineaire 2} and \fref{thetrapped:morawetz:eq:controle lineaire 1}. For $\delta \ll \delta_p$ this gives the identity \fref{thetrapped:morawetz:eq:controle lineaire} we had to prove.
\end{proof}

We can now state the control in the full nonlinear wave equation:
\begin{proposition}\label{thetrapped:pr:morawetz}\emph{(Control of the local term by the Morawetz identity)}
We suppose all the parameters of Proposition \ref{prop:bootstrap} are fixed in their range, except $s_0$. For $s_0$ and $A$ large enough, there holds for $s_0\leq s <s^*$:
\begin{equation}
\frac{d}{ds}\mathcal{M}\geq \frac{\delta}{2N^{\delta}} \mathcal{E}_{s_L,\text{loc}}-\frac{C(M)}{A^{\delta}}\mathcal{E}_{s_L}-C(A)\sqrt{\mathcal{E}_{s_L}}b_1^{L+3}  ,
\end{equation}
($\mathcal{E}_{s_L}$ and $\mathcal{E}_{s_L,\text{loc}}$ were defined in \fref{thetrapped:eq:def mathcalEsL} and \fref{thetrapped:eq:def mathcalEsLloc}).
\end{proposition}

\begin{remark}
As:
$$
\frac{d}{dt}\frac{\mathcal{M}}{\lambda^{2(s_L-s_c)}}=2(s_L-s_c)\frac{b_1\mathcal{M}}{\lambda^{2(s_L-s_c)+1}}+\frac{1}{\lambda^{2(s_L-s_c)+1}}\frac{d}{ds}\mathcal{M} ,
$$
from the control \ref{thetrapped:eq:morawetz controle par high sobo}, the result of the lemma implies (remember $b_1\leq \frac{1}{s_0}$ in the bootstrap regime, and that $s_0$ is chosen in last so than $b_1$ can be arbitrarily small compared to the other constants) :
$$
\frac{d}{dt}\left(\frac{\mathcal{M}}{\lambda^{2(s_L-s_c)}}\right) \geq \frac{1}{\lambda^{2(s_L-s_c)+1}}\left(\frac{\delta}{2N^{\delta}} \mathcal{E}_{s_L,\text{loc}}-\frac{C(M)}{A^{\delta}}\mathcal{E}_{s_L}-C(A,M)\sqrt{\mathcal{E}_{s_L}}b_1^{L+3}\right) .
$$
This is because the impact of the scale changing in the estimate we want to prove is of lower order, so we can work both at level $\bos{\varepsilon}$ or $\bos{w}$.
\end{remark}

\begin{proof}[Proof of Proposition \ref{thetrapped:pr:morawetz}]
The control comes from the previous lemma, and the new terms in the full (NLW) will be showed to be negligible. The time evolution of $\mathcal{M}$ is ($f_k$ being the $k$-th adapted derivative of $f$ defined in \fref{linearized:eq:def derivees adaptees}):
\be \label{thetrapped:morawetz:eq:expression dtM}
\begin{array}{r c l}
&\frac{d}{ds}\mathcal{M}= -\int \nabla \phi_A .\nabla [ (-\frac{\lambda_s}{\lambda}\Lambda^{(1)}\varepsilon^{(1)}+\varepsilon^{(2)}-\tilde{\psi}_b^{(1)}-\tilde{Mod}(t)^{(1)})]_{s_L-1}\varepsilon^{(2)}_{s_L-1} \\
&- \int \frac{(1-\delta)\Delta \phi_A}{2}(-\frac{\lambda_s}{\lambda}\Lambda^{(1)}\varepsilon^{(1)}+\varepsilon^{(2)}-\tilde{\psi}_b^{(1)}-\tilde{Mod}(t)^{(1)})_{s_L-1} \varepsilon^{(2)}_{s_L-1} \\
&+\int \nabla \phi_A .\nabla \varepsilon^{(1)}_{s_L-1}[\mathcal{L}\varepsilon^{(1)}+\frac{\lambda_s}{\lambda}\Lambda^{(2)}\varepsilon^{(2)}+\tilde{\psi}_b^{(2)}+\tilde{Mod}(t)^{(2)}-L(\varepsilon)-NL(\varepsilon)]_{s_L-1} \\
&+ \int \frac{(1-\delta)\Delta \phi_A}{2}\varepsilon^{(1)}_{s_L-1} [\mathcal{L}\varepsilon^{(1)}+\frac{\lambda_s}{\lambda}\Lambda^{(2)}\varepsilon^{(2)}+\tilde{\psi}_b^{(2)}+\tilde{Mod}(t)^{(2)}-L(\varepsilon)+NL(\varepsilon)]_{s_L-1} .
\end{array}
\ee
And we aim at computing the effect of everything in the right hand side. The linear part produces exactly the control we want thanks to the previous Lemma \ref{thetrapped:lem:control du morawetz}:
\be \label{thetrapped:morawetz:eq:estimation lineaire}
\begin{array}{r c l}
& \int [\nabla \phi_A .\nabla \varepsilon^{(1)}_{s_L-1}+\frac{(1-\delta)\Delta \phi_A}{2}\varepsilon^{(1)}_{s_L-1}] \mathcal{L}\varepsilon^{(1)}_{s_L-1} \\
&-\int [\nabla \phi_A .\nabla \varepsilon^{(2)}_{s_L-1}+\frac{(1-\delta)\Delta \phi_A}{2}\varepsilon^{(2)}_{s_L-1}] \varepsilon^{(2)}_{s_L-1}\\
\geq &\frac{\delta}{2N^{\delta}} \mathcal{E}_{s_L,\text{loc}} -\frac{C}{A^{\delta}}\mathcal{E}_{s_L} .
\end{array}
\ee
We are now going to show that all the other terms are of smaller order. As we work on a compact support, from the coercivity \fref{annexe:coercivitemathcalE:eq:coercivitemathcalE} and the fact that $\frac{\lambda_s}{\lambda}\sim -b_1$ from \fref{thetrapped:eq:modulation leq L-1}:
\be \label{thetrapped:morawetz:eq:estimation lambdavarepsilon}
\begin{array}{r c l}
&\left|\int [\nabla \phi_A .\nabla(\frac{\lambda_s}{\lambda}\Lambda^{(1)}\varepsilon^{(1)}_{s_L-1})+\frac{(1-\delta)\Delta \phi_A}{2}\frac{\lambda_s}{\lambda}\Lambda^{(1)}\varepsilon^{(1)}_{s_L-1}]\varepsilon^{(2)}_{s_L-1}\right| \\
+&\left|\int [\nabla \phi_A. \nabla(\varepsilon^{(1)}_{s_L-1})+\frac{(1-\delta)\Delta \phi_A}{2}\varepsilon^{(1)}_{s_L-1}]\frac{\lambda_s}{\lambda}\Lambda^{(2)}\varepsilon^{(2)}_{s_L-1}\right| \\
\leq & b_1 C(A) \mathcal{E}_{s_L} ,
\end{array}
\ee
so with $b_1$ small enough it is negligeable. Still from the compactness of the support of $\phi_A$, for $b_1$ small enough we do not see the bad tail of $\tilde{\bos{\psi}}_b$ (remember that for $y\leq B_1$, $\tilde{\bos{\psi}}_b=\bos{\psi}_b$). Hence:
\be \label{thetrapped:morawetz:eq:estimation psib}
\begin{array}{r c l}
&\left|\int [\nabla \phi_A .\nabla(\tilde{\psi}^{(1)}_{b,s_L-1})+\frac{(1-\delta)\Delta \phi_A}{2}\tilde{\psi}^{(1)}_{b,s_L-1}]\varepsilon^{(2)}_{s_L-1}\right| \\
&+\left|\int [\nabla \phi_A .\nabla(\varepsilon^{(1)}_{s_L-1})+\frac{(1-\delta)\Delta \phi_A}{2}\varepsilon^{(1)}_{s_L-1}]\tilde{\psi}^{(2)}_{b,s_L-1}\right|\\
\leq & C(A) \sqrt{\mathcal{E}_{s_L}} (\parallel \tilde{\psi}_{b,s_L}^{(1)}\parallel_{L^2(\leq A)} +\parallel \tilde{\psi}_{b,s_L-1}^{(2)}\parallel_{L^2(\leq A}) \leq C(A)\sqrt{\mathcal{E}}_{\sigma}b_1^{L+3} .
\end{array}
\ee
The small linear term is also estimated easily. Indeed, we have that:
$$
L(\varepsilon)=p(Q^{p-1}-\tilde{Q}_b^{p-1})\varepsilon^{(1)}= b_1\varepsilon^{(1)}O(1)
$$
for $y\leq A$ for $b_1$ small enough. This gives using Cauchy-Schwarz:
\be \label{thetrapped:morawetz:eq:estimation L}
\left|\int [\nabla \phi_A .\nabla(\varepsilon^{(1)}_{s_L-1})+\frac{(1-\delta)\Delta \phi_A}{2}\varepsilon^{(1)}_{s_L-1}]L(\varepsilon)_{s_L-1}\right|\leq C(A)b_1\mathcal{E}_{s_L} .
\ee
For the nonlinear term we use what we already showed during the proof of the monotonicity formula for the high Sobolev norm, see \fref{thetrapped:highsobo:eq:estimation NL}:
\be \label{thetrapped:morawetz:eq:estimation NL}
\begin{array}{r c l}
\left|\int [\nabla \phi_A .\nabla(\varepsilon^{(1)}_{s_L-1})+\frac{(1-\delta)\Delta \phi_A}{2}\varepsilon^{(1)}_{s_L-1}]NL(\varepsilon)_{s_L-1}\right|&\leq& C(A)\sqrt{\mathcal{E}_{s_L}}\parallel NL(\varepsilon)_{s_L-1}\parallel_{L^2} \\
&\leq & C(A)b_1\mathcal{E}_{s_L} ,
\end{array}
\ee
which is negligeable for $b_1$ small enough as we said before. Finally it just remains to control the modulation terms. We just compute for the second coordinate, a similar estimate holding for the first one. Let $i$ be odd, $1\leq i \leq L$. As $A\ll B_1$ for $s_0$ large enough, we do not see the the cut $\chi_{B_1}$ in the integral: $\chi_{B_1}\equiv 1$ for $y\leq 2A$. Because $\bos{H}^{i}\bos{T}_i=(\bos{T}_i)_{i-1}(-1)^{\frac{i+1}{2}}=(-1)^i\bos{\Lambda} \bos{Q}$ this term cancels in the integral because $(\bos{T}_i)_{s_L-1}=((\bos{T}_i)_{i-1})_{s_L-i}=0$ as $s_L-i=L+k_0-i\geq 1$.
$$
\int [\nabla \phi_A \nabla(\varepsilon^{(1)}_{s_L-1})+\frac{(1-\delta)\Delta \phi_A}{2}\varepsilon^{(1)}_{s_L-1}](b_{i,s}+(i-\alpha)b_1b_i-b_{i+1})(\chi_{B_1}T_i)^{(2)}_{s_L-1}=0 .
$$
For the terms of the form $\frac{\partial S_j}{\partial b_i}$ we always have at least one parameter $b_i$ involved in this expression, which gives that for $y\leq A$ there holds: $\left|\frac{\partial S_j}{\partial b_i}(y)\right| \leq C(A) b_1$. We then use the modulation equation proven in Lemma \ref{lem:modulation} to estimate:
$$
\begin{array}{r c l}
&\left| \int [\nabla \phi_A. \nabla(\varepsilon^{(1)}_{s_L-1})+\frac{(1-\delta)\Delta \phi_A}{2}\varepsilon^{(1)}_{s_L-1}](b_{i,s}+(i-\alpha)b_1b_i-b_{i+1})(\chi_{B_1}\frac{\partial S_j}{\partial b_i})^{(2)}_{s_L-1} \right| \\
\leq& C(A,M)\mathcal{E}_{s_L} b_1+C(A,M)\sqrt{\mathcal{E}_{s_L}}b_1^{L+3} .
\end{array}
$$
As we said, the same reasoning applies to treat the first coordinate. Consequently we have the following bound for the modulation terms:
\be \label{thetrapped:morawetz:eq:estimation mod}
\begin{array}{r c l}
&\Bigl| \int \nabla \phi_A .\nabla [ \tilde{Mod}(t)^{(1)})]_{s_L-1}\varepsilon^{(2)}_{s_L-1} + \int \frac{(1-\delta)\Delta \phi_A}{2}(\tilde{Mod}(t)^{(1)})_{s_L-1} \varepsilon^{(2)}_{s_L-1} \\
&\int \nabla \phi_A .\nabla \varepsilon^{(1)}_{s_L-1}[\tilde{Mod}(t)^{(2)}]_{s_L-1} +\int \frac{(1-\delta)\Delta \phi_A}{2}(\varepsilon^{(1)})_{s_L-1} [\tilde{Mod}(t)^{(2)}]_{s_L-1} \Bigr| \\
\leq& C(A,M)\mathcal{E}_{s_L} b_1+C(A,M)\sqrt{\mathcal{E}_{s_L}}b_1^{L+3}.
\end{array}
\ee
We now come back to our initial decomposition \fref{thetrapped:morawetz:eq:expression dtM}. We have the expected control from the linear term in \fref{thetrapped:morawetz:eq:estimation lineaire}, and have estimated all the other terms in \fref{thetrapped:morawetz:eq:estimation lambdavarepsilon}, \fref{thetrapped:morawetz:eq:estimation psib}, \fref{thetrapped:morawetz:eq:estimation L}, \fref{thetrapped:morawetz:eq:estimation NL} and \fref{thetrapped:morawetz:eq:estimation mod}. It gives the desired result.
\end{proof}


\section{End of the proof:}
\label{end:sec:end of the proof}
\subsection{End of the Proof of Proposition \ref{prop:bootstrap}}
We now end the proof of the proposition \ref{prop:bootstrap}. We will reintagrate in time the equations giving the time evolution of the parameters and the norms for the error term to obtain improved bounds. The definition of the minimal time $s^*$ for which the bootstrap assumptions are violated implies that at time $s^*$ at least one of the following three facts is true:
\begin{itemize}
\item[(i)] \emph{The error term has grown too big:}
$$
\mathcal{E}_{s_L}(s^*)=K_1b_1(s^*)^{2L+2(1-\delta_0)(1+\eta)} \ \text{or} \ \mathcal{E}_{\sigma}= K_2b_1(s^*)^{2(\sigma-s_c)\frac{\ell}{\ell-\alpha}} ,
$$
\item[(ii)] \emph{Exit of the stable modes}
$$
V_1(s^*)=\frac{1}{(s^*)^{\tilde{\eta}}} \ \text{or} \ |b_k(s^*)|=\frac{\epsilon_k}{(s^{*})^{k+\tilde{\eta}}} ,
$$
\item[(iii)] \emph{Exit of the instable modes:}

$$
(V_2(s^*),...,V_{\ell}(s^*))\in \mathcal{S}^{l-1}\left(\frac{1}{(s^*)^{\tilde{\eta}}}\right) .
$$

\end{itemize}
We will show in this section that the cases (i) and (ii) never happen for any initial solution. Indeed, the estimates of the error term can be improved using all the preceding monotonicity formulas, and are in fact smaller than what we asked for. The exit of the stable modes is impossible because their evolution is governed by a linear equation for which $0$ is an attractor, plus a force term whose size is under control.\\
\\
There are initial data leading to the exit of the unstable modes because they are driven by unstable dynamics. Indeed from the study of the linearized equation for the parameters we have seen that $0$ is a repulsive equilibrium for these modes. However this equilibrium must persist\footnote{this is a way of speaking, there is no fixed point but one trajectory staying bounded.} when we add the perturbative term to the equation, because the contrary would go against Brouwer fixed point theorem. This part will be made clearer in our precise case later on.\\

We begin with integrating the scaling equations.

\begin{lemma}[law for the scaling in the trapped regime]\label{endoftheproof:lem:integration des equations pour le scaling}
Up to time $s^*$ there holds the following estimations for the scaling:
\begin{equation} \label{end:eq:asymptotique lambda}
\lambda(s)=\left( \frac{s_0}{s}\right) ^{\frac{\ell}{\ell-\alpha}}\left[1+O\left( \frac{1}{s_0^{\tilde{\eta}}} \right) \right] .
\end{equation}
\end{lemma}

\begin{proof}[Proof of Lemma \ref{endoftheproof:lem:integration des equations pour le scaling}]
Untill $s^*$, we have under the bootstrap assumptions \fref{eq:bootstrap modes stables} and \fref{eq:bootstrap modes instables} for the parameters that $b_i(s)=b_i^e+\frac{U_i}{s^{i+1}}$ with $U_i\leq \frac{1}{s^{\tilde{\eta}}}$. So we use the modulation equation proved in Lemma \ref{lem:modulation}:
$$
-\frac{\lambda_s}{\lambda}=b_1+O\left( b_1 \mathcal{E}_{s_L}+b_1^{L+3})\right)=\frac{\ell}{(\ell-\alpha)s}+O\left(\frac{1}{s^{1+\tilde{\eta}}}\right) .
$$
We rewrite this equation as:
$$
\left|\frac{d}{ds}(\text{log}(s^{\frac{\ell}{\ell-\alpha}}\lambda)) \right|\lesssim \frac{1}{s^{1+\tilde{\eta}}} .
$$
After integration gives:
$$
\lambda(s)=\left( \frac{s_0}{s}\right) ^{\frac{\ell}{\ell-\alpha}}\left[1+O\left( \frac{1}{s_0^{\tilde{\eta}}} \right) \right] .
$$
\end{proof}

We now rule out the case $(i)$. We recall that $K_1$ and $K_2$ are used to quantify the control of the error term $\bos{\varepsilon}$ in the trapped regime of proposition \ref{prop:bootstrap}.

\begin{lemma}[Integrating the evolution equations for the norms]\label{endoftheproof:lem:integration des equations pour les normes}
Assume all the other constants of Proposition \ref{prop:bootstrap} are fixed in their range. There exist $K_1,K_2>0$, $N>0$, $\nu>0$ and $\epsilon$ such that for $s_0$ big enough, $\eta$ small enough, under the bootstrap assumptions untill time $s^*$ the norms enjoy a better estimation. There holds in fact:
\begin{equation}
\mathcal{E}_{s_L}\leq \frac{K_1}{2}b_1^{2L+2(1-\delta_0)(1+\eta)} ,
\end{equation}
and:
\begin{equation}
\mathcal{E}_{\sigma}\leq \frac{K_2}{2}b_1^{2(\sigma-s_c)\frac{\ell}{\ell-\alpha}} .
\end{equation}
\end{lemma}

\begin{remark}
The constant $\frac{1}{2}$ is not really important, we could have stated it for any constant.
\end{remark}

\begin{proof}[ Proof of Lemma \ref{endoftheproof:lem:integration des equations pour les normes}] \underline{The low Sobolev norm:} We recall the bound on the time evolution of the low Sobolev norm from Proposition \ref{trappedregime:pr:low sobo}:
$$
\frac{d}{dt}\left\{ \frac{\mathcal{E}_{\sigma}}{\lambda^{2(\sigma-s_c)}} \right\} \leq \frac{b_1 \sqrt{\mathcal{E}_{\sigma}}b_1^{(\sigma-s_c)(1+\nu)}}{\lambda^{2(\sigma-s_c)+1}} \left[ b_1^{\frac{\alpha}{2L}+O\left(\frac{\sigma-s_c}{L} \right)} + b_1^{\frac{\alpha}{2L}+O\left(\frac{\sigma-s_c}{L} \right)}\sum_{k=2}^p\left( \frac{\sqrt{\mathcal{E}_{\sigma}}}{b_1^{\sigma-s_c}}\right)^{k-1}\right] 
$$
with $\nu=\frac{\alpha}{\ell-\alpha}$. One has $\sum_{k=2}^p\left( \frac{\sqrt{\mathcal{E}_{\sigma}}}{b_1^{\sigma-s_c}}\right)^{k-1}\ll 1$ under the bootstrap conditions \fref{eq:bootstrap estimations sur epsilon}. Therefore, we see that there exists a small constant $0<\delta\ll 1$, such that if one chooses $s_0$ large enough, this equation can be rewriten as:
$$
\frac{d}{ds} \left\{ \frac{\mathcal{E}_{\sigma}}{\lambda^{2(s_L-s_c)}} \right\}\leq \frac{1}{\lambda^{2(\sigma-s_c)}}\frac{b_1}{\lambda}\sqrt{\mathcal{E}_{\sigma}}b_1^{(\sigma-s_c)\frac{\ell}{\ell-\alpha}+\delta}.
$$
Still under the bootstrap assumption we can integrate this equation:
\be \label{end:eq:estimation low intermediaire}
\mathcal{E}_{\sigma}(s)\leq \mathcal{E}_{\sigma}(0)\lambda^{2(\sigma-s_c)}+\lambda^{2(\sigma-s_c)}\int_{s_0}^s \frac{b_1}{\lambda^{2(\sigma-s_c)}}\sqrt{K_2} b_1^{2(\sigma-s_c)\frac{\ell}{\ell-\alpha}+\delta} .
\ee
We recall that $\lambda(0)=1$ and from \fref{end:eq:asymptotique lambda} and the bootstrap assumptions \fref{eq:bootstrap modes stables} and \fref{eq:bootstrap modes instables} on $b_1$:
$$
\left| \lambda(s)-\left(\frac{s_0}{s}\right)^{\frac{\ell}{\ell-\alpha}} \right|\leq  \frac{1}{s_0^{c\tilde{\eta}}}\left(\frac{s_0}{s} \right)^{\frac{\ell}{\ell-\alpha}} \ \text{and} \ \left|b_1-\frac{c_1}{s} \right|\leq \frac{1}{s^{1+\tilde{\eta}}} .
$$
It implies: $\lambda(s)\leq \frac{C}{s^{\frac{\ell}{\ell-\alpha}}}$ and $b_1\sim\frac{c_1}{s}$. Consequently:
$$
\mathcal{E}_{\sigma}(0)\lambda^{2(\sigma-s_c)}\leq C\mathcal{E}_{\sigma}(0)b_1^{2(\sigma-s_c)\frac{\ell}{\ell-\alpha}} .
$$
Given the initial condition \fref{eq:estimation initiale de epsilon0} on $\mathcal{E}_{\sigma}(0)$ it yields:
\be \label{end:eq:estimation low echelle}
\mathcal{E}_{\sigma}(0)\lambda^{2(\sigma-s_c)}\leq b_1^{2(\sigma-s_c)\frac{\ell}{\ell-\alpha}}.
\ee
For the integral term one has:
$$
\lambda^{2(\sigma-s_c)}\int_{s_0}^s \frac{b_1}{\lambda^{(2(\sigma-s_c)}}b_1^{2\frac{\ell}{\ell-\alpha}(\sigma-s_c)+\delta} \leq C\lambda^{2(\sigma-s_c)}\leq C b_1^{2(\sigma-s_c)\frac{\ell}{\ell-\alpha}}
$$
because the integral is convergent ($\frac{b_1}{\lambda^{(2(\sigma-s_c)}}b_1^{2\frac{\ell}{\ell-\alpha}(\sigma-s_c)}\leq s^{-1-\delta}$). Therefore:
\be \label{end:eq:estimation low force}
\lambda^{2(\sigma-s_c)}\int_{s_0}^s \frac{b_1}{\lambda^{(2(\sigma-s_c)}}b_1^{2\frac{\ell}{\ell-\alpha}(\sigma-s_c)}\sqrt{K_2}\leq C \sqrt{K_2}b_1^{\frac{\ell}{\ell-\alpha}(\sigma-s_c)} .
\ee
Injecting the two estimates \fref{end:eq:estimation low echelle} and \fref{end:eq:estimation low force} we found in \fref{end:eq:estimation low intermediaire} gives:
$$
\mathcal{E}_{\sigma}(s)\leq b_1^{2(\sigma-s_c)\frac{\ell}{\ell-\alpha}}\left( 1+C\sqrt{K_2}\right) ,
$$
and $\left( 1+C\sqrt{K_2}\right)\leq \frac{K_2}{2}$ for $K_2$ large enough.\\

\underline{The high Sobolev norm:} We recall the estimate of Proposition \ref{trappedregime:pr:high sobo}:
$$
\begin{array}{r c l}
\frac{d}{dt}\left\{ \frac{\mathcal{E}_{s_L}}{\lambda^{2(s_L-s_c)}}+O\left( \frac{\mathcal{E}_{s_L}b_1^{\eta(1-\delta_0)}}{\lambda^{2(s_L-s_c)}}\right) \right\}  &\leq& \frac{C(M)}{\lambda^{2(s_L-s_c)}}\frac{b_1}{\lambda}  \Bigl[ \mathcal{E}_{s_L}b_1^{\frac{\alpha}{2L}+O\left( \frac{\sigma-s_c}{L} \right)}\underset{k=2}{\overset{p}{\sum}} \left[\frac{\sqrt{\mathcal{E}_{\sigma}}}{b_1^{\sigma-s_c}} \right]^{k-1} \\
&&+C(N)\mathcal{E}_{s_L,\text{loc}}+\frac{\mathcal{E}_{s_L}}{N^{\frac{\delta_0}{2}}} + \sqrt{\mathcal{E}_{s_L}}b_1^{L+(1-\delta_0)(1+\eta)} \Bigr]
\end{array}
$$
with $C(M)$ independent of $N$. In the trapped regime \fref{eq:bootstrap estimations sur epsilon}, by taking $s_0$ large enough one has:
$$
\mathcal{E}_{s_L}b_1^{\frac{\alpha}{2L}+O\left(\frac{\sigma-s_c}{L}\right)} \underset{k=2}{\overset{p}{\sum}}\left(\frac{\sqrt{\mathcal{E}_{\sigma}}}{b_1^{\sigma-s_c}} \right)^{k-1}\leq \frac{C\mathcal{E}_{s_L}}{N^{\frac{\delta_0}{2}}}.
$$
So the previous equation becomes:
$$
\begin{array}{r c l}
&\frac{d}{ds}\left\{ \frac{\mathcal{E}_{s_L}}{\lambda^{2(s_L-s_c)}}+O\left( \frac{\mathcal{E}_{s_L}}{\lambda^{2(s_L-s_c)}}b_1^{\eta(1-\delta_0)}\right) \right\} \\ \leq& \frac{Cb_1}{\lambda^{2(s_L-s_c)}}
\times \left[ \frac{\mathcal{E}_{s_L}}{N^{\frac{\delta_0}{2}}} + \sqrt{\mathcal{E}_{s_L}}b_1^{L+(1-\delta_0)(1+\eta)}+C(N)\mathcal{E}_{s_L,loc}\right] , 
\end{array}
$$
(by multipliying the constant $C$ by 2). We also have by the Proposition \ref{thetrapped:pr:morawetz}:
$$
\frac{d}{ds} \left(\frac{\mathcal{M}}{\lambda^{2(s_L-s_c}}\right)\geq \frac{\delta}{2N^{\delta}\lambda^{2(s_L-s_c)}} \mathcal{E}_{s_L,loc}-\frac{C}{A^{\delta}\lambda^{2(s_L-s_c)}}\mathcal{E}_{s_L}-\frac{C(A,N)\sqrt{\mathcal{E}_{s_L}}}{\lambda^{2(s_L-s_c)}}b_1^{L+3} .
$$
Let $a>0$. Once $N$, $K_1$ and $A$ are chosen, for $s_0$ small enough we have:
$$
\frac{C C(N)b_1}{\lambda^{2(s_L-s_c)}}\mathcal{E}_{s_L,\text{loc}} \leq \frac{C(N)}{a}\left( \frac{d}{ds}\left( \frac{\mathcal{M}}{\lambda^{2(s_L-s_c)}} \right) \right)+\frac{C(N,M)b_1}{A^{\delta}\lambda^{2(s_L-s_c)}}\mathcal{E}_{s_L}+\frac{C(A,N)}{\lambda^{2(s_L-s_c)}}\sqrt{\mathcal{E}_{s_L}} ,
$$
which gives for the evolution of the high Sobolev norm the following monotonicity formula:
$$
\begin{array}{r c l}
&\frac{d}{ds}\left\{ \frac{\mathcal{E}_{s_L}}{\lambda^{2(s_L-s_c)}}+O\left( \frac{\mathcal{E}_{\sigma}}{\lambda^{2(s_L-s_c)}}b_1^{\eta(1-\delta_0)}\right) \right\}\\
 \leq & \frac{Cb_1}{\lambda^{2(s_L-s_c)}}   \big{[} \frac{\mathcal{E}_{\sigma}}{N^{\frac{\delta_0}{2}}} + C(A,N)\sqrt{\mathcal{E}_{s_L}}b_1^{L+(1-\delta_0)(1+\eta)} +\frac{C(N)}{A^{\delta}}\mathcal{E}_{s_L}\big{]}+\frac{C(N)}{a} \frac{d}{ds}\left( \frac{\mathcal{M}}{\lambda^{2(s_L-s_c)}} \right) .
\end{array}
$$
Let $a'>0$ be a large constant. By letting $N$ be large enough, then by letting $A$ and $a$ be large enough we can reformulate it as:
$$
\begin{array}{r c l}
\frac{d}{ds}\left\{ \frac{\mathcal{E}_{s_L}}{\lambda^{2(s_L-s_c)}}+O\left( \frac{\mathcal{E}_{\sigma}}{\lambda^{2(s_L-s_c)}}b_1^{\eta(1-\delta_0)}\right) \right\} &\leq& \frac{b_1}{\lambda^{2(s_L-s_c)}} \left[ \frac{\mathcal{E}_{\sigma}}{a'} + C\sqrt{\mathcal{E}_{s_L}}b_1^{L+(1-\delta_0)(1+\eta)}\right]\\
&&+\frac{1}{a'} \frac{d}{ds}\left( \frac{\mathcal{M}}{\lambda^{2(s_L-s_c)}} \right) ,
\end{array}
$$
with $C$ independent of $a'$. We will now integrate it in time as we did for the low Sobolev norm, using the bootstrap assumption \fref{eq:bootstrap estimations sur epsilon}:
$$
\begin{array}{r c l}
\mathcal{E}_{s_L}(s)&\leq& C(s_0)(\mathcal{E}_{s_L}(s_0)+|\mathcal{M}(s_0|)\lambda^{2(s_L-s_c)}+\frac{1}{a'}|\mathcal{M}(s)| \\
&+&\lambda^{2(s_L-s_c)}\int_{s_0}^s \frac{b_1}{\lambda^{(2(s_L-s_c)}}\left( \frac{K_1}{a'}+C\sqrt{K_1}\right)b_1^{2(L+(1-\delta_0)(1+\eta))} .
\end{array}
$$
We recall that: $|\mathcal{M}|\leq C(A)\mathcal{E}_{s_L}$, so:
$$
\left| \frac{\mathcal{M}}{a'} \right|\leq \frac{C(M)}{a'}\mathcal{E}_{s_L}  .
$$
We then compare using the equivalents for $b_1$ and $\lambda$: 
$$
b_1^{2L+2(1-\delta_0)(1+\eta)}\approx \frac{1}{s^{2L+2(1-\delta_0)(1+\eta)}} .
$$
$$
\lambda^{2(s_L-s_c)}\sim \frac{1}{s^{2\frac{\ell}{\ell-\alpha}(L+k_0-\frac{d}{2}+\frac{2}{p-1})}}\lesssim \frac{1}{s^{2L+\frac{\alpha}{\ell}L+O\left(\frac{1}{L^2} \right)}}  .
$$
This implies $\lambda^{2(s_L-s_c)}=o(b_1^{2(L+(1-\delta_0)(1+\eta)})$ (remember that $\ell \ll L$). Because of the initial bound \fref{eq:estimation initiale de epsilon0} on $\mathcal{E}_{s_L}(0)$ there holds for all $s_0\leq s\leq s^*$:
$$
C(s_0)(\mathcal{E}_{s_L}(0)+|\mathcal{M}(s_0)|)\lambda^{2(s_L-s_c)}\leq b_1^{2L+2(1-\delta_0)(1+\eta)} .
$$
We now treat the integral term using the equivalents for $\lambda(s)$ and $b_1(s)$:
$$
\begin{array}{r c l}
&\lambda^{2(s_L-s_c)}\int_{s_0}^s \frac{b_1}{\lambda^{(2(s_L-s_c)}}b_1^{2(L+(1-\delta_0)(1+\eta))}\\
\leq & C s^{-2(s_L-s_c)\frac{\ell}{\ell-\alpha}}\int_{s_0}^s s^{-1-2(L+(1-\delta_0)(1+\eta)+2(s_L-s_c)(\frac{\ell}{\ell-\alpha})} \\
\leq & C s^{-2L-2(1-\delta_0)(1+\eta)} \leq  Cb_1^{2L+2(1-\delta_0)(1+\eta)} ,
\end{array}
$$
with the constant $C$ just depending on $c_1$ and $s_0$. The integral is indeed divergent from $-2(L+(1-\delta_0)(1+\eta))+2(s_L-s_c)\frac{\ell}{\ell-\alpha}>0 $ (as $\ell \ll L$). Eventually the three estimations we have shown allow us to conclude:
$$
\left(1-\frac{C(N)}{a'}\right)\mathcal{E}_{s_L}(s)\leq b_1^{2L+2(1-\delta_0)(1+\eta)}\left(\frac{C}{a'}K_1+C\sqrt{K_1}+C \right) .
$$
For $a'$ and $K_1$ big enough one has:
$$
\frac{\frac{C}{a'}K_1+C\sqrt{K_1}+C}{1-\frac{C(N)}{a'}}\leq \frac{K_1}{2} 
$$
(remember that here $\frac{C(N)}{a'}=\frac{C(N)}{a}$ and since we choose $a$ after $M$ this term can be arbitrarily small.
\end{proof}

We now rule out case (ii) in the possible exit scenarios. We recall that the small coefficients $(\epsilon_i)_{\ell+1\leq i \leq L}$ are used to quantify the control over the stable modes in the trapped regime of Proposition \ref{prop:bootstrap}.

\begin{lemma}[control of the stable modes]\label{endoftheproof:lem: controle des modes stables}
After having choosen the other constants correctly, there exists small enough constants $\tilde{\eta}$, and $(\epsilon_i)_{\ell+1\leq i \leq L}$ such that for $s_0$ big enough, untill time $s^*$ there holds:
\begin{equation}
|V_1|\leq \frac{1}{2s^{\tilde{\eta}}}, \ \text{and} \ |b_k(s)|\leq \frac{\epsilon_i}{2s^{k+\tilde{\eta}}} \ \text{for} \ \ell+1\leq k \leq L.
\end{equation}
\end{lemma}

\begin{proof}[Proof of Lemma \ref{endoftheproof:lem: controle des modes stables}]
\underline{The stable modes for $\ell+1\leq i \leq L-1$:} Let $i$ be an integer, $\ell+1\leq i \leq L-1$. We recall that the evolution of $b_i$ is given by:
$$
\begin{array}{r c l}
b_{i,s}&=&-(i-\alpha)b_1b_i+b_{i+1}+O(b_1\sqrt{\mathcal{E}_{s_L}}+b_1^{L+3}) \\
&=&-\frac{c_1(i-\alpha)}{s}b_i-(i-\alpha)\frac{U_1b_i}{s}+b_{i+1}+O(s^{-L-1-(1-\delta_0)})\\
&=&-\frac{c_1(i-\alpha)}{s}b_i+b_{i+1}+O(s^{-1-i-2\tilde{\eta}}),
\end{array}
$$
for $\tilde{\eta}$ small enough, because $U_1b_i=O(s^{-2\tilde{\eta}})$ under the bootstrap assumptions. Hence for $s_0$ large enough it gives:
$$
|b_{i,s}+(i-\alpha)c_1\frac{b_i}{s}|\leq \frac{2\epsilon_{i+1}}{s^{i+1+\tilde{\eta}}} ,
$$
which we rewrite as:
\be \label{end:eq:edo bi leq L-1}
|\frac{d}{ds}(s^{(i-\alpha)c_1}b_i) |\leq 2\epsilon_{i+1}s^{(i-\alpha)c_1-(i+1+\tilde{\eta})} .
\ee
We notice that $(i-\alpha)c_1=\frac{l(i-\alpha)}{l-\alpha}>i$. So for $\tilde{\eta}$ small enough one has $(i-\alpha)c_1\geq i+\tilde{\eta}$. With these two facts in mind we integrate the last equation and estimate using the initial condition \fref{eq:estimation initiale des modes stables}:
$$
\begin{array}{r c l}
|b_i(s)|&\leq& b_i(0)\frac{s^{(i-\alpha)c_1}}{s^{(i-\alpha)c_1}}+\frac{2\epsilon_{i+1}}{s^{(i-\alpha)c_1}}\int_{s_0}^s \tau^{(i-\alpha)c_1-(i+1+\tilde{\eta})}d\tau \\
&\leq & \frac{\epsilon_i}{10s^{i+\tilde{\eta}}}+\frac{2\epsilon_{i+1}}{((i-\alpha)c_1-i)s^{i+\tilde{\eta}}} ,
\end{array}
$$
the integral that appeared being divergent. We therefore see here that we can choose the constants of initial smallness $(\epsilon_i)_{\ell+1\leq i \leq L}$ one after each other: once $\epsilon_i$ is choosen we can take $\epsilon_{i+1}$ small enough to produce $\frac{\epsilon_i}{10}+\frac{2\epsilon_{i+1}}{(i-\alpha)c_1}<\frac{\epsilon_i}{2} $. This, of course, makes only sense if one is able to bootstrap the estimate on the last parameter $b_L$.\\

\underline{The stable mode $i=L$:} We recall the improved modulation equation for $b_L$:
\be \label{end:eq:edo bL}
\begin{array}{r c l}
&& \left| b_{L,s}+(L-\alpha)b_1 b_L -\frac{d}{ds}\left[ \frac{\langle \bos{H}^{L}\bos{\varepsilon} , \chi_{B_0} \bos{\Lambda} \bos{Q}\rangle}{\Bigl\langle  \chi_{B_0} \Lambda^{(1)} Q, \Lambda^{(1)} Q +(-1)^{\frac{L-1}{2}}\left(\frac{\partial S_{L+2}}{\partial b_L}\right)_{L-1}\Bigr\rangle }\right] \right| \\
&\leq& \frac{1}{B_0^{\delta_0}}C(M)\left[  \sqrt{\mathcal{E}_{s_L}} +b_1^{L+(1-\delta_0)(1+\eta)}\right] .
\end{array}
\ee
We have seen in \fref{thetrapped:highsobo:estimation C(xi)} that:
$$
\left| \frac{\langle \bos{H}^{L}\bos{\varepsilon } , \chi_{B_0} \bos{\Lambda} \bos{Q}\rangle}{\Bigl\langle  \chi_{B_0} \Lambda^{(1)} Q, \Lambda^{(1)} Q +(-1)^{\frac{L-1}{2}}\left(\frac{\partial S_{L+2}}{\partial b_L}\right)_{L-1}\Bigr\rangle}\right|\leq C\sqrt{\mathcal{E}_{s_L}}b_1^{\delta_0-1}\lesssim s^{-L-1-\eta(1-\delta_0)},
$$
We integrate the time evolution of $b_L$ the same way we did for the other stable modes. This time, however, the force term comes from the error $\bos{\varepsilon}$. We reformulate \fref{end:eq:edo bL}
\be
\frac{d}{ds}(s^{(L-\alpha)c_1}b_L)=s^{(L-\alpha)c_1}\frac{d}{ds}O\left( \frac{1}{s^{L+\eta(1-\delta_0)}} \right)+s^{(L-\alpha)c_1}O\left(\frac{1}{s^{L+1+\eta(1-\delta_0)}} \right).
\ee
Then as for the $b_i$'s for $\ell+1\leq i \leq L-1$, we integrate and use integration by parts to find, under the initial smallness assumption on $b_L$ and for $\tilde{\eta}$ small enough:
$$
|b_L(s)|\leq \frac{\epsilon_{L}}{10s^{L+\tilde{\eta}}}+\frac{C}{s^{L+\eta(1-\delta_0)}},
$$
where $C$ is just some integration constant. Hence by choosing $s_0$ large enough and $\tilde{\eta}<\eta(1-\delta_0)$ we have: $|b_L(s)|\leq \frac{\epsilon_L}{2s^{L+\tilde{\eta}}}$.\\

\underline{control of $V_1$.} We recall that $V_1$ is the eigenvector associated to the eigenvalue $-1$ of the linearized operator $A_{\ell}$, defined by \fref{thetrapped:eq:def Vi}: $V_1=(P_{\ell}U)_1=\sum_1^{\ell}p_{1,i}U_i$. We first calculate the time evolution of the $U_i$'s for for $1\leq i \leq \ell$ thanks to the modulation equation \fref{lem:modulation}:
$$
\begin{array}{r c l}  \label{endoftheproof:eq:eq devolution des Ui}
U_{i,s}&=&\frac{(AU)_i}{s}+\frac{O(|U^2|)}{s}+s^i O(b_1C(M)\sqrt{\mathcal{E}_{s_L}}+C(M)b_1^{L+3})\\
&:=&\frac{(AU)_i}{s}+\frac{O(|U^2|)}{s}+s^i g_i(s),
\end{array}
$$
where $g_i(s)$ stands for the terms added in the full equation. It leads to the following expression for the time evolution of $V_1$:
\begin{equation} 
V_{1,s}=-\frac{1}{s}V_1 +\frac{O(|V|^2)}{s}+\sum_{j=1}^L p_{1,j} s^j g_j(s)+q_1s^{\ell}b_{\ell+1},
\end{equation}
where $q_1$ is a constant defined by \fref{linearized:eq:diagonalisation}. We reformulate it under the bootstrap assumptions as:
$$
\frac{d}{ds} (sV_1)=sO\left( \frac{1}{s^{1+2\tilde{\eta}}}+\frac{1}{s^{L-l}}\right)+sq_1s^{\ell}b_{\ell+1} .
$$
As $|b_{\ell+1} |\leq \epsilon_{\ell+1}s^{-\tilde{\eta}}$ under the bootstrap assumptions, for $s_0$ large enough the time integration gives:
$$
|V_1(s)|\leq \frac{s_0|V_1(s_0)|}{s}+O\left( \frac{\epsilon_{\ell+1}}{s^{\tilde{\eta}}} \right) .
$$
We recall the initial assumption $V_1(s_0)\leq \frac{1}{10s_0}$. For $\epsilon_{\ell+1}$ small enough the last equation becomes:
$$
|V_1(s)|\leq \frac{1}{2s^{\tilde{\eta}}} .
$$
\end{proof}

We now fix all the constants of the analysis, and the constants of smallness, so that the last two lemmas hold. We just allow us to increase the initial time $s_0$ if necessary, as it still make these two lemmas hold.

\begin{remark}
we now know that $s^*$ is characterized by:
$$
(V_2(s^*),...,V_{\ell}(s^*))\in \mathcal{S}^{\ell-1}\left(\frac{1}{s^{\tilde{\eta}}} \right) .
$$
\end{remark} \label{end:rem:condition sortie}
We fix $\varepsilon(s_0)$, $V_1(s_0)$ and $b_i(s_0)$ satisfying the smallness assumptions \fref{eq:estimation initiale de epsilon0} and \fref{eq:estimation initiale des modes stables}. we define the following application:
\begin{equation}
\begin{array}{l l l l}
f: &\mathcal{D}(f)\subset \mathcal{B}^{\ell-1}\left(\frac{1}{s_0^{\tilde{\eta}}} \right) &\rightarrow & \mathcal{S}^{\ell-1}\left(\frac{1}{s_0^{\tilde{\eta}}} \right) \\
& (V_2(s_0),...V_{\ell}(s_0)) & \mapsto & \frac{(s^*)^{\tilde{\eta}}}{s_0^{\tilde{\eta}}} (V_2(s^*),...,V_{\ell}(s^*)) ,
\end{array}
\end{equation}
With domain:
\begin{equation}
\mathcal{D}(f)=\left\{(V_2(s_0),...,V_{\ell}(s_0))\in \mathcal{B}^{\ell-1}\left(\frac{1}{s_0^{\tilde{\eta}}} \right), \ \text{such} \ \text{that} \ s^* <+\infty  \right\} .
\end{equation}
We prove in the following lemma that $\mathcal{D}$ is non empty, open in $\mathcal{B}^{\ell-1}\left(\frac{1}{s_0^{\tilde{\eta}}} \right)$, that $f$ is continuous and equivalent to the identity on the sphere $ \mathcal{S}^{\ell-1}\left(\frac{1}{s_0^{\tilde{\eta}}} \right) $.

\begin{lemma}\label{endoftheproof:lem:proprietes de f}
\emph{(Topological properties of $f$)} The following properties hold:
\begin{itemize}
\item[(i)] $\mathcal{D}(f)$ is non empty and open, satisfying $\mathcal{S}^{\ell-1}\left(\frac{1}{s_0^{\tilde{\eta}}} \right)\subset \mathcal{D}(f)$ .
\item[(ii)] $f$ is continuous and is the identity on the sphere $\mathcal{S}^{\ell-1}\left(\frac{1}{s_0^{\tilde{\eta}}} \right)$.
\end{itemize}
\end{lemma}

\begin{proof}[Proof of Lemma \ref{endoftheproof:lem:proprietes de f}]
We recall that $V_i$ is the projection of $U$ on the unstable direction $v_i$ associated to the eigenvalue $\frac{i\alpha}{\ell-\alpha}$ of the matrix $A_{\ell}$, see Lemma \fref{lem:linearisationsystemdyn}. To ease notation we will write $\mu_i:=\frac{i\alpha}{\ell-\alpha}$ the eigenvalues. From the time evolution of $U_i$ for $1\leq i \leq \ell$ computed in \fref{endoftheproof:eq:eq devolution des Ui} we get that the time evolution of $V_i$ is:
$$
\begin{array}{r c l}
V_{i,s}&=&\frac{\mu_i}{s}V_i +O(s^{-1-2\tilde{\eta}})+O(s^{L-\ell})+O(\epsilon_{\ell +1}s^{-1-\tilde{\eta}})\\
&=& \frac{\mu_i}{s}V_i +O(\epsilon_{\ell +1}s^{-1-\tilde{\eta}}) .
\end{array}
$$
Let $(V_2(s_0),...,V_{\ell}(s_0))\in \mathcal{S}^{\ell-1}\left(\frac{1}{s_0^{\tilde{\eta}}} \right)$ be an initial data on the sphere. We claim that $s^*=0$ which implies of course:
$$
f((V_2(s_0)),...V_l(s_0)))=(V_2(s_0)),...V_l(s_0)) .
$$
This will prove that $\mathcal{D}(f)$ is non empty and that $f$ is equivalent to the identity on $\mathcal{S}^{\ell-1}\left(\frac{1}{s_0^{\tilde{\eta}}} \right)$. To prove that, we just compute the scalar product between the time derivative of $(V_2(s),...V_{\ell}(s))$ and an outgoing normal vector to the sphere at the point $(V_2(s_0),...V_{\ell}(s_0))$:
$$
(V_2(s_0),...V_{\ell}(s_0)).(V_{2,s}(s_0),...V_{\ell,s}(s_0)) =\sum_{i=2}^{\ell} \frac{\mu_i}{s_0}|V_i|^2 +O(\epsilon_{\ell+1}s_0^{-1-2\tilde{\eta}})>0
$$
for $\epsilon_{\ell+1}$ small enough. In addition, this inequality uniformly holds on the sphere. For any small time $s'$, we have that $(V_2(s_0+s'),...V_{\ell}(s_0+s'))$ is outside the ball, which implies $s^*=s_0$.\\
\\
At $s=s_0$, this says that close to the border of the ball $\mathcal{B}^{\ell-1}(\frac{1}{s_0^{\tilde{\eta}}})$ the force term whose size is $O(\epsilon_{\ell+1}s_0^{-1-\tilde{\eta}})$ is overtaken by the linear repulsive dynamics. We are going to show that this is also true for $s_0\leq s \leq s^* $.\\
\\
We now prove that $f$ is continuous. Let $s$ be such that $s_0\leq s\leq s^*$ and let $(V_2(s_0),...,V_{\ell}(s_0))$ be an initial data such that at time $s$, $\frac{1}{2s^{\tilde{\eta}}}\leq(V_2(s),...,V_{\ell}(s))$. The same computation gives:
$$
\begin{array}{r c l}
\frac{d}{ds}|V|^2&=&(V_2(s),...V_{\ell}(s)).(V_{2,s}(s),...V_{\ell,s}(s)) \\
&\geq& \text{min}((\mu_i)_{2\leq i \leq \ell})\frac{1}{4s^{1+2\tilde{\eta}}} +O(\frac{\epsilon_{\ell+1}}{s^{1+2\tilde{\eta}}} ) \\
&>&0 ,
\end{array}
$$
once again provided one has taken $\epsilon_{\ell+1}$ small enough. It implies that at time $s$ fixed, there exists a small enough time $s^+>0$ and a small enough distance $r>0$ such that:
$$
\frac{1}{s^{\tilde{\eta}}}-r\leq |V(s)| \leq \frac{1}{\tilde{\eta}} \ \text{implies} \ s\leq s^* \leq s^+ ,
$$
ie the orbit leaves the ball $\mathcal{B}^{\ell-1}(\frac{1}{s^{\tilde{\eta}}})$ in finite time. Let now $(V_2(s_0),...,V_{\ell}(s_0))$ be an initial data such that $s^*<+\infty$. Since the time evolution of $V$ is a lipischitz continuous function of our problem, there is local continuity of the trajectories. Take $s^-<s^*$ close enough to $s^*$ so that $1/s^{\tilde{\eta}}-\frac{r}{2}\leq|V(s^-)|$, there exists a small enough distance $r_0>0$ such that if $|V'(s_0)-V(s_0)|<r_0$ then $|V'(s)-V(s)|<\frac{r}{4}$ for $s_0\leq s \leq s^-$. The exit result we just stated implies that $s^-<s^*(V')$ and that $1/s^{\tilde{\eta}}-\frac{3r}{4}\leq V'(s^-)$. So that $s^-\leq s^*(V') \leq s^-+s^+$. We have proven that $\mathcal{D}(f)$ is open.\\
\\
From direct inspection, with the use of continuity properties, it is easy to prove in the same spirit that the function $s^*$ is continuous on $\mathcal{D}$, and that $f$ is continuous too on $\mathcal{D}(f)$.
\end{proof}

We have reached the end of the proof. Indeed, if for all choices of initial data $(V_2(s_0),...,V_{\ell}(s_0))$ we had $s^*<+\infty$, ie that no solution stayed in the trapped regime for all time, then $f$ would be a continuous function from the ball $\mathcal{B}^{\ell-1}(\frac{1}{s_0^{\tilde{\eta}}})$ onto the sphere $\mathcal{S}^{\ell-1}(\frac{1}{s_0^{\tilde{\eta}}})$ being equal to the identity at the border. This would be a contradiction to Brouwer's fixed point theorem. It implies the existence of at least one initial data $(V_2(s_0),...,V_{\ell}(s_0))\in \mathcal{B}^{\ell-1}\left(\frac{1}{s_0^{\tilde{\eta}}} \right)$ such that the solution of (NLW) stays in the trapped regime described by Proposition \ref{prop:bootstrap}.\\
\\
We now end the proof of the main theorem. We know from Proposition \ref{prop:bootstrap} that there exists an orbit satisfying the assumptions of the trapped regime. We have computed that in that case there exists a constant $c>0$ such that:
$$
\frac{1}{c}s^{-\frac{\ell}{\ell-\alpha}}\leq \lambda \leq c s^{-\frac{\ell}{\ell-\alpha}} .
$$
Since $\frac{ds}{dt}=\frac{1}{\lambda}$ it gives:
$$
\frac{1}{c'}s^{\frac{\ell}{\ell-\alpha}}\leq \frac{ds}{dt}\leq c's^{\frac{\ell}{\ell-\alpha}}.
$$
This is an explosive ODE, we have that there exists a maximal time $T$ with:
$$
s \sim C(\bos{u}(0)) (T-t)^{-\frac{\ell-\alpha}{\alpha}} \ \text{as} \ t\rightarrow T .
$$
This implies:
$$
\frac{1}{c} (T-t)^{\frac{\ell}{\alpha}}\leq \lambda(t)\leq c (T-t)^{\frac{\ell}{\alpha}} \ \text{as} \ t \rightarrow T .
$$


\subsection{Behavior of Sobolev norms near blow-up time}
\label{end:subsection:normes}
We now prove the convergence of the norms \fref{intro:eq:comportement norme critique}, \fref{intro:eq:comportement norme critique 2}, \fref{intro:eq:convergence surcritique} and \fref{intro:eq:bornitude sous critique}. First note that our analysis relies only on the study of supercritical Sobolev norms $(\dot{H}^{\sigma}\cap \dot{H}^{s_L})\times(\dot{H}^{\sigma-1}\cap \dot{H}^{s_L-1})$ for the perturbative term $\tilde{\bos{\alpha}}_{b,\frac{1}{\lambda}}+\bos{w}$. For this reason, the finiteness of the $\dot{H}^1\times L^2$ norm of the initial data is not a requirement. Still, it is worth studying the behavior of lower order Sobolev norms because it applies when taking "nice" initial data, say smooth and with compact support, and because their asymptotic really corresponds to the concentration of a critical object. We still consider a solution described by Proposition \ref{prop:bootstrap} but now under the following decompositions:
\be
\bos{u}=\bos{Q}_{\frac{1}{\lambda}}+\tilde{\bos{w}}=(\bos{Q}+\tilde{\bos{\varepsilon}})_{\frac{1}{\lambda}}, \ \ \text{ie} \ \tilde{\bos{w}}=\bos{w}+\tilde{\bos{\alpha}}_{b,\frac{1}{\lambda}}, \ \ \text{and} \ \tilde{\bos{\varepsilon}}=\bos{\varepsilon}+\tilde{\bos{\alpha}}_{b},
\ee
\be \label{end:eq:decomposition soliton coupe}
\bos{u}=\chi \bos{Q}_{\frac{1}{\lambda}}+\bos{w}'=(\chi_{\frac{1}{\lambda}}\bos{Q}+\bos{\varepsilon}')_{\frac{1}{\lambda}}, \ \ \text{ie} \ \bos{w}'=\tilde{\bos{w}}+((1-\chi_{\frac{1}{\lambda}})\bos{Q})_{\frac{1}{\lambda}}, \ \ \text{and} \ \bos{\varepsilon}'=\tilde{\bos{\varepsilon}}+(1-\chi_{\frac{1}{\lambda}})\bos{Q}.
\ee
We recall that the subscript $\frac{1}{\lambda}$ has a different meaning when it applies to $\chi$, see \fref{intro:eq:definition scaling chi}. First note that because of \fref{eq:bootstrap estimations sur epsilon} and because $\mathcal{E}_{s_L}$ controls the usual Sobolev norms, see \fref{annexe:coercivitenormes:eq:mathcalEsL sim HsL}, one has by interpolation:
\be
\parallel \bos{\varepsilon} \parallel_{\dot{H}^{s}\times \dot{H}^{s-1}} \underset{t\rightarrow T}{\rightarrow} 0 \ \text{for} \ \text{all} \ \sigma\leq s \leq s_L.
\ee
Moreover, this convergence is also true for the perturbation on the manifold of approximate blow-up solutions:
$$
\parallel \tilde{\bos{\alpha}}_{b} \parallel_{\dot{H}^{s}\times \dot{H}^{s-1}} \underset{t\rightarrow T}{\rightarrow} 0 \ \text{for} \ \text{all} \ \sigma\leq s \leq s_L.
$$
so we get for the perturbation:
\be \label{end:eq:convergence erreur vers 0 geq sigma}
\parallel \tilde{\bos{\varepsilon}} \parallel_{\dot{H}^{s}\times \dot{H}^{s-1}} \underset{t\rightarrow T}{\rightarrow} 0 \ \text{for} \ \text{all} \ \sigma\leq s \leq s_L.
\ee
We suppose from now on that $\parallel \bos{u}(0) \parallel_{\dot{H}^{1}\times L^2}$ is finite. This implies the boundedness of the perturbation at initial time: $\parallel \bos{\varepsilon}'(0) \parallel_{\dot{H}^{1}\times L^2}= \parallel \bos{w}'(0)\parallel_{\dot{H}^{1}\times L^2}\leq C(\bos{u}(0)) $. We show first that this last quantity stays bounded.

\begin{lemma}[Boundedness in $\dot{H}^1\times L^2$] \label{end:lem:borne L2}
Suppose $\bos{u}$ is a solution described by Proposition \ref{prop:bootstrap}, such that $\bos{u}(0)\in \dot{H}^1\times L^2$. Then there exists a constant $C(\bos{u}(0))$ such that for all $0\leq t<T$:
\be \label{end:eq:borne L2}
\parallel \bos{u}\parallel_{\dot{H}^1\times L^2} \leq C(\bos{u}(0))
\ee
\end{lemma}

\begin{proof}[Proof of Lemma \ref{end:lem:borne L2}]
We first compute that under the decomposition \fref{end:eq:decomposition soliton coupe}, the soliton's contribution to the $\dot{H}^1$ norm is finite:
\be \label{end:eq:contribution finie du soliton}
\parallel \chi Q_{\frac{1}{\lambda}}\parallel_{\dot{H}^1} = \frac{1}{\lambda^{1-s_c}} \parallel \chi_{\frac{1}{\lambda}}Q \parallel_{\dot{H}^1} \leq  \frac{1}{\lambda^{1-s_c}} C \left( \int_1^{\frac{1}{\lambda}} y^{d-\frac{4}{p-1}-2} \right)^{\frac{1}{2}}\leq C.
\ee
Therefore, the lemma is proven once we show that the $\dot{H}^1\times L^2$ norm of $\bos{w}'$ stays finite. We are going to prove this by computing its time evolution under the bootstrap regime. We claim that:
\be \label{end:eq:borne evolution temporelle w'L2}
\frac{d}{dt}\parallel \bos{w}' \parallel_{\dot{H}^{1}\times L^2}^2\leq C\parallel \bos{w}' \parallel_{\dot{H}^{1}\times L^2}+C \sum_{k=1}^{p} \parallel \bos{w} \parallel_{\dot{H}^1\times L^2}^{2-c_k} \parallel \bos{w} \parallel_{\dot{H}^{\sigma}\times \dot{H}^{\sigma-1}}^{c_k'},
\ee
where for each $k$, $0<c_k\leq 2$. We start by proving this bound. The time evolution of $\bos{w}'$ is:
\be
\partial_t \bos{w}'=\bos{L}+\frac{1}{\lambda}\bos{\mathcal{F}}_{\frac{1}{\lambda}}+\frac{1}{\lambda}\bos{I}_{\frac{1}{\lambda}}
\ee
where $\bos{L}$ is the linear part, $\bos{L}:= \begin{pmatrix} w^{(2)} \\ \Delta w^{(1)}  \end{pmatrix}$, $\bos{\mathcal{F}}$ is the force term:
$$
\bos{\mathcal{F}}=\begin{pmatrix}  \lambda_t \chi_{\frac{1}{\lambda}}\Lambda^{(1)}Q \\ \chi_{\frac{1}{\lambda}}Q^p(\chi_{\frac{1}{\lambda}}^{p-1}-1)+(\lambda^2(\partial_{rr} \chi)_{\frac{1}{\lambda}}+\frac{d-1}{r}\lambda (\partial_r \chi)_{\frac{1}{\lambda}})  Q+2\lambda(\partial_r \chi)_{\frac{1}{\lambda}}\partial_r Q \end{pmatrix} ,
$$
and $\bos{I}$ is the interaction term: $\bos{I}=\begin{pmatrix} 0 \\ \sum_{k=1}^p C_k (\chi_{\frac{1}{\lambda}}Q)^{p-k}(\varepsilon^{'(1)})^k \end{pmatrix}$. It leads to the following expression for the time derivative of the norm:
\be \label{end:eq:expression derivee temporelle w'L2}
\frac{d}{dt} \parallel \bos{w}' \parallel_{\dot{H}^{1}\times L^2}^2 = 2\int \nabla w^{'(1)}.\nabla(L^{(1)}+\frac{1}{\lambda}\mathcal{F}^{(1)} )+  2\int w^{'(2)} (L^{(2)}+\frac{1}{\lambda}\mathcal{F}^{(2)}+\frac{1}{\lambda}I^{(2)} ).
\ee
We now want to estimate everything in the right hand side of \fref{end:eq:expression derivee temporelle w'L2}. The linear term's contribution is null:
\be \label{end:eq:contribution L}
\int \nabla w^{'(1)}.\nabla w^{'(2)}+w^{'(2)}\Delta w^{'(1)}=0.
\ee
We then compute the size of the force term. For the first coordinate:
\be \label{end:eq:norme H1 terme force haut}
\begin{array}{r c l}
\int \frac{1}{\lambda^2} |\nabla \mathcal{F}^{(1)}|^2  &=&   \frac{1}{\lambda^2} \frac{1}{\lambda^{2(1-s_c)}} \int \lambda_t^2 |\nabla(\chi_{\frac{1}{\lambda}}\Lambda^{(1)}Q)|^2 \leq  C\frac{1}{\lambda^{2(2-s_c)}}\int_1^{\frac{1}{\lambda}} y^{d-2\gamma-2-1}dy \\
&\leq& C\frac{\lambda_t^2}{\lambda^{2(2-s_c)+d-2\gamma-2}}\leq C\lambda_t^2 \lambda^{2\alpha-2}\leq C,
\end{array}
\ee
because $\alpha>2$ and $\lambda_t =b_1\rightarrow 0$ as $t\rightarrow T$. For the second coordinate:
\be \label{end:eq:norme L2 terme force bas}
\begin{array}{r c l}
\int \frac{1}{\lambda^2} |\mathcal{F}^{(2)}|^2&=&\frac{1}{\lambda^2} \frac{1}{\lambda^{2(1-s_c)}} \Bigl( \int | \chi_{\frac{1}{\lambda}}Q^p(\chi_{\frac{1}{\lambda}}^{p-1}-1)\\
&&+(\lambda^2(\partial_{rr} \chi)_{\frac{1}{\lambda}}+\frac{d-1}{r}\lambda (\partial_r \chi)_{\frac{1}{\lambda}})  Q+2\lambda(\partial_r \chi)_{\frac{1}{\lambda}}\partial_r Q   |^2\Bigr) \\
&\leq & C \frac{1}{\lambda^{2(2-s_c)}} \int_1^{\frac{1}{\lambda}} y^{d-4-\frac{4}{p-1}-1}dy\leq C \frac{1}{\lambda^{2(2-s_c)}}\frac{1}{\lambda^{d-\frac{4}{p-1}-4}}=C.
\end{array}
\ee
The bounds \fref{end:eq:norme H1 terme force haut} and \fref{end:eq:norme L2 terme force bas} imply the bound for the force term's contribution:
\be \label{end:eq:contribution F}
\left| \int \frac{1}{\lambda}\nabla w^{'(1)}.\nabla \mathcal{F}^{(1)}+\frac{1}{\lambda} w^{'(2)}\mathcal{F}^{(2)} \right| \leq C \parallel \bos{w}'  \parallel_{\dot{H}^1\times L^2}.
\ee 
We now turn to the $L^2$ norm of the interaction term. First we rescale:
\be \label{end:eq:norme L2 I intermediaire}
\left| \frac{1}{\lambda}\int w^{'(2)} I^{(2)}\right| \leq \frac{C}{\lambda^{1+2(1-s_c)}} \sum_{k=1}^p \int |\varepsilon^{'(2)}|(\chi_{\frac{1}{\lambda}}Q)^{(p-k)} |\varepsilon^{'(1)}|^{k}.
\ee
We first take $k=1$. Because of the asymptotic $Q^{p-1}\sim \frac{c}{y^2}$ we use Hardy inequality and interpolation:
$$
\begin{array}{r c l}
\int |\varepsilon^{'(2)}|(\chi_{\frac{1}{\lambda}}Q)^{(p-1)} |\varepsilon^{'(1)}|&\leq & C\parallel \varepsilon^{'(2)}  \parallel_{L^2} \parallel \nabla^{2}\varepsilon^{'(1)} \parallel_{L^2}\\
&\leq& C \parallel \varepsilon^{'(2)}  \parallel_{L^2} \parallel \varepsilon^{'(1)} \parallel_{\dot{H}^1}^{\frac{\sigma-2}{\sigma-1}} \parallel \varepsilon^{'(1)} \parallel_{\dot{H}^{\sigma}}^{\frac{1}{\sigma-1}}.
\end{array}
$$
As $\frac{\sigma-2}{\sigma-1}(1-s_c)+\frac{\sigma-s_c}{\sigma-1}=2-s_c$ this gives the the estimate when applying the scale change:
\be \label{end:eq:norme L2 I k=1}
\frac{1}{\lambda^{1+2(1-s_c)}} \int |\varepsilon^{'(2)}|(\chi_{\frac{1}{\lambda}}Q)^{(p-1)} |\varepsilon^{'(1)}| \leq C \parallel w^{'(2)}  \parallel_{L^2} \parallel w^{'(1)} \parallel_{\dot{H}^1}^{\frac{\sigma-2}{\sigma-1}} \parallel w^{'(1)} \parallel_{\dot{H}^{\sigma}}^{\frac{1}{\sigma-1}}.
\ee
Now let $k$ be an integer, $2\leq k \leq p$. We have the asymptotic: $Q^{p-k}\sim \frac{c}{y^{2\frac{p-k}{p-1}}}$. We put this weighted decay on $\varepsilon^{'(2)}$, use Hardy inequality and interpolation:
\be \label{end:eq:norme L2 I hardy}
\parallel (\chi_{\frac{1}{\lambda}}Q)^{p-k}\varepsilon^{'(2)}  \parallel_{L^2}\leq C\parallel \nabla^{\frac{2(p-k)}{p-1}}\varepsilon^{'(2)}   \parallel_{L^2} \leq  C\parallel \varepsilon^{'(2)} \parallel_{L^2}^{1-\theta} \parallel \nabla^{\sigma-1} \varepsilon^{'(2)} \parallel_{L^2}^{\theta}
\ee
for $\theta=\frac{2(p-k)}{(p-1)(\sigma-1)}$. From Sobolev injection $|\varepsilon^{'(1)}|^k\in L^q$ for $q\in [\frac{2d}{k(d-2)},\frac{2d}{k(d-2\sigma)}]$. Because we work in a high dimension $d\geq 11$ and $p$ is an integer $\geq 2$ one has:
$$
\frac{2d}{k(d-2)}\leq 2 \leq \frac{2d}{k(d-2\sigma)}=\frac{(p-1)d}{2k}+O(\sigma-s_c).
$$
This implies that $\varepsilon^{'(1)k}\in L^2 $ with the estimate:
\be \label{end:eq:norme L2 I sobolev}
\begin{array}{r c l}
\parallel \varepsilon^{'(1)k} \parallel_{L^2} = \parallel \varepsilon^{'(1)}  \parallel_{L^{2k}}^k \leq C \parallel \varepsilon^{'(1)}\parallel_{\dot{H}^1}^{k(1-\theta')} \parallel \varepsilon^{'(1)}\parallel_{\dot{H}^{\sigma}}^{k\theta'},
\end{array}
\ee
for $\frac{(1-\theta')(d-2)}{2d}+\frac{\theta'(d-2\sigma)}{2d}=\frac{1}{2k}$. The estimates \fref{end:eq:norme L2 I hardy} and \fref{end:eq:norme L2 I sobolev} allow us to apply Cauchy Schwarz and find:
$$ 
\int |\varepsilon^{'(2)}|(\chi_{\frac{1}{\lambda}}Q)^{(p-k)} |\varepsilon^{'(1)}|^k \leq C \parallel \varepsilon^{'(2)}  \parallel_{L^2}^{1-\theta} \parallel \varepsilon^{'(2)}  \parallel_{\dot{H}^{\sigma-1}}^{\theta}  \parallel \varepsilon^{'(1)}\parallel_{\dot{H}^1}^{k(1-\theta')} \parallel \varepsilon^{'(1)}\parallel_{\dot{H}^{\sigma}}^{k\theta'}.
$$
We now compute: 
$$(1-\theta)(1-s_c)+\theta(\sigma-s_c)+k(1-\theta')(1-s_c)+k\theta'(\sigma-s_c=1+2(1-s_c)).
$$
 Hence when applying the scale change the last estimate gives:
\be \label{end:eq:norme L2 I k geq 2}
\frac{1}{\lambda^{1+2(1-s_c)}} \int |\varepsilon^{'(2)}|(\chi_{\frac{1}{\lambda}}Q)^{(p-k)} |\varepsilon^{'(1)}|^k \leq C \parallel \bos{w}  \parallel_{\dot{H}^1\times L^2}^{1-\theta+k(1-\theta')} \parallel \bos{w} \parallel_{\dot{H}^{\sigma}\times\dot{H}^{\sigma-1}}^{\theta+k\theta'} 
\ee
we compute the power involved for the $\parallel \bos{w}  \parallel_{\dot{H}^1\times L^2} $ term:
$$
1-\theta+k(1-\theta')=2-\frac{1-(k-1)(\sigma-s_c)}{\sigma-1}=2-c_k.
$$
We now go back to the expression \fref{end:eq:norme L2 I intermediaire}. We have computed the right hand side for the linear case in \fref{end:eq:norme L2 I k=1}, and in the non linear case in \fref{end:eq:norme L2 I k geq 2}. We have computed the coefficient condition for the non linear case in the last equation (it is straightforward in the linear case). Therefore we have the following estimate for the interaction term:
\be \label{end:eq:contribution I}
\left| \frac{1}{\lambda} \int w^{'(2)}I^{(2)}   \right| \leq C \sum_{k=1}^{p} \parallel \bos{w} \parallel_{\dot{H}^1\times L^2}^{2-c_k} \parallel \bos{w} \parallel_{\dot{H}^{\sigma}\times \dot{H}^{\sigma-1}}^{c_k'} .
\ee
We now come back to the identity \fref{end:eq:expression derivee temporelle w'L2}. We estimated the right hand side in \fref{end:eq:contribution L}, \fref{end:eq:contribution F} and \fref{end:eq:contribution I}, proving the bound \fref{end:eq:borne evolution temporelle w'L2} we claimed. We now integrate this equation in time. We recall that $\bos{w}'=\bos{w}+\tilde{\bos{\alpha}}_{b,\frac{1}{\lambda}}+(1-\chi_{\frac{1}{\lambda}}\bos{Q})_{\frac{1}{\lambda}} $. We take $s$ slightly supercritical: $s_c<s\leq \sigma$. The profile $\tilde{\bos{\alpha}}_{b,\frac{1}{\lambda}}$ has finite supercritical norm:
\be \label{end:eq:finitude normes surcritiques alphab}
\parallel \tilde{\bos{\alpha}}_{b,\frac{1}{\lambda}} \parallel_{\dot{H}^{s}\times \dot{H}^{s-1}}\underset{t\rightarrow T}{\rightarrow} 0.
\ee
The tail of the soliton has also a bounded size:
\be \label{end:eq:finitude normes surcritiques queue soliton}
\parallel ((1-\chi_{\frac{1}{\lambda})\bos{Q}})_{\frac{1}{\lambda}} \parallel_{\dot{H}^{s}\times \dot{H}^{s-1}} \leq C.
\ee
From the bound \fref{eq:bootstrap estimations sur epsilon}, the same property holds for $\bos{w}$ for $s=\sigma$: $\parallel \bos{w} \parallel_{\dot{H}^{\sigma}\times \dot{H}^{\sigma-1}}\leq C$.  Consequently, we have the boundedness of the $\sigma$ Sobolev norm for $\bos{w}'$: 
$$
\parallel \bos{w}' \parallel_{\dot{H}^{\sigma}\times \dot{H}^{\sigma-1}}\leq C.
$$ 
Coming back to the identity \fref{end:eq:borne evolution temporelle w'L2} it gives:
$$
\frac{d}{dt}\parallel \bos{w}' \parallel_{\dot{H}^{1}\times L^2}^2\leq C\parallel \bos{w}' \parallel_{\dot{H}^{1}\times L^2}+C \sum_{k=1}^{p} \parallel \bos{w} \parallel_{\dot{H}^1\times L^2}^{2-c_k}.
$$
The growth of this quantity is sub linear: it stays bounded until time $T$.
\end{proof}

We now know from the previous Lemma \ref{end:lem:borne L2} that our solution stays bounded in $L^2$ until blow-up time. Using \fref{end:eq:contribution finie du soliton} we have that:
$$
\parallel \bos{w}' \parallel_{\dot{H}^1\times L^2}\leq C.
$$
This implies for the renormalized error:
$$
\parallel \bos{\varepsilon}' \parallel_{\dot{H}^1\times L^2}\leq \lambda^{1-s_c} C.
$$
On the other hand, the bootstrap bound \fref{eq:bootstrap estimations sur epsilon} gives:
$$
\parallel \bos{\varepsilon}' \parallel_{\dot{H}^{\sigma}\times \dot{H}^{\sigma-1}}\leq \lambda^{\sigma-s_c} C.
$$
By interpolation, we get that for any $1\leq s \leq \sigma$:
\be \label{end:eq:interpolation sobolev}
\parallel \bos{\varepsilon}' \parallel_{\dot{H}^{s}\times \dot{H}^{s-1}}\leq \lambda^{s-s_c} C.
\ee
We now come back to the decomposition: $\bos{\varepsilon}'=\bos{\varepsilon}+\tilde{\bos{\alpha}}_{b}+(1-\chi_{\frac{1}{\lambda}}\bos{Q})$. From \fref{end:eq:finitude normes surcritiques alphab} and \fref{end:eq:finitude normes surcritiques queue soliton} the perturbation $\tilde{\bos{\alpha}}_{b}$ and the tail of the solitary waves enjoy the bound:
$$
\parallel \tilde{\bos{\alpha}}_b+(1-\chi_{\frac{1}{\lambda}})\bos{Q} \parallel_{\dot{H}^{s}\times \dot{H}^{s-1}}\leq \lambda^{s-s_c} C.
$$
Combined with the previous bound \fref{end:eq:interpolation sobolev}, it gives for the original error term:
$$
\parallel \bos{\varepsilon} \parallel_{\dot{H}^{s}\times \dot{H}^{s-1}}\leq \lambda^{s-s_c} C\rightarrow 0 \ \text{as} \ t\rightarrow T.
$$
This proves the convergence to $0$ of the renormalized perturbation in slightly supercritical norms:
\be \label{end:eq:convergence erreur vers 0 leq sigma}
\parallel \tilde{\bos{\varepsilon}} \parallel_{\dot{H}^{s}\times \dot{H}^{s-1}}\rightarrow 0 \ \ \text{as}  \ \ t\rightarrow T, \ \ \text{for} \  \ s_c<s\leq \sigma.
\ee
We now put \fref{end:eq:convergence erreur vers 0 geq sigma} and \fref{end:eq:convergence erreur vers 0 leq sigma} together: for any $s_c<s\leq s_L$,
\be \label{end:eq:convergence erreur vers 0}
\parallel \tilde{\bos{\varepsilon}} \parallel_{\dot{H}^{s}\times \dot{H}^{s-1}}\rightarrow 0 \ \ \text{as} \ \ t\rightarrow T.
\ee
Now we turn to subcritical Sobolev norms. Let $s$ be such that $1\leq s <s_c$. From \fref{end:eq:interpolation sobolev}, the perturbation has finite subcritical norms:
$$
\parallel \bos{w}' \parallel_{\dot{H}^{s}\times \dot{H}^{s-1}}\leq  C.
$$
As the localized soliton also has finite subcritical norms:
$$
\parallel (\chi_{\frac{1}{\lambda}}\bos{Q})_{\frac{1}{\lambda}} \parallel_{\dot{H}^{s}\times \dot{H}^{s-1}}\leq  C,
$$
this means that the full solution stays bounded in subcritical norms:
\be
\parallel \bos{u} \parallel_{\dot{H}^{s}\times \dot{H}^{s-1}}\leq  C(\bos{u}(0)).
\ee
We now turn the the critical norm. From \fref{end:eq:interpolation sobolev}, the perturbation has finite critical and slightly supercritical norms:
$$
\parallel \bos{w}' \parallel_{\dot{H}^{s}\times \dot{H}^{s-1}}\leq  C(\bos{u}(0)) \ \text{for} \ s_c\leq s \leq \sigma
$$
As the soliton is located on the first coordinate, this implies the boundedness of the time derivative in the critical and slightly critical spaces:
\be
\parallel \partial_t u^{(1)} \parallel_{\dot{H}^{s-1}}=\parallel u^{(2)} \parallel_{\dot{H}^{s-1}}\leq  C(\bos{u}(0)) \ \text{for} \ s_c\leq s \leq \sigma
\ee
The critical norm for the first coordinate comes then from the soliton cut in a fixed zone:
\be
\parallel  u^{(1)} \parallel_{\dot{H}^{s_c}}\sim \parallel  \chi Q_{\frac{1}{\lambda}} \parallel_{\dot{H}^{s_c}} = C(d,p)\sqrt{\ell} \sqrt{|\text{log}(T-t)|} (1+o(1)) \ \ \text{as} \ t\rightarrow T.
\ee


\section{Lipschitz aspect and codimension of the set of solutions described by Proposition \ref{prop:bootstrap}}

In Proposition \ref{pr:profilapprochecoupe}, we have constructed an approximate blow up profile described by a set of $L+1$ parameters $\lambda,b_1,...,b_L$: $\tilde{\bos{Q}}_{b,\frac{1}{\lambda}}$. We studied the approximate dynamics of (NLW) for such profiles, and found in Lemma \ref{lem:soldusystemd} that for each integer $\ell>\alpha$, the time dependent profile $\tilde{\bos{Q}}_{b^e,\frac{1}{\lambda^e}}$ was a good approximate blow up solution. In Proposition \ref{prop:bootstrap}, we showed the existence of a real solution of (NLW), under the form $\tilde{\bos{Q}}_{b^e+(\frac{U_1}{s},...,\frac{U_L}{s^L}),\frac{1}{\lambda}}+\bos{w} $,
that stayed close to this approximate blow up solution.\\
\\
To prove it, we studied the parameters $V_1,...,V_{\ell},U_{\ell+1},..,U_L$ (obtained from the $U_i$'s by a linear change of variables). We showed that at leading order, $V_1,U_{\ell+1},...,U_L$ were evolving according to a stable linear dynamics, whereas $V_2,...,V_{\ell}$ were evolving via a unstable linear one. The error $\bos{w}$ was showed to be a stable perturbation. For each initial values of the stable parameters $V_1(s_0),U_{\ell+1}(s_0),...,U_L(s_0)$ and error $\bos{w}(s_0)$, we proved in Lemma \ref{endoftheproof:lem:proprietes de f} that we could apply Brouwer's continuity argument to find the existence of at least one initial perturbation $V_2(s_0),...,V_{\ell}(s_0)$ such that the orbit $V_2,...,V_{\ell}$ stayed small., giving the existence of the real blow up solution.\\
\\
Now one could wonder: is the choice $V_2(s_0),...V_{\ell}(s_0)$ unique? If yes, how does it depend on the initial perturbation along the stable directions $V_1(s_0),U_{\ell+1}(s_0)$,..., $U_L(s_0)$ and $\bos{w}(s_0)$? We show in this section the uniqueness and the Lipschitz dependence. It will imply that the set of type II blow up solutions described by Proposition \ref{prop:bootstrap} is a Lipschitz manifold of codimension $\ell-1$.

\begin{theorem} \label{variete:thm;variete}
We keep the assumptions and notations of Proposition \ref{prop:bootstrap}, and recall that $\sigma$ and $s_L$ are defined in \fref{thetrapped:eq:def sigma} and \fref{intro:eq:def sL}. There exists a choice of constants implied in this proposition such that its result still holds, and that moreover the set of initial data leading to such solutions is a locally Lipschitz manifold of codimension $\ell-1$ in the space $\left(\dot{H}^{\sigma}\cap \dot{H}^{s_L} \right)\times \left(\dot{H}^{\sigma-1}\cap \dot{H}^{s_L-1} \right)$.
\end{theorem}

Roughly speaking, the proof of Theorem \ref{variete:thm;variete} is the adaptation of everything we did in the proof of Proposition \ref{prop:bootstrap}, this time to study the difference of two solutions and to see what informations we can get. For this reason, some technical points in the proofs to come will be treated in a faster way whenever we already treated them in Section $3$.\\
\\
Our strategy of the proof is the following:
\begin{itemize}
\item[(i)] \emph{Lipschitz aspect of the unstable modes under extra assumptions}. We first prove that for initial data starting at the same scale and having extra regularity assumptions, the coefficients of the unstable modes $V_2(s_0),...,V_{\ell}(s_0)$ have Lipschitz dependence on the stable modes $V_1(s_0)$, $b_{\ell+1}(s_0),...,b_{L}(s_0)$ and $\bos{w}(s_0)$.
\item[(ii)] \emph{removal of the extra assumptions}. We then show how to remove the extra assumptions we needed in the first step: it just consists in performing a lower order decomposition at initial time. Instead of studying the decomposition $\bos{U}=(\tilde{\bos{Q}}_{b}+\bos{\varepsilon})_{\frac{1}{\lambda}}$ for $b$ a $L$-tuple $b=(b_1,...,b_L)$, we study the decomposition $\bos{U}=(\tilde{\bos{Q}}_{\bar{b}}+\overline{\bos{\varepsilon}} )_{\frac{1}{\bar{\lambda}}}$ for $\bar{b}$ a $L-1$-tuple. We apply the result of the first step to this new decomposition. As $b_L$ is small because it represents a small perturbation along the last stable mode, it imply the result for the original decomposition.
\end{itemize}


\subsection{Lipschitz dependence of the unstable modes under extra assumptions}
We now perform the first part of the analysis. Let $\bos{U}$ and $\bos{U}'$ be two solutions described by Proposition \fref{prop:bootstrap}. For $\bos{U}$ we keep the notations introduced in the analysis throughout the previous section. For $\bos{U}'$ we adapt them:
$$
\boldsymbol{U'}:=(\boldsymbol{Q}_{b'}+\boldsymbol{\varepsilon'})_{\frac{1}{\lambda'}}=\boldsymbol{Q}_{b',\frac{1}{\lambda'}}+\boldsymbol{w}',
$$
with $\boldsymbol{\varepsilon}'$ satisfying the orthogonality conditions \fref{eq:condition d'ortho pour epsilon}. Its renormalized time is $s'$ (defined by \fref{thetrapped:eq:def s}), its associated scale $\lambda'$, and associated parameters $U'$ and $V'$. We use the same notation for the norms of the error we already used and introduce a higher derivative adapted norm:
$$
\mathcal{E}_{\sigma}':=\int |\nabla^{\sigma}\varepsilon^{'(1)}|^2 + |\nabla^{\sigma-1}\varepsilon^{'(2)}|^2, \ \mathcal{E}_{s_L}:=\int |\varepsilon^{'(1)}_{s_L}|^2 + |\varepsilon^{'(2)}_{s_L-1}|^2,
$$
$$
\mathcal{E}_{s_L+1}':=\int |\varepsilon^{'(1)}_{s_L+1}|^2 + |\varepsilon^{'(2)}_{s_L}|^2.
$$
We introduce the following notations for the differences:
\be
\triangle U_i':=U_i-U_i', \ \triangle V_i:= V_i-V'_i, \ \triangle V_{uns}:=(\triangle V_{2},...,\triangle V_{l}),
\ee
\be
\triangle \mathcal{E}_{s_L}:= \int |(\varepsilon^{(1)} - \varepsilon '^{(1)})_{s_{L}}|^2 + |(\varepsilon^{(2)} - \varepsilon '^{(2)})_{s_{L}-1}|^2,
\ee
\be \label{variete:eq:def triangle mathcalEsigma}
\triangle \mathcal{E}_{\sigma}:= \int |\nabla^{\sigma}(\varepsilon^{(1)} - \varepsilon '^{(1)})|^2 + |\nabla^{\sigma-1}(\varepsilon^{(2)} - \varepsilon '^{(2)})|^2.
\ee
In the analysis, it will be easier to use the following renormalized differences:
\be
\triangle_r \mathcal{E}_{s_L}:= b_1^{-2L-2 (1-\delta_0)(1+\frac{\eta}{2})} \triangle \mathcal{E}_{s_L}, \ \triangle_r \mathcal{E}_{\sigma}:= b_1^{-2(\sigma-s_c)(1+\nu)} \triangle \mathcal{E}_{\sigma}.
\ee
The presence of $\frac{\eta}{2}$ instead of the usual $\eta$ is just technical. Here is the main proposition of this subsection, the Lipschitz dependence of the unstable coefficients under some extra assumptions: the two solutions start at the same scale and have some additional regularity.

\begin{proposition}
\label{variete:prop:parametres lipschitz}
Suppose that $\boldsymbol{U}_0=\boldsymbol{Q}_{b_0}+\boldsymbol{\varepsilon}_0$ and $\boldsymbol{U'}_0=\boldsymbol{Q}_{b'_0}+\boldsymbol{\varepsilon'}_0$ are two initial data of solutions described by Proposition \fref{prop:bootstrap}, starting at the same scale. Suppose that they are close initially:
\be
b_0=b^e(s_0)+\left(\frac{U_1(s_0)}{s_0},...,\frac{U_L(s_0)}{s_0^L}\right), b'_0=b^e(s_0)+\left(\frac{U'_1(s_0)}{s_0},...,\frac{U'_L(s_0)}{s_0^L}\right),
\ee
which means $s_0=s_0'$. Suppose moreover that we have the following additional regularity for $\boldsymbol{\varepsilon'}$:
\be \label{variete:eq:estimation supplementaire varepsilon'}
\mathcal{E}_{s_L+1}'(s')\leq K_3 (b_1')^{(2L+2+2(1-\delta_0)(1+\eta)}, \ for \ all \ s_0\leq s',
\ee
for some constant $K_3=K_3(K_1,K_2)$. Then there exist a constant $C>0$ such that for $s_0$ small enough the following bound at initial time holds:
\be \label{variete:eq:lipschitz}
|\triangle V_{\text{uns}}(s_0)|\leq C\left( |\triangle V_1(s_0)|+\sum_{\ell+1}^L|\triangle U_i(s_0)|+\sqrt{\triangle_r \mathcal{E}_{\sigma}(s_0)}+\sqrt{\triangle_r \mathcal{E}_{s_L}(s_0)} \right).
\ee
\end{proposition}

The next subsubsections are devoted to the proof of this Proposition. We first introduce an adapted time for comparison $\hat{s}'$, and associate to $\boldsymbol{U'}$ the adapted variables for the analysis $\hat{\boldsymbol{\varepsilon '}}$, $\hat{U}'$ and $\hat{V}'$. We then write the time evolution equation for the differences of the parameters and error, yielding a system of coupled equations. We study this system and we show that if the initial size of the difference of the unstable parameters is too big compared to the initial size of the differences of the stable parameters and error, then repellency wins and it cannot converge to zero, preventing one of the two solutions to stay forever in the trapped regime.


\subsubsection{Adapted time for comparison, notations and identities}

The two solutions $\boldsymbol{U}$ and $\boldsymbol{U'}$ might blow up at different times. In addition, we have seen that the values of $\lambda$, $s$ and the parameters $b$ are correlated, see the equivalences in the trapped regime \fref{linearized:eq:def ci}, \fref{eq:bootstrap modes instables}, \fref{eq:bootstrap modes stables} and \fref{end:eq:asymptotique lambda}. Thus, we do not compare $\boldsymbol{U'}$ to $\boldsymbol{U}$ at the same time $t$, but at the times for which their scale coincide: $\lambda=\lambda'$.

\begin{definition}[adapted time and variables for comparison] \label{variete:def:changement de variable}
We define
\be \label{variete:eq:def hats'}
\hat{s}'(s)=(\lambda')^{-1}(\lambda(s)),
\ee
as the adapted time for comparison, where $\lambda':[s_0;+\infty[\rightarrow \ ]0;1]$ is seen as $C^1$ diffeomorphism (remember that $\lambda'_{s'}\sim -\lambda' \frac{c_1}{s'}<0$ from \fref{thetrapped:eq:modulation leq L-1}). We associate to $\boldsymbol{U'}$ the variables $\hat{\boldsymbol{\varepsilon}}'$, $\hat{b}'$, $\hat{U}'$, $\hat{V}'$ defined by ($P_{\ell}$ being defined in \fref{linearized:eq:def P}):
\be
\hat{\bos{w}}'(t)=\bos{w}'(t(\hat{s}'(s))), \ \ \hat{\boldsymbol{\varepsilon}}'(s)=\boldsymbol{\varepsilon}'(\hat{s}'(s)), \ \ \hat{b}'(s)=b'(\hat{s}'(s)),
\ee
\be
\hat{U}'_i(s)=\left( \frac{s}{\hat{s}'(s)} \right)^i   U'_i(\hat{s}'(s)), \ \text{for} \ 1 \leq i\leq L, \ \text{and} \ \hat{V}'=P_{\ell}(\hat{U}').
\ee
\end{definition}

We use the following notations for the norms of $\hat{\boldsymbol{\varepsilon}}'$:
\be
\hat{\mathcal{E}}_{\sigma}':=\int |\nabla^{\sigma}\hat{\varepsilon}^{'(1)}|^2 + |\nabla^{\sigma-1}\hat{\varepsilon}^{'(2)}|^2, \ \hat{\mathcal{E}}_{s_L+i}:=\int |\hat{\varepsilon}^{'(1)}_{s_L+i}|^2 + |\hat{\varepsilon}^{'(2)}_{s_L-1+i}|^2, \ i=0,1 .
\ee
We now prove that the times $s$ and $\hat{s}'$ are close. All the analysis bounds of the trapped regime for $\boldsymbol{U}'$, expressed in function of $\hat{b}'_1$, then still hold interchanging $\hat{b}_1'$ with $b_1$.

\begin{lemma}[Bounds on the change of variables]
\label{variete:lem:borne varepsilon' vers hatvarepsilon'}

The following bound holds:
\be \label{variete:eq: equivalence s et hats'}
\hat{s}'=s(1+O(s_0^{-\tilde{\eta}} )).
\ee
The bounds of the trapped regime \fref{eq:bootstrap estimations sur epsilon} and the bound \fref{variete:eq:estimation supplementaire varepsilon'} can be written as:
\be  \label{variete:eq:equivalence normes varepsilon'}
\hat{\mathcal{E}}'_{s_L+i}\leq 2K_2 b_1^{2L+2i+2(1-\delta_0)(1+\eta)}, \ i=0,1\ \text{and} \ \hat{\mathcal{E}}_{\sigma}'\leq 2 K_1 b_1^{2(\sigma-s_L)(1+\nu)}.
\ee
The parameters also enjoy the same estimates:
\be \label{variete:eq:equivalence parametres}
|\hat{V}_1'|\ \leq \frac{C}{s^{\tilde{\eta}}}, \ |\hat{V}_i'|\ \leq \frac{C}{s^{\tilde{\eta}}} \ \text{for} \ 2\leq i \leq \ell, \ |\hat{b}_i'|\ \leq \frac{C}{s^{i+\tilde{\eta}}} \ \text{for} \ \ell+1 \leq i \leq L,
\ee
the constant $C$ being independent of the other parameters. Moreover, $\hat{\bos{\varepsilon}}'$ still enjoys the orthogonality conditions:
\be \label{variete:eq:conditions d'orthogonalite pour hatvarepsilon'}
\langle \hat{\bos{\varepsilon}}', \bos{H}^*\bos{\Phi}_M\rangle=0, \ \text{for} \ 1\leq i \leq L.
\ee
\end{lemma}

\begin{proof}[Proof of Lemma \ref{variete:lem:borne varepsilon' vers hatvarepsilon'}]. The orthogonality conditions are a straightforward consequence of those for $\bos{\varepsilon}'$, see \fref{eq:condition d'ortho pour epsilon}. We use the formula \fref{end:eq:asymptotique lambda} relating $\lambda$ and $s$:
$$
\lambda(t)=\left(\frac{s_0}{s(t)}\right)^{\frac{l}{l-\alpha}}(1+O( s_0^{-\tilde{\eta}})).
$$
This implies:
$$
s(t)= \frac{s_0}{\lambda(t)^{\frac{l-\alpha}{l}}}(1+O(s_0^{-\tilde{\eta}})), \  \ \text{and} \ \ \hat{s}'(s)= \frac{s_0}{\lambda(t)^{\frac{l-\alpha}{l}}}(1+O(s_0^{-\tilde{\eta}})).
$$
From that we get the first bound of the lemma: $\frac{s(t)}{s'(t')}=1+O(s_0^{-\tilde{\eta}})$. Now we recall that in the trapped regime: $b_1(s)=\frac{c_1}{s}+\frac{U_1}{s}=\frac{c_1}{s}+O(s^{-1-\tilde{\eta}})$ and $\hat{b}'_1(s)=\frac{c_1}{\hat{s}'}+\frac{U_1'(\hat{s}')}{\hat{s}'}=\frac{c_1}{\hat{s}'}+O((\hat{s}')^{-1-\tilde{\eta}})$. Hence, \fref{variete:eq: equivalence s et hats'} implies $\frac{b_1(s)}{\hat{b}'_1(s)}=1+O(s_0^{-\tilde{\eta}}) $. The bounds  \fref{variete:eq:equivalence normes varepsilon'} and  \fref{variete:eq:equivalence parametres} are just a rewriting of the bootstrap bounds \fref{eq:bootstrap estimations sur epsilon} and \fref{variete:eq:estimation supplementaire varepsilon'} knowing this equivalence.

\end{proof}

We use the following notation for the differences (all terms taken at time $s$):
\be \label{variete:eq:def Vuns}
\triangle \hat{b}_i:= b_i-\hat{b}_i', \ \triangle \hat{U}_i:=U_i-\hat{U}'_i, \ \triangle \hat{V}_i:= V_i-\hat{V}'_i, \ \triangle \hat{V}_{\text{uns}}:=(\triangle \hat{V}_2,...,\triangle \hat{V}_{\ell})
\ee
\be
\triangle \hat{\mathcal{E}}_{s_L}:= \int |(\varepsilon^{(1)} - \hat{\varepsilon} '^{(1)})_{s_{L}}|^2 + |(\varepsilon^{(2)} - \hat{\varepsilon} '^{(2)})_{s_{L}-1}|^2,
\ee
\be  \label{variete:eq:def norme low}
\triangle \hat{\mathcal{E}}_{\sigma}:= \int |\nabla^{\sigma}(\varepsilon^{(1)} - \hat{\varepsilon} '^{(1)})|^2 + |\nabla^{\sigma-1}(\varepsilon^{(2)} - \hat{\varepsilon} '^{(2)})|^2,
\ee
In the analysis, it will be easier to use the following renormalized differences:
\be \label{variete:eq:def normes renormalisees}
\triangle_r \hat{\mathcal{E}}_{s_L}:= b_1^{-2L-2 (1-\delta_0)(1+\frac{\eta}{2})} \triangle \hat{\mathcal{E}}_{s_L}, \ \ \triangle_r \hat{\mathcal{E}}_{\sigma}:= b_1^{-2(\sigma-s_c)(1+\nu)} \triangle \hat{\mathcal{E}}_{\sigma},
\ee
The coefficient $\frac{\eta}{2}$ instead of the $\eta$ we had previously is because we will loose a bit in the analysis later on. 
We adapt the notation for the terms involved the analysis\footnote{We do not mention the dependance of $\bos{L}$ and $\bos{NL}$ in $\bos{\varepsilon}$ and $\bos{w}$ anymore to ease notations, as it will be clear to which variable we are refering to in future computations.}:
\be
\boldsymbol{\psi}_{\hat{b}'}(s):=\tilde{\boldsymbol{\psi}}_{b'}(\hat{s}'), \ \hat{\boldsymbol{L}}'(s):= \boldsymbol{L'}(\boldsymbol{\varepsilon}')(\hat{s}'(s)), \ \hat{\boldsymbol{NL}}'(s):=\boldsymbol{NL}'(\bos{\varepsilon}'(\hat{s}'(s))).
\ee
\be
\hat{\boldsymbol{Mod}}'(s):= \frac{d\hat{s}'}{ds} \tilde{\boldsymbol{Mod}}'(\hat{s}'(s)), \ \hat{B}_1':=(\hat{b}_1')^{-(1+\eta)}, \ \text{and} \ \hat{\bos{S}}'_i:=\bos{S}_i(\hat{b}').
\ee
The change of variables of Definition \fref{variete:def:changement de variable} produces the following identities:
\be \label{variete:eq:identity hatb'}
\hat{b}'_i(s)=\frac{c_i}{\hat{s}^{'i}}+\frac{\hat{U}'_i}{s^i},
\ee
\be
\begin{array}{r c l}
\hat{\boldsymbol{Mod}}'(s)&=& \chi_{\hat{B}_1'}  \sum_{i=1}^L (\hat{b}'_{i,s}+\frac{d\hat{s}'}{ds}((i-\alpha)\hat{b}'_1\hat{b}'_i-\hat{b}'_{i+1}))\left( \boldsymbol{T}_i +\sum_{j=i+1}^{L+2} \frac{\partial \boldsymbol{S}_j}{\partial b_i}(\hat{b}')  \right) \\
&&-(\frac{\lambda_s}{\lambda}+\frac{d\hat{s}'}{ds}\hat{b}'_1)\boldsymbol{\Lambda} \tilde{\boldsymbol{Q}}_{\hat{b}'}.
\end{array}
\ee
We introduce the following notation for $1\leq i \leq L$:
\be \label{variete:eq:def triangle hatmodi}
\begin{array}{r c l}
\triangle \hat{\bos{Mod}}_i&:=&(b_{i,s}+(i-\alpha)b_1b_i-b_{i+1})\chi_{B_1}(\bos{T}_i+\underset{j=i+1}{\overset{L+2}{\sum}}\frac{\partial \bos{S}_{j}}{\partial b_i})\\
&&-(\hat{b}'_{i,s}+\frac{d\hat{s}'}{ds}((i-\alpha)\hat{b}'_1\hat{b}'_i-\hat{b}'_{i+1}))\chi_{\hat{B}'_1}(\bos{T}_i+\underset{j=i+1}{\overset{L+2}{\sum}}\frac{\partial \hat{\bos{S}}'_{j}}{\partial b_i}),
\end{array}
\ee
and define:
\be \label{variete:eq:def triangle hatmod0}
\triangle \hat{\bos{Mod}}_0 := -(\frac{\lambda_s}{\lambda}+b_1)\boldsymbol{\Lambda} \tilde{\boldsymbol{Q}}_{b}+(\frac{\lambda_s}{\lambda}+\frac{d\hat{s}'}{ds}\hat{b}'_1)\boldsymbol{\Lambda} \tilde{\boldsymbol{Q}}_{\hat{b}'}
\ee
So that $\tilde{\bos{Mod}}-\hat{\bos{Mod}}'=\underset{i=0}{\overset{L}{\sum}}\triangle \hat{\bos{Mod}}_i $. With these new notations the time evolution of the difference of errors in renormalized variables is given by:
\be \label{variete:eq:evolution difference epsilon}
\begin{array}{r c l}
&\frac{d}{ds}(\boldsymbol{\varepsilon}-\hat{\boldsymbol{\varepsilon}}' )-\frac{\lambda_s}{\lambda}\bos{\Lambda}(\boldsymbol{\varepsilon}-\hat{\boldsymbol{\varepsilon}}' )+\bos{H}(\boldsymbol{\varepsilon}-\hat{\boldsymbol{\varepsilon}}' )+(1-\frac{d\hat{s}'}{ds})\bos{H}(\hat{\boldsymbol{\varepsilon}}') \\
=& -\tilde{\boldsymbol{Mod}}+\hat{\boldsymbol{Mod}}'-\tilde{\boldsymbol{\psi}}_b+\frac{d\hat{s}'}{ds}\hat{\boldsymbol{\psi}}_{\hat{b}'} +\boldsymbol{NL}-\frac{d\hat{s}'}{ds}\hat{\boldsymbol{NL}}'+\boldsymbol{L}-\frac{d\hat{s}'}{ds}\hat{\boldsymbol{L}}'.
\end{array}
\ee
The time evolution of the original variables $\bos{w}-\hat{\bos{w}}'$ is:
\be \label{variete:eq:evolution difference w}
\begin{array}{r c l}
&\frac{d}{dt}(\boldsymbol{w}-\hat{\boldsymbol{w}}' )+\bos{H}_{\frac{1}{\lambda}}(\boldsymbol{w}-\hat{\boldsymbol{w}}' )+(1-\frac{d\hat{s}'}{ds})\bos{H}_{\frac{1}{\lambda}}(\hat{\boldsymbol{w}}')\\
=&  -\frac{1}{\lambda}\tilde{\boldsymbol{Mod}}_{\frac{1}{\lambda}}+\frac{1}{\lambda}\hat{\boldsymbol{Mod}}'_{\frac{1}{\lambda}}-\frac{1}{\lambda}\tilde{\boldsymbol{\psi}}_{b,\frac{1}{\lambda}}+\frac{d\hat{s}'}{ds}\frac{1}{\lambda}\hat{\boldsymbol{\psi}}_{\hat{b}'} +\boldsymbol{NL}-\frac{d\hat{s}'}{ds}\hat{\boldsymbol{NL}}'+\boldsymbol{L}-\frac{d\hat{s}'}{ds}\hat{\boldsymbol{L}}'.
\end{array}
\ee


\subsubsection{Modulation equations for the difference}

In this subsection we compute the time evolution of the difference of parameters between the first solution and the modified second solution defined in Definition \fref{variete:def:changement de variable}. We relate it to the difference $\boldsymbol{\varepsilon}-\hat{\boldsymbol{\varepsilon}}' $ and itself. We start with a technical lemma linking the differences of some profiles to the differences of the parameters.

\begin{lemma}[Asymptotic for some differences of profiles for $y\leq 2B_0$:]
\label{variete:lem:asymptotique differences de profils yleqB0}
The following bounds hold,  $k$ denoting an integer $k\in \mathbb{N}$ .
\begin{itemize}
\item[(i)] \emph{Differences of potentials:} For $1\leq j\leq p-1$:
\be \label{variete:eq:asymptotique diff Qp-j}
|\partial_y^k((\tilde{Q}_{\hat{b}'}^{(1)})^{p-j}-(\tilde{Q}_b^{(1)})^{p-j})|\leq \frac{Cb_1}{1+y^{\frac{2(p-j)}{p-1}-1+\alpha+k-C_k\eta}}\underset{1\leq i \leq L}{\text{sup}}(b_1^{-i}|\triangle \hat{b}_i|).
\ee
\item[(ii)] \emph{Difference of the errors in the central zone:} For $y\leq 2B_0$, one has that $\tilde{\bos{\psi}}_b-\tilde{\bos{\psi}}_{\hat{b}'}=\begin{pmatrix}0\\ \psi_b-\psi_{\hat{b}'} \end{pmatrix}$ is on the second coordinate and there holds: 
\be \label{variete:eq:asymptotique diff psib leqB0}
| \partial_y^k(\psi_b-\psi_{\hat{b}'} )|\leq  \frac{Cb_1^{L+3}}{1+y^{\gamma+g'-L-1+k}}   \underset{1\leq i \leq L}{\text{sup}}(b_1^{-i}|\triangle \hat{b}_i|).
\ee
\end{itemize}
\end{lemma}

\begin{proof}[Proof of Lemma \ref{variete:lem:asymptotique differences de profils yleqB0}]
\underline{Step 1:} \emph{Differences of polynomials of parameters}. We claim that for any $L$-tuple $J$ there holds:
\be \label{variete:eq:difference polynomes parametres}
|b^J-\hat{b}^{'J}|\leq C b_1^{|J|_2}\underset{1\leq i \leq L}{\text{sup}}(b_1^{-i}|\triangle \hat{b}_i|).
\ee
We recall the notations $|J|_1=\underset{i=1}{\overset{L}{\sum}}J_i$ and $|J|_2=\underset{i=1}{\overset{L}{\sum}}iJ_i$. We show this bound by iteration. It is true for the trivial case $|J|_1=0$. Take now $i\geq 1$ and suppose that it is true for all $J'$ satisfying $|J'|_1\leq i-1$. Take $J$ satisfying $|J|_1=i$. Let $j$ be the first coordinate for which $J$ is non null. We have then:
$$
b^J-\hat{b}^{'J}=b_jb^{J'}-\hat{b}'_j\hat{b}^{'J'}=(b_j-\hat{b}'_j)b^{J'}+\hat{b}_j'(b^{J'}-\hat{b}^{'J'})
$$
for some $L$ tuple $J'$ satisfying $|J'|_1=i-1$ and $|J'|_2=|J|_2-j$. The bound \fref{variete:eq:equivalence parametres} imply that the parameters of the two solutions have the same size: $|b_j|,|\hat{b}_j'|\lesssim b_1^j$. For the first term of the previous identity one then has:
$$
|(b_j-\hat{b}'_j)b^{J'}|\leq C|\triangle \hat{b}_j|b_1^{|J'|_2}\leq C b_1^{|J|_2}(b_1^{-j}|\triangle \hat{b}_j|).
$$
For the second term, from the induction hypothesis for $J'$ one has:
$$
|\hat{b}_j'(b^{J'}-\hat{b}^{'J'})|\leq C b_1^jb_1^{|J'|_2}\underset{1\leq i \leq L}{\text{sup}}(b_1^{-i}|\triangle \hat{b}_i|)= C b_1^{|J|_2}\underset{1\leq i \leq L}{\text{sup}}(b_1^{-i}|\triangle \hat{b}_i|).
$$
This implies that the property is true for $i$.\\

\underline{Proof of (i):} The difference of the two potentials is:
\be \label{variete:asymptotique:Qp-j expression}
(\tilde{Q}_{\hat{b}'}^{(1)})^{p-j}-(\tilde{Q}_b^{(1)})^{p-j}=\underset{i=1}{\overset{p-1}{\sum}}C_i Q^{p-j-i}\left(\chi_{\hat{B}'_1}^i((\alpha_{\hat{b}'}^{(1)})^i-(\alpha_b^{(1)})^i)+(\alpha_b^{(1)})^i(\chi_{\hat{B}'_1}^i-\chi_{B_1}^i)\right)
\ee
for some constants $(C_i)_{1\leq i \leq p-j}$. Let $i$ be fixed, with $1\leq i \leq p-1$. We first study the first term in the right hand side of \fref{variete:asymptotique:Qp-j expression}. There holds:
$$
(\alpha_{\hat{b}'}^{(1)})^i-(\alpha_b^{(1)})^i=\underset{|J|_1=i}{\sum}C_J\Bigl{(} \underset{n=2 \ \text{even}}{\overset{L-1}{\prod}}b_n^{J_n}T_n^{J_n} \underset{n=2 \ \text{even}}{\overset{L+1}{\prod}}S_n^{\tilde{J}_n} -\underset{n=2 \ \text{even}}{\overset{L-1}{\prod}}\hat{b}_n^{'J_n}T_n^{J_n} \underset{n=2\ \text{even}}{\overset{L+1}{\prod}}\hat{S}_n^{'\tilde{J}_n} \Bigr{)}
$$
and the profiles $\bos{S}_n$ are homogeneous of degree $(n,n-g',n \ \text{mod}2,n)$ in the sense of Definition \ref{def:homogeneousfunctions}. This means that for $n$ even: 
$$
S_n(b)=\underset{J' \in\mathcal{J}, \ |J'|_2=n }{\sum} b^{J'}f_{J'},
$$
the sum being finite $\# \mathcal{J}<+\infty$, and the profiles satisfying $\partial_y^k f_{J'}=O\left(\frac{1}{1+y^{\gamma-n+g'+k}} \right)$. Therefore one has the identity:
$$
\begin{array}{r c l}
&\underset{n=2, \ \text{even}}{\overset{L-1}{\prod}}b_n^{J_n}T_n^{J_n} \underset{n=2, \ \text{even}}{\overset{L+1}{\prod}}S_n^{\tilde{J}_n} -\underset{n=2, \ \text{even}}{\overset{L-1}{\prod}}\hat{b}_n^{'J_n}T_n^{J_n} \underset{n=2,\ \text{even}}{\overset{L+1}{\prod}}\hat{S}_n^{'\tilde{J}_n}= \underset{G\in \mathcal{G}}{\sum} [b^G -(\hat{b}')^G]g_G,
\end{array}
$$
the sum being finite $\# \mathcal{G}<+\infty$, for some determined profiles $g_G$ having the asymptotic: $\partial_y^k g_G=O\left( \frac{1}{1+y^{i\gamma+g'\sum_{2}^{L+2}\tilde{J}_n-|G|_2+k}}\right)$. Using the bound \fref{variete:eq:difference polynomes parametres} on $b^G-\hat{b}^{'G} $, one has for $y\leq 2\hat{B}_1'$:
$$
\begin{array}{r c l}
|\partial_y^k  [b^G -(\hat{b}')^G]g_G|  &\leq& C(\underset{1\leq k \leq L}{\text{sup}}b_1^{-k}|\triangle \hat{b}_k|) \frac{1}{1+y^{i\gamma+k}}\frac{b_1^{|G|_2}}{1+y^{-|G|_2}}\\
&\leq& C(\underset{1\leq k \leq L}{\text{sup}}b_1^{-k}|\triangle \hat{b}_k|) \frac{b_1}{1+y^{i\gamma-1+k+O(\eta)}}.
\end{array}
$$
With \fref{linearized:eq:asymptotique soliton} one obtains the desired bound (i) for the first term in \fref{variete:asymptotique:Qp-j expression}:
\be \label{variete:asymptotique:Qp-j 1}
\partial_y^k \left( \underset{i=1}{\overset{p-1}{\sum}} C_i Q^{p-j-i}\chi_{\hat{B}'_1}^i(\alpha_{\hat{b}'}^i-\alpha_b^i)\right)=O\left(\frac{b_1\underset{1\leq k \leq L}{\text{sup}}(b_1^{-k}|\triangle \hat{b}_k|) }{1+y^{\frac{2(p-j)}{p-1}-1+\alpha+k+O(\eta)}}    \right).
\ee
We now turn to the second term in \fref{variete:asymptotique:Qp-j expression}. First we factorize:
$$
\chi_{\hat{B}'_1}^i-\chi_{B_1}^i=(\chi_{\hat{B}'_1}-\chi_{B_1})\sum_{n=0}^{i-1}C_n \chi_{\hat{B}_1'}^{n}\chi_{B_1}^{i-1-n},
$$
for some constants $(C_n)_{0\leq n \leq i-1}$ and then we use the integral formulation:
\be \label{variete:eq:formulation integrale diff chi}
\chi_{B_1}(y)-\chi_{\hat{B}_1'}(y)=y(b_1^{1+\eta}-\hat{b}_1^{'(1+\eta)} )\int_0^1 \partial_y \chi (y((1-\theta)\hat{b}_1^{'(1+\eta)}+\theta b_1^{1+\eta})d\theta,
\ee
to find that: $\partial_y^k (\chi_{\hat{B}'_1}^i-\chi_{B_1}^i)=O\left( \frac{1}{1+y^{k}}b_1^{-1}|\triangle \hat{b}_1|  \right) $. We know from the asymptotic of the $T_i$'s and $S_i$'s that for $y\leq 2\text{max}(B_1,\hat{B}_1')$:
$$
\partial_y^k (Q^{p-j-i}\alpha_b^i)=O\left(\frac{b_1}{1+y^{\frac{2(p-j)}{p-1}-1+i\alpha+O(\eta)}} \right).
$$
The two last asymptotics give the desired bound for the second term in \fref{variete:asymptotique:Qp-j expression}:
\be \label{variete:asymptotique:Qp-j 2}
\partial_y^k \left( \underset{i=1}{\overset{p-1}{\sum}} C_i Q^{p-j-i}(\alpha_b^i(\chi_{\hat{B}'_1}^i-\chi_{B_1}^i))\right)=O\left(\frac{b_1b_1^{-1}|\triangle \hat{b}_1|}{1+y^{\frac{2(p-j)}{p-1}-1+\alpha+k+O(\eta)}} \right).
\ee
Injecting \fref{variete:asymptotique:Qp-j 1} and \fref{variete:asymptotique:Qp-j 2} in \fref{variete:asymptotique:Qp-j expression} gives the desired result \fref{variete:eq:asymptotique diff Qp-j}.\\
\\

\underline{Proof of (ii):} As we are in the zone $y\leq 2B_0$, from the localization property of Proposition \ref{pr:profilapprochecoupe} the error is given by:
$$
\psi_b=\underset{J\in \mathcal{J}, \ |J|_2\geq L+3}{\sum}b^Jf_J,
$$
the sum being finite $\#\mathcal{J}<+\infty$ and the profiles satisfying $\partial_y^k f_J=O\left(\frac{1}{1+y^{\gamma+g'+1-|J|_2}} \right)$. The difference of the primary errors then writes: $\psi_b-\psi_{\hat{b}'}=\underset{J\in \mathcal{J}, \ |J|_2\geq L+3}{\sum}(b^J-\hat{b}^{'J})f_J$. Therefore, the bound \fref{variete:eq:asymptotique diff psib leqB0} of the lemma is a consequence of the asymptotic of the $f_J$'s and of the bound \fref{variete:eq:difference polynomes parametres} on $b^J-\hat{b}^{'J}$.
\end{proof}

We can now relate the time evolution of the difference of the parameters to the difference of the errors $\bos{\varepsilon}-\hat{\bos{\varepsilon}}'$ and to itself.

\begin{lemma}[Modulation estimates for the difference] 
\label{variete:lem:modulation diff}

There holds the following identities. The difference of the two times obeys to:
\be \label{variete:modulation:eq:dhats (triangle b)}
\frac{d\hat{s}'}{ds}-1= \frac{\triangle \hat{b}_1}{\hat{b}'_1}+O  \left(b_1^{L+(1-\delta_0)(1+\frac{\eta}{2})}(b_1^{\frac{\eta}{2}(1-\delta_0)}\underset{1\leq i\leq L}{\text{sup}}b_1^{-i} |\triangle \hat{b}_i|+\sqrt{\triangle_r \hat{\mathcal{E}}_{s_L}})\right).
\ee
For the parameters, for $1\leq i \leq L-1$ one has:
\be \label{variete:modulation:eq:leqL-1 (triangle b)}
\begin{array}{r c l}
&\left| b_{i,s}+(i-\alpha)b_1b_i-b_{i+1}-[\hat{b}'_{i,s}+\frac{d\hat{s}'}{ds}((i-\alpha)\hat{b}'_1\hat{b}'_i-\hat{b}'_{i+1})] \right|\\
\leq& C b_1^{L+1+(1-\delta_0)(1+\frac{\eta}{2})}\left(b_1^{\frac{\eta}{2}(1-\delta_0)}\underset{1\leq i\leq L}{\text{sup}}b_1^{-i}|\triangle \hat{b}_i|+\sqrt{\triangle_r \hat{\mathcal{E}}_{s_L}}\right),
\end{array}
\ee
and for the last one we have the primary bound:
\be \label{variete:modulation:eq:L (triangle b)}
\begin{array}{r c l}
&\left| b_{L,s}+(L-\alpha)b_1b_L-[\hat{b}'_{L,s}+\frac{d\hat{s}'}{ds}(L-\alpha)\hat{b}'_1\hat{b}'_L] \right| \\
\leq& C b_1^{L+(1-\delta_0)(1+\frac{\eta}{2})}\left(b_1^{\frac{\eta}{2}(1-\delta_0)}\underset{1\leq i\leq L}{\text{sup}}b_1^{-i}|\triangle \hat{b}_i|+\sqrt{\triangle_r \hat{\mathcal{E}}_{s_L}}\right).
\end{array}
\ee

\end{lemma}

\begin{proof}[Proof of Lemma \ref{variete:lem:modulation diff}]
We take the scalar product of \fref{variete:eq:evolution difference epsilon} with the profile $\bos{H}^{*i}\bos{\Phi_M}$ for $i=0,...,L$. It gives, because of the orthogonality conditions \fref{eq:condition d'ortho pour epsilon} and \fref{variete:eq:conditions d'orthogonalite pour hatvarepsilon'}:
\be \label{variete:modulation:eq:expression generale}
\begin{array}{r c l}
&\langle \tilde{\boldsymbol{Mod}}-\hat{\boldsymbol{Mod}}',\bos{H}^{*i} \bos{\Phi}_M\rangle-\langle \frac{\lambda_s}{\lambda}\bos{\Lambda}(\boldsymbol{\varepsilon}- \hat{\boldsymbol{\varepsilon}}' ),\bos{H}^{*i} \bos{\Phi}_M\rangle+\langle \bos{H}(\boldsymbol{\varepsilon}-\hat{\boldsymbol{\varepsilon}}' ),\bos{H}^{*i} \bos{\Phi}_M\rangle \\
=& \langle \frac{d\hat{s}'}{ds}\hat{\boldsymbol{\psi}}_{\hat{b}'} -\tilde{\boldsymbol{\psi}}_b+(\frac{d\hat{s}'}{ds}-1)\bos{H}(\hat{\boldsymbol{\varepsilon}}') +\boldsymbol{NL}-\frac{d\hat{s}'}{ds}\hat{\boldsymbol{NL}}'+\boldsymbol{L}(\boldsymbol{\varepsilon})-\frac{d\hat{s}'}{ds}\hat{\boldsymbol{L}}(\hat{\boldsymbol{\varepsilon}}'), \bos{H}^{*i} \bos{\Phi}_M\rangle.
\end{array}
\ee
To simplify the analysis we introduce the following intermediate quantity:
$$
\begin{array}{r c l}
\triangle D(t)&=&\left|b_1-\frac{d\hat{s}'}{ds}\hat{b}_1'\right|\\
&&+ \sum_{i=1}^{L}\left|b_{i,s}+(i-\alpha)b_1b_i-b_{i+1}-[\hat{b}'_{i,s}+\frac{d\hat{s}'}{ds}((i-\alpha)\hat{b}'_1\hat{b}'_i-\hat{b}'_{i+1})]\right|.
\end{array}
$$
We notice that as $\bos{\varepsilon}-\hat{\bos{\varepsilon}}'$ still satisfy the orthogonality conditions \fref{eq:condition d'ortho pour epsilon} we can still use the coercivity of $\triangle \hat{\mathcal{E}}_{s_L}$ given by Corollary \fref{annexe:lem:coercivite des normes adaptees}.\\

\underline{Step 1:} Law for $\frac{d\hat{s}'}{ds}$. We take $i=0$ in the previous equation \fref{variete:modulation:eq:expression generale}. The linear terms disappear because of the orthogonality conditions \fref{eq:condition d'ortho pour epsilon} and \fref{variete:eq:conditions d'orthogonalite pour hatvarepsilon'}:
\be \label{variete:modulation:eq:dhats lineaire}
\langle \bos{H}(\boldsymbol{\varepsilon}-\hat{\boldsymbol{\varepsilon}}' ),\bos{\Phi}_M\rangle-\langle(\frac{d\hat{s}'}{ds}-1)\bos{H}(\hat{\boldsymbol{\varepsilon}}') ,\bos{\Phi}_M\rangle=0
\ee
The non linear, small linear and error terms are not on the first coordinate, so:
\be \label{variete:modulation:eq:dhats NL L psib}
\left\langle -\tilde{\boldsymbol{\psi}}_b+\frac{d\hat{s}'}{ds}\hat{\boldsymbol{\psi}}_{\hat{b}'}+\boldsymbol{NL}-\frac{d\hat{s}'}{ds}\hat{\boldsymbol{NL}}'+\boldsymbol{L}(\boldsymbol{\varepsilon})-\frac{d\hat{s}'}{ds}\hat{\boldsymbol{L}}(\hat{\boldsymbol{\varepsilon}}'), \bos{H}^{*i} \bos{\Phi}_M\right\rangle=0.
\ee
For the the scale changing term, the coercivity and the fact that $\frac{\lambda_s}{\lambda}\sim b_1$ give:
\be \label{variete:modulation:eq:dhats Lambda}
\left| \langle \frac{\lambda_s}{\lambda}\bos{\Lambda}(\boldsymbol{\varepsilon}- \hat{\boldsymbol{\varepsilon}}' ),\bos{H}^{*i} \bos{\Phi}_M\rangle \right| \leq C(M)b_1^{L+1+(1-\delta_0)(1+\frac{\eta}{2})}\sqrt{\triangle_r \hat{\mathcal{E}}_{s_L}}.
\ee
The $\tilde{\bos{Mod}}$ term catches the dynamics on the manifold $(\tilde{\bos{Q}}_{b,\lambda})_{\lambda,b}$. Taking $i=0$ in \fref{variete:modulation:eq:expression generale} means that we are computing the law for the scaling. But by the very definition \fref{variete:eq:def hats'} of the time $\hat{s}'$, the two solutions have the same scale. This property induces the law for $\hat{s}'$ as we are going to see. Using the notations \fref{variete:eq:def triangle hatmodi} and \fref{variete:eq:def triangle hatmod0} one writes:
\be \label{variete:modulation:eq:dhats mod intermediaire}
\langle \tilde{\boldsymbol{Mod}}-\hat{\boldsymbol{Mod}}', \bos{\Phi}_M\rangle=\sum_0^L \langle \triangle \hat{\bos{Mod}}_i, \bos{\Phi}_M\rangle.
\ee
Using the orthogonality conditions \fref{thetrapped:eq:orthogonalite PhiM} and the fact that $M\ll B_1,B_1'$ one decomposes for $1\leq i \leq L$:
\be \label{variete:modulation:eq:dhats mod intermediaire 1}
\begin{array}{r c l}
&\langle \triangle \hat{\bos{Mod}}_i, \bos{\Phi}_M\rangle  \\
=& \langle \sum_{i=1}^L (b_{i,s}+(i-\alpha)b_1b_i-b_{i+1})(\sum_{j=i+1}^{L+2}\frac{\partial \bos{S}_{j}}{\partial b_i}-\frac{\partial \hat{\bos{S}}'_{j}}{\partial b_i})  ,\bos{\Phi}_M \rangle \\
&+  \underset{i=1}{\overset{L}{\sum}} (b_{i,s}+(i-\alpha)b_1b_i-b_{i+1}-(\hat{b}'_{i,s}+\frac{d\hat{s}'}{ds}((i-\alpha)\hat{b}'_1\hat{b}'_i-\hat{b}'_{i+1})))\langle \underset{j=i+1}{\overset{L+2}{\sum}}\frac{\partial \hat{\bos{S}}'_{j}}{\partial b_i}  ,\bos{\Phi}_M \rangle 
\end{array}
\ee
Now we recall that $\bos{S}_j$ is an homogeneous profile of degree $(j,j-g',j \ \text{mod}2,j) $. It implies that for $1\leq i <j\leq L+2$, one has the bound: $\left|\frac{\partial \bos{S}_j}{\partial b_i}\right|\leq C(L,M)b_1$ on $y\leq 2M$ (and similarly for $\hat{\bos{S}}'$). Hence the bound for the second term in \fref{variete:modulation:eq:dhats mod intermediaire 1}:
\be \label{variete:modulation:eq:dhats mod intermediaire 2}
\begin{array}{r c l}
&\left| \underset{i=1}{\overset{L}{\sum}} (b_{i,s}+(i-\alpha)b_1b_i-b_{i+1}-(\hat{b}'_{i,s}+\frac{d\hat{s}'}{ds}((i-\alpha)\hat{b}'_1\hat{b}'_i-\hat{b}'_{i+1})))\langle \underset{j=i+1}{\overset{L+2}{\sum}}\frac{\partial \hat{\bos{S}}'_{j}}{\partial b_i}  ,\bos{\Phi}_M \rangle \right| \\
\leq& C(L,M)b_1 \triangle D(t).
\end{array}
\ee
The homogeneity of the $\bos{S}_j$'s means that: $ \frac{\partial \bos{S}_j}{\partial b_i}=\sum_{J\in \mathcal{J}} b^J f_J$ and $ \frac{\partial \hat{\bos{S}}'_j}{\partial b_i}=\sum_{J\in \mathcal{J}} \hat{b}^{'J} f_J$ where the $J$'s are non null: $J\neq (0,...,0)$. Using the bound \fref{variete:eq:difference polynomes parametres} on $b^J-\hat{b}^{'J}$ we obtain that for $y\leq 2M$, $\left|\frac{\partial \bos{S}_j}{\partial b_i}- \frac{\partial \hat{\bos{S}}'_j}{\partial b_i}\right| \leq b_1C(L,M)\underset{1\leq i\leq \text{min}(|J|_2-1,L)}{\text{sup}}|b_1^{-i}\triangle \hat{b}_i|$. Moreover, we know from the modulation equations \fref{thetrapped:eq:modulation leq L-1} and \fref{thetrapped:eq:premiere modulation L} that $|b_{i,s}+(i-\alpha)b_1b_i-b_{i+1}|\leq C(L,M)b_1^{L+(1-\delta_0)(1+\eta)} $. Hence we get the following bound for the second term in the right hand side of \fref{variete:modulation:eq:dhats mod intermediaire 1}:
\be \label{variete:modulation:eq:dhats mod intermediaire 3}
\begin{array}{r c l}
&\left| \left\langle \sum_{i=1}^L (b_{i,s}+(i-\alpha)b_1b_i-b_{i+1})(\sum_{j=i+1}^{L+2}\frac{\partial \bos{S}_{j}}{\partial b_i}-\frac{\partial \hat{\bos{S}}'_{j}}{\partial b_i})  ,\bos{\Phi}_M \right\rangle\right|\\
\leq& C b_1^{L+1+(1-\delta_0)(1+\eta)}\underset{1\leq i\leq L}{\text{sup}}|b_1^{-i}\triangle \hat{b}_i|.
\end{array}
\ee
The identity \fref{variete:modulation:eq:dhats mod intermediaire 1} and the bounds \fref{variete:modulation:eq:dhats mod intermediaire 2} and \fref{variete:modulation:eq:dhats mod intermediaire 3} give for $1\leq i \leq L$:
\be \label{variete:modulation:eq:dhats mod intermediaire 4}
\left| \langle \triangle \hat{\bos{Mod}}_i,\bos{\Phi}_M\rangle \right| \leq C(L,M)[b_1^{L+(1-\delta_0)(1+\eta)+1}\underset{1\leq i\leq L}{\text{sup}}|\triangle \hat{b}_i|+b_1 \triangle D(t)].
\ee
We now look at the first term in the sum in the right hand side of \fref{variete:modulation:eq:dhats mod intermediaire}. Using the same ideas we just used for the others, one gets:
\be \label{variete:modulation:eq:dhats mod intermediaire 5}
\begin{array}{r c l}
&\langle \triangle \hat{\bos{Mod}}_0, \bos{\Phi}_M\rangle = \langle (\frac{\lambda_s}{\lambda}+b_1)(\bos{\Lambda}\tilde{\bos{Q}}_b-\bos{\Lambda}\tilde{\bos{Q}}_{\hat{b}'})+(b_1-\frac{d \hat{s}'}{ds}\hat{b}'_1))\bos{\Lambda}\tilde{\bos{Q}}_{\hat{b}'}, \bos{\Phi}_M\rangle \\
=& O(b_1^{L+1+(1-\delta_0)(1+\eta)}\underset{1\leq i\leq L}{\text{sup}}|b_1^{-i}\triangle \hat{b}_i|)+O(b_1\triangle D(t)) +(b_1-\frac{d\hat{s}'}{ds}\hat{b}'_1)\langle \bos{\Lambda} \bos{Q},\bos{\Phi}_M \rangle.
\end{array}
\ee
We have estimated all terms involved in the identity \fref{variete:modulation:eq:dhats mod intermediaire} for the modulation term in \fref{variete:modulation:eq:dhats mod intermediaire 4} and \fref{variete:modulation:eq:dhats mod intermediaire 5}, giving:
\be \label{variete:modulation:eq:dhats mod}
\begin{array}{r c l}
\langle \tilde{\boldsymbol{Mod}}-\hat{\boldsymbol{Mod}}', \bos{\Phi}_M\rangle&=&(b_1-\frac{d\hat{s}'}{ds}\hat{b}'_1)\langle \bos{\Lambda} \bos{Q},\bos{\Phi}_M \rangle\\
&&+O(b_1^{L+1+(1-\delta_0)(1+\eta)}\underset{1\leq i\leq L}{\text{sup}}|\triangle \hat{b}_i|+b_1\triangle D(t)).
\end{array}
\ee
We can now come back to the modulation equation \fref{variete:modulation:eq:expression generale} for $i=0$. We have calculated all terms in the right hand side in \fref{variete:modulation:eq:dhats lineaire}, \fref{variete:modulation:eq:dhats NL L psib}, \fref{variete:modulation:eq:dhats Lambda} and \fref{variete:modulation:eq:dhats mod}, so it now writes (because $\langle \bos{\Lambda}\bos{Q},\bos{\Phi}_M\rangle\sim cM^{2k_0+2\delta_0}>0$ for $c>0$):
\be \label{variete:modulation:eq:dhats}
\left| b_1-\frac{d\hat{s}'}{ds}\hat{b}'_1\right|\leq C [b_1\triangle D(t)+b_1^{L+1+(1-\delta_0)(1+\frac{\eta}{2})}(b_1^{\frac{\eta}{2}(1-\delta_0)}\underset{1\leq i\leq L}{\text{sup}}b_1^{-i}|\triangle \hat{b}_i|+\sqrt{\triangle_r \hat{\mathcal{E}}_{s_L}})].
\ee
This identity gives a first bound for the law of $\hat{s}'$:
\be \label{variete:modulation:eq:dhats-1}
1-\frac{d\hat{s}'}{ds}=\frac{-\triangle \hat{b}_1}{\hat{b}'_1} +O[ \triangle D(t)+b_1^{L+(1-\delta_0)(1+\frac{\eta}{2})}(b_1^{\frac{\eta}{2}(1-\delta_0)}\underset{1\leq i\leq L}{\text{sup}}b_1^{-i}|\triangle \hat{b}_i|+\sqrt{\triangle_r \hat{\mathcal{E}}_{s_L}})].
\ee

\underline{Step 2:} Law for $\triangle \hat{b}_i$ for $1\leq i \leq L-1$. We look at \fref{variete:modulation:eq:expression generale} for $1\leq i\leq L-1$. The linear term disappear because of orthogonality conditions:
\be \label{variete:modulation:eq:leqL-1 lineaire}
\left\langle \bos{H}(\boldsymbol{\varepsilon}-\hat{\boldsymbol{\varepsilon}}' )-(\frac{d\hat{s}'}{ds}-1)\bos{H}(\hat{\boldsymbol{\varepsilon}}'),\bos{H}^{*i} \bos{\Phi}_M\right\rangle=0.
\ee
The scale changing term is estimated as before:
\be  \label{variete:modulation:eq:leqL-1 Lambda}
\left| \langle \frac{\lambda_s}{\lambda}\bos{\Lambda}(\boldsymbol{\varepsilon}- \hat{\boldsymbol{\varepsilon}}' ),\bos{H}^{*i} \bos{\Phi}_M\rangle \right|\leq C(L,M) b_1^{L+1+(1-\delta_0)(1+\frac{\eta}{2})}\sqrt{\triangle_r \hat{\mathcal{E}}_{s_L}}.
\ee
The bounds \fref{variete:eq:asymptotique diff psib leqB0}, \fref{variete:modulation:eq:dhats-1} and \fref{linearized:eq:estimations locales non coupe} on on $\tilde{\bos{\psi}}_b-\tilde{\bos{\psi}}_{\hat{b}'}$, $|\frac{d\hat{s}'}{ds}-1|$ and $\hat{\bos{\psi}}_{\hat{b}'}$ imply: 
\be \label{variete:modulation:eq:leqL-1 psib}
\begin{array}{r c l}
\left| \langle \tilde{\boldsymbol{\psi}}_b-\frac{d\hat{s}'}{ds}\hat{\boldsymbol{\psi}}_{\hat{b}'},\bos{H}^{*i}\bos{\Phi}_M\rangle \right| &=& \left| \langle \tilde{\boldsymbol{\psi}}_b-\hat{\boldsymbol{\psi}}_{\hat{b}'}+(1-\frac{d\hat{s}'}{ds})\hat{\boldsymbol{\psi}}_{\hat{b}'},\bos{H}^{*i}\bos{\Phi}_M\rangle \right| \\
&\leq& Cb_1^{L+3}(\underset{1\leq i \leq L}{\text{sup}} b_1^{-i}|\triangle b_i|+\triangle D(t)+\sqrt{\triangle_r \hat{\mathcal{E}}_{s_L}}).
\end{array}
\ee
For the nonlinear terms, we have that $NL=\sum_{j=0}^{p-2} C_j\tilde{Q}_b^{(1)j}\varepsilon^{(1)(p-j)}$ and similarly for $\hat{NL}'$. Fix $j$, $1\leq j \leq p-2$. We estimate, using the bound \fref{variete:eq:asymptotique diff Qp-j} on $\tilde{Q}_b^{(1)j}-\tilde{Q}_{\hat{b}'}^{(1)j}$:
$$
\begin{array}{r c l}
&\parallel \tilde{Q}_b^{(1)j}\varepsilon^{(1)(p-j)}-\tilde{Q}_{\hat{b}'}^{(1)j}(\hat{\varepsilon}^{'(1)})^{(p-j)} \parallel_{L^1,y\leq 2M}\\
\leq & \parallel \tilde{Q}_b^{(1)j}-\tilde{Q}_{\hat{b}'}^{(1)j} \parallel_{L^{\infty},y\leq 2M}\parallel \parallel \varepsilon^{(1)(p-j)} \parallel_{L^1,y\leq 2M}\\
&+C\parallel \varepsilon^{(1)(p-j)}-(\hat{\varepsilon}^{'(1)})^{(p-j)} \parallel_{L^1,y\leq 2M} \\
\leq & Cb_1^{2L+(1-\delta_0)(2+\eta)}\underset{1\leq i \leq L}{\text{sup}} |\triangle \hat{b}_i |+Cb_1^{2L+(1-\delta_0)(2+\eta)}\sqrt{\triangle_r \hat{\mathcal{E}}_{s_L}}.
\end{array}
$$
For $j=0$ one has: $\parallel  \varepsilon^{(1)p}-(\hat{\varepsilon}^{'(1)})^{p}  \parallel_{L^1,y\leq 2M}\leq Cb_1^{2L+(1-\delta_0)(2+\eta)}\sqrt{\triangle_r \hat{\mathcal{E}}_{s_L}}$. The previous bounds and the bound \fref{variete:modulation:eq:dhats-1} on $\frac{d\hat{s}'}{ds}-1$ finally imply:
\be \label{variete:modulation:eq:leqL-1 NL}
\begin{array}{l l l l l}
&\left|\langle \boldsymbol{NL}-\frac{d\hat{s}'}{ds}\hat{\boldsymbol{NL}}', \bos{H}^{*i}\bos{\Phi}_M \rangle \right|= \left| \langle \boldsymbol{NL}-\hat{\boldsymbol{NL}}'+(1-\frac{d\hat{s}'}{ds})\hat{\boldsymbol{NL}}', \bos{H}^{*i}\bos{\Phi}_M\rangle \right|  \\
\leq & C\parallel  NL-\hat{NL}' \parallel_{L^1,y\leq 2M}+C|1-\frac{d\hat{s}'}{ds}|\parallel \hat{NL}' \parallel_{L^1,y\leq 2M} \\
\leq& Cb_1^{2L+(1-\delta_0)(2+\eta)-1}(\underset{1\leq i \leq L}{\text{sup}} |\triangle \hat{b}_i |+\sqrt{\triangle_r \hat{\mathcal{E}}_{s_L}}+\triangle D(t)).
\end{array}
\ee
We treat the same way the small linear term:
\be \label{variete:modulation:eq:leqL-1 L}
\begin{array}{r c l}
&\left| \langle \boldsymbol{L}(\boldsymbol{\varepsilon})-\frac{d\hat{s}'}{ds}\hat{\boldsymbol{L}}(\hat{\boldsymbol{\varepsilon}}'), \bos{H}^{*i} \bos{\Phi}_M\rangle \right| \\
\leq & C\parallel (\tilde{Q}_b^{(1)(p-1)}-Q^{p-1})\varepsilon^{(1)}-(\tilde{Q}_{\hat{b}'}^{(1)(p-1)}-Q^{p-1})\hat{\varepsilon}^{'(1)}  \parallel_{L^1,y\leq 2M}\\
&+C|1-\frac{d\hat{s}'}{ds}|\parallel \hat{L}(\hat{\varepsilon}') \parallel_{L^1} \\
\leq & Cb_1^{L+1+(1-\delta_0)(1+\frac{\eta}{2})}[b_1^{\frac{\eta}{2}(1-\delta_0)}\underset{1\leq i \leq L}{\text{sup}} |\triangle \hat{b}_i|+\sqrt{\triangle_r \hat{\mathcal{E}}_{s_L}}+b_1^{\frac{\eta}{2}(1-\delta_0)}\triangle D(t)].
\end{array}
\ee
Finally, for the modulation term, using the same tools employed for $i=0$ we obtain:
\be \label{variete:modulation:eq:leqL-1 mod}
\begin{array}{l l}
\langle \tilde{\boldsymbol{Mod}}-\hat{\boldsymbol{Mod}}',\bos{H}^{*i} \bos{\Phi}_M\rangle = O(b_1\triangle D(t)+b_1^{L+1+(1-\delta_0)(1+\eta)}\underset{1\leq i \leq L}{\text{sup}}|\triangle \hat{b}_i|).\\
+(b_{i,s}+(i-\alpha)b_1b_i-b_{i+1}-(\hat{b}'_{i,s}+\frac{d\hat{s}'}{ds}((i-\alpha)\hat{b}'_1\hat{b}'_i-\hat{b}'_{i+1})))\langle \bos{\Lambda}\bos{Q},\bos{\Phi}_M\rangle.

\end{array}
\ee
We now collect all the estimates we have showed, \fref{variete:modulation:eq:leqL-1 lineaire}, \fref{variete:modulation:eq:leqL-1 Lambda}, \fref{variete:modulation:eq:leqL-1 psib}, \fref{variete:modulation:eq:leqL-1 NL}, \fref{variete:modulation:eq:leqL-1 L} and \fref{variete:modulation:eq:leqL-1 mod} and inject them in \fref{variete:modulation:eq:expression generale}. This gives:
\be \label{variete:modulation:eq:leqL-1}
\begin{array}{r c l}
&\left| b_{i,s}+(i-\alpha)b_1b_i-b_{i+1}-(\hat{b}'_{i,s}+\frac{d\hat{s}'}{ds}((i-\alpha)\hat{b}'_1\hat{b}'_i-\hat{b}'_{i+1})) \right|\\
\leq& C [b_1\triangle D(t)+b_1^{L+1+(1-\delta_0)(1+\frac{\eta}{2})}(b_1^{\frac{\eta}{2}(1-\delta_0)}\underset{1\leq i\leq L}{\text{sup}}b_1^{-i}|\triangle \hat{b}_i|+\sqrt{\triangle_r \hat{\mathcal{E}}_{s_L}})].
\end{array}
\ee

\underline{Step 3:} Law for $\triangle \hat{b}_L$. The computations we made in the previous step, to find the estimates \fref{variete:modulation:eq:leqL-1 psib}, \fref{variete:modulation:eq:leqL-1 Lambda}, \fref{variete:modulation:eq:leqL-1 NL}, \fref{variete:modulation:eq:leqL-1 L} and \fref{variete:modulation:eq:leqL-1 mod} still work when taking $i=L$. The difference is that the linear term does not cancel anymore. Namely, using the bound \fref{variete:modulation:eq:dhats-1} on $\frac{d\hat{s}'}{ds}-1$:
$$
\begin{array}{r c l}
&\left| \langle \bos{H}(\boldsymbol{\varepsilon}-\hat{\boldsymbol{\varepsilon}}' )-(\frac{d\hat{s}'}{ds}-1)\bos{H}(\hat{\boldsymbol{\varepsilon}}'),\bos{H}^{*i} \bos{\Phi}_M\rangle \right| \\
\leq & C\parallel \boldsymbol{\varepsilon}-\hat{\boldsymbol{\varepsilon}}'  \parallel_{L^2,y\leq 2M}+|\frac{d\hat{s}'}{ds}-1|\parallel \hat{\boldsymbol{\varepsilon}}' \parallel_{L^2,y\leq 2M} \\
\leq &C b_1^{L+(1-\delta_0)(1+\frac{\eta}{2})}\sqrt{\triangle_r \hat{\mathcal{E}}_{s_L}}+Cb_1^{L+(1-\delta_0)(1+\eta)}(\triangle D(t)+\underset{1\leq i \leq L}{\text{sup}}b_1^{-i}|\triangle \hat{b}_i|).
\end{array}
$$
So for $i=L$ in \fref{variete:modulation:eq:expression generale} one obtains:
\be \label{variete:modulation:eq:L}
\begin{array}{r c l}
&\left| b_{L,s}+(L-\alpha)b_1b_L-(\hat{b}'_{L,s}+\frac{d\hat{s}'}{ds}(L-\alpha)\hat{b}'_1\hat{b}'_L) \right|\\
\leq& C \left( b_1\triangle D(t)+b_1^{L+(1-\delta_0)(1+\frac{\eta}{2})}(b_1^{\frac{\eta}{2}(1-\delta_0)}\underset{1\leq i\leq L}{\text{sup}}b_1^{-i}|\triangle \hat{b}_i|+\sqrt{\triangle_r \hat{\mathcal{E}}_{s_L}}) \right).
\end{array}
\ee

\underline{Step 4:} Gathering the bounds. We now put together the primary bounds we found so far for the scaling \fref{variete:modulation:eq:dhats}, for the parameters \fref{variete:modulation:eq:leqL-1} and \fref{variete:modulation:eq:L} to find that:
$$
|D(t)|\leq  C b_1^{L+(1-\delta_0)(1+\frac{\eta}{2})}[b_1^{\frac{\eta}{2}(1-\delta_0)}\underset{1\leq i\leq L}{\text{sup}}b_1^{-i}|\triangle \hat{b}_i|+\sqrt{\triangle_r \hat{\mathcal{E}}_{s_L}})].
$$ 
We reinject it in the previous primary bounds \fref{variete:modulation:eq:dhats}, \fref{variete:modulation:eq:leqL-1} and \fref{variete:modulation:eq:L} to obtain the bounds \fref{variete:modulation:eq:dhats (triangle b)}, \fref{variete:modulation:eq:leqL-1 (triangle b)} and \fref{variete:modulation:eq:L (triangle b)} claimed in the lemma.

\end{proof}

We are now going to improve our control over $\triangle b_L$ by the same technique we used in Lemma \fref{trappedregime:improvedmodulation:lem:improvedmodulation}. After an integration by parts in time, the time evolution of $\triangle \hat{b}_L$ enjoys a sufficiently good estimate for our purpose, as the ones we just proved for $\triangle \hat{b}_i$ for $1\leq i \leq L-1$ in Lemma \ref{variete:lem:modulation diff}.

\begin{lemma}[Improved modulation equation for $\triangle \hat{b}_L$]
\label{variete:lem:improved modulation}

There holds\footnote{the denominator being strictly positive from \fref{thetrapped:eq:improved modulation taille denominateur}.} :
\begin{equation} \label{variete:eq:improved modulation}
\begin{array}{r c l}
& (b_{L,s}+(L-\alpha)b_1b_L-(\hat{b}'_{L,s}+\frac{d\hat{s}'}{ds}(L-\alpha)\hat{b}'_1\hat{b}'_L)) \\
=&\frac{d}{ds}\left[ \frac{\langle \bos{H}^{L}(\bos{\varepsilon}-\hat{\bos{\varepsilon}}') , \chi_{B_0} \bos{\Lambda} \bos{Q}\rangle-  \hat{b}_L'\int \chi_{B_0}\Lambda^{(1)} Q\left( \frac{\partial S_{L+2}}{\partial b_L}-\frac{\partial \hat{S}'_{L+2}}{\partial b_L} \right)_{L-1}}{\Bigl\langle  \chi_{B_0}\Lambda^{(1)}Q,\Lambda^{(1)}Q+(-1)^{\frac{L-1}{2}}\left(\frac{\partial S_{L+2}^{(2)}}{\partial b_L}\right)_{L-1} \Bigr\rangle} \right] \\
&+O[b_1^{L+1+\frac{\eta}{2}(1-\delta_0)}(\sqrt{\triangle_r \hat{\mathcal{E}}_{s_L}} + b_1^{\frac{\eta}{2}(1-\delta_0)}\text{sup} \ b_1^{-i}|\triangle \hat{b}_i|)].
\end{array}
\end{equation}
The quantity appearing via its derivative in time has the following size:
\be \label{variete:eq:gain improved modulation}
\begin{array}{r c l}
&\left| \frac{\langle \bos{H}^{L}(\bos{\varepsilon}-\hat{\bos{\varepsilon}}') , \chi_{B_0} \bos{\Lambda} \bos{Q}\rangle-  \hat{b}_L'\int \chi_{B_0}\Lambda^{(1)} Q\left( \frac{\partial S_{L+2}}{\partial b_L}-\frac{\partial \hat{S}'_{L+2}}{\partial b_L} \right)_{L-1}}{\Bigl\langle  \chi_{B_0}\Lambda^{(1)}Q,\Lambda^{(1)}Q+(-1)^{\frac{L-1}{2}}\left(\frac{\partial S_{L+2}^{(2)}}{\partial b_L}\right)_{L-1} \Bigr\rangle}\right| \\
 \leq& C b_1^{L+\frac{\eta}{2} (1-\delta_0)}(\sqrt{\triangle_r \hat{\mathcal{E}}_{s_L}}+b_1^{g'+O(\eta)}\underset{1\leq i \leq 2}{\text{sup}}b_1^{-i}|\triangle \hat{b}_i|).
\end{array}
\ee
\end{lemma}

\begin{proof}[Proof of Lemma \ref{variete:lem:improved modulation}]
We will do the same computations we did to prove Lemma \fref{trappedregime:improvedmodulation:lem:improvedmodulation}, this time expressing everything in function of the differences $\triangle \hat{b}_i$ and $\bos{\varepsilon}-\hat{\bos{\varepsilon}}'$.\\
 \underline{Step 1:} Time derivative of the numerator in \fref{variete:eq:improved modulation}. We compute for the first term:
\be \label{variete:eq:expression numerateur}
\frac{d}{ds}  \langle \bos{H}^L (\bos{\varepsilon}-\hat{\bos{\varepsilon}}'), \chi_{B_0}\bos{\Lambda} \bos{Q} \rangle = \langle \bos{H}^L (\bos{\varepsilon}_s-\hat{\bos{\varepsilon}}'_s), \chi_{B_0}\bos{\Lambda} \bos{Q} \rangle + \langle \bos{H}^L (\bos{\varepsilon}-\hat{\bos{\varepsilon}}'), b_{1,s} y\partial_y \chi (\frac{y}{B_0})\bos{\Lambda} \bos{Q} \rangle .
\ee
We now calculate everything in the right hand side. For the second term:
\be \label{variete:eq:expression deuxieme terme improved}
\left| \langle \bos{H}^L (\bos{\varepsilon}-\hat{\bos{\varepsilon}}'), b_{1,s} \partial_y \chi (\frac{y}{B_0})\bos{\Lambda} \bos{Q} \rangle \right| \leq C(M)\sqrt{\triangle \hat{\mathcal{E}}_{s_L}} b_1^{-(2k_0+\delta_0)} .
\ee
We will now estimate the first term in the right hand side of \fref{variete:eq:expression numerateur}. From the time evolution of the difference \fref{variete:eq:evolution difference epsilon}, one gets:
\be \label{variete:eq:expression premier terme improved}
\begin{array}{r c l}
& (-1)^{\frac{L+1}{2}}\langle \bos{H}^{L} (\bos{\varepsilon}_s-\hat{\bos{\varepsilon}}_s), \chi_{B_0}\bos{\Lambda} \bos{Q}\rangle \\
=&  \int \chi_{B_0} \Lambda^{(1)}Q \times \bigl{(}\mathcal{L}(\hat{\varepsilon}^{'(1)}-\varepsilon^{(1)})+\frac{\lambda_s}{\lambda}\Lambda^{(2)}(\varepsilon^{(2)}-\hat{\varepsilon}^{'(2)}) -(\tilde{Mod}(t)^{(2)}-\hat{Mod}{'(2)})\\
&-\tilde{\psi}_b^{(2)}+\tilde{\psi}_{\hat{b}'}^{(2)}+NL-\hat{NL}'+L-\hat{L}'+(\frac{d\hat{s}'}{ds}-1)(\mathcal{L}\hat{\varepsilon}^{'(1)}+\tilde{\psi}_{\hat{b}'}^{(2)}-\hat{NL}'-\hat{L}' )  \bigr{)}_{L-1} 
\end{array}
\ee
and we now consider each term in the right hand side. \\
$\bullet$ \emph{Linear term:} One has the bound from coercivity:
\be \label{variete:eq:estimation premier terme improved intermediaire lineaire}
\left| \int \chi_{B_0}\Lambda^{(1)}Q (\mathcal{L}(\varepsilon^{(1)}-\hat{\varepsilon}^{'(1)}))_{L-1}  \right| \leq C(M) \sqrt{\triangle_r \hat{\mathcal{E}}_{s_L}} b_1^{-2(k_0+\delta_0)+L+1+\frac{\eta}{2}(1-\delta_0)} .
\ee
$\bullet$\emph{Scale changing term:} One has the same bound:
\be  \label{variete:eq:estimation premier terme improved intermediaire lambdavarepsilon}
\left| \int\chi_{B_0}\Lambda^{(1)}Q \frac{\lambda_s}{\lambda} (\Lambda^{(2)}(\varepsilon^{(2)}-\hat{\varepsilon}^{'(2)}))_{L-1} \right| \leq C(M) \sqrt{\triangle_r \hat{\mathcal{E}}_{s_L}} b_1^{-2(k_0+\delta_0)+L+1+\frac{\eta}{2}(1-\delta_0)} .
\ee
$\bullet$ \emph{Error term:} As we are in the zone $y\leq 2B_0$ we can use the asymptotic \fref{variete:eq:asymptotique diff psib leqB0}:

\be \label{variete:eq:estimation premier terme improved intermediaire psib}
\left| \int \chi_{B_0}\Lambda^{(1)}Q (\tilde{\psi_b}^{(2)}-\tilde{\psi}_{\hat{b}'}^{(2)})_{L-1} \right| \leq  C b_1^{-2(k_0+\delta_0)+L+1+g'}\underset{1\leq i \leq L}{\text{sup}}b_1^{-i}|\triangle \hat{b}_i| .
\ee
$\bullet$ \emph{$NL$ term:} We start by puting all the adapted derivatives on $\chi_{B_0}\Lambda^{(1)}Q$, localizing the integral in $B_0\leq y\leq2B_0$ as $A\Lambda^{(1)}Q=0$:
$$
\left| \int \chi_{B_0} \Lambda^{(1)}Q (NL-\hat{NL}')_{L-1}  \right| \leq C  \int_{B_0}^{2B_0} \frac{1}{y^{\gamma+L-1}} |NL-\hat{NL}'| .
$$
We know that $NL$ is a sum of terms of the form\footnote{we write $\tilde{Q}_b$ instead of $\tilde{Q}_b^{(1)}$ to ease notations.}: $\tilde{Q}_b^{p-k}\varepsilon^{(1)k}$ for $k>2$, and similarly for $\hat{NL}'$. Suppose that $k=p$, then:
\be \label{variete:improved:eq:NL 1}
\begin{array}{r c l}
\int_{B_0}^{2B_0} \frac{|\varepsilon^{(1)p}-\hat{\varepsilon}^{'(1)p}|}{y^{\gamma+L-1}} &\leq& C\text{max}(\parallel \varepsilon^{(1)}\parallel_{L^{\infty}}^{p-1},\parallel \varepsilon^{'(1)}\parallel_{L^{\infty}}^{p-1})\int_{B_0}^{2B_0} \frac{1}{y^{\gamma+L-1}} |\varepsilon^{(1)}-\hat{\varepsilon}^{'(1)}| \\
&\leq & C  \left( \frac{\sqrt{\mathcal{E}_{\sigma}}+\sqrt{\hat{\mathcal{E}}'_{\sigma}}}{b_1^{\sigma-s_c}} \right)^{p-1} b_1^{2+\frac{2\alpha}{L}+O\left(\frac{\sigma-s_c}{L} \right)}  \sqrt{\triangle \hat{\mathcal{E}}_{s_L}}b_1^{-2(k_0+\delta_0)-2} \\
&\leq & C b_1^{-2(k_0+\delta_0)+L+1+\frac{\alpha}{2}+O\left(\eta,\frac{\sigma-s_c}{L} \right)} \sqrt{\triangle _r \hat{\mathcal{E}}_{s_L}},
\end{array}
\ee
where we used the estimates \textbf{mettre ref} of the trapped regime (we recall that they hold for both $\bos{\varepsilon}$ and $\bos{\varepsilon}'$ as $b_1\sim \hat{b}'_1$ from \fref{variete:eq: equivalence s et hats'}). Suppose now $2\leq k\leq p-1$. We start by splitting in two parts:
$$
\begin{array}{r c l}
\left|\int_{B_0}^{2B_0} \frac{\tilde{Q}_b^{p-k}\varepsilon^{(1)k}-\tilde{Q}_{\hat{b}'}^{p-k}\hat{\varepsilon}^{'(1)k}}{y^{\gamma+L-1}}\right| &=& \left|\int_{B_0}^{2B_0} \frac{(\tilde{Q}_b^{p-k}-\tilde{Q}_{\hat{b}'}^{p-k})\varepsilon^{(1)k}}{y^{\gamma+L-1}}+\frac{\tilde{Q}_{\hat{b}'}^{p-k}(\varepsilon^{(1)k}-\hat{\varepsilon}^{'(1)k})}{y^{\gamma+L-1}}\right|.
\end{array}
$$
For the first part, from the bound \fref{variete:eq:asymptotique diff Qp-j} for $\tilde{Q}_b^{p-k}-\tilde{Q}_{\hat{b}'}^{p-k}$, one gets:
$$
\begin{array}{r c l}
\left|\int_{B_0}^{2B_0} \frac{(\tilde{Q}_b^{p-k}-\tilde{Q}_{\hat{b}'}^{p-k})\varepsilon^{(1)k}}{y^{\gamma+L-1}}\right|&\leq& Cb_1^{\frac{2(p-k)}{p-1}+\alpha+O(\eta)}\underset{1\leq i \leq L}{\text{sup}}b_1^{-i}|\triangle \hat{b}_i| \parallel \varepsilon^{(1)} \parallel_{L^{\infty}}^{k-1}\int_{B_0}^{2B_0} \frac{|\varepsilon^{(1)}|}{y^{\gamma+L-1}} \\
&\leq & Cb_1^{-2(k_0+\delta_0)+L+1+\alpha+O\left( \eta,\sigma-s_c,\frac{1}{L} \right)}\underset{1\leq i \leq L}{\text{sup}}b_1^{-i}|\triangle \hat{b}_i|.
\end{array}
$$
For the second part, as $|\tilde{Q}_{\hat{b}'}^{(p-k)}|\leq Cb_1^{\frac{2(p-k)}{p-1}} $ for $B_0\leq y \leq 2B_0$ one gets using again the $L^{\infty}$ estimate and coercivity:
$$
\begin{array}{r c l}
\left| \underset{B_0}{\overset{2B_0}{\int}} \frac{(\tilde{Q}_{\hat{b}'}^{p-k}(\varepsilon^{(1)k}-\hat{\varepsilon}^{'(1)k}))}{y^{\gamma+L-1}}\right|&\leq &\text{max}(\parallel \varepsilon^{(1)}\parallel_{L^{\infty}}^{k-1},\parallel \hat{\varepsilon}^{'(1)}\parallel_{L^{\infty}}^{k-1}) b_1^{-\frac{2(p-k)}{p-1}-(2k_0+\delta_0)}\sqrt{\triangle \hat{\mathcal{E}}_{s_L}}\\
&\leq& b_1^{-2(k_0+\delta_0)+L+1+\frac{2(k-1)\alpha}{(p-1)L}+O\left(\eta,\frac{\sigma-s_c}{L} \right)}\sqrt{\triangle_r \hat{\mathcal{E}_{s_L}}}.
\end{array}
$$
As $\eta \ll 1$ the last bounds give the following estimate for the non linear term:
\be \label{variete:eq:estimation premier terme improved intermediaire NL}
\begin{array}{r c l}
&\left| \int \chi_{B_0} \Lambda^{(1)} Q (NL-\hat{NL}')_{L-1} \right|\\
\leq& Cb_1^{-2(k_0+\delta_0)+L+1+\frac{\eta}{2}(1-\delta_0)}(\sqrt{\triangle_r \hat{\mathcal{E}_{s_L}}} +b_1^{\frac{\eta}{2}(1-\delta_0)}\underset{1\leq i \leq L}{\text{sup}}b_1^{-i}|\triangle \hat{b}_i|) .
\end{array}
\ee
$\bullet$ \emph{Small linear term:} One has: $L=(\tilde{Q}_b^{p-1}-Q^{p-1})\varepsilon^{(1)} $ and similarly for $\hat{L}'$. As for the non-linear we start by decomposing:
$$
\left| \int \chi_{B_0} \Lambda^{(1)}Q (L-\hat{L}')_{L-1} \right| \leq C \int_{B_0}^{2B_0} \frac{|\tilde{Q}_b^{p-1}-\tilde{Q}_{\hat{b}'}^{p-1}||\varepsilon^{(1)}|+|\tilde{Q}_{\hat{b}'}^{p-1}||\varepsilon^{(1)}-\hat{\varepsilon}^{'(1)})|}{y^{\gamma+L-1}} .
$$
For the first term we use the asymptotic \fref{variete:eq:asymptotique diff Qp-j} for $\tilde{Q}_b^{p-1}-\tilde{Q}_{\hat{b}'}^{p-1}$, yielding:
$$
\int_{B_0}^{2B_0} \frac{1}{y^{\gamma+L-1}}|\tilde{Q}_b^{p-1}-\tilde{Q}_{\hat{b}'}^{p-1}||\varepsilon^{(1)}|\leq b_1^{-2(k_0+\delta_0)+L+1+\alpha+O(\eta)}(\underset{1\leq i \leq L}{\text{sup}}b_1^{-i}|\triangle \hat{b}_i|).
$$
For the second term, from $|\tilde{Q}_{\hat{b}'}^{p-1}|\leq C b_1^{2} $ for $B_0\leq y\leq 2B_0$ one gets:
$$
\int_{B_0}^{2B_0} \frac{1}{y^{\gamma+L-1}}(|\tilde{Q}_{\hat{b}'}^{p-1}||\varepsilon^{(1)}-\hat{\varepsilon}^{'(1)})|)\leq Cb_1^{-2(k_0+\delta_0)+L+\frac{\eta}{2}(1-\delta_0)}\sqrt{\triangle_r \hat{\mathcal{E}}_{s_L}} .
$$
The last two bounds show that for the small linear term:
\be \label{variete:eq:estimation premier terme improved intermediaire L}
\bigl{|} \int \chi_{B_0} \Lambda^{(1)}Q (L-\hat{L}')_{L-1} \bigr{|} \leq Cb_1^{-2(k_0+\delta_0)+L+\frac{\eta (1-\delta_0)}{2}}(\sqrt{\triangle_r \hat{\mathcal{E}}_{s_L}}+b_1^{\frac{\eta (1-\delta_0)}{2}}\underset{1\leq i \leq L}{\text{sup}}b_1^{-i}|\triangle \hat{b}_i|)
\ee
$\bullet$ \emph{The modulation term:} From the localization of the $T_i$' and $S_i$'s (\fref{linearized:eq:localisation Ti} and \fref{linearized:eq:degre Si}), and because $(T_i)_{L-1}=0$ for $i<L-1$:
\be \label{variete:eq:estimation premier terme improved intermediaire mod}
\begin{array}{r c l}
&\int (\tilde{Mod}^{(2)}-\hat{Mod}^{'(2)})_{L-1}\chi_{B_0}\Lambda^{(1)} Q \\
=& \int \chi_{B_0}\Lambda^{(1)}Q (b_{L,s}+(L-\alpha)b_1b_L) (T_L+\frac{\partial S_{L+2}}{\partial b_L})_{L-1}\\
&-\int\chi_{B_0}\Lambda^{(1)}Q(\hat{b}'_{L,s}+\frac{d\hat{s}'}{ds}(L-\alpha)\hat{b}'_1\hat{b}'_L) (T_L+\frac{\partial \hat{S}'_{L+2}}{\partial b_L}))_{L-1} \\
& +\underset{i=1}{\overset{L-1}{\sum}}\int \chi_{B_0}\Lambda^{(1)}Q ((b_{i,s}+(i-\alpha)b_1b_i-b_{i+1}) \left( \underset{j=i+1, \ j \ \text{odd}}{\overset{L+2}{\sum}}\frac{\partial S_{j}}{\partial b_i} \right)_{L-1} \\
&-\underset{i=1}{\overset{L-1}{\sum}}\int \chi_{B_0}\Lambda^{(1)}Q ((\hat{b}'_{i,s}+\frac{d\hat{s}'}{ds}((i-\alpha)\hat{b}'_1\hat{b}'_i-\hat{b}'_{i+1})) \left( \underset{j=i+1, \ j \ \text{odd}}{\overset{L+2}{\sum}}\frac{\partial \hat{S}'_{j}}{\partial b_i}\right)_{L-1} \\
&-\int \chi_{B_0}\Lambda^{(1)}Q((\frac{\lambda_s}{\lambda}+b_1)\Lambda^{(2)}\tilde{Q}_b^{(2)}-(\frac{\lambda_s}{\lambda}+\frac{d\hat{s}'}{ds}\hat{b}'_1)\Lambda^{(2)}\tilde{Q}_{\hat{b}'}^{(2)})_{L-1} .
\end{array}
\ee
We start by studying the first term in \fref{variete:eq:estimation premier terme improved intermediaire mod}. Since $\bos{H}(\bos{T}_L)=(-1)^{L}\bos{\Lambda Q}$:
$$
\begin{array}{r c l}
&\int \chi_{B_0}\Lambda^{(1)}Q (b_{L,s}+(L-\alpha)b_1b_L) (T_L+\frac{\partial S_{L+2}}{\partial b_L})_{L-1}\\
&-\int\chi_{B_0}\Lambda^{(1)}Q(\hat{b}'_{L,s}+\frac{d\hat{s}'}{ds}(L-\alpha)\hat{b}'_1\hat{b}'_L) (T_L+\frac{\partial \hat{S}'_{L+2}}{\partial b_L}))_{L-1}\\
=& (-1)^{\frac{L-1}{2}}(b_{L,s}+(L-\alpha)b_1b_L-(\hat{b}'_{L,s}+\frac{d\hat{s}'}{ds}(L-\alpha)\hat{b}'_1\hat{b}'_L) )\\
&\times \int \chi_{B_0}\Lambda^{(1)}Q \left(\Lambda^{(1)}Q+\left(\frac{\partial S_{L+2}}{\partial b_L}\right)_{L-1}\right) \\
&+(\hat{b}'_{L,s}+\frac{d\hat{s}'}{ds}(L-\alpha)\hat{b}'_1\hat{b}'_L)\int \chi_{B_0}\Lambda^{(1)}Q\left( \frac{\partial S_{L+2}}{\partial b_L}-\frac{\partial \hat{S}'_{L+2}}{\partial b_L} \right)_{L-1}\\
=& (-1)^{\frac{L-1}{2}}(b_{L,s}+(L-\alpha)b_1b_L-(\hat{b}'_{L,s}+\frac{d\hat{s}'}{ds}(L-\alpha)\hat{b}'_1\hat{b}'_L) )\\
&\times  \int \chi_{B_0}\Lambda^{(1)}Q \left(\Lambda^{(1)}Q+\left(\frac{\partial S_{L+2}}{\partial b_L}\right)_{L-1}\right) \\
&+\hat{b}'_{L,s}\int \chi_{B_0}\Lambda^{(1)}Q\left( \frac{\partial S_{L+2}}{\partial b_L}-\frac{\partial \hat{S}'_{L+2}}{\partial b_L} \right)_{L-1}+O(b_1^{-2(k_0+\delta_0)+L+1+g'}\underset{1\leq i \leq L}{\text{sup}}b_1^{-i}|\triangle \hat{b}_i|).
\end{array}
$$
For the second, third and fourth terms in \fref{variete:eq:estimation premier terme improved intermediaire mod}, using the modulation bounds \fref{variete:modulation:eq:dhats (triangle b)} and \fref{variete:modulation:eq:leqL-1 (triangle b)} from the proof of the last Lemma and splitting as we did before:
$$
\begin{array}{r c l}
& \bigl{|}\underset{i=1}{\overset{L-1}{\sum}}\int \chi_{B_0}\Lambda^{(1)}Q ((b_{i,s}+(i-\alpha)b_1b_i-b_{i+1}) \left( \underset{j=i+1, \ j \ \text{odd}}{\overset{L+2}{\sum}}\frac{\partial S_{j}}{\partial b_i} \right)_{L-1} \\
&-\underset{i=1}{\overset{L-1}{\sum}}\int \chi_{B_0}\Lambda^{(1)}Q ((\hat{b}'_{i,s}+\frac{d\hat{s}'}{ds}((i-\alpha)\hat{b}'_1\hat{b}'_i-\hat{b}'_{i+1})) \left( \underset{j=i+1, \ j \ \text{odd}}{\overset{L+2}{\sum}}\frac{\partial \hat{S}'_{j}}{\partial b_i}\right)_{L-1} \\
& -\int \chi_{B_0}\Lambda^{(1)}Q((\frac{\lambda_s}{\lambda}+b_1)\Lambda^{(2)}\tilde{Q}_b^{(2)}-(\frac{\lambda_s}{\lambda}+\frac{d\hat{s}'}{ds}\hat{b}'_1)\Lambda^{(2)}\tilde{Q}_{\hat{b}'}^{(2)})_{L-1} \bigr{|}\\
\leq & b_1^{-2(k_0+\delta_0)+L+1+g'+(1-\delta_0)+O(\eta)}(\sqrt{\triangle_r \hat{\mathcal{E}}_{s_L}} + \text{sup} \ b_1^{-i}|\triangle \hat{b}_i|)).
\end{array}
$$
With the previous computations, \fref{variete:eq:estimation premier terme improved intermediaire mod} becomes eventually:
\be \label{variete:eq:estimation premier terme improved mod}
\begin{array}{r c l}
& (-1)^{\frac{L-1}{2}} \int (\tilde{Mod(t)}^{(2)}-\hat{Mod}^{'(2)})_{L-1}\chi_{B_0}\Lambda^{(1)} Q \\
=& (b_{L,s}+(L-\alpha)b_1b_L-(\hat{b}'_{L,s}+\frac{d\hat{s}'}{ds}(L-\alpha)\hat{b}'_1\hat{b}'_L) )\\
&\times  \langle \chi_{B_0}\Lambda^{(1)}Q, \Lambda^{(1)}Q+(-1)^{\frac{L-1}{2}}(\frac{\partial S_{L+2}}{\partial b_L})_{L-1}\rangle \\
&+O(b_1^{-2(k_0+\delta_0)+g'+L+1}\sqrt{\triangle_r \hat{\mathcal{E}}_{s_L}})+\hat{b}'_{L,s}\int \chi_{B_0}\Lambda^{(1)}Q\left( \frac{\partial S_{L+2}}{\partial b_L}-\frac{\partial \hat{S}'_{L+2}}{\partial b_L} \right)_{L-1}\\
& +O[\text{sup} \ b_1^{-i}|\triangle \hat{b}_i|(b_1^{-2(k_0+\delta_0)+g'+L+1})].
\end{array}
\ee
$\bullet$ \emph{The time error term:} Using the upper bound \fref{variete:modulation:eq:dhats (triangle b)} for $\left|\frac{d\hat{s}'}{ds}-1 \right|$ and the previous bounds\fref{thetrapped:eq:estimation premier terme improved intermediaire lineaire}, \fref{thetrapped:eq:estimation premier terme improved intermediaire psib}, \fref{thetrapped:eq:estimation premier terme improved intermediaire L} and \fref{thetrapped:eq:estimation premier terme improved intermediaire NL} from the original Lemma about the improved modulation:
\be
\label{variete:eq:estimation premier terme improved intermediaire dhats'-1}
\begin{array}{r c l}
&\left| \int \chi_{B_0}\Lambda^{(1)}Q (\frac{d\hat{s}'}{ds}-1)(\mathcal{L}\hat{\varepsilon}^{'(1)}+\tilde{\psi}_{\hat{b}'}^{(2)}-\hat{NL}'-\hat{L}' )_{L-1} \right| \\
\leq & b_1^{-2(k_0+\delta_0)+L+1}(b_1^{L+1-\delta_0+O(\eta)}\sqrt{\triangle_r \hat{\mathcal{E}}_{s_L}} + b_1^{\eta(1-\delta_0)}\text{sup} \ b_1^{-i}|\triangle \hat{b}_i|).
\end{array}
\ee

We can now gather all the bounds \fref{variete:eq:estimation premier terme improved intermediaire lineaire}, \fref{variete:eq:estimation premier terme improved intermediaire lambdavarepsilon}, \fref{variete:eq:estimation premier terme improved intermediaire psib}, \fref{variete:eq:estimation premier terme improved intermediaire NL}, \fref{variete:eq:estimation premier terme improved intermediaire L}, \fref{variete:eq:estimation premier terme improved intermediaire mod} and \fref{variete:eq:estimation premier terme improved intermediaire dhats'-1}, inject them in \fref{variete:eq:expression premier terme improved} to find that the first term in the rhs of \fref{variete:eq:expression numerateur} is:
\be \label{variete:eq:estimation premier term improved modulation}
\begin{array}{r c l}
&\langle \bos{H}^L (\bos{\varepsilon}_s-\hat{\bos{\varepsilon}}'_s), \chi_{B_0}\bos{\Lambda} \bos{Q} \rangle-\hat{b}'_{L,s}\int \chi_{B_0}\Lambda^{(1)}Q\left( \frac{\partial S_{L+2}}{\partial b_L}-\frac{\partial \hat{S}'_{L+2}}{\partial b_L} \right)_{L-1}\\
=& (b_{L,s}+(L-\alpha)b_1b_L-(\hat{b}'_{L,s}+\frac{d\hat{s}'}{ds}((L-\alpha)\hat{b}'_1\hat{b}'_L)) \\
&\times \Bigl\langle  \chi_{B_0}\Lambda^{(1)}Q,\Lambda^{(1)}Q+(-1)^{\frac{L-1}{2}}\left(\frac{\partial S_{L+2}^{(2)}}{\partial b_L}\right)_{L-1} \Bigr\rangle\\
&+O(b_1^{-2(k_0+\delta_0)+L+1+\frac{\eta}{2}(1-\delta_0)}(\sqrt{\triangle_r \hat{\mathcal{E}}_{s_L}} + b_1^{\frac{\eta}{2}(1-\delta_0)}\text{sup} \ b_1^{-i}|\triangle \hat{b}_i|)).
\end{array}
\ee
Combining the two computations we made, \fref{variete:eq:estimation premier term improved modulation} and \fref{variete:eq:expression deuxieme terme improved}, the time evolution of the first term of the numerator in \fref{variete:eq:expression numerateur} is now:
\be\label{variete:eq:evolution temporelle numerateur}
\begin{array}{r c l}
&\frac{d}{ds}\langle \bos{H}^L (\bos{\varepsilon}-\hat{\bos{\varepsilon}}'), \chi_{B_0}\bos{\Lambda} \bos{Q} \rangle-\hat{b}'_{L,s}\int \chi_{B_0}\Lambda^{(1)}Q\left( \frac{\partial S_{L+2}}{\partial b_L}-\frac{\partial \hat{S}'_{L+2}}{\partial b_L} \right)_{L-1}\\
=& (b_{L,s}+(L-\alpha)b_1b_L-(\hat{b}'_{L,s}+\frac{d\hat{s}'}{ds}((L-\alpha)\hat{b}'_1\hat{b}'_L)) \\
&\times \Bigl\langle  \chi_{B_0}\Lambda^{(1)}Q,\Lambda^{(1)}Q+(-1)^{\frac{L-1}{2}}\left(\frac{\partial S_{L+2}^{(2)}}{\partial b_L}\right)_{L-1} \Bigr\rangle\\
&+O(b_1^{-2(k_0+\delta_0)+L+1+\frac{\eta}{2}(1-\delta_0)}(\sqrt{\triangle_r \hat{\mathcal{E}}_{s_L}} +b_1^{\frac{\eta}{2}(1-\delta_0)} \text{sup} \ b_1^{-i}|\triangle \hat{b}_i|)).
\end{array}
\ee

\underline{Step 2:} End of the computation. We can now end the proof of the Lemma. We recall that the denominator in \fref{variete:eq:improved modulation} and its time derivative have the following size:
$$
\begin{array}{l l}
\left\langle  \chi_{B_0}\Lambda^{(1)}Q,\Lambda^{(1)}Q+(-1)^{\frac{L-1}{2}}\left(\frac{\partial S_{L+2}^{(2)}}{\partial b_L}\right)_{L-1} \right\rangle\sim c b_1^{-2k_0-2\delta_0}, \ \ (c \ \text{a constant}, \ c>0)\\
\left| \frac{d}{ds}\Bigl\langle  \chi_{B_0}\Lambda^{(1)}Q,\Lambda^{(1)}Q+(-1)^{\frac{L-1}{2}}\left(\frac{\partial S_{L+2}^{(2)}}{\partial b_L}\right)_{L-1} \Bigr\rangle \right| \leq C b_1^{-2(k_0+\delta_0)+1}.
\end{array}
$$
We get by coercivity of the adapted norm:
\be \label{variete:improvedmodulation:gain inte}
|\langle \bos{H}^L (\bos{\varepsilon}-\hat{\bos{\varepsilon}}') ,\chi_{B_0}\bos{\Lambda} \bos{Q}\rangle |\leq C  b_1^{-2(k_0+\delta_0)-1+L+\frac{\eta}{2}(1-\delta_0)}\sqrt{\triangle_r \hat{\mathcal{E}}_{s_L}}.
\ee
The last three bounds, together with the identity \fref{variete:eq:evolution temporelle numerateur} we derived in Step 1, give:
\be \label{presque}
\begin{array}{r c l}
&\frac{d}{ds}\left[ \frac{\langle \bos{H}^{L}(\bos{\varepsilon}-\hat{\bos{\varepsilon}}') , \chi_{B_0} \bos{\Lambda} \bos{Q}\rangle}{\Bigl\langle  \chi_{B_0}\Lambda^{(1)}Q,\Lambda^{(1)}Q+(-1)^{\frac{L-1}{2}}\left(\frac{\partial S_{L+2}^{(2)}}{\partial b_L}\right)_{L-1} \Bigr\rangle}\right]\\
&-\hat{b}'_{L,s}\frac{\int \chi_{B_0}\Lambda^{(1)}Q\left( \frac{\partial S_{L+2}}{\partial b_L}-\frac{\partial \hat{S}'_{L+2}}{\partial b_L} \right)_{L-1}}{\Bigl\langle  \chi_{B_0}\Lambda^{(1)}Q,\Lambda^{(1)}Q+(-1)^{\frac{L-1}{2}}\left(\frac{\partial S_{L+2}^{(2)}}{\partial b_L}\right)_{L-1} \Bigr\rangle}\\
=& (b_{L,s}+(L-\alpha)b_1b_L-(\hat{b}'_{L,s}+\frac{d\hat{s}'}{ds}(L-\alpha)\hat{b}'_1\hat{b}'_L)) \\
&+O(b_1^{L+1+\frac{\eta}{2}(1-\delta_0)}(\sqrt{\triangle_r \hat{\mathcal{E}}_{s_L}} + b_1^{\frac{\eta}{2}(1-\delta_0)}\text{sup} \ b_1^{-i}|\triangle \hat{b}_i|)).
\end{array}
\ee
As $\frac{\partial \bos{S}_{L+2}}{\partial b_L}$ is homogeneous of degree $(L+2,L+2,1,2)$ and does not depend on $b_L$, we have using the modulation bounds \fref{thetrapped:eq:modulation leq L-1} and \fref{variete:modulation:eq:leqL-1 (triangle b)}:
$$
\begin{array}{r c l}
&\left| \frac{d}{ds}\left[ \frac{\int \chi_{B_0}\Lambda^{(1)}Q\left( \frac{\partial S_{L+2}}{\partial b_L}-\frac{\partial \hat{S}'_{L+2}}{\partial b_L} \right)_{L-1}}{\Bigl\langle  \chi_{B_0}\Lambda^{(1)}Q,\Lambda^{(1)}Q+(-1)^{\frac{L-1}{2}}\left(\frac{\partial S_{L+2}^{(2)}}{\partial b_L}\right)_{L-1} \Bigr\rangle}\right]\right| \\
&\leq Cb_1^{g'+1}\left(\underset{1\leq i \leq L}{\text{sup}}b_1^{-i}|\triangle \hat{b}_i|+\sqrt{\triangle_r \hat{\mathcal{E}}_{s_L}}\right).
\end{array}
$$
Integrating by parts then yields:
$$
\begin{array}{r c l}
&\frac{\hat{b}'_{L,s}\int \chi_{B_0}\Lambda^{(1)}Q\left( \frac{\partial S_{L+2}}{\partial b_L}-\frac{\partial \hat{S}'_{L+2}}{\partial b_L} \right)_{L-1}}{\Bigl\langle  \chi_{B_0}\Lambda^{(1)}Q,\Lambda^{(1)}Q+(-1)^{\frac{L-1}{2}}\left(\frac{\partial S_{L+2}^{(2)}}{\partial b_L}\right)_{L-1} \Bigr\rangle} +O[b_1^{L+g'+1 }(\underset{1\leq i \leq L}{\text{sup}}b_1^{-i}|\triangle \hat{b}_i|+\sqrt{\triangle_r \hat{\mathcal{E}}_{s_L}})]\\
=&\frac{d}{ds}\left( \hat{b}'_{L}\frac{\int \chi_{B_0}\Lambda^{(1)}Q\left( \frac{\partial S_{L+2}}{\partial b_L}-\frac{\partial \hat{S}'_{L+2}}{\partial b_L} \right)_{L-1}}{\Bigl\langle  \chi_{B_0}\Lambda^{(1)}Q,\Lambda^{(1)}Q+(-1)^{\frac{L-1}{2}}\left(\frac{\partial S_{L+2}^{(2)}}{\partial b_L}\right)_{L-1} \Bigr\rangle} \right).
\end{array}
$$
Injecting this last identity in \fref{presque} give the identity \fref{variete:eq:improved modulation} we had to prove. To finish, the gain when integrating is a consequence of \fref{variete:improvedmodulation:gain inte}, of the size of the denominator \fref{thetrapped:eq:improved modulation taille denominateur}, and of the asymptotic:
$$
\left( \frac{\partial S_{L+2}}{\partial b_L}-\frac{\partial \hat{S}'_{L+2}}{\partial b_L} \right)_{L-1}=O(y^{-\gamma-g'+2}b_1^2 \underset{1\leq i \leq 2}{\text{sup}}b_1^{-i}|\triangle \hat{b}|_i )
$$

\end{proof}


\subsubsection{Energy identities for the difference of errors}

In the previous section, the key norm of $\bos{\varepsilon}$ we had to control was the adapted high Sobolev norm $\mathcal{E}_{s_L}$. We recall the non linear tools we used to find a sufficient estimate: we control $\bos{\varepsilon}$ at another level of regularity to close the non linear term, integrate in time the modulation part that is not controlled directly, and derive a Morawetz type identity to manage a local term. Here we want to know how the time evolution of the adapted high Sobolev norm of the difference of the errors, $\bos{\varepsilon}-\hat{\bos{\varepsilon}}'$ depends on the differences of the parameters and itself, and will do it using the same non linear tools.\\
\\
We start with a technical lemma linking the difference of the profiles to the difference of the parameters.

\begin{lemma}[Bounds on the differences of profiles:]
\label{variete:lem:estimations differences de profils}
The following bounds hold:
\be \label{variete:eq:estimation low difference psib}
\parallel \tilde{\bos{\psi}}_b-\tilde{\bos{\psi}}_{\hat{b}'}  \parallel_{\dot{H}^{\sigma}\times \dot{H}^{\sigma-1}} \leq Cb_1^{\alpha+1+O(\sigma-s_c,\eta)}\left( \underset{1\leq i \leq L}{\text{sup}}b_1^{-i}|\triangle \hat{b}_i|+\sqrt{\triangle_r \hat{\mathcal{E}}_{s_L}  }\right),
\ee
\be \label{variete:eq:estimation high difference psib}
\begin{array}{r c l}
& \parallel (\tilde{\psi}_b^{(1)}-\tilde{\psi}_{\hat{b}'}^{(1)})_{s_L}  \parallel_{L^2}+\parallel (\tilde{\psi}_b^{(2)}-\tilde{\psi}_{\hat{b}'}^{(2)})_{s_L-1}  \parallel_{L^2}\\
\leq & Cb_1^{L+1+(1-\delta_0)(1+\eta)}\underset{1\leq i \leq L}{\text{sup}}(b_1^{-i}|\triangle \hat{b}_i|)+Cb_1^{2L+2-2\delta_0+O(\eta)}\sqrt{\triangle_r \hat{\mathcal{E}}_{s_L}}.
\end{array}
\ee
\end{lemma}

\begin{proof}[Proof of Lemma \fref{variete:lem:estimations differences de profils}]
We recall from  \fref{eq:def:psibtilde} the expression of the differences of the errors:
\be \label{variete:eq:low psi expression}
\begin{array}{ r c l}
\tilde{\bos{\psi}}_b-\tilde{\bos{\psi}}_{\hat{b}'}&=& \chi_{B_1}\bos{\psi}_b- \chi_{\hat{B}'_1}\bos{\psi}_{\hat{b}'} + \partial_s(\chi_{B_1})\bos{\alpha}_b-\partial_{\hat{s}'}(\chi_{B_1})\bos{\alpha}_{\hat{b}'}\\
&&+b_1(\bos{\Lambda} \tilde{\bos{Q}_b}-\chi_{B_1}\bos{\Lambda} \bos{Q}_b)-\hat{b}_1'(\bos{\Lambda} \tilde{\bos{Q}_{\hat{b}'}}-\chi_{B_1}\bos{\Lambda} \bos{Q}_{\hat{b}'})  \\
&& -(\bos{F}(\tilde{\bos{Q}}_b)-\bos{F}(\bos{Q})-\chi_{B_1}(\bos{F}(\bos{Q}_b)-\bos{F}(\bos{Q}))\\
&& +(\bos{F}(\tilde{\bos{Q}}_{\hat{b}'})-\bos{F}(\bos{Q})-\chi_{B_1}(\bos{F}(\bos{Q}_{\hat{b}'})-\bos{F}(\bos{Q})) ,
\end{array} 
\ee
We have to estimate everything in the right hand side. It always rely on finding the asymptotic of the profiles and relating it to the difference of the parameters. We will just do it for the first two terms: the same methodology giving the same results for the others. The first one is on the second coordinate and we decompose:
\be \label{variete:eq:low psi premier terme expression}
\chi_{B_1}\bos{\psi}_b- \chi_{\hat{B}'_1}\bos{\psi}_{\hat{b}'}=\begin{pmatrix}0\\ \chi_{B_1}(\psi_b-\psi_{\hat{b}'})+\psi_{\hat{b}'}(\chi_{B_1}-\chi_{\hat{B}'_1})\end{pmatrix}.
\ee
For the first term in \fref{variete:eq:low psi premier terme expression}, from the asymptotic  \fref{variete:eq:asymptotique diff psib leqB0} of $\psi_b-\psi_{\hat{b}'}$ we obtain:
\be \label{variete:eq:lowsobo psi premier terme 1}
\parallel \chi_{B_1}(\psi_b-\psi_{\hat{b}'}) \parallel_{\dot{H}^{\sigma-1}}\leq Cb_1^{\alpha+1+g'+O(\eta,\sigma-s_c)} \underset{1\leq i \leq L}{\text{sup}}(b_1^{-i}|\triangle \hat{b}_i|),
\ee
\be \label{variete:eq:highsobo psi premier terme 1}
\parallel (\chi_{B_1}(\psi_b-\psi_{\hat{b}'}))_{s_L-1} \parallel_{L^2}\leq Cb_1^{L+1+(1-\delta_0)+g'+O(\eta)} \underset{1\leq i \leq L}{\text{sup}}(b_1^{-i}|\triangle \hat{b}_i|),
\ee
We now turn to the second term in \fref{variete:eq:low psi premier terme expression}. The integral formula \fref{variete:eq:formulation integrale diff chi} for $\chi_{B_1}-\chi_{\hat{B}'_1}$ implies that $ \chi_{B_1}-\chi_{\hat{B}'_1}=(b_1^{1+\eta}-\hat{b}_1^{'(1+\eta)})f(y) $ with the function $f$ having its support in $[\text{min}(B_1,\hat{B}'_1),2\text{max}(B_1,\hat{B}'_1)]$, and satisfying: $\partial_y^k f=O(y^{1-k})$. As one has $|b_1^{1+\eta}-\hat{b}^{'(1+\eta)}|\leq C|b_1-\hat{b}_1'|$, using the previous result \fref{linearized:eq:global bound jleqL} we get:
\be \label{variete:eq:lowsobo psi premier terme 2}
\parallel \psi_{\hat{b}'}(\chi_{B_1}-\chi_{\hat{B}'_1}) \parallel_{\dot{H}^{\sigma-1}}\leq Cb_1^{\alpha+1+g'+O(\eta,\sigma-s_c)} b_1^{-1}|\triangle \hat{b}_1|.
\ee
\be \label{variete:eq:highsobo psi premier terme 2}
\parallel ( \psi_{\hat{b}'}(\chi_{B_1}-\chi_{\hat{B}'_1}) )_{s_L-1}\parallel_{L^2}\leq Cb_1^{L+1+(1-\delta_0)+g'+O(\eta,\sigma-s_c)} b_1^{-1}|\triangle \hat{b}_1|.
\ee
The decomposition \fref{variete:eq:low psi premier terme expression} and the bounds \fref{variete:eq:lowsobo psi premier terme 1}, \fref{variete:eq:lowsobo psi premier terme 2}, \fref{variete:eq:highsobo psi premier terme 1} and \fref{variete:eq:highsobo psi premier terme 2} imply for the following bounds for the first term in \fref{variete:eq:low psi expression}:
\be \label{variete:eq:low psi premier terme}
\parallel \chi_{B_1}\bos{\psi}_b- \chi_{\hat{B}'_1}\bos{\psi}_{\hat{b}'} \parallel_{\dot{H}^{\sigma}\times \dot{H}^{\sigma-1}} \leq Cb_1^{\alpha+1+g'+O(\eta,\sigma-s_c)} \underset{1\leq i \leq L}{\text{sup}}(b_1^{-i}|\triangle \hat{b}_i|).
\ee
\be \label{variete:eq:high psi premier terme}
\parallel (\chi_{B_1}\psi_b- \chi_{\hat{B}'_1}\psi_{\hat{b}'})_{s_L-1} \parallel_{L^2} \leq Cb_1^{L+1+(1-\delta_0)+g'+O(\eta,\sigma-s_c)} \underset{1\leq i \leq L}{\text{sup}}(b_1^{-i}|\triangle \hat{b}_i|).
\ee
We now turn to the second difference of terms in \fref{variete:eq:low psi expression}. We compute:
\be \label{variete:eq:low psi second terme expression}
\begin{array}{r c l}
&\partial_s(\chi_{B_1})\bos{\alpha}_b-\partial_{\hat{s}'}(\chi_{B_1})\bos{\alpha}_{\hat{b}'}:=y(1+\eta)(\bos{A}_1+\bos{A}_2+\bos{A}_3+\bos{A}_4)\\
=& y(1+\eta)\Bigl{[}(b_{1,s}-\hat{b}'_{1,s})b_1^{\eta}\partial_y\chi(yb_1^{1+\eta})\bos{\alpha}_b+\hat{b}'_{1,s}(b_1^{\eta}-\hat{b}_1^{'\eta})\partial_y\chi(yb_1^{1+\eta})\bos{\alpha}_b\\
&+\hat{b}'_{1,s}\hat{b}_1^{'\eta}(\partial_y\chi(yb_1^{1+\eta})-\partial_y \chi (y\hat{b}_1^{'(1+\eta)}))\bos{\alpha}_b+\hat{b}'_{1,s}\hat{b}_1^{'\eta}\partial_y\chi(yb_1^{'(1+\eta)})(\bos{\alpha}_b-\bos{\alpha}_{\hat{b}'})\Bigr{]}.
\end{array}
\ee
and will estimate everything in the right hand side. From the expressions \fref{variete:modulation:eq:leqL-1 (triangle b)} and \fref{variete:modulation:eq:dhats (triangle b)} for $b_{1,s}-\hat{b}_{1,s}'$ and $\frac{d\hat{s}'}{ds}-1$ we deduce that for the first term:
$$
|b_{1,s}-\hat{b}'_{1,s}|\leq Cb_1^2\underset{1\leq i \leq L}{\text{sup}}b_1^{-i}|\triangle \hat{b}_i|+b_1^{L+1+(1-\delta_0)(1+\frac{\eta}{2})}\sqrt{\triangle_r \hat{\mathcal{E}}_{s_L}}.
$$
For the second term one has $|b_{1,s}(b_1^{\eta}-\hat{b}_1^{'\eta})|\leq b_1^{2+\eta}|b_1-\hat{b}'_1|$. For the third term an integral formula similar to \fref{variete:eq:formulation integrale diff chi} holds, giving: 
$$
\partial_y^k (\partial_y\chi(yb_1^{1+\eta})-\partial_y \chi (y\hat{b}_1^{'(1+\eta)})=O\left( \frac{b_1^{\eta}|\triangle \hat{b}_1|}{1+y^{-1+k}}\right).
$$
Therefore we get for the first three terms in \fref{variete:eq:low psi second terme expression}:
\be \label{variete:eq:low psi second terme 1}
\begin{array}{r c l}
\parallel y(1+\eta)(\bos{A}_1+\bos{A}_2+\bos{A}_3) \parallel_{\dot{H}^{\sigma}\times \dot{H}^{\sigma-1}} &\leq& Cb_1^{\alpha+1+O(\eta,\sigma-s_c)} \underset{1\leq i \leq L}{\text{sup}}(b_1^{-i}|\triangle \hat{b}_i|)\\
&&+Cb_1^{L+1+\alpha+O(\eta,\sigma-s_c)}\sqrt{\triangle _r \hat{\mathcal{E}}_{s_L}},
\end{array}
\ee
\be \label{variete:eq:high psi second terme 1}
\begin{array}{r c l}
&\parallel (y(A_1^{(1)}+A_2^{(1)}+A_3^{(1)}))_{s_L} \parallel_{L^2}+ \parallel (y(A_1^{(2)}+A_2^{(2)}+A_3^{(2)}))_{s_L-1} \parallel_{L^2} \\
\leq& Cb_1^{L+1+(1-\delta_0)(1+\eta)} \underset{1\leq i \leq L}{\text{sup}}(b_1^{-i}|\triangle \hat{b}_i|)+Cb_1^{2L+2-2\delta_0+O(\eta)}\sqrt{\triangle _r \hat{\mathcal{E}}_{s_L}}.
\end{array}
\ee
We turn to the fourth term in \fref{variete:eq:low psi second terme expression}. One has: 
$$
\bos{\alpha}_b-\bos{\alpha}_{\hat{b}'}=\sum_{1}^L(b_i-\hat{b}'_i)\bos{T}_i+\sum_2^{L+2}\bos{S}_i-\hat{\bos{S}}_i'.
$$
The bound \fref{variete:eq:difference polynomes parametres}, the fact that the $\bos{S}_i$'s are homogeneous, using their asymptotic and the one of the $\bos{T}_i$'s yield:
\be \label{variete:eq:low psi second terme 2}
\begin{array}{r c l}
&\parallel y\hat{b}'_{1,s}\hat{b}_1^{'\eta}\partial_y\chi(\frac{y}{\hat{B}'_1})(\bos{\alpha}_b-\bos{\alpha}_{\hat{b}'}) \parallel_{\dot{H}^{\sigma}\times \dot{H}^{\sigma-1}} \leq& Cb_1^{\alpha+1+O(\eta,\sigma-s_c)} \underset{1\leq i \leq L}{\text{sup}}(b_1^{-i}|\triangle \hat{b}_i|)\\
\end{array}
\ee
\be \label{variete:eq:high psi second terme 2}
\begin{array}{r c l}
&\parallel (y\hat{b}'_{1,s}\hat{b}_1^{'\eta}\partial_y\chi(\frac{y}{\hat{B}'_1})(\alpha^{(1)}_b-\alpha^{(1)}_{\hat{b}'})_{s_L} \parallel_{L^2}+ \parallel (y\hat{b}'_{1,s}\hat{b}_1^{'\eta}\partial_y\chi(\frac{y}{\hat{B}'_1})(\alpha^{(2)}_b-\alpha^{(2)}_{\hat{b}'})_{s_L-1} \parallel_{L^2} \\
\leq& Cb_1^{L+1+(1-\delta_0)(1+\eta)} \underset{1\leq i \leq L}{\text{sup}}(b_1^{-i}|\triangle \hat{b}_i|).
\end{array}
\ee
because $|\hat{b}_{1,s}'\hat{b}_1^{'\eta}|\leq Cb_1^{2+\eta}$. We collect the bounds \fref{variete:eq:low psi second terme 1}, \fref{variete:eq:low psi second terme 2}, \fref{variete:eq:high psi second terme 1} and \fref{variete:eq:high psi second terme 2} to find that for the second term in \fref{variete:eq:low psi expression}:
\be \label{variete:eq:low psi second terme}
\begin{array}{r c l}
\parallel \partial_s(\chi_{B_1})\bos{\alpha}_b-\partial_{\hat{s}'}(\chi_{B_1})\bos{\alpha}_{\hat{b}'} \parallel_{\dot{H}^{\sigma}\times \dot{H}^{\sigma-1}} &\leq &Cb_1^{\alpha+1+O(\eta,\sigma-s_c)} \underset{1\leq i \leq L}{\text{sup}}(b_1^{-i}|\triangle \hat{b}_i|)\\
&&+b_1^{L+1+\alpha+O(\eta,\sigma-s_c)}\sqrt{\triangle_r\hat{\mathcal{E}}_{s_L}},
\end{array}
\ee
\be \label{variete:eq:high psi second terme}
\begin{array}{r c l}
&\parallel (\partial_s(\chi_{B_1})\bos{\alpha}_b^{(1)}-\partial_{\hat{s}'}(\chi_{B_1})\bos{\alpha}_{\hat{b}'}^{(1)})_{s_L} \parallel_{L^2}+\parallel \partial_s(\chi_{B_1})\bos{\alpha}_b^{(2)}-\partial_{\hat{s}'}(\chi_{B_1})\bos{\alpha}_{\hat{b}'}^{(2)})_{s_L-1} \parallel_{L^2}\\
\leq& Cb_1^{L+1+(1-\delta_0)(1+\eta)} \underset{1\leq i \leq L}{\text{sup}}(b_1^{-i}|\triangle \hat{b}_i|)+Cb_1^{2L+2-2\delta_0+O(\eta)}\sqrt{\triangle_r \hat{\mathcal{E}}_{s_L}}.
\end{array}
\ee
We claim that the bounds \fref{variete:eq:low psi second terme} and \fref{variete:eq:high psi second terme} also holds for the last two differences of profiles in \fref{variete:eq:low psi expression} and that they can be proven using verbatim the same tools we employed so far. This fact give us the bounds for the remaining terms in \fref{variete:eq:low psi expression}, which combined with the previous estimates for the first two terms \fref{variete:eq:low psi premier terme}, \fref{variete:eq:high psi premier terme}, \fref{variete:eq:low psi second terme} and \fref{variete:eq:high psi second terme} proves the two estimates \fref{variete:eq:estimation low difference psib} and \fref{variete:eq:estimation high difference psib} of the lemma.

\end{proof}


We state now how the time evolution of the low Sobolev norm of the difference of the errors $\bos{\varepsilon}-\hat{\bos{\varepsilon}}'$ is influenced by itself and the difference between the parameters and the renormalized times. It is the analogue of Proposition \fref{trappedregime:pr:low sobo}.

\begin{lemma}\label{variete:lem:low sobo}
\emph{(Time evolution of the low Sobolev norm of $\bos{\varepsilon}-\hat{\bos{\varepsilon}}'$)}. We keep the assumptions and notations of Proposition \ref{variete:prop:parametres lipschitz}. There holds:
\be \label{variete:eq:lowsobo}
\begin{array}{r c l}
\frac{d}{dt}\left\{ \frac{\triangle \mathcal{E}_{\sigma}}{\lambda^{2(\sigma-s_c)}} \right\} \leq \frac{Cb_1^{1+2(\sigma-s_c)(1+\nu)+\frac{\alpha}{2L}}}{\lambda^{2(\sigma-s_c)+1}} \left( \triangle_r \hat{\mathcal{E}}_{\sigma}+\triangle _r \hat{\mathcal{E}}_{s_L}+ (\underset{1\leq i \leq L}{\text{sup}}b_1^{-i}|\triangle \hat{b}_i|)^2\right)
\end{array}
\ee
(the norm $\triangle \mathcal{E}_{\sigma}$ is defined in  \fref{variete:eq:def norme low}, the renormalized norms $\triangle_r \hat{\mathcal{E}}_{\sigma}$ and $\triangle _r \hat{\mathcal{E}}_{s_L}$ are defined in  \fref{variete:eq:def normes renormalisees}).
\end{lemma}

\begin{proof}[Proof of Lemma \ref{variete:lem:low sobo}]
We start by computing the following identity:
\be \label{variete:lowsobo:eq:expression derivee}
\begin{array}{r c l}
&\frac{d}{dt}\left\{ \frac{\triangle \hat{\mathcal{E}}_{\sigma}}{\lambda^{2(\sigma-s_c)}} \right\} \\
=& \int \nabla^{\sigma}(w^{(1)}-\hat{w}^{'(1)}).\nabla^{\sigma}(w^{(2)}-\hat{w}^{'(2)}+\frac{(\hat{Mod}(t)^{'(1)}-\tilde{Mod}(t)^{(1)}+\tilde{\psi}_{\hat{b}'}^{(1)}-\tilde{\psi}_b^{(1)})_{\frac{1}{\lambda}}}{\lambda} \\
&+ (1-\frac{d\hat{s}'}{ds})\int \nabla^{\sigma}(w^{(1)}-\hat{w}^{'(1)}).\nabla^{\sigma}(-\hat{w}^{'(1)}-\frac{1}{\lambda}\tilde{\psi}_{\hat{b}',\frac{1}{\lambda}}^{(2)}) \\
& + \int \nabla^{\sigma-1}(w^{(2)}-\hat{w}^{'(2)}).\nabla^{\sigma-1}\bigl{(}\mathcal{L}(\hat{w}^{'(1)}-w^{(1)})+\frac{(\hat{Mod}'(t)^{(2)}-\tilde{Mod}(t)^{(2)}+\tilde{\psi}_{\hat{b}'}^{(2)}-\tilde{\psi_b}^{(2)})_{\frac{1}{\lambda}}}{\lambda}\\
&+NL-\hat{NL}'+L-\hat{L}'+(1-\frac{d\hat{s}'}{ds})(-\mathcal{L}\hat{w}^{'(1)}-\frac{1}{\lambda}\tilde{\psi}_{\hat{b}',\frac{1}{\lambda}}^{(2)}+\hat{NL}'+\hat{L}') \bigr{)}.
\end{array}
\ee
We now compute the size of every term in the right hand side of equation \fref{variete:lowsobo:eq:expression derivee}.\\
$\bullet$ \emph{Linear terms:} The norm studied here being adapted to a wave equation:
$$
\begin{array}{r c l}
& \int \nabla^{\sigma}(w^{(1)}-\hat{w}^{'(1)}).\nabla^{\sigma}(w^{(2)}-\hat{w}^{'(2)})+ \nabla^{\sigma-1}(w^{(2)}-\hat{w}^{'(2)}).\nabla^{\sigma-1}\mathcal{L} (w^{(1)}-\hat{w}^{'(1)})  \\
=& \int \nabla^{\sigma-1}(w^{(2)}-\hat{w}^{'(2)}). \nabla^{\sigma-1}(pQ_{\frac{1}{\lambda}}^{p-1}(w^{(1)}-\hat{w}^{'(1)})) \\
\leq & O( \parallel \nabla^{\sigma}(w^{(2)}-\hat{w}^{'(2)}) \parallel_{L^2} \parallel \nabla^{\sigma-2}(Q_{\frac{1}{\lambda}}^{p-1}(w^{(1)}-\hat{w}^{'(1)})) \parallel_{L^2}) .
\end{array}
$$
We recall the asymptotic $Q^{p-1}\sim\frac{c}{x^{2}}$ ($c>0$). Using the weighted Hardy estimate from Lemma \ref{annexe:lem: hardy frac a poids} one has for the second term:
$$
\parallel \nabla^{\sigma-2}(Q_{\frac{1}{\lambda}}^{p-1}(w^{(1)}-\hat{w}^{'(1)})) \parallel_{L^2} \leq \frac{C}{\lambda^{\sigma-s_c}} \parallel \nabla^{\sigma} (\varepsilon^{(1)}-\hat{\varepsilon}^{'(1)})\parallel_{L^2} = C \frac{\sqrt{\triangle \mathcal{E}_{\sigma}}}{\lambda^{\sigma-s_c}} .
$$
By interpolation, we get for the other term:
$$
\parallel \nabla^{\sigma} (w^{(2)}-\hat{w}^{'(2)}) \parallel_{L^2}\leq \frac{C}{\lambda^{\sigma-s_c+1}} \sqrt{\triangle \mathcal{E}_{\sigma}}^{1-\frac{1}{s_L-\sigma}}\sqrt{ \triangle \mathcal{E}_{s_L}}^{\frac{1}{s_L-\sigma}} .
$$
Using the definition of the renormalized norms of the difference \fref{variete:eq:def normes renormalisees} and the fact that $\frac{L+(1-\delta_0)(1+\eta)-(\sigma-s_c)(1+\nu)}{s_L-\sigma}=1+\frac{\alpha}{L}+O((\sigma-s_c)L^{-1},\eta L^{-1},L^{-2}) $ we conclude:
\be \label{variete:lowsobo:eq:estimation lineaire}
\begin{array}{r c l}
&\left| \int \nabla^{\sigma}(w^{(1)}-\hat{w}^{'(1)}.\nabla^{\sigma}(w^{(2)}-\hat{w}^{'(2)})- \nabla^{\sigma-1}w^{(2)}.\nabla^{\sigma-1}\mathcal{L} (w^{(1)}-\hat{w}^{'(1)})\right|\\
\leq & \frac{Cb_1^{2(\sigma-s_c)(1+\nu)+1+\frac{\alpha}{L}+O\left(\frac{\sigma-s_c}{L},\frac{\eta}{L},\frac{1}{L^2} \right)}}{\lambda^{2(\sigma-s_c)}+1}\sqrt{\triangle_r \mathcal{E}_{\sigma}}^{2-\frac{1}{s_L-\sigma}}\sqrt{\triangle_r \mathcal{E}_{s_L}}^{\frac{1}{s_L-\sigma}} .
\end{array}
\ee
$\bullet$ \emph{$\tilde{\bos{Mod}}(t)$ terms:} We only compute for the $\tilde{Mod}^{(2)}$ terms, the calculation being the same for the first coordinate. Rescaling, using Cauchy-Schwarz and the notations \fref{variete:eq:def triangle hatmodi} and \fref{variete:eq:def triangle hatmod0}:
\be \label{variete:lowsobo:eq:mod cs}
\begin{array}{r c l}
&\left| \frac{1}{\lambda}\int \nabla^{\sigma-1}(w^{(2)}-\hat{w}^{'(2)}).\nabla^{\sigma-1} (\tilde{Mod}^{(2)}-\hat{Mod}^{'(2)})_{\frac{1}{\lambda}}\right| \\
\leq& \frac{1}{\lambda^{2(\sigma-s_c)+1}}\sqrt{\triangle \mathcal{E}_{\sigma}} \left( \sum_{i=0}^L \parallel \nabla^{\sigma-1}\triangle \hat{Mod}_i^{(2)}\parallel_{L^2} \right).
\end{array}
\ee
We will just compute a bound for the last term: $\triangle \hat{Mod}_L^{(2)}$. Indeed it is for this one that we have the worst bound, see Lemma \ref{variete:lem:modulation diff}. We first split:
\be \label{variete:lowsobo:eq:estimation mod intermediaire}
\begin{array}{r c l}
\triangle \hat{Mod}_L^{(2)}&=& (b_{L,s}+(L-\alpha)b_1b_i-(\hat{b}'_{L,s}+\frac{d\hat{s}'}{ds}(L-\alpha)\hat{b}'_1\hat{b}'_i))\chi_{B_1}(T_L+\frac{\partial S_{L+2}}{\partial b_L})\\
&&+(\hat{b}'_{L,s}+\frac{d\hat{s}'}{ds}(L-\alpha)\hat{b}'_1\hat{b}'_i)(\chi_{B_1}(T_L+\frac{\partial S_{L+2}}{\partial b_L})-\chi_{\hat{B}'_1}(T_L+\frac{\partial \hat{S}'_{L+2}}{\partial b_L})).
\end{array}
\ee
For the first term, the bound \fref{variete:modulation:eq:L (triangle b)} derived in the previous Lemma \ref{variete:lem:modulation diff} implies:
\be \label{variete:lowsobo:eq:estimation mod intermediaire 2}
\begin{array}{r c l}
&\parallel \nabla^{\sigma-1}(b_{L,s}+(L-\alpha)b_1b_i-(\hat{b}'_{L,s}+\frac{d\hat{s}'}{ds}(L-\alpha)\hat{b}'_1\hat{b}'_i))\chi_{B_1}(T_L+\frac{\partial S_{L+2}}{\partial b_L}) \parallel_{L^2}\\
\leq& Cb_1^{\alpha+1-\delta_0+O(\eta,\sigma-s_c)}(\sqrt{\triangle_r \hat{\mathcal{E}}_{s_L}}+\underset{1\leq i \leq L}{\text{sup}}(b_1^{-i}|\triangle \hat{b}_i|)).
\end{array}
\ee
We now want to estimate the second term in \fref{variete:lowsobo:eq:estimation mod intermediaire}. We decompose:
$$
\begin{array}{r c l}
\chi_{B_1}(T_L+\frac{\partial S_{L+2}}{\partial b_L})-\chi_{\hat{B}'_1}(T_L+\frac{\partial \hat{S}'_{L+2}}{\partial b_L})&=&(\chi_{B_1}-\chi_{\hat{B}_1'})(T_L+\frac{\partial S_{L+2}}{\partial b_L})\\
&&+\chi_{B_1}(\frac{\partial S_{L+2}}{\partial b_L}-\frac{\partial S_{L+2}}{\partial b_L}).
\end{array}
$$
The identity \fref{variete:eq:formulation integrale diff chi} gives that $\chi_{B_1}(y)-\chi_{\hat{B}_1'}(y)=(b_1^{1+\eta}-\hat{b}_1^{'(1+\eta)} )f_{b_1,\hat{b}_1'}(y)$ with $f_{b_1,\hat{b}_1'}(y)$ being a $C^{\infty}$ function with support in $[\text{min}(B_1,\hat{B}'_1),2\text{max}(B_1,\hat{B}'_1)]$ satisfying: $|\partial_y^k f_{b_1,\hat{b}'_1}|\leq Cb_1^{k-1+C_k\eta}$. We recall the meaning of our notation:
$$
\hat{b}'_{L,s}+\frac{d\hat{s}'}{ds}(L-\alpha)\hat{b}'_1\hat{b}'_i=\frac{d\hat{s}'}{ds}\left (\frac{d}{d\hat{s}'}b'_L(\hat{s}')+(L-\alpha)b_1'(\hat{s}')b_L'(\hat{s}') \right).
$$
From the bound on the modulation \fref{thetrapped:eq:premiere modulation L}, and from $|b_1^{1+\eta}-\hat{b}_1^{'(1+\eta)}|\leq C|b_1-\hat{b}'_1|$ one gets using the asymptotic of $T_L$ and $\frac{\partial S_{L+2}}{\partial b_L}$:
$$
\begin{array}{r c l}
&\parallel \nabla^{\sigma-1}(\hat{b}'_{L,s}+\frac{d\hat{s}'}{ds}(L-\alpha)\hat{b}'_1\hat{b}'_i)(\chi_{B_1}-\chi_{\hat{B}_1'})(T_L+\frac{\partial S_{L+2}}{\partial b_L}) \parallel_{L^2}  \\
\leq& Cb_1^{\alpha+1-\delta_0+O(\eta,\sigma-s_c)-1}|b_1-\hat{b}'_1|.
\end{array}
$$
For the second part, using again the bound \fref{thetrapped:eq:premiere modulation L}, the fact that $\frac{\bos{S}_{L+2}}{\partial b_L}$ is homogeneous of degree $(L+2,L+2-g',1,2)$ and the bound \fref{variete:eq:difference polynomes parametres}:
$$
\begin{array}{r c l}
&\parallel \nabla^{\sigma-1}(\hat{b}'_{L,s}+\frac{d\hat{s}'}{ds}(L-\alpha)\hat{b}'_1\hat{b}'_i)\chi_{B_1}(\frac{\partial S_{L+2}}{\partial b_L}-\frac{\partial S_{L+2}}{\partial b_L}) \parallel_{L^2}\\
\leq& Cb_1^{\alpha+1-\delta_0+O(\eta,\sigma-s_c)}\underset{1\leq i \leq L}{\text{sup}}(b_1^{-i}|b_i-\hat{b}'_i)|.
\end{array}
$$
Eventually we have found, gathering the two previous bounds:
\be \label{variete:lowsobo:eq:estimation mod intermediaire 3}
\begin{array}{r c l}
&\parallel \nabla^{\sigma-1}(\hat{b}'_{L,s}+\frac{d\hat{s}'}{ds}(L-\alpha)\hat{b}'_1\hat{b}'_i)(\chi_{B_1}(T_L+\frac{\partial S_{L+2}}{\partial b_L})-\chi_{\hat{B}'_1}(T_L+\frac{\partial \hat{S}'_{L+2}}{\partial b_L})) \parallel_{L^2}\\
\leq& Cb_1^{\alpha+1-\delta_0+O(\eta,\sigma-s_c)}\underset{1\leq i \leq L}{\text{sup}}(b_1^{-i}|\triangle \hat{b}_i|)).
\end{array}
\ee
We can now go back to \fref{variete:lowsobo:eq:estimation mod intermediaire} and inject the bounds \fref{variete:lowsobo:eq:estimation mod intermediaire 2} and \fref{variete:lowsobo:eq:estimation mod intermediaire 3} for the terms in the right hand side. This gives for the $L$-th modulation term:
\be \label{variete:lowsobo:eq:estimation mod L}
\begin{array}{r c l}
\parallel \nabla^{\sigma-1} \triangle \hat{Mod}^{(2)}_L \parallel_{L^2} \leq Cb_1^{\alpha+1-\delta_0+O(\eta,\sigma-s_c)}(\underset{1\leq i \leq L}{\text{sup}}(b_1^{-i}|\triangle \hat{b}_i|)+\sqrt{\triangle_r \hat{\mathcal{E}}_{s_L}}).
\end{array}
\ee
The primary modulation bounds for the evolution of $b_L$ and $\triangle \hat{b}_L$ being worst than the ones for $b_i$ and $\triangle \hat{b}_i$ (compare \fref{thetrapped:eq:modulation leq L-1} and \fref{thetrapped:eq:premiere modulation L}, \fref{variete:modulation:eq:leqL-1} and \fref{variete:modulation:eq:L}) we claim that a better estimate than \fref{variete:lowsobo:eq:estimation mod L} also holds for the other terms in \fref{variete:lowsobo:eq:mod cs} and that it also work for the first coordinate, yielding when injected in \fref{variete:lowsobo:eq:mod cs}:
\be \label{variete:lowsobo:eq:estimation mod}
\begin{array}{r c l}
&\bigl{|} \frac{1}{\lambda}\int \nabla^{\sigma-1}(w^{(2)}-\hat{w}^{'(2)}).\nabla^{\sigma-1} (\tilde{Mod}^{(2)}-\hat{Mod}^{'(2)})_{\frac{1}{\lambda}}\\
&+\frac{1}{\lambda}\int\nabla^{\sigma}(w^{(1)}-\hat{w}^{'(1)}).\nabla^{\sigma} (\tilde{Mod}^{(1)}-\hat{Mod}^{'(1)})_{\frac{1}{\lambda}} \bigr{|} \\
\leq& C\frac{b_1^{1+\alpha-\delta_0+O(\eta,\sigma-s_c)}}{\lambda^{2(\sigma-s_c)+1}}\sqrt{\triangle_r \mathcal{E}_{\sigma}} (\underset{1\leq i \leq L}{\text{sup}}(b_1^{-i}|\triangle \hat{b}_i|)+\sqrt{\triangle_r \hat{\mathcal{E}}_{s_L}}),\\
\end{array}
\ee
and we recall that $\alpha-\delta_0>2-\delta_0>1$.\\
$\bullet$ \emph{$\tilde{\bos{\psi}}_b$ term:} We use the bound \fref{variete:eq:estimation low difference psib} on $\tilde{\bos{\psi}}_b-\tilde{\bos{\psi}}_{\hat{b}'}$:
\be \label{variete:lowsobo:eq:estimation psib}
\begin{array}{r c l}
& \left| \frac{1}{\lambda}\int \nabla^{\sigma}(w^{(1)}-\hat{w}^{'(1)}).\nabla^{\sigma}(\tilde{\psi}_{b}-\tilde{\psi}_{\hat{b}'})_{\frac{1}{\lambda}}^{(1)} +\nabla^{\sigma-1}(w^{(2)}-\hat{w}^{'(2)}).\nabla^{\sigma-1}(\tilde{\psi}_{b}-\tilde{\psi}_{\hat{b}'})_{\frac{1}{\lambda}}^{(2)} \right| \\
\leq &C \frac{b_1^{1+\alpha+O(\eta,\sigma-s_c)}}{\lambda^{2(\sigma-s_c)+1}}\sqrt{\triangle_r \mathcal{E}_{\sigma}} (\underset{1\leq i \leq L}{\text{sup}}(b_1^{-i}|\triangle \hat{b}_i|+\sqrt{\triangle_r \hat{\mathcal{E}}_{s_L}}), \\
\end{array}
\ee
and we recall that $\alpha>2$.\\
$\bullet$ \emph{$L(w)$ term:} We compute the following identity:
$$
\begin{array}{r c l}
L-\hat{L}'&=&(p-1)(Q^{p-1}-\tilde{Q}_b^{p-1})\varepsilon^{(1)}-(p-1)(Q^{p-1}-\tilde{Q}_{\hat{b}'}^{p-1})\hat{\varepsilon}^{'(1)} \\
&=&(p-1)(Q-\tilde{Q}_b^{p-1})(\varepsilon^{(1)}-\hat{\varepsilon}^{'(1)})+(p-1)(\tilde{Q}_{\hat{b}'}^{p-1}-\tilde{Q}_b^{p-1})\hat{\varepsilon}^{'(1)}.
\end{array}
$$
We recall that thanks to the asymptotic \fref{thetrapped=lowsobo:eq:asymptotique L} and to the fractional Hardy inequality one has for the first term:
$$
\begin{array}{l l l l}
&\parallel \nabla^{\sigma-1}(Q-\tilde{Q}_b^{p-1})(\varepsilon^{(1)}-\hat{\varepsilon}^{'(1)})\parallel_{L^2}\leq Cb_1\parallel \nabla^{\sigma+\frac{1}{p-1}}(\varepsilon^{(1)}-\hat{\varepsilon}^{'(1)})  \parallel_{L^2}\\
\leq&Cb_1^{1+\frac{1}{p-1}+O(L^{-1},\eta,\sigma-s_c)}\sqrt{\triangle _r \hat{\mathcal{E}}_{\sigma}}^{1-\frac{1}{(p-1)(s_L-\sigma)}}\sqrt{\triangle _r \hat{\mathcal{E}}_{s_L}}^{\frac{1}{(p-1)(s_L-\sigma)}}.
\end{array}
$$
For the second one, the bound \fref{variete:eq:asymptotique diff Qp-j} on the asymptotic of $\tilde{Q}_{\hat{b}'}^{p-1}-\tilde{Q}_b^{p-1}$ and the Hardy inequality yield:
$$
\parallel \nabla^{\sigma-1}(\tilde{Q}_{\hat{b}'}^{p-1}-\tilde{Q}_b^{p-1})\hat{\varepsilon}^{'(1)} \parallel_{L^2}\leq Cb_1^{1+\frac{1}{p-1}+O(L^{-1},\eta,\sigma-s_c)}\underset{1\leq i \leq L}{\text{sup}}(b_1^{-i}|\triangle \hat{b}_i|).
$$
Therefore we end up with the following bound on the small linear term:
\be \label{variete:lowsobo:eq:estimation L}
\begin{array}{r c l}
&\left| \int \nabla^{\sigma-1}(w^{(2)}-\hat{w}^{'(2)}).\nabla^{\sigma-1}(L-\hat{L}')\right|\leq \frac{C\parallel w^{(2)}-\hat{w}^{'(2)} \parallel_{\dot{H}^{\sigma-1}} \parallel L-\hat{L}' \parallel_{\dot{H}^{\sigma-1}}}{\lambda^{2(\sigma-s_c)+1}}  \\
\leq &   C\frac{b_1^{1+\frac{1}{p-1}+O(L^{-1},\sigma-s_c,\eta)}\sqrt{\triangle_r \hat{\mathcal{E}}_{\sigma}}(\underset{1\leq i \leq L}{\text{sup}}(b_1^{-i}|\triangle \hat{b}_i|+\sqrt{\triangle _r \hat{\mathcal{E}}_{\sigma}}^{1-\frac{1}{(p-1)(s_L-\sigma)}}\sqrt{\triangle _r \hat{\mathcal{E}}_{s_L}}^{\frac{1}{(p-1)(s_L-\sigma)}})}{\lambda^{2(\sigma-s_c)+1}} .
\end{array}
\ee
$\bullet$ \emph{$NL$ term:} The difference of the non linear terms is\footnote{we make here the abuse of notation $\tilde{Q}_b^{p-j}=\tilde{Q}_b^{(1)(p-j)}$ to ease notations.}:
\be \label{variete:lowsobo:eq:estimation NL expression}
\begin{array}{r c l}
NL-\hat{NL}'&=&\sum_{j=2}^{p}C_j\tilde{Q}_{b}^{p-j}\varepsilon^{(1)j}-\sum_{j=2}^{p}C_j\tilde{Q}_{\hat{b}'}^{p-j}\hat{\varepsilon}^{'(1)j} \\
&=& \sum_{j=2}^{p}C_j\tilde{Q}_{b}^{p-j}(\varepsilon^{(1)j}-\hat{\varepsilon}^{'(1)j})+\sum_{j=2}^{p}C_j(\tilde{Q}_{b}^{p-j}-\tilde{Q}_{\hat{b}'}^{p-j})\hat{\varepsilon}^{'(1)j}.
\end{array}
\ee
for some coefficients $C_j$ appearing when developing the polynomial $(X+Y)^p$. We start with the second term of this identity, assuming $j\neq p$. We now recall the bound \fref{thetrapped:lowsobo:eq:estimation NL 2} we found for the non-linear term in the proof of Proposition \fref{trappedregime:pr:low sobo}:
$$
\left\Vert \nabla^{\sigma-2+(j-1)(\sigma-s_c)}(v(y)\varepsilon^{'(1)j}) \right\Vert_{L^2} \leq C \sqrt{\hat{\mathcal{E}}'_{\sigma}}^j,
$$
for potentials $v$ satisfying $\partial_y^k v=O\left(\frac{1}{1+y^{2\frac{p-j}{p-1}+k}}  \right)$. Here, thanks to the asymptotic \fref{variete:eq:asymptotique diff Qp-j}, the potential is even better because of an extra gain $y^{-\alpha}$, therefore:
$$
\left\Vert \nabla^{\sigma-2+(j-1)(\sigma-s_c)}((\tilde{Q}_{\hat{b}'}^{p-j}-\tilde{Q}_b^{p-j})\hat{\varepsilon}^{'(1)j}) \right\Vert_{L^2} \leq  C\sqrt{\hat{\mathcal{E}}'_{\sigma}}^j \underset{1\leq i \leq L}{\text{sup}}(b_1^{-i}|\triangle \hat{b}_i|).
$$
This last bound imply that, integrating by part, for the second term in  \fref{variete:lowsobo:eq:estimation NL expression}:
\be \label{variete:lowsobo:eq:estimation NL 1}
\begin{array}{r c l}
&\left| \int \nabla^{\sigma-1}(w^{(1)}-\hat{w}^{'(1)}).\nabla^{\sigma-1}((\tilde{Q}_{b}^{p-j}-\tilde{Q}_{\hat{b}'}^{p-j})\hat{\varepsilon}^{'(1)j} ) \right| \\
\leq & \frac{C}{\lambda^{2(\sigma-s_c)+1}}\parallel \varepsilon^{(1)}-\hat{\varepsilon}^{'(1)}\parallel_{\dot{H}^{\sigma-(j-1)(\sigma-s_c)}}\left\Vert (\tilde{Q}_{\hat{b}'}^{p-j}-\tilde{Q}_b^{p-j})\hat{\varepsilon}^{'(1)j} \right\Vert_{\dot{H}^{\sigma-2+(j-1)(\sigma-s_c)}}\\
\leq & \frac{Cb_1^{2(\sigma-s_c)(1+\nu)+1+\frac{\alpha}{L}+O\left(\frac{\sigma-s_c}{L} \right)}\underset{1\leq i \leq L}{\text{sup}}(b_1^{-i}|\triangle \hat{b}_i|)\sqrt{\triangle _r \hat{\mathcal{E}}_{\sigma}}^{1-\frac{1-(j-1)(\sigma-s_c)}{s_L-\sigma}}\sqrt{\triangle _r \hat{\mathcal{E}}_{s_L}}^{\frac{1-(j-1)(\sigma-s_c)}{s_L-\sigma}}}{\lambda^{2(\sigma-s_c)+1}} .\\
\end{array}
\ee
We now turn to the first term in \fref{variete:lowsobo:eq:estimation NL expression}. We factorize the non linear term:
$$
(\varepsilon^{(1)j}-\hat{\varepsilon}^{'(1)j})=(\varepsilon^{(1)}-\hat{\varepsilon}^{'(1)})\sum_{i=0}^{j-1}C_i \varepsilon^{(1)i} (\hat{\varepsilon}^{'(1)})^{j-1-i},
$$
for some coefficients $(C_i)_{0\leq i \leq j-1}$. We can then apply the same reasoning we used in the proof of the bound \fref{thetrapped:lowsobo:eq:estimation NL 2}, giving this time:
$$
\left\Vert \nabla^{\sigma-2+(j-1)(\sigma-s_c)}(\tilde{Q}_{b}^{p-j}(\varepsilon^{(1)j}-\hat{\varepsilon}^{'(1)j})) \right\Vert_{L^2} \leq  C \sqrt{\triangle \hat{\mathcal{E}}_{\sigma}} \sum_{i=0}^{j-1} \sqrt{\mathcal{E}_{\sigma}}^{i} \sqrt{\hat{\mathcal{E}}'_{\sigma}}^{j-1-i} .
$$
As we did previously for the second term in \fref{variete:lowsobo:eq:estimation NL expression}, we now use interpolation and inject the bootstrap bounds \fref{eq:bootstrap estimations sur epsilon} to find:
\be \label{variete:lowsobo:eq:estimation NL 2}
\begin{array}{r c l}
&\left| \int \nabla^{\sigma-1}(w^{(1)}-\hat{w}^{'(1)}).\nabla^{\sigma-1}(\tilde{Q}_{b}^{p-j}(\varepsilon^{(1)j}- \hat{\varepsilon}^{'(1)j} )) \right| \\
\leq & \frac{C}{\lambda^{2(\sigma-s_c)+1}}\parallel \varepsilon^{(1)}-\hat{\varepsilon}^{'(1)} \parallel_{\dot{H}^{\sigma-(j-1)(\sigma-s_c)}}\left\Vert \tilde{Q}_b^{p-j}(\varepsilon^{(1)j}-\hat{\varepsilon}^{'(1)j}) \right\Vert_{\dot{H}^{\sigma-2+(j-1)(\sigma-s_c)}}\\
\leq & \frac{Cb_1^{2(\sigma-s_c)(1+\nu)+1+\frac{\alpha}{L}+O\left(\frac{\sigma-s_c}{L} \right)}}{\lambda^{2(\sigma-s_c)+1}} \sqrt{\triangle _r \hat{\mathcal{E}}_{\sigma}}^{2-\frac{1-(j-1)(\sigma-s_c)}{s_L-\sigma}}\sqrt{\triangle _r \hat{\mathcal{E}}_{s_L}}^{\frac{1-(j-1)(\sigma-s_c)}{s_L-\sigma}}.\\
\end{array}
\ee
In \fref{variete:lowsobo:eq:estimation NL 1} and \fref{variete:lowsobo:eq:estimation NL 2} we have found an estimate for the two terms in the right hand side of \fref{variete:lowsobo:eq:estimation NL expression}, giving the following bound for the non linear terms contribution:
\be \label{variete:lowsobo:eq:estimation NL}
\begin{array}{r c l}
&\left| \int \nabla^{\sigma-1}(w^{(1)}-\hat{w}^{'(1)}).\nabla^{\sigma-1}(NL-\hat{NL}') \right| \\
\leq & \frac{Cb_1^{2(\sigma-s_c)(1+\nu)+1+\frac{\alpha}{L}+O\left(\frac{\sigma-s_c}{L} \right)}}{\lambda^{2(\sigma-s_c)+1}} \sqrt{\triangle _r \hat{\mathcal{E}}_{\sigma}}^{1-\frac{1-(j-1)(\sigma-s_c)}{s_L-\sigma}}\sqrt{\triangle _r \hat{\mathcal{E}}_{s_L}}^{\frac{1-(j-1)(\sigma-s_c)}{s_L-\sigma}} \\
&\times (\sqrt{\triangle _r \hat{\mathcal{E}}_{\sigma}}+\underset{1\leq i \leq L}{\text{sup}}(b_1^{-i}|\triangle \hat{b}_i|)).
\end{array}
\ee
$\bullet$ \emph{The time difference terms:} We now look for a bound for the terms involving $\frac{d\hat{s}'}{ds}-1$ in \fref{variete:lowsobo:eq:expression derivee}. We have already computed the size of most of the terms in \fref{thetrapped:lowsobo:eq:estimation psib}, \fref{thetrapped:lowsobo:eq:estimation L} and \fref{thetrapped:lowsobo:eq:estimation NL}, yielding:
$$
\begin{array}{r c l}
&\bigl{|} \int \nabla^{\sigma}(w^{(1)}-\hat{w}^{'(1)}).\nabla^{\sigma}(-\frac{1}{\lambda}\tilde{\psi}_{\hat{b}',\frac{1}{\lambda}}^{(2)}) \\
&+\int \nabla^{\sigma-1}(w^{(2)}-\hat{w}^{'(2)}).\nabla^{\sigma-1}(-\frac{1}{\lambda}\tilde{\psi}_{\hat{b}',\frac{1}{\lambda}}^{(2)}+\hat{NL}'+\hat{L}') \bigr{|}\\
\leq& \frac{Cb_1^{2(\sigma-s_c)(1+\nu)+\frac{\alpha}{L}+O\left(\frac{\sigma-s_c}{L}\right)}(\sqrt{\triangle _r \hat{\mathcal{E}}_{\sigma}}+\underset{k=2}{\overset{p}{\sum}}\sqrt{\triangle _r \hat{\mathcal{E}}_{\sigma}}^{1-\frac{1-(k-1)(\sigma-s_c)}{s_L-\sigma}}\sqrt{\triangle _r \hat{\mathcal{E}}_{s_L}}^{\frac{1-(k-1)(\sigma-s_c)}{s_L-\sigma}})}{\lambda^{2(\sigma-s_c)+1}}.
\end{array}
$$
With the bound \fref{variete:modulation:eq:dhats (triangle b)} on $|\frac{d\hat{s}'}{ds}-1|$ we obtain:
\be \label{variete:lowsobo:eq:dhats'-1 psib NL L}
\begin{array}{r c l}
&\bigl{|} (1-\frac{d\hat{s}'}{ds})\int \nabla^{\sigma}(w^{(1)}-\hat{w}^{'(1)}).\nabla^{\sigma}(-\frac{1}{\lambda}\tilde{\psi}_{\hat{b}',\frac{1}{\lambda}}^{(2)}) \\
&+(1-\frac{d\hat{s}'}{ds}) \int \nabla^{\sigma-1}(w^{(2)}-\hat{w}^{'(2)}).\nabla^{\sigma-1}(-\frac{1}{\lambda}\tilde{\psi}_{\hat{b}',\frac{1}{\lambda}}^{(2)}+\hat{NL}'+\hat{L}')\bigr{|}\\
\leq& \frac{Cb_1^{2(\sigma-s_c)(1+\nu)+\frac{\alpha}{L}+O\left(\frac{\sigma-s_c}{L}\right)}}{\lambda^{2(\sigma-s_c)+1}}( \sqrt{\triangle _r \hat{\mathcal{E}}_{s_L}}+\underset{1\leq i \leq L}{\text{sup}}(b_1^{-i}|\triangle \hat{b}_i|)  )  \\
&\times  \left( \sqrt{\triangle _r \hat{\mathcal{E}}_{\sigma}}+\underset{k=2}{\overset{p}{\sum}}\sqrt{\triangle _r \hat{\mathcal{E}}_{\sigma}}^{1-\frac{1-(k-1)(\sigma-s_c)}{s_L-\sigma}}\sqrt{\triangle _r \hat{\mathcal{E}}_{s_L}}^{\frac{1-(k-1)(\sigma-s_c)}{s_L-\sigma}}\right).
\end{array}
\ee
The only term we did not really estimate in the proof of Proposition \ref{trappedregime:pr:low sobo} is the linear one, because we had a natural cancellation, the norm being adapted to a wave equation. We start with the terms involving derivatives:
\be \label{variete:lowsobo:eq:dhats'-1 1}
\begin{array}{r c l}
& \left| \int \nabla^{\sigma}(w^{(1)}-\hat{w}^{'(1)}).\nabla^{\sigma}(-\hat{w}^{'(2)})  + \nabla^{\sigma-1}(w^{(2)}-\hat{w}^{'(2)}).\nabla^{\sigma-1}(\Delta \hat{w}^{'(1)})\right|\\
\leq & \frac{C}{\lambda^{2(\sigma-s_c)+1}}\sqrt{\triangle \hat{\mathcal{E}_{\sigma}}}\parallel \hat{\bos{\varepsilon}}'\parallel_{\dot{H}^{\sigma+1}\times \dot{H}^{\sigma}}\leq \frac{Cb_1^{1+2(\sigma-s_c)(1+\nu)+\frac{\alpha}{L}+O\left(\frac{\sigma-s_c}{L} \right)}}{\lambda^{2(\sigma-s_c)+1}}\sqrt{\triangle_r \hat{\mathcal{E}_{\sigma}}}.
\end{array}
\ee
For the term involving the potential, integrating by parts, using Hardy inequality (as $Q^{p-1}=O(y^{-2}))$ and interpolation yields:
\be \label{variete:lowsobo:eq:dhats'-1 2}
\begin{array}{r c l}
& \left| \int \nabla^{\sigma-1}(w^{(2)}-\hat{w}^{'(2)}).\nabla^{\sigma-1}((p-1)Q_{\frac{1}{\lambda}}^{p-1} \hat{w}^{'(1)})\right|\\
\leq & \frac{C}{\lambda^{2(\sigma-s_c)+1}}\parallel \nabla^{\sigma}(w^{(2)}-\hat{w}^{'(2)}) \parallel_{L^2}   \parallel \nabla^{\sigma-2}(Q_{\frac{1}{\lambda}}^{p-1} \hat{w}^{'(1)}) \parallel_{L^2}\\
\leq & \frac{Cb_1^{1+2(\sigma-s_c)(1+\nu)+\frac{\alpha}{L}+O\left(\frac{\sigma-s_c}{L} \right)}}{\lambda^{2(\sigma-s_c)+1}}\sqrt{\triangle _r \hat{\mathcal{E}}_{\sigma}}^{1-\frac{1}{s_L-\sigma}}  \sqrt{\triangle _r \hat{\mathcal{E}}_{\sigma}}^{\frac{1}{s_L-\sigma}}.
\end{array}
\ee
The two previous bounds \fref{variete:lowsobo:eq:dhats'-1 1} and \fref{variete:lowsobo:eq:dhats'-1 2}, combined with the bound \fref{variete:modulation:eq:dhats (triangle b)} on $|\frac{d\hat{s}'}{ds}-1|$ give for the linear term:
\be \label{variete:lowsobo:eq:dhats'-1 lineaire}
\begin{array}{r c l}
& \left| (\frac{d\hat{s}'}{ds}-1)\int \nabla^{\sigma}(w^{(1)}-\hat{w}^{'(1)}).\nabla^{\sigma}\hat{w}^{'(2)}  + \nabla^{\sigma-1}(w^{(2)}-\hat{w}^{'(2)}).\nabla^{\sigma-1}\mathcal{L} \hat{w}^{'(1)}\right|\\
\leq & \frac{Cb_1^{1+2(\sigma-s_c)(1+\nu)+\frac{\alpha}{L}+O\left(\frac{\sigma-s_c}{L} \right)}\sqrt{\triangle _r \hat{\mathcal{E}}_{\sigma}}^{1-\frac{1}{s_L-\sigma}}  \sqrt{\triangle _r \hat{\mathcal{E}}_{\sigma}}^{\frac{1}{s_L-\sigma}}( \sqrt{\triangle _r \hat{\mathcal{E}}_{\sigma}}+\underset{1\leq i \leq L}{\text{sup}}(b_1^{-i}|\triangle \hat{b}_i|))}{\lambda^{2(\sigma-s_c)+1}}.
\end{array}
\ee
The bounds \fref{variete:lowsobo:eq:dhats'-1 psib NL L} and \fref{variete:lowsobo:eq:dhats'-1 lineaire} imply that for the terms in \fref{variete:lowsobo:eq:expression derivee} involving $\frac{d\hat{s}'}{ds}-1$:
\be \label{variete:lowsobo:eq:estimation dhats'-1}
\begin{array}{r c l}
&\bigl{|} (1-\frac{d\hat{s}'}{ds})\int \nabla^{\sigma}(w^{(1)}-\hat{w}^{'(1)}).\nabla^{\sigma}(-\hat{w}^{'(1)}-\tilde{\psi}_{\hat{b}',\frac{1}{\lambda}}^{(2)}) \\
& + (1-\frac{d\hat{s}'}{ds})\int \nabla^{\sigma-1}(w^{(2)}-\hat{w}^{'(2)}).\nabla^{\sigma-1}(-\mathcal{L}\hat{w}^{'(1)}-\tilde{\psi}_{\hat{b}',\frac{1}{\lambda}}^{(2)}+\hat{NL}'+\hat{L}' )\bigr{|}\\
\leq &\frac{Cb_1^{1+2(\sigma-s_c)(1+\nu)+\frac{\alpha}{L}+O\left(\frac{\sigma-s_c}{L} \right)}\sqrt{\triangle _r \hat{\mathcal{E}}_{\sigma}}^{1-\frac{1}{s_L-\sigma}}  \sqrt{\triangle _r \hat{\mathcal{E}}_{\sigma}}^{\frac{1}{s_L-\sigma}}( \sqrt{\triangle _r \hat{\mathcal{E}}_{\sigma}}+\underset{1\leq i \leq L}{\text{sup}}(b_1^{-i}|\triangle \hat{b}_i|))}{\lambda^{2(\sigma-s_c)+1}}.
\end{array}
\ee

\underline{Step 2:} Gathering the bounds. We have made the decomposition \fref{variete:lowsobo:eq:expression derivee} and have computed an upper bound for all terms in the right hand side in \fref{variete:lowsobo:eq:estimation lineaire}, \fref{variete:lowsobo:eq:estimation mod}, \fref{variete:lowsobo:eq:estimation psib}, \fref{variete:lowsobo:eq:estimation L}, \fref{variete:lowsobo:eq:estimation NL} and \fref{variete:lowsobo:eq:estimation dhats'-1}. Consequently:
$$
\begin{array}{r c l}
\frac{d}{dt}\left\{ \frac{\triangle \mathcal{E}_{\sigma}}{\lambda^{2(\sigma-s_c)}} \right\} \leq \frac{Cb_1^{1+2(\sigma-s_c)(1+\nu)+\frac{\alpha}{2L}}}{\lambda^{2(\sigma-s_c)+1}} \left( \triangle_r \hat{\mathcal{E}}_{\sigma}+\triangle _r \hat{\mathcal{E}}_{s_L}+(\underset{1\leq i \leq L}{\text{sup}}b_1^{-i}|\triangle \hat{b}_i|)^2 \right),
\end{array}
$$
which is the bound we had to prove.
\end{proof}


We now turn to the control of the most important of the two norms of the difference of errors $\bos{\varepsilon}-\hat{\bos{\varepsilon}}'$: the adapted one at a high level of regularity. We state a similar result as the one in Proposition \ref{trappedregime:pr:high sobo}, this time relating the time evolution to the differences of the parameters and errors. Again, we will not be able to control directly a local norm, relegating it to the next lemma.

\begin{lemma}\label{variete:lem:high sobo}
\emph{(Lyapunov monotonicity for the high Sobolev norm:)} We recall that $\triangle \mathcal{E}_{s_L}$ and $\triangle \mathcal{E}_{s_L,\text{loc}}$ are defined in \fref{thetrapped:eq:def mathcalEsL} and \fref{thetrapped:eq:def mathcalEsLloc}). There holds for $s_0\leq s$:
\begin{equation} \label{variete:eq:high sobo}
\begin{array}{r c l}
&\frac{d}{dt}\left\{ \frac{\triangle \hat{\mathcal{E}}_{s_L}}{\lambda^{2(s_L-s_c)}}+O \left( \frac{b_1^{2L+2(1-\delta_0)(1+\eta)}}{\lambda^{2(s_L-s_c)}}(\triangle_r \hat{\mathcal{E}}_{s_L}+|\underset{1\leq i \leq L}{\text{sup}}b_1^{-i}|\triangle \hat{b}_i||^2) \right) \right\} \\
\leq& \frac{Cb_1^{2L+2(1-\delta_0)(1+\frac{\eta}{2})+1}}{\lambda^{2(s_L-s_c)+1}}(b_1^{\frac{\eta}{2}(1-\delta_0)}\sqrt{\triangle_r \hat{\mathcal{E}}_{s_L}}\underset{1\leq i \leq L}{\text{sup}}b_1^{-i}|\triangle \hat{b}_i|+C(N)\triangle_r \hat{\mathcal{E}}_{s_L,\text{loc}} \\
&+\frac{C}{N^{\frac{\delta_0}{2}}}(\triangle_r \hat{\mathcal{E}}_{s_L}+\triangle_r \hat{\mathcal{E}}_{\sigma})+b_1^{\frac{\eta}{2}(1-\delta_0)}\underset{1\leq i \leq L}{\text{sup}}b_1^{-i}|\triangle \hat{b}_i|)^2) ,\\
\end{array}
\end{equation}
for some universal constant $C$ that does not depend on $N$.
\end{lemma}

\begin{proof}[Proof of Lemma \ref{variete:lem:high sobo}] The strategy of the proof of Lemma \ref{variete:lem:high sobo} is similar to the one of the proof of Proposition \fref{trappedregime:pr:high sobo}. We start by computing the following identity:
\be \label{variete:highsobo:eq:expression generale 1}
\begin{array}{r c l}
&\frac{d}{dt}\left( \frac{\triangle \mathcal{E}_{s_L}}{2\lambda^{2(s_L-s_c)}} \right) \\
=& \int (w^{(1)}-\hat{w}^{'(1)})\mathcal{L}_{\frac{1}{\lambda}}^{k_0+L+1} \bigl{[}w^{(2)}-\hat{w}^{'(2)}-\frac{1}{\lambda}\tilde{\psi}_{b,\frac{1}{\lambda}}^{(1)}+\frac{1}{\lambda}\tilde{\psi}_{\hat{b}',\frac{1}{\lambda}}^{(1)} -\frac{1}{\lambda} \tilde{Mod}_{\frac{1}{\lambda}}^{(1)}\\
&+\frac{1}{\lambda}\hat{Mod}^{'(1)}_{\frac{1}{\lambda}} +(\frac{d\hat{s}'}{ds}-1)(\hat{w}^{'(2)}-\frac{1}{\lambda}\tilde{\psi}^{(1)}_{\hat{b}',\frac{1}{\lambda}}) \Bigr{]}\\
&+\int (w^{(2)}-\hat{w}^{'(2)})\mathcal{L}_{\frac{1}{\lambda}}^{k_0+L} \Bigl{[}-\mathcal{L}_{\frac{1}{\lambda}}(w^{(1)}-\hat{w}^{'(1)})-\frac{1}{\lambda}(\tilde{\psi}_{b,\frac{1}{\lambda}}^{(2)}-\tilde{\psi}_{\hat{b}',\frac{1}{\lambda}}^{(2)}) -\frac{1}{\lambda} \tilde{Mod}_{\frac{1}{\lambda}}^{(2)}\\
&+\frac{1}{\lambda}\hat{Mod}_{\frac{1}{\lambda}}^{'(2)}+L-\hat{L}'+NL-\hat{NL}'+(1-\frac{d\hat{s}'}{ds})(-\mathcal{L}\hat{w}^{'(1)}-\frac{1}{\lambda}\tilde{\psi}_{\hat{b}',\frac{1}{\lambda}}^{(2)}+\hat{NL}'+\hat{L}')  \Bigr{]}\\
&+\frac{1}{2}\sum_{i=1}^{k_0+L+1} \int (w^{(1)}-\hat{w}^{'(1)}) \mathcal{L}_{\frac{1}{\lambda}}^{i-1}\frac{d}{dt}\left( \mathcal{L}_{\frac{1}{\lambda}}\right)\mathcal{L}_{\frac{1}{\lambda}}^{k_0+L+1-i} (w^{(1)}-\hat{w}^{'(1)}) \\
&+\frac{1}{2}\sum_{i=1}^{k_0+L} \int (w^{(2)}-\hat{w}^{'(2)}) \mathcal{L}_{\frac{1}{\lambda}}^{i-1}\frac{d}{dt}\left( \mathcal{L}_{\frac{1}{\lambda}}\right) \mathcal{L}_{\frac{1}{\lambda}}^{k_0+L-i} (w^{(2)}-\hat{w}^{'(2)}) .
\end{array}
\ee
We now manage all terms in the right hand side.\\

\underline{Step 1:} Direct bounds. The linear, non linear, error, and time error terms can be estimated via a direct bound. We claim the following identity:
\be \label{variete:highsobo:eq:expression generale 2}
\begin{array}{r c l}
&\frac{d}{dt}\left( \frac{\triangle \mathcal{E}_{s_L}}{2\lambda^{2(s_L-s_c)}} \right) \\
=& \int (w^{(1)}-\hat{w}^{'(1)})\mathcal{L}_{\frac{1}{\lambda}}^{k_0+L+1} \bigl{[}-\frac{1}{\lambda} \tilde{Mod}_{\frac{1}{\lambda}}^{(1)}+\frac{1}{\lambda}\hat{Mod}^{'(1)}_{\frac{1}{\lambda}}  \Bigr{]}\\
&+\int (w^{(2)}-\hat{w}^{'(2)})\mathcal{L}_{\frac{1}{\lambda}}^{k_0+L} \Bigl{[}-\frac{1}{\lambda} \tilde{Mod}_{\frac{1}{\lambda}}^{(2)}+\frac{1}{\lambda}\hat{Mod}_{\frac{1}{\lambda}}^{'(2)}+L-\hat{L}'\Bigr{]}\\
&+\frac{1}{2}\sum_{i=1}^{k_0+L+1} \int (w^{(1)}-\hat{w}^{'(1)}) \mathcal{L}_{\frac{1}{\lambda}}^{i-1}\frac{d}{dt}\left( \mathcal{L}_{\frac{1}{\lambda}}\right)\mathcal{L}_{\frac{1}{\lambda}}^{k_0+L+1-i} (w^{(1)}-\hat{w}^{'(1)}) \\
&+\frac{1}{2}\sum_{i=1}^{k_0+L} \int (w^{(2)}-\hat{w}^{'(2)}) \mathcal{L}_{\frac{1}{\lambda}}^{i-1}\frac{d}{dt}\left( \mathcal{L}_{\frac{1}{\lambda}}\right) \mathcal{L}_{\frac{1}{\lambda}}^{k_0+L-i} (w^{(2)}-\hat{w}^{'(2)}) \\
&+O\Bigl{(}\frac{Cb_1^{2L+(1-\delta_0)(2+\frac{3}{2}\eta)+1}(\sqrt{\triangle_r \hat{\mathcal{E}}_{s_L}}\underset{1\leq i \leq L}{\text{sup}}b_1^{-i}|\triangle \hat{b}_i|+b_1^{\frac{\alpha+O(\sigma-s_c,\eta)}{L}}(\triangle_r \hat{\mathcal{E}}_{s_L}+\triangle_r \hat{\mathcal{E}}_{\sigma}))}{\lambda^{2(s_L-s_c)+1}} \Bigr{)},
\end{array}
\ee
which we are now going to prove by finding upper bounds for each term in the right hand side of \fref{variete:highsobo:eq:expression generale 1}.\\
$\bullet$ \emph{Linear terms:} The fact that the form of the norm is adapted to the linear wave equation with operator $\mathcal{L}$ induces:
\be \label{variete:highsobo:eq:estimation lineaire} 
\int (w^{(1)}-\hat{w}^{'(1)}) \mathcal{L}_{\frac{1}{\lambda}}^{k_0+L+1}(w^{(2)}-\hat{w}^{'(2)})+(w^{(2)}-\hat{w}^{'(2)}) \mathcal{L}_{\frac{1}{\lambda}}^{k_0+L}(-\mathcal{L}_{\frac{1}{\lambda}}(w^{(1)}-\hat{w}^{'(1)}))=0 .
\ee
$\bullet$ \emph{Error terms:} Using the bound \fref{variete:eq:estimation high difference psib} on $\tilde{\bos{\psi}}_b-\tilde{\bos{\psi}}_{\hat{b}'}$:
\be \label{variete:highsobo:eq:estimation psib} 
\begin{array}{r c l}
&\left|\frac{1}{\lambda}\int (w^{(1)}-\hat{w}^{'(1)}) \mathcal{L}_{\frac{1}{\lambda}}^{k_0+L+1}(\tilde{\psi}_{b,\frac{1}{\lambda}}^{(1)}-\tilde{\psi}_{\hat{b}',\frac{1}{\lambda}}^{(1)})+(w^{(2)}-\hat{w}^{'(2)}) \mathcal{L}_{\frac{1}{\lambda}}^{k_0+L}(\tilde{\psi}_{b,\frac{1}{\lambda}}^{(2)}-\tilde{\psi}_{\hat{b}',\frac{1}{\lambda}}^{(2)})\right| \\
\leq& \frac{C\sqrt{\triangle_r \hat{\mathcal{E}}_{s_L}}}{\lambda^{2(s_L-s_c)+1}}( b_1^{2L+(1-\delta_0)(2+\frac{3}{2}\eta)+1}\underset{1\leq i \leq L}{\text{sup}}(b_1^{-i}|\triangle \hat{b}_i|)+b_1^{3L+3-3\delta_0+O(\eta)}\sqrt{\triangle_r\hat{\mathcal{E}}_{s_L}}).
\end{array}
\ee
$\bullet$ \emph{Non linear terms:} We know that $NL$ is a sum of terms of the form $\tilde{Q}^{p-k}_b\varepsilon^{(1)k}$ for $2\leq k \leq p$. Therefore we start by decomposing:
\be \label{variete:highsobo:eq:estimation NL expression}
\begin{array}{r c l}
\parallel (NL-\hat{NL}')_{k_0+L}\parallel_{L^2} & \leq & C\sum_{2}^p\parallel (\tilde{Q}_b^{p-k}(\varepsilon^{(1)k}-\hat{\varepsilon}^{'(1)k}))_{k_0+L}\parallel_{L^2}\\
&&+\parallel (\hat{\varepsilon}^{'(1)k}(\tilde{Q}_b^{(1)(p-k)}-\tilde{Q}_{\hat{b}'}^{(1)(p-k)}))_{k_0+L}\parallel_{L^2}
\end{array}
\ee
For the first term of this identity, we can do the same reasoning we used in the proof of the direct bound \fref{thetrapped:highsobo:eq:estimation NL} in the proof of Proposition \ref{trappedregime:pr:high sobo}. What changes here is that we do not have to treat $\varepsilon^{(1)k}$, but $(\varepsilon^{(1)}-\hat{\varepsilon}^{'(1)})\varepsilon^{(1)i}\hat{\varepsilon}^{'(1)(k-1-i)}$ because of the factorization:
$$
\varepsilon^{(1)k}-\hat{\varepsilon}^{'(1)k}=(\varepsilon^{(1)}-\hat{\varepsilon}^{'(1)})\sum_{i=0}^{k-1} C_i\varepsilon^{(1)i}\hat{\varepsilon}^{'(1)(k-1-i)}
$$
for some constants $(C_0)_{1\leq i \leq k-1}$. We recall that using various decompositions, Hardy inequalities and Sobolev injections, in \fref{thetrapped:highsobo:eq:estimation NL} we proved:
\be \label{variete:highsobo:eq:estimation NL 1}
\parallel (NL)_{s_L-1}\parallel_{L^2}^2\leq C(K_1,K_2) b_1^{2L+2+2(1-\delta_0)(1+\eta)+\frac{2\alpha}{L}+O\left(\frac{\sigma-s_c}{L} \right)}.
\ee
Whenever interpolating between $\triangle \hat{\mathcal{E}}_{\sigma} $ and $\triangle \hat{\mathcal{E}}_{s_L}$ one has for $0\leq \theta\leq 1$:
$$
\begin{array}{r c l}
\triangle \hat{\mathcal{E}}_{\sigma} ^{\theta} \triangle \hat{\mathcal{E}}_{s_L}^{1-\theta}&\leq &b_1^{2\theta(\sigma-s_c)(1+\nu)+2(1-\theta)(L+(1-\delta_0)(1+\eta))}\triangle _r\hat{\mathcal{E}}_{\sigma} ^{\theta} \triangle_r \hat{\mathcal{E}}_{s_L}^{1-\theta}\\
&\leq& b_1^{2\theta(\sigma-s_c)(1+\nu)+2(1-\theta)(L+(1-\delta_0)(1+\eta))}(\triangle _r\hat{\mathcal{E}}_{\sigma}+\triangle_r \hat{\mathcal{E}}_{s_L}).
\end{array}
$$
This is why in this case, \fref{variete:highsobo:eq:estimation NL 1} transforms into:
\be \label{variete:highsobo:eq:estimation NL 2}
\parallel (\tilde{Q}_b^{p-k}(\varepsilon^{(1)k}-\hat{\varepsilon}^{'(1)k}))_{s_L-1}\parallel_{L^2}^2\leq Cb_1^{2L+2+2(1-\delta_0)(1+\eta)+\frac{2\alpha}{L}+O\left(\frac{\sigma-s_c}{L} \right)}(\triangle _r\hat{\mathcal{E}}_{\sigma}+\triangle_r \hat{\mathcal{E}}_{s_L}).
\ee
We now turn to the second term in \fref{variete:highsobo:eq:estimation NL expression}. Using the bound \fref{variete:eq:asymptotique diff Qp-j} and again the same reasoning that proved  \fref{thetrapped:highsobo:eq:estimation NL} one gets:
\be \label{variete:highsobo:eq:estimation NL 3}
\begin{array}{r c l}
&\parallel ((\tilde{Q}_b^{(1)(p-k)}-\tilde{Q}_{\hat{b}'}^{(1)(p-k)})\hat{\varepsilon}^{'(1)k})_{s_L-1}\parallel_{L^2}^2\\
\leq &Cb_1^{2L+2+2(1-\delta_0)(1+\eta)+\frac{2\alpha}{L}+O\left(\frac{\sigma-s_c}{L} \right)}\underset{1\leq i \leq L}{\text{sup}}(b_1^{-i}|\triangle \hat{b}_i|).
\end{array}
\ee
We can now come back to the identity \fref{variete:highsobo:eq:estimation NL expression}, inject the bounds \fref{variete:highsobo:eq:estimation NL 2} and \fref{variete:highsobo:eq:estimation NL 3} to find that the size of the nonlinear term is:
\be \label{variete:highsobo:eq:estimation NL 4}
\begin{array}{r c l}
&\parallel (NL-\hat{NL}')_{k_0+L} \parallel_{L^2}\\
\leq & Cb_1^{L+1+(1-\delta_0)(1+\eta)+\frac{\alpha}{L}+O\left(\frac{\sigma-s_c}{L}  \right)}(\sqrt{\triangle_r \hat{\mathcal{E}}_{\sigma}}+\sqrt{\triangle_r \hat{\mathcal{E}}_{s_L}}+\underset{1\leq i \leq L}{\text{sup}}b_1^{-i}|\triangle \hat{b}_i|).
\end{array}
\ee
After rescaling and applying Cauchy-Schwarz, this gives the following bound on the nonlinear term's contribution:
\be \label{variete:highsobo:eq:estimation NL}
\begin{array}{r c l}
& \left| \int (w^{(2)}-\hat{w}^{'(2)})\mathcal{L}^{s_L-1}_{\frac{1}{\lambda}}(NL-\hat{NL}')\right|\\
\leq & \frac{Cb_1^{2L+2(1-\delta_0)(1+\eta)+1+\frac{\alpha}{L}+O\left(\frac{\sigma-s_c}{L} \right)}\sqrt{\triangle_r \hat{\mathcal{E}}_{s_L}}}{\lambda^{2(s_L-s_c)+1}}(\underset{1\leq i \leq L}{\text{sup}}b_1^{-i}|\triangle \hat{b}_i|+\sqrt{\triangle_r \hat{\mathcal{E}}_{\sigma}}+\sqrt{\triangle \hat{\mathcal{E}}_{s_L}}).
\end{array}
\ee
$\bullet$ \emph{Time difference terms:} For the small linear term involving $\hat{w}'$ we recall \fref{thetrapped:highsobo:eq:estimation intermediaire L}:
$$
\parallel \hat{L}'_{s_L-1} \parallel_{L^2} \leq Cb_1 \left\Vert \frac{\hat{\varepsilon}^{'(1)}_{s_L}}{1+y^{\delta}}\right\Vert_{L^2}\leq Cb_1 \parallel \varepsilon^{'(1)}_{s_L}\parallel_{L^2}\leq Cb_1^{L+(1-\delta_0)(1+\eta)+1} .
$$
For the linear term, we need the extra assumption \fref{variete:eq:estimation supplementaire varepsilon'} on the higher derivative of $\hat{\varepsilon}'$, it produces:
$$
\parallel \hat{\varepsilon}^{'(2)}_{s_L} \parallel_{L^2}+\parallel \hat{\varepsilon}^{'(1)}_{s_L+1} \parallel_{L^2} \leq Cb_1^{L+(1-\delta_0)(1+\eta)+1}.
$$
The two previous inequalities, with the estimates \fref{thetrapped:highsobo:eq:estimation psib} and \fref{thetrapped:highsobo:eq:estimation NL} we already derived for the non linear and error terms, plus the bound \fref{variete:modulation:eq:dhats (triangle b)} on $|\frac{d\hat{s}'}{ds}-1|$ yield:
\be \label{variete:highsobo:eq:estimation dhats'-1 1}
\begin{array}{r c l}
&|\frac{d\hat{s}'}{ds}-1|\left[ \parallel (\hat{w}^{'(2)} -\frac{1}{\lambda}\tilde{\psi}^{(1)}_{\hat{b}',\frac{1}{\lambda}})_{s_L}\parallel_{L^2} +\parallel (\mathcal{L}\hat{w}^{'(1)}+\frac{1}{\lambda}\tilde{\psi}_{\hat{b}',\frac{1}{\lambda}}^{(2)}-\hat{NL}'-\hat{L}')_{s_L-1}  \parallel_{L^2} \right]\\
\leq & \frac{Cb_1^{L+(1-\delta_0)(1+\eta)+1}}{\lambda^{(s_L-s_c)+1}}(\underset{1\leq i \leq L}{\text{sup}}b_1^{-i}|\triangle \hat{b}_i|+b_1^{L+(1-\delta_0)} \sqrt{\triangle_r \hat{\mathcal{E}}_{s_L}} ).
\end{array}
\ee
This implies that the contribution of the terms involving the difference of the evolution of the renormalized times $\frac{d\hat{s}'}{ds}-1$ in \fref{variete:highsobo:eq:expression generale 1} is:
\be \label{variete:highsobo:eq:estimation dhats'-1}
\begin{array}{r c l}
&\Bigl{|} (\frac{d\hat{s}'}{ds}-1)\int (w^{(1)}-\hat{w}^{'(1)})\mathcal{L}^{s_L} (\hat{w}^{'(2)} -\frac{1}{\lambda}\tilde{\psi}^{(1)}_{\hat{b}',\frac{1}{\lambda}}\\
&+\int (w^{(2)}-\hat{w}^{'(2)})\mathcal{L}^{s_L-1}(-\mathcal{L}\hat{w}^{'(1)}-\frac{1}{\lambda}\tilde{\psi}_{\hat{b}',\frac{1}{\lambda}}^{(2)}+\hat{NL}'+\hat{L}')  \Bigr{|} \\
\leq & \frac{Cb_1^{2L+(1-\delta_0)(2+\frac{3}{2}\eta)+1}}{\lambda^{2(s_L-s_c)+1}}\sqrt{\triangle_r \hat{\mathcal{E}}_{s_L}}(\underset{1\leq i \leq L}{\text{sup}}b_1^{-i}|\triangle \hat{b}_i|+b_1^{L+1-\delta_0} \sqrt{\triangle_r \hat{\mathcal{E}}_{s_L}} ).
\end{array}
\ee
We reach the end of the proof of the first step. We now inject the bounds \fref{variete:highsobo:eq:estimation lineaire} for the linear terms, \fref{variete:highsobo:eq:estimation psib} for the error terms, \fref{variete:highsobo:eq:estimation NL} for the non linear terms and \fref{variete:highsobo:eq:estimation dhats'-1} for the time error term in \fref{variete:highsobo:eq:expression generale 1}, yielding the intermediate equation \fref{variete:highsobo:eq:expression generale 2} claimed in this step 1.\\

\underline{Step 2:} Terms making appear a local part that cannot be estimated directly. The small linear terms and the scale changing terms cannot be estimated directly. The aim of this step is to decompose their contribution into two parts: one that can be bounded directly and the other that requires the study of a Morawetz type quantity, see next Lemma \ref{variete:lem:morawetz}. We claim that \fref{variete:highsobo:eq:expression generale 2} can be transformed into:
\be \label{variete:highsobo:eq:expression derivee step2}
\begin{array}{r c l}
\frac{d}{dt}\left( \frac{\triangle \mathcal{E}_{s_L}}{2\lambda^{2(s_L-s_c)}} \right) &=& \int (w^{(1)}-\hat{w}^{'(1)})\mathcal{L}_{\frac{1}{\lambda}}^{k_0+L+1} \bigl{[}-\frac{1}{\lambda} \tilde{Mod}_{\frac{1}{\lambda}}^{(1)}+\frac{1}{\lambda}\hat{Mod}^{'(1)}_{\frac{1}{\lambda}}  \Bigr{]}\\
&&+\int (w^{(2)}-\hat{w}^{'(2)})\mathcal{L}_{\frac{1}{\lambda}}^{k_0+L} \Bigl{[}-\frac{1}{\lambda} \tilde{Mod}_{\frac{1}{\lambda}}^{(2)}+\frac{1}{\lambda}\hat{Mod}_{\frac{1}{\lambda}}^{'(2)}\Bigr{]}\\
&&+O\Bigl{(}\frac{Cb_1^{2L+2(1-\delta_0)(1+\frac{\eta}{2})+1}}{\lambda^{2(s_L-s_c)+1}}(b_1^{\frac{\eta}{2}(1-\delta_0)}\sqrt{\triangle_r \hat{\mathcal{E}}_{s_L}}\underset{1\leq i \leq L}{\text{sup}}b_1^{-i}|\triangle \hat{b}_i|\\
&&+C(N)\triangle_r \hat{\mathcal{E}}_{s_L,\text{loc}} +\frac{C}{N^{\frac{\delta_0}{2}}}(\triangle_r \hat{\mathcal{E}}_{s_L}+\triangle_r \hat{\mathcal{E}}_{\sigma}) \Bigr{)},\\
\end{array}
\ee
the constant $C$ being independent of $N$. We now prove this identity by establishing bounds on the small linear terms and the scale changing terms in \fref{variete:highsobo:eq:expression generale 2}. \\
$\bullet$ \emph{The small linear terms:} We start by decomposing:
\be \label{variete:highsobo:eq:estimation L expression}
\begin{array}{r c l}
L-\hat{L}'&=&p(\tilde{Q}_{b}^{(1)(p-1)}-Q^{p-1})\varepsilon^{(1)}-p(\tilde{Q}_{\hat{b}}^{(1)(p-1)}-Q^{p-1})\hat{\varepsilon}^{'(1)}\\
&=&p(\tilde{Q}_{b}^{(1)(p-1)}-\tilde{Q}_{\hat{b}'}^{(1)(p-1)})\varepsilon^{(1)}+p(\tilde{Q}_{\hat{b}}^{(1)(p-1)}-Q^{p-1})(\varepsilon^{(1)}-\hat{\varepsilon}^{'(1)})
\end{array}
\ee
and we now estimate each term. For the first term in \fref{variete:highsobo:eq:estimation L expression}, from the bound \fref{variete:eq:asymptotique diff Qp-j} on $\tilde{Q}_{b}^{(1)(p-1)}-\tilde{Q}_{\hat{b}'}^{(1)(p-1)}$ one gets:
\be \label{variete:highsobo:eq:estimation L 1}
\parallel ((\tilde{Q}_{b}^{(1)(p-1)}-\tilde{Q}_{\hat{b}'}^{(1)(p-1)})\varepsilon^{(1)})_{s_L-1} \parallel_{L^2}\leq b_1^{L+1+(1-\delta_0)(1+\eta)}\underset{1\leq i \leq L}{\text{sup}}b_1^{-i}|\triangle \hat{b}_i|.
\ee
Now for the second term in \fref{variete:highsobo:eq:estimation L expression}, using the same reasoning we used to prove \fref{thetrapped:highsobo:eq:estimation intermediaire L} we obtain: 
$$
\parallel ((\tilde{Q}_{\hat{b}}^{(1)(p-1)}-Q^{p-1})(\varepsilon^{(1)}-\hat{\varepsilon}^{'(1)}))_{s_L-1}  \parallel_{L^2}^2 \leq C b_1^2 \left\Vert \frac{(\varepsilon^{(1)}-\hat{\varepsilon}^{'(1)})_{s_L}}{1+y^{\frac{\delta_0}{2}}} \right\Vert_{L^2}^2 .
$$
By cutting at a distance $N$ from the origin one gets:
\be \label{variete:highsobo:eq:estimation L 2}
\left\Vert ((\tilde{Q}_{\hat{b}}^{(1)(p-1)}-Q^{p-1})(\varepsilon^{(1)}-\hat{\varepsilon}^{'(1)}))_{s_L-1} \right\Vert_{L^2}^2 \leq  \frac{b_1^2C}{N^{\delta_0}}\triangle \hat{\mathcal{E}}_{s_L}+C(N)b_1^2\triangle \hat{\mathcal{E}}_{s_L,\text{loc}}.
\ee
We now come back to the expression \fref{variete:highsobo:eq:estimation L expression} for which we have found bounds in \fref{variete:highsobo:eq:estimation L 1} and \fref{variete:highsobo:eq:estimation L 2}, yielding the following size for the small linear terms:
\be \label{variete:highsobo:eq:estimation L 1}
\begin{array}{r c l}
&\parallel (L-\hat{L}')_{s_L-1}\parallel_{L^2}\\
\leq & Cb_1^{L+(1-\delta_0)(1+\frac{\eta}{2}\eta)+1}\left (\frac{\sqrt{\triangle_r \hat{\mathcal{E}}_{s_L}}}{N^{\frac{\delta_0}{2}}}+C(N)\sqrt{\triangle_r \hat{\mathcal{E}}_{s_L,\text{loc}}}+b_1^{\frac{\eta}{2}(1-\delta_0)}\underset{1\leq i \leq L}{\text{sup}}b_1^{-i}|\triangle \hat{b}_i| \right).
\end{array}
\ee
After rescaling, applying Cauchy-Schwarz inequality the contribution of the small linear terms can be split into:
\be \label{variete:highsobo:eq:estimation L}
\begin{array}{r c l}
&\left| \int (w^{(2)}-\hat{w}^{'(2)})\mathcal{L}_{\frac{1}{\lambda}}^{k_0+L} (L-\hat{L}')\right|\\
\leq& \frac{Cb_1^{2L+2(1-\delta_0)(1+\frac{1}{2}\eta)+1}\left(b_1^{\frac{\eta}{2}(1-\delta_0)}\sqrt{\triangle_r \hat{\mathcal{E}}_{s_L}} \underset{1\leq i \leq L}{\text{sup}}b_1^{-i}|\triangle \hat{b}_i|+\frac{1}{N^{\frac{\delta_0}{2}}}\triangle_r \hat{\mathcal{E}}_{s_L}+C(N)\triangle_r \hat{\mathcal{E}}_{s_L,\text{loc}}\right)}{\lambda^{2(s_L-s_c)+1}}.
\end{array}
\ee
$\bullet$ \emph{The scale changing term:} Using verbatim the same methodology we used to prove \fref{thetrapped:highsobo:eq:estimation lambdaL} we get:
\be \label{variete:highsobo:eq:estimation lambdaL}
\begin{array}{r c l}
&\Bigl{|} \sum_{i=1}^{k_0+L+1} \int (w^{(1)}-\hat{w}^{'(1)}) \mathcal{L}_{\frac{1}{\lambda}}^{i-1}\frac{d}{dt}\left( \mathcal{L}_{\frac{1}{\lambda}}\right)\mathcal{L}_{\frac{1}{\lambda}}^{k_0+L+1-i} (w^{(1)}-\hat{w}^{'(1)}) \\
&+\sum_{i=1}^{k_0+L} \int (w^{(2)}-\hat{w}^{'(2)}) \mathcal{L}_{\frac{1}{\lambda}}^{i-1}\frac{d}{dt}\left( \mathcal{L}_{\frac{1}{\lambda}}\right) \mathcal{L}_{\frac{1}{\lambda}}^{k_0+L-i} (w^{(2)}-\hat{w}^{'(2)}) \Bigr{|} \\
\leq &\frac{C(M)b_1^{2L+2(1-\delta_0)(1+\frac{\eta}{2})+1}}{\lambda^{2(s_L-s_c)+1}} \left(\frac{\triangle_r \hat{\mathcal{E}}_{s_L}}{N^{\frac{\delta_0}{2}}} +C(N) \triangle_r \hat{\mathcal{E}}_{s_L,\text{loc}} \right) ,
\end{array}
\ee
Coming back to the identity \fref{variete:highsobo:eq:expression generale 2} we derived in step 1, and injecting the bounds \fref{variete:highsobo:eq:estimation L}  on the small linear terms and \fref{variete:highsobo:eq:estimation lambdaL} on the scale changing terms gives the identity \fref{variete:highsobo:eq:expression derivee step2} that we had to prove in this step 2.\\

\underline{Step 3:} The modulation term. We need to find a proper integration by parts in time to deal with the modulation terms. We claim that:
\be \label{variete:highsobo:eq:estimation mod}
\begin{array}{r c l}
&\int (w^{(1)}-\hat{w}^{'(1)})\mathcal{L}_{\frac{1}{\lambda}}^{s_L} \frac{(\tilde{Mod}^{(1)}-\hat{Mod}^{'(1)})_{\frac{1}{\lambda}}}{\lambda} + (w^{(2)}-\hat{w}^{'(2)})\mathcal{L}_{\frac{1}{\lambda}}^{s_L-1}\frac{(\tilde{Mod}^{(2)}-\hat{Mod}^{'(2)})_{\frac{1}{\lambda}}}{\lambda} \\
=& \partial_t \left[O \left( \frac{b_1^{2L+2(1-\delta_0)(1+\eta)}}{\lambda^{2(s_L-s_c)}}(\triangle_r \hat{\mathcal{E}}_{s_L}+|\underset{1\leq i \leq L}{\text{sup}}b_1^{-i}|\triangle \hat{b}_i||^2) \right)\right]\\
&+O\left( \frac{b_1^{2L+1+2(1-\delta_0)(1+\eta)}}{\lambda^{2(s_L-s_c)+1}}(\triangle_r \hat{\mathcal{E}}_{s_L}+|\underset{1\leq i \leq L}{\text{sup}}b_1^{-i}|\triangle \hat{b}_i||^2) \right).
\end{array}
\ee
Once this bound is proven, we finish the proof of the proposition by injecting it in \fref{variete:highsobo:eq:expression derivee step2}. Therefore to finish to proof we now prove \fref{variete:highsobo:eq:estimation mod}. We recall that $\triangle \hat{\bos{Mod}}_i$ is defined by \fref{variete:eq:def triangle hatmodi} and \fref{variete:eq:def triangle hatmod0}, and that $ \tilde{\bos{Mod}}-\hat{\bos{Mod}}'=\sum_{i=0}^{L} \triangle \hat{\bos{Mod}}_i $. First we find a direct bound for the all the modulation terms other than the $L$-th. Let $i$ denote an even integer, $1\leq i \leq L-1$. The fact that we assume $i$ even is just to have a precise location for the profiles. In that case one decompose:
\be \label{variete:highsobo:modulation:eq:leqL-1 expression}
\begin{array}{l l l l}
&\triangle \hat{\bos{Mod}}_i:=\bos{A}_1+\bos{A}_2\\
=&(b_{i,s}+(i-\alpha)b_1b_i-b_{i+1}-(\hat{b}'_{i,s}+\frac{d\hat{s}'}{ds}((i-\alpha)\hat{b}'_1\hat{b}'_i-\hat{b}'_{i+1})))\chi_{B_1}(\bos{T}_i+\underset{j=i+1}{\overset{L+2}{\sum}}\frac{\partial \bos{S}_{j}}{\partial b_i})\\
&+(\hat{b}'_{i,s}+\frac{d\hat{s}'}{ds}((i-\alpha)\hat{b}'_1\hat{b}'_i-\hat{b}'_{i+1}))(\chi_{B_1}(\bos{T}_i+\underset{j=i+1}{\overset{L+2}{\sum}}\frac{\partial \bos{S}_{j}}{\partial b_i})-\chi_{\hat{B}'_1}(\bos{T}_i+\underset{j=i+1}{\overset{L+2}{\sum}}\frac{\partial \hat{\bos{S}}'_{j}}{\partial b_i})).
\end{array}
\ee
For the first term of the previous equation, we employ the bound \fref{variete:modulation:eq:leqL-1 (triangle b)} on the modulation of the parameters $b_i$ for $1\leq i \leq L-1$, yielding:
\be \label{variete:highsobo:modulation:eq:leqL-1 1}
\parallel (A_1^{(1)})_{s_L} \parallel_{L^2}+\parallel (A_1^{(2)})_{s_L-1} \parallel_{L^2}\leq Cb_1^{L+3-\delta_0+O(\eta)}(\underset{1\leq i \leq L}{\text{sup}}b_1^{-i}|\triangle \hat{b}_i|+\sqrt{\triangle_r \hat{\mathcal{E}}_{s_L}}).
\ee
For the second term \fref{variete:eq:difference polynomes parametres} and \fref{variete:eq:formulation integrale diff chi} imply that:
$$
\begin{array}{r c l}
&\left\Vert \left(\chi_{B_1}(\bos{T}_i+\underset{j=i+1 \ \text{even}}{\overset{L+2}{\sum}}\frac{\partial \bos{S}_{j}}{\partial b_i})-\chi_{\hat{B}'_1}(\bos{T}_i+\underset{j=i+1 \ \text{even} }{\overset{L+2}{\sum}}\frac{\partial \hat{\bos{S}}'_{j}}{\partial b_i})\right)_{s_L} \right\Vert_{L^2}\\
&+\left\Vert \left(\chi_{B_1}\underset{j=i+1 \ \text{odd}}{\overset{L+2}{\sum}}\frac{\partial \bos{S}_{j}}{\partial b_i}-\chi_{\hat{B}'_1}\underset{j=i+1 \ \text{odd}}{\overset{L+2}{\sum}}\frac{\partial \hat{\bos{S}}'_{j}}{\partial b_i}\right)_{s_L-1} \right\Vert \\
\leq& Cb_1 \underset{1\leq i \leq L}{\text{sup}}b_1^{-i}|\triangle \hat{b}_i|.
\end{array}
$$
We then use the primary bound \fref{thetrapped:eq:modulation leq L-1} on the modulation to find that:
\be \label{variete:highsobo:modulation:eq:leqL-1 2}
\parallel (A_2^{(1)})_{s_L} \parallel_{L^2}+\parallel (A_2^{(2)})_{s_L-1} \parallel_{L^2}\leq Cb_1^{L+3-\delta_0+O(\eta)}\underset{1\leq i \leq L}{\text{sup}}b_1^{-i}|\triangle \hat{b}_i|.
\ee
We come back to the decomposition  \fref{variete:highsobo:modulation:eq:leqL-1 expression} for which we have found bounds for the terms in the right hand side in \fref{variete:highsobo:modulation:eq:leqL-1 1} and \fref{variete:highsobo:modulation:eq:leqL-1 2}, in the case where $i$ is even. Now if $i$ is odd or $i=0$ the very same computations show that they still hold, yielding:
\be \label{variete:highsobo:modulation:eq:leqL-1 1}
\begin{array}{r c l}
&\parallel \sum_{i=0}^{L-1} (\triangle \hat{Mod}_i^{(1)})_{s_L}\parallel_{L^2}+ \parallel \sum_{i=0}^{L-1} (\triangle \hat{Mod}_i^{(2)})_{s_L-1}\parallel_{L^2}\\
\leq & Cb_1^{L+3-\delta_0+O(\eta)}(\underset{1\leq i \leq L}{\text{sup}}b_1^{-i}|\triangle \hat{b}_i|+\sqrt{\triangle_r \hat{\mathcal{E}}_{s_L}}).
\end{array}
\ee
The previous bound \fref{variete:highsobo:modulation:eq:leqL-1 1} then imply the intermediate identity:
\be \label{variete:highsobo:modulation:expression}
\begin{array}{r c l}
&\int (w^{(1)}-\hat{w}^{'(1)})\mathcal{L}_{\frac{1}{\lambda}}^{s_L} \frac{(\tilde{Mod}^{(1)}-\hat{Mod}^{'(1)})_{\frac{1}{\lambda}}^{(1)}}{\lambda} + (w^{(2)}-\hat{w}^{'(2)})\mathcal{L}_{\frac{1}{\lambda}}^{s_L-1}\frac{(\tilde{Mod}^{(2)}-\hat{Mod}^{'(2)})_{\frac{1}{\lambda}}}{\lambda}\\
=& \frac{1}{\lambda} \int  (w^{(1)}-\hat{w}^{'(1)})\mathcal{L}_{\frac{1}{\lambda}}^{s_L}\triangle \hat{Mod}_{L,\frac{1}{\lambda}}^{(1)}+ (w^{(2)}-\hat{w}^{'(2)})\mathcal{L}_{\frac{1}{\lambda}}^{s_L-1}\triangle \hat{Mod}_{L,\frac{1}{\lambda}}^{(2)}  \\
&+O\left( \frac{b_1^{2L+2(1-\delta_0)+2+O(\eta)}}{\lambda^{2(s_L-s_c)+1}}(\sqrt{\triangle_r \hat{\mathcal{E}}_{s_L}}\underset{1\leq i \leq L}{\text{sup}}b_1^{-i}|\triangle \hat{b}_i|+\triangle_r \hat{\mathcal{E}}_{s_L}) \right).
\end{array}
\ee
We now have to deal with the last modulation term. We know by the improved bound for the evolution of $\triangle \hat{b}_L$, see Lemma \ref{variete:lem:improved modulation} that 
$$
b_{L,s}+(L-\alpha)b_1b_L-(\hat{b}'_{L,s}+(L-\alpha)\hat{b}'_1\hat{b}'_L)$$
is small enough up to the derivative in time of the projection of $\bos{\varepsilon}-\hat{\bos{\varepsilon}}'$ onto $\bos{H}^{*L}\chi_{B_1}\bos{\Lambda Q}$. We claim the following identity:
\be \label{variete:highsobo:modulation:L}
\begin{array}{r c l}
&\frac{1}{\lambda}\int (w^{(1)}-\hat{w}^{'(1)})\mathcal{L}_{\frac{1}{\lambda}}^{s_L}\triangle \hat{Mod}^{(1)}_{L,\frac{1}{\lambda}}+\frac{1}{\lambda}\int (w^{(2)}-\hat{w}^{'(2)})\mathcal{L}_{\frac{1}{\lambda}}^{s_L-1}\triangle \hat{Mod}^{(2)}_{L,\frac{1}{\lambda}}\\
=&\partial_t \left[O \left( \frac{b_1^{2L+2(1-\delta_0)(1+\eta)}}{\lambda^{2(s_L-s_c)}}(\triangle_r \hat{\mathcal{E}}_{s_L}+|\underset{1\leq i \leq L}{\text{sup}}b_1^{-i}|\triangle \hat{b}_i||^2 \right)\right]\\
&+O\left( \frac{b_1^{2L+1+2(1-\delta_0)(1+\eta)}}{\lambda^{2(s_L-s_c)+1}}(\triangle_r \hat{\mathcal{E}}_{s_L}+|\underset{1\leq i \leq L}{\text{sup}}b_1^{-i}|\triangle \hat{b}_i||^2 \right).
\end{array}
\ee
Once this identity is proven, we can combine it with \fref{variete:highsobo:modulation:expression} to obtain the identity \fref{variete:highsobo:eq:estimation mod} we claimed in this step 3. The rest of the proof is now devoted to the proof of \fref{variete:highsobo:modulation:L}. We define two radiations:
$$ 
\begin{array}{r c l}
\bos{\xi} &:=&  \frac{\langle \bos{H}^{L}(\bos{\varepsilon}-\hat{\bos{\varepsilon}}') , \chi_{B_0} \bos{\Lambda} \bos{Q}\rangle-  \hat{b}_L'\int \chi_{B_0}\Lambda^{(1)} Q\left( \frac{\partial S_{L+2}}{\partial b_L}-\frac{\partial \hat{S}'_{L+2}}{\partial b_L} \right)_{L-1}}{\Bigl\langle  \chi_{B_0}\Lambda^{(1)}Q,\Lambda^{(1)}Q+(-1)^{\frac{L-1}{2}}\left(\frac{\partial S_{L+2}^{(2)}}{\partial b_L}\right)_{L-1} \Bigr\rangle} \\
&&\times \left[\chi_{B_1}\left(\bos{T}_L+\frac{\partial \bos{S}_{L+1}}{\partial b_L}+\frac{\partial \bos{S}_{L+2}}{\partial b_L} \right)\right]_{\frac{1}{\lambda}}, \\
\end{array}
$$
$$
\begin{array}{r c l}
\bos{\xi}' &:=&  \frac{\langle \bos{H}^{L}(\hat{\bos{\varepsilon}}' , \chi_{\hat{B}'_0} \bos{\Lambda} \bos{Q}\rangle}{\Bigl\langle  \chi_{\hat{B}'_0} \Lambda^{(1)} Q, \Lambda^{(1)} Q +(-1)^{\frac{L-1}{2}}\left(\frac{\partial \hat{S}'_{L+2}}{\partial b_L}\right)_{L-1}\Bigr\rangle} \Bigl{[}\chi_{B_1}\left( \bos{T}_L+\frac{\partial \bos{S}_{L+1}}{\partial b_L}+\frac{\partial \bos{S}_{L+2}}{\partial b_L} \right)\\
&&-\chi_{\hat{B}'_1}\left( \bos{T}_L+\frac{\partial \hat{\bos{S}}'_{L+1}}{\partial b_L}+\frac{\partial \hat{\bos{S}}'_{L+2}}{\partial b_L} \right)\Bigr{]}_{\frac{1}{\lambda}} .
\end{array}
$$
They enjoy the bound for $i=0,1$:
\be \label{variete:highsobo:modulation:estimation radiations}
\begin{array}{r c l}
&\parallel (\xi^{(1)}+\xi^{'(1)})_{s_L+i}\parallel_{L^2}+\parallel (\xi^{(2)}+\xi^{'(2)})_{s_L-1+i}\parallel_{L^2}\\ 
\leq &C\frac{b_1^{L+(1-\delta_0)(1+\frac{3}{2}\eta)+i}}{\lambda^{s_L-s_c+i}}(\sqrt{\triangle_r \hat{\mathcal{E}}_{s_L}}+\underset{1\leq i \leq L}{\text{sup}}b_1^{-i}|\triangle \hat{b}_i|).
\end{array}
\ee
From \fref{variete:eq:improved modulation} and \fref{thetrapped:eq:improved modulation} one has:
\be \label{variete:highsobo:modulation:derivee temporelle radiations}
\partial_t (\bos{\xi}+\bos{\xi}')= \frac{1}{\lambda}\triangle \hat{\bos{Mod}}_{L,\frac{1}{\lambda}}+\bos{R},
\ee
where $\bos{R}$ is a remainder satisfying:
\be \label{variete:highsobo:modulation:estimation reste}
\parallel R^{(1)}_{s_L}\parallel_{L^2}+\parallel R^{(2)}_{s_L-1}\parallel_{L^2}\leq \frac{Cb_1^{L+(1+\frac{3}{2}\eta)(1-\delta_0)+1}}{\lambda^{s_L-s_c+1}}(\sqrt{\triangle_r \hat{\mathcal{E}_{s_L}}}+\underset{1\leq i \leq L}{\text{sup}}b_1^{-i}|\triangle \hat{b}_i|)
\ee
In the time evolution of $\bos{w}-\hat{\bos{w}}'$, \fref{variete:eq:evolution difference w}, we found a bound for almost all the terms in the right hand side in \fref{variete:eq:estimation high difference psib}, \fref{variete:highsobo:eq:estimation NL 1}, \fref{variete:highsobo:eq:estimation dhats'-1 1}, \fref{variete:highsobo:eq:estimation L 1} and \fref{variete:highsobo:modulation:eq:leqL-1 1}. With the identity \fref{variete:highsobo:modulation:derivee temporelle radiations} and the bound \fref{variete:highsobo:modulation:estimation reste} it gives the following identity:
$$
\begin{array}{r c l}
\partial_t(\bos{w}-\hat{\bos{w}}')+\bos{H}_{\frac{1}{\lambda}}(\bos{w}-\hat{\bos{w}}')&=&-\frac{1}{\lambda}\triangle \hat{\bos{Mod}}_{L,\frac{1}{\lambda}}+\bos{R}'\\
&=&-\partial_t(\bos{\xi}+\bos{\xi}')+\bos{R}'-\bos{R},
\end{array}
$$
$\bos{R}'$ being a remainder with the following size:
$$
\parallel R^{'(1)}_{s_L}\parallel_{L^2}+\parallel R^{'(2)}_{s_L-1}\parallel_{L^2}\leq C\frac{b_1^{L+1+(1-\delta_0)(1+\frac{\eta}{2})}}{\lambda^{s_L-s_c+1}}(\sqrt{\triangle_r \hat{\mathcal{E}}_{s_L}}+\underset{1\leq i \leq L}{\text{sup}}b_1^{-i}|\triangle \hat{b}_i|).
$$
With the previous relations, we perform the following integration by parts in time:
$$
\begin{array}{r c l}
&\frac{1}{\lambda}\int (w^{(1)}-\hat{w}^{'(1)})\mathcal{L}_{\frac{1}{\lambda}}^{s_L}\triangle \hat{Mod}^{(1)}_{L,\frac{1}{\lambda}}+\frac{1}{\lambda}\int (w^{(2)}-\hat{w}^{'(2)})\mathcal{L}_{\frac{1}{\lambda}}^{s_L-1}\triangle \hat{Mod}^{(2)}_{L,\frac{1}{\lambda}}\\
=&\partial_t\Bigl{[} \int (w^{(1)}-\hat{w}^{'(1)})\mathcal{L}_{\frac{1}{\lambda}}^{s_L}(\xi^{(1)}+\xi^{'(1)})+ \int (w^{(2)}-\hat{w}^{'(2)})\mathcal{L}_{\frac{1}{\lambda}}^{s_L-1}(\xi^{(2)}+\xi^{'(2)})\\
&+\frac{1}{2}\int (\xi^{(1)}+\xi^{'(1)})\mathcal{L}_{\frac{1}{\lambda}}^{s_L}(\xi^{(1)}+\xi^{'(1)})+ \frac{1}{2}\int (\xi^{(2)}+\xi^{'(2)})\mathcal{L}_{\frac{1}{\lambda}}^{s_L-1}(\xi^{(2)}+\xi^{'(2)}) \Bigr{]}\\
&- \int (w^{(1)}-\hat{w}^{'(1)})\partial_t(\mathcal{L}_{\frac{1}{\lambda}}^{s_L})(\xi^{(1)}+\xi^{'(1)})+ \int (w^{(2)}-\hat{w}^{'(2)})\partial_t(\mathcal{L}_{\frac{1}{\lambda}}^{s_L-1})(\xi^{(2)}+\xi^{'(2)})\\
&-\frac{1}{2}\int (\xi^{(1)}+\xi^{'(1)})\partial_t(\mathcal{L}_{\frac{1}{\lambda}}^{s_L})(\xi^{(1)}+\xi^{'(1)})+ \frac{1}{2}\int (\xi^{(2)}+\xi^{'(2)})\partial_t(\mathcal{L}_{\frac{1}{\lambda}}^{s_L-1})(\xi^{(2)}+\xi^{'(2)})\\
&+O\left( \frac{b_1^{2L+1+2(1-\delta_0)(1+\eta)}}{\lambda^{2(s_L-s_c)+1}}(\triangle_r \hat{\mathcal{E}}_{s_L}+|\underset{1\leq i \leq L}{\text{sup}}b_1^{-i}|\triangle \hat{b}_i||^2) \right).
\end{array}
$$
Using the degeneracy of the derivative in time of the potential \fref{linearized:eq:degenerescence scaling} one has the bound for the third and fourth terms in the previous identity:
$$
\begin{array}{r c l}
&\Bigl{|} \int (w^{(1)}-\hat{w}^{'(1)})\partial_t(\mathcal{L}_{\frac{1}{\lambda}}^{s_L})(\xi^{(1)}+\xi^{'(1)})+ \int (w^{(2)}-\hat{w}^{'(2)})\partial_t(\mathcal{L}_{\frac{1}{\lambda}}^{s_L-1})(\xi^{(2)}+\xi^{'(2)})\\
&-\frac{1}{2}\int (\xi^{(1)}+\xi^{'(1)})\partial_t(\mathcal{L}_{\frac{1}{\lambda}}^{s_L})(\xi^{(1)}+\xi^{'(1)})+ \frac{1}{2}\int (\xi^{(2)}+\xi^{'(2)})\partial_t(\mathcal{L}_{\frac{1}{\lambda}}^{s_L-1})(\xi^{(2)}+\xi^{'(2)}) \bigr{|}\\
\leq& C  \frac{b_1^{2L+1+2(1-\delta_0)(1+\eta)}}{\lambda^{2(s_L-s_c)+1}}(\triangle_r \hat{\mathcal{E}}_{s_L}+|\underset{1\leq i \leq L}{\text{sup}}b_1^{-i}|\triangle \hat{b}_i||^2).
\end{array}
$$
Hence we can write:
$$
\begin{array}{r c l}
&\frac{1}{\lambda}\int (w^{(1)}-\hat{w}^{'(1)})\mathcal{L}_{\frac{1}{\lambda}}^{s_L}\triangle \hat{Mod}^{(1)}_{L,\frac{1}{\lambda}}+\frac{1}{\lambda}\int (w^{(2)}-\hat{w}^{'(2)})\mathcal{L}_{\frac{1}{\lambda}}^{s_L-1}\triangle \hat{Mod}^{(2)}_{L,\frac{1}{\lambda}}\\
=&\partial_t\Bigl{[} \int (w^{(1)}-\hat{w}^{'(1)})\mathcal{L}_{\frac{1}{\lambda}}^{s_L}(\xi^{(1)}+\xi^{'(1)})+ \int (w^{(2)}-\hat{w}^{'(2)})\mathcal{L}_{\frac{1}{\lambda}}^{s_L-1}(\xi^{(2)}+\xi^{'(2)})\\
&+\frac{1}{2}\int (\xi^{(1)}+\xi^{'(1)})\mathcal{L}_{\frac{1}{\lambda}}^{s_L}(\xi^{(1)}+\xi^{'(1)})+ \frac{1}{2}\int (\xi^{(2)}+\xi^{'(2)})\mathcal{L}_{\frac{1}{\lambda}}^{s_L-1}(\xi^{(2)}+\xi^{'(2)}) \Bigr{]}\\
&+O\left( \frac{b_1^{2L+1+2(1-\delta_0)(1+\eta)}}{\lambda^{2(s_L-s_c)+1}}(\triangle_r \hat{\mathcal{E}}_{s_L}+|\underset{1\leq i \leq L}{\text{sup}}b_1^{-i}|\triangle \hat{b}_i||^2 \right).
\end{array}
$$
We now take the previous equation, inject the bound \fref{variete:highsobo:modulation:estimation radiations} for the terms integrated in time, it gives the intermediate identity \fref{variete:highsobo:modulation:L} we had to prove. 

\end{proof}


To control the local term in \fref{variete:eq:high sobo}, we study a Morawetz type quantity localized near the origin. We recall that $\phi_A$ is defined by \fref{thetrapped:eq:def phiA}. We define the following quantity:
\begin{equation}\label{thetrapped:eq:definition M}
\begin{array}{r c l}
\triangle \mathcal{M}&=& -\int \left[\nabla \phi_A .\nabla (\varepsilon^{(1)}-\hat{\varepsilon}^{'(1)})_{s_L-1}+\frac{(1-\delta)\Delta \phi_A}{2}(\varepsilon^{(1)}-\hat{\varepsilon}^{'(1)})_{s_L-1}\right] \\
&&\times (\varepsilon^{(2)}-\hat{\varepsilon}^{'(2)})_{s_L-1} .
\end{array}
\end{equation}
$\triangle \mathcal{M}$ is controlled by the high Sobolev norm of the difference:
\begin{equation}\label{variete:eq:morawetz controle par high sobo}
|\triangle \mathcal{M} |\leq  C(A,M) \triangle \mathcal{E}_{s_L}
\end{equation}

At the linear level of the dynamics \fref{variete:eq:evolution difference epsilon} of $\bos{\varepsilon}-\hat{\bos{\varepsilon}}'$, this quantity controls the local term $\triangle \hat{\mathcal{E}}_{s_L,\text{loc}}$. Indeed, from Lemma \ref{thetrapped:lem:control du morawetz} one has:

\begin{equation} \label{variete:morawetz:eq:controle lineaire}
\begin{array}{r c l}
&  \int [\nabla \phi_A .\nabla (\varepsilon^{(1)}-\hat{\varepsilon}^{'(1)})_{s_L-1}+\frac{(1-\delta)\Delta \phi_A}{2}(\varepsilon^{(1)}-\hat{\varepsilon}^{'(1)})_{s_L-1}] (\mathcal{L}(\varepsilon^{(1)}-\hat{\varepsilon}^{'(1)})_{s_L-1}) \\
&-\int [\nabla \phi_A .\nabla (\varepsilon^{(2)}-\hat{\varepsilon}^{'(2)})_{s_L-1}+\frac{(1-\delta)\Delta \phi_A}{2}(\varepsilon^{(2)}-\hat{\varepsilon}^{'(2)})_{s_L-1}] (\varepsilon^{(2)}-\hat{\varepsilon}^{'(2)})_{s_L-1}\\
\geq &\frac{\delta}{2N^{\delta}} \triangle \mathcal{E}_{s_L,loc} -\frac{C(M)}{A^{\delta}}\triangle \mathcal{E}_{s_L} .
\end{array}
\end{equation}

This control remains in the full non linear equation. We have the following result:

\begin{lemma}[Control of the local term by a Morawetz type identity]\label{variete:lem:morawetz}

One has the following lower bound on the evolution of $\triangle \mathcal{M}$:

\begin{equation} \label{variete:eq:morawetz}
\begin{array}{r c l}
\frac{d}{ds}\triangle \mathcal{M}&\geq & \frac{\delta}{2N^{\delta}} \triangle \mathcal{E}_{s_L,\text{loc}}-\frac{C(M)}{A^{\delta}}\triangle \mathcal{E}_{s_L}\\
&& -C(A)\sqrt{\triangle \mathcal{E}_{s_L}}b_1^{L+1+(1-\delta_0)+O(\eta)}(\sqrt{\triangle_r \mathcal{E}_{\sigma}}+\underset{1\leq i \leq L}{\text{sup}}b_1^{-i}|\triangle \hat{b}_i|) .
\end{array}
\end{equation}

\end{lemma}

\begin{proof}[Proof of Lemma \ref{variete:lem:morawetz}]

To prove the identity of the lemma, we first compute the time evolution of $\triangle \mathcal{M}$, use the control \fref{variete:morawetz:eq:controle lineaire} obtained at the linear level, and show that the other terms are negligible. The time evolution of $\triangle \mathcal{M}$ is:

\be \label{variete:morawetz:eq:expression dtM}
\begin{array}{r c l}
&\frac{d}{ds}\triangle \mathcal{M}\\
=& -\int \nabla \phi_A .\nabla \Bigl{[} (\frac{\lambda_s}{\lambda}\Lambda^{(1)}(\varepsilon^{(1)}-\hat{\varepsilon}^{'(1)})+\varepsilon^{(2)}-\hat{\varepsilon}^{'(2)}-\tilde{\psi}_b^{(1)}+\tilde{\psi}_{\hat{b}'}^{(1)}-\tilde{Mod}^{(1)}\\
&+\hat{Mod}^{'(1)}+(\frac{d\hat{s}'}{ds}-1)(\tilde{\psi}_{\hat{b}'}^{(1)}-\hat{\varepsilon}^{'(2)})  \Bigr{]}_{s_L-1} (\varepsilon^{(2)}-\hat{\varepsilon}^{'(2)})_{s_L-1} \\
&- \int \frac{(1-\delta)\Delta \phi_A}{2}\Bigl{[} (\frac{\lambda_s}{\lambda}\Lambda^{(1)}(\varepsilon^{(1)}-\hat{\varepsilon}^{'(1)})+\varepsilon^{(2)}-\hat{\varepsilon}^{'(2)}-\tilde{\psi}_b^{(1)}+\tilde{\psi}_{\hat{b}'}^{(1)}-\tilde{Mod}^{(1)}\\
&+\hat{Mod}^{'(1)}+(\frac{d\hat{s}'}{ds}-1)(\tilde{\psi}_{\hat{b}'}-\hat{\varepsilon}^{'(2)})  \Bigr{]}_{s_L-1} (\varepsilon^{(2)}-\hat{\varepsilon}^{'(2)})_{s_L-1} \\
&-\int \nabla \phi_A .\nabla (\varepsilon^{(1)}-\hat{\varepsilon}^{'(1)})_{s_L-1}\Bigl{[}-\mathcal{L}(\varepsilon^{(1)}-\hat{\varepsilon}^{'(1)})-\frac{\lambda_s}{\lambda}\Lambda^{(2)}(\varepsilon^{(2)}-\hat{\varepsilon}^{'(2)})\\
&-\tilde{\psi}_b^{(2)}+\tilde{\psi}_{\hat{b}'}^{(2)}-\tilde{Mod}^{(2)}+\hat{Mod}^{'(2)}\\
&+L-\hat{L}'+NL-\hat{NL}'+(\frac{d\hat{s}'}{ds}-1)(\tilde{\psi}_{\hat{b}'}^{(2)}+\mathcal{L}\hat{\varepsilon}^{'(1)}-\hat{L}'-\hat{NL}')\Bigr{]}_{s_L-1} \\
&- \int \frac{(1-\delta)\Delta \phi_A}{2}(\varepsilon^{(1)}-\hat{\varepsilon}^{'(1)})_{s_L-1} \Bigl{[}-\mathcal{L}(\varepsilon^{(1)}-\hat{\varepsilon}^{'(1)})-\frac{\lambda_s}{\lambda}\Lambda^{(2)}(\varepsilon^{(2)}-\hat{\varepsilon}^{'(2)})\\
&-\tilde{\psi}_b^{(2)}+\tilde{\psi}_{\hat{b}'}^{(2)}-\tilde{Mod}^{(2)}+\hat{Mod}^{'(2)}\\
&+L-\hat{L}'+NL-\hat{NL}'+(\frac{d\hat{s}'}{ds}-1)(\tilde{\psi}_{\hat{b}'}^{(2)}+\mathcal{L}\hat{\varepsilon}^{'(1)}-\hat{L}'-\hat{NL}')\Bigr{]}_{s_L-1} .
\end{array}
\ee

We now compute everything in the right hand side. The linear part produces exactly the control we want thanks to the identity \fref{variete:morawetz:eq:controle lineaire}:

\begin{equation} \label{variete:morawetz:eq:estimation lineaire}
\begin{array}{r c l}
&  \int [\nabla \phi_A .\nabla (\varepsilon^{(1)}-\hat{\varepsilon}^{'(1)})_{s_L-1}+\frac{(1-\delta)\Delta \phi_A}{2}(\varepsilon^{(1)}-\hat{\varepsilon}^{'(1)})_{s_L-1}] (\mathcal{L}(\varepsilon^{(1)}-\hat{\varepsilon}^{'(1)})_{s_L-1}) \\
&-\int [\nabla \phi_A .\nabla (\varepsilon^{(2)}-\hat{\varepsilon}^{'(2)})_{s_L-1}+\frac{(1-\delta)\Delta \phi_A}{2}(\varepsilon^{(2)}-\hat{\varepsilon}^{'(2)})_{s_L-1}] (\varepsilon^{(2)}-\hat{\varepsilon}^{'(2)})_{s_L-1}\\
\geq &\frac{\delta}{2N^{\delta}} \triangle \mathcal{E}_{s_L,loc} -\frac{C(M)}{A^{\delta}}\triangle \mathcal{E}_{s_L} .
\end{array}
\end{equation}
Now $\phi_A$ is of compact support. Hence by integrating by parts and using coercivity we can control the scale changing term:
\be \label{variete:morawetz:eq:estimation lambdavarepsilon}
\begin{array}{r c l}
&\int [\nabla \phi_A .\nabla(\frac{\lambda_s\Lambda^{(1)}(\varepsilon^{(1)}-\hat{\varepsilon}^{'(1)})_{s_L-1}}{\lambda})+\frac{(1-\delta)\Delta \phi_A\lambda_s \Lambda^{(1)}(\varepsilon^{(1)}-\hat{\varepsilon}^{'(1)})_{s_L-1}}{2\lambda}](\varepsilon^{(2)}-\hat{\varepsilon}^{'(2)})_{s_L-1} \\
&+\int [\nabla \phi_A. \nabla(\varepsilon^{(1)}-\hat{\varepsilon}^{'(1)})_{s_L-1})+\frac{(1-\delta)\Delta \phi_A(\varepsilon^{(1)}-\hat{\varepsilon}^{'(1)})_{s_L-1}}{2}]\frac{\lambda_s(\Lambda^{(2)}\varepsilon^{(2)}-\hat{\varepsilon}^{'(2)})_{s_L-1}}{\lambda} \\
= & O(b_1 C(A) \triangle \mathcal{E}_{s_L}) ,
\end{array}
\ee
As we work on a compact set, we do not see the bad tail of the error terms. Hence their contribution is:
\be \label{variete:morawetz:eq:estimation psib}
\begin{array}{r c l}
&\left|\int [\nabla \phi_A .\nabla(\tilde{\psi}^{(1)}_b-\tilde{\psi}_{\hat{b}'})_{s_L-1})+\frac{(1-\delta)\Delta \phi_A}{2}(\tilde{\psi}^{(1)}_b-\tilde{\psi}_{\hat{b}'}^{(1)})_{s_L-1})](\varepsilon^{(2)}-\hat{\varepsilon}^{'(2)})_{s_L-1}\right| \\
&+\left|\int [\nabla \phi_A .\nabla(\varepsilon^{(1)}-\hat{\varepsilon^{'(1)}}_{s_L-1})+\frac{(1-\delta)\Delta \phi_A}{2}(\varepsilon^{(1)}-\hat{\varepsilon}^{'(1)})_{s_L-1}](\tilde{\psi}^{(2)}_b-\tilde{\psi}_{\hat{b}'}^{(2)})_{s_L-1}\right|\\
\leq& C(A) \sqrt{\triangle \mathcal{E}}_{\sigma} b_1^{L+3} \underset{1\leq i \leq L}{\text{sup}}b_1^{-i}|\triangle \hat{b}_i|.
\end{array}
\ee
For the small linear terms we use the decomposition:
$$
L-\hat{L}'=p(\tilde{Q}_{b}^{(1)(p-1)}-\tilde{Q}_{\hat{b}'}^{(1)(p-1)})\varepsilon^{(1)}+p(\tilde{Q}_{\hat{b}}^{(1)(p-1)}-Q^{p-1})(\varepsilon^{(1)}-\hat{\varepsilon}^{'(1)}).
$$
From \fref{variete:eq:asymptotique diff Qp-j} one has for the first term in this decomposition:
$$
\begin{array}{r c l}
&\Bigl{|}\int [\nabla \phi_A .\nabla(\varepsilon^{(1)}-\hat{\varepsilon}^{'(1)})_{s_L-1}+\frac{(1-\delta)\Delta \phi_A}{2}(\varepsilon^{(1)}-\hat{\varepsilon}^{'(1)})_{s_L-1}]\\
&\times((\tilde{Q}_{b}^{(1)(p-1)}-\tilde{Q}_{\hat{b}'}^{(1)(p-1)})\varepsilon^{(1)})_{s_L-1}\Bigr{|}\\
\leq & C(A)\sqrt{\triangle \hat{\mathcal{E}}_{s_L}}b_1^{L+2-\delta_0+O(\eta)}\underset{1\leq i \leq L}{\text{sup}}b_1^{-i}|\triangle \hat{b}_i|.
\end{array}
$$
For the second term, as $|\partial_y^k(\tilde{Q}_{\hat{b}'}^{(1)(p-1)}-Q^{p-1)}|\leq C(A,k)b_1$ because of the compactness of the support, one gets:
$$
\begin{array}{r c l}
&\Bigl{|}\int [\nabla \phi_A .\nabla(\varepsilon^{(1)}-\hat{\varepsilon}^{'(1)})_{s_L-1}+\frac{(1-\delta)\Delta \phi_A}{2}(\varepsilon^{(1)}-\hat{\varepsilon}^{'(1)})_{s_L-1}]\\
&\times ((\tilde{Q}_{\hat{b}}^{(1)(p-1)}-Q^{p-1})(\varepsilon^{(1)}-\hat{\varepsilon}^{'(1)})_{s_L-1}\Bigr{|}\\
\leq & C(A)\triangle \hat{\mathcal{E}}_{s_L}b_1.
\end{array}
$$
Hence the contribution of the small linear terms is:
\be \label{variete:morawetz:eq:estimation L}
\begin{array}{r c l}
&\left|\int [\nabla \phi_A .\nabla(\varepsilon^{(1)}-\hat{\varepsilon}^{'(1)})_{s_L-1})+\frac{(1-\delta)\Delta \phi_A}{2}(\varepsilon^{(1)}-\hat{\varepsilon}^{'(1)})_{s_L-1}](L-\hat{L}')_{s_L-1}\right|\\
\leq& C(A)b_1^{2L+3-2\delta_0+O(\eta)}\sqrt{\triangle _r \hat{\mathcal{E}}_{s_L}} (\sqrt{\triangle _r\hat{\mathcal{E}}_{s_L}}+ \underset{1\leq i \leq L}{\text{sup}}b_1^{-i}|\triangle \hat{b}_i|) .
\end{array}
\ee
For the nonlinear terms we use the bound \fref{variete:highsobo:eq:estimation NL 4} we derived in the proof of the monotonicity of the adapted high Sobolev norm to find:
\be \label{variete:morawetz:eq:estimation NL}
\begin{array}{r c l}
&\left|\int [\nabla \phi_A .\nabla(\varepsilon^{(1)}-\hat{\varepsilon}^{'(1)})_{s_L-1}+\frac{(1-\delta)\Delta \phi_A}{2}(\varepsilon^{(1)}-\hat{\varepsilon}^{'(1)})_{s_L-1}](NL-\hat{NL}')_{s_L-1}\right|\\
\leq & C(A)\sqrt{\mathcal{E}_{s_L}}b_1^{L+2-\delta_0+\frac{\alpha}{2L}+O\left( \frac{\sigma-s_c}{L},\eta\right)}(\sqrt{\triangle_r \hat{\mathcal{E}}_{s_L}}+\sqrt{\triangle_r \hat{\mathcal{E}}_{\sigma}}+\underset{1\leq i \leq L}{\text{sup}}b_1^{-i}|\triangle \hat{b}_i|).
\end{array}
\ee
Using the bounds \fref{thetrapped:highsobo:eq:estimation psib}, \fref{thetrapped:highsobo:eq:estimation L}, \fref{thetrapped:highsobo:eq:estimation NL} we derived in the proof of Proposition \ref{trappedregime:pr:high sobo} plus the assumption \fref{variete:eq:estimation supplementaire varepsilon'} and the bound \fref{variete:modulation:eq:dhats (triangle b)} on $\frac{d\hat{s}'}{ds}-1$ one gets for the terms involving the evolution of the time difference:
\be \label{variete:morawetz:eq:estimation dhats'-1}
\begin{array}{r c l}
&\Bigl{|}(\frac{d\hat{s}'}{ds}-1)\int [\nabla \phi_A .\nabla(\varepsilon^{(1)}-\hat{\varepsilon}^{'(1)})_{s_L-1}+\frac{(1-\delta)\Delta \phi_A}{2}(\varepsilon^{(1)}-\hat{\varepsilon}^{'(1)})_{s_L-1}]\\
&\times(\tilde{\psi}_{\hat{b}'}^{(2)}+\mathcal{L}\hat{\varepsilon}^{'(1)}-\hat{L}'+\hat{NL}')_{s_L-1}\Bigr{|}\\
&+\left|(\frac{d\hat{s}'}{ds}-1)\int [\nabla \phi_A .\nabla(\tilde{\psi}_{\hat{b}'}^{(1)}-\hat{\varepsilon}^{'(2)})_{s_L-1}+\frac{\Delta \phi_A(\tilde{\psi}_{\hat{b}'}^{(1)}-\hat{\varepsilon}^{'(2)})_{s_L-1}}{2}](\varepsilon^{(2)}-\hat{\varepsilon}^{'(2)})_{s_L-1}\right|\\
\leq & C(A)\sqrt{\triangle \mathcal{E}_{s_L}}b_1^{L+1+(1-\delta_0)(1+\frac{3}{2}\eta)}(\underset{1\leq i \leq L}{\text{sup}}b_1^{-i}|\triangle \hat{b}_i|+b_1^{L+(1-\delta_0)(1+\eta)}\sqrt{\triangle_r \hat{\mathcal{E}}_{s_L}}).
\end{array}
\ee
To finish the proof it remains to estimate the modulation terms. We just compute for one difference of modulation terms located in the second coordinate, that is to say a term of the form:
$$
\int [\nabla \phi_A .\nabla(\varepsilon^{(1)}-\hat{\varepsilon}^{'(1)})_{s_L-1}+\frac{(1-\delta)\Delta \phi_A}{2}(\varepsilon^{(1)}-\hat{\varepsilon}^{'(1)})_{s_L-1}]\triangle \hat{Mod}^{(2)}_i)_{s_L-1}
$$
where we recall that $\triangle \hat{Mod}_i$ is defined by \fref{variete:eq:def triangle hatmodi}. We suppose also that $i$ is odd. We claim that the same computations yield the same result for the other modulation terms. As we work on a compact support, we do not see the two cut off $\chi_{B_1}$ and $\chi_{\hat{B}'_1}$. So the profile $T_i$ cancels as $(T_i)_{s_L-1}=0$. Therefore the quantity we have to estimate simplifies into:
$$
\begin{array}{r c l}
&(\triangle \hat{Mod}^{(2)}_i)_{s_L-1}\\
=& (b_{i,s}+(i-\alpha)b_1b_i-b_{i+1}-(\hat{b}'_{i,s}+\frac{d\hat{s}'}{ds}((i-\alpha)\hat{b}'_1\hat{b}'_i+\hat{b}'_{i+1}))) \underset{j=i+1, \ j \ \text{odd}}{\overset{L+2}{\sum}} \frac{\partial S_j}{\partial b_i} \\
&+(\hat{b}'_{i,s}+\frac{d\hat{s}'}{ds}((i-\alpha)\hat{b}'_1\hat{b}'_i+\hat{b}'_{i+1})) \underset{j=i+1, \ j \ \text{odd}}{\overset{L+2}{\sum}} \frac{\partial S_j}{\partial b_i}-\frac{\partial \hat{S}'_j}{\partial b_i}
\end{array}
$$
for $y\leq 2A$. Therefore, using the modulation bounds \fref{variete:modulation:eq:leqL-1 (triangle b)} and \fref{variete:modulation:eq:L (triangle b)} one gets that the contribution of this term is:
$$
\begin{array}{r c l}
&\left| \int [\nabla \phi_A .\nabla(\varepsilon^{(1)}-\hat{\varepsilon}^{'(1)})_{s_L-1}+\frac{(1-\delta)\Delta \phi_A}{2}(\varepsilon^{(1)}-\hat{\varepsilon}^{'(1)})_{s_L-1}](\triangle \hat{Mod}^{(2)}_i)_{s_L-1} \right| \\
\leq & C(A) \sqrt{\triangle \hat{\mathcal{E}}_{s_L}}b_1^{L+2-\delta_0+O(\eta)}(\underset{1\leq i \leq L}{\text{sup}}b_1^{-i}|\triangle \hat{b}_i|+\sqrt{\triangle_r \hat{\mathcal{E}}_{s_L}}).
\end{array}
$$
For the other terms involved in the modulation terms, the same reasoning yield the same estimate, hence:
\be \label{variete:morawetz:eq:estimation mod}
\begin{array}{r c l}
&\Bigl{|}(\int [\nabla \phi_A .\nabla(\varepsilon^{(1)}-\hat{\varepsilon}^{'(1)})_{s_L-1}+\frac{(1-\delta)\Delta \phi_A}{2}]\\
&\times(\tilde{Mod}^{(2)}-\hat{Mod}^{'(2)})_{s_L-1}\Bigr{|}\\
&+\Bigl{|}\int [\nabla \phi_A .\nabla(\tilde{Mod}^{(1)}-\hat{Mod}^{'(1)})_{s_L-1}+\frac{(1-\delta)\Delta \phi_A}{2}(\tilde{Mod}^{(1)}-\hat{Mod}^{'(1)})_{s_L-1}]\\
&\times (\varepsilon^{(2)}-\hat{\varepsilon}^{'(2)})_{s_L-1}\Bigr{|}\\
\leq & C(A) \sqrt{\triangle \hat{\mathcal{E}}_{s_L}}b_1^{L+2-\delta_0+O(\eta)}(\underset{1\leq i \leq L}{\text{sup}}b_1^{-i}|\triangle \hat{b}_i|+\sqrt{\triangle_r \hat{\mathcal{E}}_{s_L}}).
\end{array}
\ee
Now, gathering together all the bounds we have proven: the control on the linear terms \fref{variete:morawetz:eq:estimation mod}, the bounds on the error terms \fref{variete:morawetz:eq:estimation psib}, on the scale changing terms \fref{variete:morawetz:eq:estimation lambdavarepsilon}, on the small linear and non linear terms \fref{variete:morawetz:eq:estimation L} and \fref{variete:morawetz:eq:estimation NL}, and the time difference terms \fref{variete:morawetz:eq:estimation dhats'-1} and on the modulation terms \fref{variete:morawetz:eq:estimation mod} one gets the bound \fref{variete:eq:morawetz} claimed in the lemma.
\end{proof}


\subsubsection{Study of the coupled dynamical system, end of the proof of Proposition \fref{variete:prop:parametres lipschitz}}

So far in this section, we introduced new variables $(\hat{\bos{\varepsilon}}',\hat{b}')$ that we could compare with the other solution $(\bos{\varepsilon},b)$. We then computed the time evolution of the difference of relevant quantities. In the Lemmas \ref{variete:lem:modulation diff} and \ref{variete:lem:improved modulation} we calculated the time evolution of the difference of the parameters, and in Lemma \ref{variete:lem:high sobo} we related the time evolution of the adapted high Sobolev norm of the difference of errors to the difference of parameters. The two other Lemmas \ref{variete:lem:low sobo} and \ref{variete:lem:morawetz} for the low Sobolev norm and for the Morawetz quantity are just additional tools to close an estimate for the previous norm.\\
\\
Thus, at this point we found a quite complicated coupled dynamical system for the differences of the variables of the two solutions $b_i-\hat{b}'_i$ for $1\leq i \leq L$, $s-\hat{s}'$ and $\bos{\varepsilon}-\hat{\bos{\varepsilon}}'$. In the following lemma we analyse this dynamical system, and find that it is only weakly coupled. Namely: the difference of the unstable parameters evolves according to an repelling linear dynamic plus a smaller feedback from the difference of the stable parameters and errors, the difference of stable parameters evolves according to an attractive linear dynamic plus a smaller feedback from the difference of the unstable parameters and errors and the dynamics of the difference of the errors is also stable.

\begin{lemma} \label{variete:lem:bootstrap differences}
For any $0<\kappa \ll 1$, there exists universal constants $\tilde{C}$, $(C_i)_{\ell +1\leq i \leq L}$, $C_1$, $C_{\triangle \hat{s}}$, $0<\kappa_1<\kappa$, $0<\kappa_i<\kappa$ for $\ell+1 \leq i\leq L$ and $\bar{s}$ such that if $s_0\geq \bar{s}$ the following holds for $s_0\leq s$:
\begin{itemize}
\item[(i)] \emph{Estimates on the stable parameters:} for $\ell+1\leq i \leq L$ one has
\be \label{variete:eq:bootstrap V1}
\begin{array}{r c l}
|\triangle \hat{V}_1(s)| &\leq& C_1\left( \underset{\ell+1\leq i \leq L}{\text{sup}}|\triangle \hat{U}_i(s_0)| + |\triangle \hat{V}_1(s_0)|+\sqrt{\triangle_r \hat{\mathcal{E}}_{\sigma}(s_0)}+\sqrt{\triangle_r \hat{\mathcal{E}}_{s_L}(s_0)}\right)\\
&&+ \kappa_1 \underset{s_0\leq s' \leq s, \ 2\leq i \leq \ell}{\text{sup}}|\triangle \hat{V}_i|,
\end{array}
\ee
\be \label{variete:eq:bootstrap Ui}
\begin{array}{r c l}
|\triangle \hat{U}_i(s)| &\leq& C_i\left( \underset{\ell+1\leq i \leq L}{\text{sup}}|\triangle \hat{U}_i(s_0)| + |\triangle \hat{V}_1(s_0)|+\sqrt{\triangle_r \hat{\mathcal{E}}_{\sigma}(s_0)}+\sqrt{\triangle_r \hat{\mathcal{E}}_{s_L}(s_0)}\right)\\
&&+ \kappa_i \underset{s_0\leq s' \leq s, \ 2\leq i \leq \ell}{\text{sup}}|\triangle \hat{V}_i|, \ \ \text{for} \ \ell+1\leq i \leq L.
\end{array}
\ee
For the difference of renormalized times there holds:
\be \label{variete:eq:bootstrap s-hats'}
\begin{array}{r c l}
\frac{|s-\hat{s}'(s)|}{s\text{log}(s)} &\leq& C_{\triangle \hat{s}}\Bigl{(} \underset{\ell+1\leq i \leq L}{\text{sup}}|\triangle \hat{U}_i(s_0)| + |\triangle \hat{V}_1(s_0)|+\sqrt{\triangle_r \hat{\mathcal{E}}_{\sigma}(s_0)}+\sqrt{\triangle_r \hat{\mathcal{E}}_{s_L}(s_0)}\\
&& + \underset{s_0\leq s' \leq s, \ 2\leq i \leq \ell}{\text{sup}}|\triangle \hat{V}_i|\Bigr{)}.
\end{array}
\ee
\item[(ii)] \emph{Estimates on the difference of errors:} One has the bounds:
\be \label{variete:eq:bootstrap triangle varepsilon low}
\begin{array}{r c l}
\sqrt{\triangle_r \hat{\mathcal{E}}_{\sigma}(s)}  &\leq& \tilde{C}\Bigl{(} \underset{\ell+1\leq i \leq L}{\text{sup}}|\triangle \hat{U}_i(s_0)| + |\triangle \hat{V}_1(s_0)|+\sqrt{\triangle_r \hat{\mathcal{E}}_{\sigma}(s_0)}+\sqrt{\triangle_r \hat{\mathcal{E}}_{s_L}(s_0)}\\
&&+  \underset{s_0\leq s' \leq s, \ 2\leq i \leq \ell}{\text{sup}}|\triangle \hat{V}_i|\Bigr{)},
\end{array}
\ee
\be \label{variete:eq:bootstrap triangle varepsilon high}
\begin{array}{r c l}
\sqrt{\triangle_r \hat{\mathcal{E}}_{s_L}(s)} &\leq& \tilde{C}\Bigl{(} \underset{\ell+1\leq i \leq L}{\text{sup}}|\triangle \hat{U}_i(s_0)| + |\triangle \hat{V}_1(s_0)|+\sqrt{\triangle_r \hat{\mathcal{E}}_{\sigma}(s_0)}+\sqrt{\triangle_r \hat{\mathcal{E}}_{s_L}(s_0)}\\
&&+  \underset{s_0\leq s' \leq s, \ 2\leq i \leq \ell}{\text{sup}}|\triangle \hat{V}_i|\Bigr{)}.
\end{array}
\ee
\end{itemize}

\end{lemma}

\begin{proof}[Proof of Lemma \fref{variete:lem:bootstrap differences}]
The proof is based on a bootstrap technique: we inject the bounds of the lemma in the evolution equations, and find that they can be bootstrapped. From now on we fix the constants of the Lemma \fref{variete:lem:bootstrap differences}: $\kappa$ is small $\kappa\ll 1$, and the $C$'s are large. We just allow us to increase $\bar{s}$ if necessary. The bounds of the lemma are verified at least on a small interval of time $[s_0,s']$, so we define $s_1$ as the supremum of times $s'$ such that all the bounds of the Lemma are verified on $[s_0,s_1[$. If $s_1=+\infty$ the lemma is proven. So we now assume $s_1<+\infty$ and look for a contradiction.\\
\\
We recall that we have the following relation: $ \tilde{\eta}\ll \eta \ll 1 $. We first state the following identity:
\be \label{variete:bootstrap:eq:expression trianglehatbi}
\begin{array}{r c l}
\triangle \hat{b}_i= b_i-\hat{b}_i'=\frac{c_i}{s^i}+\frac{U_i}{s^i}-\frac{c_i}{(\hat{s}')^i}-\frac{\hat{U}'_i}{s^i}=c_i\frac{(\hat{s}')^i-s^i}{s^i (\hat{s}')^i}+\frac{\triangle \hat{U}_i}{s^i}.
\end{array}
\ee
To ease notations, we let:
$$
D_{\text{stab}}(s_0)= \underset{\ell+1\leq i \leq L}{\text{sup}}|\triangle \hat{U}_i(s_0)| + |\triangle \hat{V}_1(s_0)|+\sqrt{\triangle_r \hat{\mathcal{E}}_{\sigma}(s_0)}+\sqrt{\triangle_r \hat{\mathcal{E}}_{s_L}(s_0)}.
$$

\underline{Step 1:} \emph{the time difference}. We recall that because the two solutions we are studying are in the trapped regime one has: $b_1\sim s^{-1}$ and $|U_i|+|\hat{U}_i|\lesssim s^{-\tilde{\eta}}$. We inject the identity \fref{variete:bootstrap:eq:expression trianglehatbi} in the time evolution of $s-\hat{s}'$ given by \fref{variete:modulation:eq:dhats (triangle b)}:
\be \label{variete:bootstrap:eq:expression dhats'}
\frac{d}{ds}\left(\frac{s-\hat{s}'}{s} \right)=O\left(\frac{|s-\hat{s}'|}{s^{2+\tilde{\eta}}}+\frac{|\triangle \hat{U}_1|}{s}+\frac{|\triangle U|+\sqrt{\triangle_r \hat{\mathcal{E}}_{s_L}}}{s^{L+(1-\delta_0)(1+\frac{\eta}{2})+1}} \right),
\ee
the constant in the $O()$ being independent on the constants of the Lemma we are proving. We integrate till $s_1$. As $\triangle \hat{U}_1$ is a linear combination of the $\triangle \hat{V}_i$ for $1\leq i \leq \ell$, injecting the bounds \fref{variete:eq:bootstrap Ui}, \fref{variete:eq:bootstrap V1}, \fref{variete:eq:bootstrap s-hats'} and \fref{variete:eq:bootstrap triangle varepsilon high} gives:
$$
\begin{array}{r c l}
&\left| \int_{s_0}^{s_1}O\left(\frac{|s-\hat{s}'|}{s^{2+\tilde{\eta}}}+\frac{|\triangle \hat{U}_1|}{s}+\frac{|\triangle U|+\sqrt{\triangle_r \hat{\mathcal{E}}_{s_L}}}{s^{L+(1-\delta_0)(1+\frac{\eta}{2})+1}} \right) \right| \\
\leq& C\text{log}(s_1) \left( \frac{\text{log}(s_0)}{s_0^{\tilde{\eta}}}C_{\triangle \hat{s}}+C_1+\frac{\tilde{C}+\sum_{\ell+1}^L C_i}{s_0^{L+(1-\delta_0)(1+\frac{\eta}{2})}} \right)D_{\text{stab}}(s_0)\\
&+C\text{log}(s_1)\left( 1+\frac{\tilde{C}}{s_0^{L+(1-\delta_0)(1+\frac{\eta}{2})}}+ \frac{\text{log}(s_0)}{s_0^{\tilde{\eta}}}C_{\triangle \hat{s}} \right)  \underset{s_0\leq s' \leq s, \ 2\leq i \leq \ell}{\text{sup}}|\triangle \hat{V}_i| ,
\end{array}
$$
for some constant $C$ independent of the bootstrap constants (the $\kappa_i$'s do not appear as they are small, $\kappa_i\ll 1$). Now we recall that at initial time $\hat{s}'(s_0)=s_0$. Hence when integrating \fref{variete:bootstrap:eq:expression dhats'}:
$$
\begin{array}{r c l}
|s_1-\hat{s}'(s_1))|& \leq &  s_1\text{log}(s_1) \left( CC_1+O(s_0^{-\frac{\tilde{\eta}}{2}})(C_{\triangle \hat{s}}+\sum_{\ell+1}^L C_i+\tilde{C}) \right)D_{\text{stab}}(s_0)\\
&&+s_1\text{log}(s_1)\left( C+O(s_0^{-\frac{\tilde{\eta}}{2}})(\tilde{C}+ C_{\triangle \hat{s}}) \right)  \underset{s_0\leq s' \leq s, \ 2\leq i \leq \ell}{\text{sup}}|\triangle \hat{V}_i|.
\end{array}
$$
It means, as $C_1$ is a big constant and $1\ll s_0$, that the inequality \fref{variete:eq:bootstrap s-hats'} is strict at time $s_1$  provided:
\be \label{variete:bootstrap:condition hats'}
C_{\triangle \hat{s}}> CC_1+O(s_0^{-\frac{\tilde{\eta}}{2}})(\tilde{C}+\sum_{\ell+1}^L C_i),
\ee
where the constant $C$ and the constants hidden in the $O()$ are independent of the other constants of the Lemma we are proving.\\

\underline{Step 2:} \emph{the parameter $V_1$}. The identity \fref{variete:eq:identity hatb'} implies that for $1\leq i \leq \ell$:
\be
\begin{array}{r c l}
&b_{i,s}+(i-\alpha)b_1b_i-b_{i+1}-(\hat{b}'_{i,s}+\frac{d\hat{s}'}{ds}((i-\alpha)\hat{b}'_1\hat{b}'_i-\hat{b}'_{i+1}))\\
=& \frac{1}{s^i}(\triangle \hat{U}_{i,s}-\frac{(A_{\ell}\triangle \hat{U})_i}{s}+O(s^{-1-\tilde{\eta}}(\frac{|\hat{s}'-s|}{s}+|\triangle \hat{U}|+|\frac{d\hat{s}'}{ds}-1|))).
\end{array}
\ee
We now inject it in \fref{variete:modulation:eq:leqL-1 (triangle b)} using the bound \fref{variete:modulation:eq:dhats (triangle b)} on $\frac{d\hat{s}'}{ds}-1$ to find:
\be \label{variete:eq:expression dttriangleUi}
\triangle \hat{U}_{i,s}=\frac{(A_{\ell}\triangle \hat{U})_i}{s}+O(s^{-1-\tilde{\eta}}(|\frac{\hat{s}'-s}{s}|+|\triangle \hat{U}|))+O(s^{-L-(1-\delta_0)(1+\frac{\eta}{2})+i}\sqrt{\triangle_r \hat{\mathcal{E}}_{s_L}}),
\ee
the constants in the $O()$ being independent of the constants of the Lemma we are proving. As $\triangle \hat{V}_1$ is a linear combination of the $\triangle \hat{U}_i$ for $1\leq i \leq \ell$ only, see \fref{thetrapped:eq:def Vi}, and because of the shape of the matrix $A_{\ell}$, see \fref{linearized:eq:diagonalisation}, the previous identity yields:
$$
\triangle \hat{V}_{1,s}=\frac{-\triangle \hat{V}_1}{s}+\frac{q_1\triangle U_{\ell+1}}{s}+O(|\frac{\hat{s}'-s}{s^{2+\tilde{\eta}}}|+\frac{|\triangle \hat{U}|}{s^{1+\tilde{\eta}}}))+O(s^{-L-(1-\delta_0)(1+\frac{\eta}{2})+\ell}\sqrt{\triangle_r \hat{\mathcal{E}}_{s_L}}),
$$
for some coefficient $q_1$ coming from the change of variable. This can be rewritten the following way:
\be \label{variete:bootstrap:eq:expression V1}
\frac{d}{ds}(s\triangle \hat{V}_1)=q_1\triangle U_{\ell+1}+O(s^{-\tilde{\eta}}(|\frac{\hat{s}'-s}{s}|+|\triangle \hat{U}|))+O(s^{-L-(1-\delta_0)(1+\frac{\eta}{2})+\ell+1}\sqrt{\triangle_r \hat{\mathcal{E}}_{s_L}}).
\ee
We now integrate till $s_1$ this identity. Injecting the bootstrap bounds \fref{variete:eq:bootstrap Ui}, \fref{variete:eq:bootstrap V1}, \fref{variete:eq:bootstrap s-hats'} and \fref{variete:eq:bootstrap triangle varepsilon high} one finds:
$$
\begin{array}{r c l}
&\frac{1}{(s_1-s_0)}\left|\underset{s_0}{\overset{s_1}{\int}}q_1\triangle U_{\ell+1}+O(|\frac{\hat{s}'-s}{s^{1+\tilde{\eta}}}|+\frac{|\triangle \hat{U}|}{s^{\tilde{\eta}}})+O(s^{-L-(1-\delta_0)(1+\frac{\eta}{2})+\ell+1}\sqrt{\triangle_r \hat{\mathcal{E}}_{s_L}}) \right| \\
\leq & (q_1C_{\ell+1}+C(\frac{\text{log}(s_0)}{s_0^{\tilde{\eta}}}C_{\triangle \hat{s}}+\frac{C_1+\sum_{i=\ell+1}^L C_i}{s_0^{\tilde{\eta}}}+\frac{\tilde{C}}{s_0^{L+(1-\delta_0)(1+\frac{\eta}{2})-\ell-1}}))D_{\text{stab}}(s_0)\\
&+(q_1\kappa_{\ell+1}+C(\frac{\text{log}(s_0)C_{\triangle \hat{s}}+\kappa_1+\sum_{i=\ell+1}^L \kappa_i}{s_0^{\tilde{\eta}}}+\frac{\tilde{C}}{s_0^{L+(1-\delta_0)(1+\frac{\eta}{2})-\ell-1}}))\underset{s_0\leq s' \leq s, \ 2\leq i \leq \ell}{\text{sup}}|\triangle \hat{V}_i|.
\end{array}
$$ 
So after integrating \fref{variete:bootstrap:eq:expression V1} one obtains:
$$
\begin{array}{r c l}
&|\triangle \hat{V}_1(s_1)|\\
\leq& (1+q_1C_{\ell+1}+C(\frac{\text{log}(s_0)}{s_0^{\tilde{\eta}}}C_{\triangle \hat{s}}+\frac{C_1+\sum_{i=\ell+1}^L C_i}{s_0^{\tilde{\eta}}}+\frac{\tilde{C}}{s_0^{L+(1-\delta_0)(1+\frac{\eta}{2})-i-1}}))D_{\text{stab}}(s_0)\\
&+(q_1\kappa_{\ell+1}+C(\frac{\text{log}(s_0)C_{\triangle \hat{s}}+\kappa_1+\sum_{i=\ell+1}^L \kappa_i}{s_0^{\tilde{\eta}}}+\frac{\tilde{C}}{s_0^{L+(1-\delta_0)(1+\frac{\eta}{2})-i-1}}))\underset{s_0\leq s' \leq s, \ 2\leq i \leq \ell}{\text{sup}}|\triangle \hat{V}_i|.
\end{array}
$$
As $\ell\ll L$ and $1\ll s_0$ the inequality \fref{variete:eq:bootstrap V1} is thus strict at time $s_1$ provided:
\be \label{variete:bootstrap:condition V1}
\begin{array}{l l}
C_1>2+2q_1C_{\ell+1}+O(s_0^{-\frac{\tilde{\eta}}{2}})(C_{\triangle \hat{s}}+\sum_{i=\ell+1}^L C_i+\tilde{C}), \\
\kappa_1 > 2q_1\kappa_{\ell+1}+O(s_0^{-\frac{\tilde{\eta}}{2}})(C_{\triangle \hat{s}}+\sum_{i=\ell+1}^L \kappa_i+\tilde{C}).
\end{array}
\ee

\underline{Step 3:}  \emph{the parameters $U_i$ for $\ell+1\leq i \leq L-1$}. Pick $i$ satisfying $\ell+1\leq i \leq L-1$. One has the identity:
$$
\begin{array}{r c l}
&b_{i,s}+(i-\alpha)b_1b_i-b_{i+1}-(\hat{b}'_{i,s}+\frac{d\hat{s}'}{ds}((i-\alpha)\hat{b}'_1\hat{b}'_i-\hat{b}'_{i+1}))\\
=& \frac{1}{s^i}(\triangle \hat{U}_{i,s}-\frac{(i-(i-\alpha)c_1)\triangle \hat{U}_i+\triangle \hat{U}_{i+1}}{s}+O(s^{-1-\tilde{\eta}}(\frac{|\hat{s}'-s|}{s}+|\triangle \hat{U}|+|\frac{d\hat{s}'}{ds}-1|))).
\end{array}
$$
Hence, using the bound \fref{variete:modulation:eq:dhats (triangle b)}, the modulation equation \fref{variete:modulation:eq:leqL-1 (triangle b)} can be rewritten as:
$$
\begin{array}{r c l}
\triangle \hat{U}_{i,s}&=&\frac{(i-(i-\alpha)c_1)\triangle \hat{U}_i}{s}+\frac{\triangle \hat{U}_{i+1}}{s}\\
&&+O(s^{-1-\tilde{\eta}}(\frac{|s-\hat{s}'|}{s}+|\triangle \hat{U}|)+s^{-L-1-(1-\delta_0)(1+\frac{\eta}{2})+i}\sqrt{\triangle_r \hat{\mathcal{E}}_{s_L}}).
\end{array}
$$
As $i-(i-\alpha)c_1<0$, we can inject the bootstrap bounds \fref{variete:eq:bootstrap Ui}, \fref{variete:eq:bootstrap V1}, \fref{variete:eq:bootstrap s-hats'} and \fref{variete:eq:bootstrap triangle varepsilon high} in the previous equation, and integrate till time $s_1$ as we did in the previous steps to find that:
$$
\begin{array}{r c l}
|\triangle \hat{U}_{i}(s_1)| &\leq& (1+CC_{i+1}+O(s_0^{\frac{-\tilde{\eta}}{2}})(C_{\triangle \hat{s}}+C_1+\underset{j=\ell+1}{\overset{L}{\sum}} C_j+\tilde{C}))D_{\text{stab}}(s_0)\\
&&+(C\kappa_{i+1}+O(s_0^{\frac{-\tilde{\eta}}{2}})(C_{\triangle \hat{s}}+\kappa_1+\underset{j=\ell+1}{\overset{L}{\sum}} \kappa_i+\tilde{C}))\underset{s_0\leq s' \leq s, \ 2\leq i \leq \ell}{\text{sup}}|\triangle \hat{V}_i|.
\end{array}
$$
Thus the inequality \fref{variete:eq:bootstrap Ui} is strict at time $s_1$ provided:
\be \label{variete:bootstrap:condition Ui}
\begin{array}{l l}
C_i>2+CC_{i+1}+O(s_0^{-\frac{\tilde{\eta}}{2}})(C_1+C_{\triangle \hat{s}}+\sum_{j=\ell+1, \ j\neq i}^L C_j+\tilde{C}), \\
\kappa_i > C\kappa_{i+1}+O(s_0^{-\frac{\tilde{\eta}}{2}})(\kappa_1+C_{\triangle \hat{s}}+\sum_{j=\ell+1, \ j \neq i}^L \kappa_j+\tilde{C}),
\end{array}
\ee
the constant $C$ being independent on the constants of the Lemma.\\

\underline{Step 4:}  \emph{the last parameter $U_L$}. Similarly, we rewrite \fref{variete:eq:improved modulation} as:
$$
\begin{array}{r c l}
&\left| \frac{d}{ds}\left(s^{(L-\alpha)c_1-L)}\triangle \hat{U}_L+O(s^{(L-\alpha)c_1-L-\frac{\eta}{2}(1-\delta_0)}(\sqrt{\triangle_r \hat{\mathcal{E}_{s_L}}}+\frac{|\hat{s}'-s|}{s}+|\triangle \hat{U}|)) \right)\right| \\
\leq & Cs^{(L-\alpha)c_1-L-1}(s^{-\frac{\eta}{2}(1-\delta_0)}\sqrt{\triangle_r \hat{\mathcal{E}_{s_L}}}+s^{-\tilde{\eta}}(\frac{|\hat{s}'-s|}{s}+|\triangle \hat{U}|))
\end{array}
$$
because of the bound \fref{variete:eq:gain improved modulation} (the constant in the $O()$ being independent on the other constants of the lemma we are proving). Because $(L-\alpha)c_1-L>0$, when integrating this equation till time $s_1$ one gets:
$$
\begin{array}{r c l}
|\triangle \hat{U}_L(s_1)| &\leq& (1+O(s_0^{\frac{-\tilde{\eta}}{2}})(C_{\triangle \hat{s}}+C_1+\sum_{j=\ell+1}^L C_j+\tilde{C}))D_{\text{stab}}(s_0)\\
&&+O(s_0^{\frac{-\tilde{\eta}}{2}})(C_{\triangle \hat{s}}+\kappa_1+\sum_{i=\ell+1}^L \kappa_i+\tilde{C})\underset{s_0\leq s' \leq s, \ 2\leq i \leq \ell}{\text{sup}}|\triangle \hat{V}_i|.
\end{array}
$$
Thus the inequality \fref{variete:eq:bootstrap Ui} is strict at time $s_1$ provided:
\be \label{variete:bootstrap:condition UL}
\begin{array}{l l}
C_L>2+O(s_0^{-\frac{\tilde{\eta}}{2}})(C_1+C_{\triangle \hat{s}}+\sum_{j=\ell+1}^{L-1} C_j+\tilde{C}), \\
\kappa_L > O(s_0^{-\frac{\tilde{\eta}}{2}})(\kappa_1+C_{\triangle \hat{s}}+\sum_{j=\ell+1}^{L-1} \kappa_j+\tilde{C}),
\end{array}
\ee
the constants in the $O()$ being independent on the constants of the Lemma.\\

\underline{Step 5:}  \emph{the low Sobolev norm}. We consider the time evolution of the low Sobolev norm of the difference of the errors given by \fref{variete:eq:lowsobo}. Because $\lambda^{2(\sigma-s_c)}\sim c b_1^{2(\sigma-s_c)(1+\nu)}$ for some constant $c>0$ one can rewrite it as:
$$
\begin{array}{r c l}
\frac{d}{ds}\left\{ \frac{\triangle \mathcal{E}_{\sigma}}{\lambda^{2(\sigma-s_c)}} \right\} \leq Cb_1^{1+\frac{\alpha}{2L}} \left( \triangle_r \hat{\mathcal{E}}_{\sigma}+\triangle _r \hat{\mathcal{E}}_{s_L}+ (\underset{1\leq i \leq L}{\text{sup}}b_1^{-i}|\triangle \hat{b}_i|)^2\right).
\end{array}
$$
Now, in a similar way as we did in all the previous step, we inject the bootstrap bounds, and integrate this identity till time $s_1$, to find that the bound \fref{variete:eq:bootstrap triangle varepsilon low} is strict at time $s_1$ provided:
\be \label{variete:bootstrap:condition mathcalEsigma}
\tilde{C}>2+O(s_0^{-\frac{\alpha}{4L}})(C_1+C_{\triangle \hat{s}}+\sum_{j=\ell+1}^{L-1} C_j), 
\ee
the constants in the $O()$ being independent on the constants of the Lemma.\\

\underline{Step 6:}  \emph{the high Sobolev norm}. We consider the time evolution of the adapted high Sobolev norm of the difference of the errors given by \fref{variete:eq:high sobo}. We inject the control on the local term given by the Morawetz estimate \fref{variete:eq:morawetz}, knowing $|\mathcal{M}|\lesssim \triangle \hat{\mathcal{E}}_{s_L}$, and rewrite it as (taking $s_0$ large enough and using Young's inequality):
$$
\begin{array}{r c l}
&\left| \frac{d}{ds}\left\{ \frac{\triangle \hat{\mathcal{E}}_{s_L}}{\lambda^{2(s_L-s_c)}}+O \left( \frac{b_1^{2L+2(1-\delta_0)(1+\eta)}}{\lambda^{2(s_L-s_c)}}(\triangle_r \hat{\mathcal{E}}_{s_L}+|\underset{1\leq i \leq L}{\text{sup}}b_1^{-i}|\triangle \hat{b}_i||^2 \right) \right\} \right| \\
\leq& \frac{Cb_1^{2L+2(1-\delta_0)(1+\frac{\eta}{2})+1}}{\lambda^{2(s_L-s_c)}}\Bigl{[}\frac{\triangle \hat{\mathcal{E}}_{s_L}}{A^{\delta}}C(N) \\
&+\frac{C}{N^{\frac{\delta_0}{2}}}(\triangle_r \hat{\mathcal{E}}_{s_L}+\triangle_r \hat{\mathcal{E}}_{\sigma})+C(N,A))b_1^{\frac{\eta}{2}(1-\delta_0)}(\underset{1\leq i \leq L}{\text{sup}}b_1^{-i}|\triangle \hat{b}_i|)^2\Bigr{]}. \\
\end{array}
$$
We inject the bootstrap bounds \fref{variete:eq:bootstrap Ui}, \fref{variete:eq:bootstrap V1}, \fref{variete:eq:bootstrap s-hats'}, \fref{variete:eq:bootstrap triangle varepsilon high} and \fref{variete:eq:bootstrap triangle varepsilon low} in the previous identity and integrate this identity till time $s_1$ (we recall that $b_1\sim \frac{c}{s}$ and $\lambda \sim \frac{c}{s^{\frac{\ell}{\ell-\alpha}}}$):
\be \label{variete:bootstrap:condition mathcalEsL}
\begin{array}{r c l}
\triangle_r \hat{\mathcal{E}}_{s_L}(s_1) &\leq& C(D_{\text{stab}}(s_0)+\underset{s_0\leq s' \leq s, \ 2\leq i \leq \ell}{\text{sup}}|\triangle \hat{V}_i|) \Bigl{[} 1+\left(\frac{1}{N^{\frac{\delta_0}{2}}}+\frac{C(N)}{A^{\delta}}\right)\tilde{C}^2 \\
&&+O\left(\frac{\text{log}(s_0)}{s_0^{\frac{\eta}{2}(1-\delta_0)}}\right)(C_1^2+\sum_{\ell+1}^{L}C_i^2+C_{\triangle \hat{s}}^2)\Bigr{]}.
\end{array}
\ee
The $\kappa$'s do not appear as they are small. The constant $C$ is independent on the other constants. Thus, the bound \fref{variete:eq:bootstrap triangle varepsilon high} is strict at time $s_1$ provided:
$$ 
\tilde{C}^2>C\left[1+\left(\frac{1}{N^{\frac{\delta_0}{2}}}+\frac{C(N)}{A^{\delta}}\right)\tilde{C}^2 +O\left(\frac{\text{log}(s_0)}{s_0^{\frac{\eta}{2}(1-\delta_0)}}\right)(C_1^2+\sum_{\ell+1}^{L}C_i^2+C_{\triangle \hat{s}}^2)\right], 
$$
the constants in the $O()$ being independent on the other constants. Taking $s_0$, $N$, then $A$ large enough, the previous inequality is met if:
\be \label{variete:bootstrap:condition mathcalEsL}
\tilde{C}^2> C\left[1+O\left(\frac{\text{log}(s_0)}{s_0^{\frac{\eta}{2}(1-\delta_0)}}\right)(C_1^2+\sum_{\ell+1}^{L}C_i^2+C_{\triangle \hat{s}}^2)\right], 
\ee
for some constant $C$ independent on the other constants.\\
\underline{Step 7:} \emph{end of the proof}. We have seen that the bootstrap inequalities \fref{variete:eq:bootstrap Ui}, \fref{variete:eq:bootstrap V1}, \fref{variete:eq:bootstrap s-hats'}, \fref{variete:eq:bootstrap triangle varepsilon low} and \fref{variete:eq:bootstrap triangle varepsilon high} are strict at time $s_1$ provided that the conditions \fref{variete:bootstrap:condition hats'}, \fref{variete:bootstrap:condition V1}, \fref{variete:bootstrap:condition Ui}, \fref{variete:bootstrap:condition UL}, \fref{variete:bootstrap:condition mathcalEsigma} and \fref{variete:bootstrap:condition mathcalEsL} are met. Now, if one takes $s_0$ large enough, one can see that there exists constants $C_1$, $\tilde{C}$, $C_{\triangle \hat{s}}$, $(C_i)_{\ell+1\leq i \leq L}$, $\kappa_1\leq \kappa$, $(\kappa_i)_{1\leq i \leq L}$ with $\kappa_i\leq \kappa$ that satisfies all these conditions. Thus, if the time $s_1$ were finite, all the bootstrap bounds would be strict at this time, which is impossible from a continuity argument.
\end{proof}

Thanks to the previous Lemma we can now end the proof of Proposition \fref{variete:prop:parametres lipschitz}.

\begin{proof}[Proof of Proposition \fref{variete:prop:parametres lipschitz}] Let $\bos{U}$ and $\bos{U}'$ be two solutions satisfying the assumptions of Proposition \fref{variete:prop:parametres lipschitz}. We recall that $\triangle \hat{V}_{\text{uns}}$ is defined by \fref{variete:eq:def Vuns}. At time $s_0$ one has: $\triangle \hat{V}_{\text{uns}}=\triangle V_{\text{uns}}$. Let $i$ be an integer, $2\leq i \leq \ell$. As $\triangle \hat{V}_i$ is a linear combination of the $\triangle \hat{U}_j$ for $1\leq i \leq \ell$ only, see \fref{thetrapped:eq:def Vi}, and because of the shape of the matrix $A_{\ell}$, see \fref{linearized:eq:diagonalisation}, the identity \fref{variete:eq:expression dttriangleUi} gives that the time evolution of $\triangle \hat{V}_i$ is:
\be \label{variete:eq:evolution trianglehatVi}
\triangle \hat{V}_{i,s}=\mu_i \frac{\triangle \hat{V}_1}{s}+q_i\frac{\triangle U_{\ell+1}}{s}+O(|\frac{\hat{s}'-s}{s^{2+\tilde{\eta}}}|+\frac{|\triangle \hat{U}|}{s^{1+\tilde{\eta}}})+O(s^{-L-(1-\delta_0)(1+\frac{\eta}{2})+i}\sqrt{\triangle_r \hat{\mathcal{E}}_{s_L}}),
\ee
where $\mu_i>0$ denotes the $i$-th eigenvalue of the matrix $A_{\ell}$, see Lemma \ref{lem:linearisationsystemdyn}, and $q_i$ is some constant coefficient coming from the change of variables from $\triangle \hat{U}$ to $\triangle \hat{V}$. Now let $\mu:=\underset{2\leq i \leq \ell }{\text{min}}\mu_i$ and $q:=\underset{2\leq i \leq \ell }{\text{max}}|q_i|$. Using Cauchy-Schwarz inequality, the identity \fref{variete:eq:evolution trianglehatVi} gives for the evolution of the unstable parameters:
\be \label{variete:eq:dtVuns}
\begin{array}{r c l}
&\frac{d}{ds}|\triangle \hat{V}_{\text{uns}}|^2\\
\geq& \frac{|\triangle \hat{V}_{\text{uns}}|}{s}(\frac{\mu}{2} |\triangle \hat{V}_{\text{uns}}|-q|\triangle \hat{U}_{\ell+1}|-\frac{1}{s^{\frac{\tilde{\eta}}{2}}}(|\frac{\hat{s}'-s}{s}|+|\triangle \hat{V}_1+|\sum_{\ell+1}^L|\triangle \hat{U}_i|+\sqrt{\triangle_r \hat{\mathcal{E}}_{s_L}}) ),
\end{array}
\ee
if one has chosen $s_0$ big enough. Now, as $q$ and $\mu$ are fixed constants of the problem, one can ask that:
\be \label{variete:eq:choix kappa}
q\kappa<\frac{\mu}{10}.
\ee
Let the constants $\tilde{C}$, $C_1$, $(C_i)_{\ell +1\leq i \leq L}$, $C_{\triangle \hat{s}}$, $0<\kappa_1<\kappa$, $0<\kappa_i<\kappa$ for $\ell+1 \leq i\leq L$ and $\bar{s}$ be such that the previous Lemma \fref{variete:lem:bootstrap differences} holds. In particular, one can take $s_0$ big enough such that:
\be \label{variete:eq:choix s}
\frac{1}{s^{\frac{\tilde{\eta}}{2}}}(\text{log}(s)C_{\triangle \hat{s}}+\tilde{C}+\kappa_1+\sum_{i=\ell+1}^L \kappa_i)\leq \frac{\mu}{10}
\ee
We now argue by contradiction. Suppose one has at initial time:
\be \label{variete:eq:condition initiale Vuns}
|\triangle V_{\text{uns}}(s_0)| > \frac{10}{\mu}(C_1+\tilde{C}+C_{\triangle \hat{s}}+qC_{\ell+1}+\sum_{\ell+2}^L|\triangle \hat{U}_i|)  D_{\text{stab}}(s_0).
\ee
We are going to show that this leads to a contradiction. Indeed, \fref{variete:eq:dtVuns} implies that at initial time the differences of unstable modes are growing:
\be \label{variete:eq:croissance Vuns}
\frac{d}{ds}|\triangle \hat{V}_{\text{uns}}|^2>0.
\ee
Let $s_1$ denote the supremum of all times $s$ with $s_0\leq s$ such that \fref{variete:eq:croissance Vuns} holds on $[s_0,s_1]$. We are going to prove that $s_1=+\infty$. Indeed, suppose $s_1$ were finite. Then at time $s_1$ one has:
$$
\underset{s_0\leq s' \leq s_1, \ 2\leq i \leq \ell}{\text{sup}}|\triangle \hat{V}_i| \leq |\triangle \hat{V}_{\text{uns}}(s_1)|
$$
because of the monotonicity \fref{variete:eq:croissance Vuns} on $[s_0,s_1]$. Injecting the bounds \fref{variete:eq:bootstrap triangle varepsilon high}, \fref{variete:eq:bootstrap s-hats'}, \fref{variete:eq:bootstrap V1} and \fref{variete:eq:bootstrap Ui} in \fref{variete:eq:dtVuns} give, because of the inequalities \fref{variete:eq:choix kappa} and \fref{variete:eq:choix s} between the constants:
$$
\begin{array}{r c l}
&\frac{d}{ds}|\triangle \hat{V}_{\text{uns}}|^2\\
\geq& \frac{\mu|\triangle \hat{V}_{\text{uns}}|}{2s}\Bigl{(}  |\triangle \hat{V}_{\text{uns}}|(1-\frac{2q\kappa_{\ell+1}}{\mu}-\frac{2}{\mu s^{\frac{\tilde{\eta}}{2}}}(\text{log}(s)C_{\triangle \hat{s}}+\tilde{C}+\kappa_1+\sum_{i=\ell+1}^L \kappa_i))\\
&-\frac{2}{\mu}(qC_{\ell+1}+\frac{1}{s^{\frac{\tilde{\eta}}{2}}}(\text{log}(s)C_{\triangle \hat{s}}+\tilde{C}+C_1+\sum_{i=\ell+1}^L C_i))D_{\text{stab}}(s_0)  \Bigr{)} \\
\geq &\frac{\mu|\triangle \hat{V}_{\text{uns}}|}{2s}\Bigl{(}  |\triangle \hat{V}_{\text{uns}}|\frac{1}{2} \\
&-\frac{2}{\mu}(qC_{\ell+1}+\frac{1}{s^{\frac{\tilde{\eta}}{2}}}(\text{log}(s)C_{\triangle \hat{s}}+\tilde{C}+C_1+\sum_{i=\ell+1}^L C_i))D_{\text{stab}}(s_0)  \Bigr{)}.
\end{array}
$$
But because $|\triangle \hat{V}_{\text{uns}}|$ is increasing on $[s_0,s_1]$, and because at initial time \fref{variete:eq:condition initiale Vuns} holds, one has:
$$
|\triangle \hat{V}_{\text{uns}}(s_1)|\frac{1}{2}-\frac{2}{\mu}(qC_{\ell+1}+\frac{1}{s^{\frac{\tilde{\eta}}{2}}}(\text{log}(s)C_{\triangle \hat{s}}+\tilde{C}+C_1+\sum_{i=\ell+1}^L C_i))D_{\text{stab}}(s_0)>0
$$
which in turn implies that at time $s_1$: $\frac{d}{ds}|\triangle \hat{V}_{\text{uns}}|^2>0$, contradicting the definition of $s_1$. Hence $s_1=+\infty$. But if $s_1=+\infty$, that means that $|\triangle \hat{V}_{\text{uns}}|$ does not converge toward $0$. This is the desired contradiction, because as $\bos{U}$ and $\bos{U}'$ stay in the trapped regime, this should be true.
\end{proof}


\subsection{Removal of extra assumptions, end of the proof of Theorem \ref{variete:thm;variete}}

In the proof of Proposition \ref{prop:bootstrap}, we have seen that in order to control the projection of a solution on the first $L$ iterates of the kernel of $\bos{H}$, one needs to control the $k_0+1+L$ adapted derivative of $\bos{\varepsilon}$. Therefore, we will decompose only on the first $L-1$ modes, which will allow us to work with the $k_0+L$-th adapted derivative, while keeping the bound \fref{variete:eq:estimation supplementaire varepsilon'} for the $k_0+1+L$-th one. It will allow us to remove the regularity assumption \fref{variete:eq:estimation supplementaire varepsilon'} in Proposition \ref{variete:prop:parametres lipschitz}. An other extra assumption in this proposition was the fact that the two solutions started with the same scale, what we will also remove. Our main result is the following improvement of Proposition \fref{variete:prop:parametres lipschitz}:

\begin{proposition}
\label{variete:prop:parametres lipschitz 2}
Suppose $\boldsymbol{U}(s_0)=\Bigl{(}\tilde{\boldsymbol{Q}}_{b,\frac{1}{\lambda}}+\boldsymbol{w}\Bigr{)}(s_0)$, $\boldsymbol{U'}(s_0)=\Bigl{(}\tilde{\boldsymbol{Q}}_{b',\frac{1}{\lambda'}}+\boldsymbol{w'}\Bigr{)}(s_0)$ are two initial data whose solutions stay in the trapped regime described by Proposition \ref{prop:bootstrap}. Suppose that they are close initially, that is to say that:
\be
b(s_0)=b^e(s_0)+\left(\frac{U_1(s_0)}{s_0},...,\frac{U_L(s_0)}{s_0^L}\right), b'(0)=b^e(s_0)+\left(\frac{U'_1(s_0)}{s_0},...,\frac{U'_L(s_0)}{s_0^L}\right).
\ee
Suppose that the scales are close to one:
\be \label{variete:removal:eq:lambda sim 1}
|\lambda(s_0)-1|+|\lambda'(s_0)-1|\leq s_0^{-L}
\ee
Then there exists $C>0$ such that for $s_0$ small enough the following bound holds:
\be
\begin{array}{r c l}
|\triangle V_{\text{uns}}(s_0)|&\leq & C\Bigl( |\triangle V_1(s_0)|+\sum_{\ell+1}^L|\triangle U_i(s_0)|+|\lambda'(s_0)-\lambda(s_0)|\\
&& C(s_0) \parallel \bos{w}(s_0)-\bos{w}'(s_0)  \parallel_{\dot{H}^{\sigma}\cap\dot{H}^{s_L}\times \dot{H}^{\sigma-1}\cap\dot{H}^{s_L-1}}   \Bigr).
\end{array}
\ee
\end{proposition}

\subsubsection{Lower order decomposition}

We start by lowering the number of modes on which we project on the manifold of approximate solutions $(\tilde{\bos{Q}}_{b,\lambda})_{b,\lambda}$. We let:
\be
\bar{L}=L-1.
\ee

\begin{definition}[Lower order decomposition] \label{variete:def:decomposition adaptee varepsilon}
Suppose $\boldsymbol{U}=\tilde{\boldsymbol{Q}}_{b,\frac{1}{\lambda}}+\boldsymbol{w}=(\tilde{\boldsymbol{Q}}_{b}+\boldsymbol{w})_{\frac{1}{\lambda}}$ is a solution satisfying the assumptions of Proposition \ref{variete:prop:parametres lipschitz 2}. We define the $\bar L$-tuple of real numbers $\bar{b}$, the scale $\bar{\lambda}$, and the error terms $\overline{\bos{\varepsilon}}$ and $\overline{\bos{w}} $ by:
\be
\bos{U}(t)=\tilde{\bos{Q}}_{\bar{b},\frac{1}{\bar{\lambda}}}+\overline{\bos{w}}(t)= (\tilde{\bos{Q}}_{\bar{b}}+\overline{\bos{\varepsilon}}(\bar{s}))_{\frac{1}{\bar{\lambda}}},
\ee
where $\overline{\bos{\varepsilon}}$ satisfies the $L$ orthogonality conditions:
\be
\langle \overline{\bos{\varepsilon}},\bos{H}^{*i}\bos{\Phi}_M \rangle= 0, \ for \ 0 \leq i \leq L-1.
\ee
The renormalized time is given by:
\be 
\bar{s}:=s_0+\int_0^t \frac{1}{\bar{\lambda}(\tau)}d\tau.
\ee
This decomposition is possible for $\bos{U}$ because as it is a solution given by Proposition \ref{prop:bootstrap}, the result of subsubsection \ref{thetrapped:subsubsection:modulation} applies for the integer $\bar L$. We then define the tuples of parameters $\bar{U}$ and $\bar{V}$ as ($\bar{P}_{\ell}$ being the analogue of $P_{\ell}$ defined by \fref{linearized:eq:def P}):
\be
\bar{U}_i:=\bar{s}^i (\bar{b}_i-\frac{c_i}{\bar{s}^i}), \ \text{for} \ 1\leq i \leq L, \ \text{and} \ \bar{V}:=\bar{P}_{\ell}(\bar{U}).
\ee
\end{definition}

We introduce the following notation for the norms of $\overline{\bos{\varepsilon}}$:
\be 
\overline{\mathcal{E}}_{\sigma}:=\int |\nabla^{\sigma}\bar{\varepsilon}^{(1)}|^2 +|\nabla^{\sigma-1}\bar{\varepsilon}^{(2)}|^2, \
\overline{\mathcal{E}}_{i}:=\int |\bar{\varepsilon}^{(1)}_{i}|^2 +|\nabla^{\sigma-1}\bar{\varepsilon}^{(2)}_{i-1}|^2, \ i=s_{\bar{L}},s_{\bar{L}}+1.
\ee
This decomposition works as follows: we have approximately $\bar{b}\sim(b_1,...,b_{L-1})$ and $\overline{\bos{\varepsilon}}\sim \bos{\varepsilon}+b_L\bos{T}_L$. The bounds of the trapped regime for the original decomposition transform into bounds for the lower order decomposition. This way we obtain a solution of the trapped regime (with respect to the integer $\bar{L}$ instead of $L$) with the extra higher regularity bound \fref{variete:eq:estimation supplementaire varepsilon'}: this is the type of solution for which we proved a primary Lipschitz bound in Proposition \ref{variete:prop:parametres lipschitz}.

\begin{lemma}[Bounds for the lower order decomposition]
\label{variete:lem:borne sur changement epsilon epsilonbar}
We keep the assumptions and notations of Definition \ref{variete:def:decomposition adaptee varepsilon}. The following estimates for $0\leq t < T$ hold:
\begin{itemize}
\item[(i)] \emph{Global closeness for the parameters:} The renormalized time satisfies: 
\be 
\bar{s}=s+O\left(\frac{1}{s_0^{L}}\right).
\ee
For all $1\leq i \leq L-1$ there holds:
\be \label{variete:lower:proximite U}
|U_i-\bar{U}_i|= O(\bar{s}^{-1}).
\ee
These two bounds imply in particular that:
\be 
\bar{b}_1\sim b_1.
\ee
\item[(ii)] \emph{Bounds for the high adapted derivatives:} for $i=0,1$ one has
\be  \label{variete:lower:eq:proximite high}
\overline{\mathcal{E}}_{s_{\bar{L}}+i}\leq C(L,M)K_2 \bar{b}_1^{\bar{L}+i+(1-\delta_0)(1+\eta)}.
\ee
\item[(iii)] \emph{Bound at $\sigma$ level of regularity:}
\be 
\overline{\mathcal{E}}_{\sigma}\leq C(L,M)K_1\bar{b}_1^{2(\sigma-s_c)(1+\nu)}.
\ee
\end{itemize}
\end{lemma}

We denote the canonical projection from $\mathbb{R}^{L}$ to $\mathbb{R}^{\bar{L}}$ by:
\be 
\pi \ : \ (b_1,...,b_L)\mapsto (b_1,...,b_{L-1}).
\ee
The difference between $\tilde{\bos{Q}}_b$ and $\tilde{\bos{Q}}_{\pi(b)}$ is denoted by:
\be \label{variete:eq:def barSL+1}
\tilde{\bos{Q}}_b=\tilde{\bos{Q}}_{\pi(b)}+\chi_{B_1}(b_L \bos{T}_L+\bos{S}_{L+2}+\bos{S}_{L+1}-\overline{\bos{S}}_{L+1})
\ee
where $\bos{S}_{L+1}$ is given by Proposition \textbf{mettre renvoi} for the L-tuple $b$ and $\overline{\bos{S}}_{L+1}$ is the profile given by the same proposition, but for the $L-1$ tuple $\pi (b)$.

\begin{proof}[Proof of Lemma \ref{variete:lem:borne sur changement epsilon epsilonbar}]
\underline{Proof of (i):}\\
$\bullet $ \emph{Step 1:} primary bound. We claim that for all $0\leq t <T$:
\be \label{variete:eq:borne meme temps decomposition lambda} 
\left| \frac{\bar{\lambda}(t)}{\lambda(t)}-1 \right|\leq C(L,M)b_1(t)^{L+1},
\ee
\be \label{variete:eq:borne meme temps decomposition b}
\left| b_i(t)-\bar{b}_i(t) \right|\leq C(L,M)b_1(t)^{L+1}.
\ee
We start by proving these two estimates. $\overline{\bos{\varepsilon}}$ is given by:
\be \label{variete:eq:expression barvarepsilon}
\begin{array}{r c l}
\overline{\bos{\varepsilon}}&=&\bos{\varepsilon}_{\frac{\bar{\lambda}}{\lambda}}+(\tilde{\bos{Q}}_{\frac{\bar{\lambda}}{\lambda}}-\bos{Q})+(\tilde{\bos{\alpha}}_{\pi(b),\frac{\bar{\lambda}}{\lambda}}-\tilde{\bos{\alpha}}_{\bar{b}})\\
&&+(\chi_{B_1}(b_L\bos{T}_L+\bos{S}_{L+2}+\bos{S}_{L+1}-\overline{\bos{S}}_{L+1}))_{\frac{\bar{\lambda}}{\lambda}}.
\end{array}
\ee
We take the scalar product between $\overline{\bos{\varepsilon}}$ and $\bos{H}^{*i}\bos{\Phi}_M$ for $0\leq i \leq \bar{L}$. For $i=0$ we obtain a bound for the scaling.
$$
-\langle \bos{Q}_{\frac{\bar{\lambda}}{\lambda}}-\bos{Q},\bos{\Phi}_M \rangle = \langle \bos{\varepsilon}_{\frac{\bar{\lambda}}{\lambda}}+ \tilde{\bos{\alpha}}_{\pi(b),\frac{\bar{\lambda}}{\lambda}}-\tilde{\bos{\alpha}}_{\bar{b}}+(b_L\bos{T}_L+\bos{S}_{L+2}+\bos{S}_{L+1} -\overline{\bos{S}}_{L+1})_{\frac{\bar{\lambda}}{\lambda}},\bos{\Phi}_M \rangle.
$$ 
The left hand side is:
$$
\langle \bos{Q}_{\frac{\bar{\lambda}}{\lambda}}-\bos{Q},\bos{\Phi}_M \rangle=(\frac{\bar{\lambda}}{\lambda}-1)\langle \bos{\Lambda} \bos{Q},\bos{\Phi}_M \rangle+O(|\frac{\bar{\lambda}}{\lambda}-1|^2)
$$
We now look at the terms in the right hand side. Performing a change of variables:
$$
\langle \bos{\varepsilon}_{\frac{\bar{\lambda}}{\lambda}},\bos{\Phi}_M\rangle=\left( \frac{\lambda}{\bar{\lambda}}\right)^{d-\frac{4}{p-1}}\langle \bos{\varepsilon},\bos{\Phi}_{M,\frac{\lambda}{\bar{\lambda}}}\rangle=O\left( b_1^{L+(1-\delta_0)(1+\eta))}(\frac{\bar{\lambda}}{\lambda}-1) \right).
$$
For the second term we decompose:
$$
\langle \tilde{\bos{\alpha}}_{\pi(b),\frac{\bar{\lambda}}{\lambda}}-\tilde{\bos{\alpha}}_{\bar{b}},\bos{\Phi}_M \rangle= \langle \tilde{\bos{\alpha}}_{\pi(b),\frac{\bar{\lambda}}{\lambda}}-\tilde{\bos{\alpha}}_{\pi (b)}+\tilde{\bos{\alpha}}_{\pi (b)}-\tilde{\bos{\alpha}}_{\bar{b}},\bos{\Phi}_M \rangle.
$$
There holds for the first part:
$$
\langle \tilde{\bos{\alpha}}_{\pi (b),\frac{\bar{\lambda}}{\lambda}}-\tilde{\bos{\alpha}}_{\pi(b)},\bos{\Phi}_M\rangle = O\left( b_1^2(\frac{\bar{\lambda}}{\lambda}-1) \right).
$$
For the second part, because of the orthogonality property \fref{thetrapped:eq:orthogonalite PhiM}:
$$
\langle \tilde{\bos{\alpha}}_{\pi (b)}-\tilde{\bos{\alpha}}_{\bar{b}},\bos{\Phi}_M \rangle =\left\langle \sum_{i=2}^{L-2} \bos{S}_i(\pi(b))-\bos{S}_i(\bar{b})+\overline{\bos{S}}_{L+1}(\pi (b))-\overline{\bos{S}}_{L+1}(\bar{b}),\bos{\Phi}_M\right\rangle,
$$
where we recall that $\overline{\bos{S}}_{L+1}$ is defined in \fref{variete:eq:def barSL+1}. All these terms are of the form:
$$
\int \Phi_M f (\prod_1^{L-1}b_i^{J_i}-\prod_1^{L-1}\bar{b}_i^{J_i})
$$
where $|J|_2\geq 2$ (the notation for the tupples are defined in \textbf{mettre ref}) and $f$ is bounded. The bound \fref{variete:eq:difference polynomes parametres} on the difference of polynomials of the $b_i$'s then gives:
$$
\langle \tilde{\bos{\alpha}}_{\pi (b)}-\tilde{\bos{\alpha}}_{\bar{b}},\bos{\Phi}_M \rangle=O(b_1\text{sup}(|b_i-\bar{b}_i|)).
$$
The last term gives:
$$
\langle (\chi_{B_1}(b_L\bos{T}_L+\bos{S}_{L+2}+\bos{S}_{L+1}-\overline{\bos{S}}_{L+1}))_{\frac{\bar{\lambda}}{\lambda}},\bos{\Phi}_M \rangle=O(b_1^{L+1}).
$$
Put together, all the previous computations yield:
\be \label{variete:eq:controle diff lambda decomposition}
(\frac{\bar{\lambda}}{\lambda}-1)=O(b_1^{L+1})+O(b_1 sup(|b_i-\bar{b}_i|)).
\ee
Similarly, taking the scalar product of \fref{variete:eq:expression barvarepsilon} with $\bos{H}^{*i}\bos{\Phi}_M$ for $1\leq i\leq \bar L$ yields:
\be \label{variete:eq:controle diff b decomposition}
(b_i-\bar{b}_i)=O(b_1^{L+1})+O(b_1sup(|b_i-\bar{b}_i|))+O\left(\left(b_1+|\frac{\bar{\lambda}}{\lambda}-1|\right)|\frac{\bar{\lambda}}{\lambda}-1|\right) .
\ee
By summing \fref{variete:eq:controle diff b decomposition} and \fref{variete:eq:controle diff lambda decomposition} one obtains the primary bounds we claimed: \fref{variete:eq:borne meme temps decomposition b} and \fref{variete:eq:borne meme temps decomposition lambda}.\\
$\bullet$ \emph{Step 2:} integration of the primary bounds. Equation \fref{variete:eq:borne meme temps decomposition lambda} gives a control on the renormalized time difference:
$$
\frac{d\bar{s}}{ds}= \frac{d \bar{s}}{dt} \frac{dt}{ds}=\frac{\lambda}{\bar{\lambda}}=1+O(b_1^{L+1}).
$$
As $b_1\lesssim s^{-1}$ an integration in time yields:
$$
\bar{s}=s+O\left( \frac{1}{s_0^L}\right).
$$
This implies in particular that for $1\leq i \leq L$:
$$
b_i^e=\bar{b}^e_i+O(s^{-(i+1)}),
$$
which, combined with the primary bound \fref{variete:eq:borne meme temps decomposition b} ends the proof of (i).

\underline{Proof of (ii):} We proved in the previous step that $s \sim \bar{s}$ and $b_1 \sim \bar{b}_1$. We first prove the bound at the level of regularity $s_{\bar{L}}+1=s_L$. We have to compute the adapted norm of the right hand side of \fref{variete:eq:expression barvarepsilon}. We just show here the computations for the second coordinate $\bar{\varepsilon}^{(2)}$, because the estimate for the first one can be proven using the very same calculations. As $\varepsilon$ satisfies the result of Proposition \ref{prop:bootstrap}, and as $\bar{\lambda}\sim \lambda$, see \fref{variete:eq:borne meme temps decomposition lambda}, there holds:
$$
\int |(\varepsilon^{(2)}_{\frac{\bar{\lambda}}{\lambda}})_{s_L-1}|^2\leq CK_2\bar{b}_1^{2L+2(1-\delta_0)(1+\eta)}
$$
with  $C$ independent of the other constants. For the second term we decompose:
$$
\int |(\tilde{\alpha}_{\pi (b),\frac{\bar{\lambda}}{\lambda}}^{(2)}-\tilde{\alpha}_{\bar{b}}^{(2)})_{s_L-1}|^2\lesssim \int |(\tilde{\alpha}_{\pi (b),\frac{\bar{\lambda}}{\lambda}}^{(2)}-\tilde{\alpha}_{\pi (b)}^{(2)})_{s_L-1}|^2+|(\tilde{\alpha}_{\pi (b)}^{(2)}-\tilde{\alpha}_{\bar{b}}^{(2)})_{s_L-1}|^2.
$$
The first term of the right hand side satisfies:
$$
(\tilde{\alpha}_{\pi (b),\frac{\bar{\lambda}}{\lambda}}^{(2)}-\tilde{\alpha}_{\pi (b)}^{(2)})=(\frac{\bar{\lambda}}{\lambda}-1)\int_0^1 \frac{1}{1-\theta+\theta\frac{\bar{\lambda}}{\lambda}}(\Lambda^{(2)}\tilde{\alpha}^{(2)}_{\pi (b)})_{1-\theta+\theta\frac{\bar{\lambda}}{\lambda}}d\theta.
$$
And as:
$$
\int |(\Lambda^{(2)}\tilde{\alpha}^{(2)}_{\pi (b')})_{1-\theta+\theta\frac{\bar{\lambda}}{\lambda'},s_L-1}|^2<+\infty,
$$
we conclude using \fref{variete:eq:borne meme temps decomposition lambda} that:
$$
\int |(\tilde{\alpha}_{\pi (b),\frac{\bar{\lambda}}{\lambda}}^{(2)}-\tilde{\alpha}_{\pi (b)}^{(2)})_{s_L-1}|^2\leq |\frac{\bar{\lambda}}{\lambda}-1|^2\lesssim \bar{b}_1^{2L+2}\lesssim \bar{b}_1^{2L+2(1-\delta_0)(1+\eta)}.
$$
For the other term we compute:
$$
\begin{array}{r c l}
 |(\tilde{\alpha}_{\pi (b)}^{(2)}-\tilde{\alpha}_{\bar{b}}^{(2)})_{s_L-1}|^2 &\lesssim&  \sum_{i=2, \ \text{odd}}^{L-1} |((b_i\chi_{B_1}T_i-\bar{b}_i\chi_{\bar{B}_1}T_i)_{s_L-1}|^2\\
 &&+\sum_{i=2, \ \text{odd}}^{L+1}|(\chi_{B_1}S_i(\pi(b)-\chi_{\bar{B}_1}S_i(\bar{b}))_{s_L-1}|^2.
\end{array}
$$
We have:
$$
\int |((b_i\chi_{B_1}T_i-\bar{b}_i\chi_{\bar{B}_1}T_i)_{s_L-1}|^2\lesssim \int |(b'_i(\chi_{B_1}-\chi_{\bar{B}_1})T_i)_{s_L-1}|^2+|((b_i-\bar{b}_i)\chi_{\bar{B}_1}T_i)_{s_L-1}|^2
$$
and we estimate the two parts:
$$
\int |(b_i(\chi_{B_1}-\chi_{\bar{B}_1})T_i)_{s_L-1}|^2\lesssim \bar{b}_1^{2L+2(1-\delta_0)(1+\eta)},
$$
$$
\int |((b_i-\bar{b}_i)\chi_{\bar{B}_1}T_i)_{s_L-1}|^2 \lesssim \bar{b}_1^{2L+2(1-\delta_0)(1+\eta)+2(L+1-i)},
$$
where we used \fref{variete:eq:borne meme temps decomposition b} for the second inequality. A similar argument gives a similar control for the $S_i$'s contribution, hence yielding to:
$$
\int |(\tilde{\alpha}_{\pi (b)}^{(2)}-\tilde{\alpha}_{\bar{b}}^{(2)})_{s_L-1}|^2\lesssim \bar{b}_1^{2L+2(1-\delta_0)(1+\eta)}.
$$
We go on, estimating the next term. From the asymptotic of $T_L$, see Lemma \ref{lem:profilsTi}:
$$
\int b_L^{2} |((\chi_{B_1}T_L)_{\frac{\bar{\lambda}}{\lambda}})_{s_L-1}|^2\leq |b_L|^2 b_1^{2(1+\eta)(1-\delta_0)}\leq C\bar{b}_1^{2L+2(1-\delta_0)(1+\eta)},
$$
with a constant $C$ that just depends on the bootstrap constant\footnote{remember that $\epsilon_L$ quantify the size of $b_L$, see \fref{eq:bootstrap modes stables}.} $\epsilon_L$ and on $L$, but which, if $L$ is fixed, is uniformly bounded in $\epsilon_L$. Similarly, from \fref{linearized:eq:degre Si}:
$$
\begin{array}{r c l}
\int |((\chi_{B_1}S_{L+2})_{\frac{\bar{\lambda}}{\lambda}})_{s_L-1}|^2&\leq&
\left\{ \begin{array}{l l}
C \bar{b}_1^{2L+2(1-\delta_0)-C'\eta+2g'} \ \text{if} \ 2\delta_0+2-2g'>0 \\
C \bar{b}_1^{2L+4} \ \text{if} \ 2\delta_0+2-2g'<0
\end{array} \right. \\
&\leq& C \bar{b}_1^{2L+2(1-\delta_0)(1+\eta)},
\end{array}
$$
for $\eta$ small enough, (we recall the assumption $0<\delta_0$). All the previous estimates show the bound (ii) for the second coordinate:
$$
\int |\bar{\varepsilon}^{(2)}_{s_L-1}|^2\leq C(K_1,\epsilon_L,L,M)\bar{b}_1^{2L+2(1-\delta_0)(1+\eta)}.
$$
We claim that the estimate for the first coordinate can be shown making verbatim the same computations. For the sake of completeness, we just show how to deal with the term involving the soliton. We compute first:
\be \label{variete:eq:expression Qscale-Q}
Q_{\frac{\bar{\lambda}}{\lambda}}-Q= (\frac{\bar{\lambda}}{\lambda}-1) \int_0^1 \frac{1}{1-\theta+\theta (\frac{\bar{\lambda}}{\lambda})}(\Lambda^{(1)} Q)_{1-\theta+\theta (\frac{\bar{\lambda}}{\lambda'})}d\theta.
\ee
As for all $\theta$, $\int |((\Lambda Q)_{1-\theta+\theta (\frac{\bar{\lambda}}{\lambda'})})_{s_L}|^2<+\infty$, using \fref{variete:eq:borne meme temps decomposition lambda} and because $0<\delta_0$ we get:
$$
\int |(Q_{\frac{\bar{\lambda}}{\lambda}}-Q)_{s_L}|^2\leq C\bar{b}_1^{L+1}\leq \bar{b}_1^{2L+2(1-\delta_0)(1+\eta))}
$$
for $\eta$ small enough. This way we get the bound (ii) for $i=1$. To prove (ii) for $i=0$ we need to use the energy estimate we used to control the error in the proof of Proposition \ref{prop:bootstrap}. In the proof of this proposition, we saw (see Section \ref{end:sec:end of the proof}) that if a solution started in the trapped regime, the only way to escape it was by having unstable mode growing too big. Here the unstable modes are under control from the previous bounds \fref{variete:lower:proximite U}. So if it starts in the trapped regime described by proposition \ref{prop:bootstrap} associated to the integer $\bar{L}$, it will imply the control \fref{variete:lower:eq:proximite high} for $i=0$. We compute the adapted $s_{\bar{L}}$ norm of the right hand side of \fref{variete:eq:expression barvarepsilon} at initial time $s_0$. One has for the error by interpolation of \fref{eq:estimation initiale de epsilon0}:
$$
\parallel \varepsilon^{(1)}(s_0)_{\frac{\bar{\lambda}}{\lambda},s_{\bar{L}}} \parallel_{L^2} +\parallel \varepsilon^{(2)}(s_0)_{\frac{\bar{\lambda}}{\lambda},s_{\bar{L}-1}} \parallel_{L^2} \leq Cb_1(s_0)^{\bar{L}+2+(1-\delta_0)(1+\eta)}.
$$
For the $L$-th mode one has using the bound \fref{eq:estimation initiale des modes stables}:
$$
\parallel b_L(s_0)(\chi_{B_1}T_{L})_{s_{\bar{L}-1}} \parallel_{L^2}\leq C|b_L(s_0)|b_1^{-\delta_0(1+\eta)}\leq C b_1^{\bar{L}+1-\delta_0+\alpha\frac{L-\ell}{\ell-\alpha}+O(\eta))}.
$$
We claim that for all the other terms in the right hand side of \fref{variete:eq:expression barvarepsilon}, the same computations we did for the proof of (ii) in the case $i=1$ yield similar results. Hence at initial time one has:
$$
\parallel \bar{\varepsilon}^{(1)}(s_0)_{s_{\bar{L}}} \parallel_{L^2} +\parallel \bar{\varepsilon}^{(2)}(s_0)_{s_{\bar{L}-1}} \parallel_{L^2} \leq Cb_1(s_0)^{\bar{L}+(1-\delta_0)(1+\eta)}.
$$
Hence we use the result Remark \ref{end:rem:condition sortie}: as the unstable modes are under control from \fref{variete:lower:proximite U}, we get the desired bound for all time:
$$
\parallel \bar{\varepsilon}^{(1)}_{s_{\bar{L}}} \parallel_{L^2} +\parallel \bar{\varepsilon}^{(2)}_{s_{\bar{L}-1}} \parallel_{L^2} \leq Cb_1^{\bar{L}+(1-\delta_0)(1+\eta)}.
$$

\underline{Proof of (iii):} The estimate for $\overline{\mathcal{E}}_{\sigma}$ can be done by direct computation as we did for (ii) using similar computations. We estimate again all the terms in the right hand side of \fref{variete:eq:expression barvarepsilon}. We only show the estimate for the first coordinate, as the proof for the second one relies on similar computations. From $\bar{\lambda}\sim \lambda$, and as $\varepsilon$ satisfies the bound \fref{eq:bootstrap estimations sur epsilon} one gets:
$$
\int |\nabla^{\sigma}\varepsilon_{\frac{\bar{\lambda}}{\lambda}}^{(1)}|^2 |\nabla^{\sigma-1}\varepsilon_{\frac{\bar{\lambda}}{\lambda}}^{(2)}|^2\leq = \left|\frac{\bar{\lambda}}{\lambda}\right|^{2(\sigma-s_c)}\mathcal{E}_{\sigma}\leq CK_1 \bar{b}_1^{2(\sigma-s_c)(1+\nu)}
$$
for a constant $C$ independent of the other constants. For the soliton term, we use the expression \fref{variete:eq:expression Qscale-Q} and Fubini to estimate:
$$
\int |\nabla^{\sigma}(Q_{\frac{\bar{\lambda}}{\lambda}}-Q)|^2 \leq \left|\frac{\bar{\lambda}}{\lambda}-1\right|^2 \underset{\theta\in [0,1] }{sup} \left|1-\theta+\theta\frac{\bar{\lambda}}{\lambda}\right|^{2(\sigma-s_c)-2} \int |\nabla^{\sigma}(\Lambda^{(1)}Q)|^2\leq C \bar{b}_1^{2L+2}.
$$
We used the bound $|\frac{\bar{\lambda}}{\lambda}-1|\lesssim \bar{b}_1^{L+1}$ and the fact that $ \int |\nabla^{\sigma}(\Lambda^{(1)}Q)|^2 <+\infty$ from the asymptotic \fref{eq:degenerescencesoliton}. For the following term, we decompose:
$$
\tilde{\alpha}_{\pi(b),\frac{\bar{\lambda}}{\lambda}}^{(1)}-\tilde{\alpha}_{\bar{b}}^{(1)})= \tilde{\alpha}_{\pi(b),\frac{\bar{\lambda}}{\lambda}}^{(1)}-\tilde{\alpha}_{\pi(b)}^{(1)}+ \tilde{\alpha}_{\pi(b)}^{(1)}-\tilde{\alpha}_{\bar{b}}^{(1)}.
$$
For the first part, using the analogue of formula \fref{variete:eq:expression Qscale-Q}:
$$
\begin{array}{r c l}
\int |\nabla^{\sigma} \tilde{\alpha}_{\pi(b),\frac{\bar{\lambda}}{\lambda}}^{(1)}-\tilde{\alpha}_{\pi(b)}^{(1)}|^2&\leq& \left|\frac{\bar{\lambda}}{\lambda}-1\right|^2 \underset{\theta\in [0,1] }{sup} \left|1-\theta+\theta\frac{\bar{\lambda}}{\lambda}\right|^{2(\sigma-s_c)-2}\\
&&\times \int |\nabla^{\sigma}(\Lambda^{(1)}\tilde{\alpha}_{\pi(b)}^{(1)})|^2\leq C(L,M)\bar{b}_1^{2L+2}
\end{array}
$$
because $\int |\nabla^{\sigma}(\Lambda^{(1)}\tilde{\alpha}_{\pi(b)}^{(1)})|^2<+ \infty$ from the asymptotic \fref{linearized:eq:degre Si} and Lemma \ref{lem:profilsTi}. For the other part, \fref{variete:eq:difference polynomes parametres}, \fref{variete:eq:borne meme temps decomposition b} and again the same asymptotics yield:
$$
\int |\nabla^{\sigma} (\tilde{\alpha}_{\pi(b)}^{(1)}-\tilde{\alpha}_{\bar{b}}^{(1)} )|^2\leq \bar{b}_1^4.
$$
Putting together the last two estimates gives:
$$
\int |\nabla^{\sigma}(\tilde{\alpha}_{\pi(b),\frac{\bar{\lambda}}{\lambda}}^{(1)}-\tilde{\alpha}_{\bar{b}}^{(1)})|^2\lesssim \bar{b}_1^4.
$$
The last remaining term is estimated similarily:
$$
\int |\nabla^{\sigma}((S_{L+1}-\overline{S}_{L+1})_{\frac{\bar{\lambda}}{\lambda}})|^2\lesssim \bar{b}_1^4
$$
The estimate we have done for each term of the right hand side of \fref{variete:eq:expression barvarepsilon} give:
$$
\int |\nabla^{\sigma}\bar{\varepsilon}^{(1)}|^2\leq C(K_2) \bar{b}_1^{2(\sigma-s_c)(1+\nu)}.
$$
Using the very same method, one finds the same estimation for the second coordinate, leading to the result (iii).
\end{proof}

The same lower order decomposition also applies for the other solution $\bos{U}'$, and we have the analogue of the previous lemma. What we want to do now is to apply the Proposition \ref{variete:prop:parametres lipschitz} associated to the integer $\bar{L}$ to these two new solutions in the trapped regime associated to the integer $\bar{L}$. There remains two steps: we have to check that the differences between the parameters and errors under the lower order decomposition can be related to the original decomposition, and we have to deal with a possible scale difference at initial time. We use the following notations for the lower order decomposition associated to $\bos{U}'$ by the Definition \fref{variete:def:decomposition adaptee varepsilon} :
\be
\bos{U}'(t)=\tilde{\bos{Q}}_{\bar{b}',\frac{1}{\bar{\lambda}'}}+\overline{\bos{w}}'(t)= (\tilde{\bos{Q}}_{\bar{b}'}+\overline{\bos{\varepsilon}}(\bar{s}'))_{\frac{1}{\bar{\lambda}'}},
\ee
where $\overline{\bos{\varepsilon}}'$ satisfies the $L$ orthogonality conditions:
\be
\langle \overline{\bos{\varepsilon}}',\bos{H}^{*i}\bos{\Phi}_M \rangle= 0, \ for \ 0 \leq i \leq \bar L.
\ee
Similarly we define ($\bar{P}_{\ell}$ being the analogue of $P_{\ell}$ defined by \fref{linearized:eq:def P}):
\be 
\bar{s}':=s_0+\int_0^t \frac{1}{\bar{\lambda}'(\tau)}d\tau,
\ee
\be
\bar{U}_i':=(\bar{s}')^i (\bar{b}'_i-\frac{c_i}{(\bar{s}')^i}), \ \text{for} \ 1\leq i \leq L, \ \text{and} \ \bar{V}:=\bar{P}_{\ell}(\bar{U}).
\ee
We use the following notations for the differences under lower order decomposition:
$$
\triangle \bar{b}:=\bar{b}_i-\bar{b}_i', \ \triangle \bar{U}_i:=\bar{U}_i-\bar{U}'_i, \ \triangle \bar \lambda:=\bar \lambda - \bar \lambda '.
$$
We make now a slight change regarding the former norm notations. They now concern $\bos{w}$ instead of $\bos{\varepsilon}$:
$$
\triangle \overline{\mathcal{E}}_{\sigma}:=\parallel \overline{\bos{w}}-\overline{\bos{w}}' \parallel_{\dot{H}^{\sigma}\times \dot{H}^{\sigma-1}}^2, \ \triangle_r \overline{\mathcal{E}}_{\sigma}:=b_1^{-2(\sigma-s_c)(1+\nu)}\triangle \overline{\mathcal{E}}_{\sigma},
$$
$$
\triangle \overline{\mathcal{E}}_{s_\bar{L}}:=\int (\bar{w}^{(1)}-\bar{w}^{'(1)})_{s_{\bar{L}}}^2+(\bar{w}^{(2)}-\bar{w}^{'(2)})_{s_{\bar{L}}}^2, \ \triangle_r \overline{\mathcal{E}}_{s_\bar{L}}:=b_1^{-2\bar{L}-(1-\delta_0)(2+\eta)}\triangle \overline{\mathcal{E}}_{s_\bar{L}}.
$$

In the following lemma we relate the differences between the lower order decomposition and the original decomposition at initial time. Basically, the differences of the two solutions in lower or higher order are almost the same.

\begin{lemma}[Bounds for the differences at initial time] \label{variete:lem: estimation difference ini}
We keep the assumptions and notations from Definitions \ref{variete:def:decomposition adaptee varepsilon} and Proposition \ref{variete:prop:parametres lipschitz 2}. There holds initially:
\begin{itemize}
\item[(i)] \emph{bounds on the parameters:} For $1 \leq i \leq \bar{L}$:
\be \label{variete:removal:eq:diff parametres}
\begin{array}{r c l}
\triangle \bar{U}_i(s_0) &=&\triangle U_i (s_0) + O[b_1|\triangle U(s_0)|+b_1^{(1-\delta_0)(1+\eta)}|\triangle \lambda|]\\
&&+O(C(s_0) \parallel \bos{w}-\bos{w}'\parallel_{\dot{H}^{\sigma}\cap \dot{H}^{s_L}\times \dot{H}^{\sigma-1}\cap \dot{H}^{s_L-1}}),
\end{array}
\ee
\item[(ii)] \emph{bounds on the errors:}
\be \label{variete:removal:eq:diff high}
\sqrt{\triangle_r \overline{\mathcal{E}}_{s_{\bar{L}}}}\leq C(s_0)\parallel \bos{w}-\bos{w}'\parallel_{\dot{H}^{\sigma}\cap \dot{H}^{s_L}\times \dot{H}^{\sigma-1}\cap \dot{H}^{s_L-1}}+C(b_1 |\triangle U|+b_1^{\frac{\eta}{2}}|\triangle \lambda|),
\ee
\be \label{variete:removal:eq:diff low}
\sqrt{\triangle_r \overline{\mathcal{E}}_{\sigma}}\leq C(s_0)\parallel \bos{w}-\bos{w}'\parallel_{\dot{H}^{\sigma}\cap \dot{H}^{s_L}\times \dot{H}^{\sigma-1}\cap \dot{H}^{s_L-1}}+Cb_1^{\bar{L}}(|\triangle U|+|\triangle \lambda|).
\ee
\item[(iii)] \emph{bound on the scales:}
\be \label{variete:removal:eq:diff scale}
\begin{array}{r c l}
\triangle \bar{\lambda} (s_0)&=&\triangle \lambda (s_0) + O[b_1|\triangle U(s_0)|+b_1^{(1-\delta_0)(1+\eta)}|\triangle \lambda|]\\
&&+O( C(s_0)\parallel \bos{w}-\bos{w}'\parallel_{\dot{H}^{\sigma}\cap \dot{H}^{s_L}\times \dot{H}^{\sigma-1}\cap \dot{H}^{s_L-1}}),
\end{array}
\ee
\end{itemize}
for some constant $C$ independent of the other constants.

\end{lemma}

\begin{remark}
\label{variete:rem:dependance high sobo lambda}
In all the previous computations, $\bos{w}$ and $\bos{w}'$, or $\overline{\bos{w}}$ and $\overline{\bos{w}}'$ were always at the same scale: there was no confusion regarding orthogonality conditions or adapted norms. Now, in the case of Lemma \fref{variete:lem: estimation difference ini}, each error has a different scale: $\lambda$, $\lambda'$, $\bar \lambda$ and $\bar \lambda '$. From \fref{variete:removal:eq:lambda sim 1} and  \fref{variete:eq:borne meme temps decomposition lambda} they are all close to one:
$$
|\lambda-1|+|\lambda'-1|+|\bar{\lambda}-1|+|\bar{\lambda}'-1|\lesssim b_1^{L}.
$$
From coercivity (see \fref{thetrapped:highsobo:eq:estimation lambdaL}) we obtain that for $\bos{f}\in \dot{H}^{\sigma}\cap \dot{H}^{s_L}\times \dot{H}^{\sigma-1}\cap \dot{H}^{s_L-1}$ satisfying the orthogonality conditions \fref{eq:condition d'ortho pour epsilon} and $\tilde{\lambda}$ close enough to $1$:
$$
\parallel f^{(1)}_{s_L}-((f_{\tilde{\lambda}}^{(1)})_{s_L})_{\frac{1}{\tilde{\lambda}}} \parallel_{L^2}+\parallel f^{(2)}_{s_L}-((f_{\tilde{\lambda}}^{(2)})_{s_L})_{\frac{1}{\tilde{\lambda}}} \parallel_{L^2} \lesssim |\tilde{\lambda}-1| (\parallel f^{(1)}_{s_L} \parallel_{L^2}+\parallel f^{(2)}_{s_L-1} \parallel_{L^2}),
$$
from what we deduce that the scale does not matter for this adapted norm:
$$
\parallel ((f_{\tilde{\lambda}}^{(1)})_{s_L})_{\frac{1}{\tilde{\lambda}}} \parallel_{L^2}+\parallel ((f_{\tilde{\lambda}}^{(2)})_{s_L})_{\frac{1}{\tilde{\lambda}}} \parallel_{L^2} \sim \parallel f^{(1)}_{s_L} \parallel_{L^2}+\parallel f^{(2)}_{s_L-1} \parallel_{L^2}.
$$
\end{remark}

\begin{proof}[Proof of Lemma \ref{variete:lem: estimation difference ini}]. To ease notations, we do not mention the dependence with respect to time: all objects are taken at time $s_0$. At this initial time, one has:
\be \label{variete:lower:eq:expression}
\tilde{\bos{Q}}_{b,\frac{1}{\lambda}}-\tilde{\bos{Q}}_{b',\frac{1}{\lambda '}}-(\tilde{\bos{Q}}_{\bar{b},\frac{1}{\bar{\lambda}}}-\tilde{\bos{Q}}_{\bar{b}',\frac{1}{\bar{\lambda} '}})+\bos{w}-\bos{w}'-(\overline{\bos{w}}-\overline{\bos{w}}')=0.
\ee
We introduce the following the notation:
$$
D=\sum_i^{\bar{L}}|\triangle b_i-\triangle \bar{b}_i|.
$$
Throughout the proof we will use Remark \ref{variete:rem:dependance high sobo lambda} and the fact that from \fref{variete:lower:proximite U} the parameters have the same size:
$$
b_i\approx b_i'\approx \bar{b}_i\approx \bar{b}'_i \ \text{for} \ 1\leq i \leq \bar{L}.
$$
$\overline{\bos{w}}'$ satisfies the orthogonality conditions \fref{eq:condition d'ortho pour epsilon}, but at the scale $\frac{1}{\bar \lambda '}$. To deal with the problem of the scale in orthogonality conditions and adapted norms, we introduce:
\be
\overline{\bos{v}}':=\overline{\bos{w}}'-\sum_0^{\bar L}\frac{\langle \overline{\bos{w}}',(\bos{H}^{*i}\bos{\Phi}_M)_{\frac{1}{\bar \lambda}}\rangle}{\langle \Lambda^{(1)}Q,\chi_M \Lambda^{(1)}Q \rangle}\chi_{B_1}\bos{T}_i.
\ee
Thus, $\overline{\bos{v}}'$ satisfies the orthogonality conditions \fref{eq:condition d'ortho pour epsilon} at the scale $\bar{\lambda}$. It is very close to $\overline{\bos{w}}'$ and one has the estimates:
\be
\parallel \overline{\bos{w}}'-\overline{\bos{v}}'\parallel_{\dot{H}^{\sigma}\times \dot{H}^{\sigma-1}}\leq |\triangle \bar \lambda|b_1^{\alpha+1-\delta_0+O(\eta,\sigma-s_c)},
\ee
\be \label{variete:eq:estimation high v}
\parallel (\overline{w}^{'(1)}-\overline{v}^{'(1)})_{s_{\bar{L}}}\parallel_{L^2}+\parallel (\overline{w}^{'(2)}-\overline{v}^{'(2)})_{s_{\bar{L}}-1}\parallel_{L^2}\leq |\triangle \bar \lambda|b_1^{\bar{L}+2(1-\delta_0)(1+\eta)},
\ee
\be \label{variete:eq:estimation infty v}
\parallel \overline{\bos{w}}'-\overline{\bos{v}}'\parallel_{L^{\infty}\times L^{\infty}(y\leq 2M)}\leq C|\triangle \bar \lambda|b_1^{\bar{L}+(1-\delta_0)(1+\eta)},
\ee

\underline{Step 1:} \emph{Difference of differences of polynomials of parameters}. We claim that for any $L$-tuple $J$ there holds:
\be \label{variete:removal:eq:difference polynomes parametres}
|b^J-\hat{b}^{'J}-(\bar{b}^J-\bar{b}^{'J})|\leq C (b_1^{|J|_2} |\triangle U-\triangle \bar{U}|+b_1^{L+1}|\triangle b|).
\ee
We show this bound by iteration on $|J|_1=i$. It is obviously true for $i=0$. We take $i\geq 1$ and $J$ satisfying $|J|_1=i$ and suppose it is true for all $J'$ with $|J'|_1\leq i-1$. Let $j$ be the first coordinate for which $J$ is non null and write $b^J=b_jb^{J'}$ with $|J'|= i-1$. We decompose:
$$
\begin{array}{r c l}
b^J-\hat{b}^{'J}-(\bar{b}^J-\bar{b}^{'J})&=&\bar{b}_j(b^{J'}-b^{'J'}-(\bar{b}^{J'}-\bar{b}^{'J'}))+\bar{b}^{'J'}(b_j-b'_j-(\bar{b}_j-\bar{b}'_j))\\
&&+(b_j-\bar{b}_j)(b^{J'}-b^{'J'})+(b^{'J'}-\bar{b}^{'J'})(b_j-b_j').
\end{array}
$$
From \fref{variete:lower:proximite U} one gets for the last two terms $|(b_j-\bar{b}_j)(b^{J'}-b^{'J'})+(b^{'J'}-\bar{b}^{'J'})(b_j-b_j')|\leq b_1^{L+1}|\triangle b|$. For the first two terms we apply the iteration hypothesis for $J'$ and conclude. \\

\underline{step 2: the scale}. We claim the first bound:
\be \label{variete:removal:eq:scale}
\triangle \bar{\lambda}=\triangle \lambda +O(b_1D)+ O[b_1^{L+1}|\triangle U|+b_1^{\bar{L}+(1-\delta_0)(1+\eta)}|\triangle \lambda|)]+O(b_1^{L}\sqrt{\triangle \overline{\mathcal{E}_{s_{\bar{L}}}}}).
\ee
We prove it by taking the scalar product of \fref{variete:lower:eq:expression} with $(\bos{\Phi}_M)_{\frac{1}{\lambda}}$. For the part on the manifold of approximate solutions one has the following decomposition:
\be \label{variete:removal:eq:Q expression}
\begin{array}{r c l}
\tilde{\bos{Q}}_{b,\frac{1}{\lambda}}-\tilde{\bos{Q}}_{b',\frac{1}{\lambda '}}-(\tilde{\bos{Q}}_{\bar{b},\frac{1}{\bar{\lambda}}}-\tilde{\bos{Q}}_{\bar{b}',\frac{1}{\bar{\lambda} '}})&=&(\tilde{\bos{Q}}_b-\tilde{\bos{Q}}_{b'}-(\tilde{\bos{Q}}_{\bar{b}}-\tilde{\bos{Q}}_{\bar{b}'}))_{\frac{1}{\bar{\lambda}}}\\
&&+((\tilde{\bos{Q}}_b-\tilde{\bos{Q}}_{b'})_{\frac{1}{\lambda}}-(\tilde{\bos{Q}}_b-\tilde{\bos{Q}}_{b'})_{\frac{1}{\bar{\lambda}}})\\
&&+((\tilde{\bos{Q}}_{b'}-\tilde{\bos{Q}}_{\bar{b}'})_{\frac{1}{\lambda}}-(\tilde{\bos{Q}}_{b'}-\tilde{\bos{Q}}_{\bar{b}'})_{\frac{1}{\lambda '}})\\
&&+(\tilde{\bos{Q}}_{\bar{b}',\frac{1}{\lambda}}-\tilde{\bos{Q}}_{\bar{b}',\frac{1}{\lambda '}}-(\tilde{\bos{Q}}_{	\bar{b}',\frac{1}{\bar{\lambda}}}-\tilde{\bos{Q}}_{\bar{b}',\frac{1}{\bar{\lambda} '}})).
\end{array}
\ee
We aim at estimating the contribution of each term in the right hand side. For the first term, from the orthogonality conditions \fref{thetrapped:eq:proprietes PhiM} and the localization \fref{thetrapped:eq:localisation PhiM}:
$$
\begin{array}{r c l}
&\langle(\tilde{\bos{Q}}_b-\tilde{\bos{Q}}_{b'}-(\tilde{\bos{Q}}_{\bar{b}}-\tilde{\bos{Q}}_{\bar{b}'}))_{\frac{1}{\bar{\lambda}}},\bos{\Phi}_{M,\frac{1}{\lambda}} \rangle \\
=& \langle(\tilde{Q}_b^{(1)}-\tilde{Q}_{b'}^{(1)}-(\tilde{Q}_{\bar{b}}^{(1)}-\tilde{Q}_{\bar{b}'}^{(1)})),\Phi_M^{(1)}\rangle \\
&+O(b_1^{L}\parallel \tilde{Q}^{(1)}_b-\tilde{Q}_{b'}^{(1)}-(\tilde{Q}_{\bar{b}}^{(1)}-\tilde{Q}_{\bar{b}'}^{(1)}) \parallel_{ L^2(\leq 2M)}) \\
=& \langle(S_{L+2}-S'_{L+2}+S_{L+1}-\bar{S}_{L+1}-(S'_{L+1}-\bar{S}'_{L+1}),\Phi_M^{(1)}\rangle\\
&+O(\parallel \sum_{i=1, \ \text{even}}^{L} S_i- S_i'-(\overline{S}_i-\overline{S}'_i) \parallel_{L^2(\leq 2M)})\\
 &+O(b_1^{L}\parallel \tilde{Q}^{(1)}_b-\tilde{Q}_{b'}^{(1)}-(\tilde{Q}_{\bar{b}}^{(1)}-\tilde{Q}_{\bar{b}'}^{(1)}) \parallel_{ L^2(\leq 2M)})
\end{array}
$$
Now, one decomposes the profiles $S_i$'s for $1\leq i \leq L$ as a finite sum $S_i=\sum b^J f$ with $|J|_2=i$ and $f$ a $C^{\infty}$ function. Applying \fref{variete:removal:eq:difference polynomes parametres} gives (we recall that $D$ is defined at the begining of the proof):
$$
\parallel S_i- S_i'-(\overline{S}_i-\overline{S}'_i) \parallel_{L^2(\leq 2M)}=O(b_1D)+O(b_1^{L+1}|\triangle b|).
$$
So for the first term in \fref{variete:removal:eq:Q expression} we obtain:
$$
\langle\tilde{\bos{Q}}_b-\tilde{\bos{Q}}_{b'}-(\tilde{\bos{Q}}_{\bar{b}}-\tilde{\bos{Q}}_{\bar{b}'})_{\frac{1}{\bar{\lambda}}},\bos{\Phi}_{M,\frac{1}{\lambda}}\rangle=O(b_1 D)+O(b_1^{L+1}|\triangle U|).
$$
From \fref{variete:eq:borne meme temps decomposition lambda} we get for the second:
$$
\langle(\tilde{\bos{Q}}_b-\tilde{\bos{Q}}_{b'})_{\frac{1}{\lambda}}-(\tilde{\bos{Q}}_b-\tilde{\bos{Q}}_{b'})_{\frac{1}{\bar{\lambda}}},\bos{\Phi}_{M,\frac{1}{\lambda}}\rangle=O(b_1^{L+1}|\triangle b|).
$$
For the third from \fref{variete:lower:proximite U} one has:
$$
\langle(\tilde{\bos{Q}}_{b'}-\tilde{\bos{Q}}_{\bar{b}'})_{\frac{1}{\lambda}}-(\tilde{\bos{Q}}_{b'}-\tilde{\bos{Q}}_{\bar{b}'})_{\frac{1}{\lambda '}},\bos{\Phi}_{M,\frac{1}{\lambda}}\rangle=O(b_1^{L+1}|\triangle \lambda|).
$$
For the fourth we decompose, use \fref{variete:lower:proximite U} and \fref{variete:eq:borne meme temps decomposition lambda} to find:
$$
\begin{array}{r c l}
&\langle \tilde{\bos{Q}}_{\bar{b}',\frac 1 \lambda}-\tilde{\bos{Q}}_{\bar{b}',\frac{1}{\lambda '}}-(\tilde{\bos{Q}}_{\bar{b}',\frac{1}{\bar{\lambda}}}-\tilde{\bos{Q}}_{\bar{b}',\frac{1}{\bar{\lambda} '}}),\bos{\Phi}_{M,\frac{1}{\lambda}}\rangle \\
=& \langle (\tilde{\bos{Q}}_{\bar{b}',\frac{1}{\lambda}}-\tilde{\bos{Q}}_{\bar{b}',\frac{1}{\lambda '}})-(\tilde{\bos{Q}}_{\bar{b}',\frac{1}{\lambda}}-\tilde{\bos{Q}}_{\bar{b}',\frac{1}{\lambda '}})_{\frac{\lambda}{\bar{\lambda}}},\bos{\Phi}_{M,\frac{1}{\lambda}}\rangle +\langle \tilde{\bos{Q}}_{\bar{b}',\frac{1}{\bar{\lambda}'}}-\tilde{\bos{Q}}_{\bar{b}',\frac{\lambda }{\bar{\lambda}\lambda'}},\bos{\Phi}_M\rangle \\
=& O(b_1^{L+1}|\triangle \lambda |)-(\bar{\lambda}'-\frac{\lambda' \bar{\lambda}}{\lambda})(\langle \chi_M \Lambda^{(1)}Q,\Lambda^{(1)}Q\rangle +O(b_1)) \\
=& (\triangle \lambda-\triangle \bar{\lambda} )(\langle \chi_M \Lambda^{(1)}Q,\Lambda^{(1)}Q\rangle +O(b_1))+O(b_1^{L+1}|\triangle \lambda|)
\end{array}
$$
from the identity $\bar{\lambda}'-\frac{\lambda'\bar{\lambda}}{\lambda}=\triangle \bar{\lambda}-\triangle \lambda+\frac{1}{\lambda}(\lambda-\bar{\lambda})(\lambda'-\lambda)$. The decomposition \fref{variete:removal:eq:Q expression} and the four previous equations give for the contribution of the difference of differences of approximate profiles in \fref{variete:lower:eq:expression}:
\be \label{variete:removal:eq:Q}
\begin{array}{r c l}
&\langle \tilde{\bos{Q}}_{b,\frac{1}{\lambda}}-\tilde{\bos{Q}}_{b',\frac{1}{\lambda '}}-(\tilde{\bos{Q}}_{\bar{b},\frac{1}{\bar{\lambda}}}-\tilde{\bos{Q}}_{\bar{b}',\frac{1}{\bar{\lambda} '}}),\bos{\Phi}_{M,\frac{1}{\lambda}}\rangle\\
 =& (\triangle \lambda-\triangle \bar{\lambda} )(\langle \chi_M \Lambda^{(1)}Q,\Lambda^{(1)}Q\rangle +O(b_1))+O[b_1^{L+1}(|\triangle U|+|\triangle \lambda|)]+O(b_1 D).
\end{array}
\ee
We now turn to the contribution of the difference of differences of errors in \fref{variete:lower:eq:expression}. We compute using the orthogonality conditions \fref{eq:condition d'ortho pour epsilon} and the hypothesis $|1-\lambda|\leq b_1^{L}$:
$$
\langle \bos{w}-\bos{w}',\bos{\Phi}_{M,\frac{1}{\lambda}} \rangle = -\langle \bos{w}',\bos{\Phi}_{M,\frac{1}{\lambda}}-\bos{\Phi}_{M,\frac{1}{\lambda'}} \rangle = O(b_1^{L+(1-\delta_0)(1+\eta)}|\triangle \lambda |).
$$
Using the variable $\overline{\bos{v}}'$ introduced at the begining of the proof, one can use coercivity thanks to Remark \fref{variete:rem:dependance high sobo lambda}:
$$
\begin{array}{r c l}
\langle \overline{\bos{w}}-\overline{\bos{w}}',\bos{\Phi}_{M,\frac{1}{\lambda}} \rangle &=& \langle \overline{\bos{w}}-\overline{\bos{v}}',\bos{\Phi}_{M,\frac{1}{\lambda}} \rangle+\langle \overline{\bos{v}}'-\overline{\bos{w}}',\bos{\Phi}_{M,\frac{1}{\lambda}} \rangle \\
&=& \langle \overline{\bos{w}}-\overline{\bos{v}}',\bos{\Phi}_{M,\frac{1}{\lambda}}-\bos{\Phi}_{M,\frac{1}{\bar{\lambda}}} \rangle+\langle \overline{\bos{v}}'-\overline{\bos{w}}',\bos{\Phi}_{M,\frac{1}{\lambda}} \rangle \\
 &=& O(b_1^{L}\sqrt{\triangle \overline{\mathcal{E}}_{s_\bar{L}}})+O(b_1^{\bar{L}+(1-\delta_0)(1+\eta)}|\triangle \bar{\lambda} |).
\end{array}
$$
where we used the estimates \fref{variete:eq:estimation high v} and \fref{variete:eq:estimation infty v}. We put the two previous estimates for the contribution of the errors and \fref{variete:removal:eq:Q} in \fref{variete:lower:eq:expression}, it gives the estimate \fref{variete:removal:eq:scale} we claimed in this step 2.\\

\underline{step 3: the parameters}. We claim that the techniques employed in the previous step adapts when we consider the scalar product between \fref{variete:lower:eq:expression} and $(\bos{H}^{*i}\bos{\Phi}_M)_{\frac{1}{\lambda}}$ for $1\leq i \leq L$, yielding:
$$
\begin{array}{r c l}
\triangle \bar{b}_i&=&\triangle b_i +O(b_1D)+ O[b_1^{L+1}(|\triangle U|+|\triangle \lambda|)]+O(b_1^{L}\sqrt{\triangle \overline{\mathcal{E}_{s_{\bar{L}}}}})\\
&&+O(b_1^{\bar{L}+(1-\delta_0)(1+\eta)}|\triangle \bar{\lambda}|).
\end{array}
$$
Injecting the bound \fref{variete:removal:eq:parametres}, the previous equation simplifies into:
\be \label{variete:removal:eq:parametres}
\triangle \bar{b}_i=\triangle b_i +O(b_1D)+ O[b_1^{L+1}(|\triangle U|)+b_1^{\bar{L}+(1-\delta_0)(1+\eta)}|\triangle \lambda|)]+O(b_1^{L}\sqrt{\triangle \overline{\mathcal{E}_{s_{\bar{L}}}}})
\ee

\underline{Step 4:improving the bounds}. We sum the previous identity \fref{variete:removal:eq:parametres} from $i=1$ to $\bar{L}$, it gives:
$$
D=O[b_1^{L+1}(|\triangle U|)+b_1^{\bar{L}+(1-\delta_0)(1+\eta)}|\triangle \lambda|)]+O(b_1^{L}\sqrt{\triangle \overline{\mathcal{E}_{s_{\bar{L}}}}})+O(b_1^{L}\sqrt{\triangle \mathcal{E}_{s_L}}).
$$
Putting back this bound in \fref{variete:removal:eq:scale} and \fref{variete:removal:eq:parametres} yield:
\be \label{variete:removal:eq:scale 2}
\triangle \bar{\lambda}=\triangle \lambda + O[b_1^{L+1}(|\triangle U|)+b_1^{\bar{L}+(1-\delta_0)(1+\eta)}|\triangle \lambda|)]+O(b_1^{L}\sqrt{\triangle \overline{\mathcal{E}_{s_{\bar{L}}}}}),
\ee
\be \label{variete:removal:eq:parametres 2}
\triangle \bar{b}_i=\triangle b_i + O[b_1^{L+1}(|\triangle U|)+b_1^{\bar{L}+(1-\delta_0)(1+\eta)}|\triangle \lambda|)]+O(b_1^{L}\sqrt{\triangle \overline{\mathcal{E}_{s_{\bar{L}}}}}).
\ee

\underline{step 5: the error terms}. The difference between the two error terms in lower order decomposition is:
$$
\overline{\bos{w}}-\overline{\bos{w}}'=\tilde{\bos{Q}}_{b,\frac{1}{\lambda}}-\tilde{\bos{Q}}_{b',\frac{1}{\lambda '}}-(\tilde{\bos{Q}}_{\bar{b},\frac{1}{\bar{\lambda}}}-\tilde{\bos{Q}}_{\bar{b}',\frac{1}{\bar{\lambda} '}})+\bos{w}-\bos{w}'.
$$
Injecting the bounds \fref{variete:removal:eq:scale 2} and \fref{variete:removal:eq:parametres 2} in the decomposition \fref{variete:removal:eq:Q expression} gives:
$$
\begin{array}{r c l}
\sqrt{\triangle \overline{\mathcal{E}}_{s_{\bar{L}}}}=\sqrt{\triangle \mathcal{E}}_{s_{\bar{L}}}+O[b_1^{L+1}(|\triangle U|)+b_1^{\bar{L}+(1-\delta_0)(1+\eta)}|\triangle \lambda|)]+O(b_1^{L}\sqrt{\triangle \overline{\mathcal{E}_{s_{\bar{L}}}}}),
\end{array}
$$
Now, as:
$$
\sqrt{\triangle \mathcal{E}}_{s_{\bar{L}}}\lesssim \parallel \bos{w}-\bos{w}' \parallel_{\dot{H}^{\sigma}\cap\dot{H}^{s_L}\times \dot{H}^{\sigma-1}\cap \dot{H}^{s_L-1}}
$$
it gives:
$$
\sqrt{\triangle_r \overline{\mathcal{E}}_{s_{\bar{L}}}}\leq C(s_0)\parallel \bos{w}-\bos{w}' \parallel_{\dot{H}^{\sigma}\cap\dot{H}^{s_L}\times \dot{H}^{\sigma-1}\cap \dot{H}^{s_L-1}}+\sqrt{\triangle_r \mathcal{E}}_{\sigma})+C(b_1 |\triangle U|+b_1^{\frac{\eta}{2}}|\triangle \lambda|).
$$
We turn back to the previous identities \fref{variete:removal:eq:scale 2} and \fref{variete:removal:eq:parametres 2}, inject the bound we just found to obtain:
$$
\triangle \bar{\lambda}=\triangle \lambda + O[b_1^{L+1}|\triangle U|+b_1^{\bar{L}+(1-\delta_0)(1+\eta)}|\triangle \lambda|]+O( \parallel \bos{w}-\bos{w}' \parallel_{\dot{H}^{\sigma}\cap\dot{H}^{s_L}\times \dot{H}^{\sigma-1}\cap \dot{H}^{s_L-1}}),
$$
$$
\triangle \bar{b}_i=\triangle b_i + O[b_1^{L+1}|\triangle U|+b_1^{\bar{L}+(1-\delta_0)(1+\eta)}|\triangle \lambda|]+O(\parallel \bos{w}-\bos{w}' \parallel_{\dot{H}^{\sigma}\cap\dot{H}^{s_L}\times \dot{H}^{\sigma-1}\cap \dot{H}^{s_L-1}}),
$$
where the constant in the $O()$ depends on $s_0$. These two bounds allow us to compute the last norm of $\overline{\bos{w}}-\overline{\bos{w}}'$:
$$
\sqrt{\triangle \overline{\mathcal{E}}_{\sigma}}\leq C(b_1^{L+1}|\triangle U|+b_1^{\bar{L}+(1-\delta_0)(1+\eta)}|\triangle \lambda|)+C(s_0)\parallel \bos{w}-\bos{w}' \parallel_{\dot{H}^{\sigma}\cap\dot{H}^{s_L}\times \dot{H}^{\sigma-1}\cap \dot{H}^{s_L-1}}.
$$
The four last bounds directly imply the bounds of the lemma we had to prove.

\end{proof}

We are now ready to end the proof of Proposition \ref{variete:prop:parametres lipschitz 2}.

\begin{proof}[Proof of Proposition \ref{variete:prop:parametres lipschitz 2}] 
Let $\bos{U}$ and $\bos{U}'$ be two solutions described by the Proposition \ref{variete:prop:parametres lipschitz 2}. We associate to the two solutions their lower order decomposition described by Definition \ref{variete:def:decomposition adaptee varepsilon}. Without loss of generality, we can assume $\bar{\lambda}(s_0)\leq \bar{\lambda}'(s_0)$, which means that in lower order decomposition, the second solution starts at a higher scale than the first one.\\
\\
We let the second solution evolve with time and define $\bar{s}'_1$ as the time at which its scale is the same as the initial scale of the first solution in lower order decomposition: $\bar{\lambda}'(\bar{s}'_1)=\bar{\lambda}(s_0)$. We now estimate the difference between the second solution, in lower order decomposition, taken at these two times. From the equation \fref{thetrapped:eq:modulation leq L-1} governing the time evolution of the scale one has:
$$
|\bar{s}'_1-s_0|\leq Cb_1^{-1}|\bar{\lambda}'(s_0)-\bar{\lambda}(s_0)|.
$$
We can then estimate, from \fref{thetrapped:eq:modulation leq L-1} and \fref{thetrapped:eq:improved modulation} the time variation of the parameters:
$$
|\bar{b}'_i(\bar{s}'_1)-\bar{b}'_i(s_0)|\leq Cb_1^i |\bar{\lambda}'(s_0)-\bar{\lambda}(s_0)|.
$$
Let us now quantify how the error changed. In the proof of the energy estimate for the high adapted Sobolev norm \fref{thetrapped:eq:high sobo}, we computed the size of everything in the right hand side of \fref{eq:evolution W}. We computed also the influence of the scale changing in \fref{thetrapped:highsobo:eq:estimation lambdaL}. The form of this energy estimate was meant to cancel the linear part, see \fref{thetrapped:highsobo:eq:estimation lineaire}. But we have here the additional regularity \fref{variete:eq:estimation supplementaire varepsilon'} for the second solution under the lower order decomposition. Thus all these estimates yield:
$$
\parallel \frac{d}{d\bar{s}'}[(((\bar{w}^{'(1)}_{\bar{\lambda}'})_{s_{\bar{L}}})_{\frac{1}{\bar{\lambda}'}}],\frac{d}{d\bar{s}'}[(((\bar{w}^{'(2)}_{\bar{\lambda}'})_{s_{\bar{L}-1}})_{\frac{1}{\bar{\lambda}'}}]\parallel_{L^2\times L^2} \leq C b_1^{\bar{L}+1+(1-\delta_0)(1+\eta)}.
$$
From that we deduce (using Remark \ref{variete:rem:dependance high sobo lambda}):
$$
\begin{array}{r c l}
&\parallel \bar{w}^{'(1)}_{s_{\bar{L}}}(\bar{s}'_1)-((\bar{w}^{'(1)}_{\frac{\bar{\lambda}'}{\bar{\lambda}}}(s_0))_{s_{\bar{L}}})_{\frac{\bar{\lambda}}{\bar{\lambda}'}},\bar{w}^{'(2)}_{s_{\bar{L}}-1}(\bar{s}'_1)-((\bar{w}^{'(2)}_{\frac{\bar{\lambda}'}{\bar{\lambda}}}(s_0))_{s_{\bar{L}}-1})_{\frac{\bar{\lambda}}{\bar{\lambda}'}}\parallel_{L^2\times L^2} \\
\leq& C b_1^{\bar{L}+(1-\delta_0)(1+\eta)}|\bar{\lambda}'(s_0)-\bar{\lambda}(s_0)|.
\end{array}
$$
A similar result holds for the low regularity Sobolev norm:
$$
\parallel \bar{w}^{'(1)}(\bar{s}'_1)-\bar{w}^{'(1)}(s_0),\bar{w}^{'(2)}(\bar{s}'_1)-\bar{w}^{'(2)}(s_0)\parallel_{\dot{H}^{\sigma}\times \dot{H}^{\sigma-1}} \leq C b_1^{(\sigma-s_c)(1+\nu)}|\bar{\lambda}'(s_0)-\bar{\lambda}(s_0)|.
$$
We now apply the result of Proposition \ref{variete:prop:parametres lipschitz} to $\overline{\bos{U}}(s_0)$ and $\overline{\bos{U}}'(\bar{s}'_1)$. It gives the primary Lipschitz bound (using Remark \ref{variete:rem:dependance high sobo lambda}):
$$
\begin{array}{r c l}
&|\bar{V}_{\text{uns}}(s_0)-\bar{V}'_{\text{uns}}(\bar{s}'_1)|\\
\leq& C\Bigl( | \bar{V}_1(s_0)-\bar{V}'_1(\bar{s}'_1)|+\sum_{\ell+1}^L| \bar{U}_i(s_0)-\bar{U}'_i(\bar{s}'_1)| \\
&+b_1^{-(\sigma-s_c)(1+\nu)}\parallel \overline{\bos{w}}(s_0)-\overline{\bos{w}}'(\bar{s}'_1) \parallel_{\dot{H}^{\sigma}\times\dot{H}^{\sigma-1}}\\
&+b_1^{-\bar{L}-(1-\delta_0)(1+\eta)}\parallel \bar{w}^{(1)}_{s_{\bar{L}}}(s_0)-\bar{w}^{'(1)}_{s_{\bar{L}}}(\bar{s}'_1),\bar{w}^{(2)}_{s_{\bar{L}}-1}(s_0)-\bar{w}^{'(1)}_{s_{\bar{L}}-1}(\bar{s}'_1)  \parallel_{L^2\times L^2} \Bigr)
\end{array}
$$
for the variables under lower order decomposition and at different times $s_0$ and $\bar{s}'_1$. We now use the four previous bounds that link the variables for the second solution under lower order decomposition between the times $s_0$ and $\bar{s}'_1$ to obtain from the previous equation:
$$
\begin{array}{r c l}
&|\bar{V}_{\text{uns}}(s_0)-\bar{V}'_{\text{uns}}(s_0)|\\
\leq& C\Bigl( | \bar{V}_1(s_0)-\bar{V}'_1(s_0)|+\sum_{\ell+1}^L| \bar{U}_i(s_0)-\bar{U}'_i(s_0)|+|\bar{\lambda}(s_0)-\bar{\lambda}'(s_0)| \\
&+b_1^{-(\sigma-s_c)(1+\nu)}\parallel \overline{\bos{w}}(s_0)-\overline{\bos{w}}'(s_0) \parallel_{\dot{H}^{\sigma}\times\dot{H}^{\sigma-1}}\\
&+b_1^{-\bar{L}-(1-\delta_0)(1+\eta)}\parallel \bar{w}^{(1)}_{s_{\bar{L}}}(s_0)-\bar{w}^{'(1)}_{s_{\bar{L}}}( s_0),\bar{w}^{(2)}_{s_{\bar{L}}-1}(s_0)-\bar{w}^{'(1)}_{s_{\bar{L}}-1}(s_0)  \parallel_{L^2\times L^2} \Bigr),
\end{array}
$$
where we used again Remark \ref{variete:rem:dependance high sobo lambda}. The previous identity is the Lipschitz aspect under lower order decomposition. To relate it to the original higher order decomposition, we use the bounds \fref{variete:removal:eq:diff parametres}, \fref{variete:removal:eq:diff high}, \fref{variete:removal:eq:diff low} and \fref{variete:removal:eq:diff scale} of the previous Lemma \fref{variete:lem: estimation difference ini}, and we obtain the result of the Proposition.

\end{proof}

We can now end the proof of the main Theorem \ref{variete:thm;variete} of this section.

\begin{proof}[Proof of Theorem \fref{variete:thm;variete}]
We let $\bos{X}:=\left(\dot{H}^{\sigma}\cap \dot{H}^{s_L} \right)\times \left(\dot{H}^{\sigma-1}\cap \dot{H}^{s_L-1} \right)$ and $\bos{U}_0\in \bos{X}$ be a solution leading to a type II blow up as described by Proposition \ref{prop:bootstrap}. Without loss of generality we can assume that its scale is $1$. We then write:
$$
\bos{U}_0=\tilde{\bos{Q}}_{b_0}+\bos{w}_0,
$$
with $b_0=b^e(s_0)+(\frac{U_1(s_0)}{s_0},...,\frac{U_L(s_0)}{s_0^L})$ according to the decomposition explained in Subsubsection \ref{thetrapped:subsubsection:modulation}.\\

\underline{Step 1:} Flattening the non linear coordinates. Let $\bos{U}'_0\in \bos{X}$ be another initial datum. It can be written as:
\be \label{variete:eq:decomposition lineaire}
\bos{U}'_0=\bos{U}_0+\delta \lambda \frac{\partial}{\partial \lambda}(\tilde{\bos{Q}}_{b_0,\frac{1}{\lambda}})_{| \lambda=1}+\sum_1^L \frac{\delta U_i}{s_0^i} \frac{\partial}{\partial b_i}(\tilde{\bos{Q}}_{b})_{| b=b_0}+\delta \bos{w},
\ee
where $\delta \bos{w}\in X$ satisfies fixed orthogonality conditions at scale $1$: $\langle \delta \bos{w},\bos{H}^{*i}\bos{\Phi}_M\rangle=0$ for $0\leq i \leq M$. We have seen that for the parameters one had stable directions of perturbation $V_1,U_{\ell+1},...,U_L$, unstable ones $V_2,...V_{\ell}$ and that the error $\bos{w}$ was a stable perturbation. We recall the notation $V_i=\sum_1^{\ell}p_{i,j}U_j$. With the decomposition we just stated we can define the stable and unstable spaces of linearized directions of perturbation:
$$
\begin{array}{l l l}
\bos{X}_s&:=\Bigl\{& \delta \lambda \frac{\partial}{\partial \lambda}(\tilde{\bos{Q}}_{b_0,\frac{1}{\lambda}})_{| \lambda=1}+\delta V_1(\underset{1}{\overset{\ell}{\sum}} \frac{p_{1,j}}{s_0^j}\frac{\partial}{\partial b_j}(\tilde{\bos{Q}}_{b})_{| b=b_0})+\underset{\ell+1}{\overset{L}{\sum}} \frac{\delta U_i}{s_0^i} \frac{\partial}{\partial b_i}(\tilde{\bos{Q}}_{b})_{| b=b_0}+ \bos{w},\\
&&(\delta \lambda, \delta V_1,\delta U_{\ell+1},...,\delta U_L)\in \mathbb{R}^{L-\ell+2},\\
&&  \delta \bos{w}\in \bos{X}, \ \langle \delta \bos{w},\bos{H}^{*i}\bos{\Phi}_M\rangle=0 \ \text{for} \ 0\leq i \leq M  \Bigr\},\\
\bos{X}_u&:=\Bigl\{ & \sum_2^{\ell} \delta V_i (\sum_1^{\ell} \frac{p_{i,j}}{s_0^j} \frac{\partial}{\partial b_j}(\tilde{\bos{Q}}_{b})_{| b=b_0}),  \ (\delta V_2,...,\delta V_{\ell})\in \mathbb{R}^{\ell-1} \Bigr\}.
\end{array}
$$
So that we decompose in an affine way $\bos{X}=\bos{U}_0+(\bos{X}_u \oplus\bos{X}_s )$. \\

\underline{Step 2:} From linear to adapted coordinates. To be able to use the results of Proposition \ref{prop:bootstrap} and Proposition \ref{variete:prop:parametres lipschitz 2} we consider the following mapping:
$$
\begin{array}{l l l l}
\bos{\phi}:&\bos{X}&\rightarrow &\bos{X} \\
&\bos{U}&\mapsto & \tilde{\bos{Q}}_{b+\delta b,\frac{1}{1+\delta \lambda}}+\tilde{\bos{w}} 
\end{array}
$$
where, using the decomposition \fref{variete:eq:decomposition lineaire}, we define $\delta b$ as $\delta b:=(\frac{\delta U_1}{s_0},...,\frac{\delta U_L}{s_0^L})$ and:
\be \label{variete:eq:decomposition non lineaire}
\tilde{\bos{w}}:=\bos{w}+\delta \bos{w} -\sum_0^L \frac{\langle \bos{w}+\delta \bos{w},(\bos{H}^{*i}\bos{\Phi}_M)_{\frac{1}{1+\delta \lambda}} \rangle}{\langle \bos{T}_{i,\frac{1}{1+\delta \lambda}},(\bos{H}^{*i}\bos{\Phi}_M)_{\frac{1}{1+\delta \lambda}} \rangle} \bos{T}_{i,\frac{1}{1+\delta \lambda}}
\ee
satisfies the orthogonality conditions \fref{eq:condition d'ortho pour epsilon} at the scale $\frac{1}{1+\delta \lambda}$: 
$$
\langle \delta \bos{w},(\bos{H}^{*i}\bos{\Phi}_M)_{\frac{1}{1+\delta \lambda}}\rangle=0 \ \text{for} \ 0\leq i \leq M.
$$
$\bos{\phi}$ is a $C^{\infty}$ diffeomorphism that preserves $\bos{U}_0$: $\bos{\phi}(\bos{U}_0)=\bos{U}_0$.\\

\underline{Step 3:} the Lipschitz manifold properties. Let 
$$
 \delta \lambda \frac{\partial}{\partial \lambda}(\tilde{\bos{Q}}_{b_0,\frac{1}{\lambda}})_{| \lambda=1}+ \delta V_1(\sum_1^{\ell} \frac{p_{1,j}}{s_0^j}\frac{\partial}{\partial b_j}(\tilde{\bos{Q}}_{b})_{| b=b_0})+\sum_{\ell+1}^L \frac{\delta U_i}{s_0^i} \frac{\partial}{\partial b_i}(\tilde{\bos{Q}}_{b})_{| b=b_0}+ \bos{w}:=\delta \bos{U}_s\in X_s
$$
be small enough. We apply the result of Proposition \ref{prop:bootstrap} to $\bos{\phi}(\bos{U}_0+\delta \bos{U}_s)$. There exists a choice of unstable modes $\delta V_2,...,\delta V_{\ell}$ such that $\tilde{\bos{U}}:= \tilde{\bos{Q}}_{b+\delta b,\frac{1}{1+\delta \lambda}}+\tilde{\bos{w}} $ is an initial datum leading to a blow up as described in this Proposition, where $\delta b:=(\frac{\delta U_1}{s_0},...,\frac{\delta U_L}{s_0^L})$, and $\delta V_i:=\sum_1^{\ell} p_{i,j}U_j$ for $1\leq i \leq \ell$. Moreover, from Proposition \fref{variete:prop:parametres lipschitz 2} the $\ell-1$-tuple $\delta V_2,...,\delta V_{\ell}$ is unique. We then have:
$$
\bos{\phi}^{-1}(\tilde{\bos{U}})=\bos{U}_0+\delta \bos{U}_s+\delta \bos{U}_u,
$$
with $\delta \bos{U}_u:=\sum_2^{\ell} \delta V_i (\sum_1^{\ell} \frac{p_{i,j}}{s_0^j} \frac{\partial}{\partial b_j}(\tilde{\bos{Q}}_{b})_{| b=b_0})\in \bos{X}_u$. Let $\bos{\mathcal{O}}$ be a small enough open set of $\bos{X}$ with $\bos{0}\in \bos{\mathcal{O}}$. We define the application $\bos{f}$ as:
$$
\begin{array}{l l l l}
\bos{f}:& \delta \bos{\mathcal{O}}\cap \bos{X}_s & \rightarrow & \delta \bos{X}_u \\
&\bos{U}_s & \mapsto & \bos{U}_u
\end{array}
$$
with $\bos{X}_u$ being defined by the previous construction. For $\bos{U}_s\in \bos{X}_s\cap \bos{\mathcal{O}}$, the function $\bos{\phi}(\bos{U}_0+\delta \bos{U}_s+\bos{f}(\delta \bos{U}_s)) $ yields a type II blow up as described by Proposition \ref{prop:bootstrap}. Moreover, Proposition \ref{variete:prop:parametres lipschitz 2} implies that $\bos{f}$ is a Lipsichitz mapping. Let $\bos{\mathcal{M}}$ denote the set of initial data described by Proposition \ref{prop:bootstrap}. We just have proved that $\bos{\phi}^{-1}(\bos{\mathcal{M}}\cap (\bos{U}_0+\bos{\mathcal{O}}))$ is the graph of the Lipschitz mapping $\bos{f}:\bos{X}_s \cap \bos{\mathcal{O}}\rightarrow \bos{X}_u$ with $\bos{X}=\bos{X}_u\oplus\bos{X}_s $ and $\bos{X}_u$ of dimension $\ell-1$. This means that $\bos{\phi}^{-1}(\bos{\mathcal{M}}\cap (\bos{U}_0+\bos{\mathcal{O}}))$ is a Lipschitz manifold of codimension $\ell-1$. As $\bos{\phi}$ is a $C^{\infty}$ diffeomorphism, it implies that $\bos{\mathcal{M}}\cap (\bos{U}_0+\bos{\mathcal{O}}))$ is a Lipschitz manifold of codimension $\ell-1$. Hence $\bos{\mathcal{M}}$ is a locally Lipschitz manifold of codimension $\ell-1$ in $\bos{X}$.

\end{proof}


\begin{appendix}

\section{Properties of the stationnary state}

We state here the fundamental decomposition for the asymptotic of the stationary state Q. These results are now standard, see \cite{YiLi} \cite{GNW} for exemple, and see also \cite{MRRod2} for its role in type II blow-up involving $Q$ for the Schr\"odinger equation. An important fact, the non nullity of the second term in the expansion, is however not proven in these works. We therefore prove it hereafter.

\begin{lemma}\label{lem:expansion_soliton}
(\emph{Asymptotic expansion for the stationary state:})
We have the expansion:
\begin{equation}
\partial_y^k Q(y)=\partial_y^k\left( \frac{c_{\infty}}{y^{\frac{2}{p-1}}}+\frac{a_1}{y^{\gamma}}\right)+O\left( \frac{1}{y^{\gamma+g+k}} \right) \ \ \ \text{as} \ y \ \text{goes} \ \text{to} \ +\infty ,
\end{equation}
with $a_1$ being a strictly negative (in particular $a_1\neq 0$) coefficient:
\be \label{annexe:eq:a1 strict negatif}
a_1<0
\ee
\end{lemma}

In \cite{GNW} and references therein, the authors show the expansion, but they do not show that $a_1\neq 0$. This appendix is devoted to prove this fact. In the paper the authors show the following result:

\begin{lemma}[Gui Ni Wang, \cite{GNW}, Theorem 2.5]
We recall that $0<\alpha_1<\alpha_2$ are the roots of the polynomial:
\begin{equation}
X^2-\left(d-2-\frac{4}{p-1}\right)X+2\left(d-2-\frac{2}{p-1}\right) .
\end{equation}
Then the following expansion is true.
\begin{itemize}
\item[(i)] \emph{If $\frac{\alpha_2}{\alpha_1}\notin \mathbb{N}$}, then for all $k_1, \ k_2 \in \mathbb{N}$, as $y\rightarrow +\infty$ one has:
\begin{equation}
Q(y)=\frac{c_{\infty}}{y^{\frac{2}{p-1}}}+\sum_{i,j=1}^{k_1,k_2}\frac{a_{i,j}}{y^{\frac{2}{p-1}+k_1\alpha_1+k_2\alpha_2}}+O\left(\frac{1}{y^{\frac{2}{p-1}+(k_1+1)\alpha_1}} \right) .
\end{equation}
\item[(ii)] \emph{If $\frac{\alpha_2}{\alpha_1}=k+1\in \mathbb{N}$:} then as $y\rightarrow +\infty$ one has::
\begin{equation}
Q(y)=\frac{c_{\infty}}{y^{\frac{2}{p-1}}}+\sum_{i=1}^{k+1}\frac{a_i}{y^{\frac{2}{p-1}+i\alpha_1}}+\frac{a_k \text{log}(y)+a_k'}{y^{\frac{2}{p-1}+k\alpha_1}}+O \left( \frac{1}{y^{\frac{2}{p-1}+(k+1)\alpha_1}} \right).
\end{equation}
As in the previous case the expansion can be continued to higher terms, but it does not matter for the analysis of the present paper.
\item[(iii)] This expansion adapts for higher derivatives of $Q$.
\end{itemize}
\end{lemma}

This proves the expansion of Lemma \ref{lem:expansion_soliton}. The rest of this section is devoted to the proof that $a_1$ is strictly negative.

\begin{proof}[Proof of the assertion \fref{annexe:eq:a1 strict negatif}]
As a consequence of the previous lemma we get that, noting $k:=E[\frac{\lambda_2}{\lambda_1}]$ if $\frac{\alpha_2}{\alpha_1}\notin \mathbb{N}$, and $k:=\frac{\alpha_2}{\alpha_1}-1$ if $\frac{\alpha_2}{\alpha_1}\in \mathbb{N}$ we have in both cases:
\begin{equation}\label{eq:expansionlambdaQ}
\Lambda^{(1)}Q= \sum_{i=1}^{k}-i\alpha_1 \frac{a_1}{y^{\frac{2}{p-1}+i\alpha_1}}+O\left( \frac{\text{log}(y)}{y^{\frac{2}{p-1}+\alpha_2}}\right) ,
\end{equation}
and:
\begin{equation}\label{eq:expansion_nablalambdaQ}
\partial_y \Lambda^{(1)} Q= \sum_{i=1}^{k}(i\alpha_1)\left(\frac{2}{p-1}+i\alpha_1\right) \frac{a_1}{y^{\frac{2}{p-1}+i\alpha_1+1}}+O\left( \frac{\text{log}(y)}{y^{\frac{2}{p-1}+\alpha_2+1}}\right) .
\end{equation}

The key point is that the coefficient $a_i$ are linked with a reccurence relation:

\begin{lemma}\label{lem:relationrecurrencecoefasymptotique}
For $1\leq i \leq k$, $a_i$ is given by $a_i=P_i(a_1)$ where $P_i$ is a polynomial such that $P_i(0)=0$ for all $1\leq i \leq k$.
\end{lemma}

This lemma is proved later. Hence we have the following alternative:
\begin{equation}
either \ a_1 \neq 0 \ or \ \partial_y \Lambda^{(1)} Q= O\left( \frac{log(y)}{y^{\frac{2}{p-1}+\alpha_2+1}}\right) .
\end{equation}

The remainder term of \fref{eq:expansion_nablalambdaQ} is in $L^2$. Indeed, we compute:
$$
d-2\frac{2}{p-1}-2\alpha_2-2=-\sqrt{\triangle}<0 .
$$
So If $a_1=0$ then $\Lambda^{(1)}Q \in \dot{H}^1$. The term associated to $a_1$ is not in $L^2$ because $d-2\frac{2}{p-1}-2\alpha_1-2=\sqrt{\triangle}>0$, see \fref{defgamma}.
\\But we know from \cite{KaS} that $\mathcal{L}$ is positive definite on $\dot{H}^1$, and that $\mathcal{L}\Lambda^{(1)} Q=0$. We then must have $\Lambda^{(1)}Q \notin \dot{H}^1$. Considering what was said previously, this implies $a_1\neq 0$.\\
\\
We also know from \cite{KaS} that $\Lambda^{(1)}Q>0$. From the expansion \fref{eq:expansionlambdaQ} This implies that $a_1$ is strictly negative.
\end{proof}

We now give the proof of the recurrence relation between the $a_i$'s stated in Lemma \ref{lem:relationrecurrencecoefasymptotique}.

\begin{proof}[Proof of Lemma \ref{lem:relationrecurrencecoefasymptotique}]
We use here the ideas developped in \cite{YiLi}. In this paper or in references therin, the following facts are proven:
\begin{lemma}[\cite{YiLi} Lemmas 4.3 and 4.4] The following holds:
\begin{itemize} 
\item[(i)] the solitary wave exists and has $C^{\infty}$ regularity.
\item[(ii)] $y^{\frac{2}{p-1}}Q(y)$ has a limit as $y\rightarrow +\infty$, denoted $c_{\infty}$.
\item[(iii)] If we renormalise the space variable by $y=e^t$ and define:
\begin{equation}
W(t)=y^{\frac{2}{p-1}}Q(y)-c_{\infty} .
\end{equation}
 $W$ then satisfies the differential equation for $t$ large:
\begin{equation}\label{eq:equadiffdeW}
W_{tt}+\left(d-2-\frac{4}{p-1}\right)+2\left(d-2-\frac{2}{p-1}\right)W+P(W)=0 ,
\end{equation}
where $P$ denotes the polynomial:
\be \label{annexe:eq:expression P}
(X+c_{\infty})^p-c_{\infty}^p-pc_{\infty}^{p-1}X.
\ee
\item[(iv)] $W$ has the following begining of expansion at infinity:
\begin{equation}
W(t)=\left\{ \begin{array}{l l} 
a_1 e^{-\alpha_1t}+O(e^{-\alpha_2 t}) \ \text{if} \ \alpha_2<2\alpha_1 \\
a_1 e^{-\alpha_1t}+O(te^{-\alpha_2 t}) \ \text{if} \ \alpha_2=2\alpha_1 \\
a_1 e^{-\alpha_1t}+O(e^{-2\alpha_1 t}) \ \text{if} \ \alpha_2>2\alpha_1 .
\end{array}
\right.
\end{equation}
\end{itemize}
\end{lemma}

We will now compute the other coefficients of the expansion. As $W$ is a solution of \fref{eq:equadiffdeW}, basic ODE theory states that there exists two coefficients $a$ and $b$ such that:
\begin{equation}\label{eq:formuledonnantW}
W(t)=ae^{-\alpha_1 t}+be^{-\alpha_2 t} +\frac{1}{\alpha_2-\alpha_1}\int_{T_0}^t(e^{\alpha_2(s-t}-e^{\alpha_1(s-t)})P(W)ds .
\end{equation}
We now prove lemma \ref{lem:relationrecurrencecoefasymptotique} by iteration. Our iteration hypothesis is the following for $1\leq j\leq k-1$:
\be \label{annexe:hp de reccurence}
\mathcal{H}(j): \ \begin{array}{ll} W(t)=\sum_{i=1}^{j} a_ie^{-i\alpha_1 t}+O(e^{-(j+1)\alpha_1 t}), \ \text{with} \ a_i=P_i(a_1), \\
P_i \ \text{being a polynomial such that} \  P_i(0)=0.\end{array}
\ee
\underline{Initialization:} For $i=1$, $a_1=P_1(a_1)$ with $P_1=X$ and of course $P_1(0)=0$. Because of the preliminary expansion (iv), the property is true for $j=1$.\\
\\
\underline{Heredity:} We now suppose it is true for $1\leq j \leq k-1$. We then plug the expansion \fref{annexe:hp de reccurence} into \fref{eq:formuledonnantW}. It gives the following expression for $W$:
\begin{equation}\label{eq:formule_W_pour_obtention_coefficients}
W(t)=ae^{-\alpha_1 t}+\frac{1}{\alpha_2-\alpha_1}+\int_{T_0}^t(e^{\alpha_2(s-t}-e^{\alpha_1(s-t)})P(W)ds+O(e^{-(j+1)\alpha_1 t}) ,
\end{equation}
since $(j+1)\alpha_1<\alpha_2$ (because $1\leq j\leq k-1$). But with the definition \fref{annexe:eq:expression P} of $P$ and the hypothesis \fref{annexe:hp de reccurence} on the $a_i$ for $i\leq j$ we have that:
$$
P(W(t))=\sum_{i=2}^{j+1}\tilde{a_i}e^{-i\alpha_1 t}+O(e^{-(j+2)\alpha_1 t}) ,
$$
where $\tilde{a_i}=\tilde{P}_i(a_1)$ with $\tilde{P}_i$ being a polynomial such that $\tilde{P}_i(0)=0$. We now put this expression in \fref{eq:formule_W_pour_obtention_coefficients} and compute the integral of the right hand side. For $2\leq i \leq j+1$:
$$
\begin{array}{r c l}
\int_{T_0}^t(e^{\alpha_2(s-t}-e^{\alpha_1(s-t)})e^{-i\alpha_1 s}ds
&=& \frac{1}{\alpha_2-i\alpha_1}e^{-\alpha_2 t}-\frac{1}{(i-1)\alpha_1}e^{-\alpha_1 t} \\
&& +\left(\frac{1}{\alpha_2-i\alpha_1}+\frac{1}{(i-1)\alpha_1}\right)e^{-i\alpha_1 t},
\end{array}
$$
and:
\begin{equation}\label{eq:calcul_iteration_remainder}
\begin{array}{r c l}
\int_{T_0}^t(e^{\alpha_2(s-t}-e^{\alpha_1(s-t)})O(e^{-(j+2)\alpha_1 s})ds &=& e^{-\alpha_2 t} \int_{T_0}^t O(e^{(\alpha_2-(j+2)\alpha_1) s})ds \\
&&-e^{-\alpha_1 t} \int_{T_0}^t O(e^{-(j+1)\alpha_1 s})ds . \\
\end{array}
\end{equation}
Since $\alpha_2>(j+2)\alpha_1$ the first integral diverges, the second term is integrable. Hence:
$$
\begin{array}{r c l}
\int_{T_0}^t(e^{\alpha_2(s-t}-e^{\alpha_1(s-t)})O(e^{-(j+2)\alpha_1 s})ds &=& e^{-\alpha_2 t}O(\int_{T_0}^t e^{(\alpha_2-(j+2)\alpha_1) s}ds ) \\
&& - e^{-\alpha_1 t} \big{(}\int_{T_0}^{+\infty} O(e^{-(j+1)\alpha_1}) \\
&& -\int_t^{+\infty}O(e^{-(j+1)\alpha_1}\big{)} \\
&=& Ce^{-\alpha_1 t}+O(e^{-(j+2)\alpha_1 t}) .
\end{array}
$$
So we finally get for a constant $C$:
\begin{equation}
\begin{array}{r c l}
W(t)&=& Ce^{-\alpha_1 t}+\sum_{i=2}^{j+1} \tilde{a_i}\frac{1}{\alpha_2-\alpha_2}\left(\frac{1}{\alpha_2-i\alpha_1}-\frac{1}{-(i-1)\alpha_1}\right)e^{-i\alpha_1 t} \\
&+& O(e^{-(i+2)\alpha_1 t}) .
\end{array}
\end{equation}
By identifying this last identity with the expansion \fref{annexe:hp de reccurence} given by the induction hypothesis, one finds that in fact $C=a_1$ and $a_i=\tilde{a_i}$ for $i\leq j$. Therefore the property $\mathcal{H}(j) $ is true for $j+1$.\\
\\
By induction, we have proved that \fref{annexe:hp de reccurence} is valid for $j=k-1$. To finish the proof one needs to do the same computation that we did before for the case $j=k-1$
\begin{itemize}
\item[(i)] If $\frac{\alpha_2}{\alpha_1}\neq \mathbb{N}$. Then the only things that changes is that we do not have $e^{-\alpha_2 t}=O(e^{-(k+1)\alpha_1 t})$, so we cannot throw away the terms involving $e^{-\alpha_2 t}$ and we get:
$$
W(t)=Ce^{-\alpha_1 t}+\sum_{i=2}^{k} \tilde{a_i} e^{-i\alpha_1 t}+O(e^{-(k+1)\alpha_1 t} ).
$$
\item[(ii)] If $\frac{\alpha_2}{\alpha_1}$ is an integer, and $k=\frac{\alpha_2}{\alpha_1}-1$ to go from $k-1$ to $k$ we also do the same computations as before. Now what changes is that we have a $t$ corrective term in \fref{eq:calcul_iteration_remainder}:
$$
\int_{T_0}^t e^{\alpha_2(s-t)}O(e^{-(k+1)\alpha_1 t})=O(te^{-\alpha_2 t}) .
$$
which is what produces the log term in the expansion of $Q$ in that case.
\end{itemize}

\end{proof}


\section{Equivalence of norms}

In this subsection we show that the notion of degree for admissible functions (see Definition \ref{linearized:def:fonctions admissibles}) is equivalent for usual derivatives and adapted ones. We also show that the weighted usual Sobolev norms are equivalent, to some extent, to the weighted adapted ones.

\begin{lemma}\label{annexe:lem:equivalence degre}\emph{(equivalence of the degree)}
Let $p_2$ be a real number and $f$ a $C^{\infty}$ radial function. We recall that $f_k$ is the $k$-th adapted derivative defined in \fref{linearized:eq:def derivees adaptees}. The two following proposition are equivalents:
\begin{itemize}
\item[(i)] $\forall k\geq 0$, $\partial_y^k f=O\left( \frac{1}{y^{p_2+k}}\right) \ \text{as} \ y \rightarrow +\infty$ .
\item[(ii)] $\forall k\geq 0$, $f_k=O\left( \frac{1}{y^{p_2+k}}\right) \ \text{as} \ y\rightarrow +\infty$ .
\end{itemize} 
Let $a\in \mathbb{R}$. For any $u\in C^{\infty}_{rad}$ there holds:\footnote{the quantity need not be finite. By $x\sim y$ we mean here $\frac{x}{c}\leq y \leq c x $ for $c>0$.}
\be \label{annexe:eq:equivalence derivees adaptees et usuelles hors origine}
\sum_{i=0}^k \int_{y\geq 1} \frac{|\partial_y^i u|^2}{1+y^{2k-2i+2a}}\sim \sum_{i=0}^k \int_{y\geq 1} \frac{|u_i|^2}{1+y^{2k-2i+2a}} .
\ee

\end{lemma}

\begin{proof}[Proof of Lemma \ref{annexe:lem:equivalence degre}] We just show that (i) implies (ii), the other implication being similar. So we suppose:
$$
f\in C^{\infty}_{rad}, \ \text{with} \ \forall k\geq 0, \ f_k=O\left( \frac{1}{y^{p_2+k}}\right) \ \text{as} \ y\rightarrow + \infty .
$$
We are going to show to following property by induction: for $i$ an integer:, for all $0\leq j \leq i$ and $k\in\mathbb{N}$ there holds: 
$$
\mathcal{H}(i) \ \ \ \partial_y^k f_j =O\left( \frac{1}{y^{p_2+j+k}}\right)  \ \text{for all} \ 0\leq j \leq i \ \text{and} \ k\in\mathbb{N} .
$$
The property $\mathcal{H}(0)$ is obviously true from the supposition on $f$. Suppose now $\mathcal{H}(i)$ is true for $i$, and let $k\in \mathbb{N}$, suppose in addition that $i$ is odd. Then:
$$
\partial_y^k f_{i+1}=\partial_y^k(\mathcal{A^*}f_i)=\partial_y^k\left[\partial_y f_i+\left(\frac{d-1}{y}+W \right)f_i\right] .
$$
As $\partial_y^l\left(\frac{d-1}{y}+W \right)=O\left( \frac{1}{y^{l+1}}\right) $ the property $\mathcal{H}(i+1)$ is then true. If $i$ is even, then replacing $A^*$ by $A$ leads to the same result as they have the same structure (they divide or multiply by a potential similar to $y^{-1}$) at infinity. We have proven that if $\mathcal{H}(i)$ is true then so is $\mathcal{H}(i+1)$. Hence we have showed the first proposition of the lemma by induction.\\
\\
For the equivalence of the weighted norms away from the origin, we note that what we have just proven is the fact that for any integer $i$:
$$
\partial_y^i f=\sum_{j=0}^{i} a_{i,j}f_j \ \ \text{and } \ \ f_i=\sum_{j=0}^{i} \tilde{a}_{i,j}\partial_y^j f ,
$$
the functions $a_{i,j}$ and $\tilde{a}_{i,j}$ being radial and $C^{\infty}$ outside the origin, with $a_{i,j}=O(y^{-(i-j)})$ and $\tilde{a}_{i,j}=O(y^{-(i-j)})$ as $y\rightarrow +\infty$. This implies \fref{annexe:eq:equivalence derivees adaptees et usuelles hors origine}.
\end{proof}

We recall that the Laplace based derivatives of a $C^{\infty}$ functions are:
$$
D^{2k}u:=\Delta^k u, \ \text{and} \ D^{2k+1}u:=\partial_y \Delta^{k}u .
$$
\begin{lemma}\label{equivalence of norms}
\emph{(Equivalence of weighted adapted norms)}
There holds for all $u\in C^{\infty}$ radial function and integer $k$:
\be \label{annexe:equivalencenormes:eq:equivalences normes}
\sum_{i=0}^k \int \frac{u_i^2}{1+y^{2k-2i}}\sim \sum_{i=0}^k \int \frac{|D^i u|^2}{1+y^{2k-2i}} .
\ee
\end{lemma}

\begin{proof}[Proof of Lemma \ref{equivalence of norms}]
\underline{step 1:} \emph{Leibnitz rule}. Let $f$ and $u$ be $C^{\infty}$ radial, with:
$$
\partial_y^k f =O\left( y^{a-k} \right) \ \text{as} \ y\rightarrow +\infty ,
$$
for some real number $a$. We will show the following property by induction: for any integer $i$:
$$
\mathcal{H}(i): \ \ \ (fu)_i=\sum_{j=0}^i V_{i,j}(f)u_j ,
$$
$V_{i,j}(f)\in C^{\infty} $ depending just on $f$, with $\partial_y^k V_{i,j}(f)\sim y^{a-(j-i)-k}$, and with the regularity $\frac{V_{i,j}(f)}{y}\in C^{\infty}$ for $i-j$ odd.\\
\\
he property $\mathcal{H}(0)$ is obviously true. Suppose now it is true for $i$ odd:
$$
\begin{array}{r c l}
(fu)_{i+1}&=&A^*((fu)_i)=\sum_{j=0, \ j \ \text{even}}^i A^*(V_{i,j}u_j)+\sum_{j=0, \ j \  \text{odd}}^iA^*(V_{i,j}u_j) \\
&=& \sum_{j=0, \ j \ \text{even}}^i \left( -A+2W+\frac{d-1}{y}\right) (V_{i,j}u_{j+1})\\ &&+\sum_{j=0, \ j \ \text{odd}}^i \partial_y V_{i,j}u_j+V_{i,j}u_{j+1} \\
&=& \sum_{j=0, \ j \ \text{even}}^i V_{i,j}u_{j+1}+\left( \partial_y V_{i,j} +2WV_{i,j}+\frac{(d-1)V_{i,j}}{y}   \right)u_j \\
&&+\sum_{j=0, \ j \ \text{odd}}^i \partial_y V_{i,j}u_j+V_{i,j}u_{j+1} \\
&=& \sum_{j=0, \ (i+1-j) \text{even}} \left(\partial_y V_{i,j}+2W V_{i,j}+\frac{d-1}{y}V_{i,j}+V_{i,j-1}\right) u_j \\
&& +\sum_{j=0, \ (i+1-j) \text{odd}} \partial_y V_{i,j}u_j +V_{i,j-1}u_j .
\end{array}
$$
For the terms in the first sum we have: $\partial_y V_{i,j}+2W V_{i,j}+\frac{d-1}{y}V_{i,j}+V_{i,j-1}\in C^{\infty}$ because of the property for $i$, and it satisfies the decay propriety:
$$
\partial_y^k\left(\partial_y V_{i,j}+2W V_{i,j}+\frac{d-1}{y}V_{i,j}+V_{i,j-1}\right)=O\left( y^{a-(i+1-j)-k} \right) .
$$
For the second one the asymptotic property is also true from the induction hypothesis $\mathcal{H}(i)$, and we have indeed: $\frac{1}{y}(\partial_y(V_{i,j})+V_{i,j-1})\in C^{\infty}$. We have showed that if $\mathcal{H}(i)$ is true for $i$ odd, then $\mathcal{H}(i+1)$ is true. For $i$ even a similar reasoning gives also that $\mathcal{H}(i)$ implies $\mathcal{H}(i+1)$. Consequently, the proprerty $\mathcal{H}(i)$ holds for all $i\in \mathbb{N}$.\\
\\
\underline{Step 2:} \emph{passing from one derivation to the other:} We now claim that for any integer $i$ another property holds:
$$
\mathcal{H}'(i) \ \ \ D^iu=\sum_{j=0}^i \tilde{V_{i,j}} u_j ,
$$
with $V_{i,j}\in C^{\infty}$ satisfaying $\partial_y^k V_{i,j}\sim y^{-(i-j)-k}$, and for $j-i$ odd $\frac{1}{y}\tilde{V}_{i,j}\in C^{\infty}$. We show this property also by induction. It is true for $i=0,1,2$. Suppose now it is true for $i\geq 2$. Suppose $i$ even, then:
$$
\begin{array}{r c l}
D^{i+1}u&=&\partial_y(D^iu)=\sum_{j=0, \ j \ \text{even}}^i(-A+W)(V_{i,j}u_j) \\
&& + \sum_{j=0, \ j \ \text{odd}}^i (A^*-W-\frac{d-1}{y})(V_{i,j}u_j) \\
&=& \sum_{j=0, \ j \ \text{even}}^i -V_{i,j}u_{j+1}+\partial_y V_{i,j}u_j \\
&& + \sum_{j=0, \ j \ \text{odd}}^i V_{i,j}u_{j+1}+(\partial_y V_{i,j}-WV_{i,j}-\frac{d-1}{y}V_{i,j})u_j .
\end{array}
$$
The asymptotic behavior of the potentials is easily checked from the induction hypothesis. For $i+1-j$ odd we have: $\tilde{V}_{i+1,j}=\partial_y V_{i,j}+V_{i,j-1}$, which verifies indeed $\frac{1}{y}\tilde{V}_{i+1,j}\in C^{\infty}$ from the induction hypothesis $\mathcal{H}'(i)$. Hence $\mathcal{H}'(i+1)$ is true. We have shown $\mathcal{H}(i)$ implies $\mathcal{H}'(i+1)$ for $i$ even and claim that for $i$ odd a very similar proof shows the heredity. Therefore, the propriety $\mathcal{H}'(i)$ is true for any integer $i$\\
\\
This implies:
$$
\int |D^i u|^2 \leq C \sum_{j=0}^i \int \frac{u_j^2}{1+y^{2(i-j)}} ,
$$
which implies the control of the Laplace derivatives by adapted derivatives in the Lemma. The other inequality of the equivalence can be proved exactly the same way. The opposite formula holds indeed also:
$$
u_i=\sum_{j=0}^i \tilde{V}'_{i,j} D^{j}u ,
$$
with $\tilde{V}'_{i,j}\in C^{\infty}$, $\partial_y^k \tilde{V}'_{i,j}\sim y^{-(i-j)-k}$ and $\frac{1}{y}\tilde{V}'_{i,j}\in C^{\infty}$ if $i-j$ odd. The proof is left to the reader.

\end{proof}


\section{Hardy inequalities}

In this subsection we recall the standard Hardy estimates we used in the paper, in order to make this paper self contained. We use them to derive Hardy type estimates for the adapted norms, see next subsection. These analysis results, used to relate a norm that is adapted to a linear flow to the standard $L^2$ norms for usual derivatives, is now used in a canonical way in some works about blow-up, see  for exemple \cite{Rsc3} in a more subtle critical setting, \cite{MRRod2} in a supercritical setting.
\begin{lemma}\emph{(Hardy inequality with best constant)}\label{lemme:annexe:hardy1}
\begin{itemize}
\item[(i)]\emph{Hardy near the origin:} Let $u\in \cap_{0<r<1} H^1(\mathcal{C}(r,1))$, then:\footnote{Note that the quantities can be infinite.}
\begin{equation}
\int_{y\leq 1} |\partial_y u|^2y^{d-1}dy\geq \frac{(d-2)^2}{4}\int_{y\leq 1} \frac{u^2}{y^2}y^{d-1} dy -C(d) u^2(1) .
\end{equation}
\item[(ii)]\emph{Hardy away from the origin, non critical exponent:} Let $p>0$, $p\neq \frac{d-2}{2}$, and $u\in \cap_{1<R} H^1(\mathcal{C}(1,R))$. If $p$ is supercritical, $p>\frac{d-2}{2}$ then:\footnote{Note that the quantities can be infinite.}
\begin{equation}
\int_{y\geq 1} \frac{|\partial_y u|^2}{y^{2p}}y^{d-1}dy\geq \left( \frac{d-(2p+2)}{2}\right)^2 \int_{y\geq 1} \frac{u^2}{y^{2p+2}}y^{d-1}dy-C(d,p)u^2(1) ,
\end{equation}
\be \label{annexe:hardy:identite surcritique}
\frac{(2p+2-d)^2}{4} \int_1^R \frac{u^2}{y^{2p+2}}y^{d-1}dy \leq \int_1^R \frac{|\partial_y u|^2}{y^{2p}}y^{d-1}dy +C(d,p)u^2(1).
\ee
If $p$ is subcritical, $0<p<\frac{d-2}{2}$, if:\footnote{we need integrability this time, a constant function violates this rule for example.}
\begin{equation}\label{annexe:eq:condition d'inte A*}
\int_{y\geq 1} \frac{|u |^2}{y^{2p+2}}y^{d-1}dy <+\infty ,
\end{equation}
then:
\begin{equation}
\int_{y\geq 1} \frac{|\partial_y u|^2}{y^{2p}}y^{d-1}dy\geq \left( \frac{d-(2p+2)}{2}\right)^2 \int_{y\geq 1} \frac{u^2}{y^{2p+2}}y^{d-1}dy .
\end{equation}
\end{itemize}

\end{lemma}

\begin{proof}[Proof of Lemma \ref{lemme:annexe:hardy1}]
A proof of this lemma can be found in \cite{MRRod2}.
\end{proof}

We now state a useful refined version of Hardy inequality for arbitrary weight function and number of derivatives. We denote by $x:=(x_1,...,x_d)$ an element $x\in \mathbb{R}^d$. We introduce a notation for the partial derivatives of a function:
\be \label{annexe:hardyfrac:notations derivees}
\partial^{\kappa} f=\frac{\partial f}{\partial_{x_1}^{\kappa_1}...\partial_{x_d}^{\kappa_d} }
\ee
for a $d$-tuple $\kappa:=(\kappa_1,...,\kappa_d)$ with $|\kappa|_1=\sum_{i=1}^d \kappa_i$.
\begin{lemma}\label{annexe:lem: hardy frac a poids}
\emph{(Weighted Fractional Hardy :)} Let:
$$
0<\nu<1, \ k\in \mathbb{N} \ \text{and} \ 0<\alpha \ \text{satisfying} \ \alpha+\nu+k <\frac{d}{2} ,
$$
and let $f$ be a smooth function with decay estimates:
\begin{equation} \label{annexe:hardyfrac:eq:condition non radiale}
|\partial^{\kappa} f(x)|\leq \frac{C(f)}{1+|x|^{\alpha+i}}, \ \text{for} \ |\kappa|_1=i, \ i=0,1,...,k+1 , 
\end{equation}
then for $\varepsilon\in \dot{H}^{\alpha+k+\nu}$, there holds $\varepsilon f\in \dot{H}^{\nu+k} $ with:
\begin{equation} \label{annexe:hardyfrac:eq:hardyfrac}
\parallel \nabla^{\nu+k} (\varepsilon f)\parallel_{L^2} \leq C(C(f),\nu,k,\alpha,d) \parallel \nabla^{\alpha+k+\nu}\varepsilon \parallel_{L^2} .
\end{equation}
If $f$ is a smooth radial function satisfying:
\be \label{annexe:hardyfrac:eq:condition radiale}
|\partial_{|x|}^i f(|x|)|\leq \frac{C(f)}{1+|x|^{\alpha+i}}, \ i=0,1,...,k+1,
\ee
then \fref{annexe:hardyfrac:eq:hardyfrac} holds.
\end{lemma}

\begin{proof}[Proof of Lemma \ref{annexe:lem: hardy frac a poids}]
We first proove for $f$ satisfying the non radial condition \fref{annexe:hardyfrac:eq:condition non radiale}, and show after that for a radial function, this condition is equivalent to \fref{annexe:hardyfrac:eq:condition radiale} the radial condition mentionned in the Lemma.\\

 \underline{Step 1:} case for $k=0$. A proof of the case $k=0$ can be found in \cite{MRRod2} for example.\\

\underline{Step 2:} Proof for $k\geq 1$. Let $f$, $\varepsilon$, $\alpha$, $\nu$ and $k$ satisfying the conditions of the lemma, with $k\geq 1$. Using Liebnitz rule for the integer part of the derivation:
\be \label{annexe:hardyfrac:eq:kgeq1 decomposition}
\parallel \nabla^{\nu+k}(\varepsilon f) \parallel_{L^2}^2\leq C \sum_{(\kappa,\tilde{\kappa}), |\kappa||_1+|\tilde{\kappa}|_1=k} \parallel \nabla^{\nu} (\partial^{\kappa_k}\varepsilon\partial^{\tilde{\kappa}_k}f \parallel_{L^2}^2
\ee
We can now apply the result obtained for $k=0$ to the norms $\parallel \nabla^{\nu} (\partial^{\kappa_k}\varepsilon\partial^{\tilde{\kappa}_k}f \parallel_{L^2}^2$ in \fref{annexe:hardyfrac:eq:kgeq1 decomposition}. We have indeed that $\partial^{\kappa_k}\varepsilon\in \dot{H}^{\alpha+k_2+\nu} $, and that $\partial^{\tilde{\kappa}}$ satisfyies the decay property from \fref{annexe:hardyfrac:eq:condition non radiale}. It implies that for all $\kappa,\tilde{\kappa}$:
$$
\parallel \nabla^{\nu} (\partial^{\kappa_k}\varepsilon\partial^{\tilde{\kappa}_k}f \parallel_{L^2}^2\leq C \parallel \nabla^{\nu+\alpha+k} \varepsilon \parallel_{L^2}^2
$$
which implies the result: $\parallel \nabla^{\nu+k}(\varepsilon f) \parallel_{L^2}^2\leq C(C(f),\nu,d,k,\alpha) \parallel \nabla^{\nu+\alpha+k} \varepsilon \parallel_{L^2}^2$.\\

 \underline{Step 3:} equivalence between the decay properties. We want to show that  \fref{annexe:hardyfrac:eq:condition non radiale} and \fref{annexe:hardyfrac:eq:condition radiale} are equivalents for radial smooth functions, therefore implying the last assertion of the lemma. Suppose that $f$ is smooth, radial, and satisfies  \fref{annexe:hardyfrac:eq:condition non radiale}. Then one has:
$$
\partial_y^{i} f(y)= \frac{\partial f}{\partial^{i}_{x_1}}(|y|e_1)
$$
where $e_1$ stands for the unit vector $(1,...,0)$ of $\mathbb{R}^d$. From this formula, we see that the condition \fref{annexe:hardyfrac:eq:condition non radiale} on $\frac{\partial f}{\partial^{i}_{x_1}}(|y|e_1) $ implies the radial condition \fref{annexe:hardyfrac:eq:condition radiale}. We now suppose that $f$ is a smooth radial function satisfying the radial condition \fref{annexe:hardyfrac:eq:condition radiale}. Then there exists a smooth radial function $\phi$ such that:
$$
f(y)=\phi(y^2).
$$
With a proof by iteration left to the reader one has that the decay property \fref{annexe:hardyfrac:eq:condition radiale} for $f$ implies the following decay property for $\phi$:
$$
|\partial_y^i \phi(y)|\leq \frac{C(f)}{1+y^{\frac{\alpha}{2}+i}}, \ i=0,1,...,k+1,
$$
Now the standard derivatives of $f$ are easier to compute with $\phi$. We claim that for all $d$-tuple $\kappa$ there exists a finite number of polynomials $P_i(x):=C_ix_1^{i_1}...x_d^{i_d}$, for $1\leq i \leq l(\kappa)$, such that:
$$
\partial^{\kappa}f(x)= \sum_{i=1}^{l(\kappa)} P_i(x)\partial_{|x|}^{q(i)} \phi (|x|^2) 
$$
with for all $i$, $2q(i)-\sum_{j=1}^d i_j=|\kappa|_1$. This fact is also left to the reader. The decay property for $\phi $ then implies:
$$
|P_i(x)\partial_{|x|}^{q(i)} \phi (|x|^2) |\leq \frac{C}{1+y^{\alpha+2q(i)-\sum_{j=1}^{d}i_j}}= \frac{C}{1+y^{\alpha+|\kappa|_1}},
$$
which implies the property \fref{annexe:hardyfrac:eq:condition non radiale}.
\end{proof}


\section{Coercivity of the adapted norms}

Here we derive Hardy type inequalities for the operators $A$, $A^*$ and $\mathcal{L}$. Such quantities are easier to manipulate for the linear flow of the operator $\bos{H}$ (defined in \fref{eq:def:H}). As for the previous section of the Appendix, this kind of bounds is now standard and we refer to the papers quoted therein for the use of similar techniques. We start with $A^*$, then $A$, and after that we are able to deal with the coercivity of the adapted norms.\\
\\
We recall that the profile $\bos{\Phi}_M$ is defined by equation \fref{thetrapped:eq:def:phi M}. Its main properties that we will use in this section are its localization on the first coordinate and its non-orthogonality with respect to $\bos{\Lambda Q}$ (from \fref{thetrapped:eq:localisation PhiM} and \fref{thetrapped:eq:proprietes PhiM}):
\be \label{annexe:coercivitenormes:eq:proprietes PhiM}
\bos{\Phi}_M=\begin{pmatrix} \Phi_M \\ 0  \end{pmatrix}, \ \langle \bos{\Phi}_M,\bos{\Lambda Q}\rangle=\langle \Phi_m,\Lambda^{(1)}Q\rangle \sim CM^{2k_0+2\delta_0}>0 \ (C>0).
\ee
We also recall the structure of the two first order differential operators on radial functions $A$ and $A^*$:
\be \label{annexe:coercivitenormes:eq:def A}
A^*=\partial_y+\left(\frac{d-1}{y}+W\right), \ A=-\partial_y +W,
\ee
where $W$ is a smooth radial function with the asymptotic at infinity from \fref{linearized:eq:asymptotique W}:
\be \label{annexe:coercivitenormes:eq:asymptotique W}
W=\frac{-\gamma}{y}+O\left( \frac{1}{y^{1+g}}  \right) \ \text{as} \ y \ \rightarrow +\infty
\ee

\begin{lemma}\label{lem:coercivite pour A*}
\emph{(Weighted coercivity for $A^*$)}. Let $p$ be a non negative real number. Then there exists a constant $c_p>0$ such that for all radial $u\in H^1_{\text{loc}}(\mathbb{R}^d)$ there holds\footnote{The quantities need not be finite.}:
\begin{equation} \label{annexe:coerciviteA*:eq:coerciviteA*}
\int \frac{|A^*u|^2}{1+y^{2p}}\geq c_p \left[ \int \frac{u^2}{y^2(1+y^{2p})}+\int \frac{|\partial_y u|^2}{1+y^{2p}}  \right] .
\end{equation}
\end{lemma}

\begin{proof}[Proof of Lemma \ref{lem:coercivite pour A*}]
We take $u$ satisfying the conditions of the lemma.\\

\underline{Step 1:} \emph{Subcoercivity for $A^*$}. We claim the subcoercivity lower bound:
\begin{equation}\label{annexe:eq:estimation de sous coercivite A*}
\begin{array}{r c l}
\int \frac{|A^*u|^2}{1+y^{2p}}&\geq& c\left[ \int \frac{u^2}{y^2(1+y^{2p})}+\int \frac{|\partial_y u|^2}{1+y^{2p}}  \right] \\
&& -\frac{1}{c}\left[ u^2(1) +\int \frac{u^2}{1+y^{2p+g}}\right] ,
\end{array}
\end{equation}
for a universal constant $c=c(d,p)>0$. We introduce the operator: $\tilde{W}:=W+\frac{d-1}{y}$. First we estimate close to the origin:
\be \label{annexe:coerciviteA*:eq: pres}
\begin{array}{r c l}
\int_{y\leq 1}|A^*u|^2 &=& \int_{y\leq 1} (|\partial_y u|^2+\tilde{W}^2u^2+2\tilde{W}u\partial_y u) \\
&=& \int_{y\leq 1} |\partial_y u|^2 +\int_{y\leq 1}u^2\left(\tilde{W}^2-\frac{1}{y^{d-1}}(y^{d-1}\tilde{W})\right)+W(1)^2u(1)^2 \\
&\geq & \int_{y\leq 1} |\partial_y u|^2 +\int_{y\leq 1}u^2(\frac{(d-1)^2-(d-1)(d-2)}{y^2}+O(1))\\
&=& \int_{y\leq 1} |\partial_y u|^2 +(d-1)\int_{y\leq 1}\frac{u^2}{y^2}+O(\int_{y\leq 1} u^2) .
\end{array}
\ee
Away from the origin, from the asymptotic \fref{annexe:coercivitenormes:eq:asymptotique W}:
\be \label{annexe:coerciviteA*:eq:loin intermediaire}
\begin{array}{r c l}
\int_1^R \frac{|A^*u|^2}{y^{2p}}&=& \int_1^R \frac{1}{y^{2p}}(\partial_y u+\frac{d-1-\gamma}{y}u+O(\frac{1}{y^{1+g}})u)^2 \\
&=& \int_1^R \frac{1}{y^{2p}}[\partial_y u+\frac{d-1-\gamma}{y}u]^2  + \int_1^R uO\left( \frac{1}{y^{2p+1+g}} \right)\left( \partial_y u +uO\left(\frac{1}{y} \right) \right)\\
&=& \int_1^R \frac{1}{y^{2p+2(d-1-\gamma)}}|\partial_y(y^{d-1-\gamma}u)|^2  + \int_1^R uO\left( \frac{1}{y^{2p+1+g}} \right)\left( \partial_y u +uO\left(\frac{1}{y} \right) \right) .
\end{array}
\ee
Let $v=y^{d-1-\gamma}u$, and $p'=p+d-1-\gamma$. We have: $2p'-(d-2)=2p+d-2\gamma>0$.
Hence we can apply the identity \fref{annexe:hardy:identite surcritique}:
$$
\begin{array}{r c l}
\int_1^R \frac{1}{y^{2p+2(d-1-\gamma)}}|\partial_y(y^{d-1-\gamma}u)|^2&=& \int_1^R\frac{1}{y^{2p'}}|\partial_y v|^2 \geq C(d,p) \int_1^R \frac{v^2}{y^{2p'+2}}-C'v^2(1) \\
&=& \int_1^R \frac{u^2}{y^{2p+2}}-C'u^2(1) .
\end{array}
$$
We have by developing the expression, using Cauchy Schwarz and Young's inequality:
$$
\begin{array}{r c l}
\underset{1}{\overset{R}{\int}} \frac{1}{y^{2p+2(d-1-\gamma)}}|\partial_y(y^{d-1-\gamma}u)|^2 &\geq& \int_1^R \frac{|\partial_y u|^2}{y^{2p}}+C\frac{u^2}{y^{2p+2}}
-C' \left( \int_1^R \frac{|\partial_y u|^2}{y^{2p}}\right)^{\frac{1}{2}} \left( \int_1^R \frac{u^2}{y^{2p+2}} \right)^{\frac{1}{2}} \\
&\geq& (1-\frac{\epsilon}{2} C')\int_1^R \frac{|\partial_y u|^2}{y^{2p}}+(C-\frac{C'}{2\epsilon}) \int_1^R \frac{u^2}{y^{2p+2}} .
\end{array}
$$
Combining the last two estimates gives:
\be \label{annexe:coerciviteA*:eq:controle loin}
\int_1^R \frac{1}{y^{2p+2(d-1-\gamma)}}|\partial_y(y^{d-1-\gamma}u|^2 \geq c\left(  \int_1^R \frac{u^2}{y^{2p+2}}+\int_1^R \frac{|\partial_y u|^2}{y^{2p}}\right)-C'u^2(1) ,
\ee
for a constant $c>0$. We come back to \fref{annexe:coerciviteA*:eq:loin intermediaire} and inject the bound \fref{annexe:coerciviteA*:eq:controle loin}, it yields:
\be \label{annexe:coerciviteA*:eq:loin intermediaire 2}
\begin{array}{r c l}
\int_1^R \frac{|A^*u|^2}{y^{2p}} &\geq& c\left(  \int_1^R \frac{u^2}{y^{2p+2}}+\int_1^R \frac{|\partial_y u|^2}{y^{2p}}\right)-\frac{1}{c}u^2(1) \\
&&+ \int_1^R uO\left(\frac{1}{y^{2p+1+g}} \right)\left( \partial_y u +uO\left(\frac{1}{y} \right)\right).
\end{array}
\ee
We now use Cauchy-Schwarz and Young inequalities on better decaying term:
$$
\begin{array}{r c l}
&\left|\int_1^R uO\left(\frac{1}{y^{2p+1+g}} \right)\left( \partial_y u +uO\left(\frac{1}{y} \right)\right) \right| \\
\leq& C\epsilon \int_1^{R} \frac{|\partial_y u|^2}{y^{2p}}+\frac{C}{\epsilon}  \int_1^{R} \frac{|u|^2}{y^{2p+2+2g}}+C \int_1^{R} \frac{|u|^2}{y^{2p+2+g}}.
\end{array}
$$
Taking $\epsilon$ small enough and combining this bound with \fref{annexe:coerciviteA*:eq:loin intermediaire 2} gives for a constant $c>0$:
$$
\int_1^R \frac{|A^*u|^2}{y^{2p}} \geq c\left(  \int_1^R \frac{u^2}{y^{2p+2}}+\int_1^R \frac{|\partial_y u|^2}{y^{2p}}\right)-\frac{1}{c}\left(u^2(1)+ \int_1^R \frac{u^2}{y^{2p+2+g}}\right) \\
$$
Because of the additional decay in the last term we have that if $\frac{u^2}{y^{2p+2}}$ or $\frac{|\partial_y u|^2}{y^{2p}}$ is non integrable at infinity, then going to the limit $R\rightarrow 0$ gives that $\frac{|A^*u|^2}{y^{2p}}$ is non integrable. Therefore in that case all quantities in \fref{annexe:coerciviteA*:eq:coerciviteA*} are infinite and the inequality is proven. Now, if they are integrable, then going to the limit $R\rightarrow +\infty$ in the last inequality and combining it with the estimate close to the origin \fref{annexe:coerciviteA*:eq: pres}we proved earlier gives the subcoercivity bound \fref{annexe:eq:estimation de sous coercivite A*}.\\

\underline{Step 2:} Coercivity. We argue by contradiction. We suppose that there exists a sequence of functions $(u_n)_{n\in\mathbb{N}}$ such that, up to a renormalization:
\begin{equation}\label{annexe:eq:estimation un dans argument par l'absurde A*}
\int \frac{|A^*u|^2}{1+y^{2p}}\leq \frac{1}{n}, \ \ \text{and} \ \int \frac{u^2}{y^2(1+y^{2p})}+\int \frac{|\partial_y u|^2}{1+y^{2p}}=1
\end{equation}
From the subcoercivity estimate \fref{annexe:eq:estimation de sous coercivite A*} it implies that:
$$
u_n(1)^2+\int \frac{u_n^2}{1+y^{2p+2+g}}\gtrsim 1 .
$$
And by \fref{annexe:eq:estimation un dans argument par l'absurde A*} we have that $u_n$ is uniformly bounded in $H^1[r,R]$. Hence by compacity and by an extraction argument there exists a limit profile $u_{\infty}\in H^1_{loc}$ such that up to a subsequence,
$$
u_n \rightharpoonup u_{\infty} \ \text{in} \ H^1_{\text{loc}} .
$$
From continuity of functions in $H^1$ in one dimension, and from compactness of the injection $H^1\hookrightarrow L^2$ on compact sets we have also:
$$
u_n \rightarrow u_{\infty} \ \text{in} \ L^2_{\text{loc}}, \ u_n(1)\rightarrow u_{\infty}(1) .
$$
We now show that $u_\infty \neq 0$. We have that $u_n^2(1)\rightarrow u_{\infty}^2(1)$. Indeed the continuity of the $H^1_{loc}$ functions in 1 dimension, the strong convergence $L^2$ and of the equi-continuity of the family $\{ u_n \}$ implies the convergence in $L^{\infty}$. If $u_{\infty}^2(1)\neq 0$, then $u_{\infty} \neq 0$. If $u_{\infty}(1)=0$ then the subcoercivity bound implies that $\int \frac{u_n^2}{1+y^{2p +2+g}}\gtrsim 1$. The local $L^2$ convergence, and the fact that $\int \frac{u_n^2}{y^2(1+y^{2p})}$ is uniformly bounded implies that:
$$
\int \frac{u_n^2}{1+y^{2p+2+g}}\rightarrow \int \frac{u_{\infty}^2}{1+y^{2p+2+g}}.
$$
Hence $\int \frac{u_{\infty}^2}{1+y^{2p+2+g}}>0$ so $u_{\infty}\neq 0$. In any cases we have found: $u_{\infty}\neq 0$. On the other hand from semi-continuity again we have that:
$$
A^* u_{\infty}=0 .
$$
This equation has for unique solution in $H^1$ the function $\Gamma$ up to multiplication by a scalar. Hence:
$$
u_{\infty}=c\Gamma .
$$ 
$c$ is non zero because $u_{\infty}$ is non zero. But:
$$
\int_{y\leq 1} \frac{\Gamma^2}{y^2}\gtrsim \int_{y\leq 1} \frac{y^{d-1}}{y^{2(d-2)+2}}dy=+\infty ,
$$
which contradicts \fref{annexe:eq:condition d'inte pour A}.
\end{proof}

We now focus on the coercivity of the operator $A$.

\begin{lemma}\label{annexe:lem:coercivite A}
\emph{(Weighted coercivity for $A$:)} Let $p$ be a non negative real number. Let $k_0$ and $\delta_0$ be defined by \fref{intro:eq:def k0} ($\delta_0>0$)). Then:
\begin{itemize}
\item[(i)]\emph{case $p$ small:} if $0\leq p < k_0+\delta_0-1$, then there exists a constant $c_p>0$ such that for all $u\in H^1_{\text{rad},\text{loc}}(\mathbb{R}^d)$ satisfying:
\begin{equation}\label{annexe:eq:condition d'inte pour A}
\int_{y\geq 1} \frac{u^2}{y^{2p+2}}<+\infty ,
\end{equation}
there holds the coercivity:\footnote{the quantities in the coercivity estimate need not be finite.}
\begin{equation}
\int \frac{|Au|^2}{1+y^{2p}}\geq c_k \left[ \int \frac{|\partial_y u |^2}{1+y^{2p}}+\frac{u^2}{y^2(1+y^{2p})}   \right] .
\end{equation}
\item[(ii)]\emph{case $p$ large:} let $p> k_0+\delta_0-1$, let $M$ be large enough (depending on $d$ and $p$ only), then there exists $c_{M,p}>0$ such that if $u\in H^1_{\text{rad},\text{loc}}$ satisfies:
\begin{equation}
\langle u, \Phi_M \rangle=0 .
\end{equation}
then:\footnote{idem.}
\begin{equation}
\int \frac{|Au|^2}{1+y^{2p}}\geq c_{M,p} \left[ \int \frac{|\partial_y u|^2}{1+y^{2p}} +\frac{u^2}{y^2(1+y^{2p})}\right] .
\end{equation}
\end{itemize}
\end{lemma}

\begin{proof}[Proof of Lemma \ref{annexe:lem:coercivite A}] 
As for $A^*$ we first show a subcoercivity bound and then show that if we want to violate the Hardy type inequality, one must get closer and closer to the zero of $A$ which is $\Lambda^{(1)}Q$, but this is impossible due to integrability conditions in the case $p$ small and due to the orthogonality condition for the case $p$ large.\\

 \underline{Step 1:} \emph{subcoercivity}. Let $p\geq 0$. Then we claim that if $u$ satisfies \fref{annexe:eq:condition d'inte pour A}:
\begin{equation}
\int \frac{|Au|^2}{1+y^{2p}}\geq c\left[ \int \frac{|\partial_y u|^2}{1+y^{2p}}+\frac{u^2}{y^2(1+y^{2p})}\right]- \frac{1}{c} \left[ u^2(1)+\int \frac{u^2}{1+y^{2p+2+g}} \right] ,
\end{equation}
for a universal constant $c>0$. We start by computing close to the origin using \fref{annexe:coercivitenormes:eq:def A}, with the help of the Hardy inequality close to the origin \fref{lemme:annexe:hardy1}:
$$
\begin{array}{r c l}
\int_{y\leq 1}|Au|^2 &= & \int_{y\leq 1} |\partial_y u|^2+\int_{y\leq 1} O(u^2)+\int u\partial_yu O(1) \\
&\geq& c\left( \int_{y\leq 1} |\partial_y u|^2+\frac{u^2}{y^2} \right)-\frac{1}{c}\left(u^2(1)+\int_{y\leq 1}u^2\right)+\int u\partial_yu O(1)  .
\end{array} 
$$
We apply Cauchy-Schwarz and Young inequality to control the last term:
$$
\left| \int u\partial_yu O(1) |\leq \epsilon C \int_{y\leq 1} \right| \partial_y u|^2+\frac{C}{\epsilon}\int_{y\leq 1}u^2.
$$
Taking $\epsilon$ small enough gives close to the origin:
\be \label{annexe:coerciviteA:eq:expression pres}
\int_{y\leq 1}|Au|^2\geq c\left( \int_{y\leq 1} |\partial_y u|^2+\frac{u^2}{y^2} \right)-\frac{1}{c}\left(u^2(1)+\int_{y\leq 1}u^2\right).
\ee
Away from the origin, we use the asymptotics \fref{annexe:coercivitenormes:eq:asymptotique W} of the potential $W$ to derive:
\be \label{annexe:coerciviteA:eq:expression loin}
\begin{array}{r c l}
\int_1^R \frac{|Au|^2}{y^{2p}}&=&\int_2^R \frac{1}{y^{2p}}\left[ \partial_y u +\frac{\gamma}{y}u+O\left( \frac{u^2}{y^{1+g}}\right)\right]^2 \\
&= & \int_1^R \frac{1}{y^{2p}}\left[ \partial_y u +\frac{\gamma}{y}u\right]^2  + \int_1^R O\left(\frac{u}{y^{2p+1+g}}\right)\left( \partial_y u+uO\left(\frac{1}{y}\right)\right) .
\end{array}
\ee
This time we let $v=y^{\gamma}u$, and $2p'=2p+2\gamma$. We observe: $2p'-(d-2)=2p-2k_0+2-2\delta_0<0$ in the case $p$ small and $>0$ in the case $p$ large. For $p$ small we have from \fref{annexe:hardy:eq:calcul sous critique}:
\be \label{annexe:coerciviteA:eq:controle loin}
\begin{array}{r c l}
\int_1^R \frac{1}{y^{2p}}\left[ \partial_y u +\frac{\gamma}{y}u\right]^2 = \int_1^R \frac{|\partial_y v|^2}{y^{2p'}} &\geq & c\int_1^R \frac{v^2}{y^{2p'+2}}-\frac{R^{d-2p'-2}}{d-2-2p'}v^2(R) \\
&=& c\int_1^R \frac{u^2}{y^{2p+2}}-\frac{R^{d-2-2k}}{d-2-2p}u^2(R) .
\end{array}
\ee
As we did in the proof of the sub-coercivity estimate for $A^*$, the identity \fref{annexe:coerciviteA:eq:expression loin} and the control \fref{annexe:coerciviteA:eq:controle loin} imply using Cauchy-Schwarz and Young inequality:
$$
\int_1^R \frac{|Au|^2}{y^{2p}} \geq c'\left( \int_1^R \frac{u^2}{y^{2p+2}}+\frac{|\partial_y u|^2}{y^{2p}}\right) -\frac{1}{c}\left(\frac{R^{d-2-2p}}{d-2-2p}u^2(R) +\int_1^R \frac{u^2}{y^{2p+2+g}} \right).
$$
The integrability condition \fref{annexe:eq:condition d'inte pour A} gives that along a sequence $R_n$ the $u(R_n)$ term goes to zero. This allow us to conclude that if $\frac{|\partial_y u|^2}{y^{2p}}$ is not integrable, then $\frac{|A^*u|^2}{y^{2p}}$ is not integrable neither. This gives the Hardy inequality in the case the quantities are infinite. We can now suppose that the involved quantities are finite. We go to the limit in the previous equation along $R_n$ and combine it with \fref{annexe:coerciviteA:eq:expression pres} to obtain the subcoercivity estimate.\\
\\
For $p$ large we are in the supercritical case in the standard Hardy inequality for $v$. We can do verbatim the same reasoning we did for the proof of the subcoercivity estimate for $A^*$.\\

\underline{Step 2:} \emph{Coercivity}. We argue by contradiction. If the hardy inequality we want to show was wrong, there would exist a sequence $(u_n)_{n\in \mathbb{N}}$, such that:
$$
\int \frac{|\partial_y u_n|^2}{1+y^{2p}}+\frac{u_n^2}{y^2(1+y^{2p})}=1, \ \int \frac{|Au|^2}{1+y^{2p}}\rightarrow 0 .
$$
From the subcoercivity estimate implies:
$$
u_n^2(1)+\int \frac{u_n^2}{1+y^{2p+2+g}}\gtrsim 1 ,
$$
and $u_{n}\rightharpoonup u_{\infty}$ in $H^1_{\text{loc}}(]0,+\infty[)$. The quantities go the same way to the limit and we find that $u_{\infty}$ is not zero and must satisfy:
$$
Au=0 .
$$
This implies $u_{\infty}=c\Lambda^{(1)} Q$, $c\neq 0$.\\
\\
If $k\geq k_0$ then the orthogonality condition goes to the limit with the weak topology and we find $\langle u_{\infty}, \Phi_M\rangle =0$ which violates \fref{annexe:coercivitenormes:eq:proprietes PhiM}. If $k\leq k_0-1$, we have from lower semi continuity that:
$$
\int \frac{u_{\infty}^2}{1+y^{2p+2}}<+\infty ,
$$
but $\Lambda^{(1)}Q$ does not satisfy this inequality because as $-2\gamma-2p-2+d=2(k_0-p)-2(1-\delta_0)>0$ we have:
$$
\int \frac{\Lambda^{(1)}Q^2}{1+y^{2p+2}}=+\infty .
$$
In both cases there is a contradiction. Hence the lemma are proven.
\end{proof}

Once the coercivity properties of $A$ and $A^*$ have been established, we can turn to the core of this part: the coercivity estimates for the adapted norms provided some orthogonality conditions are satisfied.

\begin{lemma}[Coercivity of $\mathcal{E}_k$]\label{annexe:lem:coercivite des normes adaptees}
We still assume $\delta_0\neq 0$. $k$ denotes an integer. We recall that $u_j$, the $j$-th adapted derivative of $u$, is defined in \fref{linearized:eq:def derivees adaptees}.
\begin{itemize}
\item[(i)] \emph{case $k$ small} Let $0\leq k \leq k_0$ and $0\leq \delta< \delta_0$. Then there exists a constant $c_{k,\delta}>0$ such that for all $u\in H^{k}_{\text{rad},\text{loc}}(\mathbb{R}^d)$ satisfying:
\begin{equation}\label{annexe:eq:condition int coercivite normes adaptees}
\sum_{p=0}^{k} \int \frac{u_p^2}{1+y^{2k-2p}}<+\infty ,
\end{equation}
there holds:
\begin{equation}
\int \frac{u_k^2}{1+y^{2\delta}} \geq c_k \sum_{p=0}^{k-1} \int \frac{u_p^2}{1+y^{2k-2p+2\delta}} .
\end{equation}
\item[(ii)] \emph{case $k$ large} Let $k\geq k_0+1$ and $0\leq \delta< \delta_0$, let $j=E(\frac{k-k_0}{2})$. Then for $M=M(k)$ large enough, there exists $c_{M,k}>0$ such that for all $H^{k}_{\text{loc},\text{rad}}(\mathbb{R}^d)$ satisfying:
\begin{equation} \label{annexe:eq:condition inte pour coercivite}
\sum_{p=0}^{k} \int \frac{u_p^2}{1+y^{2k-2}}<+\infty \ \text{and} \ \langle u, \mathcal{L}^p \Phi_M \rangle=0, \ \text{for} \ 0\leq p \leq j-1,
\end{equation}
there holds:
\begin{equation}
\int \frac{u_k^2}{1+y^{2\delta}} \geq c_{M,k} \sum_{p=0}^{k-1} \int \frac{u^2_p}{1+y^{2k-2p+2\delta}} .
\end{equation}
\end{itemize}
\end{lemma}

\begin{corollary}[Coercivity of $\mathcal{E}_{s_L}$] \label{annexe:cor:coercivite mathcalEsL}
Let $L$ and $\sigma$ be defined by \fref{linearized:eq:def L} and \fref{thetrapped:eq:def sigma}  ($L$ is odd) and $0\leq \delta<\delta_0$. Then there exists a constant $c>0$ such that for all radial $\bos{\varepsilon}\in \dot{H}^{s_L}\times \dot{H}^{s_L-1}\cap \dot{H}^{\sigma}\times \dot{H}^{\sigma-1}$ satisfying:
\be \label{annexe:coercivitenormes:eq:conditions ortho}
\langle \bos{\varepsilon},\bos{H}^{*i}\bos{\Phi}_M \rangle=0  \ \text{for} \ 0\leq i \leq L,
\ee
there holds:
\be \label{annexe:coercivitemathcalE:eq:coercivitemathcalE}
\sum_{p=0}^{s_L-1} \int \frac{|\varepsilon^{(1)}_p|^2}{1+y^{2s_L-2p+2\delta}}+\sum_{p=0}^{s_L-2} \int \frac{|\varepsilon^{(2)}_p|^2}{1+y^{2s_L-2-2p+2\delta}}\leq c\left(  \int \frac{|\varepsilon^{(1)}_{s_L}|^2}{1+y^{2\delta}}+ \int \frac{|\varepsilon^{(2)}_{s_L-1}|^2}{1+y^{2\delta}} \right)
\ee
\be \label{annexe:coercivitenormes:eq:mathcalEsL sim HsL}
\parallel \bos{\varepsilon} \parallel_{ \dot{H}^{s_L}\times \dot{H}^{s_L-1}}^2 \leq c \mathcal{E}_{s_L} <+\infty,
\ee
the adapted derivatives $u_k$ being defined by \fref{linearized:eq:def derivees adaptees} and $\mathcal{E}_{s_L}$ being defined by \fref{thetrapped:eq:def mathcalEsL}.
\end{corollary}

\begin{proof}[Proof of Corollary \fref{annexe:cor:coercivite mathcalEsL}]
\underline{Step 1:} Proof that $\mathcal{E}_{s_L}<+\infty$. From the equivalence between Laplace derivatives and adapted ones, \fref{annexe:equivalencenormes:eq:equivalences normes}, one has:
$$
\int |\varepsilon^{(1)}_{s_L}|^2 \leq C\sum_{i=0}^{s_L} \int \frac{|D^i\varepsilon^{(1)}|^2}{1+y^{2s_L-2i}}.
$$
For $\sigma \leq i \leq s_L$ one has by interpolation $\int |D^i\varepsilon^{(1)}|^2<+\infty$, hence $\int \frac{|D^i\varepsilon^{(1)}|^2}{1+y^{2s_L-2i}}<+\infty$. For $0\leq i \leq \sigma$ one has $ \frac{D^i\varepsilon^{(1)}}{1+y^{\sigma-i}}\in L^2$ from the Hardy inequality \fref{annexe:hardyfrac:eq:hardyfrac}. Consequently in that case we also have $\frac{D^i\varepsilon^{(1)}}{1+y^{s_L-i}} \in L^2$. This proves:
$$
\int |\varepsilon^{(1)}_{s_L}|^2<+\infty.
$$
Similarily one has $\int |\varepsilon^{(2)}_{s_L-1}|^2<+\infty$, implying $\mathcal{E}_{s_L}<+\infty$. \underline{Step 2:} Proof of the coercivity estimate. We want to apply the previous Lemma \ref{annexe:lem:coercivite des normes adaptees} for $k=s_L$. We have seen in the previous step 1 that the integrability condition \fref{annexe:eq:condition inte pour coercivite} is met. Now from the formula \fref{thetrapped:eq:puissances de H} giving the powers of $H^{*}$ we compute that the orthogonality condition \fref{annexe:coercivitenormes:eq:conditions ortho} implies:
$$
\langle \varepsilon^{(1)},\mathcal{L}^i \Phi_M \rangle=\langle \varepsilon^{(2)}, \mathcal{L}^i \Phi_M \rangle=0 \  \text{for} \ 0\leq i \leq \frac{L-1}{2}.
$$
We compute: $E\left[\frac{k-k_0}{2}\right]=E\left[\frac{L+k_0+1-k_0}{2}\right]=\frac{L+1}{2}$. Therefore the Lemma \ref{annexe:lem:coercivite des normes adaptees} applies and gives the bound \fref{annexe:coercivitemathcalE:eq:coercivitemathcalE}. Now we use the equivalence between Laplace and adapted derivatives \fref{annexe:equivalencenormes:eq:equivalences normes}, with the bound we just proved for \fref{annexe:coercivitemathcalE:eq:coercivitemathcalE} for $\delta=0$ and it yields \fref{annexe:coercivitenormes:eq:mathcalEsL sim HsL}.
\end{proof}

\begin{proof}[Proof of Lemma \ref{annexe:lem:coercivite des normes adaptees}]
\underline{case $k$ small:} We suppose $1\leq k\leq k_0$, and that $u$ is a function satisfying the conditions of the lemma. We have, depending on the parity of $k$:
$$
u_k=A u_{k-1} \ \text{or} \ u_k=A^* u_{k-1} .
$$
In both cases, the conditions required to apply to $u_{k-1}$ Lemma \ref{annexe:lem:coercivite A} or Lemma \ref{lem:coercivite pour A*} are fulfilled. Consequently:
$$
\int \frac{u_{k}^2}{1+y^{2\delta}} \gtrsim \int \frac{u_{k-1}^2}{1+y^{2+2\delta}} .
$$
If $k-1=0$ we have finished. If not, then again, $u_{k-1}=Au_{k-2}$ or $u_{k-1}=A^* u_{k-2}$ and in both cases we can apply Lemma \ref{annexe:lem:coercivite A} or Lemma \ref{lem:coercivite pour A*} which gives:
$$
\int \frac{u_k^2}{1+y^{2\delta}} \gtrsim \int \frac{u_{k-1}^2}{1+y^{2+2\delta}} \gtrsim \int \frac{u_{k-2}^2}{1+y^{4+2\delta}} .
$$
We can iterate $k$ times what we did previously to obtain:
$$
\int \frac{u_k^2}{1+y^{2\delta}} \gtrsim \int \frac{u_{k-1}^2}{1+y^{2+2\delta}} \gtrsim ... \gtrsim \int \frac{u_1^2}{1+y^{2(k-2)+2\delta}} \gtrsim \int \frac{u^2}{1+y^{2k+2\delta}} ,
$$
which gives the result in that case.\\

\underline{Case $k$ large:} Suppose first that $k\geq k_0+1$ and that $j=\frac{k-k_0}{2} \in \mathbb{N}^*$, so $k=k_0+2j$. We can apply the result for $k$ small we just showed to derive:
$$
\int \frac{u_k^2}{1+y^{2\delta}} \gtrsim \int \frac{u_{k-k_0}^2}{1+y^{2k_0+2\delta}}= \int \frac{u_{2j}^2}{1+y^{2k_0+2\delta}} .
$$
Since $2j$ is even we know that: $u_{2j}=A^*A...A^*Au=A^* u_{2j-1}$ and we can apply Lemma \ref{lem:coercivite pour A*} to find:
$$
\int \frac{u_{2j}^2}{1+y^{2k_0+2\delta}}\gtrsim \int \frac{u_{2j-1}^2}{1+y^{2k_0+2+2\delta}}= \frac{Au_{2j-2}^2}{1+y^{2k_0+2+2\delta}} .
$$
We need an orthogonality condition for $u_{2j-2}$ in order to go on. This is given by the orthogonality condition on $u$. Indeed:
$$
\langle u_{2j-2},\Phi_M \rangle= \langle u, \mathcal{L}^{j-1}\Phi_M \rangle=0 .
$$
Hence:
$$
\int \frac{u_{2j-1}^2}{1+y^{2(k_0+1+\delta)}}\gtrsim \int \frac{u_{2j-2}^2}{1+y^{2(k_0+2+\delta)}} .
$$
We need exactly the $j$ orthogonality conditions to iterate like that till we reach $0$.\\
\\
Suppose now that $k=k_0+2j+1$. Then it works the same, indeed without use of orthogonality conditions:
$$
\int \frac{u_k^2}{1+y^{2\delta}} \gtrsim \int \frac{u_{k-1}^2}{1+y^{2+2\delta}}\gtrsim... \gtrsim \frac{u_{k-k_0}^2}{1+y^{2k_0+2\delta}}= \int \frac{|Au_{2j}|^2}{1+y^{2k_0+2\delta}} .
$$
We have exactly $j$ orthogonality conditions to go down to zero as we did before:
$$
\int \frac{|Au_{2j}|^2}{1+y^{2k_0+2\delta}}\gtrsim \int \frac{u_{2j}^2}{1+y^{2k_0+2\delta}}\gtrsim... \gtrsim \int \frac{u^2}{1+y^{2k+2\delta}} .
$$
This ends the proof.
\end{proof}


\section{Specific bounds for the analysis}

We make use here of the tools established in the last subsection to control $\bos{\varepsilon}$. Again, the use of such estimate is standard in blow-up issues, and we refer to the papers quoted in Appendix C. Although their proofs are not very hard to write once one has the previous results, we put it here for the reader's convenience. aAs the non-linearity just acts on $\varepsilon^{(1)}$ we just state results for this coordinate.

\begin{lemma}\label{annexe:lem:interpolation varepsilon}
Under the bootstrap conditions \fref{eq:bootstrap estimations sur epsilon} of Proposition \ref{prop:bootstrap} and provided that $\varepsilon$ satisfies the orthogonality conditions \fref{eq:condition d'ortho pour epsilon} there holds ($\mathcal{E}_{s_L}$ and $\mathcal{E}_{\sigma}$ being defined in \fref{thetrapped:eq:def mathcalEsL} and \fref{thetrapped:eq:def mathcalEsigma}):
\begin{itemize}
\item[(i)] \emph{Improved Hardy inequality:} For $j\in \mathbb{N}$ and $p>0$ satisfying $\sigma\leq j+p\leq s_L$:
\begin{equation}
\int_{y\geq 1} \frac{|\partial_y^j \varepsilon^{(1)}|^2}{1+y^{2p}} \leq C(M)\mathcal{E}_{\sigma}^{\frac{s_L-(j+p)}{s_L-\sigma}} \mathcal{E}_{s_L}^{\frac{j+p-\sigma}{s_L-\sigma}} ,
\end{equation}
\item[(ii)]\emph{$L^{\infty}$ control:}
\begin{equation}
\parallel \epsilon^{(1)} \parallel_{L^{\infty}} \leq C(K_1,K_2,M) \sqrt{\mathcal{E}_{\sigma}}b_1^{\left( \frac{d}{2}-\sigma\right)+\frac{2\alpha}{(p-1)L} +O\left( \frac{\sigma-s_c}{L} \right) } ,
\end{equation}
\item[(iii)] \emph{Weighted $L^{\infty}$ bound:} for $0<a<\frac{d}{2}$
\begin{equation}
\left\Vert \frac{\epsilon^{(1)}}{1+x^a} \right\Vert_{L^{\infty}} \leq C(K_1,K_2,M) \sqrt{\mathcal{E}_{\sigma}}b_1^{a+\left( \frac{d}{2}-\sigma\right)+\frac{(\frac{2}{p-1}+a)\alpha}{L}+ O\left( \frac{\sigma-s_c}{L} \right) } .
\end{equation}
\end{itemize}
\end{lemma}

\begin{proof}[Proof of Lemma \ref{annexe:lem:interpolation varepsilon}]
\underline{Proof of (i):} Let $j\in \mathbb{N}$ and $p$ satisfying $\sigma\leq j+p\leq s_L$. For a slow decaying potential, ie if $p$ satisfies in addition $p<\frac{d}{2}$ then the equivalence between Laplace derivatives and $\partial_y$ ones away from the origin, together with the weighted Hardy inequality (Lemma \ref{annexe:lem: hardy frac a poids}) gives:
$$
\int \frac{|\partial_y^j \varepsilon^{(1)}|^2}{1+y^{2p}}\leq C \int |\nabla^{j+p}\varepsilon^{(1)}|^2 ,
$$
and we conclude by interpolation. We claim now that:
$$
\sum_{i=0}^{s_L} \int_{y\geq 1} \frac{|\partial_y^i \varepsilon^{(1)}|^2}{1+y^{2(s_L-i)}}\leq C(M) \mathcal{E}_{s_L} .
$$
Indeed, from the equivalence between $\partial_y$ and adapted derivatives (Lemma \ref{annexe:lem:equivalence degre}), and from coercivity we have:
$$
\sum_{i=0}^{s_L} \int_{y\geq 1} \frac{|\partial_y^i \varepsilon^{(1)}|^2}{1+y^{2(s_L-i)}} \sim \sum_{i=0}^{s_L} \int_{y\geq 1} \frac{|\varepsilon_i^{(1)}|^2}{1+y^{2(s_L-i)}} \leq C(M) \mathcal{E}_{s_L} .
$$
This claim implies that for a fast decaying potential, ie $p=s_L-j$:
$$
\int \frac{|\partial_y^j \varepsilon^{(1)}|^2}{1+y^{2p}}\leq \mathcal{E}_{s_L} .
$$
Now, for $\frac{d}{2}\leq p\leq s_l-j$ we interpolate the last two results, as for $a\leq b\leq c$:
$$
\frac{|\varepsilon^{(1)}|^2}{1+y^{2b}}\sim \left( \frac{|\varepsilon|^2}{1+y^{2a}}\right)^{\frac{c-b}{c-a}} \left( \frac{|\varepsilon|^2}{1+y^{2b}}\right)^{\frac{b-a}{c-a}}
$$
and this gives (i).\\

\underline{Proof of (ii)}. We prove it for $\varepsilon^{(1)}$, the proof for the second coordinate being similar. By the coercivity bound \fref{annexe:coercivitenormes:eq:mathcalEsL sim HsL} we have that:
$$
\parallel \nabla^{s_L} \varepsilon^{(1)} \parallel_{L^2}^2 \leq C(M) \mathcal{E}_{s_L} .
$$
We have by interpolation that for all $\sigma \leq k \leq s_L$, $\nabla^{k}\varepsilon^{(1)}\in L^2$ with the control
$$
\parallel \nabla^{k}\varepsilon^{(1)} \parallel_{L^2}^2 \leq C(M) \mathcal{E}_{\sigma}^{\frac{s_L-k}{s_L-\sigma}} \mathcal{E}_{s_L}^{\frac{k-\sigma}{s_L-\sigma}} .
$$ 
Denoting by $\hat{\varepsilon^{(1)}}$ the Fourier transform of $\varepsilon^{(1)}$ we have:
$$
|\varepsilon^{(1)}(y)|\leq \int_{|\xi|\leq 1} \frac{|\hat{\varepsilon^{(1)}}||\xi|^{\frac{k_1}{2}}}{|\xi|^{\frac{k_1}{2}}}+\int_{|\xi|\geq 1} \frac{|\hat{\varepsilon^{(1)}}||\xi|^{\frac{k_2}{2}}}{|\xi|^{\frac{k_2}{2}}} \lesssim \parallel \nabla^{k_1}\varepsilon^{(1)} \parallel_{L^2} +\parallel \nabla^{k_2}\varepsilon^{(1)} ,\parallel_{L^{2}}
$$
with $\sigma<k_1<\frac{d}{2}<k_2<s_L$. Using the interpolation bound previously derived and taking $k_1,k_2\rightarrow\frac{d}{2}$ gives:
$$
\begin{array}{ r c l}
|\varepsilon^{(1)}(y)|^2&\leq& C\mathcal{E}_{\sigma}^{\frac{s_L-\frac{d}{2}}{s_L-\sigma}} \mathcal{E}_{s_L}^{\frac{\frac{d}{2}-\sigma}{s_L-\sigma}} \leq C\mathcal{E}_{\sigma} b_1^{2\left(L+(1-\delta_0)(1+\eta)-(\sigma-s_c)\frac{\ell}{\ell-\alpha}\right)(\frac{d}{2}-\sigma)\left( \frac{1}{s_L}+\frac{\sigma}{s_L^2}+O(\frac{1}{L^3)}\right)} \\
&=& C\mathcal{E}_{\sigma} b_1^{(\frac{d}{2}-\sigma)2+\frac{\frac{2}{p-1}(2\eta(1-\delta_0)+2\alpha)}{L}+O\left(\frac{(\sigma-s_c)}{L} \right)}\leq  C\mathcal{E}_{\sigma} b_1^{(\frac{d}{2}-\sigma)2+\frac{\frac{2}{p-1}\alpha}{L}+O\left(\frac{(\sigma-s_c)}{L} \right)}
\end{array}
$$
which gives the result.\\

\underline{Proof of (iii)} Take $a\geq 1$, $\alpha\leq a\ll s_L$. Then from (i):
$$
\begin{array}{r c l}
\parallel\nabla^{E\left[\frac{d}{2}+1 \right]}\frac{\varepsilon^{(1)}}{1+y^{a}}\parallel_{L^2}^2 &\sim& \int \left|D^{E\left[\frac{d}{2}+1 \right]}\left(\frac{\varepsilon^{(1)}}{1+{a}} \right)\right|^2 \leq  C(M)\mathcal{E}_{\sigma}^{\frac{s_L-E\left[\frac{d}{2}+1 \right]-a}{s_L-\sigma}} \mathcal{E}_{\sigma}^{\frac{E\left[\frac{d}{2}+1 \right]+a-\sigma}{s_L-\sigma}} .
\end{array}
$$
And we estimate the same way $\parallel \nabla^{E\left[\frac{d}{2}-1 \right]}\frac{\varepsilon^{(1)}}{1+y^{a}} \parallel_{L^2}^2$. We can the interpolate this two estimations to have an estimate for $\parallel \frac{\varepsilon^{(1)}}{1+y^{a}} \parallel_{L^{\infty}}$. By calculating the exponents the same way we did for the proof of (ii) we get the result of the lemma for $a$. Now we can interpolate this result with (ii) to conclude for any exponent $0\leq a \leq s_L$.
\end{proof}

\end{appendix}

 \frenchspacing
\bibliographystyle{plain}

\end{document}